\documentclass[11pt,a4paper,reqno]{amsart}
\usepackage{amsmath,amsthm,amsfonts,amssymb,bm,wasysym}
\usepackage{comment}
\usepackage{fullpage,bbm}
\usepackage{graphicx}
\usepackage{tikz,ifthen}
\usepackage{lipsum} 
\usepackage{hyperref}
\usepackage{enumerate}
\usepackage{subcaption}
\usepackage{booktabs}
\usepackage{multirow}
\usepackage{array}
\usetikzlibrary{arrows, positioning, chains}
\newenvironment{spacedtable}[2]{%
\begingroup
\renewcommand{\arraystretch}{#1}
  \setlength{\tabcolsep}{#2}
}{%
  \endgroup
}

\usetikzlibrary{calc}
\usetikzlibrary{intersections}
\makeatletter
\newcommand{\storecoords}[3]{%
  \path let \p1=(#1) in
    \pgfextra{%
      \pgfmathsetmacro{\@tempx}{\x1/1pt}%
      \pgfmathsetmacro{\@tempy}{\y1/1pt}%
      \global\edef#2{\@tempx}%
      \global\edef#3{\@tempy}%
    }%
}
\makeatother
\usepackage{mathtools}
\usetikzlibrary{decorations.pathreplacing}
\usepackage{color,latexsym,amsfonts,amssymb,bbm,comment}
\usepackage{hyperref}
\usepackage{amsmath,cite}
\usepackage{amsthm}
\usepackage{fullpage}
\usepackage{graphicx}
\usepackage{etoolbox}

\usepackage{caption,subcaption}
\usepackage[left=1.9cm,right=1.9cm,top=2.5cm,bottom=2.5cm]{geometry}
\usepackage{longtable} 
\usepackage{booktabs} 
\usepackage{thmtools}
\usepackage{thm-restate}

\usepackage{cleveref}

\usepackage{makecell}
\usepackage{multirow}
\usepackage{mathtools}
\usepackage{mathrsfs}
\usepackage{pgfplots}
\usepackage[utf8]{inputenc}
\usepackage[nolist,nohyperlinks]{acronym}
\usepackage{chngcntr}
\counterwithin{table}{section}

\usepackage{imakeidx}
\makeindex[options= -f] 

\usepackage{setspace}
\usepackage{etoolbox} 

\newcommand{\interior}[1]{%
{\kern0pt#1}^{\mathrm{o}}}

\begin{acronym}
    \acro{PLP}{Poisson line process}
    \acro{MLP}{Manhattan line process}
    \acro{MLCP}{Manhattan line Cox process}
    \acro{PLCP}{Poisson line Cox process}
\end{acronym}

\newtheorem{theorem}{Theorem}
\newtheorem{proposition}[theorem]{Proposition}
\newtheorem{corollary}[theorem]{Corollary}
\newtheorem{lemma}[theorem]{Lemma}

\newtheorem{observation}[theorem]{Observation}
\newtheorem{discussion}[theorem]{Discussion}

\newtheoremstyle{caseStyle}   
  {\topsep}                   
  {\topsep}                   
  {}                          
  {}                          
  {\bfseries}                 
  {.}                         
  {0.5em}                     
  {}                          

\theoremstyle{caseStyle}
\newcounter{proofno}
\AtBeginEnvironment{proof}{\refstepcounter{proofno}}

\newcounter{case}[subsubsection]

\renewcommand{\thecase}{Case \arabic{case}}

\newenvironment{case}[1][]{\refstepcounter{case}\par\noindent\textbf{\textit{\thecase.}} #1}
  {\par}

\newcounter{subcase}[case]
\renewcommand{\thesubcase}{Subcase \arabic{case}.\arabic{subcase}}

\newenvironment{subcase}[1][]
{\refstepcounter{subcase}\par\noindent\textbf{\textit{\thesubcase.}} #1 }{\par}

\newcounter{subsubcase}[subcase]
\renewcommand{\thesubsubcase}{Subsubcase \arabic{case}.\arabic{subcase}.\arabic{subsubcase}}

\newenvironment{subsubcase}[1][]
{\refstepcounter{subsubcase}\par\noindent\textbf{\textit{\thesubsubcase.}} #1 }{\par}

\newcommand{\startappendix}{%
  \appendix 
  \makeatletter
  \@removefromreset{case}{subsubsection}
  \@addtoreset{case}{subsection}
  \makeatother
}

\newtheoremstyle{remarkstyle}
  {0.35em\topsep}   
  {\topsep}   
  {\itshape}  
  {}          
  {\bfseries} 
  {.}         
  {.5em}      
  {}          

\theoremstyle{remarkstyle}
\newtheorem{remark}[theorem]{Remark}

\newtheoremstyle{definition} 
  {\topsep}                     
  {\topsep}                     
  {\itshape}                    
  {}                            
  {\bfseries\itshape}                   
  {.}                           
  {.5em}                        
  {}                            

\theoremstyle{definition}
\newtheorem{definition}[theorem]{Definition}

\newtheoremstyle{identity} 
  {\topsep}                     
  {\topsep}                     
  {\itshape}                    
  {}                            
  {\bfseries\itshape}                   
  {.}                           
  {0.5em}                        
  {}                            

\theoremstyle{identity}
\newtheorem{identity}{Identity}

\numberwithin{equation}{section}
\numberwithin{theorem}{section}
\numberwithin{lemma}{section}
\numberwithin{proposition}{section}
\numberwithin{corollary}{section}
\numberwithin{observation}{section}
\numberwithin{definition}{section}
\numberwithin{remark}{section}
\numberwithin{discussion}{section}
\numberwithin{figure}{section}
\theoremstyle{plain}

\newcommand{\bbinom}[2]{\scalebox{0.7}{$\bigg[ \begin{array}{c} #1 \\ #2 \end{array} \bigg]$}}

\newcommand{\plainref}[1]{%
  \tikz[baseline=(X.base)] \node[inner sep=1pt, minimum width=1em, minimum height=2em, align=center](X){\ref{#1}};%
}

\newcommand{\blue}[1]{\textcolor{blue}{#1}}
\newcommand{\red}[1]{\textcolor{red}{#1}}

\newcommand{\one}   {{\mathds{1}}}

\newcommand{\R}     {\mathbb{R}} 
\newcommand{\Z}     {\mathbb{Z}}
\newcommand{\N}     {\mathbb{N}}
 
\renewcommand{\P}   {\mathbb{P}} 
 
\newcommand{\E}     {\mathbb{E}}

\newcommand{\Acal}  {{\mathcal A}}
\newcommand{\Bcal}  {{\mathcal B}}
\newcommand{\Ccal}   {{\mathcal C }}

\newcommand{\Fcal}   {{\mathcal F }} 
 
\newcommand{\Hcal}   {{\mathcal H }}

\newcommand{\Lcal}   {{\mathcal L }} 
\newcommand{\Mcal}   {{\mathcal M }}

\newcommand{\Pcal}   {{\mathcal P }} 
 
\newcommand{\Rcal}   {{\mathcal R }} 
 
\newcommand{\Tcal}   {{\mathcal T }} 
\newcommand{\Ucal}   {{\mathcal U }} 
\newcommand{\Vcal}   {{\mathcal V }} 
\newcommand{\Wcal}   {{\mathcal W }}

\newcommand{\nn}   {{\nonumber}} 
\newcommand{\half}   {{\frac{1}{2}}}

\def\one{\mathbbmss{1}}

\def\a{\alpha}

\def\eps{\varepsilon}
\def\p{\phi}

\def\g{\gamma}
\def\la{\lambda}

\def\t{{\tau}}
\def\th{\theta}
\def\x{\xi}

\def\Th{{\Theta}}
\def\La{\Lambda}

\pgfplotsset{compat=1.18}

\begin{document}
\author{Fran\c{c}ois Baccelli\textsuperscript{1,2}}
\thanks{\textsuperscript{1} INRIA Paris,  \textsuperscript{2} \'{E}cole Normale Sup\'{e}rieure Paris, \textsuperscript{3}T\'{e}l\'{e}com Paris, Correspondence: sanjoy.jhawar@telecom-paris.fr}
\address[FB]{MATHNET, INRIA Paris, 48 Rue Barrault, 75013 Paris, France, and \'{E}cole Normale Sup\'{e}rieure Paris, 45 Rue d'Ulm, 75005 Paris, France}
\email{francois.baccelli@inria.fr}

\author{Sanjoy Kumar Jhawar\textsuperscript{3,1}}
\address[SKJ]{INFRES, T\'{e}l\'{e}com Paris, 19, place Marguerite Perey, 91123 Palaiseau, France, and MATHNET, INRIA Paris, 48 Rue Barrault, 75013 Paris, France}
\email{sanjoy.jhawar@telecom-paris.fr, sanjoy-kumar.jhawar@inria.fr}

\title{On a Class of Dynamical Poisson-Voronoi Tessellations}


\date{\today}

\vspace{-2in}

\begin{abstract}
Consider a dynamical network model featuring mobile stations on the
Euclidean plane. The initial locations of the stations are given by
a homogeneous Poisson point process. The stations are all moving at
a constant speed and in a random direction. Consider fixed users
located in the Euclidean plane, which are served by the mobile stations.
Each user stays connected to the nearest station at any given point
of time. Since the stations are moving, a user disconnects and
connects with different stations over time, by always selecting which
ever station is the closest. This gives rise to a dynamical version of
the Poisson-Voronoi tessellation. The focus of this paper is on the
sequence of ``handover'' events of a typical user, which are the
epochs when its association changes. This defines a point process on the
time-axis, the ``handover point process''. We show that this point process
is stationary and we determine its main properties, in particular its
intensity and the joint distribution of its inter-event times. We also
analyze the handover Palm distributions of several variables of
practical interest. This includes the distance to the closest mobile
stations and the point process of all other mobile stations at handover
epochs. The analysis is conducted both in the single-speed and in the
multi-speed scenarios. Closed form expressions are obtained in both cases. It leads to the identification of the three
dimensional state variables that ``Markovize'' the association dynamics.
The analysis is based on  a specific system of non-compact particles.
The motivations are in the modeling of low or medium orbit satellite
wireless communication networks. The model studied here is a planar
``caricature'' of this problem, which is initially defined on the sphere.
\end{abstract}

\maketitle

\vspace{-0.38in}


\vspace{-0.07in}


\vspace{-0.15in}
{\small 
\tableofcontents
}
\vspace{-0.45in}

\section{Introduction}\label{sec-Moti}
Consider stations moving at constant speed and in random directions on the Euclidean plane.
A typical initial condition for the stations is the set of atoms of a homogeneous Poisson point process $\Phi$ on $\R^2$.
Consider a user located at the origin (a user located at any other locus would
see the same statistics as those described below).
This user is bound to keep connected to the closest station at any time.

Here is a set of practical questions within this setting.
At what frequency does the user have to swap from a station to another one?
These events are called {\em handovers} and the corresponding time points are called {\em handover epochs}.
More generally, what is the full structure of the handover (time) point process, for instance
the distribution of the duration between two handovers or more precisely, the joint distribution of the inter-handover times?

Here are further questions, which assume that one can define 
the state of the planar dynamical system of interest at a typical handover epoch,
or more precisely the Palm distribution of the planar dynamical system
with respect to the time handover point process:
What is the distribution of the distance from the user to the two stations involved in the handover?
What is the distribution of the point process made of the stations not involved in the handover?

For the first set of questions, it is interesting to consider the following dual dynamical
system where the stations are not mobile and the user is mobile in place, with a constant speed
and a random direction. 
Here is a first teaser for the reader: for some motion speed, are there more handovers in the mobile station, fixed
user model or in the dual fixed stations, mobile user case?

For the second set of questions, in the PPP model, the distribution, at a typical time of
the dynamical system, of the distance $R$ to the closest station is
well known to be Rayleigh, and the other stations are known to form a Poisson point process of the
same intensity outside the ball of radius $R$.
Here is a second teaser for the reader: under the Palm distribution of the
handover point process, is the distribution of the distance to the
stations involved in the handover still Rayleigh?
What is the law of the point process made of the other stations?

The aim of the machinery introduced in the present paper is to answer these questions and
many others of the same kind. 
These questions are of independent mathematical interest as the simplest and most natural questions about the dynamics of Poisson-Voronoi tessellations. They are also of central importance in the wireless communication setting
alluded to above. The frequency of handovers is important as any such
event means an interruption of the communication stream and a signaling overhead 
from the serving station to the next.
The distribution of the distance to the closest station at handover times
determines the signal power at these epochs.
That of the other stations determines the interference power.
Altogether, this allows one to determine the distribution of the signal to interference
and noise ratio (SINR) and signal to noise ratio (SNR), which in turn determine the distribution of the Shannon rate 
of the link to the user at these critical epochs. These quantities are some of the key performance metrics  in the context of satellite communication network~\cite{Baccelli-Choi}.

The model described in this article can be considered as a first step towards understanding a wide variety of dynamical communication models using unmanned aerial vehicles (UAVs)~\cite{Banagar-etal}, high altitude platform stations (HAPS)~\cite{Tanelli-etal},~\cite{Slim-etal} or constellation of low-Earth orbit (LEO) and medium-Earth orbit (MEO) satellites~\cite{Baccelli-Choi},~\cite{Choi-Baccelli-MA},~\cite{Choi-Baccelli-WC}. Communication networks based on Low Earth Orbit (LEO) or Medium Earth Orbit (MEO) satellite constellations
feature moving stations, more precisely moving on orbits. One can for instance consider the random locations of the satellites to be
a Cox point process on certain Earth orbits, at the same or different altitudes with respect to the surface of the Earth. 
The ever changing locations of the satellites and the rotation of the Earth lead to networks with
stations having a dynamics of the type described in the abstract when
devices on the Earth's surface connect to the closest visible satellite at any given time. 
The central object in this regard is the spherical dynamical tessellations 
formed by the projection of the satellite point process on the surface of the Earth.   
One of the important questions is to mathematically characterize the structure and the properties of the handover point process
under this dynamical setting. Spherical tessellations under various types of LEO constellation of satellites, were first considered in \cite{Okati-2020}. We refer the reader to \cite{Okati-thesis} and the references therein for a study of communication networks based on LEO and MEO satellite constellations, from a stochastic geometry perspective.

The objective of this work is to investigate this class of problems in a simplified setting. The first simplification is that where the spherical geometry is replaced by a planar one. The second one is that where the Cox motion alluded to above is simplified to a simple random way-point model.
These simplifications, which may look extreme, are in fact natural, or even necessary to define first steps that are
tractable in this line of thoughts. As we shall see, even the planar caricature of the problem is already challenging. 
The steps towards relaxing these simplifications are discussed in the paper in terms of future research in Section~\ref{sec-future_work}.

Consider the static scenario where all the stations are fixed and a mobile user is moving along a straight line at speed $v$, without changing its direction. In this case the user performs handover whenever it crosses the boundary of the static Voronoi tessellation. It has been shown in \cite{Baccelli_Zuyev2} using the linear contact length distribution of Voronoi tessellation for stationary point processes in \cite{L_Heinrich}, that the handover frequency is $\frac{4 v\sqrt{\la}}{\pi}$, which is expressed in terms of the intensity $\la$
of stations and the speed $v$ of the user. This result is also discussed in \cite{Baccelli_Madadi_Gustavo}, \cite{BKLZ}, \cite{Baccelli_Zuyev}, \cite{Miles}. The distribution of inter-handover time is also studied in~\cite{Salehi-Hossain}, in a multi-tier setting. Further statistics of the $2$ and higher dimensional Poisson-Voronoi tessellations have also been studied in~\cite{Gilbert},~\cite{Mecke1981},~\cite{Jasper_Moller},~\cite{Muche2010},~\cite{Miles}, where as generalizations to the case of stationary point processes have been studied in~\cite{L_Heinrich}. Mobility in terms of dynamic random geometric graphs studied in~\cite{Diaz-etal}.

In a computational geometry perspective, dynamic Voronoi tessellations were first considered in~\cite{Gowda83} in a finite setting for the study of geometric data structure using the underlying Delaunay triangulation of the Voronoi tessellation evolving over time and further explored in in~\cite{Roos-th},~\cite{Albers-Th},~\cite{Albers-Roos},~\cite{T-Roos93}. To the best of our knowledge, the current work about the dynamic Poisson-Voronoi tessellations with applications to wireless communications in a stochastic geometry setting, appears to be considered for the first time.


\section{Setting and main objects} \label{sec-Results}
Consider a grand probability space $(\Omega, \Fcal, \P)$ on which all our measures and random variables are defined. To start with a simple model on $\R^2$, consider a static user $u$ located at the origin $o=(0,0)$. At time $0$, consider stations located at all atoms of a homogeneous Poisson point process $\Phi_{\la}$ on $\R^2$ with intensity $\la$. For simplicity, we write $\Phi$ for $\Phi_\la$, as long as the intensity $\la$ is fixed. We denote the intensity measure by $\mu$ and it's density by ${\rm d} \mu=\la {\rm d} x$ and denote the probability distribution of $\Phi$ as $\P_{\Phi}$.
We refer to the locations of the mobile stations at time $0$ as their original locations ($\Phi$). At time $0$, the configuration of the atoms $\{X_i\}_{i\in \N}\subset \R^2$ of the point process $\Phi$ are such that $X_0:=\arg\min\{\vert X_i\vert \,\mbox{:}\, i\in \N\}$, where $|x|$ denotes the norm $\vert|x\vert|_2$.\label{notation:norm} This is the location of the nearest station to the user at time $0$. \label{notation:phi} \label{notation:PPsi} 

Let us now define how stations move. At time $0$, they are located at the atoms of $\Phi$. We suppose that each station is moving at a constant speed $v$, but in a random direction, chosen uniformly and independently in $(-\pi, \pi]$, with respect to the $x$-axis. The stations move without any obstacle, collision or change of direction. In the simplest model, the velocities of the mobile stations located at the atoms $\{X_i\}_{i\in \N}$ of $\Phi$, are defined as an i.i.d. collection of vectors $\{V_n\}_{n\in \N}$, where $V_n:=(v\cos \Theta_n, v\sin \Theta_n)$ and $\Theta_n\stackrel{D}{\sim}U[-\pi, \pi)$. One can think of the point process $\Phi$ as being equipped with marks $\{\Theta_i\}_{i\in \N_0}$ representing the i.i.d. directions of motion of the stations. Incorporating the direction of motion with the initial locations, we construct a marked Poisson point process $\tilde{\Phi}:=\sum_{i\in \N} \delta_{(X_i, \Theta_i)}$ with {\em intensity measure} $m_{\tilde{\Phi}}$ having density 
\[
m_{\tilde{\Phi}}({\rm d} x, {\rm d} \theta):=\la {\rm d} x \otimes \frac{1}{2\pi}{\rm d} \theta.\label{notation:m}
\]
As we will see later, our analysis remains the same in case the law for the direction of motion is arbitrary, for instance $\delta_\th$ for $\th\in [-\pi,\pi)$ fixed or $\sum_{i=1}^{n}p_i\delta_{\th_i}$, for some $\th_i\in [-\pi,\pi)$ fixed, $n\in \N$ and $p_i\in [0,1]$ such that $\sum_{i=1}^n p_i=1$. Nevertheless, we proceed with $U[-\pi,\pi)$ distributed random variables for the directions in several parts of this work.

At time $t$, the location of a mobile station, which was initially located at $X_i^0=X_i$, is $X_i^t:=X_i+V_i t$. We denote the set of locations of all the mobile stations at time $t$ by $\tilde{\Phi}^t$. For brevity, we shall write $\tilde{\Phi}^0=\Phi$. By the displacement theorem~\cite[Section 5.5]{kingman}, \cite[Theorem 2.2.17]{Baccelli-Bartek-Karray}, the point process $\Phi$ corresponding to the initial locations of the stations and the new locations, given by $\tilde{\Phi}^t$, have the same distribution for all times $t$. The displacement theorem also guarantees that there is no collision almost surely under the dynamics. For completeness, we shall detail the result in Appendix~\ref{sec-appendix} (Proposition~\ref{proposition:Phi-Phi-t}). We refer the reader to~\cite{Alex-Beutel}, for an interactive simulation of the joint mobility of the stations.

The collective motion of the points on the plane gives rise to a {\em dynamical Voronoi tessellation} on the plane. Due to the motion of the stations, the corresponding Voronoi cells change their shapes continuously. The user being fixed at the origin, the nearest station to the user gets changed to a new one whichever becomes the closest. In other words, the Voronoi cell that covers the user is replaced by a new cell corresponding to the station nearest to the user after some time. The old cell spends a random amount of time covering the user, before another one replaces it. We call the event of change of cell a {\em handover} and the change of cell events create a point process on the time axis, called as the {\em handover point process}. Indeed, as we will show, under this dynamics, we have a countable collection of handover epochs without accumulation on the time axis, for each realization of the marked Poisson point process $\tilde{\Phi}$.  The handover point process is denoted by $\Vcal_\la:=\sum_{t\in S}\delta_t$, where informally,
\begin{equation}
S:=\left\{t\in \R \,\mbox{:}\, \exists (X_i, \Theta_i), (X_j, \Theta_j) \text{ atoms of } \tilde{\Phi}, \text{ such that } |X_i^t|=|X_j^t| \text{ and } \tilde{\Phi}^t(B_{|X_i^t|}(o))=0\right\}.\nn
\end{equation}
We write $B_r(x)$ for the open ball of radius $r$, centered at $x\in \R^2$.\label{notation:ball}\! For simplicity we shall write $\Vcal$ for $\Vcal_\la$. We are interested in the statistical properties of the point process $\Vcal$, for example its stationarity and its intensity, denoted by $\la_\Vcal$. \label{notation:Vcal1} \label{notation:LVcal1}

Another important quantity we are interested in this context is the duration between two successive handovers, namely the {\em inter-handover time}. In terms of the dynamic Voronoi tessellation, the inter-handover time is nothing but the time spent by a station while serving the user. The structure of the ``history'' of the inter-handover times is important as well.  

It is equally important to consider the case when the stations are moving at different speeds and the directions are sampled uniformly at random as above. Suppose the set of speed is $v_{[n]}:=\{v_i\}_{i\in [n]}$ for some $n\in \N$. For each $i\in [n]$, a station of type $i$ moves with speed $v_i$. The initial locations of the stations given by the point process $\Phi$ can be thought of as a superposition $\sum_{i\in [n]} \Phi^{i}$ of the collection of Poisson point processes $\Phi^i$ with intensity $\la_i$, corresponding to the stations of type $i$. The total  intensity of $\Phi$ is $\la=\sum_{i\in [n]}\la_i$. The directions of motion are again chosen uniformly at random as above. Although, the directions having arbitrary distribution give us the same result. The displacement theorem (similar to the one in single-speed case in Proposition~\ref{proposition:Phi-Phi-t}) can also be shown to hold, even in the multi-speed setting. 

Like in the single-speed setting, we will suppose that the system obeys the  same nearest neighbor association rule. The corresponding dynamical Voronoi tessellation is of slightly different nature. The Voronoi cells corresponding to the faster stations pierce through the tessellation while respecting the space required for other cells.  In short, the handover point process consists of points representing handover between stations of the same type and handover between  stations of different types. The inter-handover time is again defined as the duration between two consecutive handovers, which is essentially the time spent by a station, of certain type, while serving a user.  

The multi-speed scenario serves as a natural generalization of the model with single-speed. Apart from this, the motivation for considering this type of multi-speed mobility model comes from the non-terrestrial networks. Consider a communication network consisting of LEO satellite constellation where satellites are moving in different altitude orbits and are serving the users on the Earth directly. The speeds of satellites are different at different orbital altitudes, like for instance in satellite constellation consisting of multi-altitude LEO satellites. In this type of communication networks, the user on Earth performs a ``handover'' whenever it receives a better signal than the previous one. As the power of signal is a function of distance, the best signal can be considered to be from the nearest visible satellite and in this case, the two satellite types have different velocities.    
%


\section{Questions, strategy and structure of the paper} \label{sec-Structure}
\input{Input-Structure}


\section{Handover frequency: single-speed case}\label{sec-Palm_handover}
In what follows, we devise a method allowing us to define the handover point process in a rigorous way.
For this, suppose that at time instant $0$, there is a mobile station at $X$,
an atom of $\Phi$ with polar coordinates $(R=\vert X\vert , \Psi)$, and with velocity vector $V:=(\cos\Theta, \sin\Theta)$. 
For the sake of simplicity, we assume that the stations are moving at unit speed.
The results for a general speed parameter $v$ hold up to a scaling that will be discussed below.
Let $\a\equiv(\Psi-\Theta)\pmod {2\pi}$ be the relative angle between the location vector $X$ and the velocity vector $V$. 
Our method relies on the use of a {\em distance function} (a function of the time variable) defined for each mobile stations, and 
which helps to simplify our analysis by keeping a minimum number of variables. 

For a station starting at $X$, the new location at time $t$ is $X^{t}=X+Vt$.
The distance at time $t$ from the mobile station to the user $u$ at the origin $o$ is hence
\begin{equation}
\vert X^t\vert =\left(\vert X\vert ^2+2t\vert X\vert \cos\a+ t^2\right)^{\half},
\label{eq:bird1}
\end{equation}
where $\a\equiv(\Psi-\Theta)\pmod {2\pi}$ is the relative angle. So the distance at time $t$
only depends on the initial distance $|X|$, the initial relative angle $\a$, and the time $t$. It will be convenient to use the point process $\tilde{\Phi}$ capturing the initial distances of mobile stations
along with their relative motion directions.\label{notation:phiT} It is easy to see that the marked point process
$\tilde{\Phi}:=\sum_{i\in \N}\delta_{(R_i, \a_i)}$, on the space $\R^+ \times [-\pi, \pi)$, is Poisson
with intensity measure $\mu$ of density 
\begin{equation}
{\rm d} \mu:= 2\pi r \la\,{\rm d} r \otimes \frac{1}{2\pi}{\rm d} \a.
\label{eq:mu_r}
\end{equation}
The variable $\alpha$ captures the randomness coming from the relative angle of motion of the station
with respect to the user, which is uniform on the interval $[-\pi, \pi)$.
\begin{remark}
One can also consider the direction of motion $\Theta$ having any law, for instance $\delta_\th$ for fixed $\th\in [-\pi,\pi)$, or $\sum_{i=1}^{n}p_i\delta_{\th_i}$, for fixed $\th_i\in [-\pi,\pi)$, some $n\in \N$ and $p_i\in [0,1]$ such that $\sum_{i=1}^n p_i=1$. In both cases, the relative angle $\a\equiv(\Psi-\Theta)\pmod {2\pi}$, can be proved to be a uniformly distributed random variable on $[-\pi,\pi)$.
\end{remark}
\textbf{Distance function: } For an atom $(R,\a)$ of $\tilde{\Phi}$, corresponding to a station starting at $X$, 
the distance from the user at time $t$ is 
\begin{equation}
\vert X^t\vert =\left(R ^2+2t R \cos \alpha+ t^2\right)^{\half} :=f((R,\a), t).
\label{eq:bird2a}
\end{equation}
We call $f((R,\a),\cdot)$ the {\em distance function}. As a function of the time variable $t$, it possesses a unique minimum.
Heuristically the collection of all minimum time-space locations gives
rise to a point process on the upper half plane, which we analyze in the following subsection.  

\subsection{Head point process}\label{subsection:HPP} Let $\mathbb{H}^+:=\R \times \R^+$, where the
first component represents the {\em time coordinate} and the second one the {\em height coordinate}. \label{notation:UH}
The unique minimum of the function  $f((R,\a),\cdot)$, defined in (\ref{eq:bird2a}), with respect to
the time variable, gives the distance of the mobile station to the user along it's trajectory.
Thus the nearest time and nearest location of the mobile station starting at $X$ are 
\begin{align}
(T, H)&:=\big(\arg\inf_{t\in \R}\{f((R,\a), t)\}, f(\arg\inf_{t\in \R}\{f((R,\a), t)\})\big)= \big(-R \cos\a,R \,\vert\sin\a\vert \big).
\label{eq:heada}
\end{align}
The locations $(T, H)$ are intrinsic and do not change with time, once the initial distances and the
relative angles are fixed.
\begin{definition}[Head point process]
The point process 
	\begin{equation}
		\Hcal:=\sum_{i\in \N}\delta_{(T_i, H_i)},
	\end{equation}
on $\mathbb{H}^+$, formed by the minimum time-space points $(T_i, H_i)$, as defined in (\ref{eq:heada}), associated with the atoms
$(R_i,\a_i)$ of $\tilde{\Phi}$, will be called the {\em head point process}.
\end{definition} \label{notation:Hcal}
\begin{lemma}[Mapping Lemma] 
The point process $\Hcal=\sum_{i\in \N}\delta_{(T_i, H_i)}$ is a homogeneous Poisson point process on
$\mathbb{H}^+$ with intensity measure $\nu$, where ${\rm d}\nu:={\rm d} t\otimes 2\la\, {\rm d} h$.
\label{lemma:tips_density}
\end{lemma}
\begin{proof}[Proof Lemma~\ref{lemma:tips_density}] Consider two Poisson point processes
$\tilde{\Phi}^1$  on $\R^+\times [0, \pi)$ and $\tilde{\Phi}^2$ on $\R^+\times [-\pi, 0)$, of
intensity measures $\mu_1, \mu_2$ respectively, where $d\mu_1:= \pi r \la\,{\rm d} r \otimes \frac{1}{\pi}{\rm d} \a$
and $d\mu_2:=\pi r \la\,{\rm d} r \otimes \frac{1}{\pi}{\rm d} \a$. The reason we need the intensity measures as defined above is that the location of mobile station at distance $r$ from the user
can be uniformly distributed on the semi-circle of radius $r$ and centered at $o$.
The point processes $\tilde{\Phi}^1, \tilde{\Phi}^2$ are two sub point processes of $\tilde{\Phi}$, and we
can write $\tilde{\Phi}:=\tilde{\Phi}^1+ \tilde{\Phi}^2$. Define two bijective maps
$g_1: \R^+\times[0, \pi) \to\mathbb{H}^+$, $g_2:\R^+\times[-\pi, 0) \to\mathbb{H}^+$ as follows:
\begin{equation}
g_1(r, \alpha):= (-r \cos\alpha, r\sin\alpha), \,\,\,g_2(r, \alpha):= (-r \cos\alpha, -r\sin\alpha).
\label{eq:g1g2}
\end{equation}
By the mapping theorem, the Poisson point processes $\tilde{\Phi}^1$ under $g_1$ and $\tilde{\Phi}^2$ under  $g_2$
map to two independent point processes $\tilde{\Phi}^1\circ g_1^{-1} $ and $ \tilde{\Phi}^2\circ g_2^{-1}$, with intensity measures
$\nu_1:=\mu_1\circ g_1^{-1}$ and $\nu_2:=\mu_2\circ g_2^{-1}$
on $\mathbb{H}^+$, respectively. Thus $\Hcal:=(\tilde{\Phi}^1\circ g_1^{-1})+ (\tilde{\Phi}^2\circ g_2^{-1})$
is also a Poisson point process with intensity measures $\nu:=\mu_1\circ g_1^{-1} \oplus \mu_2\circ g_2^{-1}$ 
on $\mathbb{H}^+$. We give below an exact expression for its intensity measure $\nu$ or it's density.

For any measurable function $f: \mathbb{H}^+\to \R$ we have, by the change of variable formula, 
\begin{align}
   \lefteqn{\int_{0}^{\infty}\int_{-\infty}^{\infty} f(t, h){\rm d} t\otimes2\la\, {\rm d} h}\nn\\
   &= 2 \la \int_{0}^{\infty}\int_{0}^{\pi} f(-r \cos \alpha, r \sin \alpha){\rm d} \a\, r\,{\rm d} r\nn\\
    &= \la \int_{0}^{\infty}\int_{0}^{\pi}f(-r \cos \alpha, r\sin \alpha){\rm d} \a\, r\,{\rm d} r +\la \int_{0}^{\infty}\int_{-\pi}^{0}f(-r \cos \alpha, -r\sin \alpha){\rm d} \a\, r\,{\rm d} r\nn\\
    &=  \int_{0}^{\infty}\int_{0}^{\pi} f(-r \cos \alpha, r\sin \alpha)\frac{1}{\pi}{\rm d} \a \otimes \pi r \la\,{\rm d} r +  \int_{0}^{\infty}\int_{-\pi}^{0} f(-r \cos \alpha, -r\sin \alpha)\frac{1}{\pi}{\rm d} \a \otimes \pi r \la\,{\rm d} r.
\label{eq:map0}
\end{align}
Using the maps $g_1, g_2$ defined in (\ref{eq:g1g2}) in (\ref{eq:map0}), we obtain that 
\begin{align}
\lefteqn{\int_{0}^{\infty}\int_{-\infty}^{\infty} f(t, h){\rm d} t\otimes2\la\, {\rm d} h}\nn\\
&=\int_{0}^{\infty}\int_{0}^{\pi}  f\circ g_1(r, \alpha)\,  \frac{1}{\pi}\, {\rm d} \a \otimes \pi r \la\,{\rm d} r +\int_{0}^{\infty}\int_{-\pi}^{0}  f\circ g_2(r, \alpha)\,  \frac{1}{\pi}\, {\rm d} \a \otimes \pi r \la\,{\rm d} r.
\label{eq:map1}
\end{align}
Above, for both maps $g_1(r, \alpha)= (-r \cos\alpha, r\sin\alpha)$ and $g_2(r, \alpha)= (-r \cos\alpha, -r\sin\alpha)$, the Jacobian is $r$. In particular for any measurable set $A\subset \mathbb{H}^+$,
\begin{align}
\E[\Hcal(A)]&=\E[\tilde{\Phi}^1\circ g_1^{-1}(A)]+ \E[\tilde{\Phi}^2\circ g_2^{-1}(A)]=\mu_1\circ g_1^{-1}(A) + \mu_2\circ g_2^{-1}(A)= \nu(A).\nn
\end{align}
The last equivalences establish the fact that the Poisson point processes $\tilde{\Phi}^1, \tilde{\Phi}^2$
on $\R^+\times[0, \pi)$ and $\R^+\times[-\pi, 0)$, with intensity measures $\mu_1, \mu_2$, respectively,
can be mapped to another Poisson point process on $\mathbb{H}^+$ of intensity measure $\nu$, with density
${\rm d}\nu:={\rm d} t\otimes 2\la\, {\rm d} h$, bijectively, through the maps $g_1, g_2$ together. On the other hand,
we get a $2$ in the density ${\rm d} t\otimes 2\la {\rm d} h$ of the intensity measure of the heads because
there exist two points, one on each of the spaces $\R^+\times[0, \pi)$ and $\R^+\times[-\pi, 0)$,
that have the same image in $\mathbb{H}^+$ under the maps $g_1, g_2$.
\end{proof}
\begin{remark}
We write the probability measure $\P_{\Hcal}$ as the law of the head point process $\Hcal$ and $\E_\Hcal$ denotes the expectation under $\P_\Hcal$.\label{notation:PalmPsiE} From now onward, our analysis will essentially rely on the head point process $\Hcal$,
instead the station point process $\tilde{\Phi}$, exploiting the power of the mapping lemma (Lemma~\ref{lemma:tips_density}).
\label{remark:Hcal-Phi}
\end{remark} \label{notation:PPsiE}
\begin{remark}~\label{remark:r1}(Speed $v\neq 1$)
Suppose the stations move at a speed $v\neq 1$. As before, we consider $\tilde{\Phi}$ as a
marked Poisson point process on $\R^+\times [-\pi, \pi]$. For any $(R,\a)\in \R^+\times [-\pi, \pi]$, 
corresponding to any station, it's nearest location to the user is $(T, H)= \left(-\frac{R}{v}\cos \a, R|\sin\a|\right)$.
Indeed, it is the time-space minimum of the distance function 
\begin{equation}
f_v((R,\a), t)= \left(R^2+2tv R \cos \alpha+ v^2t^2\right)^{\half}.
\label{eq:birda}
\end{equation}
In this case. the intensity of the head point process $\Hcal$ is given by
\[
{\rm d}\nu=v{\rm d}t\otimes 2\la\, {\rm d} h.
\]
Indeed, the maps required in Lemma~\ref{lemma:tips_density} are
\[
g_1(R, \alpha)= \left(-\frac{R}{v} \cos\alpha, R\sin\alpha\right), \,\,\,g_2(R, \alpha)= \left(-\frac{R}{v} \cos\alpha, -R\sin\alpha\right),\]
and, in (\ref{eq:map1}), the corresponding Jacobian for both maps $g_1, g_2$ turns out to be $R/v$.
This remark will be very useful in determining the handover frequency in the single-speed ($v\neq 1$) and multi-speed cases.
    \label{remark:onetov}
\end{remark}
\begin{observation}
In order to determine which station is the nearest at any given time,
it will be important to keep track of the time and distance coordinate data 
$\{\left(t, f((R,\a),t)\right):t\in \R\}\subsetneq \mathbb{H}^+$, for each station
corresponding to $(R, \a)$ in the point process $\tilde{\Phi}$.
\end{observation}
\begin{remark}[Distance function as a branch of a hyperbola on $\mathbb H^+$]
Consider a station initially at distance $R$ and moving at a speed $v$ with a relative angle $\a$. The equation of the distance function in (\ref{eq:birda}) can be rephrased as the equation of the hyperbola:
\begin{equation}
h^2- v^2\left(t+\frac{R}{v} \cos\alpha\right)^2
= R^2 \sin^2\alpha,
\label{eq:hyperbola-bird}
\end{equation}
in $\mathbb H^+$ in the $(t,h)$-coordinate, where $h$ denotes the distance of the station at time $t$. The center of the hyperbola is $\left(-\frac{R}{v} \cos\alpha, 0\right)$ and its vertex is $\left(-\frac{R}{v} \cos\alpha, R\vert\sin\alpha\vert\right)$.
\end{remark}
\begin{definition}[Radial bird closed set]
For each $(R,\a)$ atom of  $\tilde{\Phi}$, let \label{notation:BCset1}
\begin{equation}
C_{(R,\a)}:=\left\{(t, f((R,\a), t)): t\in \R\right\}.
\label{eq:bird2}
\end{equation} 
We call the infinite closed set $C_{(R,\a)}\subset\mathbb{H}^+$ a \textbf{radial bird closed set} or
a \textbf{radial bird particle} or simply a \textbf{radial bird}.   
\end{definition}
\begin{remark}
For each $(R,\a)$ atom of the point process $\tilde{\Phi}$, the closed set $C_{(R,\a)}$ defined by the
radial distance of the mobile station to the user at different times and the set looks like a bird, hence the name.
\end{remark}
In the following, we showcase some geometric properties of the distance function and that of the corresponding radial bird.
This is a preparation to the ``radial bird particle process'' in Subsection~\ref{subsection:BPPss}.
\subsection{Properties of the distance function} Consider an atom $(R,\a)$ of $\tilde{\Phi}$.
The corresponding distance function and radial bird, $f((R, \a), \cdot)$ and $C_{(R,\a)}$, respectively,
possess the following features:
\begin{enumerate}[1.]
\item \label{Obs2a} \textbf{Head:} the minimum time-space location $(T, H)$ is the {\em head} of the radial bird.
\vspace{0.04in}
\item \label{Obs2}\textbf{Distance to the head: }the distance function $f((R,\a), \cdot)$, which gives the distance of the station to the user at time $t$, satisfies the relation,
\begin{align}
f((R,\a),t)&= \left(R^2+2tR \cos \alpha+t^2\right)^{\half}= \left((t+R \cos \alpha)^2+ (R|\sin \alpha|)^2\right)^{\half}=\vert (t,0)-(T, H)\vert, 
\label{eq:two-point}
\end{align}
for any $t\in \R$. A key observation in (\ref{eq:two-point}) is that the distance at time $t$
to the mobile station starting at $X$ is also $\vert (t,0)-(T, H)\vert$, see Figure~\ref{figure:RB_properties}~(picture~\subref{sfig:dist2head}). This essentially means
that the radial bird closed sets are purely determined by their heads.
\begin{figure}[ht!]
\begin{subfigure}[t]{0.45\linewidth}
    \centering
\begin{tikzpicture}[scale=1,
every node/.style={scale=0.7}]
\pgftransformxscale{0.8}  
\pgftransformyscale{0.8}    
    \draw[->] (-2, 0) -- (3.2, 0) node[right] {$t$};
    \draw[-] (0, 0) -- (0, 1.58);
    \draw[-] (0, 0) -- (1.51, 0.5);
    \draw[blue, domain=-1.5:3, smooth] plot (\x, {(\x*\x-3*\x+2.5)^0.5});
    \draw[blue](1.51,0.5) node{$\bullet$};
    \draw[](2.2,0.35) node{$(T,H)$};
    \draw[red] (-1.58,0) arc (180:0:1.58);
    \draw[](0,-0.31) node{$(t,0)$};
    \draw[->] (0.9,1.9)--(0.2, 0.5);
    \draw[->] (1,1.9)--(0.7,0.3);
    \draw[](1.2,2.2) node{$|(t,0)-(T,H)|$};
    \end{tikzpicture}
    \caption{The distance of the station at time $t$ is the distance between $(t,0)$ and the head point at $(T,H)$.}
\label{sfig:dist2head}
    \end{subfigure}
    \qquad
      \begin{subfigure}[t]{0.45\linewidth}
    \centering
\begin{tikzpicture}[scale=0.9,
every node/.style={scale=0.7}]
\pgftransformxscale{0.2}  
\pgftransformyscale{0.28}    
    \draw[->] (-8, 0) -- (14, 0) node[right] {$t$};
    \draw[->] (0, -1) -- (0, 10) node[above] {$h$};
    \draw[->] (4, 0) -- (-6.3, 10);
    \draw[->] (4, 0) -- (14.3, 10);
    \draw[](4,-0.8) node{$(s,0)$};      \draw[blue, domain=-6.3:14.3, smooth] plot (\x, {(\x*\x-8*\x+18)^0.5});
       \draw[blue, domain=-6.3:14.3, smooth] plot (\x, {(\x*\x-8*\x+20)^0.5});
        \draw[blue, domain=-6.3:14.3, smooth] plot (\x, {(\x*\x-8*\x+25)^0.5});
        \draw[blue, domain=-6.3:14.3, smooth] plot (\x, {(\x*\x-8*\x+30)^0.5});
    \end{tikzpicture}
    \caption{A family of curves with their heads located at a vertical line $t=s$ and the straight lines $h=s+t$, $h=s-t$ are their asymptotes.}
    \label{sfig:asymptotes}
    \end{subfigure}
    \caption{Properties of the distance function.}
\label{figure:RB_properties}
\end{figure}
\item \label{Obs2b} \textbf{Family of curves: } Fix some $s\in \R$. For $u\in \R^+$, $g_{s, u}(t):=\vert (t,0)-(s, u)\vert$
is the height at time $t$ of the bird with head at $(s,u)$. Then $\{g_{s, u}(\cdot)\}_{u\in \R^+}$ defines the set of all
radial birds with head at $(s,u)$ when varying $u\in \R^+$. Each element in this family has the lines $h=s+t$ and $h=s-t$,
for $h\in \R^+$, as its asymptotes (see picture~\subref{sfig:asymptotes} of Figure~\ref{figure:RB_properties}).
\vspace{0.04in}
\item \label{Obs3} \textbf{Translation:}  Observe that, for fixed $u\in \R^+$ and $s_1< s_2\in \R$, the two subsets
of $\mathbb{H}^+$, $A_1:=\{(t, g_{s_1, u}(t))\,\mbox{:}\,t\in \R\}$ and $A_2:=\{(t, g_{s_2, u}(t))\,\mbox{:}\,t\in \R\}$ are just translated
versions of each other. For instance, $A_2=(s_2-s_1, 0)+ A_1$. This is the pivotal reason for the time-stationarity
of the spatially marked point process described later.
\vspace{0.04in}
\item \textbf{Unique intersection: }\label{Obs4}For $(s_1, u_1), (s_2,u_2)\in \mathbb{H}^+$ with $s_1\neq s_2$, 
the intersection of the two sets, $A_1:=\{(t, g_{s_1, u_1}(t))\,\mbox{:}\,t\in \R\}$ and $A_2:=\{(t, g_{s_2, u_2}(t))\,\mbox{:}\,t\in \R\}$ contains exactly one point.
\vspace{0.04in}
\item \textbf{Non-unit speed: }\label{Obs1} If $v\neq 1$, the head points are given by $(T, H)= \left(-\frac{R}{v}\cos \a, R|\sin\a|\right)$, for some $(R,\a)$ from $\tilde \Phi$. Then similarly to (\ref{eq:two-point}) the distance function $f_v((R,\a), \cdot)$ is given by 
\begin{align}
f_v((R,\a), t)&= \left(v^2\left(t+\frac{R}{v} \cos \alpha\right)^2+ (R |\sin \alpha|)^2\right)^{\half} = \left(v^2(t-T)^2+ H^2\right)^{\half},
\label{eq:two-point-v}
\end{align}
where we only scale the time coordinate. For fixed $v\neq 1$, let $g^{(v)}_{s, u}(t):=\left(v^2(t-s)^2+u^2\right)^{\half}$. Then each element of the family of curves $\left\{g^{(v)}_{s, u}(\cdot)\right\}_{u\in \R^+}$, has the lines $h=s+vt$ and $h=s-vt$,
for $h\in \R^+$, as its asymptotes.  All the properties, $(\ref{Obs2a}),\, (\ref{Obs2b}),\, (\ref{Obs3}), \,(\ref{Obs4})$ also hold true for any $v\neq 1$.
\end{enumerate}

\subsection{Radial bird particle process}
\label{subsection:BPPss}

\begin{definition}[Radial Bird Particle Process (RBPP)] \vspace{0.1in}

The \textbf{radial bird particle process} 
	is the point process 
\[
\Pcal_{c}:=\sum_{i\in \N}\delta_{C_i},
\]
on $\Ccal(\mathbb{H}^+)$, where $C_i$ is the particle corresponding to the atom $(R_i, \a_i)$ of $\tilde{\Phi}$.
\label{definition:BPPss}
\end{definition}\label{notation:RBPP}
We have the following result about $\Pcal_c$ as a Poisson set/particle process. See \cite[Definition 9.2.17]{Baccelli-Bartek-Karray}, for a reference on Poisson set processes. 
\begin{lemma} The following holds true for the radial bird particle process $\Pcal_c$:
\begin{enumerate}[(i)]
\item \label{RBPP1} $\Pcal_c$ is a Poisson particle process on the
space $(\Ccal(\mathbb{H}^+), \Bcal(\Ccal(\mathbb{H}^+)))$,
with support on the set of closed sets of the form (\ref{eq:bird2}).
\item \label{RBPP2} Both $\Pcal_c$ and $\cup_i C_i$ are stationary along the time coordinate.
\end{enumerate}
\label{lemma:bird_time_stationary}
\end{lemma}
\begin{figure}[ht!]
    \centering
\begin{tikzpicture}[scale=1.2,every node/.style={scale=0.7}]
\pgftransformxscale{0.7}  
    \pgftransformyscale{0.7}    
    \draw[<->] (-2, 0) -- (5, 0) node[right] {$t$};
    \draw[->] (-2.5, 1) -- (-2.5, 2) node[above] {$h$};
    \draw[domain=-2:5, smooth] plot (\x, {(\x*\x+2.5)^0.5});
    \draw[](-0.1,1.57) node{$\bullet$};
    \draw[domain=-2:5, smooth] plot (\x, {(\x*\x-3*\x+2.5)^0.5});
    \draw[](1.51,0.5) node{$\bullet$};
    %
    %
    %
    %
    \draw[domain=-2:5, smooth] plot (\x, {(\x*\x-2*\x+2)^0.5});
    \draw[](0.95,1) node{$\bullet$};
    \draw[domain=-2:5, smooth] plot (\x, {(\x*\x+2.1*\x+1.2)^0.5});
    \draw[](-1.05,0.31) node{$\bullet$};
    \draw[domain=-2:5, smooth] plot (\x, {(\x*\x-6*\x+10)^0.5});
    \draw[](3,1) node{$\bullet$};
    \draw[domain=-2:5, smooth] plot (\x, {(\x*\x-8*\x+16.5)^0.5});
    \draw[](4,0.7) node{$\bullet$};
%
%
\end{tikzpicture}
\caption{The birds together with their heads form the {\em radial bird particle process $\Pcal_c\equiv \Hcal_c$}.}
\label{figure:BPP-1}
\end{figure}
\begin{proof}[Proof of Lemma~\ref{lemma:bird_time_stationary}~(\ref{RBPP1})]
There is a bijection between the atom $(R_i, \a_i)$, the corresponding radial bird $C_i\equiv C_{(R_i,\a_i)}$ 
and its head $(T_i, H_i)$.
The proof then follows from the inverse construction of set valued marked Poisson point process from
the set process described in \cite[Proposition 10.2.10]{Baccelli-Bartek-Karray}, using the Poisson point process
$\Hcal=\sum_{i\in \N}\delta_{(T_i, H_i)}$ of the heads as the background process, see Figure~\ref{figure:BPP-1}.
\end{proof}
\begin{proof}[Proof of Lemma~\ref{lemma:bird_time_stationary}~(\ref{RBPP2})]
Let $t\in \R$ and $S_t:\Ccal(\mathbb{H}^+)\to \Ccal(\mathbb{H}^+)$ be the map defined as
$S_t(C):= C-(t, 0)$. The particle process $S_t(\Pcal_c)$ is the particle process $\Pcal_c$ shifted
by $(t, 0)$, namely
\[S_t(\Pcal_{c}):=\sum_{i\in \N}\delta_{S_t(C_i)}= \sum_{i\in \N}\delta_{C_i-(t,0)}.\]
Under the shift $S_t$, each of the radial bird particles is shifted towards the left along the time axis. Using Property~\ref{Obs3}, $S_t(\Pcal_{c})$ has the same distribution as $\Pcal_c$. Hence we have time-stationarity of $\Pcal_{c}$. Alternatively, the shape of the set valued mark depends only
on the height coordinate of the head, not on the time coordinate, see Properties~\ref{Obs2} and ~\ref{Obs3}.
So the time-stationarity of the radial bird particle process $\Pcal_c$ follows from the time stationarity of 
the head point process $\Hcal$.
\end{proof}
\begin{remark}
Rather than the particle process $\Pcal_{c}$, one can consider the marked point process
	\begin{equation}
		\Hcal_c := \sum_{i\in\N}\delta_{(T_i, H_i, C_i)},
		\label{eq:myhc}
	\end{equation}
on $\mathbb{H}^+$, with marks on the space of closed sets of $\mathbb{H}^+$.
It is clear that this Poisson point process is stationary along the time axis as well.
\label{remark:bird-to-mhead}
\end{remark} \label{notation:MHPP}
\begin{discussion}[\textbf{RBPP: stochastic process on $(\R^+)^{\N}$}] For each $t\in \R$, we
can construct a sequence of positive random variables that represent the distances of the stations at time $t$. 
For each atom $(R_i, \a_i)$ of $\tilde{\Phi}=\sum_{i\in \N}\delta_{(R_i, \a_i)}$, let $\xi^t_i:= f((R_i, \a_i),t)$.
We shall re-order the elements of the sequence $\{\xi^t_i\}_{i\in\N}$ by defining $\xi^t_{(i)}$ to be the
distance of the $i$-th closest station at time $t$. Formally, for each $t$, we construct the
sequence $\xi(t):=\left\{\xi^t_{(k)}\right\}_{k\in \N}$ as
\[
\xi^t_{(1)}:=\inf\{\xi^t_i\}_{i\in \N} \text{ and, for } k> 1, \xi^t_{(k)}:=\inf\left\{\{\xi^t_i\}_{i\in \N}\setminus \{\xi^t_{(i)}\}_{1\leq i<k}\right\},
\]
which defines the ordered sequence of distances of the stations. So the radial bird particle process $\Pcal_c$
allows one to build a stochastic process $\Xi:=\{\xi(t)\}_{t\in \R}$ on the space $(\R^+)^{\N}$.
For any $t_0\in \R$, the intersection of the vertical line $t=t_0$ with the union of the radial bird closed sets,
$\cup_{i\in \N}C_i$, is given by the set $\left\{(t_0, \xi^{t_0}_{(i)})\right\}_{i\in \N}$, as seen in Figure~\ref{figure:BPP},
where $C_i\equiv C_{R_i, \a_i}$ with $(R_i, \a_i)$ the points of $\tilde{\Phi}$. What is even more useful is the data
$\xi(t)=\left\{\xi^t_{(i)}\right\}_{i\in \N}$, at time $t$, which correspond to the distances to the stations. In the words of wireless communication, this enable us to compute the signal and the interference power experienced by the user at any given time $t$. In Subsection~\ref{subsection:typT}, we determine the distribution of $\xi(t)$ as a point process on $\R^+$, at different time epochs of interest, in particular at handover epochs and beyond.
\begin{figure}[ht!]
    \centering
\begin{tikzpicture}[scale=1,every node/.style={scale=0.7}]
\pgftransformxscale{0.9}  
\pgftransformyscale{0.9}    
    \draw[thick,<->] (-2.2, 0) -- (5.2, 0) node[right] {$t$};
    \draw[->] (-2.5, 1) -- (-2.5, 2) node[above] {$h$};
     \draw[blue, ->] (2.5, 0) -- (2.5, 5);
    \draw[domain=-2:5, smooth] plot (\x, {(\x*\x+2.5)^0.5});
    \draw[](-0.1,1.57) node{$\bullet$};
    \draw[domain=-2:5, smooth] plot (\x, {(\x*\x-3*\x+2.5)^0.5});
    \draw[](1.51,0.5) node{$\bullet$};
    \draw[->] (1.51, 0) -- (1.51, 5);
    \draw[](1.51,1.12) node{$\bullet$};
    \draw[](1.51,1.78) node{$\bullet$};
    \draw[](1.51,2.2) node{$\bullet$};
    \draw[](1.51,2.57) node{$\bullet$};
    \draw[](1.51,2.60) node{$\bullet$};
    \draw[blue](2.5,1.1) node{$\bullet$};
    \draw[blue](2.5,1.62) node{$\bullet$};
    \draw[blue](2.5,1.8) node{$\bullet$};
    \draw[blue](2.5,2.95) node{$\bullet$};
    \draw[blue](2.5,3.58) node{$\bullet$};
    \draw[blue, ->] (3.25, 0) -- (3.25, 5);
    \draw[blue](3.25,1.02) node{$\bullet$};
    \draw[blue](3.25,1.81) node{$\bullet$};
    \draw[blue](3.25,2.45) node{$\bullet$};
    \draw[blue](3.25,3.59) node{$\bullet$};
    \draw[blue](3.25,4.3) node{$\bullet$};
    \draw[blue, ->] (0.48, 0) -- (0.48, 5);
    \draw[blue](0.48,1.1) node{$\bullet$};
    \draw[blue](0.48,1.54) node{$\bullet$};
    \draw[blue](0.48,1.66) node{$\bullet$};
    \draw[blue](0.48,2.7) node{$\bullet$};
    \draw[blue](0.48,3.59) node{$\bullet$};
    \draw[blue, ->] (0.18, 0) -- (0.18, 5);
    \draw[blue](0.18,1.28) node{$\bullet$};
    \draw[blue](0.18,1.42) node{$\bullet$};
    \draw[blue](0.18,1.6) node{$\bullet$};
    \draw[blue](0.18,3) node{$\bullet$};
    \draw[blue](0.18,3.85) node{$\bullet$};
    \draw[domain=-2:5, smooth] plot (\x, {(\x*\x-2*\x+2)^0.5});
    \draw[](0.95,1) node{$\bullet$};
    \draw[domain=-2:5, smooth] plot (\x, {(\x*\x+2.1*\x+1.2)^0.5});
    \draw[](-1.05,0.31) node{$\bullet$};
     \draw[->] (-1.05, 0) -- (-1.05, 5.5);
    \draw[](-1.05,1.9) node{$\bullet$};
    \draw[](-1.05,2.27) node{$\bullet$};
    \draw[](-1.05,2.6) node{$\bullet$};
    \draw[](-1.05,4.17) node{$\bullet$};
    \draw[](-1.05,5.1) node{$\bullet$};
    \draw[domain=-2:5, smooth] plot (\x, {(\x*\x-6*\x+10)^0.5});
    \draw[](3,1) node{$\bullet$};
     \draw[->] (3, 0) -- (3, 5);
    \draw[](3,1.21) node{$\bullet$};
    \draw[](3,1.58) node{$\bullet$};
    \draw[](3,2.22) node{$\bullet$};
    \draw[](3,3.38) node{$\bullet$};
    \draw[](3,4.05) node{$\bullet$};
    \draw[domain=-2:5, smooth] plot (\x, {(\x*\x-8*\x+16.5)^0.5});
    \draw[](4,0.7) node{$\bullet$};
     \draw[->] (4, 0) -- (4, 5.5);
    \draw[](4,1.4) node{$\bullet$};
    \draw[](4,2.55) node{$\bullet$};
    \draw[](4,3.15) node{$\bullet$};
    \draw[](4,4.3) node{$\bullet$};
    \draw[](4,5.05) node{$\bullet$};
\draw[blue] (0.18, 0) node {$\times$};
\draw[blue] (0.48, 0) node {$\times$};\draw[blue] (2.5, 0) node {$\times$};
\draw[blue] (3.25, 0) node {$\times$};
     \draw[red, ->] (-0.55, 0) -- (-0.55, 5.5);
    \draw[red](-0.55,0.6) node{$\bullet$};
    \draw[red](-0.55,1.67) node{$\bullet$};
    \draw[red](-0.55,1.85) node{$\bullet$};
    \draw[red](-0.55,2.1) node{$\bullet$};
    \draw[red](-0.55,3.7) node{$\bullet$};
    \draw[red](-0.55,4.6) node{$\bullet$};
    \end{tikzpicture}
\captionsetup{width=0.9\linewidth}
    \caption{The blue $\blue{\times}$'s forms the handover point process on the time-axis. The blue nodes ($\blue{\bullet}$), black nodes ($\bullet$) and red nodes ($\red{\bullet}$), respectively on individual vertical blue, black and red lines, form the point processes corresponding to the distance of all the mobile stations at handover times, at a time corresponding to a visible head and at a typical time, respectively.}
    \label{figure:BPP}
\end{figure}
\label{discussion:RBPP-SINR}
\end{discussion}
Having defined the radial bird particle process, we are now interested in understanding it's properties,
in particular the {\em lower envelope process}, which plays a fundamental role for defining the handover epochs.
\begin{definition}[\textbf{Lower envelope process}] Associated with the marked Poisson point process of the stations
$\tilde{\Phi}=\sum_{i\in \N}\delta_{(R_i,\a_i)}$, define the function
$$t\mapsto L(t):=\inf_{i\in \N} f((R_i,\a_i), t).$$
The stochastic process $\{\Lcal(t)=(t,L(t))\}_{t\in \R}$ is called the lower envelope process, where the random variable $L(t)$ represents the distance of the closest station at time $t$. 
The \textbf{lower envelope} is the random closed set on $\mathbb{H}^+$ defined as 
\begin{equation}
	\Lcal_e:=\{(t, L(t)),\ t\in \mathbb R\} \subset \mathbb{H}^+.\nn
\end{equation}
\label{definition:Lower_env}
\end{definition} \label{notation:Le}
\begin{lemma}
The lower envelope process $\Lcal(\cdot)$ is stationary with respect to time.
\label{lemma:L_stationary}
\end{lemma}
\begin{proof}[Proof of Lemma~\ref{lemma:L_stationary}] The time stationarity of
$\Lcal(\cdot)$ follows from the time stationarity of the head point process $\Hcal$. The lower envelope process $\Lcal(\cdot)$ is a factor of the head point process $\Hcal$, from Lemma~\ref{lemma:bird_time_stationary}. 
\end{proof}
\begin{definition}
Let us denote the open region above the lower envelope, as $\Lcal^+_e$, i.e.,
    \begin{equation}
    \Lcal^+_e:=\{(t,h)\in \mathbb H^+: h>L(t)\}.    
    \end{equation}
\end{definition}
\subsection{Characterization of the lower envelope $\Lcal_e$}
In this subsection, we give a characterization of the lower envelope $\Lcal_e$, which will provide a criterion for a point
$(t,h)\in \mathbb{H}^+$ to be a handover point, which will be
instrumental in what follows. We first consider the unit speed case. For $h\in \R^+$, let $U_h:= B_h(o) \cap \mathbb{H}^+$, be the upper open half-ball of radius $h$ and centered at $(0,0)$. 
We shall write $U_{h}^{s}:= (s, 0)+U_h$, so that $U_{h}^{0}:=U_h$. \label{notation:UHB} 
\begin{observation}
Let $(s, h), (s', h')\in\mathbb{H}^+$. Let $C$ be the radial bird with its head at $(s, h)$.
Let $\hat{h}$ be the height at which the radial bird $C$ intersects the vertical line $t=s'$. 
Then $(s, h)\in U_{h'}^{s'}$ if and only if $\hat{h}<h'$, as depicted in Figure~\ref{fig:in-out}.
This follows from the alternative description of the distance function in (\ref{eq:two-point}).
\label{observation:in-out}
\end{observation}
\begin{figure}[ht!]
\centering
\begin{tikzpicture}[scale=1,every node/.style={scale=0.7}]
\pgftransformxscale{0.9}  
\pgftransformyscale{0.9}    
    \draw[->] (-2, 0) -- (3, 0) node[right] {$t$};
    \draw[<-] (0, 3) -- (0, 0) node[below] {$(s',0)$};
    \draw[<->] (0.1, 0) -- (0.1, 1.1);
    \draw[](0.25,0.5) node{$\hat{h}$};
    \draw[red, domain=-0.5:3.7, smooth] plot (\x, {(\x*\x-5*\x+2.5*2.5+0.2)^0.5});
    \draw[red](2.5,0.447) node{$\bullet$};
    \draw[](-0.6,1.8) node{$(s',h')$};
    \draw[](0,1.58) node{$\bullet$};
    \draw[](-0.6,1.8) node{$(s',h')$};
    \draw[](1,0.5) node{$U^{s'}_{h'}$};
    \draw[] (-1.58,0) arc (180:0:1.58);
   \draw[blue, domain=-2:1.5, smooth] plot (\x, {(\x*\x+2.1*\x+1.2)^0.5});
    \draw[blue](-1.05,0.31) node{$\bullet$};
    \draw[](-0.9,0.8) node{$(s,h)$};
    \end{tikzpicture}
    \captionsetup{width=0.9\linewidth}
    \caption{Scenario in Observation~\ref{observation:in-out}. All birds with head point inside and outside the half-ball of radius $h'$, intersect the vertical line below and above the level $h'$, respectively.}
\label{fig:in-out}
\end{figure}

\begin{lemma}
For any point $(t, h)\in \cup_{i\in \N}C_i$, the following holds:
\begin{enumerate}[(i).]
\item \label{emptyU} The point $(t, h)$ belongs to $\Lcal_e$ if and only if $\Hcal(U_h^t)= 0$, where $\Hcal$ is the head point process.
\item \label{emptyUP} The probability of such an event is
$$\P((t, h) \in \Lcal_e)=e^{-\la\pi h^2}.$$
\end{enumerate}
\label{lem:semicircle}
\end{lemma}
\begin{remark}
We often refer to the key geometric result, Lemma~\ref{lem:semicircle}, as the \textbf{half-ball condition}.
\end{remark}
\begin{proof}[Proof of Lemma~\ref{lem:semicircle}~(\ref{emptyU})]
Without loss of generality, we shall prove the result for $t=0$. Suppose $(0,h)\in \Lcal_e$. Then there exists
a radial bird $C'\equiv C_{(R',\a')}$, for some atom $(R',\a')$ of $\tilde{\Phi}$, such that $(0, h)\in C'$.
We have two cases, depending on whether $(0,h)$ is a head of the radial bird $C'$ or not.
Let $(\mathbb{H}^+)^*$ denote the open upper half plane $\R\times (0, \infty)$ and, for any open set $A\subset \R^2$,
let $\partial^*A:= \partial A\cap (\mathbb{H}^+)^*$, denote the boundary of an open set $A$ contained inside
$(\mathbb{H}^+)^*$. We refer to the horizontal axis as the $t-$axis and to the vertical axis as the $h-$axis.

\begin{case}[$(0,h)$ is a head point] \label{head} Suppose $(0,h)$ is the head of $C'$, as in Figure~(\ref{figure:h-n-h})~(picture~\subref{sfig:head}).
We shall first prove that any radial bird with its head on $\partial^* U_h$, passes through $(0, h)$.  
Suppose $(s,u)$ is an atom of $\Hcal$ such that $(s, u)\in \partial U_h$, then $s^2+u^2 = h^2$. On the other hand,
for the same $(s, u)$, there exists $(R, \a)$, an atom of $\tilde{\Phi}$ such that 
$(s, u)=(-R \cos \alpha, R  |\sin \alpha|)$ and $s^2+u^2=R^2$. The corresponding radial bird is given by
\begin{equation}
f((R,\a), t)= \left(R ^2+2tR \cos \alpha+t^2\right)^{1/2}.
\label{eq:two-points}
\end{equation}
The curve (\ref{eq:two-points}) passes through $(0, h)$, since $f((R,\a), 0)= R =(s^2+u^2)^\half=h$. 
The result follows from  Observation~\ref{observation:in-out}, that any radial bird with its head inside $U_h$ 
intersects the $h$-axis at a height lower than $h$. So, if there is any atom of $\Hcal\vert_{U_h}$, 
then the corresponding radial bird prevents $(0,h)$ from being on the lower envelope $\Lcal_e$.
\begin{figure}[ht!]
\begin{subfigure}[t]{0.47\linewidth}
    \centering
\begin{tikzpicture}[scale=1,every node/.style={scale=0.7}]
\pgftransformxscale{1}  
    \pgftransformyscale{1}    
    \draw[->] (-2, 0) -- (2.8, 0) node[right] {$t$};
    \draw[->] (0, 0) -- (0, 3) node[above] {$h$};
    \draw[red, domain=-2:2, smooth] plot (\x, {(\x*\x+2.5)^0.5})node[right]{$C'$};
    \draw[red](0,1.58) node{$\bullet$};
    \draw[red](0.5,2.1) node{$(0,h)$};
    \draw[blue, domain=-1.5:3, smooth] plot (\x, {(\x*\x-3*\x+2.5)^0.5});
    \draw[blue](1.51,0.5) node{$\bullet$};
    \draw[](2,0.3) node{$(s,u)$};
    \draw[red] (-1.58,0) arc (180:0:1.58);
   \draw[domain=-2:2, smooth] plot (\x, {(\x*\x+2.1*\x+1.2)^0.5});
    \draw[](-1.05,0.31) node{$\bullet$};
    \end{tikzpicture}
    \caption{The point $(0,h)$ is the head of a radial bird (in red). A radial bird (in blue) with its head on $\partial^* U_h$ passes through $(0,h)$ and any other radial bird (in black) with its head inside $U_h$ intersects the $h$-axis below the level $h$.}
\label{sfig:head}
    \end{subfigure}
    \qquad
      \begin{subfigure}[t]{0.47\linewidth}
    \centering
\begin{tikzpicture}[scale=0.9,every node/.style={scale=0.7}]
\pgftransformxscale{1}  
    \pgftransformyscale{1}    
    \draw[->] (-2, 0) -- (2.8, 0) node[right] {$t$};
    \draw[->] (0, 0) -- (0, 3) node[above] {$h$};
    \draw[blue, domain=-1.5:3, smooth] plot (\x, {(\x*\x-3*\x+2.5)^0.5});
    \draw[blue](1.51,0.5) node{$\bullet$};
    \draw[](2,0.3) node{$(s,u)$};
    \draw[red, domain=-1.5:3, smooth] plot (\x, {(\x*\x-2.24*\x+2*1.12*1.12)^0.5})node[right]{$C'$};
    \draw[red](1.12,1.12) node{$\bullet$};
    \draw[red](0.5,2.1) node{$(0,h)$};
    \draw[red] (-1.58,0) arc (180:0:1.58);
    \draw[domain=-2:2, smooth] plot (\x, {(\x*\x+2.1*\x+1.2)^0.5});
    \draw[](-1.05,0.31) node{$\bullet$};
    \end{tikzpicture}
    \caption{The point $(0,h)$ is on a radial bird (in red). A radial bird (in blue) on $\partial^* U_h$ passes through $(0,h)$ and any other radial bird (in black) with its head inside $U_h$ intersects the $h$-axis below level $h$.}
    \label{sfig:not-head}
    \end{subfigure}
    \caption{The two scenarios of $(0,h)$ being or not being a head point.}
    \label{figure:h-n-h}
\end{figure}
\end{case}
\begin{case}[$(0,h)$ is not a head point] \label{otherpoint} Suppose $(0,h)$ is any other point of $C'$ than its head, 
as in Figure~(\ref{figure:h-n-h})~(picture~\subref{sfig:not-head}). Let $(s, u)$ be a point on $\partial^* U_h$. 
Using the same argument as in case~(\ref{head}), one can prove that the radial bird with its head at $(s, u)$ 
passes through $(0,h)$. Due to Observation~\ref{observation:in-out}, any other radial bird with its heads inside $U_h$ 
intersects the $h$-axis below level $h$. So it must be the case that $ \Hcal(U_h)=0$; otherwise $(0,h)\notin \Lcal_e$. 
Hence we have the result.
\end{case}
\end{proof}
\begin{proof}[Proof of Lemma~\ref{lem:semicircle}~(\ref{emptyUP})]
Using part (\ref{emptyU}), the event a point $(t,h)$ lies on the lower envelope we have, 
\[
\{(t, h) \in \Lcal_e\}= \{\Hcal(U_h^s)=0\}.
\]
The probability of such an event equals $e^{-2\la |U_h^s| }=e^{-2\la \pi h^2/2}= e^{-\la \pi h^2}$. 
\end{proof}
\begin{remark}\label{remark:two-point}
Let $(0,h)$ be the intersection point of two radial birds with their heads at $(t_1,h_1)$ and $(t_2,h_2)$, respectively.  
Then the semi-circle $\partial^* U_h$ passes through both heads at $(t_1,h_1)$ and $(t_2,h_2)$
(see Figure~\ref{figure:h-n-h1}~(picture~\subref{sfig:intersection})). This implies that any other radial bird with its head inside $U_h$ 
intersects the height-axis below the level $h$, and this prevents $(0,h)$ from being a handover point. 
As a consequence, we have that $\Hcal(U_h)=0$.    
\end{remark}
\begin{figure}[ht!]
\begin{subfigure}[t]{0.45\linewidth}
    \centering
\begin{tikzpicture}[scale=1,every node/.style={scale=0.7}]
\pgftransformxscale{1}  
    \pgftransformyscale{1}    
        \draw[->] (-2, 0) -- (2.7, 0) node[right] {$t$};
    \draw[-] (0.2, 0) -- (0.2, 1.28);
    \draw[](0.1,1.9) node{$(0,h)$};
    \draw[domain=-1:2.5, smooth] plot (\x, {(\x*\x-2*\x+2)^0.5});
    \draw[](0.95,1) node{$\bullet$};
    \draw[red] (-1.1,0) arc (180:0:1.28);
    \draw[domain=-2:1, smooth] plot (\x, {(\x*\x+2.1*\x+1.2)^0.5});
    \draw[](-1.05,0.31) node{$\bullet$};
    \end{tikzpicture}
    \caption{The semicircle of radius $h$, passing through the intersection of two birds, passes through their heads, for any pair of birds. For $(0,h)$ to be a handover point, the regions $U_{h}$ must be empty of heads.}
\label{sfig:intersection}
    \end{subfigure}
    \quad
      \begin{subfigure}[t]{0.48\linewidth}
    \centering
\begin{tikzpicture}[scale=0.6, every node/.style={scale=0.7}]
\pgftransformxscale{1}  
\pgftransformyscale{1}  
\draw[->] (-2, 0) -- (4.5, 0) node[right] {$t$};
\draw[domain=-2:2, smooth] plot (\x, {((3/2)^2*\x*\x+2^2)^0.5});
\draw[blue](0,2) node{$\bullet$};
\draw[](-1,1.8) node{$(t_1,h_1)$};
\draw[domain=0:4, smooth] plot (\x, {((3/2)^2*\x*\x-2*(3/2)^2*2*\x+(3/2)^2*2^2+2.5^2)^0.5});
\draw[blue](2,2.5) node{$\bullet$};
\draw[](3.2,2.3) node{$(t_2,h_2)$};
\draw[](0.3,2.8) node{$(0,h)$};
\draw[red](0.7,0.9) node{$E^{0,v}_{h}$};
\draw[red]  (3.07,0) arc (0:180:1.82 and 2.75);
\end{tikzpicture}
\caption{The half-ellipse of semi-major axis of length $h$, passing through the intersection of two birds, passes through their heads, for any pair of birds. For $(0,h)$ to be a handover point, the regions $E^{v}_{h}$ must be empty of heads.}
\label{sfig:mixed_birds2}
\end{subfigure}
\caption{Half-ball condition and half-ellipse condition.}
\label{figure:h-n-h1}
\end{figure}
\begin{discussion}[\textbf{Speed $v\neq 1$}]
Consider the general setup where the stations are moving at speed $v\neq 1$. Then the region that needs to be empty 
of head points, for a point $(0, h)$ to be on the lower envelope process, is the region enclosed by the upper half 
ellipse with equation $v^2x^2+y^2=h^2$, see Figure~\ref{figure:h-n-h1}~(picture~\subref{sfig:mixed_birds2}). 
For any $s\in \R$, we denote by $E^{s, v}_h$ the open set enclosed by the upper half-ellipse 
given by the equation $v^2(x-s)^2+y^2=h^2$ and $E^v_h:=E^{0, v}_h$. \label{notation:El}\!
The probability that $E^{s,v}_h$ is empty of head points from $\Hcal$ is
\[
\P\left(\Hcal(E^{s,v}_h)=0\right)=e^{-2\la v|E^{s,v}_h|}
= e^{-2\la v \times \pi\frac{h^2}{2v}}=e^{-\la \pi h^2},
\]
using Remark~\ref{remark:onetov}. We refer to Lemma~\ref{lem:semicircle} as the \textbf{half-ellipse condition} 
in the case of $v\neq 1$.
\label{discussion:ellipse_v}
\end{discussion}
\begin{remark}
The void probability remains the same as in the unit speed case. In the fast movement case $(v>1)$, 
the increase in the intensity of heads is compensated by the region of the upper half ellipse having a smaller area. 
In the slow case $(v<1)$, the decrease in the intensity of heads is balanced by the area of a larger upper half ellipse. In this case, we denote the lower envelope as $\Lcal_e(v)$.  
\end{remark}
\begin{remark}
We say that an intersection point $(s,u)$ of two radial birds, does not correspond to a handover if $(s,u)\in \Lcal^+_e$. For speed $v\neq 1$, we denote the open region above the lower envelope as $\Lcal^+_{e}(v)$.
\label{remark:notHo}
\end{remark} \label{notation:Le+}
\subsection{Construction of the handover point process}\label{subsection:FC-HPP}
This subsection shows how the handover point process can be derived from the lower envelope process, more precisely using the marked head point process $\Hcal_c$ with set valued marks defined in (\ref{eq:myhc}). The construction of the handover point process $\Vcal$ from
$\Hcal_c$ is based on finding pairs of head points corresponding to a handover epoch, along with a non-negative mark attached to them. 
The former represents the time-coordinate at which the two radial birds intersect and the mark represents 
the height at which the corresponding birds intersect. A thinning is used to retain only those pairs of heads that give rise to handovers. 

By Remark~\ref{remark:bird-to-mhead}, $\Hcal_c$ is defined using the head point process 
$\Hcal$ along with the corresponding radial bird attached to each head as a mark. By Property~(\ref{Obs2}), each of the closed set valued marks is determined by the coordinate of the head point.  Also the intersection point of two birds 
is a function of the pair of heads of the corresponding birds. For this reason, we can detect the handovers just using the head point process. We adopt the convention to denote the radial bird with it's head
at $(t,h)$ by $C_{(t, h)}$, in view of Property~\ref{Obs2}. \label{notation:BCset2}
\begin{figure}[ht!]
    \centering
\begin{tikzpicture}[scale=0.85,every node/.style={scale=0.8}]
\pgftransformxscale{1.1}  
\pgftransformyscale{1.1}    
\draw[thick,<->] (-2.2, 0) -- (5.2, 0) node[right] {$t$};
    \draw[->] (-2.5, 1) -- (-2.5, 2) node[above] {$h$};
    \draw[domain=-2:5, smooth] plot (\x, {(\x*\x+2.5)^0.5});
    \draw[](-0.1,1.57) node{$\bullet$};
    \draw[domain=-2:5, smooth] plot (\x, {(\x*\x-3*\x+2.5)^0.5});
    \draw[](1.51,0.5) node{$\bullet$};
    \draw[blue, -] (2.5, 0) -- (2.5, 1.1);
    \draw[blue](2.5,1.1) node{$\bullet$};
    \draw[blue, -] (3.25, 0) -- (3.25, 1.02);
    \draw[blue](3.25,1.02) node{$\bullet$};
    \draw[blue, -] (0.48, 0) -- (0.48, 1.11);
    \draw[blue](0.48,1.11) node{$\bullet$};
    \draw[blue, -] (0.18, 0) -- (0.18, 1.28);
    \draw[blue](0.18,1.28) node{$\bullet$};
    \draw[domain=-2:5, smooth] plot (\x, {(\x*\x-2*\x+2)^0.5});
    \draw[](0.95,1) node{$\bullet$};
    \draw[domain=-2:5, smooth] plot (\x, {(\x*\x+2.1*\x+1.2)^0.5});
    \draw[](-1.05,0.31) node{$\bullet$};
    \draw[domain=-2:5, smooth] plot (\x, {(\x*\x-6*\x+10)^0.5});
    \draw[](3,1) node{$\bullet$};
    \draw[domain=-2:5, smooth] plot (\x, {(\x*\x-8*\x+16.5)^0.5});
    \draw[](4,0.7) node{$\bullet$};
\draw[dashed] (-0.25, 0) -- (-0.25, 1.6); \draw (-0.25, 0) node {$\times$};
\draw[dashed] (0, 0) -- (0, 1.6);
\draw (0, 0) node {$\times$};
\draw[dashed] (0.26, 0) -- (0.26, 1.35); 
\draw (0.26, 0) node {$\times$};
\draw[dashed] (0.63, 0) -- (0.63, 1.7);
\draw (0.63, 0) node {$\times$};
\draw[dashed] (1.08, 0) -- (1.08, 2.15);
\draw (1.08, 0) node {$\times$};
\draw[dashed] (1.25, 0) -- (1.25, 2.03);
\draw (1.25, 0) node {$\times$};
\draw[dashed] (1.52, 0) -- (1.52, 2.6);
\draw (1.52, 0) node {$\times$};
\draw[dashed] (1.77, 0) -- (1.77, 2.3);
\draw (1.77, 0) node {$\times$};
\draw[dashed] (2, 0) -- (2, 1.44);
\draw (2, 0) node {$\times$};
\draw[dashed] (2.41, 0) -- (2.41, 1.7);
\draw (2.41, 0) node {$\times$};
\draw[dashed] (2.8, 0) -- (2.8, 1.4);
\draw (2.8, 0) node {$\times$};
\draw[blue] (0.18, 0) node {$\bigtimes$};
\draw[blue] (0.48, 0) node {$\bigtimes$};
\draw[blue] (2.5, 0) node {$\bigtimes$};
\draw[blue] (3.25, 0) node {$\bigtimes$};
\end{tikzpicture}
\caption{Construction of the handover point process $\Vcal$ and its marked version $\hat \Vcal$.}
\label{figure:RBPP}
\end{figure}

Let $(t, h)\neq (t', h')\in \mathbb{H}^+$ be such that $t'<t$ and the corresponding radial birds $C_{(t, h)}, C_{(t', h')}$ 
intersect at  $(\hat{s}, \hat{h})$. The quantities  $\hat{s}$ and $\hat{h}$ 
are derived later in (\ref{eq:handover-time}) and (\ref{eq:h}), respectively, in terms of $(t, h)$ and $(t', h')$. Let $A(\hat s,\hat h)$ be the event that the intersection point $(\hat s, \hat h)$ represents a handover. By Lemma~\ref{lem:semicircle}, the event $A(\hat s,\hat h)$ is equivalent to having no points of $\Hcal$ in the upper half-ellipse $U_{\hat h}^{\hat s}$, i.e., $\{\Hcal(U_{\hat h}^{\hat s})=0\}$. In this sense the point process given by 
\begin{equation}
\hat \Vcal:=\sum_{(T_i,H_i)\in \Hcal}\;\;\sum_{(T_j,H_j)\in \Hcal\,\mbox{:}\, T_j <  T_i} \delta_{(\hat S, \hat H)}\, \one_{A(\hat S, \hat H)},
\label{eq:VcalHPP1}
\end{equation}
represents the collection of intersection points responsible for handovers, where $(\hat S, \hat H)$ is the intersection point of the radial birds at $(T_i,H_i)$ and $(T_j,H_j)$. We have considered the pair of head points $(T_i,H_i)$ and $(T_j,H_j)$ such that $T_j<T_i$, just to avoid any double counting. In Figure~\ref{figure:RBPP}, the points $\{\blue{\times}, \times\}$ are the projections of intersection of birds on the time axis, representing the corresponding pair of head points. We shall discard the pair of heads corresponding to the black ones ($\times$) and retain only the pair for blue ones $(\blue{\times})$ to form the handover point process. The blue nodes ($\blue{\bullet}$) correspond to the handover epochs and the distances at which the handovers happen. Finally, 
the handover point process is given by \label{notation:Vcal2}
\begin{equation} \Vcal:=\sum_{(T_i,H_i)\in \Hcal}\;\;\sum_{(T_j,H_j)\in \Hcal\,\mbox{:}\, T_j <  T_i} \delta_{\hat S}\, \one_{A(\hat S, \hat H)}.\label{eq:VcalHPP2}
\end{equation}
We call the point process $\hat \Vcal$ the marked version of the handover point process $\Vcal$, where the non-negative mark, associated to each handover point, represent the height at which the corresponding radial birds intersect. In other words, the mark denotes the handover distance.
The mark $\hat{H}$ is also used in for 
the retention of the pair of head points corresponding to the intersection at $(\hat{S}, \hat{H})$.
\begin{remark}
Alternatively, in view of (\ref{eq:VcalHPP2}) one can define the handover point process $\Vcal$ as
\begin{equation} \Vcal:=\sum_{(T_i,H_i)\in \Hcal}\;\;\sum_{(T_j,H_j)\in \Hcal\,\mbox{:}\, T_j >  T_i} \delta_{\hat S}\, \one_{A(\hat S, \hat H)},\label{eq:VcalHPPL}
\end{equation}
by counting the pairs of head points $(T_i,H_i), (T_j,H_j)\in \Hcal $ such that $T_j>T_i$.
\label{remark:RVcal}
\end{remark}
In the following result, the time-stationarity of the handover point process $\Vcal$ is inherited from the same property of the head point process $\Hcal$.
\begin{lemma}[Time-stationarity of $\Vcal$]
The handover point process $\Vcal$ is stationary with respect to time.
\label{lemma:stationary_Tcal}
\end{lemma}
\subsection{Handover frequency for the single-speed}
\begin{theorem}[Handover frequency for the single-speed]
The handover frequency or intensity of the
handover point process $\Vcal$ is $\la_\Vcal=\frac{4v\sqrt{\la}}{\pi}$, where $\la$ is 
the intensity of mobile stations and v is their speed.  
\label{theorem:handover_freq1}
\end{theorem} \label{notation:LVcal2}
The proof relies on the representation of the 
two-point Palm probability distribution of $\Hcal$, which we define in the following before going into the proof of the last theorem.
\subsubsection{Two point Palm probability distribution of $\Hcal$}
It is clear from (\ref{eq:VcalHPP1}) that each handover event is characterized as the intersection of
two radial bird particles that correspond to two atoms of the head point process $\Hcal$ satisfying
certain properties. This requires us to define the two-point Palm probability measure for $\Hcal$. Recall that, from Lemma~\ref{lemma:tips_density}, the head point process $\Hcal$ is Poisson with intensity measure $\nu$ having for density ${\rm d}\nu= {\rm d} t\otimes 2\la {\rm d} h$. 

Let $M(\mathbb H^+)$ be the space of all locally finite measures on $\mathbb H^+$. We define $\Mcal(\mathbb H^+)$ as the $\sigma$-algebra on $M(\mathbb H^+)$, generated by maps of the form $\mu\mapsto \mu(A)$, for any $A\in \Bcal(\mathbb H^+)$ and $\mu\in M(\mathbb H^+)$, where $\Bcal(\mathbb H^+)$ is the Borel $\sigma$-algebra on $\mathbb H^+$, see \cite{Baccelli-Bartek-Karray} for more details.
The following result is classical for Poisson point processes:
\begin{lemma}[Two-point Palm probability distribution of $\Hcal$]
The two point Palm probability distribution of the head point process is 
\begin{equation}
\P_{\Hcal}^{(t, h),(t', h')}(B):=\P\left(\Hcal^{(t, h), (t', h')}\in B\right),\nn
\end{equation} 
for any $B\in \Mcal(\mathbb{H}^+)$, with $\Hcal^{(t, h), (t', h')}:=\Hcal+ \delta_{(t, h)}+\delta_{(t', h')}$
and $\Hcal$ the stationary version of the head point process. 
\label{lemma:twop-head}
\end{lemma}
\label{notation:2point} \label{notation:2pointP}
This result extends naturally to the marked version of the head point process $\Hcal_c$ defined in (\ref{eq:myhc}).
\begin{lemma}[Two-point Palm probability distribution of $\Hcal_c$]
The two point Palm probability distribution of $\Hcal_c$ is
\begin{equation}
\P_{\Hcal_c}^{(t, h,C),(t', h',C')}(\, \cdot\, ):=\P\left(\Hcal_c^{(t, h,C), (t', h',C')}\in \cdot\, \right),\nn
\end{equation}
with $\Hcal_c^{(t, h,C), (t', h',C')}:=\Hcal_c+ \delta_{(t, h,C)}+\delta_{(t', h',C')}$
and $\Hcal_c$ the stationary version of the head point process.
Here $C$ and $C'$ are the radial birds with head point at $(t,h)$ and $(t',h')$, respectively.
\label{definition:twop-headc}
\end{lemma}
Suppose there are two radial bird particles $C,C'$ with their heads at $(t, h), (t', h')$, respectively.
Let $(\hat{s}, \hat{h})$ be the intersection point of $C, C'$, 
where $\hat{s}$ is the time when the two radial birds intersect and $\hat{h}$ is the distance at which they intersect. 
Let $U_{\hat{h}}^{\hat{s}}$ be the upper half-ball centered at $(\hat{s}, 0)$ and of radius $\hat{h}$. 
It follows from Lemma~\ref{lem:semicircle} and Remark~\ref{remark:two-point} that 
$\partial^* U_{\hat{h}}^{\hat{s}}$ also passes through both heads $(t, h)$ and  $(t', h')$, 
see Figure~\ref{figure:two-bird}. It also follows from Lemma~\ref{lem:semicircle} that 
a handover happens at $(\hat{s}, \hat{h})$ if and only if $\Hcal(U_{\hat{h}}^{\hat{s}})=0$.
\begin{figure}[ht!]
\centering
\begin{tikzpicture}[scale=1.5, every node/.style={scale=0.7}]
\pgftransformxscale{0.65}  
\pgftransformyscale{0.65}  
\draw[->] (-2, 0) -- (2.5, 0) node[right] {$t$};
\draw[-] (0.2, 0) -- (0.2, 1.28);
\draw[](0.2,-0.2) node{$(\hat{s}, 0)$};
\draw[](0.2,1.9) node{$(\hat{s},\hat{h})$};
\draw[domain=-1.7:2, smooth] plot (\x, {(\x*\x-2*\x+2)^0.5});
\draw[](0.95,1) node{$\bullet$};
\draw[](1.1,1.3) node{$(t, h)$};
\draw[](2.1,1.6) node{$C$};
\draw[red] (-1.1,0) arc (180:0:1.28);
\draw[domain=-2:1.7, smooth] plot (\x, {(\x*\x+2.1*\x+1.2)^0.5});
\draw[](-1.05,0.31) node{$\bullet$};
\draw[](-0.4,0.3) node{$(t', h')$};
\draw[](-2.1,1.2) node{$C'$};
\end{tikzpicture}
\captionsetup{width=0.91\linewidth}
\caption{Given $(t, h)$ and $(t', h')$, the intersection $(\hat{s},\hat{h})$ is a handover point if and only if  $\Hcal(U_{\hat{h}}^{\hat{s}})=0$.}
\label{figure:two-bird}
\end{figure}

We have the following result about the independence of the head point processes restricted to
$U_{\hat{h}}^{\hat{s}}$ and $ (U_{\hat{h}}^{\hat{s}})^c$, under the two point Palm probability measure
$\P_{\Hcal_c}^{(t, h,C),(t', h',C')}$. We state it as a corollary without proof, as it directly
follows from the fact that the restriction of a Poisson point process on disjoint sets are independent, together with Lemma~\ref{lem:semicircle}.
\begin{corollary}
    Under the two-point Palm probability measure $\P_{\Hcal}^{(t, h),(t', h')}$, 
    \begin{enumerate}[(i)]
    \item The head point processes $\Hcal\vert_{U^{\hat{s}}_{\hat{h}}}$ and $\Hcal\vert_{(U^{\hat{s}}_{\hat{h}})^c}$ are independent. \label{notation:PRest}
    \item Additionally, under the event $\{\Hcal(U^{\hat{s}}_{\hat{h}})=0\}$, the point process 
	    $\Hcal\vert_{(U^{\hat{s}}_{\hat{h}})^c}$ is Poisson with intensity of density ${\rm d} t\otimes 2\la {\rm d} h$.
    \end{enumerate}
    \label{corollary:CCc}
\end{corollary}
\begin{proof}[Proof of Theorem~\ref{theorem:handover_freq1} (Handover frequency)]
The proof mainly relies on the characterization 
for a point to be on the lower envelope $\Lcal_e$ in terms of void probability of half-balls (Lemma~\ref{lem:semicircle}).
Let $X_i, X_j$ be the initial locations of two stations $(R_i, \a_i)$, $(R_j, \a_j)$ corresponding
to $X_i,X_j$. Let $f((R_i, \a_i), \cdot)$ and $f((R_j, \a_j), \cdot)$ be the curves of the two radial birds
$C_i, C_j$, with their heads at $(T_i, H_i)$ and $(T_j, H_j)$, respectively. For $(T_i, H_i)$ and $(T_j, H_j)$ with $T_j<T_i$, let $(S_{i,j},H_{i,j})$ be the intersection point of the associated radial birds.
We will see later that $(S_{i,j},H_{i,j})$ can be expressed in terms of $(T_i, H_i)$ and $(T_j, H_j)$.
For any $(s,h)\in \mathbb{H}^+$, define
\begin{equation}
A(s,h) :=\{ (s,h) \text{ is a handover point}\}.
\label{eq:HOe}
\end{equation}
In short, we define $A_{i,j}:=A(S_{i,j},H_{i,j})$ and we can write for the event that, $\one_{A_{i,j}}\equiv \one_{A_{i,j}}\left( \Hcal\right)$. Then the handover frequency can be written as 
\begin{align}
\la_\Vcal=m_{\Vcal}[0,1]
%
&=\E\left[\sum_{(T_i,H_i)\in \Hcal\,\mbox{:}\,T_i\in [0,1]}\;\;\sum_{(T_j,H_j)\in \Hcal\,\mbox{:}\, T_j <  T_i} \one_{A_{i,j}}\right],
\label{eq:h1a}
\end{align}
based on the construction of the handover point process in Subsection~\ref{subsection:FC-HPP}, in particular Equation (\ref{eq:VcalHPP1}). The last expression clearly suggests using now the factorial power of order $2$ of $\Hcal$.
The point process $\Hcal$ being Poisson, any factorial power can be expressed in terms of a product measure. 
More precisely, by applying the multivariate Campbell-Mecke formula for the factorial power of order 2 of $\Hcal$, i.e., $\Hcal^{2,\neq}$ (Proposition 2.3.24 and Proposition 2.3.25, in \cite{Baccelli-Bartek-Karray}), we get that \label{notation:FPn}
\begin{align}
\la_{\Vcal}&=\E\left[\sum_{(T_i,H_i)\in \Hcal\,\mbox{:}\,T_i\in [0,1]}\;\;\sum_{(T_j,H_j)\in \Hcal\,\mbox{:}\, T_j <  T_i} \one_{A_{i,j}}\right]\nn\\
&=(2\la)^2\int_{0}^{1} \int_{0}^{\infty}\int_{-\infty}^{t_1}\int_{0}^{\infty} \E^{(t_1, h_1), (t_2, h_2)}_{\Hcal}\left[\one_{A(s,h)} \left(\Hcal\right)\right] \, {\rm d} h_2\, {\rm d} t_2\, {\rm d} h_1\, {\rm d} t_1\nn\\
&=4\la^2\int_{0}^{1} \int_{0}^{\infty}\int_{-\infty}^{t_1}\int_{0}^{\infty} \E_\Hcal\left[\one_{A(s,h)} \left(\Hcal+\delta_{(t_1,h_1)}+\delta_{(t_2,h_2)}\right)\right] \, {\rm d} h_2\, {\rm d} t_2\, {\rm d} h_1\, {\rm d} t_1,
\label{eq:h1b}
\end{align}
where $(s,h)$ is the point that corresponds to the
intersection of the radial birds associated with the heads located at $(t_1, h_1)$ and $(t_2, h_2)$. 
Here $\Hcal^{(t_1, h_1),(t_2, h_2)}=\Hcal+ \delta_{(t_1, h_1)}+\delta_{(t_2, h_2)}$, is the two-point Palm version of
$\Hcal$ and it's two-point Palm expectation $\E^{(t_1, h_1), (t_2, h_2)}_{\Hcal}$ is
as described in Lemma~\ref{lemma:twop-head}. \label{notation:2pointE}

The handover point $(s,h)$ can be expressed in terms of $(t_1,h_1), (t_2,h_2)$, as follows. 
There exists a pair of points $(r_1, \a_1), (r_2, \a_2)\in \R^+\times (-\pi, \pi]$ such that
$(t_1, h_1)=(-r_1\cos \a_1, r_1|\sin\a_1|)$ and $(t_2, h_2)=(-r_2\cos \a_2, r_2|\sin\a_2|)$. 
Let $f((r_1, \a_1), \cdot)$ and $f((r_2, \a_2), \cdot)$ be the curves of the radial birds $C_1,C_2$, 
as defined in (\ref{eq:bird2a}). The time coordinate $s$ of handover satisfies $f((r_1, \a_1), s)=f((r_2, \a_2), s)$, i.e., 
\[s^2+2sr_1\cos\a_1+r_1^2= s^2+2sr_2\cos\a_2+r_2^2.\]
This implies,
\begin{align}
 s&= \frac{r_1^2-r_2^2}{2(r_2\cos\a_2-r_1\cos\a_1)}=\frac{t_1^2+h_1^2- t_2^2-h_2^2}{2(t_1-t_2)}.
\label{eq:handover-time}
\end{align}
The time coordinate $s$ is unique, which comes from the fact that radial birds $C_1,C_2$ with different $t_1, t_2$ 
must intersect exactly once, see Property~\ref{Obs4}. Substituting the value of $s$ from (\ref{eq:handover-time}), 
we get the height $h$ of the intersection point from:
\begin{align}
    h^2&= s^2-2s t_1+ r_1^2\nn\\
    &=\frac{1}{4(t_1-t_2)^2}\left[(t_1^2+h_1^2- t_2^2-h_2^2)^2- 4t_1(t_1-t_2)(t_1^2+h_1^2- t_2^2-h_2^2)+4(t_1^2+h_1^2)(t_1-t_2)^2\right]\nn\\
    &= \frac{1}{4(t_1-t_2)^2} \left[(t_2-t_1)^4+2(h_1^2+h_2^2)(t_1-t_2)^2+(h_2^2-h_1^2)^2\right]\nn\\
    &= \frac{1}{4}\left[(t_1-t_2)^2+2(h_1^2+h_2^2)+\frac{(h_2^2-h_1^2)^2}{(t_1-t_2)^2}\right]. 
    \label{eq:h}
\end{align}
Recall that $A(s,h)$ is the event that $C_1$ and $C_2$ intersect at $(s,h)$, which is a handover point.
A handover happens at $(s,h)$ if and only if $\Hcal(U^s_h)=0$,
by Lemma~\ref{lem:semicircle}. Since $\Hcal$ is a Poisson point process, by the Slivnyak-Mecke theorem, the probability of the event $A(s,h)$ in (\ref{eq:h1b}), is 
\begin{align}
\E_{\Hcal}\left[\one_{A(s,h)} \left(\Hcal+\delta_{(t_1, h_1)}+\delta_{(t_2, h_2)}\right)\right]&=\P_{\Hcal}\left(A(s,h)\right)=\P_{\Hcal}\left(\Hcal(U^s_h)=0\right)= e^{-\la\pi h^2},
\label{eq:P_A}
\end{align}
which does not depend on the time coordinate $s$. Using the probability from (\ref{eq:P_A}) in (\ref{eq:h1b}), we obtain that
\begin{align}
    \la_\Vcal&= 4\la^2\int_{0}^{1} \int_{0}^{\infty}\int_{0}^{\infty}\int_{-\infty}^{t_1} e^{-\la \pi h^2} {\rm d} t_2\, {\rm d} h_2\, {\rm d} h_1\, {\rm d} t_1,
\label{eq:h1}
\end{align}
as $t_2$ lies on the left of $t_1$. Since the height $h$ is a function of the difference $t_1-t_2$ as seen in (\ref{eq:h}),
the integral in (\ref{eq:h1}) can be converted into an integral in three variables after changing
the variable $t_1-t_2$ to $t$, as follows 
\begin{align}
\la_\Vcal = 4\la^2\int_{0}^{\infty}\int_{0}^{\infty}\int_{0}^{\infty} e^{-\la \pi h^2} {\rm d} t\, {\rm d} h_1\, {\rm d} h_2.
\label{eq:h2}
\end{align}
Using the expression for $h^2$ from (\ref{eq:h}) and keeping aside the term $\half(h_1^2+h_2^2)$,
we evaluate the inner-most integral in (\ref{eq:h2}) only with respect to $t$ as
\begin{align}
\int_{0}^{\infty} e^{-\frac{\la \pi}{4} \left[t^2+\frac{(h_2^2-h_1^2)^2}{t^2}\right]}  {\rm d} t 
&=e^{-\frac{\la\pi}{2}|h_2^2-h_1^2|} \frac{\sqrt{\pi}}{2\sqrt{\frac{\la\pi}{4}}}= e^{-\frac{\la\pi}{2}|h_2^2-h_1^2|} \frac{1}{\sqrt{\la}},
\label{eq:h2a}
\end{align}
where we used the following relation for $a =\frac{\la\pi}{4}>0$ and $b=\frac{\la\pi}{4} (h_2^2-h_1^2)^2>0$ as:
\begin{align}
\int_{0}^{\infty} e^{- a x^2-\frac{b}{x^2}}\, {\rm d} x &= e^{-2\sqrt{ab}} \int_{0}^{\infty} e^{- \left(\sqrt{a} x-\frac{\sqrt{b}}{x}\right)^2}\, {\rm d} x
= e^{-2\sqrt{ab}} \frac{1}{\sqrt{a}} \int_{0}^{\infty} e^{-y^2} {\rm d}y 
=  \frac{\sqrt{\pi}}{2\sqrt{a}}e^{-2\sqrt{ab}},
\label{eq:gaussian_correction}
\end{align}
%
%
computed using Cauchy-Schl\"{o}milch transformation from~\cite{Amdeberhan-etal} 
\[
\int_0^\infty f\left(\left(px-q/x\right)^2\right){\rm d}x= \frac{1}{p} \int_0^\infty f(y^2){\rm d}y,
\]
for any function $f$ and constants $p,q>0$. Substituting the resultant from (\ref{eq:h2a}) to the integral (\ref{eq:h2}), along with an extra factor of $e^{-\frac{\la\pi}{2}(h_1^2+h_2^2)}$,
the right hand side in (\ref{eq:h2}) turns out to be
\begin{align}
    \la_\Vcal&=\frac{4\la^2}{\sqrt{\la}}\int_{0}^{\infty}\int_{0}^{\infty} e^{-\frac{\la\pi}{2}(h_1^2+h_2^2)} e^{-\frac{\la\pi}{2}|h_2^2-h_1^2|} {\rm d} h_1\, {\rm d} h_2\nn\\
    &=4 \la^{\frac{3}{2}} \int_{0}^{\infty}\int_{0}^{\infty}e^{-\frac{\la\pi}{2}\left(h_1^2+h_2^2+|h_2^2-h_1^2|\right)}{\rm d} h_1\, {\rm d} h_2 \nn\\
    &=4 \la^{\frac{3}{2}} \int_{0}^{\infty}\int_{0}^{\infty} e^{-\la\pi\left(h_1^2\vee h_2^2\right)}{\rm d} h_1\, {\rm d} h_2 = 4 \la^{\frac{3}{2}}\frac{1}{\la\pi} \int_{0}^{\infty}\int_{0}^{\infty} e^{-(x^2\vee y^2)}{\rm d} x\, {\rm d}y =\frac{4\sqrt{\la}}{\pi},
\label{eq:h3}
\end{align}
since it can be shown that $\int_{0}^{\infty}\int_{0}^{\infty} e^{-(x^2\vee y^2)}{\rm d} x\, {\rm d}y=1$. This completes the proof of Theorem~\ref{theorem:handover_freq1}.
\end{proof}
\begin{remark}
In view of (\ref{eq:VcalHPPL}), we can give an alternative definition of the handover frequency, instead of the one in (\ref{eq:h1a}), as follows:
\begin{align}
\la_\Vcal=m_{\Vcal}[0,1]
&=\E\left[\sum_{(T_i,H_i)\in \Hcal\, :\, T_i\in [0,1]}\;\;\sum_{(T_j,H_j)\in \Hcal \,:\, T_j >  T_i} \one_{A_{i,j}}\right].\nn
%
\end{align}
Under this definition we can carry out the entire computation of the handover frequency $\la_\Vcal$.
\label{remark:left-right}
\end{remark}
\begin{remark}[Speed $v\neq 1$]
If the mobile stations move at a general speed $v$, the handover frequency can be computed to be
$\la_{\Vcal}=\frac{4v\sqrt{\la}}{\pi}$. This is justified since, in view of Lemma~\ref{lemma:tips_density} and Remark~\ref{remark:r1},
the speed $v$ scales the time axis proportionally.
\label{remark:vspeed}
\end{remark}
\begin{remark}[Comparison with the static scenario]
Consider the same model, but with static stations and a mobile user moving along a straight line with speed $v$,
see Figure~\ref{figure:static}. The static model was studied in various papers of the literature, for example \cite{BKLZ}, \cite{Baccelli_Zuyev},
\cite{Baccelli_Madadi_Gustavo},  \cite{Baccelli_Zuyev2}. The handover frequency is computed as the intensity of the point process
corresponding to the intersection of an infinite line with the boundaries of cells of the corresponding Poisson-Voronoi tessellation. This uses the contact and cord length distribution of the Poisson-Voronoi tessellation~\cite{Muche_Stoyan}.
The frequency of handovers in the simple dual static model is also $\la_\Vcal=\frac{4\sqrt{\la}}{\pi}$, as analytically proved in \cite{Baccelli_Zuyev2}, \cite[7.22]{Jasper_Moller} and \cite{Mecke1981}. This fact will be discussed further in the following, as an alternative  proof of Theorem~\ref{theorem:handover_freq1}.
\begin{figure}[ht!]
    \centering
\includegraphics[height=0.37\linewidth, width=0.37\linewidth]{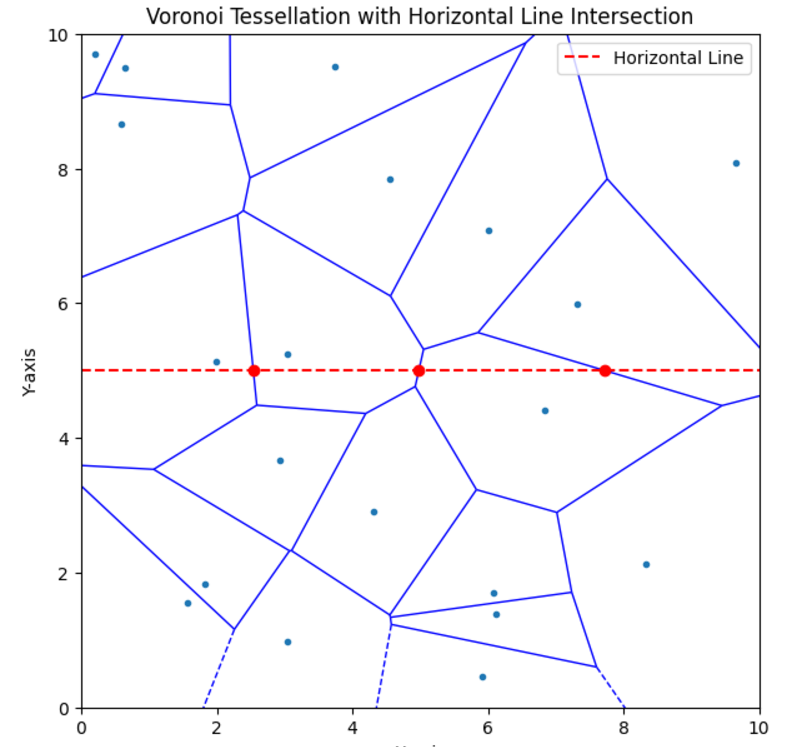}
\captionsetup{width=0.85\linewidth}
\caption{In the dual system the handover points are represented by the crossings of the Voronoi boundaries by the trajectory of an user moving along a horizontal line.}
\label{figure:static}
\end{figure}
\label{remark:static}
\end{remark}
\subsubsection{Another proof of Theorem~\ref{theorem:handover_freq1}}~\label{subsubsection:alter_HF}
We give here a simpler proof of Theorem~\ref{theorem:handover_freq1}. Unfortunately, this simpler approach does not extend to the multi-speed case, which
will centrally rely on the methodology developed for the first proof.
\begin{proof}
Let us first consider the simpler model with the direction of motion fixed and the same for all stations,
let us say $\Theta=\theta$. Then at time $t$, the distance of the mobile station, starting at $X= (|X|, \Psi)$,
from the user is given by
\[
\vert X^t\vert :=\left(\vert X\vert ^2+2t\vert X\vert \cos \a'+ t^2\right)^\half=\left(R'^2+2t R'\cos\a'+ t^2\right)^\half :=f((R',\a'), t),
\]
where $|X|:=R'$ and $\a'\equiv(\Psi-\th)\pmod {2\pi}$. Collectively the pairs of the form $(R',\a')$ give rise
to a marked Poisson point process $\tilde{\Phi}':=\sum_{i\in \N}\delta_{(R'_i, \a'_i)}$ on $\R^+\times [-\pi, \pi]$.
Subsequently, one can construct the head point process $\Hcal':=\sum_{i\in \N}\delta_{(T'_i,H'_i)}$,
where $(T'_i,H'_i):=(-R'_i\cos\a'_i, R'_i|\sin\a'_i|)$ and the marked head point process
$\Hcal'_c:=\sum_{i\in \N}\delta_{(T'_i,H'_i, C'_i)}$, where $C'_i:=\{(t, f((R'_i,\a'_i), t))\,\mbox{:}\, t\in \R\}$ are the radial birds.
The lower envelope process in this case $\Lcal'(\cdot)$ is defined as 
\[
\Lcal'(t):=\{(t, L'(t))\,\mbox{:}\,t\in\R\}, \text{ where, } L'(t):=\inf_{i\in \N} f((R'_i,\a'_i), t).
\]
The key observation is that $\a'$ has the same distribution as $\a$.
Hence the marked head point processes $\Hcal'_c$ and $\Hcal_c$, the one in the random direction case 
(see the Definition ~\ref{definition:BPPss}), have the same distribution. As a result, the lower envelope process $\Lcal'(\cdot)$ has the same distribution as the one in
the random direction case defined in Definition~\ref{definition:Lower_env}, see Figure~\ref{figure:BPP-1}. 

As we have seen in Subsection~\ref{subsection:FC-HPP}, the handover point process is totally determined 
by the lower envelope process $\Lcal'(\cdot)$ of the marked head point process $\Hcal'_c$. 
Hence, the handover point process has the same intensity in both the random and non-random direction cases.
Having just a single direction of motion $\theta$ for all mobile stations is equivalent to the situation where
the user is moving in the $\pi+\theta$ direction, while all the mobile stations are static, 
as considered in \cite{Baccelli_Zuyev}. There, the handover frequency is $\frac{4\sqrt{\la}}{\pi}$
in the  unit speed case. Since the radial bird particle process is oblivious of the direction of motion of the stations,
the handover frequency coincides with the one in the static case. 
This completes the second proof of  Theorem~\ref{theorem:handover_freq1}.
\end{proof}
\subsection{Another proof of Theorem~\ref{theorem:handover_freq1} using the displacement theorem}
\begin{proof}[Another proof of Theorem~\ref{theorem:handover_freq1} ]
Define a rotation operator $R_{\th}: \R^2\times [-\pi, \pi) \to \R^2$ as $R_\th(x,\th) := (R_\th(x), 0)$, where $R_\th(x)$ denotes the radial rotation of $x$ by an angle $\th$ clockwise. Consider the point point process 
\[
R(\tilde \Phi)= \sum_{(X,\Theta)\in \tilde \Phi} \delta_{(R_\Th(X), 0)}.
\]
It is a point process with points of $\tilde \Phi$ are displaced independently after suitable rotation. The directions of motion are parallel to $x$-axis and to the right. By the displacement theorem, $R(\tilde \Phi)$ is also a Poisson point process as $\Phi$ with same intensity $\la$. Here the direction of motion for all stations are fixed and along positive direction of $x$-axis, which is equivalent to our dual dynamical model, where user is moving in straight line along negative direction of $x$-axis. Hence all the statistics in this model have the same distribution to the same in the dual dynamical model.
\end{proof}
\begin{remark}
Moving beyond handover frequency, our analysis successfully captures the statistical properties of geometric object of interest, studied in the rest of the article. Note that our results hold for both the dynamic and the dual static model, which are captured together through  Figure~\ref{figure:BPP-MPN}.  
\end{remark}
\subsection{State at typical times and handover times}\label{subsection:typT} 
Let $L_s$ be the vertical line $t=s$, for some $s\in \R$. Observe that all the radial birds intersect the vertical line $L_s$ at certain heights. The point process of interest here, representing the ordered distances of the base stations, is given by the intersection of the birds with this vertical line, see Figure~\ref{figure:typ-dists}. The further away a head point is from $L_s$, on the right or left, the higher its intersection with the corresponding radial bird. This gives rise to a point process on $L_s$, which can be seen as a point process on $\R^+$. Recall that $\pi_2(A)=\{\pi_2(x, y): \text{ for all } (x, y)\in A \}$, for any $A\subset \R^2$, where $\pi_2:\mathbb{H}^+\to \R^+$ is the projection map to the $h$-axis such that $\pi_2(x,y)=y$ for $(x,y)\in \mathbb{H}^+$. \label{notation:L_s}

Let $\eta_s$ be the  point process on $\R^+$ corresponding to the intersection points of all the birds with the vertical line $L_s$, i.e.,\label{notation:eta_s}
\begin{equation}
\eta_s:=\sum_{i\in \N \,\mbox{:}\, (T_i,H_i)\in \Hcal}\delta_{\pi_2(C_i\cap L_s)},\nn
\end{equation}
where $C_i$ is the radial bird corresponding to $(T_i,H_i)\in \Hcal$. The point process $\eta_s$ gives the distances between the user and the mobile stations at time $s$. We would like to characterize the point processes $\{\eta_s\}_{s\in \R}$ on $\R^+$, in particular at any typical time and also at handover times. 
\begin{figure}[ht!]
    \centering
\begin{tikzpicture}[scale=1]
\pgftransformxscale{0.6}  
    \pgftransformyscale{0.6}    
    \draw[->] (-2, 0) -- (3.2, 0) node[right] {$t$};
    \draw[->] (-1.5, 0) -- (-1.5, 4.5) node[above] {$h$};
     \draw[blue, ->] (2.5, 0) -- (2.5, 4.5);
    \draw[domain=-2:3, smooth] plot (\x, {(\x*\x+2.5)^0.5});
    \draw[](-0.07,1.55) node{$\bullet$};
    \draw[domain=-2:3, smooth] plot (\x, {(\x*\x-3*\x+2.5)^0.5});
    \draw[](1.51,0.5) node{$\bullet$};
    \draw[blue](2.5,1.1) node{$\bullet$};
     \draw[blue](2.5,1.8) node{$\bullet$};
      \draw[blue](2.5,2.95) node{$\bullet$};
       \draw[blue](2.5,3.58) node{$\bullet$};
    \draw[](2.5,4.7) node{$L_s$};
    \draw[](2.5,-0.3) node{$s$};
    \draw[domain=-2:3, smooth] plot (\x, {(\x*\x-2*\x+2)^0.5});
    \draw[](0.95,1) node{$\bullet$};
    \draw[domain=-2:3, smooth] plot (\x, {(\x*\x+2.1*\x+1.2)^0.5});
    \draw[](-1.05,0.31) node{$\bullet$};
    \end{tikzpicture}
    \captionsetup{width=0.95\linewidth}
    \caption{All the radial birds intersects the line $L_s$ at different heights and creates a point process $\eta_s$.}
    \label{figure:typ-dists}
\end{figure}
Let $\hat{h}_s:=\inf\left(Supp( \eta_s)\right)$. Then $\hat h_s$ is the distance to the serving station at time $s$, which determines the power of the signal received by the user. On the other hand the point process $\eta_s-\delta_{\hat h_s}$ enables us to encode the distance of all other non-serving stations, contributing to the power of the interference experienced by the user.

Suppose there is a handover at time $0$, namely consider the system under the Palm probability measure $\P^0_\Vcal$ of handovers (to be determined rigorously in Theorem~\ref{theorem:Palm-Vcal} in Section~\ref{sec-HOdist}). Consider the point process $\eta_0$ for which $\inf(Supp(\eta_0))=\hat h$, say. At a typical handover, there are two stations at equal distance $\hat{h}$. We call $\hat h$ the typical handover distance and we derive its Palm probability distribution in Subsection~\ref{subsection:typHO}.  In the following, we give a result about the distribution of point process $\eta_0- \delta_{\hat{h}}$. 
\begin{lemma}
Conditioned on $\hat h$, under the Palm probability measure $\P^0_\Vcal$, the point process $\eta_0$ is a Poisson point process on $[\hat h, \infty)$ with intensity measure $\hat{\mu}$ with density ${\rm d}\hat{\mu}:=2\pi\la r\, {\rm d}r$.   
\label{lemma:eta_s_hat1}
\end{lemma}
\begin{proof}[Proof]
Suppose that, under the Palm probability measure $\P^0_\Vcal$ of handovers, there is a handover at time $0$ at a distance $\hat h$. Then there exists $(t,h), (t',h')\in \Hcal$ such that, the radial birds $C_{(t,h)}$ and $C_{(t',h')}$ intersect at $(0,\hat h)$ and $\Hcal(U^0_{\hat h})=0$. Since $\Hcal$ itself is a Poisson point process, we have from Corollary~\ref{corollary:CCc}, sub point processes $\Hcal\vert_{\overline{U^0_{\hat h}}}$ and $\Hcal\vert_{\left(\overline{U^0_{\hat h}}\right)^c}$ which are independent and such that $\Hcal\vert_{\left(\overline{U^0_{\hat h}}\right)^c}$ is a Poisson point process. Note that the point process $\eta_0-\delta_{\hat h}$ of the distances of the mobile stations at time $0$ is created by all the radial birds associated with the head point process $\Hcal\vert_{\left(\overline{U^0_{\hat h}}\right)^c}$ and $Supp(\eta_0)\subsetneq [\hat h, \infty)$, conditioned on $\hat{h}$. Since $\Hcal\vert_{\left(\overline{U^0_{\hat h}}\right)^c}$ is an independent Poisson point process, the point process $\eta_0-\delta_{\hat h}$ is a Poisson point process with intensity measure $\hat{\mu}$ having density ${\rm d}\hat{\mu}:=2\pi\la r\,{\rm d}r$ on $[\hat h, \infty)$, as seen in the radial part of the one at (\ref{eq:mu_r}). This completes the proof.
\end{proof}

\section{Palm probability with respect to handovers and applications} \label{sec-HOdist}
The goal of this section is to determine the Palm probability distribution w.r.t. handover epochs.
This is done in Subsection \ref{subsection:Palm-Handover}, by extending the method used in the proof of Theorem~\ref{theorem:handover_freq1}. This is then leveraged to determine the Palm distribution of the distance to the stations
involved in the handover and the Palm distribution of the time to the next handover.

\subsection{Palm probability with respect to handovers}\label{subsection:Palm-Handover}
In this section, the reference probability space is $(\Omega,\Fcal,\mathbb P)$, 
with $(\Omega,\Fcal)$ the canonical space of point processes on $\mathbb{H}^+$.
So a point $\omega$ is a realization of the point process 
$\Hcal=\sum_{i\in \N} \delta_{(T_i,H_i)}$.
One can equip this measure space with the shift $\{\theta_t\}_{t\in \mathbb{R}}$ along the time axis.
Its action on $\Hcal$ is defined by
\begin{equation}
\theta_t(\Hcal)=\sum_{i\in\N}  \delta_{(T_i-t,H_i)}.
\label{eq:shift_Hcalc}
\end{equation}
The law $\mathbb P$ of the Poisson point process $\Hcal$ is left invariant by this shift.
This shift is ergodic for $\mathbb P$.

Consider the handover point process $\Vcal$, as constructed in Subsection~\ref{subsection:FC-HPP}.
This point process is stationary ($\theta_t$-compatible, see \cite[Definition 6.1.8]{Baccelli-Bartek-Karray}) using Lemma~\ref{lemma:stationary_Tcal} and has a positive and finite intensity from Theorem~\ref{theorem:handover_freq1}.
So one can define the Palm probability w.r.t. $\Vcal$ on the probability space $(\Omega,\Fcal)$, which will be denoted as
$\mathbb P_{\Vcal}^0$. 

Consider the pairs of points $((t,h), (t',h'))$ of $\mathbb{H}^+$, such that $t'\leq t$ and the condition that
$\Hcal(U^{\hat{s}}_{\hat{h}})=0$ holds, where $(\hat{s},\hat{h})$ are the coordinates of the intersection
of the two radial birds with heads at $(t,h)$ and $(t',h')$, respectively.
The sequence of rightmost points of such pairs that are in the support of $\Hcal$ forms
a stationary ($\theta_t$-compatible) point process $\Rcal$ on $\mathbb R$.\label{notation:Rcal}\! This point process has
a positive and finite intensity.
So one can define the Palm probability w.r.t. $\Rcal$ (on the probability space $(\Omega,\Fcal)$),
which will be denoted as $\mathbb P_{\Rcal}^0$.
\begin{figure}[ht!]
\centering
\begin{center}
\begin{tikzpicture}
\pgftransformxscale{0.35}  
\pgftransformyscale{0.35}    
    \draw[->] (-8,0) -- (11,0) node[right] {$t$};
    \draw[blue, domain=1:10, smooth] plot (\x, {(\x*\x-14*\x+49+4)^0.5});
    \draw[blue](7,2) node{$\bullet$};    
    \draw[blue, domain=-4:6, smooth] plot (\x, {(\x*\x-4*\x+2^2+ 1.2^2)^0.5});
    \draw[blue](2,1.2) node{$\bullet$};    \draw[blue, domain=-7:2, smooth] plot (\x, {(\x*\x+9*\x+4.5^2+2^2)^0.5});
    \draw[blue](-4.5,2) node{$\bullet$}; 
    \draw[<->] (-1.1,3.65) arc (90:20:3.2);
    %
    \draw[<->] (5,3) arc (90:50:3);
    \draw[] (-1.42, 3.7) -- (-1.42, 0);
    \draw[] (4.75, 3.05) -- (4.75, 0);
    \draw[blue] (2, 1.3) -- (2, 0);
    \draw[blue] (7, 2.1) -- (7, 0);
    \end{tikzpicture}
    \captionsetup{width=0.85\linewidth}
    \caption{The bijection between point process $\Rcal$ and the handover point process $\Vcal$ via the head point process $\Hcal$ and the intersection points of the radial birds.}
    \label{fig:mass-transport}
\end{center}
\end{figure}
As already explained, there is a bijection between the points of $\Rcal$ 
and the points of $\Vcal$, that will be denoted by $\beta$, (see Figure~\ref{fig:mass-transport}).
Note that if $((t,h)$ and $(t',h'))$ is as above, then $\beta(t)=\hat{s}$, the abscissa of the intersection of the two radial birds with these heads.
The following lemma is obtained from Theorem 6.1.35 in \cite{Baccelli-Bartek-Karray} (which is itself a consequence
of the mass transport principle for stationary point processes):

\begin{lemma}
	For all non-negative measurable functions $f$ on $(\Omega,\Fcal)$,
	\begin{equation}
		\mathbb E_{\Vcal}^0 [f(\Hcal)] = 
		\mathbb E_{\Rcal}^0 [f (\theta_{\beta(0)}(\Hcal))].
  \label{eq:VR-MTPx}
	\end{equation}
    \label{lemma:VR-MTP}
\end{lemma}
\begin{proof}[Proof of Lemma~\ref{lemma:VR-MTP}]
The formula (\ref{eq:VR-MTPx}) is a special case of that given in \cite[Theorem 6.1.35]{Baccelli-Bartek-Karray} when taking
$\Phi=\Vcal$, $\Phi'=\Rcal$ and $g(y,\omega)= f(\omega) \one_{\{y=\beta^{-1}(0)\}}$.
\end{proof}
Let us denote the intensity of the point process $\Rcal$, as $\la_\Rcal$. Due to the bijection $\beta$ between $\Vcal$ and $\Rcal$ in the mass transport principle under the Palm probability measures, we get the equality of intensities:  
\begin{corollary}
The intensities of the point processes $\Vcal $ and $ \Rcal$ are equal: $\la_\Rcal=\la_{\Vcal}$. 
\label{corollary:eqVR}
\end{corollary}
We now use the Poisson structure of $\Hcal$ to get:
\begin{lemma}
	For all non-negative measurable functions $f$ on $(\Omega,\Fcal)$,
	\begin{align}
		\hspace{-0.1cm} \mathbb E_{\Rcal}^0 [f (\theta_{\beta(0)}(\Hcal))]
		& = 
		\frac {4\la^2} {\lambda_{\Rcal}}
		\int_{(\R^+)^3} \!\!\!
		\E_\Hcal\left[ \one_{\Hcal\left(U^{\hat{s}(0,h,-t',h')}_{\hat{h}(0,h,-t',h')}\right)=0}
		f \left(\theta_{\hat{s}(0,h,-t',h')}           \left(\Hcal+\delta_{(0,h)}+\delta_{(-t',h')}\right)\right)\right] {\rm d} h'{\rm d} h\,  {\rm d} t',
        \nonumber
	\end{align}
where, $(\hat{s}(t,h,t',h'),\hat{h}(t,h,t',h'))$
are the coordinates of the intersection of the two radial birds with heads $(t,h)$ and $(t',h')$ and $\la_\Rcal=\frac{4\sqrt{\la}}{\pi}$.
\label{lemma:Palm-Lcal}
\end{lemma}~\label{notation:hats0}
\begin{proof}[Proof of Lemma~\ref{lemma:Palm-Lcal}]
For all $(T_i,H_i)$ and $(T_j,H_j)$ in the support of $\Hcal$, let
$(\hat{S}_{i,j},\hat{H}_{i,j})$ denote the coordinates of the intersection of the two radial birds with these heads.
By the very definition of $\P_{\Rcal}^0$,
\begin{align}
\mathbb{E}_{\Rcal}^0 [f (\theta_{\beta(0)}(\Hcal))]
& = \frac 1 {\lambda_{\Rcal}}
\mathbb{E} \left[ \sum_{T_i\in \Rcal \,\mbox{:}\, 0\le T_i \le 1} f \left(\theta_{\beta(0)} \circ \theta_{T_i}(\Hcal)\right)\right].
\label{eq:Last-H-1}
\end{align}
Using the property of the shift $\{\th_t\}_{t\in \R}$, we have  $\theta_{\beta(0)} \circ \theta_{T_i}=\theta_{T_i+\beta(T_i)}$ and hence
\begin{align}
\hspace{-.7cm}\lefteqn{\mathbb{E}_{\Rcal}^0 [f (\theta_{\beta(0)}(\Hcal))]} \nonumber\\ 
& = \frac 1 {\lambda_{\Rcal}}
	\mathbb{E} \left[ \sum_{T_i\in \Rcal, 0\le T_i \le 1}
	f \left(\theta_{T_i+\beta(T_i)}(\Hcal)\right) \right]\nonumber \\
	& = 
	\frac 1 {\lambda_{\Rcal}}
	\mathbb{E} \left[ \sum_{(T_i,H_i)\in \Hcal, 0\le T_i \le 1} 
	\one_{\exists (T_j,H_j)\in \Hcal \,\mbox{:}\, T_j< T_i \mbox{ and }
	\Hcal\left(U^{\hat{S}_{i,j}}_{\hat{H}_{i,j}}\right)=0}
	f \left(\theta_{T_i+\beta(T_i)}(\Hcal)\right) \right]\nonumber \\
	& = 
	\frac 1 {\lambda_{\Rcal}}
	\mathbb{E} \left[ \sum_{(T_i,H_i)\in \Hcal, 0\le T_i \le 1} \, 
	\sum_{(T_j,H_j)\in \Hcal \,\mbox{:}\, T_j< T_i}
	\one_{\Hcal\left(U^{\hat{S}_{i,j}}_{\hat{H}_{i,j}}\right)=0}
	f \left(\theta_{T_i+\beta(T_i)}(\Hcal)\right) \right]\nonumber \\
	& = \frac 1 {\lambda_{\Rcal}}
	\mathbb{E} \left[ \sum_{(T_i,H_i)\in \Hcal, 0\le T_i \le 1} \,
	\sum_{(T_j,H_j)\in \Hcal \,\mbox{:}\, T_j< T_i}
	\one_{\Hcal\left(U^{\hat{S}_{i,j}}_{\hat{H}_{i,j}}\right)=0}
	f \left(\theta_{\hat{S}_{i,j}}(\Hcal)\right) \right].
 \label{eq:Last-H-1a}
	\end{align}
Applying the Campbell-Mecke formula to the factorial power of order $2$ of $\Hcal$, we obtain from (\ref{eq:Last-H-1a}) that
\begin{align}
\hspace{-.7cm}\lefteqn{\mathbb{E}_{\Rcal}^0 [f (\theta_{\beta(0)}(\Hcal))]} \nonumber\\
	& = \frac {(2\lambda)^2} {\lambda_{\Rcal}}
	\int_{0}^1  \int_{0}^\infty \int_{0}^\infty \int_{-\infty}^t
	\mathbb{E}^{(t,h),(t',h')}_\Hcal\left[ \one_{\Hcal\left(U^{\hat{s}(t,h,t',h')}_{\hat{h}(t,h,t',h')}\right)=0}
	f \left(\theta_{\hat{s}(t,h,t',h')}(\Hcal)\right)\right] {\rm d} t' {\rm d} h'{\rm d} h {\rm d}t \label{eq:Last1-HR} \\
	& = \frac {4\lambda^2} {\lambda_{\Rcal}}
	\int_{0}^1  \int_{0}^\infty \int_{0}^\infty \int_{-\infty}^t
	\mathbb{E}_\Hcal\left[ \one_{\Hcal\left(U^{\hat{s}(t,h,t',h')}_{\hat{h}(t,h,t',h')}\right)=0}
	f \left(\theta_{\hat{s}(t,h,t',h')}(\Hcal+\delta_{(t,h)}+\delta_{(t',h')})\right)\right] {\rm d} t' {\rm d} h'{\rm d} h {\rm d}t\nn\\
    & = \frac {4\lambda^2} {\lambda_{\Rcal}}
	\int_{0}^1  \int_{0}^\infty \int_{0}^\infty \int_{t-t'=\infty}^0
	\mathbb{E}_\Hcal\left[ \one_{\Hcal\left(U^{\hat{s}(0,h,t'-t,h')}_{\hat{h}(0,h,t'-t,h')}\right)=0}
	f \left(\theta_{\hat{s}(0,h,t'-t,h')}(\Hcal+\delta_{(0,h)}+\delta_{(t'-t,h')})\right)\right] {\rm d} t' {\rm d} h'{\rm d} h {\rm d}t,
    \label{eq:Last-H}
	\end{align}
where we used the Slivnyak-Mecke formula for the last but one equation. We obtain the final result by the change of variable $t-t'$ to $t'$ in (\ref{eq:Last-H}), similar to the one in (\ref{eq:h2}).
\end{proof}
We hence get the following result using $\la_\Rcal=\la_\Vcal=\frac{4\sqrt{\la}}{\pi}$:
\begin{theorem} For all measurable and non-negative $f$
\begin{equation}
\mathbb E_{\Vcal}^0 [f(\Hcal)] = \pi \lambda^{\frac 3 2}
\int_{(\R^+)^3}
\E_\Hcal\left[ \one_{\Hcal\left(U^{\hat{s}(0,h,-t',h')}_{\hat{h}(0,h,-t',h')}\right)=0}
f \left(\theta_{\hat{s}(0,h,-t',h')}           \left(\Hcal+\delta_{(0,h)}+\delta_{(-t',h')}\right)\right)\right] {\rm d} t'{\rm d} h'{\rm d} h.
\label{eq:HRPalm}
\end{equation}
\label{theorem:Palm-Vcal}
\end{theorem}
\begin{remark}
Observe that the steps until (\ref{eq:Last1-HR}) are true for any stationary point process on $\mathbb H^+$. In that sense Lemma~\ref{lemma:Palm-Lcal} is far more general. The same holds true for Lemma~\ref{lemma:VR-MTP}.  
\label{remark:stationary_Hcal}
\end{remark}
\subsection{Typical handover distance}\label{subsection:typHO}  Suppose $\hat{H}$ is the random distance of a typical handover. The following result evaluates the Laplace transform of $\hat{H}^2$ under the Palm distribution $\P^0_\Vcal$ of $\Vcal$, based on Theorem~\ref{theorem:Palm-Vcal}. This enables us to determine the distribution of $\hat{H}$.
\begin{lemma}
The Laplace transform of $\hat{H}^2$ under the Palm probability measure $\P^0_\Vcal$ is 
\begin{equation}
\E^0_\Vcal\left[e^{-\gamma \hat{H}^2}\right]=\left(1+\frac{\gamma}{\la \pi}\right)^{-\frac{3}{2}},
\label{eq:Lcal_hatH0}
\end{equation} 
for any $\gamma>0$. The squared handover-distance $\hat{H}^2$ follows a Gamma distribution with parameters $\left(\frac{3}{2}, \frac{1}{\la\pi}\right)$.   Moreover, the handover distance $\hat{H}$ follows a Nakagami distribution with parameters $\left(\frac{3}{2}, \frac{3}{2\la\pi}\right)$.
\label{lemma:typical_h_HO}
\end{lemma}
\begin{proof}[Proof of Lemma~\ref{lemma:typical_h_HO}] 
In order to evaluate the Laplace transform of $\hat{H}^2$ under the Palm probability measure $\P^0_\Vcal$, we directly apply Theorem~\ref{theorem:Palm-Vcal}, by taking $f(\Hcal)= e^{-\gamma \hat{H}^2}$ as follows:
\begin{align}
\E^0_\Vcal\left[e^{-\gamma \hat{H}^2}\right] & =  \pi\la^{3/2}\int_{0}^{\infty}\int_{0}^{\infty}\int_{0}^{\infty} \E_\Hcal\left[ \one_{\Hcal\left(U^{\hat{s}}_{\hat{h}}\right)=0}
e^{-\gamma (\hat h^2\circ\theta_{\hat{s}})} \right] {\rm d} t'{\rm d} h'{\rm d} h\nn\\
& =  \pi\la^{3/2}\int_{0}^{\infty}\int_{0}^{\infty}\int_{0}^{\infty} e^{-(\la \pi+\gamma) \hat{h}^2} {\rm d} t'\, {\rm d} h'\, {\rm d} h,
\label{eq:Lcal1}
\end{align}
where $\hat{s}(0,h,-t',h')\equiv \hat s$ and $\hat{h}(0,h,-t',h')\equiv \hat h$. The last step (\ref{eq:Lcal1}) is due to the fact that 
\[\E_\Hcal\left[\one_{\Hcal\left(U^{\hat{s}}_{\hat{h}}\right)=0} e^{-\gamma (\hat h^2\circ\theta_{\hat{s}})}\right]
=e^{-\gamma \hat h^2}\P_\Hcal\left(\Hcal\left(U^{\hat{s}}_{\hat{h}}\right)=0\right)
= e^{-\la\pi \hat h^2} e^{-\gamma \hat h^2},
\]
as the handover distance $\hat h$ does not change due to the shift $\theta_{\hat s}$. Recall from (\ref{eq:h}) that 
\begin{align}
\hat{h}^2= \frac{1}{4}\left[t'^2+ \frac{(h^2-h'^2)^2}{t'^2}+ 2(h^2+h'^2) \right]. 
\label{eq:hatH}
\end{align}
Using the expression for $\hat{h}^2$, keeping aside the term $\half(h^2+h'^2)$, we evaluate the innermost integral in (\ref{eq:Lcal1}), with respect to $t$, to be
\begin{align}
\int_{0}^{\infty} e^{-\frac{\la \pi+\gamma}{4} \left[t'^2+\frac{(h^2-h'^2)^2}{t'^2}\right]}  {\rm d} t' 
&=\left(\frac{\pi}{\la \pi+\gamma}\right)^{\half}e^{-\frac{\la\pi+\gamma}{2}|h^2-h'^2|},
\label{eq:h2L}
\end{align}
evaluated using the definite integral $\int_{0}^{\infty} e^{- a x^2-\frac{b}{x^2}}\, {\rm d} x= \frac{\sqrt{\pi}}{2\sqrt{a}}e^{-2\sqrt{ab}}$ from (\ref{eq:gaussian_correction}) with $a =\frac{\la\pi+\gamma}{4}>0$ and $b=\frac{\la\pi+\gamma}{4} (h^2-h'^2)^2>0$. Substituting the resultant from (\ref{eq:h2L}) to (\ref{eq:Lcal1}) along with an extra factor of $e^{-\frac{\la\pi+\gamma}{2}(h^2+h'^2)}$, the right hand side in (\ref{eq:Lcal1}) turns out to be
\begin{align}
\E^0_\Vcal\left[e^{-\gamma \hat{H}^2}\right] &=\pi\la^{3/2} \left(\frac{\pi}{\la \pi+\gamma}\right)^{\half}\int_{0}^{\infty}\int_{0}^{\infty} e^{-\frac{\la\pi+\gamma}{2}(h^2+h'^2)} e^{-\frac{\la\pi+\gamma}{2}|h^2-h'^2|} {\rm d} h'\, {\rm d} h\nn\\
    &=\left(\frac{\pi^3\la^3}{\la \pi+\gamma}\right)^{\half} \int_{0}^{\infty}\int_{0}^{\infty}e^{-\frac{\la\pi+\gamma}{2}\left(h^2+h'^2+|h^2-h'^2|\right)}{\rm d} h'\, {\rm d} h \nn\\
    &=\left(\frac{\pi^3\la^3}{\la \pi+\gamma}\right)^{\half} \int_{0}^{\infty}\int_{0}^{\infty} e^{-(\la\pi+\gamma)\left(h^2\vee h'^2\right)}{\rm d} h'\, {\rm d} h\nn\\
    &=\left(\frac{\pi^3\la^3}{\la \pi+\gamma}\right)^{\half}\frac{1}{\la\pi+\gamma} \int_{0}^{\infty}\int_{0}^{\infty} e^{-(x^2\vee y^2)}{\rm d} x\, {\rm d} y = \left(1+\frac{\gamma}{\la \pi}\right)^{-\frac{3}{2}},
\label{eq:h3L}
\end{align}
as it can be shown that $\int_{0}^{\infty}\int_{0}^{\infty} e^{-(x^2\vee y^2)}{\rm d} x\, {\rm d} y=1$. Hence it is clear that $\hat{H}^2$ is distributed as a  $\Gamma\left(\frac{3}{2}, \frac{1}{\la\pi}\right)$ random variable and, as a result, the typical handover distance $\hat{H}$ follows a Nakagami distribution with parameters $\left(\frac{3}{2}, \frac{3}{2\la\pi}\right)$. The probability density for $\hat{H}$ is 
\begin{equation}
f_{\hat{H}}(h)=4\pi\la^{3/2} h^2 e^{-\la\pi h^2}, \text{ for } h\geq 0, 
\label{eq:pdfhatH0}
\end{equation}
and its mean is $\E[\hat{H}]= \frac{2}{\pi\sqrt{\la}}$.
\end{proof}
\begin{remark}[Speed $v\neq 1$]
In the single-speed case with any non-unit station speed $v$, the typical handover distance $\hat{H}$ also follows a Nakagami distribution with parameters $\left(\frac{3}{2}, \frac{3}{2\la\pi}\right)$, as the Laplace transform of $\hat{H}^2$, under the Palm probability measure $\P^0_{\Vcal}$, remains the same, namely $\E^0_\Vcal\left[e^{-\gamma \hat{H}^2}\right] =\left(1+\frac{\gamma}{\la \pi}\right)^{-\frac{3}{2}}$. Indeed,
\begin{align}
\E^0_\Vcal\left[e^{-\gamma \hat{H}^2}\right] &= \frac{4\la^2 v^2}{\la_\Vcal} \int_{0}^{\infty}\int_{0}^{\infty}\int_0^{\infty} e^{-(\la \pi+\gamma) \hat{h}^2} {\rm d} t'\, {\rm d} h'\, {\rm d} h,
\label{eq:svn1}
\end{align}
since the head point process $\Hcal$ has intensity $2\la v$, where $\hat{h}^2= \frac{1}{4}\left[v^2t'^2+ \frac{(h^2-h'^2)^2}{v^2t'^2}+ 2(h^2+h'^2) \right]$ and $\la_\Vcal=\frac{4v\sqrt{\la}}{\pi}$. Using this in (\ref{eq:svn1}) yields
\begin{align}
\E^0_\Vcal\left[e^{-\gamma \hat{H}^2}\right] &=\pi\la^{3/2} v\int_{0}^{\infty}\int_{0}^{\infty}e^{-\frac{\la\pi+\gamma}{2}(h^2+h'^2)}\int_0^{\infty} e^{-  \frac{\la\pi+\gamma}{4}\left[v^2t'^2+\frac{(h^2-h'^2)^2}{v^2t'^2}\right]} {\rm d} t'\, {\rm d} h'\, {\rm d} h\nn\\
&=\pi\la^{3/2}  \int_{0}^{\infty}\int_{0}^{\infty}e^{-\frac{\la\pi+\gamma}{2}(h^2+h'^2)}\int_0^{\infty} e^{-  \frac{\la\pi+\gamma}{4}\left[t^2+\frac{(h^2-h'^2)^2}{t^2}\right]} {\rm d} t\, {\rm d} h'\, {\rm d} h\nn\\
&=  \pi\la^{3/2} \left(\frac{\pi}{\la \pi+\gamma}\right)^{\half}\int_{0}^{\infty}\int_{0}^{\infty}e^{-\frac{\la\pi+\gamma}{2}\left((h^2+h'^2)+|h^2-h'^2|\right)}\, {\rm d} h'\, {\rm d} h\nn\\
&=\left(\frac{\pi^3\la^3}{\la \pi+\gamma}\right)^{\half}\frac{1}{\la\pi+\gamma}  = \left(1+\frac{\gamma}{\la \pi}\right)^{-\frac{3}{2}}\nn,
\end{align}
where the last two steps are as obtained in (\ref{eq:h3L}).
\label{remark:vhatH_0}
\end{remark}

\subsection{Comparison with other typical distances}\label{subsection:CompH}
In this section, we compare the statistical properties of the distance of the two stations involved in a typical handover already determined in Lemma~\ref{lemma:typical_h_HO}, the distance to a typical visible head (Lemma~\ref{lemma:LT_visible_head}) and the distance of the nearest station at a typical time (\ref{eq:typical-time}).
\subsubsection{Distance to a typical visible head on $\Lcal_e$} \label{subsubsection:typHead}
Recall that, $\Hcal:=\sum_{i\in \N}\delta_{(T_i,H_i)}$ be the head point process and suppose $\Hcal_{\scriptscriptstyle V}\leq \Hcal$, i.e.,  $Supp(\Hcal_{\scriptscriptstyle V})\subset Supp(\Hcal)$, contains only those heads that lies on the lower envelope $\Lcal_e$. In other words, we say that $\Hcal_{\scriptscriptstyle V}$ is nothing but $\Hcal$ restricted to the lower envelope $\Lcal_e$, i.e.,  $\Hcal_{\scriptscriptstyle V}:=\Hcal\vert_{\Lcal_e}$. We use the subscript $\scriptscriptstyle V$ for the points being ``visible from the time axis'', or equivalently being on the lower envelope $\Lcal_e$. Let us write the point process $\Hcal_{\scriptscriptstyle V}$ as \label{notation:HcalV}
    \begin{equation}
    \Hcal_{\scriptscriptstyle V}:=\sum_{j\in \N}\delta_{(T^{\scriptscriptstyle V}_j, H^{\scriptscriptstyle V}_j)} = \sum_{i\in \N}\delta_{(T_i, H_i)} \one_{\{(T_i, H_i)\in \Lcal_e\}},\label{eq:Hcal_V}
    \end{equation}
where we use the notation $(T^{\scriptscriptstyle V}_j, H^{\scriptscriptstyle V}_j)$ for the visible head points. The sequence of points in $Supp(\Hcal_{\scriptscriptstyle V})$ gives rise to a stationary point process, say $\Acal_{\scriptscriptstyle V}$ on $\R$, of the abscissas of the points in $Supp(\Hcal_{\scriptscriptstyle V})$. In other words,
\begin{equation}
\Acal_{\scriptscriptstyle V}:= \sum_{j\in \N\,\mbox{:}\, (T^{\scriptscriptstyle V}_j, H^{\scriptscriptstyle V}_j)\in \Hcal_{\scriptscriptstyle V}}\delta_{T^{\scriptscriptstyle V}_j},
\label{eq:Acal_v}
\end{equation}\label{notation:supp}
\!\!is the point process of the projections of the atoms of $\Hcal_{\scriptscriptstyle V}$ on the time axis. The time-stationairty of the point processes $\Hcal_{\scriptscriptstyle V}$ and $\Acal_{\scriptscriptstyle V}$ follows from that of $\Hcal$. The following result is about the intensity of the point process $\Acal_{\scriptscriptstyle V}$, and the proof is given in Appendix~\ref{subsection:AcalV} \label{notation:AcalV} 
\begin{lemma}
    The intensity of the point process $\Acal_{\scriptscriptstyle V}$ is $\sqrt{\la}$.
    \label{lem:piP'}
\end{lemma}
\begin{remark}
In the non-unit speed case $(v\neq 1)$, the intensity of $\Acal_{\scriptscriptstyle V}$
is $v\sqrt{\la}$. This result also recovers the assertion made in \cite[Theorem 3]{Baccelli_Madadi_Gustavo}. 
\end{remark}
\begin{figure}[ht!]
\centering
\begin{center}
\begin{tikzpicture}
[scale=1, every node/.style={scale=0.7}]
\pgftransformxscale{0.2}  
\pgftransformyscale{0.2}   
\draw[->] (-9, 0) -- (7.5, 0) node[right] {$t$};
\draw[->] (-6.5, 0) -- (-6.5, 10) node[above] {$h$};
\draw[red, domain=-7:5.5, smooth] plot (\x, {(\x*\x-2*\x+1+4)^0.5});
\draw[red](1,2) node{$\bullet$};    
\draw[blue, domain=-7:5.5, smooth] plot (\x, {(\x*\x-4*\x+2^2+ 1.2^2)^0.5});
\draw[blue](2,1.2) node{$\bullet$};    
\draw[blue, domain=-7:5.5, smooth] plot (\x, {(\x*\x+9*\x+4.5^2+2^2)^0.5});
\draw[blue](-4.5,2) node{$\bullet$};    
\end{tikzpicture}
\captionsetup{width=0.85\linewidth}
\caption{The mobile station, corresponding to the radial bird in red, partly serves the user, but not during the time when it is at its closest distance.}
\label{fig:non-visible-head}
\end{center}
\end{figure}
\begin{remark}
The point process $\Hcal_{\scriptscriptstyle V}$ consists of only those atoms of $\Hcal$ that lie on the lower envelope.
There can be some atoms of $\Hcal$ for which a part of the corresponding radial bird contributes to the lower envelope,
but are such that the head point itself does not lie on $\Lcal_e$, see Figure~\ref{fig:non-visible-head}. 
In other words, the mobile stations corresponding to those points serve as the closest mobile station but not during the time when they are at their shortest distance to the user. Clearly, the point process 
$\Hcal_{\scriptscriptstyle V}$ is not sufficient to provide the true handover frequency,
since  $\la_\Vcal=\frac{4\sqrt{\la}}{\pi}>\sqrt{\la}=\la_V$. As seen earlier, we certainly need more points to obtain
the complete picture of the handover point process $\Vcal$.
\end{remark}
One can define the Palm probability, denoted by $\P^0_{\Acal_{\scriptscriptstyle V}}$, with respect to $\Acal_{\scriptscriptstyle V}$. Let $H_{\scriptscriptstyle V}$ be the distance of the typical head point corresponding to $\Acal_{\scriptscriptstyle V}$. This distance determines the best signal power received from the serving station, which is of practical interest. The proof is given in Appendix~\ref{subsection:L-LTVH} 
\begin{lemma}
For any $\gamma\geq 0$, the Laplace transform of $H_{\scriptscriptstyle V}^2$, under the Palm probability distribution $\P^0_{\Acal_{\scriptscriptstyle V}}$ of $\Acal_{\scriptscriptstyle V}$, is given by 
\begin{equation}
    \E^0_{\Acal_{\scriptscriptstyle V}}\left[e^{-\gamma H_{\scriptscriptstyle V}^2}\right] = \left(1+\frac{\gamma}{\la\pi}\right)^{-1/2}, \label{notation:LT}
\end{equation}
that is, $H_{\scriptscriptstyle V}$  follows a Nakagami distribution with parameter $\left(\half, \frac{1}{2\la\pi}\right)$.
\label{lemma:LT_visible_head}
\end{lemma}
\subsubsection{Distance of nearest station at a typical time}\label{subsubsection:Rayleigh} For any typical time, let $\tilde H$ be the distance to the nearest station to user. Suppose the typical time is $t$, then we write $\tilde H\equiv \tilde{H}_t:=\inf_{i}\vert|X_i^t\vert|$, where $X_i\in \Phi$ is the initial location of a station  and $X_i^t$ at time $t$. Then
\begin{align}
    \P(\tilde{H}_t>h)&=\P(\inf_{i}\vert|X_i^t\vert| >h)=\P(\tilde{\Phi}^t(B_h(o))=0)=\P(\Phi(B_h(o))=0)= e^{-\la\pi h^2},
    \label{eq:typical-time}
\end{align}
since $\tilde{\Phi}^t$ and $\Phi$ has the same distribution. So, $\tilde{H}$ is distributed as Rayleigh with parameter $\frac{1}{\sqrt{2\la\pi}}$.

\subsubsection{Laplace transform order} We say that two non-negative random variables $X, Y$ are ordered for the Laplace transform order, i.e., $X\geq_{Lt}Y$, if the corresponding Laplace transforms are ordered, i.e., $\Lcal_X(\gamma)\leq \Lcal_Y(\gamma)$, for all $\gamma\geq 0$. The theory of Laplace transform order is studied in \cite{Stoyan}. The following result is about the Laplace transform ordering among $\hat{H}$, $\tilde{H}$ and $H_{\scriptscriptstyle V}$. The proof is provided in Appendix~\ref{subsection:Lemma-LT}.~\label{notation:LTO}
\begin{lemma}
The typical distances $\hat{H}$, $\tilde{H}$ and $H_{\scriptscriptstyle V}$ satisfy 
\begin{equation}
\hat{H}\geq_{Lt}\tilde{H}\geq_{Lt}H_{\scriptscriptstyle V}.
\label{eq:LTO}
\end{equation}
\label{lemma:Laplace_order}
\end{lemma}
This allows one to compare the performance metrics~\cite{Andrews-etal}, namely coverage probability and Shannon rate in line with the Discussion~\ref{discussion:RBPP-SINR}, at different typical epochs, and we leave the performance evaluation for future research.
\subsection{Distribution of inter-handover time}\label{subsection:distT1} 
Let us formally define the inter-handover time under the Palm probability of the head point process $\Hcal$. For any $s\in \R$, we denote the first time a handovers happens after time $s$ as $T_1(s)$. Suppose there is a handover at time $0$ given by the intersection of two radial birds $C_{(t_1,h_1)}$ and $C_{(t_2,h_2)}$, for some $(t_1,h_1)$ and $(t_2,h_2)$ such that $t_2\leq t_1$. Let $T$ denotes the inter-handover time, which coincides with $T_1(0)$. Then the next handover will be given by an intersection that happens on the part of the radial bird $C_{(t_1,h_1)}$ lying on the right of the vertical line $t=0$. 

Consider the sequence $\{U^t_{\hat h_t}\}_{t\geq 0}$ of open half-balls of heights $\{\hat h_t\}_{t\geq 0}$, where 
$\hat{h}_t:= \left((t-t_1)^2+h_1^2\right)^\half$.
Let us also consider the collection $\left\{S_t\right\}_{t\geq 0}$, where 
\begin{equation}
S_t:=\interior{\left(\bigcup_{t'\in[0,t]} U^{t'}_{\hat{h}_{t'}}\setminus \overline{U^0_{\hat{h}_0}}\right)}.
\label{eq:st}
\end{equation}
We claim that the collection $\left\{S_t\right\}_{t\geq 0}$ of open sets is increasing in $t$, using the following   Lemma~\ref{lemma:increasingS}, for which we provide a geometric proof in Appendix~\ref{subsection:1speed-inc}.  
\begin{lemma}
The collection of sets $\left\{U^{t}_{\hat{h}_{t}}\setminus U^0_{\hat{h}_0}\right\}_{t\geq 0}$ is increasing in $t$.
\label{lemma:increasingS}
\end{lemma}
\begin{remark}[Increasing wing property]  In general, we call the decreasing and the increasing branches of the curve of a radial bird as its \textbf{decreasing wing} and \textbf{increasing wing}, respectively.
Suppose $(t_1,h_1)\in \mathbb H^+$ with $t_1\geq 0$, such that there is a radial bird $C_{(t_1,h_1)}$ with head at $(t_1,h_1)$ and $(0,\hat h_0)\in C_{(t_1,h_1)}$. Then the collection $\left\{U^t_{\hat h_t}\setminus U^0_{\hat h_0}\right\}_{t\geq t_1}$ is non-decreasing, where $\hat h_t=\left((t-t_1)^2+h_1^2\right)^\half$, since the point $(t,\hat h_t)$ lies on the increasing wing of the radial bird at $(t_1,h_1)$. We often refer to this property as the \textbf{increasing wing property} for the single speed case. 
\label{remark:inc-arm}
\end{remark}
Using the increasing property of the collection $\left\{U^{t}_{\hat{h}_{t}}\setminus U^0_{\hat{h}_0}\right\}_{t\geq 0}$, as in Lemma~\ref{lemma:increasingS}, we have 
\[
S_t=\interior{\left(\bigcup_{t'\in [0,t]}U^{t'}_{\hat{h}_{t'}}\setminus \overline{U^0_{\hat{h}_0}}\right)}= U^t_{\hat{h}_t}\setminus \overline{U^0_{\hat{h}_0}}.
\]
Also observe that the collection of sets $\left\{S_t\right\}_{t\geq 0}$ lies entirely in the region $Q_0\setminus \overline{U^{0}_{\hat h_0}}$, where $Q_s:=\left\{(t,h)\in \mathbb H^+\,\mbox{:}\, t\geq s\right\}$ is the right quadrant with respect to the vertical line $t=s$.\label{notation:Qs}

\vspace{0.1in}

\begin{definition}[Inter-handover time]
Given that there is a handover at time $0$, the time to the next handover is defined as  
\begin{equation}
T:=\inf\left\{s>0: \forall t\in (0,s], \Hcal\left(S_t \right)=0, \text{ but } \Hcal\left(\overline{S_s}\right)= 2\right\},
\label{eq:T1-0}
\end{equation}
which is equivalent to $T_1(0)$ under the handover Palm probability distribution.
\end{definition}
As a consequence of the increasing property of the sets $\left\{S_t\right\}_{t>0}$, we can conclude that  $\Hcal\left(S_{T_1(0)}\right)=0$, implies $\Hcal\left(S_t\right)=0$, for all $t\in [0,T_1(0)]$. Additionally $\Hcal\left(\overline{S_{T_1(0)}}\right)=2$. This essentially implies there exists a third head point $(t_3,h_3)$ such that $C_{(t_1,h_1)}$ and $C_{(t_3,h_3)}$ intersect at $\left(T_1(0), \hat h_{T_1(0)}\right)$,  $\Hcal\left(S_{T_1(0)}\right)=0$ and $\Hcal\left(\overline{S_{T_1(0)}}\right)=2$, see Figure~\ref{figure:firstHoT}. Hence we have the following corollary of Lemma~\ref{lemma:increasingS}:
\begin{figure}[ht!]
    \begin{tikzpicture}[scale=0.6, every node/.style={scale=0.75}]
    \pgftransformxscale{1.2}  
    \pgftransformyscale{1.2}    
    \draw[->] (-4, 0) -- (10, 0) node[right] {$t$};
    \draw[domain=-2:6, smooth] plot (\x, {(\x*\x-3.5*\x+1.75^2+2.15^2)^0.5});
    \draw[](0.3,-0.35) node{$(0,0)$};
    \draw[](4.5,-0.35) node{$(T_1(0),0)$};
    \draw[](1.75,2.15) node{$\bullet$};
    \draw[](2,3) node{$(t_1,h_1)$}; 
    \draw[]  (3,0) arc (0:180:2.5 and 2.5);
    \draw[]  (1,0) arc (-180:-360:3.5 and 3.5);
    \draw[blue]  (0.3,0) arc (-180:-360:2.35 and 2.35);
    \draw[red]  (-1.3,0) arc (-180:-360:2.3 and 2.3);
    \draw[-] (1.75, 0) -- (1.75, 4);
    \draw[->] (0.5, 0) -- (0.5, 4.3);
    \draw[] (0.5, 4.6)node{$h$};
    \draw[-] (4.5, 0) -- (4.5, 3.55);
    \draw[](0.2, 1.3) node{$\hat h_0$};
    \draw[](5.3, 2) node{$\hat{h}_{T_1(0)}$};
    \draw[](6,1) node{$U_{\hat{h}_{T_1(0)}}^{T_1(0)}$};
     \draw[](-0.5,0.6) node{$U^{0}_{\hat h_0}$};
     \draw[](-1.55,1.42)
     node{$\bullet$};
     \draw[](-2.8,1.2)
     node{$(t_2, h_2)$};
     \draw[domain=-4:2.8, smooth] plot (\x, {(\x*\x+3.1*\x+1.55^2+1.42^2)^0.5});
     \draw[](7.5,1.75) node{$\bullet$};
     \draw[](8.7,1.25)
     node{$(t_3, h_3)$};
    \draw[domain=2.7:9, smooth] plot (\x, {(\x*\x-15*\x+7.5^2+1.75^2)^0.5});
    \end{tikzpicture}
    \captionsetup{width=0.9\linewidth}
    \caption{The sequence of increasing sets of the form  $S_t=\interior{(U^t_{\hat h_t}\setminus U^{0}_{\hat h_0})}$ has no points of $\Hcal$, for all $t\in [0,T_1(0)]$, we just draw a two of them in red and blue. Here $\Hcal\left(\overline{S_{T_1(0)}}\right)=2$.}
    \label{figure:firstHoT}
\end{figure}
\begin{figure}[ht!]
    \begin{tikzpicture}[scale=0.6, every node/.style={scale=0.75}]
    \pgftransformxscale{1.2}  
    \pgftransformyscale{1.2}    
    \draw[->] (-4, 0) -- (10, 0) node[right] {$t$};
    \draw[](-1,2) node{$\bullet$};
    \draw[domain=-3:4, smooth] plot (\x, {(\x*\x+2*1*\x+1^2+2^2)^0.5});
    \draw[](-1.5,2.7) node{$(t_1,h_1)$};
    \draw[-] (-1, 0) -- (-1, 4);
    \draw[](-1,-0.35) node{$t=t_1$};
    \draw[](0.3,-0.35) node{$(0,0)$};
    \draw[](3.5,-0.35) node{$(T_1(0),0)$};
    \draw[]  (3,0) arc (0:180:2.5 and 2.5);
    \draw[]  (-1.45,0) arc (-180:-360:4.67 and 4.67);
     \draw[-] (3.2, 0) -- (3.2, 4.67);
    \draw[blue]  (-1.6,0) arc (-180:-360:3.6 and 3.6);
    \draw[blue] (2, 0) -- (2, 3.6);
    \draw[red]  (-1.7,0) arc (-180:-360:3.3 and 3.3);
    \draw[red] (1.6, 0) -- (1.6, 3.3);
    \draw[->] (0.5, 0) -- (0.5, 4.3);
    \draw[] (0.5, 4.6)node{$h$};
    \draw[](0.2, 1.3) node{$\hat h_0$};
    \draw[](3.7, 1.7) node{$\hat{h}_{T_1(0)}$};
    \draw[](6.3,1) node{$U_{\hat{h}_{T_1(0)}}^{T_1(0)}$};
     \draw[](-0.5,0.6) node{$U^{0}_{\hat h_0}$};
     \draw[](-1.55,1.42)
     node{$\bullet$};
     \draw[](-2.8,1.2)
     node{$(t_2, h_2)$};
     \draw[domain=-4:2.8, smooth] plot (\x, {(\x*\x+3.1*\x+1.55^2+1.42^2)^0.5});
     \draw[](7.5,1.75) node{$\bullet$};
     \draw[](8.7,1.25)
     node{$(t_3, h_3)$};
    \draw[domain=2.7:9, smooth] plot (\x, {(\x*\x-15*\x+7.5^2+1.75^2)^0.5});
    \end{tikzpicture}
    \captionsetup{width=0.9\linewidth}
    \caption{The sequence of increasing sets of the form  $S_t=\interior{(U^t_{\hat h_t}\setminus U^{0}_{\hat h_0})}$ has no points of $\Hcal$, for all $t\in [0,T_1(0)]$, we just draw a two of them in red and blue. Here $\Hcal\left(\overline{S_{T_1(0)}}\right)=2$.}
    \label{figure:firstHoTx}
\end{figure}
\begin{corollary}
For the time point $s$ meeting the stopping condition in the definition of $T$ in (\ref{eq:T1-0}), there exists a unique point, say $(t_3,h_3)\in \mathbb H^+$ with $t_3\geq t_1$, such that the intersection of the radial bird with that point as head and the radial bird $C_{(t_1,h_1)}$, produces a handover at time $T$. 
\label{corollary:third-pt}
\end{corollary}
In the following, we determine the Laplace transform of the inter-handover time $T$ under the Palm probability of handovers.~\label{notation:LTf0}~\label{notation:LTf}
\begin{theorem}
Consider the motion model with speed $v=1$. Under the Palm probability measure $\P^0_\Vcal$, the Laplace transform of the inter-handover time $T$ is given by,
\begin{equation}
\E^0_\Vcal\left[e^{-\rho T}\right] = 2\pi \lambda^{\frac 5 2}
\int_{(\R^+)^5} e^{-\rho (\hat s_1-\hat s)}
 e^{-2\la \vert U^{\hat s_1}_{\hat h_1}\cup U^{\hat s}_{\hat h}\vert}
 {\rm d}t_3\, {\rm d} h_3\, {\rm d} t_2\, {\rm d} h_2\, {\rm d} h_1, 
\label{eq:LT1-Palm}
\end{equation}
    for any $\rho\geq 0$, where $\hat s=\frac{h_1^2-t_2^2-h_2^2}{2t_2}, \hat h^2=\frac{1}{4}\left[t_2^2+2(h_1^2+h_2^2)+\frac{(h_2^2-h_1^2)^2}{t_2^2}\right]$, 
    $\hat s_1=\frac{t_3^2+h_3^2-h_1^2}{2t_3}$ and $\hat h_1^2=\frac{1}{4}\left[t_3^2+2(h_1^2+h_3^2)+\frac{(h_3^2-h_1^2)^2}{t_3^2}\right]$ and $\vert U^{\hat s}_{\hat h}\cup U^{\hat s_1}_{\hat h_1}\vert $ is evaluated in Lemma~\ref{lemma:UcupU}.
   \label{theorem:T-Palm}
\end{theorem}
\begin{proof}
Taking $f(\Hcal)= e^{-\rho T}$in Formula (\ref{eq:HRPalm}) from Theorem~\ref{theorem:Palm-Vcal} with $(0,h)=(0,h_1)$ and $(-t',h')=(-t_2,h_2)$, we obtain
\begin{align}
\E^0_\Vcal\left[e^{-\rho T}\right] &= \pi \lambda^{\frac 3 2}
\int_{(\R^+)^3}
\E_\Hcal\left[ \one_{\Hcal\left(U^{\hat{s}}_{\hat{h}}\right)=0}  e^{-\rho(T\circ \th_{\hat s})}\right] {\rm d} t_2\, {\rm d} h_2\, {\rm d} h_1,
\label{eq:LT1-1}
\end{align}
where  $(\hat s,\hat h)\equiv (\hat s(0, h_1, -t_2, h_2),\hat h(0, h_1, -t_2, h_2))$ and from, (\ref{eq:handover-time}) and (\ref{eq:h}),  
\begin{equation}
\hat s(0, h_1, -t_2, h_2)=\frac{h_1^2-t_2^2-h_2^2}{2t_2} \text{ and } \hat h(0, h_1, -t_2, h_2)=\frac{1}{4}\left[t_2^2+2(h_1^2+h_2^2)+\frac{(h_2^2-h_1^2)^2}{t_2^2}\right].
\label{eq:hatsh}
\end{equation}
In (\ref{eq:LT1-1}) we have 
\begin{equation}
T\circ \th_{\hat s}= T_1(\hat s)-\hat s,
\label{eq:T1-hats}
\end{equation}
where $T_1(\hat s)$ is the time of the first handover after time $\hat s$. The next handover can be determined by exploring the head points outside the region $U^{\hat s}_{\hat h}$, the radial bird of which intersects the radial bird $C_{(0,h_1)}$ at a point with abscissa more than $\hat s$. Then there is a unique head point with the property of Corollary~\ref{corollary:third-pt}. Moreover, that particular head point must lie on the right of $(0,h_1)$. This is true, since, any other radial bird with head point being on the left of $(0,h_1)$ and outside $U^{\hat s}_{\hat h}$ intersects the radial bird $C_{(0,h_1)}$ at a point, which is either contained in the open region above the lower envelope, i.e., $\Lcal_e^+$ or which produces a handover before time $\hat s$. So the next eligible head point must be on the right of $(0,h_1)$ and outside $\overline{U^{\hat s}_{\hat h}}$. So the unexplored region is $Q_0\setminus \overline{U^{\hat s}_{\hat h}}$. Suppose the next handover is given by a random head point $(T_j,H_j)$, then the intersection of $C_{(0,h_1)}$ and $C_{(T_j,H_j)}$ is $(\hat s(0,h_1,T_j,H_j), \hat h(0,h_1,T_j,H_j))$ and we can define $T_1(\hat s):=\hat s(0,h_1,T_j,H_j)$.
\begin{figure}[ht!]
    \begin{tikzpicture}[scale=0.7, every node/.style={scale=0.75}]
    \pgftransformxscale{0.95}  
    \pgftransformyscale{0.95}    
    \draw[->] (-4, 0) -- (10, 0) node[right] {$t$};
    \draw[domain=-2:7, smooth] plot (\x, {(\x*\x-3.5*\x+1.75^2+2.15^2)^0.5});
    \draw[](0.3,-0.35) node{$(\hat s,0)$};
    \draw[](1.9,-0.35) node{$(0,0)$};
    \draw[](4.5,-0.35) node{$(\hat s_1,0)$};
    \draw[](1.75,2.15) node{$\bullet$};
    \draw[](2,3) node{$(0,h_1)$};   
    \draw[]  (3,0) arc (0:180:2.5 and 2.5);
    \draw[]  (1,0) arc (-180:-360:3.5 and 3.5);
    \draw[-] (0.5, 0) -- (0.5, 2.5);
    \draw[->] (1.75, 0) -- (1.75, 4);
    \draw[](1.75, 4.3)node{$h$};
    \draw[-] (4.5, 0) -- (4.5, 3.55);
    \draw[](0.2, 1.3) node{$\hat{h}$};
    \draw[](4.9, 2) node{$\hat{h}_1$};
    \draw[](-0.5,0.6) node{$U_{\hat{h}}^{\hat s}$};
     \draw[](6,1.2) node{$U^{\hat s_1}_{\hat{h}_1}$};
     \draw[](-1.55,1.42)
     node{$\bullet$};
     \draw[](-2.8,1.2)
     node{$(-t_2, h_2)$};
     \draw[domain=-4:4, smooth] plot (\x, {(\x*\x+3.1*\x+1.55^2+1.42^2)^0.5});
     \draw[](7.5,1.75) node{$\bullet$};
     \draw[](8.7,1.25)
     node{$(t_3, h_3)$};
     \draw[domain=2.5:9, smooth] plot (\x, {(\x*\x-15*\x+7.5^2+1.75^2)^0.5});
    \end{tikzpicture}
    \captionsetup{width=0.9\linewidth}
    \caption{There exists a head point $(t_3,h_3)$ which is responsible for the first handover at time $\hat s_1$ after time $\hat s$. There should not be any points of $\Hcal$ in the region $U^{\hat s_1}_{\hat{h}_1} \setminus \overline{U_{\hat h}^{\hat s}}$.}
    \label{figure:UcupU}
\end{figure}
Based on the last discussion, we get that the inner expectation in (\ref{eq:LT1-1}) is equal to
\begin{align}
\lefteqn{\E_\Hcal\left[ \one_{\Hcal\left(U^{\hat{s}}_{\hat{h}}\right)=0}  e^{-\rho(T_1(\hat s)-\hat s)} \right]}\nn\\
&=
\E_\Hcal\left[ \one_{\Hcal\left(U^{\hat{s}}_{\hat{h}}\right)=0} \one_{\exists (T_j,H_j)\in \Hcal \,\mbox{:}\, T_j\geq 0, \, \left\{\Hcal\interior{\left(\bigcup_{t\in [0,\hat s(0,h_1,T_j,H_j)]} U^{t}_{\hat h_t}\setminus U^{\hat s}_{\hat h}\right)}=0\right\}} e^{-\rho (\hat s(0,h_1,T_j,H_j)-\hat s)}\right] \nn\\
&=
\E_\Hcal\left[ \one_{\Hcal\left(U^{\hat{s}}_{\hat{h}}\right)=0} \sum_{ (T_j,H_j)\in \Hcal \,\mbox{:}\, T_j\geq 0} \one_{\Hcal\interior{\left(\bigcup_{t\in [0,\hat s(0,h_1,T_j,H_j)]} U^{t}_{\hat h_t}\setminus U^{\hat s}_{\hat h}\right)}=0} e^{-\rho (\hat s(0,h_1,T_j,H_j)-\hat s)}\right].
\label{eq:LT1-2xx}
\end{align}
In the random sum in (\ref{eq:LT1-2xx}), there exists a unique such point $(T_j,H_j)$, as shown in Corollary~\ref{corollary:third-pt}. Using the increasing property of Lemma~\ref{lemma:increasingS} in (\ref{eq:LT1-2xx}) we have 
\begin{align}
\E_\Hcal\left[ \one_{\Hcal\left(U^{\hat{s}}_{\hat{h}}\right)=0}  e^{-\rho(T_1(\hat s)-\hat s)} \right]
&=
\E_\Hcal\left[\sum_{ (T_j,H_j)\in \Hcal \,\mbox{:}\, T_j\geq 0} \one_{\Hcal\left( U^{\hat s(0,h_1,T_j,H_j)}_{\hat h(0,h_1,T_j,H_j)}\cup U^{\hat s}_{\hat h}\right)=0} e^{-\rho (\hat s(0,h_1,T_j,H_j)-\hat s)}\right].
\label{eq:LT1-2x}
\end{align}
Using the Campbell-Mecke formula in (\ref{eq:LT1-2x}), we have
\begin{align}
\lefteqn{2 \la
\int_0^{\hat h}\int_{\left(\hat h^2-(t_3-\hat s)^2\right)^\half}^\infty
\E_\Hcal\left[   \one_{\Hcal(U^{\hat s(0,h_1,t_3,h_3)}_{\hat h(0,h_1,t_3,h_3)}\cup U^{\hat s}_{\hat h})=0} e^{-\rho (\hat s(0,h_1,t_3,h_3)-\hat s)}
\right] {\rm d}h_3\, {\rm d} t_3}\nn\\
&\;\; + 2 \la
\int_{\hat h}^\infty \int_0^\infty
\E_\Hcal\left[   \one_{\Hcal(U^{\hat s(0,h_1,t_3,h_3)}_{\hat h(0,h_1,t_3,h_3)}\cup U^{\hat s}_{\hat h})=0} e^{-\rho (\hat s(0,h_1,t_3,h_3)-\hat s)}\right] {\rm d}h_3\, {\rm d} t_3\nn\\
&=2 \la
\int_0^{\hat h}\int_{\left(\hat h^2-(t_3-\hat s)^2\right)^\half}^\infty e^{-\rho (\hat s_1-\hat s)}\E_\Hcal\left[   \one_{\Hcal(U^{\hat s_1}_{\hat h_1}\cup U^{\hat s}_{\hat h})=0} \right] {\rm d}h_3\, {\rm d} t_3\nn\\
&\;\;+ 2 \la
\int_{\hat h}^\infty\int_0^\infty e^{-\rho (\hat s_1-\hat s)}\E_\Hcal\left[   \one_{\Hcal(U^{\hat s_1}_{\hat h_1}\cup U^{\hat s}_{\hat h})=0} \right] {\rm d}h_3\, {\rm d} t_3.
\label{eq:LT1-2a}
\end{align}
Using the void probability, the last sum in (\ref{eq:LT1-2a}) equals
\begin{align}
\lefteqn{ \hspace{-3.1in} 2 \la
\int_0^{\hat h}\int_{\left(\hat h^2-(t_3-\hat s)^2\right)^\half}^\infty e^{-\rho (\hat s_1-\hat s)}e^{-2\la \big\vert U^{\hat s_1}_{\hat h_1}\cup U^{\hat s}_{\hat h}\big\vert}{\rm d}h_3\, {\rm d} t_3+ 2 \la
\int_{\hat h}^\infty\int_0^\infty e^{-\rho (\hat s_1-\hat s)}e^{-2\la \big\vert U^{\hat s_1}_{\hat h_1}\cup U^{\hat s}_{\hat h}\big\vert}{\rm d}h_3\, {\rm d} t_3}\nn\\
& \hspace{-3.1in} :=\zeta(\rho,h_1,t_2,h_2),
\label{eq:LT1-2}
\end{align}
where $(\hat s_1,\hat h_1)\equiv (\hat s(0, h_1, t_3, h_3),\hat h(0, h_1, t_3, h_3))$ and from (\ref{eq:handover-time}) and (\ref{eq:h}) 
\begin{equation}
\hat s(0, h_1, t_3, h_3):=\frac{t_3^2+h_3^2-h_1^2}{2t_3}\text{ and } \hat h^2(0, h_1, t_3, h_3):=\frac{1}{4}\left[t_3^2+2(h_1^2+h_3^2)+\frac{(h_3^2-h_1^2)^2}{t_3^2}\right].
\label{eq:hatsh1}
\end{equation}
In (\ref{eq:LT1-2}), the function $\zeta(\rho,h_1,t_2,h_2)$ denotes the Laplace transform of the inter-handover time $T$, given that there is a handover produced by $(0,h_1)$ and $(-t_2,h_2)$.
Figure~\ref{figure:UcupU} depicts the existence of the void region $U^{\hat s_1}_{\hat h_1}\cup U^{\hat s}_{\hat h}$. The handover distances $\hat h$ and $\hat h_1$ do not change due to the shift $\{\theta_t\}_t$. By replacing the resultant from (\ref{eq:LT1-2}) in (\ref{eq:LT1-1}) yields 
\begin{align}
\E^0_\Vcal\left[e^{-\rho T}\right] &= \pi \la^{\frac 3 2}\int_{(\R^+)^3}\zeta(\rho,h_1,t_2,h_2)\, {\rm d} t_2\, {\rm d} h_2\, {\rm d} h_1\nn\\
&=2\pi \lambda^{\frac 5 2}\int_{(\R^+)^3}
\int_0^{\hat h}\int_{\left(\hat h^2-(t_3-\hat s)^2\right)^\half}^\infty e^{-\rho (\hat s_1-\hat s)}e^{-2\la \big\vert U^{\hat s_1}_{\hat h_1}\cup U^{\hat s}_{\hat h}\big\vert}{\rm d}h_3\, {\rm d} t_3 {\rm d} t_2\, {\rm d} h_2\, {\rm d} h_1\nn\\
&\;\;+2\pi \lambda^{\frac 5 2}\int_{(\R^+)^3}
\int_{\hat h}^\infty\int_0^\infty e^{-\rho (\hat s_1-\hat s)}e^{-2\la \big\vert U^{\hat s_1}_{\hat h_1}\cup U^{\hat s}_{\hat h}\big\vert}{\rm d}h_3\, {\rm d} t_3{\rm d} t_2\, {\rm d} h_2\, {\rm d} h_1.
\label{eq:LT1-3}
\end{align}
The area of the region $U^{\hat s}_{\hat h}\cup U^{\hat s_1}_{\hat h_1}$ is evaluated below in Lemma~\ref{lemma:UcupU}. Using the expressions of $\hat s, \hat h, \hat s_1$ and $\hat h_1$ from (\ref{eq:hatsh}) and (\ref{eq:hatsh1}), one can evaluate the final expression for $\E^0_{\Vcal}\left[e^{-\rho T}\right]$ from (\ref{eq:LT1-3}).
\end{proof}
We now evaluate the area of the region $U^{\hat s_1}_{\hat h_1}\cup U^{\hat s}_{\hat h}$ in terms of the original variables $h_1, t_2,h_2,t_3,h_3$, using the following geometric lemma, see Figure~\ref{figure:UcupU} for illustration. This is presented in Appendix~\ref{subsection:UcapU}.
\begin{remark}
From (\ref{eq:LT1-3}), for $\rho=0$, we have
\begin{align}
\int_{(\R^+)^3}\left[
\int_0^{\hat h}\int_{\left(\hat h^2-(t_3-\hat s)^2\right)^\half}^\infty e^{-2\la \big\vert U^{\hat s_1}_{\hat h_1}\cup U^{\hat s}_{\hat h}\big\vert}{\rm d}h_3\, {\rm d} t_3 +
\int_{\hat h}^\infty\int_0^\infty e^{-2\la \big\vert U^{\hat s_1}_{\hat h_1}\cup U^{\hat s}_{\hat h}\big\vert}{\rm d}h_3\, {\rm d} t_3\right]{\rm d} t_2\, {\rm d} h_2\, {\rm d} h_1=\frac{1}{2\pi \lambda^{\frac 5 2}}.
\label{eq:rho0}
\end{align}
Also, for small values of $\rho$, we have $\frac{1}{\rho}\left(1-E^0_{\Vcal}[e^{-\rho T}]\right)\approx E^0_{\Vcal}[T]=\frac{1}{\la_\Vcal}= \frac{\pi}{4\sqrt{\la}}$ and hence from (\ref{eq:LT1-3}) we have 
\begin{align}
\lefteqn{\int_{(\R^+)^3}
\int_0^{\hat h}\int_{\left(\hat h^2-(t_3-\hat s)^2\right)^\half}^\infty e^{-\rho (\hat s_1-\hat s)}e^{-2\la \big\vert U^{\hat s_1}_{\hat h_1}\cup U^{\hat s}_{\hat h}\big\vert}{\rm d}h_3\, {\rm d} t_3 {\rm d} t_2\, {\rm d} h_2\, {\rm d} h_1}\nn\\
&\;\;+\int_{(\R^+)^3}
\int_{\hat h}^\infty\int_0^\infty e^{-\rho (\hat s_1-\hat s)}e^{-2\la \big\vert U^{\hat s_1}_{\hat h_1}\cup U^{\hat s}_{\hat h}\big\vert}{\rm d}h_3\, {\rm d} t_3{\rm d} t_2\, {\rm d} h_2\, {\rm d} h_1
\approx \frac{1-\rho \E^0_{\Vcal}[T]}{2\pi \la^{\frac 5 2}}= \frac{1-\frac{\pi\rho}{4\sqrt{\la}}}{2\pi \la^{\frac 5 2}}.
\label{eq:rho-small}
\end{align}
\end{remark}
\begin{remark}
For $v\neq 1$ we have
\begin{align}
\E^0_\Vcal\left[e^{-\rho T}\right] &=2 v^2\pi \lambda^{\frac 5 2}\int_{(\R^+)^3}
\int_0^{\hat h/v}\int_{\left(\hat h^2-v^2(t_3-\hat s)^2\right)^\half}^\infty e^{-\rho (\hat s_1-\hat s)}e^{-2v\la \big\vert E^{\hat s_1,v}_{\hat h_1}\cup E^{\hat s,v}_{\hat h}\big\vert}{\rm d}h_3\, {\rm d} t_3 {\rm d} t_2\, {\rm d} h_2\, {\rm d} h_1\nn\\
&\;\;+2v^2\pi \lambda^{\frac 5 2}\int_{(\R^+)^3}
\int_{\hat h/v}^\infty\int_0^\infty e^{-\rho (\hat s_1-\hat s)}e^{-2v\la \big\vert E^{\hat s_1,v}_{\hat h_1}\cup E^{\hat s,v}_{\hat h}\big\vert}{\rm d}h_3\, {\rm d} t_3{\rm d} t_2\, {\rm d} h_2\, {\rm d} h_1, 
\label{eq:LT1-Palmv}
\end{align}
where the area of the union of half-ellipses $E^{\hat s_1,v}_{\hat h_1}, E^{\hat s,v}_{\hat h}$ is evaluated in Lemma~\ref{lemma:EUE} in Appendix~\ref{sec-appendix}, after expressing $\hat s, \hat h, \hat s_1$ and $\hat h_1$ in terms of $h_1, t_2, h_2, t_3, h_3$ and $v$, similarly to the one in Theorem~\ref{theorem:T-Palm}.
\end{remark}
%

\section{Handover process as a factor of a Markov chain: single-speed case} \label{sec-MC1}
Let us first assume that the speed is $v=1$. The analysis for the case $v\neq 1$ is similar, except for the fact that we use the empty half-ellipse condition, instead of the empty half-ball condition. The preceding analysis of the handover frequency and the Palm distribution of inter-handover time is governed only by the head point process $\Hcal$. Indeed, the radial bird particle process $\Hcal_c$, the collection of their intersection points and the handover point process $\Vcal$ are factors of $\Hcal$. The focus of this section is the representation of the triple made of inter-handover times and distances to handover performing stations, as a factor of a Markov chain. More precisely, consider the sequence $\left\{\left(H_n^l, T^r_n, H^r_n\right)\right\}_{n\in \Z}$ of triples of $(\R^+)^3$, where $H_n^l$ is the distance of the left head point of the $n$-th handover, $T_n^r$ is the time of the right head point of the $n$-th handover, where $T_n^r$ is  measured from the time of the left head point, and $H_n^r$ is the distance of the right head point of the $n$-th handover.
\begin{theorem}
The sequence $\left\{\left(H_n^l, T^r_n, H^r_n\right)\right\}_{n\in \Z}$ forms a homogeneous Markov chain on the state space $(\R^+)^3$. A representation of the transition kernel is given in (\ref{eq:tk-MC1a}) and the transition probability density is given in (\ref{eq:tpd}).
\label{theorem:markov}
\end{theorem}
\begin{proof}
Fix $n\in \Z$. Suppose the $n$-th handover is given by the intersection of the radial birds with head points $(0,H_n^l)$ and $(T_n^r,H_n^r)$, with $T^r_n\geq 0$. Then the intersection point $\left(\hat S_n,\hat H_n\right)$ of the radial birds at the two head points $(0,H_n^l)$ and $(T_n^r,H_n^r)$ is such that $\Hcal(U^{\hat S_n}_{\hat H_n})=0$. The possible head point responsible for the next handover must be on the right of $(T_n^r,H_n^r)$ and outside $\overline{U^{\hat S_n}_{\hat H_n}}$. Indeed, any radial bird with head point lying on the left of $(T^r_n,H^r_n)$ and outside $\overline{U^{\hat S_n}_{\hat H_n}}$ intersects $C_{(T^r_n,H^r_n)}$ before $\hat S_n$, so that this intersection point is not eligible to be considered as a future handover. We define the region $Q_{T_n^r}\setminus \overline{U^{\hat S_n}_{\hat H_n}}$ to be the {\em unexplored region}, where, $Q_s:=\{(t,h)\in \mathbb H^+\,\mbox{:}\, t\geq s\}$, is the quadrant on the right of the vertical line $t=s$ on $\mathbb H^+$. 

As explained in (\ref{eq:st}), one can construct a sequence of increasing open sets $\{S_t\}_{t\geq \hat S_n}$ in the unexplored region $Q_{T_n^r}\setminus \overline{U^{\hat S_n}_{\hat H_n}}$, where $S_t:=U^{t}_{h(t)}\setminus \overline{U^{\hat S_n}_{\hat H_n}}$ and $h(t):=\left((t-T_n^r)^2+(H_n^r)^2\right)^\half$. See Figure~\ref{figure:firstHoT} and Figure~\ref{figure:firstHoTx} for two different cases. Subsequently, by Corollary~\ref{corollary:third-pt}, there exists almost surely a unique head point, say $(T_u,H_u)$, with $T_u\geq T_n^r$, in the region $Q_{T_n^r}\setminus \overline{U^{\hat S_n}_{\hat H_n}}$, such that the next handover is given by the intersection $(\hat S_{n+1}, \hat H_{n+1})$ of the radial birds at $(T_n^r,H_n^r)$ and $(T_u,H_u)$, where $\hat H_{n+1}=h(\hat S_{n+1})$. Thus $S_{\hat S_{n+1}}:=U^{\hat S_{n+1}}_{\hat H_{n+1}}\setminus \overline{U^{\hat S_n}_{\hat H_n}}\subset Q_{T_n^r}\setminus \overline{U^{\hat S_n}_{\hat H_n}}$ is the stopping set until one finds the next head point $(T_u,H_u)$ at its boundary. We define the next triple of the chain as 
\begin{equation}
\left(H_{n+1}^l,T_{n+1}^r, H_{n+1}^r\right):= \left(H_n^r,T_u-T_n^r,H_u\right).
\label{eq:stepk+1}
\end{equation}
The Markov property of the sequence $\left\{\left(H_n^l, T^r_n, H^r_n\right)\right\}_{n\in \Z}$ follows from the fact that,  given the state $(H_n^l,T_n^r,H_n^r)$ at step $n$, the state of the process as defined in (\ref{eq:stepk+1}), in this unexplored region $Q_{T_n^r}\setminus \overline{U^{\hat S_n}_{\hat H_n}}$, the head point is conditionally independent of the past and Poisson. 

For the sake of completeness we give below a representation of the transition kernel of this Markov chain. Given $\left(H_{n}^l,T_{n}^r,H_{n}^r\right)=(h^l_n,t^r_n,h^r_n)$, let $C_{(0,h_n^l)}$ and $C_{(t_n^r,h_n^r)}$ be the radial birds with head points at $(0,h_n^l)$ and $(t_n^r,h_n^r)$. Also suppose the intersection point is $(\hat s_n,\hat h_n)$. For any $t\geq \hat s_n+\hat h_n$, there exists a unique upper half circle that has its top most point lying on $C_{(t^r_n,h^r_n)}$, its right most point at $(t,0)$, and is such that the half-circle passes through $(t^r_n,h^r_n)$.
Consider the family of upper half-circles $\{D_t\}_{t\geq \hat s_n+\hat h_n}$ satisfying the last properties and defined by the equation
\[
(u-\hat s(t))^2+h^2=\hat h(t)^2,
\]
in the $(u,h)$-coordinate system, where $(\hat s(t),\hat h(t))$ is the intersection point of the radial bird at $(t,0)$, with the radial bird $C_{(t^r_n,h^r_n)}$, see Figure~\ref{figure:D_t}.
\begin{figure}[ht!]
    \begin{tikzpicture}[scale=0.8, every node/.style={scale=0.75}]
    \pgftransformxscale{0.9}  
    \pgftransformyscale{0.9}    
    \draw[->] (-4, 0) -- (10, 0) node[right] {$t$};
    \draw[domain=-2:6, smooth] plot (\x, {(\x*\x-3.5*\x+1.75^2+2.15^2)^0.5});
    \draw[](2.5,-0.35) node{$(\hat s_n+\hat h_n,0)$};
    \draw[](1.75,2.15) node{$\bullet$};
    \draw[](2,3) node{$(t^r_n,h^r_n)$}; 
    \draw[]  (3,0) arc (0:180:2.5 and 2.5);
    \draw[violet]  (1,0) arc (-180:-360:3.5 and 3.5);
    \draw[blue]  (0.3,0) arc (-180:-360:2.35 and 2.35);
    \draw[red]  (-1.3,0) arc (-180:-360:2.3 and 2.3);
    \draw[->] (-1.55, 0) -- (-1.55, 4.3);
    \draw[] (-1.55, 4.6)node{$h$};
    \draw[] (0.5, 2.9)node{$(\hat s_n,\hat h_n)$};
     \draw[](-0.5,0.6) node{$U^{\hat s_n}_{\hat h_n}$};
     \draw[](-1.55,1.42)
     node{$\bullet$};
     \draw[](-2.8,1.2)
     node{$(0,h^l_n)$};
     \draw[domain=-4:2.8, smooth] plot (\x, {(\x*\x+3.1*\x+1.55^2+1.42^2)^0.5});
    \end{tikzpicture}
    \captionsetup{width=0.9\linewidth}
    \caption{The sequence of half-circles $\{D_t\}_{t\geq \hat s_n+\hat h_n}$. We just draw a few of them, for example in red, blue and violet, respectively, and each of which passes through $(t^r_n,h^r_n)$.}
    \label{figure:D_t}
\end{figure}

Let $U_t$ be the open upper half-ball contained below the half-circle $D_t$. Then we have
\begin{align}
\P\left((T_u,H_u)\notin  U_t\Big\vert\left(H_{n}^l, T_{n}^r, H_{n}^r\right)=(h_{n}^l, t_{n}^r, h_{n}^r) \right)= e^{-2\la\vert \overline{U_t}\setminus \overline{U^{\hat s_n}_{\hat h_n}}\vert}.
\end{align}
Given $(h^l_n,t^r_n,h^r_n)$, define the stopping time
\[
\tau_{(h^l_n,t^r_n,h^r_n)}:=\inf\left\{t\geq \hat s_n+\hat h_n: \Hcal\left(\overline{U_t}\setminus \overline{U^{\hat s_n}_{\hat h_n}}\right)>0\right\}.
\]
In short we write $\t_{(h^l_n,t^r_n,h^r_n)}\equiv \t$. By Corollary~\ref{corollary:third-pt}, $\t<\infty$, $\P$-almost surely. Hence there exists a unique head point at $(T_u,H_u)$ distributed on $D_{\t}\setminus \overline{U^{\hat s_n}_{\hat h_n}}$ with a certain law which we compute below. Before that, observe that, given $(h^l_n,t^r_n,h^r_n)$, the conditional probability density $f_\t$ of $\t$ is 
\begin{align}
f_{\t}(t)&=-\frac{\partial}{\partial t}\left[e^{-2\la\vert\overline{U_t}\setminus \overline{U^{\hat s_n}_{\hat h_n}}\vert}\right] 
=2\la e^{-2\la\vert\overline{U_t}\setminus \overline{U^{\hat s_n}_{\hat h_n}}\vert} \frac{\partial}{\partial t}\vert\overline{U_t}\setminus \overline{U^{\hat s_n}_{\hat h_n}}\vert.
\label{eq:cpdtau}
\end{align}
The quantity $\frac{\partial}{\partial t}\vert\overline{U_t}\setminus \overline{U^{\hat s_n}_{\hat h_n}}\vert$ is the rate at which the exploration region grows as a function of $t\geq \hat s_n+\hat h_n$. By the application of the increasing wing property, Remark~\ref{remark:inc-arm}, we have that $\frac{\partial}{\partial t}\vert\overline{U_t}\setminus \overline{U^{\hat s_n}_{\hat h_n}}\vert\neq 0$. 
\begin{figure}[ht!]
    \begin{tikzpicture}[scale=0.8, every node/.style={scale=0.75}]
    \pgftransformxscale{0.9}  
    \pgftransformyscale{0.9}    
    \draw[->] (-4, 0) -- (10, 0) node[right] {$t$};
    \draw[domain=-2:6, smooth] plot (\x, {(\x*\x-3.5*\x+1.75^2+2.15^2)^0.5});
    \draw[](2.9,-0.35) node{$\hat s_n+\hat h_n$};
    \draw[](1.75,2.15) node{$\bullet$};
    \draw[](2,3) node{$(t^r_n,h^r_n)$}; 
    \draw[]  (3,0) arc (0:180:2.5 and 2.5);
    \draw[]  (0.92,0) arc (-180:-360:3.15 and 3.15);
    \draw[](7.6,-0.35) node{$t+\Delta t$};
    \draw[blue]  (0.85,0) arc (-180:-360:3 and 3);
    \draw[](4,1) node{$U_t$};
    \draw[](7.7,1) node{$U_{t+\Delta t}$};
    \draw[](6.8,-0.35) node{$t$};
    \draw[](5.9,-0.35) node{$s'$};
    \draw[dashed] (5.9, 0) -- (5.9, 3.1);
    \draw[dashed] (5.2, 0) -- (5.2, 3.2);
    \draw[](5.2,-0.4) node{$s$};
    \draw[->] (-1.55, 0) -- (-1.55, 4.3);
    \draw[] (-1.55, 4.6)node{$h$};
    \draw[] (-1.55, -0.3)node{$0$};
    \draw[] (0.5, 2.9)node{$(\hat s_n,\hat h_n)$};
     \draw[](-0.5,0.6) node{$U^{\hat s_n}_{\hat h_n}$};
     \draw[](-1.55,1.42)
     node{$\bullet$};
     \draw[](-2.8,1.2)
     node{$(0,h^l_n)$};
     \draw[domain=-4:2.8, smooth] plot (\x, {(\x*\x+3.1*\x+1.55^2+1.42^2)^0.5});
     \draw[red](5.5,2.65) node{$\bullet$};
     \draw[red](5.5,3.4) node{$(u,h)$};
     \draw[red,domain=1.5:8.5, smooth] plot (\x, {(\x*\x-11*\x+5.5^2+2.65^2)^0.5});
    \end{tikzpicture}
    \captionsetup{width=0.9\linewidth}
    \caption{If $\t\in [t,t+\Delta t]$, the next head point $(u,h)$ lies in the region $\overline{U_{t+\Delta t}}\setminus  \overline{U_t}$.}
    \label{figure:delta-Ut}
\end{figure}

Note that the point $(T_u,H_u)$ is non-uniformly distributed on  $D_{\t}\setminus U^{\hat s_n}_{\hat h_n}$. Indeed, if $\t\in[t,t+\Delta t]$, for $t\geq \hat s_n+\hat h_n$ and an infinitesimal element $\Delta t$, the head point, say $(u,h)$, is uniformly distributed on the region $\left(\overline{U_{t+\Delta t}}\setminus \overline{U^{\hat s_n}_{\hat h_n}}\right)\setminus \left( \overline{U_t}\setminus \overline{U^{\hat s_n}_{\hat h_n}}\right)$ with the shape of a crescent. The region is nothing but $\overline{U_{t+\Delta t}}\setminus \overline{U_t}$, see Figure~\ref{figure:delta-Ut}. It is more likely for $(u,h)$ to lie on the thicker side of $\overline{U_{t+\Delta t}}\setminus \overline{U_t}$. Let $(\hat s,\hat h)$ and $(\hat s',\hat h')$ be the intersection points of $C_{(t_n^r,h^r_n)}$ with the radial birds at the right most points $(t,0)$ and $(t+\Delta t,0)$, respectively. These intersection points can be computed using (\ref{eq:handover-time}) and (\ref{eq:h}). The equation of the two half-circles $D_t$ and $D_{t+\Delta t}$ are given by  
\begin{equation}
(u-\hat s)^2+h^2=\hat h^2 \text{ and } (u-\hat s')^2+h^2=(\hat h')^2,
\label{eq:hcs}
\end{equation}
in $(u,h)$-coordinate system, where $(t,0)$ and $(t+\Delta t,0)$ are the right most point of the half-circles $D_t$ and $D_{t+\Delta t}$.

Let $X_t=(X_t^1,X_t^2)$ be a random variable uniformly distributed on the region $\overline{U_{t+\Delta t}}\setminus \overline{U_t}$, as a representative of the next head point $(T^r_{n+1}+t^r_n,H^r_{n+1})=(T_u,H_u)$, using (\ref{eq:stepk+1}). For $x\in [t^r_n,t+\Delta t]$, let $D_t(x)$ and $D_{t+\Delta t}(x)$ be the heights derived from the equation (\ref{eq:hcs}) of the half circles $D_t$ and $D_{t+\Delta t}$, respectively, expressed as 
\[
D_t(x)=\left(\hat h^2-(x-\hat s)^2\right)^\half \text{ and }D_{t+\Delta t}(x)=\left((\hat h')^2-(x-(\hat s'))^2\right)^\half.
\]
Let $s,s'\in [t^r_n,t+\Delta t]$ be such that $s<s'$. Then
\begin{align}
\P\left(X_t^1\in [s,s']\right) &=\int_s^{s'} \frac{1}{\vert \overline{U_{t+\Delta t}}\setminus \overline{U_t}\vert}\left(D_{t+\Delta t}(x)-D_t(x)\right)\, {\rm d}x \approx  \frac{\Delta t}{\vert \overline{U_{t+\Delta t}}\setminus \overline{U_t}\vert} \int_s^{s'} \frac{\partial}{\partial t}D_t(x)\, {\rm d}x. 
\label{eq:x-direction}
\end{align}
Taking the limit $\Delta t\to 0$ in (\ref{eq:x-direction}), we have $\lim_{\Delta t\to 0} \frac{\vert \overline{U_{t+\Delta t}}\setminus \overline{U_t}\vert}{\Delta t}=\frac{\partial}{\partial t}\vert\overline{U_t}\setminus \overline{U^{\hat s_n}_{\hat h_n}}\vert$. Since $\frac{\partial}{\partial t}\vert\overline{U_t}\setminus \overline{U^{\hat s_n}_{\hat h_n}}\vert \neq 0$, we get that the conditional probability density of $T^r_{n+1}+t^r_n$ at $x \in [t^r_n,t]$ is equal to $\left(\frac{\partial}{\partial t}\vert\overline{U_t}\setminus \overline{U^{\hat s_n}_{\hat h_n}}\vert\right)^{-1} \frac{\partial}{\partial t}D_t(x)$. 
We derive the expression for $\frac{\partial D_t(x)}{\partial t}$ in the following lemma and prove it subsequently.
\begin{lemma}
The growth rate of the height of half-circle $\partial U_t$ at abscissa $x$, for $x\in [t^r_n,t]$, is given by
\[
\frac{\partial D_t(x)}{\partial t}= \frac{x-t^r_n}{t-t^r_n}  \frac{\hat h}{D_t(x)}.
\]   
\label{lemma:growth}
\end{lemma}
\begin{remark}
In other words, the conditional probability density of $T^r_{n+1}+t^r_n$ at $x \in [t^r_n,t]$ is a uniform compensated or weighted by a factor $\left(\frac{\partial}{\partial t}\vert\overline{U_t}\setminus \overline{U^{\hat s_n}_{\hat h_n}}\vert\right)^{-1}\frac{\hat h}{D_t(x)}$ to account for the consequence of  closeness of point $x$ to $t$. 
\end{remark}
For all $t\geq t^r_n$, $\Acal, \Bcal\subset \R^+$, using (\ref{eq:x-direction}), (\ref{eq:stepk+1}) and the underlying time-stationarity, we hence have
\begin{align}
\lefteqn{\hspace{-0.9in}\P\left(T^r_{n+1}\in \Acal, H^r_{n+1}\in \Bcal \Big\vert\left(H_{n}^l, T_{n}^r, H_{n}^r\right)=(h_{n}^l, t_{n}^r, h_{n}^r), \t=t \right)}\nn\\
%
%
& = \left(\frac{\partial}{\partial t}\vert\overline{U_t}\setminus \overline{U^{\hat s_n}_{\hat h_n}}\vert\right)^{-1} \int_{\Acal+t^r_n \cap [t^r_n,t]}  \frac{\partial D_t(x)}{\partial t} \one_{\Bcal}(D_t(x))\, {\rm d}x\nn\\
& = \left(\frac{\partial}{\partial t}\vert\overline{U_t}\setminus \overline{U^{\hat s_n}_{\hat h_n}}\vert\right)^{-1} \int_{\Acal+t^r_n \cap [t^r_n,t]}  \frac{x-t^r_n}{t-t^r_n}  \frac{\hat h}{D_t(x)} \one_{\Bcal}(D_t(x))\, {\rm d}x,
\end{align}
where we have used Lemma~\ref{lemma:growth} in the last step. Here is then the announced representation of the transition kernel. For all $t\geq t^r_n$, $\Acal, \Bcal, \Ccal\subset \R^+$, 
%
%
\begin{align}
\lefteqn{\P\left(T^r_{n+1}\in \Acal, H^r_{n+1}\in \Bcal, H^l_{n+1}\in \Ccal\Big\vert\left(H_{n}^l, T_{n}^r, H_{n}^r\right)=(h_{n}^l, t_{n}^r, h_{n}^r) \right)}\nn\\
&= \delta_{h^r_n}(\Ccal)\times \int_{\hat s_n+\hat h_n}^\infty \left(\frac{\partial}{\partial t}\vert\overline{U_t}\setminus \overline{U^{\hat s_n}_{\hat h_n}}\vert\right)^{-1} \left( \int_{\Acal+t^r_n  \cap [t^r_n,t]}  \frac{x-t^r_n}{t-t^r_n} \one_{\Bcal}(D_t(x))\,{\rm d}x\right)f_{\t}(t)\, {\rm d}t\nn\\
&= \delta_{h^r_n}(\Ccal)\times \int_{\hat s_n+\hat h_n}^\infty  \left( \int_{\Acal+t^r_n  \cap [t^r_n,t]} \frac{x-t^r_n}{t-t^r_n} \one_{\Bcal}(D_t(x))\,{\rm d}x\right)2\la e^{-2\la\vert\overline{U_t}\setminus \overline{U^{\hat s_n}_{\hat h_n}}\vert}\,{\rm d}t,
\label{eq:tk-MC1a}
\end{align}
where $f_\t$ is the conditional probability density of $\t$, given $(h^l_n,t^r_n,h^r_n)$, as derived in (\ref{eq:cpdtau}).
Hence the transition probability density for a transition from $(h_{n}^l, t_{n}^r, h_{n}^r)$ to $(h^l_{n+1}, t^r_{n+1}, h^r_{n+1})$ is given by
\begin{align}
\lefteqn{P\left((h_{n}^l, t_{n}^r, h_{n}^r), (h^l_{n+1}, t^r_{n+1}, h^r_{n+1})\right)}\nn\\
&=\P\left((H^l_{n+1}, T^r_{n+1}, H^r_{n+1}) = (h^l_{n+1}, t^r_{n+1}, h^r_{n+1})\Big\vert (H_{n}^l, T_{n}^r, H_{n}^r)=(h_{n}^l, t_{n}^r, h_{n}^r) \right)\nn\\
&= \delta_{h^r_n-h^l_{n+1}}    \frac{t^r_n+t^r_{n+1}-t^r_n}{\hat s_{n+1}+\hat h_{n+1}-t^r_n}  \frac{\hat h_{n+1}}{h^r_{n+1}} \delta_{h^r_{n+1}-D_{\hat s_{n+1}+\hat h_{n+1}}(x)}\, 2\la e^{-2\la\vert \overline{U_{\hat h_{n+1}}^{\hat s_{n+1}}} \setminus \overline{U^{\hat s_n}_{\hat h_n}}\vert}\nn\\
&= \delta_{h^r_n-h^l_{n+1}}    \frac{t^r_{n+1}}{\hat s_{n+1}+\hat h_{n+1}-t^r_n}  \frac{\hat h_{n+1}}{h^r_{n+1}} \delta_{h^r_{n+1}-D_{\hat s_{n+1}+\hat h_{n+1}}(x)}\, 2\la e^{-2\la\vert \overline{U_{\hat h_{n+1}}^{\hat s_{n+1}}} \setminus \overline{U^{\hat s_n}_{\hat h_n}}\vert},
\label{eq:tpd}
\end{align}
where $(\hat s_{n+1}, \hat h_{n+1})$ is the intersection point of the radial birds at $(t^r_n,h^r_n)$ and $(t^r_{n+1}+t^r_n,h^r_{n+1})$.
This completes the proof of Theorem~\ref{theorem:markov}.
\end{proof}
\begin{proof}[Proof of Lemma~\ref{lemma:growth}]
Note that from the geometry of the upper half circle $\partial U^{\hat s}_{\hat h}$, that $\hat s=t-\hat h$, where $\hat h$ is a function of $\hat s$ and also of $(t^r_n,h^r_n)$ and $t$ as in (\ref{eq:h}), given by the intersection of the radial birds $C_{(t^r_n,h^r_n)}$
and $C_{(t,0)}$ as
\begin{align}
\hat h= \left((\hat s-t_n^r)^2+(h_n^r)^2\right)^\half \mbox{ and }\hat h&=  \frac{1}{2}\left[(t-t^r_n)+\frac{(h^r_n)^2}{(t-t^r_n)}\right].
\label{eq:hath}
\end{align}
The partial derivative of $D_t(x)$ with respect to $\hat s$ is equal to 
\begin{align}
\frac{\partial D_t(x)}{\partial t}= \frac{\partial  \hat s}{\partial t}
\frac{\partial D_t(x)}{\partial  \hat s} &= \left(1-\frac{\partial  \hat h}{\partial t}\right)\frac{\partial D_t(x)}{\partial  \hat s}, 
\label{eq:Dt1}
\end{align}
using $\hat s=t-\hat h$. The partial derivative of $\hat h$ from (\ref{eq:hath}) with respect to $t$ and $\hat s$ is
\begin{align}
\frac{\partial  \hat h}{\partial \hat s} =\frac{\hat s-t^r_n}{\hat h} \mbox{ and } \frac{\partial  \hat h}{\partial t}=\frac{1}{2}\left[1-\frac{(h^r_n)^2}{(t-t^r_n)^2}\right] = 1- \frac{\hat h}{t-t^r_n}.
\label{eq:pd-hath}
\end{align}
We write $D_t(x)$ as function of $(\hat s,\hat h)$ as
\[
D_t(x)=\left(\hat h^2-(x-\hat s)^2\right)^\half,
\]
whose partial derivative with respect to $\hat s$ is
\begin{align}
\frac{\partial D_t(x)}{\partial  \hat s} &= \frac{1}{D_t(x)} \left(\hat h \frac{\partial  \hat h}{\partial \hat s} +x-\hat s \right) = \frac{1}{D_t(x)} \left(\hat s-t^r_n +x-\hat s \right) = \frac{x-t^r_n}{D_t(x)},
\label{eq:Dt2}
\end{align}
using the expression for $\frac{\partial  \hat h}{\partial \hat s}$ from (\ref{eq:pd-hath}). Finally from (\ref{eq:Dt1}), (\ref{eq:pd-hath}) and (\ref{eq:Dt2}), the partial derivative of $D_t(x)$ with respect to $\hat s$ is equal to  
\begin{align}
\frac{\partial D_t(x)}{\partial t} &= \left(1-\frac{\partial  \hat h}{\partial t}\right)\frac{\partial D_t(x)}{\partial  \hat s} = \frac{\hat h}{t-t^r_n}  \frac{x-t^r_n}{D_t(x)} = \frac{x-t^r_n}{t-t^r_n}  \frac{\hat h}{D_t(x)}.\nn
\end{align}
\end{proof}
\begin{remark}
The sojourn time of the Markov chain $\left\{\left(H_n^l, T^r_n, H^r_n\right)\right\}_{n\in \Z}$ is nothing but the inter-step time of the Markov chain, as depicted in Figure~\ref{figure:BPP-MPN}. This Markov process enables us to explore structural features of the handover process, which we leave for future research.
\end{remark}
\begin{remark}[Structural results about the Markov chain]
One can think of the successive determination of handover epochs as an exploration of sequence of head points with increasing time coordinates. While doing so, at each step of the exploration, we encounter a finite region along with the already explored region on $\mathbb H^+$, which is empty of head points.  
\begin{figure}[ht!]
    \begin{tikzpicture}[line width=0.6pt,scale=0.8, every node/.style={scale=0.75}]
    \pgftransformxscale{1}  
    \pgftransformyscale{1}    
    \draw[->] (-4, 0) -- (8, 0) node[right] {$t$};
    \draw[blue,->] (-1.55, 0) -- (-1.55, 4.3)node[above]{$L_{t^l}$};
    \draw[] (0.5, 2.9)node{$(\hat s,\hat h)$};
     \draw[](-0.5,0.6) node{$U^{\hat s}_{\hat h}$};
     \draw[](-1.55,1.42)
     node{$\bullet$};
     \draw[](-2.8,1.2)
     node{$(t^l,h^l)$};
     \draw[domain=-4:2.8, smooth] plot (\x, {(\x*\x+3.1*\x+1.55^2+1.42^2)^0.5});
    \draw[domain=-2:5, smooth] plot (\x, {(\x*\x-3.5*\x+1.75^2+2.15^2)^0.5});
    \draw[](1.75,2.15) node{$\bullet$};
    \draw[](1.9,1.5) node{$(t^r,h^r)$}; 
    \draw[](4.85,0.7) node{$\bullet$};
    \draw[](5.6,0.5) node{$(t^u,h^u)$};
    \draw[red,domain=1:7, smooth] plot (\x, {(\x*\x-2*4.85*\x+4.85^2+0.7^2)^0.5});
    \draw[] (2.65, 2.9)node{$(\hat s',\hat h')$};
    \draw[blue,->] (1.75, 2.15) -- (1.75, 4.3) node[above]{$L_{t^r}$};
    \draw[blue,dashed] (1.75, 2.15) -- (1.75,0);
    \draw[]  (3,0) arc (0:180:2.5 and 2.5);
    \draw[blue]  (3,0) arc (0:60:2.5 and 2.5);
    \draw[]  (4.96,0) arc (0:180:2.34 and 2.34);
    \draw[blue]  (4.96,0) arc (0:20:2.34 and 2.34);
    \draw[blue,->] (4.85, 0.7) -- (4.85, 4.3)node[above]{$L_{t^u}$};
    \draw[blue,dashed] (4.85, 0.7) -- (4.85,0);
    \draw[](3.5,0.6) node{$U^{\hat s'}_{\hat h'}$};
    \end{tikzpicture}
    \captionsetup{width=0.9\linewidth}
    \caption{A realization of two consecutive handovers, where $(t^u,h^u)$ being very close to the time axis, implies $Q_{t^u}\setminus \overline{U^{\hat s'}_{\hat h'}}$ approximately equal to $Q_{t^u}$.}
    \label{figure:petite-set}
\end{figure}

Suppose $(t^l,h^l), (t^r,h^r)\in \mathbb H^+$ are two head points such that $t^l\leq t^r$ and they are responsible for a handover given by an intersection of the corresponding radial birds at $(\hat s,\hat h)$. The upcoming head point must lie in the unexplored region $Q_{t^r}\setminus \overline{U^{\hat s}_{\hat h}}$. Suppose the upcoming head point $(t^u,h^u)\in Q_{t^r}\setminus \overline{U^{\hat s}_{\hat h}}$ is very close to the time axis and the corresponding handover is given by the intersection point $(\hat s',\hat h')$ of the radial birds at $(t^r,h^r), (t^u,h^u)$.

Then the next unexplored region turns out to be $Q_{t^u}\setminus \overline{U^{\hat s'}_{\hat h'}}$. This region being approximately equal to $Q_{t^u}$, and one can think of the exploration starting fresh from the head point $(t^u,h^u)$ onward, leading to emergence of a renewal structure within this Markovian exploration process. 

This suggests that one can devise a coupling argument using a ``petite set'' or ``small set'' described in~\cite[Theorem~5.2.1]{Meyn-Tweedie}, to establish the renewal property of the Markov chain and also establish various structural results about it. The renewal structure of continuous space Markov chain was first observed in~\cite{Nummelin} and~\cite{Athreya-Ney} independently. This property can be used to establish, for example, the central limit theorem~\cite{Jones-CLT}, the functional central limit theorem~\cite{Dedecker-Rio-FCLT},~\cite{JOS-FCLT}, the law of large numbers~\cite{Jensen-Rahbek-LLN}, large deviations~\cite{Fortelle-LDP-MC},~\cite{Touchette-LDP} etc. The theory extends beyond Markov chains to growth dynamics in hidden Markov models~\cite{duchemin-HMC} and to Markov chain Monte Carlo algorithms~\cite{Andrieu-Moulines-MCMC},~\cite{Roberts-Rosenthal-MCMC}. We postpone this line of work for future research.
\label{remark:petit_set}
\end{remark}
%


\section{Handover frequency: two-speed case}\label{sec-multi_speed}
\subsection{Two-speed case} 
\label{subsection:two_speed}
As already explained, this part is motivated by communication networks comprising of, for example, low earth orbit (LEO) satellites  along with medium earth orbit (MEO) satellites, or LEO satellites with two different altitudes. The LEO and MEO satellites move at different speeds depending on their altitude. Naturally, as a first step towards understanding the handover phenomenon, we would like to consider a model on Euclidean plane with two different types of mobile stations, moving at speed $v_1$ and $v_2$, respectively, as described in Section~\ref{sec-Results}. 

Suppose the initial locations are independent homogeneous Poisson point processes $\tilde \Phi^1$ and $\tilde \Phi^2$ in polar coordinates on $\R^+\times (-\pi,\pi]$, with $\tilde{\Phi}^l=\sum_{i\in \N}\delta_{(R^{(l)}_i, \a^{(l)}_i)}$, for two different speeds $v_l$, with $l=1,2$. The intensity measures $\mu_1, \mu_2$ have for density 
\begin{equation}
{\rm d}\mu_1:= 2\pi r \la_1\,{\rm d}r  \otimes \frac{1}{2\pi}{\rm d}\a \text{ and } {\rm d}\mu_2:= 2\pi r \la_2\,{\rm d}r  \otimes \frac{1}{2\pi}{\rm d}\a,
\label{eq:mu_r12}
\end{equation}
respectively, where $\la_1$ and $\la_2$ are the intensities. In short we shall write $\tilde \Phi:=\tilde \Phi^1+\tilde \Phi^2$ and $\la:= \la_1+\la_2$. Essentially the model inherits all the properties of the single-speed model, except that the radial bird particle process consists of two types of particles. Let us denote the two types of radial bird particles by {\em type 1} and {\em type 2}, corresponding to the two speeds $v_1, v_2$, respectively. A key new fact is that the birds of different types can intersect each other twice or never almost surely, unlike the unique intersection property~\ref{Obs4} of the radial birds in the single-speed case. This fact will be more clear later in Lemma~\ref{lemma:2int}, Lemma~\ref{lemma:2int-rev} and Remark~\ref{remark:as20int}.

\subsubsection{Head point process and radial bird particle process}\label{subsubsection:HPPRBPP2} For any speed $v\in\{v_1, v_2\}$, the radial bird particle corresponding to an atom $(R,\a)$ of $\tilde{\Phi}=\tilde{\Phi}^1+\tilde{\Phi}^2$ is given by the {\em distance function} 
\begin{align}
f_v((R,\a), t)&:=\left(R^2+2vtR\cos \a +v^2t^2\right)^{\half}=\left(v^2\left(t+\frac{R}{v}\cos \a\right)^2+(R|\sin \a|)^2\right)^{\half},
\label{eq:v-bird}
\end{align}
whereas the head of the radial bird is given by $\left(-\frac{R}{v}\cos \a, R|\sin \a|\right)$, see Remark~\ref{remark:r1}.  For an atom $(R^{(l)}_i, \a^{(l)}_i)$ of $\tilde{\Phi}^l$, the head point is at
\[
\left(-\frac{R^{(l)}_i}{v_l}\cos \a^{(l)}_i, R^{(l)}_i|\sin \a^{(l)}_i|\right):= \left(T^{(l)}_i, H^{(l)}_i\right).
\]
Then the head point processes on $\mathbb H^+$ are $\Hcal^{l}:=\sum_{i\in \N}\delta_{\left(T^{(l)}_i, H^{(l)}_i\right)}$, for $l=1,2$. We have a result similar to Lemma~\ref{lemma:tips_density}: \label{notation:HPPl}
\begin{lemma}
For each $l=1, 2$, $\Hcal^{l}$ is a Poisson point process on $\mathbb{H}^+$ with intensity measure $\nu_l$, where ${\rm d}\nu_l:=v_l {\rm d}t \otimes 2\la_l {\rm d}h $, and $\Hcal^{1}$ and $\Hcal^{2}$ are independent.
\label{lemma:tips_density2}
\end{lemma}
\begin{proof}[Proof Lemma~\ref{lemma:tips_density2}]
The proof follows from constructing $\Hcal^{1}$ and $\Hcal^{2}$ separately from the Poisson point processes $\tilde{\Phi}^{1}$ and $\tilde{\Phi}^{2}$, respectively, in the same way as we have constructed $\Hcal$ from $\tilde{\Phi}$ in Lemma~\ref{lemma:tips_density} for the single-speed case. The  point processes $\Hcal^{1}$ and $\Hcal^{2}$   are independent because $\tilde{\Phi}^{1}$ and $\tilde{\Phi}^{2}$ are independent.
\end{proof}
We write the law of the head point process $\Hcal:=\Hcal^1+\Hcal^2$, as $\P_\Hcal:=\P_{\Hcal^1}\otimes \P_{\Hcal^2}$, which is the product law of $\P_{\Hcal^1}$ and $\P_{\Hcal^2}$, since $\Hcal^1$ and $\Hcal^2$ are independent. We also use the notation $\E_\Hcal$, $\E_{\Hcal^l}$ to denote the expectation under $\P_\Hcal$ and $\P_{\Hcal^l}$, respectively, for $l=1,2$. Observe that, in (\ref{eq:v-bird}), the speed factor only scales the time coordinate of the space $\mathbb{H}^+$, reflected only on the intensity measure $\nu_l$, for $l=1,2$. In this sense the speed $v\in \{v_1,v_2\}$ can be considered as the shape parameter of the radial birds, which are defined as 
\begin{equation}
C^{l}:=\left\{\left(t, f_{v_l}\big((R^{(l)},\a^{(l)}), t\big)\right):t\in \R\right\},
\label{eq:cb2}
\end{equation}
where $C^l$ corresponds to the atom $(R^{(l)},\a^{(l)})$ of $\tilde{\Phi}^l$, for $l=1,2$. The radial bird particle process is defined as
\[
\Pcal_{c}(v_1,v_2):=\sum_{l=1,2}\sum_{i\in \N}\delta_{C^l_i}. 
\]
We also denote by $C^l_i$ the radial bird with head point $(T^l_i,H^l_i)\in \Hcal^l$.\label{notation:RBCSl}\! In the following, we have the same properties for $\Pcal_{c}(v_1,v_2)$ as those explored in Lemma~\ref{lemma:bird_time_stationary} for the single-speed case.
\begin{lemma} The following holds true the radial bird particle process $\Pcal_c(v_1,v_2)$:
\begin{enumerate}[(i)]
\item \label{RBPP12}$\Pcal_c(v_1,v_2)$ is a Poisson particle process on the
space $\left(\Ccal(\mathbb{H}^+), \Bcal(\Ccal(\mathbb{H}^+))\right)$,
with support on the set of closed sets of the form (\ref{eq:cb2}).
\item \label{RBPP22} Both $\Pcal_c(v_1,v_2)$ and $\cup_{l=1,2}\cup_i C^l_i$ are stationary along the time coordinate.
\end{enumerate}
\label{lemma:bird_time_stationary2}
\end{lemma}
In view of Remark~\ref{remark:bird-to-mhead}, we can equivalently write the radial bird particle process $\Pcal_c(v_1,v_2)$ in terms of marked head point process $\Hcal_c(v_1, v_2):=\Hcal_c(v_1)+\Hcal_c(v_2)$, where
\[
\Hcal_c(v_l):=\sum_{i\in\N}\delta_{\left(\left(T^{(l)}_i, H^{(l)}_i\right), C^l_i\right)},
\]
for $l=1,2$. The time-stationarity of $\Hcal_c(v_1, v_2)$ is inherited from the same property of the head point process $\Hcal=\Hcal^1+\Hcal^2$, as seen in part~(\ref{RBPP22}) of Lemma~\ref{lemma:bird_time_stationary2}, for the two-speed case. 

\subsubsection{Joint lower envelope process}\label{subsubsection:LE2}The {\em joint lower envelope process} is a stochastic process defined as $\{L_{v_1, v_2}(t)\}_{t\in \R}$, where 
\[
L_{v_1, v_2}(t):= \min_{l=1,2}\inf_{i\in\N} f_{v_l}\left(\left(R^{(l)}_i,\a^{(l)}_i\right), t\right).
\]
The {\em joint lower envelope} of all the radial birds of the two types together, is a random closed subset of $\mathbb H^+$ depicted in Figure~\ref{figure:BPP-MPN2}, defined as \label{notation:Le2} 
\begin{equation}
\Lcal_e(v_1,v_2):=\left\{(t, L_{v_1, v_2}(t)): t\in \R\right\}.
\label{eq:LE2}
\end{equation}
Let us denote the open region above the lower envelope $\Lcal_e(v_1,v_2)$ as \label{notation:Le2+}
\begin{equation}
\Lcal^+_e(v_1,v_2):=\{(t,h): h>L_{v_1,v_2}(t)\}. \label{eq:Lcal2+}
\end{equation}
The following result about the time-stationarity of the joint lower envelope process can be proved in the same way as in Lemma~\ref{lemma:L_stationary}.
\begin{lemma}
The joint lower envelope process $\{(t, L_{v_1, v_2}(t))\}_{t\in \R}$ is stationary with respect to time.
\label{lemma:L_stationary2}
\end{lemma}
In the following, we describe the characterization for any point of a radial bird to be on the lower envelope, which is analogous to Lemma~\ref{lem:semicircle}. For any point $(s, u)\in \mathbb H^+$ and for any speed $v\in\{v_1, v_2\}$, the curve with respect to the $(t,h)$-coordinates associated with the equation 
\begin{equation}
v^2(t-s)^2+h^2= u^2
\label{eq:ellipse}
\end{equation}
is an ellipse containing $(s,u)$. Let $E^{s,v}_{u}$ be the open region enclosed by the upper half-ellipse given by (\ref{eq:ellipse}) and above the time axis. Let $\la= \la_1+\la_2$.
\begin{lemma}[Half-ellipse condition] For any point $(s,u)\in \cup_{l=1,2}\cup_{i\in \N}C^l_i$, the following holds true: 
\begin{enumerate}[(i)]
\item \label{semiellipse_1} The point $(s,u) \in \Lcal_e(v_1, v_2)$ if and only if $ \Hcal^{l}(E^{s, v_l}_{u})= 0$, for $l=1,2$.
\item \label{semiellipse_2} The probability of such an event is
\begin{equation}
\P\left((s,u) \in \Lcal_e(v_1, v_2)\right)= e^{-\la \pi u^2}. 
\end{equation}
\end{enumerate}
\label{lem:semiellipse2}
\end{lemma}
\begin{proof}[Proof of Lemma~\ref{lem:semiellipse2}]Part~(\ref{semiellipse_1}) can be proved using Lemma~\ref{lem:semicircle} and Discussion~\ref{discussion:ellipse_v} for the point processes $\Hcal^{1}$ and $\Hcal^{2}$ separately. Given two speeds $v_1,v_2$, there are two upper half ellipses $E^{s,v_1}_u $ and $ E^{s,v_2}_u$ which should be empty of points from $\Hcal^1 $ and $ \Hcal^2$, respectively. 
Part~(\ref{semiellipse_2}) can be proved using the independence of $\Hcal^{1}$ and $\Hcal^{2}$ as follows:
\begin{align}
\P\left((s,u) \in \Lcal_e(v_1, v_2)\right) &= \P_{\Hcal^1}\left( \Hcal^{1}(E^{s,v_1}_{u})=0\right) \P_{\Hcal^2}\left( \Hcal^{2}(E^{s, v_2}_{u})=0\right)= e^{-2\la_1 v_1 \frac{\pi u^2}{2v_1}} e^{-2\la_2 v_2 \frac{\pi u^2}{2v_2}}=e^{-\la \pi u^2}.\nn
\end{align}
where $\la_1+\la_2=\la$.
\end{proof}
Before moving to the construction of the handover point process and deriving its intensity for the two-speed case, we present some geometric results, determining the condition on the location of heads for which there are two intersections between the two birds. For a pair of birds of different types, we will refer to the following results, Lemma~\ref{lemma:2int} and Lemma~\ref{lemma:2int-rev}, as the {\em intersection criterion}. For this we suppose, without loss of generality, that $v_1>v_2$.
\subsubsection{Intersection criterion}\label{subsubsection:IC}
Suppose there are tow radial birds $C_1$ and $C_2$ of two different types, with heads at $(t_1,h_1)$ and $(t_2,h_2)$, respectively. Unlike in the single-speed case, a pair of radial birds with different speeds  intersect each other twice or never almost surely. There exists a pair $(r_1, \a_1)$ and $(r_2, \a_2)$ of points of the initial marked Poisson point process $\tilde{\Phi}$ of the stations, such that 
\begin{equation}
(t_1, h_1)=\left(-\frac{r_1}{v_1}\cos \a_1, r_1|\sin\a_1|\right) \text{ and } (t_2, h_2)=\left(-\frac{r_2}{v_2}\cos \a_2, r_2|\sin\a_2|\right).
\label{eq:t1t2c}
\end{equation}
For finding the time-coordinate of the intersections, we solve the quadratic equation
\[
v_1^2s^2+2sv_1r_1\cos\a_1+r_1^2= v_2^2s^2+2sv_2r_2\cos\a_2+r_2^2,
\]
that is,
\begin{align}
\hspace{-0.4in} (v_1^2-v_2^2)s^2-2s(v_1^2t_1-v_2^2t_2)+r_1^2-r_2^2&=0.
\label{eq:handover-time-v}
\end{align}
The time coordinates of intersections are given by the solutions to the quadratic equation (\ref{eq:handover-time-v}), which are
\begin{equation}
\begin{aligned}
\hat s_1\equiv \hat s_1(t_1, t_2, h_1, h_2):= \frac{1}{v_1^2-v_2^2}\left[v_1^2t_1-v_2^2t_2-\sqrt{\Delta}\right],\\
 \hat s_2\equiv  \hat s_2(t_1, t_2, h_1, h_2):= \frac{1}{v_1^2-v_2^2}\left[v_1^2t_1-v_2^2t_2+\sqrt{\Delta}\right],
\end{aligned}
\label{eq:HOs}
\end{equation}
where
\begin{equation}
\Delta := \left(v_1^2t_1-v_2^2t_2\right)^2\!- \!(r_1^2-r_2^2)(v_1^2-v_2^2) = v_1^2v_2^2\left(t_1-t_2\right)^2\!-\! (h_1^2-h_2^2)(v_1^2-v_2^2).
\label{eq:Delta}
\end{equation}
The last simplification is due to the fact that, the head point $(t_i,h_i)$ corresponding to speed $v_i$ satisfies relation  $r_i^2=v_i^2t_i^2+h_i^2$, for $i=1, 2$, which can be derived from (\ref{eq:t1t2c}). In principle, the determinant $\Delta$ is a function of $|t_1-t_2|, h_1,h_2$ and we consider $\Delta\equiv \Delta(|t_1-t_2|, h_1, h_2)$. Let $t^*$ be the non-negative solution of $\Delta(t,h_1,h_2)=0$ in the variable $t$, as
\begin{equation}
t^*\equiv t^*(h_1,h_2)=\frac{1}{v_1v_2}(h_1^2-h_2^2)^\half (v_1^2-v_2^2)^\half,
\label{eq:tstar}
\end{equation}
given $h_1,h_2\geq 0$. We have the following results about describing the intersection criterion for different values of $\Delta$ depending on the variables $t_1,h_1,t_2,h_2$, and the proof of which is given in Appendix~\ref{subsection:Lemma-2int}.    
\begin{lemma}[Intersection criterion]
Let  $t_1\in \R $ and $ h_1, h_2\in \R^+$ be such that there is a bird at $(t_1,h_1)$ of speed $v_1$. Then there exist two sets $D_\ell\equiv D_\ell(t_1,h_1,h_2), D_r\equiv D_r(t_1,h_1,h_2)\subsetneq \R$, such that, for all $t_2\in (D_\ell \cup D_r)^\circ$ or $\partial(D_\ell \cup D_r)$ or $(D_\ell \cup D_r)^c$, the radial bird of speed $v_2$ and head at $(t_2,h_2)$ intersects the first bird, twice, once or never, respectively. The sets $D_\ell, D_r$ are given by,
\begin{equation}
\begin{aligned}
D_\ell&=\begin{cases}
           (-\infty,t_1], & \text{ if } h_2\geq  h_1,\\
           (-\infty, t_1-t^*], &\text{ if } h_2<h_1,
    \end{cases}
&\hspace{0.3in}
D_r&=\begin{cases}
           [t_1,\infty), & \text{ if } h_2\geq h_1,\\
           [t_1+t^*,\infty), &\text{ if } h_2<h_1,
    \end{cases}
\end{aligned}
\label{eq:DlDr}
\end{equation}
where $t^*$ as defined in (\ref{eq:tstar}). 
\label{lemma:2int}
\end{lemma}
\begin{remark}
Given $t_1\in \R $ and $ h_1, h_2\in \R^+$, suppose there is a bird of speed $v_1$ at $(t_1,h_1)$. In the case $h_1\leq h_2$, the bird of speed $v_2$ at $(t_2,h_2)$ intersects the first one exactly once if $t_2=t_1$ and $h_2=h_1$. In the other case, namely $h_1> h_2$, the bird at $(t_2,h_2)$ with speed $v_2$ intersects the first one exactly once if $t_2=t_1-t^*$ or $t_1+t^*$. Hence there exist almost surely two intersections or no intersection between birds of different types. 
\label{remark:as20int}
\end{remark}
\begin{remark}[Symmetry]
    Suppose two radial birds with heads located at $(t_1,h_1)$ and $(t_2, h_2)$, of two different types intersect. Then, $t_2\in D_\ell(t_1,h_1,h_2)$ if and only if $-t_2\in D_r(-t_1,h_1,h_2)$.
\end{remark}
We also have an {\em intersection criterion} for the dual scenario as follows, which can be proved similarly.
\begin{lemma}[Intersection criterion]
Let  $t_2\in \R $ and $ h_1, h_2\in \R^+$ be such that there is a bird at $(t_2,h_2)$ of speed $v_2$. Then there exist two sets $D'_\ell\equiv D'_\ell(t_1,h_1,h_2), D'_r\equiv D'_r(t_1,h_1,h_2)\subsetneq \R$, such that, for all $t_1\in (D'_\ell \cup D'_r)^\circ$ or $\partial(D'_\ell \cup D'_r)$ or $(D'_\ell \cup D'_r)^c$, the radial bird of speed $v_1$ and head at $(t_1,h_1)$ intersects the first bird, twice, once or never, respectively. The sets $D'_\ell, D'_r$ are given by,
\begin{equation}
\begin{aligned}
D'_\ell&=\begin{cases}
           (-\infty,t_2], & \text{ if } h_1\leq  h_2,\\
           (-\infty, t_2-t^*], &\text{ if } h_1>h_2,
    \end{cases}
&\hspace{0.3in}
D'_r&=\begin{cases}
           [t_2,\infty), & \text{ if } h_1\leq h_2,\\
           [t_2+t^*,\infty), &\text{ if } h_1>h_2,
    \end{cases}
\end{aligned}
\label{eq:DlDr-rev}
\end{equation}
where $t^*$ as defined in (\ref{eq:tstar}). 
\label{lemma:2int-rev}
\end{lemma}
\subsubsection{Geometric version of the intersection criteria}\label{subsubsection:geo-int} Suppose there are two radial birds of type $1$ and $2$, with their heads at $(t_1,h_1)$ and $(t_2,h_2)$, respectively. Keeping the head of one of them fixed, we can find the trajectory of the other one, such that the two radial birds just touch each other. This leads to a geometric version of the intersection criterion that we describe now together with some applications.
\begin{enumerate}[I.]
\setlength{\leftskip}{0pt}
\item \label{one-two} \textbf{Fixed $(t_1,h_1)$.}
For a given radial bird of type $1$ at $(t_1,h_1)$, consider the collection of points 
\[
\tilde C_{(t_1,h_1)}:=\{(t,h)\in \mathbb H^+\,\mbox{:}\, t=t_1-t^* \text{ or }t=t_1+t^*\},
\]
where $t^*=\frac{1}{v_1v_2}(h_1^2-h^2)^\half(v_1^2-v_2^2)^\half$, given $(t_1,h_1)$, as defined in (\ref{eq:tstar}). Any radial bird with head point on $\tilde C_{(t_1,h_1)}$ and of type $2$ intersects $C^1_{(t_1,h_1)}$ exactly once; in other words, they just touch each other. Geometrically, $\tilde C_{(t_1,h_1)}$ is given by the curve governed by the equation $t-t_1=\pm t^* $, i.e., 
\begin{align}
(t-t_1)^2 = \frac{1}{v_1^2v_2^2}(h_1^2-h^2)(v_1^2-v_2^2), \text{ i.e., }\frac{v_1^2v_2^2}{v_1^2-v_2^2}(t-t_1)^2+ h^2=h_1^2,
\label{eq:2bird-elp}
\end{align}
which is an ellipse in $\mathbb H^+$, see picture~\subref{subfigure:ellp_hyp1} of Figure~\ref{figure:ellp_hyp}. In short, the head point, say $(t_2,h_2)$, of the radial bird of type $2$ must lie outside the open region in $\mathbb H^+$ enclosed by the ellipse given by (\ref{eq:2bird-elp}), to have guaranteed intersection with $C^1_{(t_1,h_1)}$, as seen in Lemma~\ref{lemma:2int}. 
\begin{figure}[ht!]
 \centering
\begin{subfigure}[t]{0.45\textwidth}
\centering
\begin{tikzpicture}[scale=0.6, every node/.style={scale=0.7}]
\pgftransformxscale{1.2}  
\pgftransformyscale{1.2}    
\draw[->] (-1.5, 0) -- (5, 0) node[right] {$t$};
\draw[blue,domain=-1:3.9, smooth] plot (\x, {((3/2)^2*\x*\x- (3/2)^2*2*\x*(1.09)+(3/2)^2*(1.09)^2+2^2)^0.5});
    \draw[blue](1.1,2) node{$\bullet$};
    \draw[blue,domain=-1:2.6, smooth] plot (\x, {((3/2)^2*\x*\x- (3/2)^2*2*\x*(0.58)+(3/2)^2*(0.58)^2+1^2)^0.5});
    \draw[blue](0.58,1) node{$\bullet$};
    \draw[blue,domain=1:5, smooth] plot (\x, {((3/2)^2*\x*\x- (3/2)^2*2*\x*(3.27)+(3/2)^2*(3.27)^2+1.4^2)^0.5});
    \draw[blue](3.27,1.4) node{$\bullet$};
    \draw[red, domain=1:3, smooth] plot (\x, {(16*\x*\x-64*\x+64+2.5^2)^0.5});
    \draw[red](2,2.5) node{$\bullet$};
    \draw[](1.2,2.5) node{$(t_1,h_1)$};
    \draw[teal]  (0.45,0) arc (-180:-360:1.545 and 2.5);
    \draw[](0.4,4.9) node{$C^1_{(t_1,h_1)}$};
    \draw[](4.2,0.4) node{$\tilde C_{(t_1,h_1)}$};
    \end{tikzpicture}
   \captionsetup{width=1\linewidth} 
    \caption{For fixed $(t_1, h_1)$ of type $1$, the trajectory of the head points of the radial birds that touch $C^1_{(t_1, h_1)}$ forms an ellipse (in color teal) given by (\ref{eq:2bird-elp}).}
\label{subfigure:ellp_hyp1}
    \end{subfigure}
    \hspace{0.15in}
    \centering
\begin{subfigure}[t]{0.45\textwidth}
\centering
\begin{tikzpicture}[scale=0.6, every node/.style={scale=0.7}]
\pgftransformxscale{1.2}  
\pgftransformyscale{1.2}    
\draw[->] (-2, 0) -- (4.5, 0) node[right] {$t$};
\draw[teal,domain=-2:4.3, smooth] plot (\x, {(2.6182*\x*\x- 2.6182*2*\x*(1.1)+2.6182*(1.1)^2+2^2)^0.5});
\draw[blue,domain=-1.7:4, smooth] plot (\x, {((3/2)^2*\x*\x- (3/2)^2*2*\x*(1.1)+(3/2)^2*(1.1)^2+2^2)^0.5});
\draw[blue](1.1,2) node{$\bullet$};
\draw[](1.1,1.65) node{$(t_2,h_2)$};
\draw[red, domain=1:3, smooth] plot (\x, {(16*\x*\x-64*\x+64+2.5^2)^0.5});
\draw[red](2,2.5) node{$\bullet$};
\draw[red, domain=-0.7:1, smooth] plot (\x, {(16*\x*\x-16*2*0.2*\x+16*0.2*0.2+2.5^2)^0.5});
\draw[red](0.2,2.5) node{$\bullet$};
\draw[red, domain=2.3:3.9, smooth] plot (\x, {(16*\x*\x-16*2*3*\x+16*9+3.7^2)^0.5});
\draw[red](3,3.7) node{$\bullet$};
\draw[red, domain=-1.7:-0.3, smooth] plot (\x, {(16*\x*\x+16*2*0.8*\x+16*0.8*0.8+3.7^2)^0.5});
\draw[red](-0.8,3.7) node{$\bullet$};
\draw[](5,5.4) node{$\tilde C_{(t_2,h_2)}$};
\draw[](4.5,4.5) node{$C^2_{(t_2,h_2)}$};
\end{tikzpicture}
\captionsetup{width=1\linewidth} 
\caption{Given $(t_2, h_2)$ of type $2$, the trajectory of the head of the radial birds of type $1$, that touch $C^2_{(t_2,h_2)}$ forms a hyperbola (in color teal) given by (\ref{eq:2bird-hyp}).}
\label{subfigure:ellp_hyp2}
    \end{subfigure}
    \caption{Geometric interpretation of intersection criterion.}   \label{figure:ellp_hyp}
\end{figure} 
\item \label{hyp}
\textbf{Fixed $(t_2,h_2)$.}
For a fixed radial bird of type $2$ at $(t_2,h_2)$, similarly define 
\[
\tilde C_{(t_2,h_2)}:=\{(t,h)\in \mathbb H^+ \,\mbox{:}\, t=t_2-t^* \text{ or }t=t_2+t^*\},
\]
where $t^*=\frac{1}{v_1v_2}(h^2-h_2^2)^\half(v_1^2-v_2^2)^\half$, as defined in (\ref{eq:tstar}). Any radial bird with head point on $\tilde C_{(t_2,h_2)}$ and of type $1$ intersects $C^2_{(t_2,h_2)}$ exactly once, or equivalently, these two birds just touch each other. Here also, the set $\tilde C_{(t_2,h_2)}$ is characterised by the curve given by $t-t_2=\pm t^*$, i.e., 
\begin{align}
(t-t_2)^2 = \frac{1}{v_1^2v_2^2}(h^2-h_2^2)(v_1^2-v_2^2),
\text{ i.e., }h^2-\frac{v_1^2v_2^2}{v_1^2-v_2^2}(t-t_2)^2=h_2^2,  
\label{eq:2bird-hyp}
\end{align}
which is a branch of a hyperbola on $\mathbb H^+$, see picture~\subref{subfigure:ellp_hyp2} of Figure~\ref{figure:ellp_hyp}. This curve can also be considered as a radial bird with head at $(t_2,h_2)$, with a different speed $v_1v_2/{(v_1^2-v_2^2)^\half}$. There is guaranteed intersection between $C^2_{(t_2,h_2)}$ and the radial bird of type $1$ with head at, say $(t_1,h_1)$, if $(t_1,h_1)$ lies below the curve of the hyperbola given by (\ref{eq:2bird-hyp}), as seen in the intersection criterion of Lemma~\ref{lemma:2int-rev}.
\end{enumerate}
\vspace{0.1in}
\begin{remark}[Bijection between $C^2_{(t_2,h_2)}$ and $\tilde C_{(t_2,h_2)}$]\label{remark:two-one}  Given a radial bird of type $2$ at $(t_2,h_2)$, the set $\tilde C_{(t_2,h_2)}$ is constructed in such a way that, for each point $(\hat s,\hat h)\in  C^2_{(t_2,h_2)}$, there exists a unique point, say $(t_1,h_1)\in \tilde C_{(t_2,h_2)}$, such that there is a radial bird of type $1$ with head point at $(t_1,h_1)$, which touches the bird $C^2_{(t_2,h_2)}$ at $(\hat s,\hat h)$. In this sense, there exists a natural bijection, say $ \gamma_{(t_2,h_2)}:C^2_{(t_2,h_2)} \to \tilde C_{(t_2,h_2)}$. Indeed, given $(t_2,h_2)$, if we take $(\hat s,\hat h)\in C^2_{(t_2,h_2)}$, then there exists a unique point $(t_1,h_1)$, such that $\hat s$ satisfies (\ref{eq:HOs}) and the pair $(\hat s,\hat h)$ satisfies (\ref{eq:2bird-hyp}), i.e., 
\begin{equation}
\hat s=\frac{1}{v_1^2-v_2^2}(v_1^2t_1-v_2^2 t_2) \text{ and } \hat h^2 =\frac{v_1^2v_2^4}{(v_1^2-v_2^2)^2} (t_1-t_2)^2+h_1^2,
\end{equation}
see Figure~\ref{figure:bij-area}. Then $(t_1,h_1)$ is given by 
\begin{equation}
t_1=\frac{1}{v_1^2}\left[\hat s(v_1^2-v_2^2)+ v_2^2t_2\right],
\label{eq:t-1}
\end{equation}
and 
\begin{align}
h_1^2&=\hat h^2-\frac{v_1^2v_2^4}{(v_1^2-v_2^2)^2}\left(\frac{1}{v_1^2}\left[\hat s(v_1^2-v_2^2)+ v_2^2t_2\right]-t_2\right)^2 
=\hat h^2-\frac{v_2^4}{v_1^2}(\hat s-t_2)^2.
\label{eq:h-1}
\end{align}
Given $(t_2,h_2)\in \mathbb H^+$, $\gamma_{(t_2,h_2)}((\hat s, \hat h))= (t_1,h_1)$, which is determined by (\ref{eq:t-1}) and (\ref{eq:h-1}).
\end{remark}
\begin{figure}[ht!]
\begin{tikzpicture}[scale=0.7, every node/.style={scale=0.6}]
\pgftransformxscale{1.4}  
\pgftransformyscale{1.4}    
\draw[->] (0, 0) -- (5.5, 0) node[right] {$t$};
\draw[teal,domain=-0.6:4.2, smooth] plot (\x, {(2.6182*\x*\x- 2.6182*2*\x*(1.1)+2.6182*(1.1)^2+1^2)^0.5});
\draw[blue,domain=-0.3:4, smooth] plot (\x, {((3/2)^2*\x*\x- (3/2)^2*2*\x*(1.1)+(3/2)^2*(1.1)^2+1^2)^0.5});
\draw[blue](1.1,1) node{$\bullet$};
\draw[](1.1,1.65) node{$(t_2,h_2)$};
\draw[red, domain=1.5:3.5, smooth] plot (\x, {(16*\x*\x-16*2*2.5*\x+16*2.5*2.5+2.5^2)^0.5});
\draw[red](2.5,2.5) node{$\bullet$};
\draw[](3.5,2.65) node{$(\hat s,\hat h)$};
\draw[](1.7,2.5) node{$\gamma_{(t_2,h_2)}(\hat s,\hat h)$};
\draw[red, domain=2.42:4.2, smooth] plot (\x, {(16*\x*\x-16*2*3.3*\x+16*3.3*3.3+3.7^2)^0.5});
\draw[red](3.3,3.7) node{$\bullet$};
\draw[](4,3.7) node{$(\hat s_1,\hat h_1)$};
\draw[](2.3,3.7) node{$\gamma_{(t_2,h_2)}(\hat s_1,\hat h_1)$};
\draw[red]  (2,0) arc (-180:-360:0.85 and 2.7);
\draw[red](2.55,0) arc (-180:-360:1.259 and 4); 
\draw[->](2.4,3.2)-- (2.8,3);
\draw[](2.3,3.3)node{$S$};
\draw[](-1,3) node{$\tilde C_{(t_2,h_2)}$};
\draw[](-0.7,2.3) node{$C^2_{(t_2,h_2)}$};
\draw[](1.6,3) node{$C^1_{(t_1,h_1)}$};
\end{tikzpicture}
\captionsetup{width=0.93\linewidth}
\caption{(a). Under the bijection $\gamma_{(t_2,h_2)}$, the point $(\hat s,\hat h)$ is mapped to a point $(t_1,h_1)$, such that the radial bird $C^1_{(t_1,h_1)}$, if it exists, touches $C^2_{(t_2,h_2)}$ at $(\hat s,\hat h)$. 
(b). Note that for $(\hat s,\hat h)$ and $(\hat s_1,\hat h_1)$ to represent two consecutive handovers, it is necessary that along with $E^{\hat s,v_1}_{\hat h}\cup E^{\hat s_1,v_1}_{\hat h_1}$, the extra  region $S$, outside $E^{\hat s,v_1}_{\hat h}\cup E^{\hat s_1,v_1}_{\hat h_1}$, but below the hyperbola $\tilde C_{(t_2,h_2)}$ between the two points $\gamma_{(t_2,h_2)}(\hat s,\hat h)$ and $\gamma_{(t_2,h_2)}(\hat s_1,\hat h_1)$, be empty of head points of type $1$.}
\label{figure:bij-area}
\end{figure}
\begin{remark}[Region below $\tilde C_{(t_2,h_2)}$]
For any $(\hat s,\hat h)\in C^2_{(t_2,h_2)}$, the upper half ellipse enclosed by the curve $v_1^2(t-\hat s)^2+h^2=\hat h^2$ lies completely below the hyperbola branch given by (\ref{eq:2bird-hyp}). Define a sequence $\{\hat h_t\}_{t\in \R}$, where $\hat h_t:=\left(v_2^2(t-t_2)^2+h_2^2\right)^\half$, corresponding to the height of the radial bird $C^2_{(t_2,h_2)}$ at different time $t$.  Consider $(\hat s_1,\hat h_1)\in C^2_{(t_2,h_2)}$ such that $\hat s_1\geq \hat s$. Then, using the bijection $\gamma_{(t_2,h_2)}$, we have the image of the points $(\hat s,\hat h), (\hat s_1,\hat h_1)$, expressed as in (\ref{eq:t-1}), (\ref{eq:h-1}), lying on the curve $\tilde C_{(t_2,h_2)}$. 
Then we have 
\[
\bigcup_{t\in [\hat s,\hat s_1]}E^{t,v_1}_{\hat h_t}= E^{\hat s,v_1}_{\hat h}\cup S \cup E^{\hat s_1,v_1}_{\hat h_1},
\]
see Figure~\ref{figure:bij-area}, where $S$ is the region outside $E^{\hat s,v_1}_{\hat h}\cup E^{\hat s_1,v_1}_{\hat h_1}$ but below the hyperbola (\ref{eq:2bird-hyp}) within the interval \[
\left[\frac{1}{v_1^2}\left(\hat s(v_1^2-v_2^2)+ v_2^2t_2\right), \frac{1}{v_1^2}\left(\hat s_1(v_1^2-v_2^2)+ v_2^2t_2\right)\right].
\]
\label{remark:union}
\end{remark}
\subsubsection{Some more geometric results about mixed intersections} \label{subsubsection:geo-int2} Here, we present some more geometric properties of intersection of radial birds of different types (called as mixed intersections) and the corresponding handovers.
\begin{observation}
Suppose there are two radial birds with their heads at $(t_1,h_1), (t_2,h_2)$, with speed $v_1$ and $v_2$, respectively, such that $v_1>v_2$. Also suppose that the pair $(t_1,h_1), (t_2,h_2)$ satisfies one of the intersection criteria and they intersect each other at $(\hat s_1, \hat h_1)$ and  $(\hat s_2, \hat h_2)$, such that $\hat s_1\leq \hat s_2$. Then the following holds:
\begin{enumerate}[(i).]
    \item $\hat s_1\leq t_1\leq \hat s_2$;
    \item  $t_2\leq \hat s_1\leq t_1$, if $t_2\leq t_1$ and $t_1\leq \hat s_2\leq t_2$, if $t_1\leq t_2$.
\end{enumerate}
\label{observation:sts}
\end{observation}
In the following, we present a property that guarantees the intersection of birds of different types. The proof is provided in Appendix~\ref{subsection:L-in1out2}
\begin{lemma}
Let $(s,u)\in \mathbb H^+$ and let $(t_1,h_1), (t_2,h_2)\in \mathbb H^+$ be such that the radial bird at $(t_1,h_1)$ is of type $1$ and satisfies $(t_1,h_1)\in \partial E^{s,v_1}_u$, but $(t_2,h_2)\notin E^{s,v_2}_u$, see Figure~\ref{figure:1in2out}. Then the following holds:
\begin{enumerate}[(i).]
    \item \label{intersection} The radial bird of type $2$ intersects that of type $1$ twice almost surely.
    \item  \label{intersection_order} For $(\hat s_1, \hat h_1)$ and $(\hat s_2, \hat h_2)$ being the first and second intersections of the two radial birds, respectively, we have $\hat s_1\leq s\leq \hat s_2$.
\end{enumerate} 
\label{lemma:in1out2}
\end{lemma}
We present another property of mixed intersections which will be useful when deriving the Palm probability distribution with respect to handovers of different types. The proof of which is postponed to Appendix~\ref{subsection:L-int1mid2}.
\begin{lemma}
Suppose the radial birds at $(t_1,h_1), (t_2,h_2)\in \mathbb H^+$, of type $1$ and $2$, respectively, intersect at $(\hat s_1, \hat h_1)$ and $(\hat s_2, \hat h_2)$, where $\hat s_1\leq  \hat s_2$ and $\Hcal^l(E^{\hat s_1,v_l}_{\hat h_1})=0$, for $l=1,2$, so that $(\hat s_1,\hat h_1)$ is a handover. Also suppose there exists a radial bird of type $l$ at $(t_0,h_0)$, that intersects the bird of type $1$ at $(\tilde s,\tilde h)$, if $l=1$, and at $(\tilde s_1,\tilde h_1),(\tilde s_2,\tilde h_2)$, if $l=2$. Then the following holds:
\begin{enumerate}[(i).]
    \item For $l=1$, $\hat s_1\leq \tilde s\leq \hat s_2$ if and only if $(t_0,h_0)\in E^{\hat s_2,v_1}_{\hat h_2}\setminus E^{\hat s_1,v_1}_{\hat h_1}$.
    \item For $l=2$, $\hat s_1\leq \tilde s_2\leq \hat s_2$ if and only if $(t_0,h_0)\in E^{\hat s_2,v_2}_{\hat h_2}\setminus E^{\hat s_1,v_2}_{\hat h_1}$.
\end{enumerate}
\label{lemma:int1mid2}
\end{lemma}
\subsubsection{Distance of intersections to the time axis} Given $t_1, t_2, h_1, h_2$, suppose $\hat s_1,  \hat s_2$ in (\ref{eq:HOs}) exist and are solutions to (\ref{eq:handover-time-v}). Then the distances $\hat h_1\equiv \hat h_1(t_1, t_2, h_1, h_2) $ and  $\hat h_2\equiv\hat h_2(t_1, t_2, h_1, h_2)$, are computed from the distance function in (\ref{eq:v-bird}), as
\begin{equation}
\begin{aligned}
\hat h_1^2:= v_1^2 \hat s_1^2-2v_1^2 \hat s_1 t_1+ r_1^2 =v_2^2 \hat s_1^2-2v_2^2 \hat s_1 t_2+ r_2^2,\\
\hat h_2^2:= v_1^2 \hat s_2^2-2v_1^2 \hat s_2 t_1+ r_1^2= v_2^2 \hat s_2^2-2v_2^2 \hat s_2 t_2+ r_2^2.
\end{aligned}
\label{eq:HO_distance}
\end{equation}
where $r_i^2=v_i^2t_i^2+h_i^2$, for $i=1, 2$. Using the expressions of $\hat s_1, \hat s_2$ from (\ref{eq:HOs}) yields
\begin{equation}
\begin{aligned}
\hat h_1^2= v_1^2 \left(\frac{v_2^2(t_1-t_2)- \sqrt{\Delta}}{v_1^2-v_2^2}\right)^2+h_1^2,\qquad 
\hat h_2^2=v_1^2 \left(\frac{v_2^2(t_1-t_2)+ \sqrt{\Delta}}{v_1^2-v_2^2}\right)^2+h_1^2,
\end{aligned}
\label{eq:HO_distance1}
\end{equation}
where $\Delta$ as in (\ref{eq:Delta}). The distances $\hat h_1,\hat h_2$ are indeed functions of the difference $t_1-t_2$ and $h_1,h_2$, so we write $\hat h_k\equiv\hat h_k(t_1- t_2, h_1, h_2)$, for $k=1,2$.
\begin{observation}
A key observation is that, for fixed $v_1>v_2$ and $h_2>h_1$, the intersection distances computed in (\ref{eq:HO_distance1}) are symmetric functions of $t_1, t_2$, i.e., they remain invariant if we take $t_1> t_2$ or $t_1\leq t_2$, see Figure~\ref{figure:sym_birds}. In the case $t_1\leq t_2$, we have a different solution $\tilde{s}_1, \tilde{s}_2$ to (\ref{eq:handover-time-v}) with $\tilde{s}_2\leq t_2\leq \tilde{s}_1$,
\begin{equation}
\tilde{s}_1:= \frac{1}{v_1^2-v_2^2}\left[v_2^2t_2-v_1^2t_1-\sqrt{\Delta}\right],\qquad 
\tilde{s}_2:= \frac{1}{v_1^2-v_2^2}\left[v_2^2t_2-v_1^2t_1+\sqrt{\Delta}\right].
\label{eq:HOs-sym}
\end{equation}
It can be verified that the distances of intersections are $\tilde{h}_1(t_2-t_1, h_1, h_2)\equiv\tilde{h}_1= \ell_1$ and $\tilde{h}_1(t_2-t_1, h_1, h_2)\equiv \tilde{h}_2=\ell_2$, expressed differently as 
\begin{equation}
\begin{aligned}
\tilde{h}_1^2:= v_1^2 \tilde{s}_1^2-2v_1^2 \tilde{s}_1 t_1+ r_1^2 =v_2^2 \tilde{s}_1^2-2v_2^2 \tilde{s}_1 t_2+ r_2^2= \hat h_1^2,\\
\tilde{h}_2^2:= v_1^2 \tilde{s}_2^2-2v_1^2 \tilde{s}_2 t_1+ r_1^2= v_2^2 \tilde{s}_2^2-2v_2^2 \tilde{s}_2 t_2+ r_2^2= \hat h_2^2.
\end{aligned}
\label{eq:HO_distance_sym}
\end{equation}
\begin{figure}[ht!]
    \centering
\begin{tikzpicture}[scale=0.6, every node/.style={scale=0.7}]
\pgftransformxscale{1}  
    \pgftransformyscale{1}    
    \draw[->] (-2, 0) -- (7, 0) node[right] {$t$};
    \draw[blue, dashed] (-1, 2) -- (5, 2);
    \draw[<-] (-1, 1.95) -- (-1, 1.3)node[below]{$h_2$};
    \draw[<-] (-1, 0.05) -- (-1, 0.9);
     \draw[<-] (-1.8, 2.45) -- (-1.8, 1.6)node[below]{$h_1$};
    \draw[<-] (-1.8, 0.05) -- (-1.8, 1.25);
   \draw[blue, domain=-2:3.8, smooth] plot (\x, {((3/2)^2*\x*\x+2^2)^0.5});
    \draw[blue](0,2) node{$\bullet$};
    \draw[](0.2,1.8) node{$(t_2,h_2)$};
   \draw[blue, dashed, domain=0.2:6, smooth] plot (\x, {((3/2)^2*\x*\x- 2*4*\x*(3/2)^2+4*4*(3/2)^2+2^2)^0.5});
   \draw[blue](4,2) node{$\bullet$};
   \draw[](4.2,1.8) node{$(t'_2,h_2)$};
    %
    \draw[red, domain=0.5:3.5, smooth] plot (\x, {(16*\x*\x-64*\x+64+2.5^2)^0.5});
    \draw[red](2,2.5) node{$\bullet$};
    \draw[red, dashed] (-2, 2.5) -- (6, 2.5);
    \draw[](2.3,2.3) node{$(t_1,h_1)$};
    \draw[](1,3.2) node{$(\hat s_1,\hat h_1)$};
    \draw[](3.6,5) node{$(\hat s_2,\hat h_2)$};
    \draw[](3,3.2) node{$(\tilde{s}_1,\hat h_1)$};
    \draw[](1.6,5) node{$(\tilde{s}_2,\hat h_2)$};
    \end{tikzpicture}
    \caption{The handover distances $\hat h_1, \hat h_2$ are symmetric functions of the difference $t_1-t_2$, when $v_1>v_2$.}
\label{figure:sym_birds}
\end{figure}

So the key idea is that, in the two scenarios $t_1>t_2$ and $t_1\leq t_2$, the intersection times are different but the intersection distances remain the same.
\label{observation:equalHOdist}
\end{observation}
\subsubsection{Construction of the handover point process}\label{subsubsection:constV} In the following, we formally describe the construction of the handover point process $\Vcal$ in the two-speed case. As seen before in the single-speed case, handovers are given by intersections of two radial birds satisfying certain empty half-ball or half-ellipse condition. By Lemma~\ref{lem:semiellipse2}, for the two-speed case, an intersection point $(\hat s, \hat h)$ between two radial birds at $(t,h), (t',h')$ give rise to a handover if and only if the region $E_{\hat h}^{\hat s, v_l}$ has no point of $\Hcal^l$, for $l=1,2$.

In order to avoid over counting, we introduce some ordering for the pair of head points. With this ordering in mind, we can decompose the handover point process depending on the type of intersection. There are in total six types of intersections happening between all possible radial birds, as we will see later. Two of them are of {\em pure type} and four of them are of {\em mixed type}, defined below.
\vspace{0.15in}
\begin{definition}[Pure handovers]
For $l=1,2$, an intersection is of type $l$ if both radial birds are of type $l$. We call these intersections \textbf{pure intersections} and handovers given by them as \textbf{pure handovers}. 
\label{definition:pure-Ho}
\end{definition}
We denote the corresponding handover point process as $\Vcal_l$. \label{notation:Vcall}\!\! Formally, as in (\ref{eq:VcalHPP2}) for the single-speed case, the pure handover point process is defined as
\begin{equation} \Vcal_l:=\sum_{(T^l_i,H^l_i)\in \Hcal^l}\;\;\sum_{(T^l_j,H^l_j)\in \Hcal^l \,\mbox{:}\, T^l_j <  T^l_i} \delta_{\hat S}\, \one_{A(\hat S, \hat H)},
\label{eq:VcalHPPll}
\end{equation}
for $l=1,2$, where   $(\hat S, \hat H)$ is the intersection point between the radial birds with their heads at $(T^l_i,H^l_i) $ and $(T^l_j,H^l_j)\in \Hcal^l$ such that $T^l_j <  T^l_i$ and $A(\hat S, \hat H)$ is the event that the intersection point  $(\hat S, \hat H)$ exists and is responsible for a handover. In other words, $\Vcal_l$ is the point process of abscissas of the intersection points of any two radial bird of type $l$, which give rise to a handover.

Suppose $l\neq r\in \{1,2\}$. As seen in the {\em intersection criterion}, described in Lemma~\ref{lemma:2int} and Lemma~\ref{lemma:2int-rev}, there are either $2$ intersections or $0$ intersection almost surely, between any two radial birds of type $l$ and $r$.
\vspace{0.15in}
\begin{definition}[{\em Mixed handovers}] For $l,r,k\in \{1,2\}$ with $l\neq r$, we say that an intersection is of type $\binom{k}{l,r}$, if it is the $k$-th intersection between the radial birds of type $l$ and $r$. \label{notation:binom}\!\!\! In the combinatorial like notation $\binom{k}{l,r}$, the fact that the position of $l$ is on the \textbf{left} means that the head of the bird of type $l$ is on the left of the head of the bird of type $r$, which is on the \textbf{right}. The two values of $k$ refer to the two intersections, as in (\ref{eq:HOs}), of radial birds of two different types, where $k=1$ means the left intersection and $k=2$ means the right intersection. We call intersections of these type \textbf{mixed intersections} and the handovers given by these intersections as \textbf{mixed handovers}. 
\label{definition:mixed-Ho}
\end{definition}
We denote the point processes corresponding to the mixed handovers as $\Vcal^{(k)}_{l,r}$.\label{notation:Vcallrk}\! Formally, for $k\in \{1,2\}$,
\begin{equation} \Vcal_{2,1}^{(k)}:=\sum_{(T^1_i,H^1_i)\in \Hcal^1}\;\;\sum_{(T^2_j,H^2_j)\in \Hcal^2 \,\mbox{:}\, T^2_j \in D_\ell(T^1_i, H^1_i, H^2_j)} \delta_{\hat S^k}\, \one_{A(\hat S^k, \hat H^k)},
\label{eq:VcalHPPlmk}
\end{equation}
and 
\begin{equation} \Vcal_{1,2}^{(k)}:=\sum_{(T^2_i,H^2_i)\in \Hcal^2}\;\;\sum_{(T^1_j,H^1_j)\in \Hcal^1 \,\mbox{:}\, T^1_j \in D'_\ell(T^2_i, H^2_i, H^1_j)} \delta_{\hat S^k}\, \one_{A(\hat S^k, \hat H^k)},
\label{eq:VcalHPPlmks}
\end{equation}
where $(\hat S^k, \hat H^k)$ is the $k$-th intersection of the radial birds with their heads at $(T^1_i,H^1_i)\in \Hcal^1$ and $ (T^2_j,H^2_j)\in \Hcal^2$ such that $ T^2_j \in D_\ell(T^1_i, H^1_i, H^2_j)$ in the first case and heads at $(T^2_i,H^2_i)\in \Hcal^2 $ and $ (T^1_j,H^1_j)\in \Hcal^1 $ such that  $ T^1_j \in D'_\ell(T^2_i, H^2_i, H^1_j)$ in the second case, respectively. The set  $D_\ell(T^1_i, H^1_i, H^2_j)$ given  $T^1_i, H^1_i, H^2_j$ and $D'_\ell(T^2_i, H^2_i, H^1_j)$ given  $T^2_i, H^2_i, H^1_j$ is defined in Lemma~\ref{lemma:2int} and Lemma~\ref{lemma:2int-rev}, respectively, for the intersection criterion. Let $A(\hat S^k, \hat H^k)$ be the event that the intersection point $(\hat S^k, \hat H^k)$ is responsible for a handover. 

Equivalently, $\Vcal^{(k)}_{2,1}$ and  $\Vcal^{(k)}_{2,1}$ are the point processes of abscissas of the intersections which give rise to handovers of types $\binom{k}{2,1}$ and $\binom{k}{1,2}$, respectively, for $k=1,2$, as shown in Figure~\ref{figure:mixedHOP}. All the point processes $\Vcal_l$ for $l\in \{1,2\}$ and  $\Vcal^{(k)}_{l,r}$ for $l, r, k\in \{1,2\}$ with $l\neq r$, inherit the time-stationarity property of the head point process $\Hcal$. In total the handover point process in the two-speed case is 
\begin{equation}
\Vcal:=\sum_{l=1,2}\Vcal_l+\sum_{l,r, k=1,2\,\mbox{:}\, l\neq r}\Vcal_{l,r}^{(k)}.
\label{eq:VcalHPP}
\end{equation}\label{notation:Vcal3}
\begin{figure}[ht!]
    \centering
\begin{tikzpicture}[scale=0.6, every node/.style={scale=0.7}]
\pgftransformxscale{0.85}  
    \pgftransformyscale{0.85}    
    \draw[->] (-4, 0) -- (4, 0) node[right] {$t$};
   \draw[blue, domain=-3.8:3.8, smooth] plot (\x, {((3/2)^2*\x*\x+2^2)^0.5});
    \draw[blue](0,2) node{$\bullet$};
    \draw[red, domain=-3.5:-0.5, smooth] plot (\x, {(16*\x*\x+64*\x+64+1^2)^0.5});
    \draw[red](-2,1) node{$\bullet$};
    \draw[red, domain=0.5:3.5, smooth] plot (\x, {(16*\x*\x-64*\x+64+2.5^2)^0.5});
    \draw[red](2,2.5) node{$\bullet$};
    \draw[](1,3.2) node{$\binom{1}{2,1}$};
    \draw[](3.6,5) node{$\binom{2}{2,1}$};
    \draw[](-0.8,3.1) node{$\binom{2}{1,2}$};
    \draw[](-3.7,5) node{$\binom{1}{1,2}$};
    \end{tikzpicture}
    \captionsetup{width=0.9\linewidth}
    \caption{The labels at the intersections are the corresponding types, where the red and blue birds are of type $1$ and $2$ with speeds $v_1>v_2$.}
\label{figure:mixedHOP}
\end{figure}

In the following, we state the result about the time-stationarity of the handover point process $\Vcal$, which is inherited from the stationarity of the head point process $\Hcal$.
\begin{lemma}[Time-stationarity of $\Vcal$]
In the two-speed case, the handover point process $\Vcal$ is stationary with respect to time.
\label{lemma:stationary_Tcal2}
\end{lemma}
Let $\Rcal_l$ be the point processes on $\R$ made of the abscissas of the right most head points corresponding to pure handovers, for $l\in \{1,2\}$. \label{notation:Rcall}\!\!\! Let $\Rcal^{(k)}_{l,r}$ be the point process made of the abscissas of the right most point of a pair of heads corresponding to a mixed handover, for $l\neq r, k\in \{1,2\}$ and the right-most among the two heads, is of type $r$ and left-most is of type $l$.\label{notation:Rcallrk}\! The time-stationarity of the point processes $\Rcal_l$, for $l\in \{1,2\}$, and $\Rcal^{(k)}_{l,r}$, for $l\neq r, k\in \{1,2\}$, is also inherited from that of $\Hcal$.

Based on the bijection for the single-speed case in Lemma~\ref{lemma:VR-MTP}, in the following, we also have the same bijection among the handover point processes
and the point processes corresponding to the right-most head points. This is crucial for deriving mass transport principle extending Lemma~\ref{lemma:VR-MTP}.
\begin{lemma}
\begin{enumerate}[(i).]
    \item \label{bijection-i}For all $l\in \{1,2\}$, there exists a bijection $\beta_l$ between $\Vcal_l$ and $\Rcal_l$.
    \item \label{bijection-ijk}  For all $l\neq r, k\in \{1,2\}$, there exists a bijection $\beta^{(k)}_{l,r}$ between $\Vcal^{(k)}_{l,r}$ and $\Rcal^{(k)}_{l,r}$.
\end{enumerate}
\label{lemma:vr}
\end{lemma} 
Let $\La_l$ denote the intensity of $\Vcal_l$ for  $l\in\{1,2\}$.\label{notation:Lal}\!\!  We also denote the intensities of $\Vcal^{(k)}_{l,r}$ by $\La^{(k)}_{l,r}$ for $l\neq r, k\in \{1,2\}$.\label{notation:Lalrk}\! We obtain the equality of intensities as a corollary:
\begin{corollary}
\begin{enumerate}[(i).]
    \item For $l=1,2$, the intensities of the point processes $\Vcal_l$ and $ \Rcal_l$ are equal: $\la_{\Rcal_l}=\la_{\Vcal_l}:= \La_l$.
     \item   For all $l\neq r $ and $ k\in \{1,2\}$, the intensities of the point processes $\Vcal_{l,r}^{(k)}$ and $ \Rcal_{l,r}^{(k)}$ are equal: $\la_{\Rcal_{l,r}^{(k)}}=\la_{\Vcal_{l,r}^{(k)}}:=\La_{l,r}^{(k)}$.
\end{enumerate} 
\label{corollary:eqVR2}
\end{corollary}
\subsubsection{Handover frequency}\label{subsubsection:HOPP2f}
Suppose there are two types of stations with speed $v_1$ and $v_2$, respectively. Without loss of generality, assume that $v_1>v_2$. The total handover frequency is~\label{notation:La}
\[
\la_\Vcal= \sum_{l\in \{1,2\}}\La_l+\sum_{l,r, k=1,2\,\mbox{:}\, l\neq r}\La^{(k)}_{l,r}.
\]
We now present the result about the handover frequencies in the two-speed case. We use the same technique as in the proof of Theorem~\ref{theorem:handover_freq1} to obtain the pure handover frequency, whereas the mixed handover frequency requires a slightly different approach. Nevertheless, one can analyze both pure and mixed handovers, only using the head point processes $\Hcal^1, \Hcal^2$.
\begin{theorem}[Handover frequency: two-speed case]
Suppose there are two types of stations with intensities $\la_1, \la_2$ and speeds $v_1> v_2$. Let $\la_1+\la_2=\la$. Then we have:
\begin{enumerate}[(i).]
\item \label{MH2p1} The pure handover frequencies are 
\begin{equation}
     \Lambda_l= \frac{4v_l\sqrt{\la}}{\pi}\left(\frac{\la_l}{\la}\right)^2, \text{ for }l=1,2.
\end{equation}
\item \label{MH2p2}  The mixed handover frequencies satisfies the relations $\Lambda^{(1)}_{1,2}=\Lambda^{(2)}_{2,1}$ and $\Lambda^{(2)}_{1,2}=\Lambda^{(1)}_{2,1}$.
\item \label{MH2p3} For $k=1,2$,
\begin{align}
\Lambda^{(k)}_{2,1}&= 4\la_1\la_2v_1v_2\int_{0}^{\infty}\left[\int_0^{h_1}\int_{t^*}^\infty e^{-\la\pi\hat h^2_k(t, h_1, h_2)}  \,{\rm d}t \, {\rm d}h_2  +\int_{h_1}^{\infty}\int_{0}^{\infty} e^{-\la\pi\hat h^2_k(t, h_1, h_2)}\,{\rm d}t \, {\rm d}h_2 \right] {\rm d}h_1 ,
\end{align}
\end{enumerate}
where $\hat h^2_k$ as in (\ref{eq:HO_distance1}), being a function of $t, h_1, h_2$.
\label{theorem:HO2speed-main}
\end{theorem}
\begin{proof}[Proof of Theorem~\ref{theorem:HO2speed-main}~part~(\ref{MH2p1})]For handovers of the pure type, choose a speed, say $v_1$ in $\{v_1, v_2\}$. We have the contribution from the pure handovers which is similar to the one derived for the single-speed case as in Theorem~\ref{theorem:handover_freq1}. By applying the multivariate Campbell Mecke formula for the factorial power of order 2 of $\Hcal^1$, we obtain
\begin{align}
\Lambda_1&=(2v_1\la_1)^2 \int_{0}^{1}  \int_{0}^{\infty} \int_{-\infty}^{t_1} \int_{0}^{\infty} \E^{1,(t_1, h_1), (t_2, h_2)}_{\Hcal}\left[\one_{A(s,h)} \left(\Hcal^1+\Hcal^2\right)\right] \, {\rm d}h_2 \, {\rm d}t_2 \, {\rm d}h_1 \, {\rm d}t_1 \nn\\
&=4v_1^2\la_1^2 \int_{0}^{1}  \int_{0}^{\infty} \int_{-\infty}^{t_1} \int_{0}^{\infty}\E^1_{\Hcal}\left[\one_{A(s,h)} \left(\Hcal^1+\delta_{(t_1, h_1)}+\delta_{(t_2, h_2)}+\Hcal^2\right)\right] \, {\rm d}h_2 \, {\rm d}t_2 \, {\rm d}h_1 \, {\rm d}t_1 ,
%
%
\label{eq:h1b11}
\end{align}
where $A(s,h)$ is the event that there is a handover, at time $s$ and at a distance $h$, between the stations corresponding to the radial birds with heads located at $(t_1, h_1)$ and $(t_2, h_2)$.

Throughout this article, we make the following notational convention, which simply means that a head point at $(t,h)$ will naturally inherit the information about the type, whenever the point mass $\delta_{(t,h)}$ is written on the right of $\Hcal^1$ or $\Hcal^2$. For $i\neq j\in \{1,2\}$,
\begin{enumerate}[(1).]
\item $\Hcal^i {+} \delta_{(t_1, h_1)} {+} \delta_{(t_2, h_2)} {+} \Hcal^j$ is the two-point Palm version of $\Hcal^i{+}\Hcal^j$, with $(t_1,h_1), (t_2,h_2)$ of type $i$. \label{notation:2PHl}
\item   $\Hcal^i+\delta_{(t_1, h_1)}+\Hcal^j+\delta_{(t_2, h_2)}$ is the two-point Palm version of $\Hcal^i+\Hcal^j$, with $(t_1,h_1), (t_2,h_2)$ of type $i$ and $j$, respectively. \label{notation:2PHlm}
\item $\Hcal^i+\delta_{(t_1, h_1)}+\delta_{(t_2, h_2)}+\Hcal^j= \Hcal^j+\Hcal^i+\delta_{(t_1, h_1)}+\delta_{(t_2, h_2)}$  and $\Hcal^i+\delta_{(t_1, h_1)}+\Hcal^j+\delta_{(t_2, h_2)}= \Hcal^j+\delta_{(t_2, h_2)}+\Hcal^i+\delta_{(t_1, h_1)}$.
\end{enumerate}
This convention applies to multi-speed case as well. In (\ref{eq:h1b11}), the expectation $\E^1_\Hcal$ is under the probability measure \label{notation:2PHlP}
\begin{equation}
\P^{1, (t_1, h_1), (t_2,h_2)}_{\Hcal}:=\P^{(t_1, h_1), (t_2,h_2)}_{\Hcal^1}\otimes \P_{\Hcal^2},
\label{eq:2P-Palm1}
\end{equation}
which is the two-point Palm probability measure on $\Hcal$, conditioned on the two points $(t_1,h_1), (t_2,h_2)$ in $\Hcal^1$.
%
We can find the coordinate $(s,h)$ of the intersection point of two radial birds in terms of their heads at $(t_1, h_1)$ and $(t_2, h_2)$, similarly to (\ref{eq:handover-time}) and (\ref{eq:h}). By Lemma~\ref{lem:semiellipse2}, a handover happens at $(s,h)$ if and only if the regions $E^{s,v_1}_{h}, E^{s,v_2}_{h}$ have no point of $\Hcal^{1}, \Hcal^{2}$ respectively. Since $\Hcal^1$ is a Poisson point process, by the Slivnyak-Mecke theorem, the probability of such an event $A(s,h)$, in (\ref{eq:h1b11}), is
\begin{align}
\E^{1}_{\Hcal}\left[\one_{A(s,h)} \left(\Hcal^{1}+\delta_{(t_1, h_1)}+\delta_{(t_2, h_2)}+\Hcal^2\right)\right]
&=\P_{\Hcal^{1}}\left(\Hcal^{1}(E^{s,v_1}_{h})=0\right)\P_{\Hcal^{2}}\left( \Hcal^{2}(E^{s,v_2}_{h})=0\right)\nn\\
&= e^{-2\la_1 v_1\pi h^2/{2v_1} } e^{-2\la_2 v_2\pi h^2/{2v_2} }=e^{-\la\pi h^2},
\label{eq:void1}
\end{align}
where $\la=\la_1+\la_2$. The open regions $E^{s,v_1}_{h}$ and $E^{s,v_2}_{h}$ should have no head points of their respective types, for the point $(s, h)$ to be a handover point.

Then the frequency of the handovers is obtained from the proof of Theorem~\ref{theorem:handover_freq1} as
\begin{align}
\Lambda_1&= 4\la_1^2 v_1^2\int_{0}^{1} \int_{0}^{\infty}\int_{0}^{\infty}\int_{-\infty}^{t_1} e^{-\la \pi h^2} {\rm d}t_2 \, {\rm d}h_2 \, {\rm d}h_1 \, {\rm d}t_1 \nn\\
&= \left(\frac{\la_1}{\la}\right)^2\times 4\la^2 v_1^2\int_{0}^{1} \int_{0}^{\infty}\int_{0}^{\infty}\int_{-\infty}^{t_1} e^{-\la \pi h^2} {\rm d}t_2 \, {\rm d}h_2 \, {\rm d}h_1 \, {\rm d}t_1 
%
%
%
%
%
=\left(\frac{\la_1}{\la}\right)^2\times  \frac{4v_1\sqrt{\la}}{\pi},
\label{eq:Lam1}
\end{align}
since the multiple integral in (\ref{eq:Lam1}) along with the factor $4v_1^2\la^2$, equals $\frac{4v_1\sqrt{\la}}{\pi}$, which is the handover frequency in the single-speed case with speed $v_1$, see Remark~\ref{remark:vspeed}. Similarly we can prove that 
\begin{equation}
      \Lambda_{2}= \frac{4v_2\sqrt{\la}}{\pi} \left(\frac{\la_2}{\la}\right)^2,
    \label{eq:lam1-2}
\end{equation}
by performing the same computations as above for $\La_1$, using the two-point Palm probability measure $\P^{2, (t_1, h_1), (t_2,h_2)}_{\Hcal}:=\P_{\Hcal^1}\otimes \P^{(t_1, h_1), (t_2,h_2)}_{\Hcal^2}$.
\end{proof}
\begin{proof}[Proof of Theorem~\ref{theorem:HO2speed-main}~part~(\ref{MH2p2})]Let us first prove that $\La^{(1)}_{1,2}=\La^{(2)}_{2,1}$. The radial birds corresponding to speeds $v_1$ and $v_2$, with the assumption $v_1>v_2$, are referred to as {\em fast radial birds} and {\em slow radial birds}, respectively.

Let $(T^1_i, H^1_i), (T^2_j, H^2_j)$ be the locations of the heads of the radial birds $C_1, C_2$ corresponding to two stations moving at speeds $v_1, v_2$, respectively. Just to make sure intersection happens between $C_1$ and $C_2$, we assume the following. Given the triple $(T^2_j, H^2_j, H^1_i)$,  $T_i^1\in D'_\ell(T^2_j, H^2_j, H^1_i)$, when $T^1_i<T^2_j$ or $T_i^1\in D'_r(T^2_j, H^2_j, H^1_i)$, when $T^1_i>T^2_j$.  On the other hand, given the triple $(T^1_i, H^1_i, H^2_j)$, $T_j^2\in D_\ell(T^1_i, H^1_i, H^2_j)$ when $T^2_j<T^1_i$, or $T_j^2\in D_r(T^1_i, H^1_i, H^2_j)$ when $T^2_j>T^1_i$. The pairs of sets $(D_\ell, D_r)$ and $(D'_\ell, D'_r)$ are defined as in (\ref{eq:DlDr}) and (\ref{eq:DlDr-rev}), respectively. We say that $C_1$ lies on the left (resp. right) of $C_2$ if $T_i^1\leq T_j^2$ (resp. $T_j^2< T_i^1$). Given the pair $(T^1_i, H^1_i), (T^2_j, H^2_j)$ satisfying the intersection criteria (Lemma~\ref{lemma:2int}), consider the events $A^{(k)}_{1,2}$ and $A^{(k)}_{2,1}$, for $k\in \{1,2\}$ defined by 
\begin{align}
A^{(k)}_{1,2}:=\left\{\text{handover at } (S^k_{i,j},H^k_{i,j}) \text{ : } T^1_i\leq T^2_j\right\}, 
A^{(k)}_{2,1}:=\left\{\text{handover at } (S^k_{j,i},H^k_{j,i}) \text{ : } T^2_j\leq T^1_i\right\},
\label{eq:A12k}
\end{align}
where for $k\in \{1,2\}$, $(S^k_{i,j},H^k_{i,j})$,   and $(S^k_{j,i}, H^k_{j,i})$, are the pairs of intersection points of the birds $C_1, C_2$, when $T^1_i< T^2_j$  and $T^1_i\geq  T^2_j$, respectively, which are given by (\ref{eq:HOs}). In words, $A^{(k)}_{1,2}$ is the event that there is a handover of type $\binom{k}{1,2}$, given by the $k$-th intersection of $C_1, C_2$, with the fast radial bird on the left of the slow one, and $A^{(k)}_{2,1}$ is the event that there is a handover of type $\binom{k}{2,1}$, i.e., a handover due to the $k$-th intersection of $C_1, C_2$, with the slow radial bird being on the left of the fast one. By definition, the mixed handover frequencies are 
\begin{eqnarray}
    \Lambda^{(1)}_{1,2}=\E\Big[\sum_{(T^2_j,H^2_j)\in \Hcal^2 \,\mbox{:}\, T^2_j\in[0,1]}\,\, \sum_{(T^1_i,H^1_i)\in \Hcal^1 \,\mbox{:}\, T^1_i\in D'_\ell(T^2_j, H^2_j, H^1_i)} \one_{A^{(1)}_{1,2}}\Big],  
\label{eq:h1va1}
\end{eqnarray}
\begin{eqnarray}
\Lambda^{(2)}_{2,1}= \E\Big[\sum_{(T^1_i,H^1_i)\in \Hcal^1 \,\mbox{:}\, T^1_i\in[0,1]}\,\, \sum_{(T^2_j,H^2_j)\in \Hcal^2 \,\mbox{:}\, T^2_j\in D_\ell(T^1_i, H^1_i, H^2_j)} \one_{A^{(2)}_{2,1}}\Big]. 
\label{eq:h1vb2}
\end{eqnarray}
where the sets $D_\ell$ and $D'_\ell$ are as defined in (\ref{eq:DlDr}) and (\ref{eq:DlDr-rev}), respectively. Alternatively, when counting backward in time, due to the symmetry seen in Figure~\ref{figure:sym_birdss}, it is also true that
\begin{eqnarray}
\Lambda^{(1)}_{1,2}=\E\Big[\sum_{(T^2_j,H^2_j)\in \Hcal^2 \,\mbox{:}\, T^2_j\in[0,1]}\,\, \sum_{(T^1_i,H^1_i)\in \Hcal^1 \,\mbox{:}\, T^1_i\in D'_r(T^2_j, H^2_j, H^1_i)} \one_{A^{(2)}_{2,1}}\Big],%
\label{eq:h1v1}
\end{eqnarray}
\begin{eqnarray}
\Lambda^{(2)}_{2,1}= \E\Big[\sum_{(T^1_i,H^1_i)\in \Hcal^1 \,\mbox{:}\, T^1_i\in[0,1]}\,\, \sum_{(T^2_j,H^2_j)\in \Hcal^2 \,\mbox{:}\, T^2_j\in D_r(T^1_i, H^1_i, H^2_j)} \one_{A^{(1)}_{1,2}}\Big],
\label{eq:h1vb1}
\end{eqnarray}
where the sets $D_r$ and $D'_r$ are as defined in (\ref{eq:DlDr}) and (\ref{eq:DlDr-rev}), respectively.

The two definitions (\ref{eq:h1va1}) and (\ref{eq:h1v1}) are equivalent since, counting the number of handovers given by the first intersection between the pair of radial birds located at $(T^1_i,H^1_i)$ and $(T^2_j,H^2_j)$, with $T_i^1<T_j^2$, is equivalent to counting backward in time the number of handovers given by the second intersection between the pair of radial birds located at $(T^1_i,H^1_i)$ and $(T^2_j,H^2_j)$, with $T_j^2<T_i^1$. By the same reasoning, the two definitions (\ref{eq:h1vb2}) and (\ref{eq:h1vb1}) are also equivalent.

\begin{figure}[ht!]
    \centering
\begin{tikzpicture}[scale=0.55, every node/.style={scale=0.7}]
\pgftransformxscale{1}  
\pgftransformyscale{1}    
    \draw[->] (-12, 0) -- (4, 0) node[right] {$t$};
    \draw[dashed] (-4.5, 7) -- (-4.5, -0.2)node[below]{$0$};
    \draw[blue, dashed] (-12, 2) -- (4, 2);
    \draw[<-] (2, 1.95) -- (2, 1.3)node[below]{$h_2$};
    \draw[<-] (2, 0.05) -- (2, 0.8);
     \draw[<-] (2.8, 2.45) -- (2.8, 1.6)node[below]{$h_1$};
    \draw[<-] (2.8, 0.05) -- (2.8, 1.1);
   \draw[blue, domain=-3.8:3.8, smooth] plot (\x, {((3/2)^2*\x*\x+2^2)^0.5});
    \draw[blue](0,2) node{$\bullet$};
    \draw[](0.2,1.8) node{$(t_2,h_2)$};
   \draw[blue,dashed, domain=-12.8:-5.2, smooth] plot (\x, {((3/2)^2*\x*\x+(3/2)^2*18*\x+(3/2)^2*9^2+2^2)^0.5});
    \draw[blue](-9,2) node{$\bullet$};
    \draw[](-9,1.8) node{$(-t_2,h_2)$};
    \draw[red, domain=-3.5:-0.5, smooth] plot (\x, {(16*\x*\x+64*\x+64+2.5^2)^0.5});
    \draw[red](-2,2.5) node{$\bullet$};
     \draw[red,dashed, domain=-8.5:-5.5, smooth] plot (\x, {(16*\x*\x+16*14*\x+16*7^2+2.5^2)^0.5});
    \draw[red](-7,2.5) node{$\bullet$};
    \draw[](-6.5,3) node{$(-t_1,h_1)$};
    %
    \draw[red](-2,2.5) node{$\bullet$};
    \draw[red, dashed] (-12, 2.5) -- (4, 2.5);
    \draw[](-3,3) node{$(t_1,h_1)$};
    \draw[](-8.2,3.2) node{$(-\hat s_1,\hat h_1)$};
    \draw[](-6.9,5) node{$(-\hat s_2,\hat h_2)$};
    \draw[](-0.7,3.2) node{$(\hat s_1,\hat h_1)$};
    \draw[](-3.8,5) node{$(\hat s_2,\hat h_2)$};
    \end{tikzpicture}
    \caption{The mixed handovers seen in reverse time.}
\label{figure:sym_birdss}
\end{figure}
We now show that $\Lambda^{(1)}_{1,2}=\Lambda^{(2)}_{2,1}$. Due to the bijection $\beta_{1,2}^{(1)}$ in part~(\ref{bijection-ijk}) of Lemma~\ref{lemma:vr}, one can associate each mixed intersection to the right-most head points or equivalently to the left-most head point. Let $\Lcal^{(1)}_{1,2}$ be the point process on $\R$ made of the abscissas of the left-most head point corresponding to the handovers with the faster bird being on the left of the slower one.\label{notation:Lcallrk}\! Both the point processes $\Rcal^{(1)}_{1,2}$ and $\Lcal^{(1)}_{1,2}$ are stationary with respect to the time axis, which follows from the stationarity of $\Vcal^{(1)}_{1,2}$. Thus we can redefine the mixed handover frequency $\Lambda^{(1)}_{1,2}$ using the point process $\Lcal^{(1)}_{1,2}$, as 
\begin{equation}
\E\Big[\sum_{T_j^2\in \Rcal^{(1)}_{1,2} \,\mbox{:}\, T^2_j\in[0,1]}\,\, 1\Big]= \Lambda^{(1)}_{1,2} =\E\Big[\sum_{T_i^1\in \Lcal^{(1)}_{1,2} \,\mbox{:}\, T^1_i\in[0,1]}\,\, 1\Big]
\label{eq:Lcal-Rcal}
\end{equation}
and hence, from (\ref{eq:Lcal-Rcal}), we have
\begin{align}
\Lambda^{(1)}_{1,2}&=\E\Big[\sum_{T_i^1\in \Lcal^{(1)}_{1,2} \,\mbox{:}\, T^1_i\in[0,1]}\,\, 1\Big] \nn\\
&=\E\Big[\sum_{(T^1_i,H^1_i)\in \Hcal^1 \,\mbox{:}\, T^1_i\in[0,1]}\,\, \one_{\exists (T^2_j,H^2_j)\in \Hcal^2 \,\mbox{:}\, T^2_j\in D_r(T^1_i, H^1_i, H^2_j)} \one_{A^{(1)}_{1,2}}\Big] \nn\\
&=\E\Big[\sum_{(T^1_i,H^1_i)\in \Hcal^1 \,\mbox{:}\, T^1_i\in[0,1]}\,\, \sum_{(T^2_j,H^2_j)\in \Hcal^2 \,\mbox{:}\, T^2_j\in D_r(T^1_i, H^1_i, H^2_j)} \one_{A^{(1)}_{1,2}}\Big]\nn\\
&\stackrel{(\ref{eq:h1vb2}),( \ref{eq:h1vb1})}{=}\E\Big[\sum_{(T^1_i,H^1_i)\in \Hcal^1 \,\mbox{:}\, T^1_i\in[0,1]}\,\, \sum_{(T^2_j,H^2_j)\in \Hcal^2 \,\mbox{:}\, T^2_j\in D_\ell(T^1_i, H^1_i, H^2_j)} \one_{A^{(2)}_{2,1}}\Big] = \Lambda^{(2)}_{2,1}.
\label{eq:L12-L21}
\end{align}
We can prove the other equality $\Lambda^{(2)}_{1,2}= \Lambda^{(1)}_{2,1}$ similarly.
\end{proof} 
\begin{proof}[Proof of Theorem~\ref{theorem:HO2speed-main}~part~(\ref{MH2p3})]
Let us now compute the contribution of mixed handovers by evaluating $\Lambda^{(1)}_{2,1}$ and $\Lambda^{(2)}_{2,1}$. By applying the multivariate Campbell Mecke formula to the factorial power of order 2 of $\Hcal$ in (\ref{eq:h1va1}), similarly to what was done for Equation (\ref{eq:h1b}) in Theorem~\ref{theorem:handover_freq1} for the single-speed case, we obtain that, for $k=1,2$,
\begin{align} 
\!\!\!\!\Lambda^{(k)}_{2,1}
&= 4\la_1\la_2v_1v_2\int_{0}^{1}\!\! \int_{0}^{\infty}\!\!\int_{0}^{\infty}\!\! \int_{D_\ell(t_1, h_1, h_2)}\!\!\!
\E^{(t_1, h_1), (t_2,h_2)}_{\Hcal}\left[\one_{A^{(k)}_{2,1}}\left(\Hcal^1+\Hcal^2\right)\right] {\rm d}t_2 \,   {\rm d}h_2 \, {\rm d}h_1 \, {\rm d}t_1 \nn\\
&=4\la_1\la_2v_1v_2\int_{0}^{1}\!\! \int_{0}^{\infty}\!\!\int_{0}^{\infty}\!\! \int_{D_\ell(t_1, h_1, h_2)} \!\!\!\E_{\Hcal}\left[\one_{A^{(k)}_{2,1}}\left(\Hcal^1+\delta_{(t_1, h_1)}+\Hcal^2+\delta_{(t_2, h_2)}\right)\right] {\rm d}t_2 \,   {\rm d}h_2 \, {\rm d}h_1 \, {\rm d}t_1 ,
\label{eq:h3v}
\end{align}
where the probability measure is \label{notation:2PHlmP}
\begin{equation}
\P^{(t_1, h_1), (t_2,h_2)}_{\Hcal}:=\P^{(t_1, h_1)}_{\Hcal^1}\otimes \P^{(t_2,h_2)}_{\Hcal^2},
\label{eq:2P-Palm3}
\end{equation} 
namely the two-point Palm probability measure on $\Hcal$ with the first point belonging to $\Hcal^1$ and the second point belonging to $\Hcal^2$. Separating the integral (\ref{eq:h3v}) into the cases, $h_2\in [0, h_1]$, for which $D_\ell(t_1, h_1, h_2)=(-\infty, t_1-t^*]$, and $h_2\in (h_1, \infty)$, for which $D_\ell=(-\infty, t_1]$, from Lemma~\ref{lemma:2int}, we obtain
\begin{align} 
\!\!\!\Lambda^{(k)}_{2,1} &=  4\la_1\la_2v_1v_2\!\int_{0}^{1} \int_{0}^{\infty}\!\int_{0}^{h_1} \int_{-\infty}^{t_1-t^*}\E_{\Hcal}\left[\one_{A^{(k)}_{2,1}}\left(\Hcal^1+\delta_{(t_1, h_1)}+\Hcal^2+\delta_{(t_2, h_2)}\right)\right]\, {\rm d}t_2  \, {\rm d}h_2 \, {\rm d}h_1 \, {\rm d}t_1  \nn\\
&\quad+ \; 4\la_1\la_2v_1v_2\!\int_{0}^{1} \int_{0}^{\infty}\int_{h_1}^{\infty} \int_{-\infty}^{t_1} \E_{\Hcal}\left[\one_{A^{(k)}_{2,1}}\left(\Hcal^1+\delta_{(t_1, h_1)}+\Hcal^2+\delta_{(t_2, h_2)}\right)\right]\, {\rm d}t_2  \, {\rm d}h_2 \, {\rm d}h_1 \, {\rm d}t_1 ,
\label{eq:h3va1}
\end{align}
where $t^*$ as defined in (\ref{eq:tstar}).
Using the intersection points $(\hat s_1, \hat h_1), (\hat s_2, \hat h_2)$ of the radial birds, determined in (\ref{eq:HOs}) and (\ref{eq:HO_distance1}), we can redefine the event $A^{(k)}_{2,1}$, from  (\ref{eq:A12k}), for $k=1,2$, as
\begin{equation}
A^{(k)}_{2,1}:=\{(\hat s_k, \hat h_k) \text{ is a handover point}\}.
\label{eq:2HOPk}
\end{equation}
Under the two-point Palm probability measure $\P^{(t_1, h_1), (t_2,h_2)}_{\Hcal}$,
the required probability of the events $A^{(k)}_{2,1}$, for $k=1,2$,
turns out to be a function of $t_1-t_2$, given $h_1,h_2$.
This will be clear later in (\ref{eq:MH1}) in the proof of the following lemma. 
\begin{lemma} 
Suppose $h_1, h_2\in \R^+$ and $t_1, t_2\in \R$ such that $t_2\leq t_1$, for $h_1<h_2$ and $t_2\leq t_1-t^*$, for $h_1>h_2$, where $t^*$ is defined in (\ref{eq:tstar}). Under the two-point Palm probability measure $\P^{(t_1, h_1), (t_2,h_2)}_{\Hcal}$, the probability of the event $A^{(k)}_{2,1}$,  for $k=1,2$, is 
\begin{align}\E_{\Hcal}\left[\one_{A^{(k)}_{2,1}}\left(\Hcal^1+\delta_{(t_1, h_1)}+\Hcal^2+\delta_{(t_2, h_2)}\right)\right]&=e^{-\la\pi\hat h^2_k},\nn
\end{align}
where $\hat h_k$ is determined in (\ref{eq:HO_distance1}) as a function of $t_1-t_2, h_1, h_2$.
\label{lemma:HOProb}
\end{lemma}
Using Lemma~\ref{lemma:HOProb}, we can simplify (\ref{eq:h3va1}) to get 
\begin{align}
\Lambda^{(k)}_{2,1}&=4\la_1\la_2v_1v_2\int_{0}^{1}\int_{0}^{\infty}\int_0^{h_1}\int_{-\infty}^{t_1-t^*} e^{-\la\pi\hat h^2_k(t_1- t_2, h_1, h_2)} \,{\rm d}t_2 \, {\rm d}h_2 \, {\rm d}h_1 \, {\rm d}t_1 \nn\\
&\quad + 4\la_1\la_2v_1v_2\int_{0}^{1}\int_{0}^{\infty}\int_{h_1}^{\infty}\int_{-\infty}^{t_1} e^{-\la\pi\hat h^2_k(t_1- t_2, h_1, h_2)} \,{\rm d}t_2 \, {\rm d}h_2 \, {\rm d}h_1 \, {\rm d}t_1 .
\label{eq:h4v1}
\end{align}
Making the change of variable $t_1-t_2=t$ in the expression of $\hat h_k(t_1-t_2,h_1,h_2)$ in (\ref{eq:h4v1}) yields
\begin{align}
\Lambda^{(k)}_{2,1} &=4\la_1\la_2v_1v_2\int_{0}^{\infty} \left[\!\int_0^{h_1} \!\!\! \int_{t^*}^\infty \!\!\! e^{-\la\pi\hat h^2_k(t, h_1, h_2)} {\rm d}t {\rm d}h_2 + \int_{h_1}^{\infty}\!\!\! \int_{0}^{\infty} \!\!\! e^{-\la\pi\hat h^2_k(t, h_1, h_2)} {\rm d}t\, {\rm d}h_2 \, \right]{\rm d}h_1.
\label{eq:h4v}
\end{align}
Closed form expressions of these integrals are provided in Subsection~\ref{subsection:HOPP-Integral}. The total handover frequency is 
\begin{equation}
\la_\Vcal= \frac{4\sqrt{\la}}{\pi}  \sum_{l\in \{1,2\}} v_l\left(\frac{\la_l}{\la}\right)^2+2\sum_{l< r,k\in \{1,2\}}\La^{(k)}_{l,r}.
\label{eq:THF}
\end{equation}
This completes the proof of Theorem~\ref{theorem:HO2speed-main}, provided Lemma~\ref{lemma:HOProb} holds.
\end{proof}
\begin{proof}[Proof of Lemma~\ref{lemma:HOProb}]For the two intersection points $(\hat s_1, \hat h_1), (\hat s_2, \hat h_2)$, given by (\ref{eq:HOs}) and (\ref{eq:HO_distance1}), we have, for the event $A^{(k)}_{2,1}$ with the new definition in (\ref{eq:2HOPk}),
\begin{align}
\E_{\Hcal}\left[\one_{A^{(k)}_{2,1}}\left(\Hcal^1+\delta_{(t_1, h_1)}+\Hcal^2+\delta_{(t_2, h_2)}\right)\right]&=\P^{(t_1,h_1),(t_2,h_2)}_{\Hcal}\left((\hat s_k, \hat h_k) \text{ is a handover point}\right).
\label{eq:HO-12}
\end{align}
\begin{figure}[ht!]
 \centering
      \begin{subfigure}[t]{0.45\linewidth}
       \centering
      \begin{tikzpicture}[scale=0.6, every node/.style={scale=0.7}]
\pgftransformxscale{0.8}  
    \pgftransformyscale{0.8}    
    \draw[->] (-3, 0) -- (7, 0) node[right] {$t$};
   \draw[blue, domain=-3.8:3.8, smooth] plot (\x, {((3/2)^2*\x*\x+2^2)^0.5});
    \draw[blue](0,2) node{$\bullet$};
        \draw[](0.2,1.6) node{$(t_2,h_2)$};
    \draw[red, domain=0.5:3.5, smooth] plot (\x, {(16*\x*\x-64*\x+64+2.5^2)^0.5});
    \draw[red](2,2.5) node{$\bullet$};
    \draw[](2.8,2.5) node{$(t_1,h_1)$};
    \draw[](0.7,3.3) node{$(\hat s_1,\hat h_1)$};
    \draw[](4,5.2) node{$(\hat s_2,\hat h_2)$};
     \draw[blue]  (-0.47,0) arc (-180:-360:2 and 3.07);
     \draw[red]  (2.35,0) arc (0:180:0.8 and 3.07);
     \draw[blue](-1.2,0.6) node{$E^{\hat s_1,v_2}_{\hat h_1}$};
     \draw[red](1.5,0.6) node{$E^{\hat s_1,v_1}_{\hat h_1}$};
    \end{tikzpicture}
    \caption{The regions \red{$E^{s_1,v_1}_{\ell_1}$} and \blue{$E^{s_1,v_2}_{\ell_1}$} must be empty of heads from \red{fast} and \blue{slow} stations respectively, in case the mixed intersection $(\hat s_1, \hat h_1)$ represents an handover.}
    \label{subfigure:mixed_birds1}
    \end{subfigure}
    \hspace{0.1in}
    \centering
      \begin{subfigure}[t]{0.45\linewidth}
       \centering
      \begin{tikzpicture}[scale=0.6, every node/.style={scale=0.7}]
\pgftransformxscale{0.8}  
    \pgftransformyscale{0.8}    
    \draw[->] (-3, 0) -- (7, 0) node[right] {$t$};
   \draw[blue, domain=-3.8:3.8, smooth] plot (\x, {((3/2)^2*\x*\x+2^2)^0.5});
    \draw[blue](0,2) node{$\bullet$};
    \draw[](0.2,1.6) node{$(t_2,h_2)$};
    \draw[red, domain=0.5:3.5, smooth] plot (\x, {(16*\x*\x-64*\x+64+2.5^2)^0.5});
    \draw[red](2,2.5) node{$\bullet$};
    \draw[](2.8,2.5) node{$(t_1,h_1)$};
    \draw[](0.7,3.3) node{$(\hat s_1,\ell_1)$};
    \draw[](4,5.2) node{$(\hat s_2,\hat h_2)$};
    \draw[blue](0.7,0.6) node{$E^{\hat s_2,v_2}_{\hat h_2}$};
     \draw[red](2.6,0.6) node{$E^{\hat s_2,v_1}_{\hat h_2}$};
     \draw[red]  (1.85,0) arc (-180:-360:1.25 and 5.05);
     \draw[blue]  (6.5,0) arc (0:180:3.4 and 5.05);
    \end{tikzpicture}
    \caption{The regions \red{$E^{\hat s_2,v_1}_{\hat h_2}$} and  \blue{$E^{\hat s_2,v_2}_{\hat h_2}$} must be empty of heads from \red{fast} and \blue{slow} stations respectively, in case the mixed intersection $(\hat s_2, \hat h_2)$ represents an handover.}
    \label{subfigure:mixed_birds2}
    \end{subfigure}
    \caption{Mixed handovers due to first and second intersection of birds of two types.}
    \label{figure:mixed_birdss}
\end{figure}

In the rest of the proof, we evaluate the probability in (\ref{eq:HO-12}), for $k=1,2$. Let us now assume that $t_1> t_2$ and $v_1>v_2$ and consider that $(\hat s_k, \hat h_k)$ is a handover point. The upper half ellipses corresponding to $(\hat s_k,\hat h_k)$ are given by,
\begin{equation}
v_1^2(t-\hat s_k)^2+h^2=\hat h_k^2 \text{ and }v_2^2(t-\hat s_k)^2+h^2=\hat h_k^2
\label{eq:ep1}
\end{equation}
and the corresponding regions are $E^{\hat s_k, v_1}_{\hat h_k}$ and $E^{\hat s_k, v_2}_{\hat h_k}$, respectively. For a point $(\hat s_k, \hat h_k)$ to be a handover point, we need that the regions $E^{\hat s_k, v_1}_{\hat h_k}$ and $E^{\hat s_k, v_2}_{\hat h_k}$ be empty of points from $\Hcal^1$ and $\Hcal^2$, respectively (see Figure~\ref{figure:mixed_birdss}). Since $\Hcal^1, \Hcal^2$ are Poisson point processes, the corresponding probability is
\begin{align}
\P^{(t_1,h_1),(t_2,h_2)}_{\Hcal}\left((\hat s_k, \hat h_k) \text{ is a handover point}\right)&=\prod_{l=1,2}\P_{\Hcal^l}\left( \Hcal^l(E^{\hat s_k,v_l}_{\hat h_k})=0\right)\nn\\
&= \prod_{l=1,2}e^{-2\la_lv_l\frac{\pi \hat h_k^2}{2v_l}}= e^{-\pi \hat h_k^2(\la_1+\la_2)}=e^{-\pi\la \hat h_k^2},
\label{eq:MH1}
\end{align}
since $\la_1+\la_2=\la$.
\end{proof}
\subsection{Closed form for the integrals in~(\ref{eq:h4v}) }\label{subsection:HOPP-Integral}
This section contains the computation of mixed handover frequency in continuation of the results derived in derived in Theorem~\ref{theorem:HO2speed-main}. Recall that the mixed handover frequency $\Lambda^{(k)}_{l,r}$ correspond to the handovers due to the $k$-th intersection of the radial birds of type $l$ and type $r$, where the head point of type $l$ is on the left of that of type $r$.  We know from Theorem~\ref{theorem:HO2speed-main}, the symmetry among mixed handover intensities $\Lambda^{(1)}_{2,1}=\Lambda^{(2)}_{1,2}$ and $\Lambda^{(2)}_{2,1} = \Lambda^{(1)}_{1,2}$ and their expressions in the integral form are as follows:
\begin{align}
\Lambda^{(1)}_{2,1}  &= 4\la_1\la_2v_1v_2\int_{0}^{\infty} \left[\!\int_0^{h_1} \!\!\! \int_{t^*}^\infty \!\!\! e^{-\la\pi\hat h^2_1(t, h_1, h_2)} {\rm d}t {\rm d}h_2 +  \int_{h_1}^{\infty}\!\!\! \int_{0}^{\infty} \!\!\! e^{-\la\pi\hat h^2_1(t, h_1, h_2)} {\rm d}t\, {\rm d}h_2 \right] {\rm d}h_1,
\label{eq:Lam21-1}
\end{align}
and 
\begin{align}
\Lambda^{(2)}_{2,1}  &=4\la_1\la_2v_1v_2\int_{0}^{\infty} \left[ \!\int_0^{h_1} \!\!\! \int_{t^*}^\infty \!\!\! e^{-\la\pi\hat h^2_2(t, h_1, h_2)} {\rm d}t {\rm d}h_2     
+  \int_{h_1}^{\infty}\!\!\! \int_{0}^{\infty} \!\!\! e^{-\la\pi\hat h^2_2(t, h_1, h_2)} {\rm d}t\, {\rm d}h_2 \right] {\rm d}h_1,
\label{eq:Lam21-2}
\end{align}
where 
\begin{equation}
\begin{aligned}
\hat h_1^2= v_1^2 \left(\frac{v_2^2(t_1-t_2)- \sqrt{\Delta}}{v_1^2-v_2^2}\right)^2+h_1^2,\qquad 
\hat h_2^2 = v_1^2 \left(\frac{v_2^2(t_1-t_2)+ \sqrt{\Delta}}{v_1^2-v_2^2}\right)^2+h_1^2,
\end{aligned}
\label{eq:HO_distance1x}
\end{equation}
\[
t^*= \frac{1}{v_1v_2}(h_1^2-h_2^2)^\half (v_1^2-v_2^2)^\half, \quad \Delta= v_1^2v_2^2\left(t_1-t_2\right)^2\!-\! (h_1^2-h_2^2)(v_1^2-v_2^2).
\]
In the following, we derive the closed form expressions for the mixed handover frequencies by evaluating the integrals in (\ref{eq:Lam21-1}) and (\ref{eq:Lam21-2}), using the complete elliptic integrals of second and first kind, respectively, defined as follows. For any $m\in[0,1]$,
\begin{align}
E(m)&=\int_0^{\pi/2}\sqrt{1-m^2\sin^2\beta}\,{\rm d}\beta \mbox{ and } K(m) =\int_0^{\pi/2}\frac{1}{\sqrt{1-m^2\sin^2\beta}}\,{\rm d}\beta.
\label{eq:Elliptic-int}
\end{align}
\begin{theorem}
The mixed handover frequencies are
\begin{enumerate}
\item \label{m211}
\begin{equation}
\Lambda^{(1)}_{2,1} = \frac{\la_1\la_2}{\pi\la^{3/2}}(v_1+v_2) \left[\left(1+\frac{v_1}{v_2}\right)E\!\left(\frac{v_2}{v_1}\right)- \frac{v_1^2-v_2^2}{v_1 v_2} K\left(\frac{v_2}{v_1}\right)\right] = \Lambda^{(2)}_{1,2},
\label{eq:mhof1}
\end{equation}
\item \label{m212} \begin{equation}
\Lambda^{(2)}_{2,1} = \frac{\la_1\la_2}{\pi\la^{3/2}} (v_1-v_2)  \left[\left(1+\frac{v_1}{v_2}\right)E\!\left(\frac{v_2}{ v_1}\right)-  \frac{v_1^2-v_2^2}{v_1 v_2} K\left(\frac{v_2}{v_1}\right)-2\right] = \Lambda^{(1)}_{1,2},
\label{eq:mhof2}
\end{equation}
\end{enumerate}
where $E(\cdot)$ and $K(\cdot)$ are the complete elliptic integrals of second and first kind, respectively, defined in (\ref{eq:Elliptic-int}).
\label{theorem:mixed-hof}
\end{theorem}
\subsubsection{Proof of Theorem~\ref{theorem:mixed-hof}}
We make use of simplified expressions for $\hat h_1^2$ and $\hat h_2^2$ using $t^*$ and $\Delta$, to simplify the integrands in (\ref{eq:Lam21-1}) and (\ref{eq:Lam21-2}). This allow us to unfold each layer of the integrals at a time and eventually evaluating the entire integral. In the process we establish useful identities involving the complete elliptic integrals of second and first kind, as defined in (\ref{eq:Elliptic-int}), and erf and erfc functions as well. Application of those identities lead to the shape the final expression of the mixed handover intensities.  
\begin{proof}[Proof of Theorem~\ref{theorem:mixed-hof}~(part~\ref{m211})]
Let $(\hat s_1, \hat h_1), (\hat s_2, \hat h_2)$ be the intersection points between the radial birds located at $(t_1,h_1)$ and $(t_2,h_2)$ of type $1$ and $2$ respectively, with $t_2<t_1$. 
Note that since $t^*= \frac{1}{v_1v_2}(h_1^2-h_2^2)^\half (v_1^2-v_2^2)^\half$, we have 
\[
\Delta= v_1^2v_2^2\left(t_1-t_2\right)^2\!-\! (h_1^2-h_2^2)(v_1^2-v_2^2)=\begin{cases} v_1^2v_2^2\left((t_1-t_2)^2- (t^*)^2 \right), \text{ if } h_2<h_1\\
v_1^2v_2^2\left((t_1-t_2)^2+ (t^{**})^2\right), \text{ if } h_2\geq h_1,
\end{cases}
\]
where $t^{**}= \frac{1}{v_1v_2}(h_2^2-h_1^2)^\half (v_1^2-v_2^2)^\half$. Using this we have for $h_2<h_1$
\begin{align}
\hat h_1^2 &= \frac{v_1^2}{(v_1^2-v_2^2)^2} \left(v_2^2(t_1-t_2)- \sqrt{\Delta}\right)^2+h_1^2\nn\\
&= \frac{v_1^2}{(v_1^2-v_2^2)^2} \left(v_2^2(t_1-t_2)- v_1v_2\left((t_1-t_2)^2- (t^*)^2 \right)^\half\right)^2+h_1^2\nn\\
&= \frac{v_1^2v_2^2}{(v_1^2-v_2^2)^2} \left(v_2(t_1-t_2)- v_1\left((t_1-t_2)^2- (t^*)^2 \right)^\half\right)^2+h_1^2\nn\\
&= \frac{v_1^2v_2^2}{(v_1^2-v_2^2)^2} \left(v_2t- v_1\left(t^2- (t^*)^2 \right)^\half\right)^2+h_1^2,
\end{align}
by taking $t_1-t_2=t$. For $h_2\geq h_1$ and taking $t_1-t_2=t$ we have 
\begin{align}
\hat h_1^2 &= \frac{v_1^2v_2^2}{(v_1^2-v_2^2)^2} \left(v_2t- v_1\left(t^2+ (t^{**})^2 \right)^\half\right)^2+h_1^2.
\end{align}
Carrying on computation of the mixed handover frequency $\Lambda^{(1)}_{2,1}$ derived in Theorem~\ref{theorem:HO2speed-main}, we get
\begin{align}
\Lambda^{(1)}_{2,1}  &=4\la_1\la_2v_1v_2\int_{0}^{\infty} \left[\!\int_0^{h_1} \!\!\! \int_{t^*}^\infty \!\!\! e^{-\la\pi\hat h^2_1(t, h_1, h_2)} {\rm d}t {\rm d}h_2 +  \int_{h_1}^{\infty}\!\!\! \int_{0}^{\infty} \!\!\! e^{-\la\pi\hat h^2_1(t, h_1, h_2)} {\rm d}t\, {\rm d}h_2 \right] {\rm d}h_1\nn\\
&=4\la_1\la_2v_1v_2\int_{0}^{\infty} \! \int_0^{h_1} e^{-\la\pi h_1^2} \int_{t^*}^\infty \!\!\! e^{-\la\pi \bar v^2 \left(v_2t- v_1\left(t^2- (t^*)^2 \right)^\half\right)^2} {\rm d}t \, {\rm d}h_2 {\rm d}h_1\nn\\
& \;\;{+} 4\la_1\la_2v_1v_2\int_{0}^{\infty}   \int_{h_1}^{\infty} e^{-\la\pi h_1^2} \int_{0}^{\infty} \!\!\! e^{-\la\pi \bar {v}^2 \left(v_2t- v_1\left(t^2+ (t^{**})^2 \right)^\half\right)^2}{\rm d}t \, {\rm d}h_2\,{\rm d}h_1,
\label{eq:h4vz1}
\end{align}
where $\frac{v_1v_2}{v_1^2-v_2^2}:= \bar v$ and $t^{**}= \frac{1}{v_1v_2}(h_2^2-h_1^2)^\half (v_1^2-v_2^2)^\half$.
\begin{lemma}
Given $h_1,h_2$, the inner integrals in $t$ in (\ref{eq:h4vz1}) are 
\begin{align}
\int_{t^*}^\infty  e^{-\la\pi \bar v^2 \left(v_2t- v_1\sqrt{t^2- (t^*)^2}\right)^2} \, {\rm d}t &= \frac{1}{2 \sqrt{\la}v_1v_2} \left[ v_1 e^{\la\pi (h_1^2-h_2^2)} \text{erfc}\left( \sqrt{\la\pi}\left(h_1^2-h_2^2\right)^\half \bar{v}_1\right)\right. \nn\\
&\qquad \qquad\qquad  + \left. v_2  \left(1 + \text{erf}\left(\sqrt{\la\pi}\left(h_1^2-h_2^2\right)^\half  \bar{v}_2\right)\right) \right],
\label{two-int1}
\end{align}
\begin{align}
\int_{0}^{\infty} \!\!\! e^{-\la\pi \bar {v}^2 \left(v_2t- v_1\left(t^2+ (t^{**})^2 \right)^\half\right)^2}{\rm d}t &=  \frac{1}{2v_1v_2\sqrt{\la}} \left[ v_1 e^{-\la\pi (h_2^2-h_1^2)}  \left(1 + \text{erf}\left(\sqrt{\la\pi}\left(h_2^2-h_1^2\right)^\half \bar v_2\right)\right)\right. \nn\\
& \qquad \qquad\qquad  + \left. v_2  \text{erfc}\left( \sqrt{\la\pi}\left(h_2^2-h_1^2\right)^\half \bar v_1\right) \right],
\label{two-int2}
\end{align}
where $\bar v_1:=\frac{v_1}{(v_1^2-v_2^2)^\half}$ and $\bar v_2:=\frac{v_2}{(v_1^2-v_2^2)^\half}$.
\label{lemma:two-int}
\end{lemma}
The proof of the last lemma is given in Appendix~\ref{sec-appendix_integrals}. Using part~\ref{two-int1} of Lemma~\ref{lemma:two-int}, the first term in (\ref{eq:h4vz1}) becomes
\begin{align}
\lefteqn{4\la_1\la_2v_1v_2\int_{0}^{\infty} \! \int_0^{h_1} e^{-\la\pi h_1^2} \int_{t^*}^\infty \!\!\! e^{-\la\pi \bar v^2 \left(v_2t- v_1\left(t^2- (t^*)^2 \right)^\half\right)^2} {\rm d}t \, {\rm d}h_2 {\rm d}h_1}\nn\\ 
&= \!\! \frac{2\la_1\la_2}{\la^{1/2}} \!\! \int_{0}^{\infty} \!\!\! \int_0^{h_1} \left[ v_1 e^{-\la\pi h_2^2} \text{erfc}\left( \sqrt{\la\pi}\left(h_1^2{-}h_2^2\right)^\half \bar{v}_1\right) {+} v_2 e^{- \la\pi h_1^2}\left(1 {+} \text{erf}\left(\sqrt{\la\pi}\left(h_1^2{-}h_2^2\right)^\half  \bar{v}_2\right)\right) \right]{\rm d}h_2 {\rm d}h_1.
\label{eq:FT1}
\end{align}
Also using part~\ref{two-int2} of Lemma~\ref{lemma:two-int} the second term in (\ref{eq:h4vz1}) becomes
\begin{align}
\lefteqn{4\la_1\la_2v_1v_2\int_{0}^{\infty}   \int_{h_1}^{\infty} e^{-\la\pi h_1^2} \int_{0}^{\infty} \!\!\! e^{-\la\pi \bar {v}^2 \left(v_2t- v_1\left(t^2+ (t^{**})^2 \right)^\half\right)^2}{\rm d}t \, {\rm d}h_2\,{\rm d}h_1}\nn\\
&= \!\!\frac{2\la_1\la_2}{\la^{1/2}} \!\!  \int_{0}^{\infty} \!\!\! \int_{h_1}^{\infty} \!\! \left[ v_1 e^{-\la\pi h_2^2}  \left(1 {+} \text{erf}\left(\sqrt{\la\pi}\left(h_2^2{-}h_1^2\right)^\half \bar v_2\right)\right) {+} v_2 e^{- \la\pi h_1^2} \text{erfc}\left( \sqrt{\la\pi}\left(h_2^2{-}h_1^2\right)^\half \bar v_1\right) \right] {\rm d}h_2\,{\rm d}h_1\nn\\
&=\!\! \frac{2\la_1\la_2}{\la^{1/2}} \!\! \int_{0}^{\infty} \!\!\! \int_0^{h_2} \!\! \left[ v_1 e^{-\la\pi h_2^2}  \left(1 {+} \text{erf}\left(\sqrt{\la\pi}\left(h_2^2{-}h_1^2\right)^\half \bar v_2\right)\right) {+} v_2 e^{- \la\pi h_1^2} \text{erfc}\left( \sqrt{\la\pi}\left(h_2^2{-}h_1^2\right)^\half \bar v_1\right) \right] \! {\rm d}h_1\,{\rm d}h_2,
\label{eq:FT2}
\end{align}
by exchanging the integrals for the variables $h_1, h_2$ by applying Fubini, where $\bar v_2:=\frac{v_2}{(v_1^2-v_2^2)^\half}$. Summing up the quantities from (\ref{eq:FT1}) and (\ref{eq:FT2}) and substituting the resultant in (\ref{eq:h4vz1}) we obtain the expression for $\Lambda^{(1)}_{2,1}$ as 
\begin{align}
\lefteqn{\frac{2\la_1\la_2}{\la^{1/2}} \int_{0}^{\infty} \! \int_0^{h_1} \left[ v_1 e^{-\la\pi h_2^2} \text{erfc}\left( \sqrt{\la\pi}\left(h_1^2-h_2^2\right)^\half \bar{v}_1\right) + v_2 e^{- \la\pi h_1^2}\left(1 + \text{erf}\left(\sqrt{\la\pi}\left(h_1^2-h_2^2\right)^\half  \bar{v}_2\right)\right) \right]{\rm d}h_2 {\rm d}h_1}\nn\\
&\;\; {+} \frac{2\la_1\la_2}{\la^{1/2}}\!\!\int_{0}^{\infty} \!\!\! \int_0^{h_2} \!\!\left[ v_1 e^{-\la\pi h_2^2}  \left(1 {+} \text{erf}\left(\sqrt{\la\pi}\left(h_2^2-h_1^2\right)^\half \bar v_2\right)\right) {+} v_2 e^{- \la\pi h_1^2} \text{erfc}\left( \sqrt{\la\pi}\left(h_2^2-h_1^2\right)^\half \bar v_1\right) \right] {\rm d}h_1\,{\rm d}h_2\nn\\
&= \! \frac{2\la_1\la_2}{\la^{1/2}}(v_1{+}v_2)\!\! \int_{0}^{\infty} \!\!\! \int_0^{h_1}\!\! \left[ e^{-\la\pi h_2^2} \text{erfc}\left( \sqrt{\la\pi}\left(h_1^2{-}h_2^2\right)^\half \bar{v}_1\right) {+}  e^{- \la\pi h_1^2}\left(1 {+} \text{erf}\left(\sqrt{\la\pi}\left(h_1^2{-}h_2^2\right)^\half  \bar{v}_2\right)\right) \right]\! {\rm d}h_2 {\rm d}h_1\nn\\
&= \!\! \frac{2\la_1\la_2}{\la^{1/2}}(v_1{+}v_2) \int_{0}^{\infty} \!\!\! \int_0^{h_1} \!\!\left[ e^{-\la\pi h_2^2} \text{erfc}\left( \sqrt{\la\pi}\left(h_1^2{-}h_2^2\right)^\half \bar{v}_1\right) {-}  e^{- \la\pi h_1^2}\text{erfc}\left(\sqrt{\la\pi}\left(h_1^2{-}h_2^2\right)^\half  \bar{v}_2\right) \right] \! {\rm d}h_2 {\rm d}h_1\nn\\
&\qquad + \frac{4\la_1\la_2}{\la^{1/2}}(v_1+v_2) \int_{0}^{\infty} \! \int_0^{h_1} e^{- \la\pi h_1^2} {\rm d}h_2 {\rm d}h_1
\label{eq:erf-erfc}
\end{align}
by changing the role of $h_1$ and $h_2$ in the second term in the first step, and later writing erf function in terms of erfc function in the second step. 
\begin{lemma}
The two integrals in (\ref{eq:erf-erfc}) in the variables $h_1,h_2$ are evaluated as
\begin{align}
\int_{0}^{\infty} \! \int_0^{h_1} e^{-\la\pi h_2^2} \text{erfc}\left( \sqrt{\la\pi}\left(h_1^2-h_2^2\right)^\half \bar{v}_1\right) {\rm d}h_2 {\rm d}h_1= \frac{1}{2\pi\lambda}\left[E\!\left(\frac{v_2}{ v_1}\right)-1\right],
\label{h1h21}
\end{align}
\begin{align}
\int_{0}^{\infty} \! \int_0^{h_1}  e^{- \la\pi h_1^2}\text{erfc}\left(\sqrt{\la\pi}\left(h_1^2-h_2^2\right)^\half  \bar{v}_2\right) {\rm d}h_2 {\rm d}h_1 = \frac{1}{2\la\pi}\left[1+ \frac{1}{\bar v_2 \bar v_1}K\left(\frac{v_2}{v_1}\right)- \frac{v_1}{v_2} E\left(\frac{v_2}{ v_1}\right)\right]. 
\label{h1h22}
\end{align}
where $E(\cdot)$ and $K(\cdot)$ are complete elliptic integrals of second and first kind, respectively, defined in (\ref{eq:Elliptic-int}).
\label{lemma:h1h2}
\end{lemma}
\vspace{-0.2in}
The proof of the last lemma is given in Appendix~\ref{sec-appendix_integrals}. Using the expression of the integrals from Lemma~\ref{lemma:h1h2}, in total, we have
\begin{align}
\Lambda_{2,1}^{(1)}
&= \!\!\frac{2\la_1\la_2}{\la^{1/2}}(v_1{+}v_2) \!\int_{0}^{\infty} \!\!\! \int_0^{h_1} \!\!\left[ e^{-\la\pi h_2^2} \text{erfc}\left( \sqrt{\la\pi}\left(h_1^2{-}h_2^2\right)^\half \bar{v}_1\right) {-}  e^{- \la\pi h_1^2}\text{erfc}\left(\sqrt{\la\pi}\left(h_1^2{-}h_2^2\right)^\half  \bar{v}_2\right) \right]{\rm d}h_2 {\rm d}h_1\nn\\
&\qquad + \frac{4\la_1\la_2}{\la^{1/2}}(v_1+v_2) \int_{0}^{\infty} \! \int_0^{h_1} e^{- \la\pi h_1^2} {\rm d}h_2 {\rm d}h_1\nn\\
&= \frac{\la_1\la_2}{\pi\la^{3/2}}(v_1{+}v_2) \left[E\!\left(v_2/ v_1\right){-}1 {-} \left(1{+}\frac{K(v_2/v_1)}{\bar v_2 \bar v_1}- \frac{v_1}{v_2} E(v_2/v_1)\right)\right] {+} \frac{4\la_1\la_2}{\la^{1/2}}(v_1{+}v_2) \!\!\int_{0}^{\infty} \!\!\!h_1 e^{- \la\pi h_1^2}  {\rm d}h_1\nn\\
&= \frac{\la_1\la_2}{\pi\la^{3/2}}(v_1+v_2) \left[\left(1+\frac{v_1}{v_2}\right)E\!\left(v_2/ v_1\right)- \frac{K(v_2/v_1)}{\bar v_2 \bar v_1}-2\right] + \frac{2\la_1\la_2}{\pi\la^{3/2}}(v_1+v_2) \nn\\
&= \frac{\la_1\la_2}{\pi\la^{3/2}}(v_1+v_2) \left[\left(1+\frac{v_1}{v_2}\right)E\!\left(v_2/ v_1\right)- (v_1^2-v_2^2)\frac{K(v_2/v_1)}{v_1 v_2}\right].
\label{eq:la21-1f}
\end{align}
\end{proof}
\begin{proof}[Proof of Theorem~\ref{theorem:mixed-hof}~(part~\ref{m212})]
Recall the expression of $\hat h_2^2 = v_1^2 \left(\frac{v_2^2(t_1-t_2)+ \sqrt{\Delta}}{v_1^2-v_2^2}\right)^2+h_1^2$. Note that, for $h_2<h_1$, we have 
\begin{align}
\hat h_2^2 &= \frac{v_1^2v_2^2}{(v_1^2-v_2^2)^2} \left(v_2t+ v_1\left(t^2- (t^*)^2 \right)^\half\right)^2+h_1^2,
\end{align}
where $t^{*}= \frac{1}{v_1v_2}(h_1^2-h_2^2)^\half (v_1^2-v_2^2)^\half$. For $h_2\geq h_1$ we have 
\begin{align}
\hat h_2^2 &= \frac{v_1^2v_2^2}{(v_1^2-v_2^2)^2} \left(v_2t+ v_1\left(t^2+ (t^{**})^2 \right)^\half\right)^2+h_1^2,
\end{align}
where $t^{**}= \frac{1}{v_1v_2}(h_2^2-h_1^2)^\half (v_1^2-v_2^2)^\half$. Carrying on computation of $\Lambda^{(2)}_{2,1}$ derived in Theorem~\ref{theorem:HO2speed-main}, we re-write the mixed handover frequency $\Lambda^{(2)}_{2,1}$ as
\begin{align}
\Lambda^{(2)}_{2,1}  &=4\la_1\la_2v_1v_2\int_{0}^{\infty} \!\int_0^{h_1} \!\!\! \int_{t^*}^\infty \!\!\! e^{-\la\pi\hat h^2_2(t, h_1, h_2)} {\rm d}t {\rm d}h_2 \, {\rm d}h_1    
+ \int_{0}^{\infty} \int_{h_1}^{\infty}\!\!\! \int_{0}^{\infty} \!\!\! e^{-\la\pi\hat h^2_2(t, h_1, h_2)} {\rm d}t\, {\rm d}h_2 \, {\rm d}h_1\nn\\
&=4\la_1\la_2v_1v_2 \int_{0}^{\infty} \! \int_0^{h_1} e^{-\la\pi h_1^2} \int_{t^*}^\infty \!\!\! e^{-\la\pi \bar v^2 \left(v_2t+v_1\left(t^2- (t^*)^2 \right)^\half\right)^2} {\rm d}t \, {\rm d}h_2 {\rm d}h_1\nn\\
& \;\;{+} 4\la_1\la_2v_1v_2\int_{0}^{\infty}   \int_{h_1}^{\infty} e^{-\la\pi h_1^2} \int_{0}^{\infty} \!\!\! e^{-\la\pi \bar {v}^2 \left(v_2t+ v_1\left(t^2+ (t^{**})^2 \right)^\half\right)^2}{\rm d}t \, {\rm d}h_2\,{\rm d}h_1,
\label{eq:h4vz2}
\end{align}
where $\frac{v_1v_2}{v_1^2-v_2^2}:= \bar v$. 
\begin{lemma}
Given $h_1,h_2$, the inner integrals in $t$ in (\ref{eq:h4vz2}) are 
\begin{align}
\int_{t^*}^\infty \!\!\! e^{-\la\pi \bar v^2 \left(v_2t+v_1\left(t^2- (t^*)^2 \right)^\half\right)^2} {\rm d}t &=  \frac{1}{2 \sqrt{\la}v_1v_2} \left[ v_1 e^{\la\pi (h_1^2-h_2^2)} \text{erfc}\left( \sqrt{\la\pi}\left(h_1^2-h_2^2\right)^\half \bar{v}_1\right)\right.\nn\\
&\qquad\qquad \qquad \left. -v_2\text{erfc}\left(\sqrt{\la\pi}\left(h_1^2-h_2^2\right)^\half  \bar{v}_2\right) \right],
\end{align}
\begin{align}
\int_{0}^{\infty} \!\!\! e^{-\la\pi \bar {v}^2 \left(v_2t+ v_1\left(t^2+ (t^{**})^2 \right)^\half\right)^2}{\rm d}t &= \frac{1}{2v_1v_2\sqrt{\la}} \left[ v_1 e^{-\la\pi (h_2^2-h_1^2)}  \text{erfc}\left(\sqrt{\la\pi}\left(h_2^2-h_1^2\right)^\half \bar v_2\right) \right.\nn\\
&\qquad\qquad \qquad \left.- v_2  \text{erfc}\left( \sqrt{\la\pi}\left(h_2^2-h_1^2\right)^\half \bar v_1\right) \right].
\end{align}
\label{lemma:two-inta}
\end{lemma}
Using computation similar to that in (\ref{eq:erf-erfc}) and the expressions of the inner integrals obtained in Lemma~\ref{lemma:two-inta}, we get a simplified expression for $\Lambda^{(2)}_{2,1}$ as
\begin{align}
\lefteqn{\frac{2\la_1\la_2}{\sqrt{\la}} \int_{0}^{\infty} \! \int_0^{h_1} \left[ v_1 e^{-\la\pi h_2^2} \text{erfc}\left( \sqrt{\la\pi}\left(h_1^2-h_2^2\right)^\half \bar{v}_1\right) - v_2 e^{-\la\pi h_1^2} \text{erfc}\left(\sqrt{\la\pi}\left(h_1^2-h_2^2\right)^\half  \bar{v}_2\right) \right] {\rm d}h_2\,{\rm d}h_1}\nn\\
&\; + \frac{2\la_1\la_2}{\sqrt{\la}}  \int_{0}^{\infty} \! \int_{h_1}^\infty \left[ v_1 e^{-\la\pi h_2^2} \text{erfc}\left(\sqrt{\la\pi}\left(h_2^2-h_1^2\right)^\half \bar v_2\right) - v_2 e^{-\la\pi h_1^2}  \text{erfc}\left( \sqrt{\la\pi}\left(h_2^2-h_1^2\right)^\half \bar v_1\right) \right] {\rm d}h_2\,{\rm d}h_1\nn\\
&=\frac{2\la_1\la_2}{\sqrt{\la}} \int_{0}^{\infty} \! \int_0^{h_1} \left[ v_1 e^{-\la\pi h_2^2} \text{erfc}\left( \sqrt{\la\pi}\left(h_1^2-h_2^2\right)^\half \bar{v}_1\right) - v_2 e^{-\la\pi h_1^2} \text{erfc}\left(\sqrt{\la\pi}\left(h_1^2-h_2^2\right)^\half  \bar{v}_2\right) \right] {\rm d}h_2\,{\rm d}h_1\nn\\
&\; + \frac{2\la_1\la_2}{\sqrt{\la}}  \int_{0}^{\infty} \! \int_0^{h_2} \left[ v_1 e^{-\la\pi h_2^2} \text{erfc}\left(\sqrt{\la\pi}\left(h_2^2-h_1^2\right)^\half \bar v_2\right) - v_2 e^{-\la\pi h_1^2}  \text{erfc}\left( \sqrt{\la\pi}\left(h_2^2-h_1^2\right)^\half \bar v_1\right) \right] {\rm d}h_1\,{\rm d}h_2\nn\\
&= \frac{2\la_1\la_2}{\sqrt{\la}} (v_1{-}v_2) \!\!\int_{0}^{\infty} \!\!\! \int_0^{h_1} \left[ e^{-\la\pi h_2^2} \text{erfc}\left( \sqrt{\la\pi}\left(h_1^2{-}h_2^2\right)^\half \bar{v}_1\right) {-}  e^{-\la\pi h_1^2} \text{erfc}\left(\sqrt{\la\pi}\left(h_1^2{-}h_2^2\right)^\half  \bar{v}_2\right) \right] {\rm d}h_2\,{\rm d}h_1,
\end{align}
by exchanging the role of $h_1$ and $h_2$ in the second term in the second step. Invoking Lemma~\ref{lemma:h1h2}, we have
\[
\int_{0}^{\infty} \! \int_0^{h_1} e^{-\la\pi h_2^2} \text{erfc}\left( \sqrt{\la\pi}\left(h_1^2-h_2^2\right)^\half \bar{v}_1\right) {\rm d}h_2 {\rm d}h_1= \frac{1}{2\pi\lambda}\left[E\!\left(\frac{v_2}{ v_1}\right)-1\right],
\]
\[
\int_{0}^{\infty} \! \int_0^{h_1}  e^{- \la\pi h_1^2}\text{erfc}\left(\sqrt{\la\pi}\left(h_1^2-h_2^2\right)^\half  \bar{v}_2\right) {\rm d}h_2 {\rm d}h_1 = \frac{1}{2\la\pi}\left[1+ \frac{1}{\bar v_2 \bar v_1}K\left(\frac{v_2}{v_1}\right)- \frac{v_1}{v_2} E\left(\frac{v_2}{ v_1}\right)\right], 
\]
and hence we get 
\begin{align}
\Lambda^{(2)}_{2,1} &= \frac{2\la_1\la_2}{\sqrt{\la}} (v_1-v_2) \frac{1}{2\la\pi} \left[\left(E\!\left(\frac{v_2}{ v_1}\right)-1\right) - \left(1+ \frac{1}{\bar v_2 \bar v_1}K\left(\frac{v_2}{v_1}\right)- \frac{v_1}{v_2} E\left(\frac{v_2}{ v_1}\right)\right)\right]\nn\\
&= \frac{\la_1\la_2}{\pi\la^{3/2}} (v_1-v_2)  \left[\left(1+\frac{v_1}{v_2}\right)E\!\left(\frac{v_2}{ v_1}\right)- \frac{1}{\bar v_2 \bar v_1}K\left(\frac{v_2}{v_1}\right)-2\right],
\label{eq:la21-2f}
\end{align}
which is a function of $v_1,v_2$ in terms of complete elliptic integral of second and first kind.
\end{proof}
\begin{remark}
Let use take 
\begin{align}
c(v_1,v_2):=\left(1+\frac{v_1}{v_2}\right)E\!\left(\frac{v_2}{ v_1}\right)- \frac{v_1^2-v_2^2}{v_1 v_2}K\left(\frac{v_2}{v_1}\right).
\label{eq:cv1v2}
\end{align}
In total, the handover frequency is
\begin{align}
\la_\Vcal&= \frac{4\sqrt{\la}}{\pi}  \sum_{l\in \{1,2\}} v_l\left(\frac{\la_l}{\la}\right)^2+2\sum_{l< r,k\in \{1,2\}}\La^{(k)}_{l,r} \nn\\
&=  \frac{4\sqrt{\la}}{\pi}  \sum_{l\in \{1,2\}} v_l\left(\frac{\la_l}{\la}\right)^2+2 \frac{\la_1\la_2}{\pi\la^{3/2}} \left[(v_1+v_2) c(v_1,v_2)+ (v_1-v_2) \left(c(v_1,v_2)-2\right)\right]\nn\\
&= \frac{4\sqrt{\la}}{\pi}\left[ \sum_{l\in \{1,2\}} v_l\left(\frac{\la_l}{\la}\right)^2+\frac{\la_1\la_2}{\la^{2}} (v_1 c(v_1,v_2)-v_1+v_2)\right].
\label{eq:total-HF2}
\end{align}
\end{remark}
\begin{remark}[$v_1=v_2=v$]
In this case, we take $\la_1=p\la$, $\la_2=(1-p)\la$ for some $p\in [0,1]$. Then we have $c(v,v)= 2E(1)=2$, since $E(1)=1$ using the Elliptic integral of second kind $E(\cdot)$ from (\ref{eq:Elliptic-int}). Hence the total handover frequency in this case is
\begin{align}
\la_\Vcal &= \frac{4\sqrt{\la}}{\pi}\left[ v\left(p^2+(1-p)^2\right)+ vp(1-p)  2E(1)\right] = \frac{4v\sqrt{\la}}{\pi} \left[p^2+(1-p)^2+2 p(1-p)\right] = \frac{4v\sqrt{\la}}{\pi}, \nn
\end{align}
which matches with the handover frequency in the single-speed scenario.
\end{remark}
\begin{remark}[$\la_1=\la/2=\la_2$]
Consider the case of equal proportions of stations of different speeds, i.e., $\la_1=\la/2=\la_2$. Then, from (\ref{eq:total-HF2}), we can simplify the total handover frequency as
\begin{align}
\la_\Vcal &= \frac{4\sqrt{\la}}{\pi}\left[ \sum_{l\in \{1,2\}} \frac{v_l}{4}+\frac{1}{4} (v_1 c(v_1,v_2)-v_1+v_2)\right]  = \frac{\sqrt{\la}}{\pi}\left[ v_1 c(v_1,v_2)+2v_2\right],  
\end{align}
where $c(v_1,v_2)$ is defined in (\ref{eq:cv1v2}).
\end{remark}
%


\section{Handover Palm distribution and applications: two-speed case}\label{sec-Distribution_birds_MS}
In this section, we focus again on the two-speed case. We first determine the Palm probability measure with respect to handovers with the goal to determine the Palm probability distribution of inter-handover times. An identical analysis allows one to come up with a result similar to the mass transport principle Lemma~\ref{lemma:VR-MTP} for the two-speed case. In these initial steps, we derive the results in the two-speed case, as a first step towards the multi-speed case. 
\subsection{Palm probability with respect to handovers}\label{subsection:Palm-Handover-2}
Without loss of generality, let us assume that $v_1>v_2$. Like in the single-speed case, we consider that the reference probability space is $(\Omega,\Fcal,\mathbb P)$, 
with $(\Omega,\Fcal)$ the canonical space of head point process on $\mathbb{H}^+$. Each $\omega\in \Omega$ is a realization of the point process 
$\Hcal=\sum_{l\in \{1,2\}} \Hcal^l$. We equip this measure space with the shift $\{\theta_t\}_{t\in \mathbb{R}}$, along the time axis, defined as
\begin{equation}
\theta_t(\Hcal)=\theta_t(\Hcal^1)+\theta_t(\Hcal^2)=\sum_i \delta_{(T^{(1)}_i-t,H^{(1)}_i)}+ \sum_i \delta_{(T^{(2)}_i-t,H^{(2)}_i)}.
\label{eq:shift_Hcalc12}
\end{equation}
The law $\mathbb P$ of the head point process $\Hcal$ is left invariant by this shift and the latter is also ergodic for $\mathbb P$. 
The handover point process $\Vcal$ is constructed as in Subsubsection~\ref{subsubsection:constV}, more precisely as in (\ref{eq:VcalHPP}) using the pure handover point processes $\Vcal_l$, for $l=1,2$, and the mixed handover point processes   $\Vcal^{(k)}_{l,m}$, for $l\neq m, k\in \{1,2\}$. The stationarity, or in other words the $\theta_t$-compatibility of $\Vcal$, is established in Lemma~\ref{lemma:stationary_Tcal2}, and $\Vcal$ has a positive and finite intensity from Theorem~\ref{theorem:HO2speed-main}. We denote the Palm probability w.r.t. $\Vcal$ on the probability space $(\Omega,\Fcal)$ as $\mathbb P_{\Vcal}^0$. In what follows, we describe a variant of the mass transport principle, Lemma~\ref{lemma:VR-MTP}, for the two-speed case.

Recall the point processes $\Rcal_l$, for $l\in \{1,2\}$, and $\Rcal^{(k)}_{l,r}$, for $l\neq r, k\in \{1,2\}$, corresponding to the right-most head points for different types of handovers. The same mass transport principle, Lemma~\ref{lemma:VR-MTP}, under the bijections defined in Lemma~\ref{lemma:vr} applies to all the individual Palm probability measures, stated in the following. 
\begin{lemma}
For all non-negative measurable functions $f$ on $(\Omega,\Fcal)$, we have:
\begin{enumerate}[(i).]
\item \label{beta-i} For all $l\in \{1,2\}$,
\begin{equation}
\E_{\Vcal_l}^0 [f (\Hcal)]=\E_{\Rcal_l}^0 \left[f (\theta_{\beta_l(0)}(\Hcal))\right], 
\label{eq:VR-i}
\end{equation}
where $\beta_l$ is the bijection  between the point processes $\Vcal_l$ and  $\Rcal_l$, corresponding to handovers of pure type $l$.
\item \label{beta-ijk} For all $l\neq r, k\in \{1,2\}$,
\begin{equation}
\E_{\Vcal^{(k)}_{l,r}}^0 [f (\Hcal)]=\E_{\Rcal^{(k)}_{l,r}}^0 \left[f (\theta_{\beta^{(k)}_{l,r}(0)}(\Hcal))\right], 
\label{eq:VR-ijk}
\end{equation}
where $\beta^{(k)}_{l,r}$ is the bijection between the point processes $\Vcal_{l,r}^{(k)}$ and $\Rcal_{l,r}^{(k)}$, corresponding to handovers of mixed type $\binom{k}{l,r}$.
\end{enumerate}
\label{lemma:VR-ijkC}
\end{lemma}
Hence, we have the following decomposition of $\E_{\Vcal}^0 [f (\Hcal)]$, for any non-negative measurable function $f$ on $(\Omega,\Fcal)$, based on these different types of handovers using the Palm probability measure for the sum of multi-type independent point processes (cf. \cite[Subsection 1.6, p.~36]{Baccelli-Bremaud}),
\begin{theorem}
For all non-negative measurable functions $f$ on $(\Omega,\Fcal)$,
\begin{align}
\E_{\Vcal}^0 [f(\Hcal)]= \sum_{l\in \{1,2\}}\frac{\Lambda_l}{\la_{\Vcal}} \mathbb E_{\Vcal_l}^0 [f(\Hcal)]+ \sum_{ l\neq r, k\in\{1,2\}} \frac{\Lambda^{(k)}_{l,r}}{\la_\Vcal}\E^0_{\Vcal_{l,r}^{(k)}} [f (\Hcal)],
\label{eq:DPalm1}
\end{align}
in terms of the pure and mixed Palm measures $\mathbb E_{\Vcal_l}^0$, for $l\in\{1,2\}$, and $\E_{\Vcal^{(k)}_{l,r}}^0$, for $l\neq r, k\in\{1,2\}$.
\label{thm:decompfV}
\end{theorem}
Using the Poisson structure of $\Hcal$, the pure handover Palm probability measures $\P^0_{\Vcal_l}$, for $l=1,2$, are expressed as follows, similarly to Lemma~\ref{lemma:Palm-Lcal} for the single-speed case:
\begin{lemma}
For $l=1,2$ and for all non-negative measurable functions $f$ on $(\Omega,\Fcal)$,
\begin{align}
\lefteqn{\E^0_{\Vcal_l} [f (\Hcal)]}\nn\\
& = 
\frac {4\la_l^2v_l^2} {\Lambda_{l}}
\int_{(\R^+)^3}
\E_\Hcal\left[ f\left(\theta_{\hat{s}(0,h,-t',h')}           \left(\Hcal^l+\delta_{(0,h)}+\delta_{(-t',h')}+\Hcal^m\right)\right) \prod_{i=1,2}\one_{\Hcal^i\left(E^{\hat{s}(0,h,-t',h'),v_i}_{\hat{h}(0,h,-t',h')}\right)=0}  \right] {\rm d}t' {\rm d}h' {\rm d}h ,
\nonumber
\end{align}
where, $m\neq l$ and $(\hat{s}(t,h,t',h'),\hat{h}(t,h,t',h'))$ is the coordinates of the intersection point of the two radial birds with head points $(t,h)$ and $(t',h')$.
\label{lemma:Palm-Lcal-ii}
\end{lemma}
As a general principle, for any pair of head points $(t_1,h_1),(t_2,h_2)\in \mathbb H^+$, the intersection point of the associated radial birds of the same type is denoted by $(\hat s(t_1,h_1,t_2,h_2), \hat h(t_1,h_1,t_2,h_2))$.
The intersection points are given by (\ref{eq:HOs}) and (\ref{eq:HO_distance1}) if the corresponding birds are of different types and we denote them by $(\hat s_k(t_1,h_1,t_2,h_2), \hat h_k(t_1,h_1,t_2,h_2))$ for $k=1,2$. The proof of Lemma~\ref{lemma:Palm-Lcal-ii} follows the same steps as those of Lemma~\ref{lemma:Palm-Lcal} and of part~(\ref{beta-ijk}) of Lemma~\ref{lemma:VR-ijkC}, except for the extra condition on  multiple half-ellipses being empty of head points corresponding to different speeds. Having said this, we skip the proof of the last result and move towards the more interesting case of mixed Palm probability measures $\P^0_{\Vcal_{l,r}^{(k)}}$, for $l,r,k \in \{1,2\}$ with $l\neq r$, for which we provide more details in the following:~\label{notation:hatsk1}~\label{notation:hats}
\begin{lemma}
For $l\neq r, k=1,2$ and for all non-negative measurable functions $f$ on $(\Omega,\Fcal)$,
\begin{align}
\lefteqn{ \frac{\Lambda^{(k)}_{l,r}}{4\la_1\la_2 v_1v_2} \E^0_{\Vcal^{(k)}_{l,r}} [f (\Hcal)]}\nn\\
& = \int_{0}^\infty  \int_{0}^h \int_{t^*}^\infty 
\mathbb{E}_\Hcal\Bigg[ f \left(\theta_{\hat{s}_k(0,h,-t',h')}(\Hcal^r+\delta_{(0,h)}+ \Hcal^l+\delta_{(-t',h')})\right) \prod_{i=1,2}\one_{\Hcal^i\left(E^{\hat{s}_k(0,h,-t',h'),v_i}_{\hat{h}_k(0,h,-t',h')}\right)=0}
	\Bigg] {\rm d}t'  {\rm d}h' {\rm d}h \nn\\ 
    &\, +  \int_{0}^\infty  \int_h^\infty \int_0^{\infty}
	\mathbb{E}_\Hcal\Bigg[ f \left(\theta_{\hat{s}_k(0,h,-t',h')}(\Hcal^r+\delta_{(0,h)}+ \Hcal^l+\delta_{(-t',h')})\right) \prod_{i=1,2}\one_{\Hcal^i\left(E^{\hat{s}_k(0,h,-t',h'),v_i}_{\hat{h}_k(0,h,-t',h')}\right)=0}
	\Bigg] {\rm d}t'  {\rm d}h' {\rm d}h ,
\nn
\end{align}
where, $(\hat{s}_k(t,h,t',h'),\hat{h}_k(t,h,t',h'))$
are the coordinates of the $k$-th intersection of the two radial birds with heads $(t,h)$ and $(t',h')$ of type $r$ and $l$, respectively, $t^*=\frac{1}{v_1v_2}(h^2-h'^2)^\half (v_1^2-v_2^2)^\half$.
\label{lemma:Palm-Lcal-ij}
\end{lemma}
\begin{remark}
Because of the integration range, the expression on the right integral does not include the region where there is no intersection of the radial birds with their heads at $(0,h)$ and $(-t',h')$. This is justified using the intersection criteria, Lemma~\ref{lemma:2int} and Lemma~\ref{lemma:2int-rev}. 
\end{remark}
Let us first derive the result for $l=2,r=1$ and $k=1,2$, i.e., for $\E^0_{\Vcal^{(k)}_{2,1}} [f(\Hcal)]$. We prove it in the other case $l=1,r=2$ and $k=1,2$, i.e., for $\E^0_{\Vcal^{(k)}_{1,2}} [f(\Hcal)]$, in Appendix~\ref{subsection:l1r2k}.
\begin{proof}[Proof of Lemma~\ref{lemma:Palm-Lcal-ij}, $l=2,r=1$ and $k=1,2$] The proof follows a similar steps as those of the proof of Lemma~\ref{lemma:Palm-Lcal}, under some adaptation. Suppose  $(T^1_i,H^1_i)$ and $(T^2_j,H^2_j)$ in the support of $\Hcal^1, \Hcal^2$, respectively, such that $T_j^2 \in D_\ell(T_i^1, H_i^1, H_j^2)$. Let
$(\hat{S}^{(k)}_{2,1},\hat{H}^{(k)}_{2,1})$, in short written as $(\hat{S}^{(k)},\hat{H}^{(k)})$, denote the coordinates of the $k$-th intersection of the two radial birds with these heads.
By (\ref{eq:VR-ijk}) and the definition of $\P_{\Rcal^{(k)}_{2,1}}^0$,
\begin{align}
\E^0_{\Vcal^{(k)}_{2,1}} [f (\Hcal)]
=\E^0_{\Rcal^{(k)}_{2,1}} \left[f (\theta_{\beta^{(k)}_{2,1}(0)}(\Hcal))\right]
& = \frac 1 {\Lambda_{2,1}^{(k)}}
\mathbb{E} \left[ \sum_{T^1_i\in \Rcal^{(k)}_{2,1} \,\mbox{:}\, 0\le T^1_i \le 1} f \left(\theta_{\beta^{(k)}_{2,1}(0)} \circ \theta_{T^1_i}(\Hcal)\right)\right]\nn\\
&=\frac 1 {\Lambda_{2,1}^{(k)}}
\mathbb{E} \left[ \sum_{T^1_i\in \Rcal^{(k)}_{2,1} \,\mbox{:}\, 0\le T^1_i \le 1} f \left(\theta_{T^1_i+\beta^{(k)}_{2,1}(T^1_i)}(\Hcal)\right) \right],
\label{eq:Last-Hijjj}
\end{align}
using the fact that $\theta_{\beta^{(k)}_{2,1}(0)} \circ \theta_{T^1_i}=\theta_{T^1_i+\beta^{(k)}_{2,1}(T^1_i)}$. The expectation in (\ref{eq:Last-Hijjj}) can be evaluated as

\begin{align}
\lefteqn{ \mathbb{E} \left[ \sum_{(T^1_i,H^1_i)\in \Hcal^1\,\mbox{:}\, 0\le T^1_i \le 1} 
\!\!\!\!\!\one_{\exists (T^2_j,H^2_j)\in \Hcal^2 \,\mbox{:}\, T^2_j\in D_\ell(T^1_i,H^1_i,H^2_j)}  f \left(\theta_{T^1_i+\beta^{(k)}_{2,1}(T^1_i)}(\Hcal)\right)  \prod_{l=1,2}\one_{\Hcal^l\left(E^{\hat{S}^{(k)},v_l}_{\hat{H}^{(k)}}\right)=0}\right]}\nn \\
& = \mathbb{E} \left[ \sum_{(T^1_i,H^1_i)\in \Hcal^1\,\mbox{:}\, 0\le T^1_i \le 1}\,\sum_{(T^2_j,H^2_j)\in \Hcal^2 \,\mbox{:}\, T^2_j\in D_\ell(T^1_i,H^1_i,H^2_j)}
f \left(\theta_{T^1_i+\beta^{(k)}_{2,1}(T^1_i)}(\Hcal)\right) \prod_{l=1,2}\one_{\Hcal^l\left(E^{\hat{S}^{(k)},v_l}_{\hat{H}^{(k)}}\right)=0}
\right]\nn \\
& = 
\mathbb{E} \left[ \sum_{(T^1_i,H^1_i)\in \Hcal^1 \,\mbox{:}\,  0\le T^1_i \le 1} \;
\sum_{(T^2_j,H^2_j)\in \Hcal^2 \,\mbox{:}\, T^2_j\in D_\ell(T^1_i,H^1_i,H^2_j)}
f \left(\theta_{\hat{S}^{(k)}}(\Hcal)\right) \prod_{l=1,2}\one_{\Hcal^l\left(E^{\hat{S}^{(k)},v_l}_{\hat{H}^{(k)}}\right)=0}
	  \right],
\label{eq:Last-Hijj}
\end{align}
since $T^1_i+\beta^{(k)}_{2,1}(T^1_i)=\hat S^{(k)}$. Applying the Campbell-Mecke formula for the factorial power of order 2 based on the two-point Palm probability measure $\P^{(t, h), (t',h')}_{\Hcal}=\P^{(t, h)}_{\Hcal^1}\otimes \P^{(t',h')}_{\Hcal^2}$, we can evaluate the expectation in (\ref{eq:Last-Hijj}) as
\begin{align}
\lefteqn{\hspace{-0.18in}4\la_1\la_2 v_1 v_2 \!\!
\int_{0}^1 \!\! \int_{0}^\infty \!\! \int_{0}^\infty\!\! \int_{D_\ell(t,h,h')} \!\!\! 
\mathbb{E}^{(t,h),(t',h')}_\Hcal\Bigg[f \left(\theta_{\hat{s}_k(t,h,t',h')}(\Hcal)\right)
\prod_{l=1,2}\one_{\Hcal^l\left(E^{\hat{s}_k(t,h,t',h'),v_l}_{\hat{h}_k(t,h,t',h')}\right)=0} 
\Bigg] {\rm d}t'  {\rm d}h' {\rm d}h {\rm d}t}\label{eq:Last1-H} \\
& \hspace{-0.18in}=4\la_1\la_2 v_1 v_2	\int_{0}^1 \int_{0}^\infty  \int_{0}^\infty \int_{D_\ell(t,h,h')} 
	\mathbb{E}_\Hcal\Bigg[ f \left(\theta_{\hat{s}_k(t,h,t',h')}(\Hcal^1+\delta_{(t,h)}+ \Hcal^2+\delta_{(t',h')})\right)\nn\\
    &\hspace{3.3in} \times \prod_{l=1,2}\one_{\Hcal^l\left(E^{\hat{s}_k(t,h,t',h'),v_l}_{\hat{h}_k(t,h,t',h')}\right)=0}
	\Bigg] {\rm d}t'  {\rm d}h' {\rm d}h {\rm d}t,
\label{eq:Last-Hij}
\end{align}
where the set $D_\ell(t,h,h')$ is defined as in (\ref{eq:DlDr}). In case $h<h'$, we have $D_\ell(t,h,h')=(-\infty, t]$ but in case $h\geq h'$, we have $D_\ell(t,h,h')= (-\infty, t-t^*]$, where $t^*$ is as defined in (\ref{eq:tstar}). This simplifies the multiple integral in (\ref{eq:Last-Hij}) to
\begin{align}
\lefteqn{\int_{0}^1 \!\!\int_{0}^\infty  \!\! \int_{0}^h \!\! \int_{-\infty}^{t-t^*} \!\!
\mathbb{E}_\Hcal\Bigg[ f \left(\theta_{\hat{s}_k(t,h,t',h')}(\Hcal^1+\delta_{(t,h)}+ \Hcal^2+\delta_{(t',h')})\right) \prod_{l=1,2}\one_{\Hcal^l\left(E^{\hat{s}_k(t,h,t',h'),v_l}_{\hat{h}_k(t,h,t',h')}\right)=0}
\Bigg] {\rm d}t'  {\rm d}h' {\rm d}h {\rm d}t}\nn\\ 
&+ \int_{0}^1\!\! \int_{0}^\infty\!\!  \int_h^\infty\!\! \int_{-\infty}^{t} \!\!
\mathbb{E}_\Hcal\Bigg[ f \left(\theta_{\hat{s}_k(t,h,t',h')}(\Hcal^1+\delta_{(t,h)}+ \Hcal^2+\delta_{(t',h')})\right)  \prod_{l=1,2}\one_{\Hcal^l\left(E^{\hat{s}_k(t,h,t',h'),v_l}_{\hat{h}_k(t,h,t',h')}\right)=0}
\Bigg] {\rm d}t'  {\rm d}h' {\rm d}h {\rm d}t,\nn
\end{align}
which can later be transformed to
\begin{align}
\lefteqn{\int_{0}^\infty  \int_{0}^h \int_{t^*}^\infty 
\mathbb{E}_\Hcal\Bigg[ f \left(\theta_{\hat{s}_k(0,h,-t',h')}(\Hcal^1+\delta_{(0,h)}+ \Hcal^2+\delta_{(-t',h')})\right)  \prod_{l=1,2}\one_{\Hcal^l\left(E^{\hat{s}_k(0,h,-t',h'),v_l}_{\hat{h}_k(0,h,-t',h')}\right)=0}
\Bigg] {\rm d}t'  {\rm d}h' {\rm d}h} \nn\\ 
&+  \int_{0}^\infty  \int_h^\infty\!\! \int_0^{\infty}\!\!
\mathbb{E}_\Hcal\Bigg[ f \left(\theta_{\hat{s}_k(0,h,-t',h')}(\Hcal^1+\delta_{(0,h)}+ \Hcal^2+\delta_{(-t',h')})\right) \prod_{l=1,2}\one_{\Hcal^l\left(E^{\hat{s}_k(0,h,-t',h'),v_l}_{\hat{h}_k(0,h,-t',h')}\right)=0}
\Bigg] {\rm d}t'  {\rm d}h' {\rm d}h ,
\label{eq:Last-Hiji}
\end{align}
by making the change of variable $t-t'$ to $t'$, where $t^*=\frac{1}{v_1v_2} (h^2-h'^2)^\half (v_1^2-v_2^2)^\half $. We obtain the result for $l=2,r=1$ and $k=1,2$, by combining the three equations (\ref{eq:Last-Hijj}), (\ref{eq:Last-Hij}) and (\ref{eq:Last-Hiji}). 
\end{proof}
\begin{remark}
For the unit function, i.e., $f(\Hcal)=1$, we can check that $E_{\Vcal}^0 [f (\Hcal)]=1$. Indeed, using Lemma~\ref{lemma:Palm-Lcal-ii} and part~(\ref{MHnp1}) of Theorem~\ref{thm:HOnspeed} we have
\begin{align}
\!\!\!\!\! \E^0_{\Vcal_l} [f (\Hcal)]
& = 
\frac {4\la_l^2v_l^2} {\Lambda_{l}}
\int_{(\R^+)^3}
\E_\Hcal\left[ \prod_{i=1,2} \one_{\Hcal^i\left(E^{\hat{s}(0,h,-t',h'),v_i}_{\hat{h}(0,h,-t',h')}\right)=0} \right] {\rm d}t' \, {\rm d}h'  \, {\rm d}h \nn\\
&=\frac {4\la_l^2v_l^2} {\Lambda_{l}}
\int_{(\R^+)^3}
e^{-\la \pi \hat{h}^2(0,h,-t',h')} {\rm d}t' \, {\rm d}h'  \, {\rm d}h  = \frac{4v_l\sqrt{\la}}{\pi \Lambda_{l}}  \left(\frac{\la_l}{\la}\right)^2 =1.
\end{align}
Using Lemma~\ref{lemma:Palm-Lcal-ij}, we also have
\begin{align}
\E^0_{\Vcal^{(k)}_{l,r}} [f (\Hcal)]
    %
%
& = 
\frac {4\la_1\la_2 v_1v_2} {\Lambda^{(k)}_{l,r}}
\int_{0}^\infty \left[ \int_{0}^h \int_{t^*}^\infty 
e^{-\la \pi\hat{h}^2_k(0,h,-t',h')} {\rm d}t'  {\rm d}h'  +    \int_h^\infty \int_0^{\infty}
	e^{-\la \pi\hat{h}^2_k(0,h,-t',h')} {\rm d}t'  {\rm d}h' \right]{\rm d}h=1,\nn
\end{align}
using the expression of $\Lambda^{(k)}_{l,r}$ from the integral in (\ref{eq:h4v}). Thus using  Theorem~\ref{thm:decompfV}, Lemma~\ref{lemma:VR-ijkC} and (\ref{eq:THF}), we have 
\begin{align}
E_{\Vcal}^0 [f (\Hcal)] &= \frac{1}{\la_{\Vcal}}  \sum_{l\in\{1,2\}} \Lambda_l
+ \frac{1}{\la_{\Vcal}}\sum_{l,r, k\in \{1,2\}: l\neq r} \Lambda^{(k)}_{l,r}=1.\nn
\end{align}
\end{remark}
\subsubsection{Decomposition of $\E_{\Vcal}^0 [f(\Hcal)]$ for symmetric functions $f$}\label{subsubsection:symmEf}
The time reversal argument also provides symmetric relations among the Palm probability measures $\E^0_{\Vcal^{(k)}_{l,r}} [f (\Hcal)]$, for pairs of values of $l\neq r, k\in \{1,2\}$. It is described in the forth-coming result, under a necessary condition that the function $f$ is {\em symmetric}, which we define as follows. Define $-\Hcal:=\sum_{i}\delta_{(-T_i, H_i)}$. We say that a function $f:\Omega \to \R^+$ is said to be {\em symmetric}, if $f(-\Hcal)=f(\Hcal)$. 
\begin{lemma} 
For all non-negative measurable and symmetric functions $f$ 
\begin{equation}
\E^0_{\Vcal^{(1)}_{1,2}} \left[f (\Hcal)\right]=\E^0_{\Vcal^{(2)}_{2,1}} \left[f (\Hcal)\right]
\text{ and }
\E^0_{\Vcal^{(2)}_{1,2}} \left[f(\Hcal)\right] =\E^0_{\Vcal^{(1)}_{2,1}} \left[f (\Hcal)\right].
\label{eq:symm12}
\end{equation}
\label{lemma:symm}
\end{lemma}
\begin{proof}
Let $\Lcal_{1,2}^{(k)}$ be the point process made of the abscissas of the left most head point corresponding to handovers of type $\binom{k}{1,2}$. Then, there exist bijections $\tilde \beta^{(k)}_{1,2}$ between $\Vcal^{(k)}_{1,2}$ and $\Lcal^{(k)}_{1,2}$, defined similarly to $\beta^{(k)}_{1,2}$, for $k=1,2$, as in part~(\ref{bijection-ijk}) of Lemma~\ref{lemma:vr}. Under the symmetry assumption on the measurable function $f$, 
we have the following relations:
\begin{align}
\E^0_{\Vcal^{(1)}_{1,2}} [f (\Hcal)]
&=\E^0_{\Lcal^{(1)}_{1,2}} \left[f (\theta_{\tilde \beta^{(1)}_{1,2}(0)}(\Hcal))\right]=\E^0_{\Rcal^{(2)}_{2,1}} \left[f (\theta_{\beta^{(2)}_{2,1}(0)}(\Hcal))\right]=\E^0_{\Vcal^{(2)}_{2,1}} [f (\Hcal)],
\label{eq:sym1}
\end{align}
\begin{align}
\E^0_{\Vcal^{(2)}_{1,2}} [f (\Hcal)]
&=\E^0_{\Lcal^{(2)}_{1,2}} \left[f (\theta_{\tilde \beta^{(2)}_{1,2}(0)}(\Hcal))\right]=\E^0_{\Rcal^{(1)}_{2,1}} \left[f (\theta_{\beta^{(1)}_{2,1}(0)}(\Hcal))\right]=\E^0_{\Vcal^{(1)}_{2,1}} [f (\Hcal)].
\label{eq:sym2}
\end{align}
In both relations, the third equality can be justified through a time-reversal argument, similarly to (\ref{eq:L12-L21}, proved in the following lemma proved in Appendix~\ref{subsection:L-betaLR}. This completes the proof.  
\end{proof}
\begin{lemma} For all non-negative measurable and symmetric functions $f$, we have
\begin{equation}
\!\!\!\!\!\!\E^0_{\Lcal^{(1)}_{1,2}} \left[f (\theta_{\tilde \beta^{(1)}_{1,2}(0)}(\Hcal))\right]
\!=\!\E^0_{\Rcal^{(2)}_{2,1}} \left[f (\theta_{\beta^{(2)}_{2,1}(0)}(\Hcal))\right],
\label{eq:betaLR1}
\end{equation}
\begin{equation}
\E^0_{\Lcal^{(2)}_{1,2}} \left[f (\theta_{\tilde \beta^{(2)}_{1,2}(0)}(\Hcal))\right] \!=\!\E^0_{\Rcal^{(1)}_{2,1}} \left[f (\theta_{\beta^{(1)}_{2,1}(0)}(\Hcal))\right].
\label{eq:betaLR2}
\end{equation}
\label{lemma:betaLR}
\end{lemma}
Under this symmetry in Lemma~\ref{lemma:symm}, we can re-write the decomposition by folding the symmetric terms for mixed handovers in (\ref{eq:DPalm1}), as 
\begin{align}
\E_{\Vcal}^0 [f(\Hcal)]= \sum_{l\in \{1,2\}}\frac{\Lambda_l}{\la_{\Vcal}} \mathbb E_{\Vcal_l}^0 [f(\Hcal)]+ 2 \sum_{ r<l, k\in\{1,2\}} \frac{\Lambda^{(k)}_{l,r}}{\la_\Vcal}\E^0_{\Vcal_{l,r}^{(k)}} [f (\Hcal)].
\label{eq:DPalm2}
\end{align}

\subsection{Typical handover distance in the two-speed case}\label{subsection:Palm-Handover-2H}
Let $\hat{H}$ be the random variable for the distance of the typical handover. Based on the different types of typical handovers, we have following result about the decomposition of the Laplace transform of $\hat{H}^2$ under the Palm probability measure $\P^0_\Vcal$. 
\begin{lemma} Under the Palm probability measure $\P_\Vcal^0$ of handovers, for any parameter $\gamma\geq 0$, the Laplace transform of $\hat{H}^2$ is
\begin{align}
\E^0_\Vcal\left[e^{-\gamma \hat{H}^2}\right] &= \left(1+\frac{\gamma}{\la\pi}\right)^{-3/2}
\label{eq:LTH_0-2sp}
\end{align}
and $\hat{H}$ follows the Nakagami distribution with parameters $\left(\frac{3}{2}, \frac{3}{2\la\pi}\right)$.
\label{lemma:typH-2sp}
\end{lemma}
\begin{proof}[Proof of Lemma~\ref{lemma:typH-2sp}]
We apply the decomposition of the Palm probability measure with respect to handovers as in Theorem~\ref{thm:decompfV}, with the help of Lemma~\ref{lemma:Palm-Lcal-ii}, Lemma~\ref{lemma:Palm-Lcal-ij}, for the symmetric function $f(\Hcal)= e^{-\gamma \hat H^2}$. The decomposition is according to the different types of handovers. Let $\hat{H}_{l}$ is the typical handover distance corresponding to the handover of pure type $l$, for $l\in\{1,2\}$. For $l\neq r, k\in\{1,2\}$, $\hat H_{l,r}^{(k)}$ denotes the typical handover distance due to the typical handover of type $\binom{k}{l,r}$. Observation~\ref{observation:equalHOdist} supports the fact that
\begin{equation}
    \hat H_{2,1}^{(1)} \stackrel{d}{=}\hat H_{1,2}^{(2)} \text{ and }\hat H_{2,1}^{(2)} \stackrel{d}{=}\hat H_{1,2}^{(1)}.
    \label{eq:eqdist}
\end{equation}
Then using the symmetry in (\ref{eq:eqdist}) and the formula (\ref{eq:DPalm2}), the Laplace transform of $\hat{H}^2$ under the Palm probability measure $\P^0_\Vcal$ is 
\begin{align}
   \E^0_\Vcal\left[e^{-\gamma \hat{H}^2}\right] &= \sum_{l\in \{1, 2\}}\frac{\Lambda_l}{\la_{\Vcal}}\E^0_{\Vcal_l}\left[e^{-\gamma \hat{H}^2_{l} }\right]
    + 2 \sum_{k\in \{1, 2\}}\frac{\Lambda^{(k)}_{2, 1}}{\la_{\Vcal}} \E^0_{\Vcal^k_{2,1}}\left[e^{-\gamma (\hat{H}^{(k)}_{2,1})^2 }\right].
    \label{eq:h1L-2}
\end{align}
for any parameter $\gamma\geq 0$. Using calculations similar to those in Equation~(\ref{eq:Lam1}) and Remark~\ref{remark:vhatH_0}, we can obtain that
\begin{align}
\E^{0}_{\Vcal_l}\left[e^{-\gamma \hat{H}^2_{l} }\right]&=\frac{1}{\Lambda_l}\left(\frac{\la_l}{\la}\right)^2\frac{4v_l\sqrt{\la}}{\pi}\left(1+\frac{\gamma}{\la\pi}\right)^{-3/2}=\left(1+\frac{\gamma}{\la\pi}\right)^{-3/2},
\label{eq:H0i}
\end{align}
since $\Lambda_l=\left(\frac{\la_l}{\la}\right)^2\frac{4v_l\sqrt{\la}}{\pi}$. Thus we can say that, for the pure handover of type $l$, the handover distance $\hat{H}_l$ follows a Nakagami distribution with parameters $\left(\frac{3}{2}, \frac{3}{2\la\pi}\right)$, for each $l=1,2$, cf. a similar result, Lemma~\ref{lemma:typical_h_HO} for the single speed case. In what follows, we evaluate the individual Laplace transforms $\E^{0}_{\Vcal^{(1)}_{2,1}}\left[e^{-\gamma (\hat{H}^{(1)}_{2,1})^2}\right]$ and $\E^{0}_{\Vcal^{(2)}_{2,1}}\left[e^{-\gamma (\hat{H}^{(2)}_{2,1})^2 }\right]$, which is enough due to the symmetry in (\ref{eq:eqdist}). Using the formula in Lemma~\ref{lemma:Palm-Lcal-ij}, we get
\begin{align}
\lefteqn{\hspace{-0.1in}\frac{\Lambda_{2,1}^{(1)}} {4\la_1\la_2v_1v_2}\E^0_{\Vcal^{(1)}_{2,1}}\left[e^{-\gamma (\hat{H}^{(1)}_{2,1})^2 }\right]}\nn\\
&=\int_{0}^\infty  \int_{0}^{h} \int_{t^*}^\infty 
\mathbb{E}_\Hcal\Bigg[ e^{-\gamma \hat{h}^2_1(0,h,-t',h')}  \prod_{l=1,2}\one_{\Hcal^l\left(E^{\hat{s}_1(0,h,-t',h'),v_l}_{\hat{h}_1(0,h,-t',h')}\right)=0}
\Bigg] {\rm d}t'  {\rm d}h' {\rm d}h \nn\\ 
&\qquad  +  \int_{0}^\infty  \int_{h}^\infty \int_0^{\infty}
\mathbb{E}_\Hcal\Bigg[e^{-\gamma \hat{h}^2_1(0,h,-t',h')} \prod_{l=1,2}\one_{\Hcal^l\left(E^{\hat{s}_1(0,h,-t',h'),v_l}_{\hat{h}_1(0,h,-t',h')}\right)=0}
\Bigg] {\rm d}t'  {\rm d}h' {\rm d}h \nn\\
&{=}\int_{0}^{\infty}\!\!\int_0^{h}\!\!\int_{t^*}^\infty \!\!\! e^{-(\gamma+\la\pi)\hat{h}^2_1(0,h,-t',h')}\,{\rm d}t' \, {\rm d}h' \, {\rm d}h {+}  \!\!\int_{0}^{\infty}\!\!\!\int_{h}^{\infty}\!\!\!\int_0^{\infty}\!\!\! e^{-(\gamma+\la\pi)\hat{h}^2_1(0,h,-t',h')}\,{\rm d}t' \ {\rm d}h' \ {\rm d}h,
\label{eq:H0-ij1}
\end{align}
where $t^* =\frac{1}{v_1v_2}(h^2-h'^2)^\half (v_1^2-v_2^2)^\half$, given $h,h'$. Using a computation similar to that in (\ref{eq:la21-1f}) with $\g+\la\pi$ instead of just $\la$, we have 
\begin{align}
\lefteqn{\hspace{-0.9in}4\la_1\la_2v_1v_2 \int_{0}^{\infty}\!\left[\int_0^{h}\!\!\int_{t^*}^\infty \!\!\! e^{-(\gamma+\la\pi)\hat{h}^2_1(0,h,-t',h')}\,{\rm d}t' \, {\rm d}h'  {+} \!\!\int_{h}^{\infty}\!\!\!\int_0^{\infty}\!\!\! e^{-(\gamma+\la\pi)\hat{h}^2_1(0,h,-t',h')}\,{\rm d}t' \ {\rm d}h' \right] {\rm d}h}\nn\\
&= \frac{\la_1\la_2 \sqrt{\pi}}{(\g+\pi\la)^{3/2}}(v_1+v_2) \left[\left(1+\frac{v_1}{v_2}\right)E\!\left(v_2/ v_1\right)- (v_1^2-v_2^2)\frac{K(v_2/v_1)}{v_1 v_2}\right].
\label{eq:Hod2}
\end{align}
Substituting from (\ref{eq:Hod2}) to (\ref{eq:H0-ij1}), we get
\begin{align}
\E^0_{\Vcal^{(1)}_{2,1}}\left[e^{-\gamma (\hat{H}^{(1)}_{2,1})^2 }\right] &= \frac{1}{\Lambda_{2,1}^{(1)}} \frac{\la_1\la_2 \sqrt{\pi}}{(\g+\pi\la)^{3/2}}(v_1+v_2) \left[\left(1+\frac{v_1}{v_2}\right)E\!\left(v_2/ v_1\right)- (v_1^2-v_2^2)\frac{K(v_2/v_1)}{v_1 v_2}\right]\nn\\
&= \frac{(\la\pi)^{3/2}}{(\g+\pi\la)^{3/2}} = \left(1+\frac{\g}{\la\pi}\right)^{-3/2},
\label{eq:H0-ij1x}
\end{align}
using the expression for $\Lambda_{2,1}^{(1)}$ from (\ref{eq:mhof1}). 

Similarly, we can rewrite the other term as
\begin{align}
\lefteqn{\frac{\Lambda_{2,1}^{(2)}}{4\la_1\la_2v_1v_2}\E^0_{\Vcal^{(2)}_{2,1}}\left[e^{-\gamma (\hat{H}^{(2)}_{2,1})^2 }\right]}\nn\\
&= \!\! \int_{0}^{\infty} \!\! \int_0^{h} \!\! \int_{t^*}^\infty  \!\!\!\! e^{-(\gamma+\la\pi)\hat{h}^2_2(0,h,-t',h')}\,{\rm d}t' \, {\rm d}h' \, {\rm d}h  {+} \!\!  \int_{0}^{\infty} \!\! \int_{h}^{\infty} \!\! \int_0^{\infty} \!\!\!\! e^{-(\gamma+\la\pi)\hat{h}^2_2(0,h,-t',h')}\,{\rm d}t' \, {\rm d}h' \, {\rm d}h,  
%
\label{eq:H0-ij2}
\end{align}
where $t^* =\frac{1}{v_1v_2}(h^2-h'^2)^\half (v_1^2-v_2^2)^\half$, given $h, h'$. Using the computation for $\Lambda_{2,1}^{(2)}$ from (\ref{eq:mhof2}) we have
\begin{align}
\lefteqn{\hspace{-1in}4\la_1\la_2v_1v_2 \int_{0}^{\infty}\!\!\int_0^{h}\!\!\int_{t^*}^\infty \!\!\! e^{-(\gamma+\la\pi)\hat{h}^2_2(0,h,-t',h')}\,{\rm d}t' \, {\rm d}h' \, {\rm d}h {+}  \!\!\int_{0}^{\infty}\!\!\!\int_{h}^{\infty}\!\!\!\int_0^{\infty}\!\!\! e^{-(\gamma+\la\pi)\hat{h}^2_2(0,h,-t',h')}\,{\rm d}t' \ {\rm d}h' \ {\rm d}h}\nn\\
&= \frac{\la_1\la_2 \sqrt{\pi}}{(\g+\pi\la)^{3/2}}(v_1-v_2) \left[\left(1+\frac{v_1}{v_2}\right)E\!\left(v_2/ v_1\right)- (v_1^2-v_2^2)\frac{K(v_2/v_1)}{v_1 v_2}-2\right].
\label{eq:Hod2x}
\end{align}
Substituting from (\ref{eq:Hod2x}) to (\ref{eq:H0-ij2}) we get
\begin{align}
\E^0_{\Vcal^{(2)}_{2,1}}\left[e^{-\gamma (\hat{H}^{(2)}_{2,1})^2 }\right] &= \frac{1}{\Lambda_{2,1}^{(2)}} \frac{\la_1\la_2 \sqrt{\pi}}{(\g+\pi\la)^{3/2}}(v_1-v_2) \left[\left(1+\frac{v_1}{v_2}\right)E\!\left(v_2/ v_1\right)- (v_1^2-v_2^2)\frac{K(v_2/v_1)}{v_1 v_2}-2\right]\nn\\ 
&=   \frac{(\la\pi)^{3/2}}{(\g+\pi\la)^{3/2}} = \left(1+\frac{\g}{\la\pi}\right)^{-3/2}.
\label{eq:HV21-2}
\end{align}
Hence, under their respective probability distribution, both $\hat{H}^{(1)}_{2,1}$ and $\hat{H}^{(2)}_{2,1}$ follow the Nakagami distribution with parameters $\left(\frac{3}{2}, \frac{3}{2\la\pi}\right)$. The final expression for $\Lcal^0_{\hat{H}^2}(\gamma)$ can be obtained by substituting the resultant from (\ref{eq:H0i}), (\ref{eq:H0-ij1x}) and (\ref{eq:H0-ij3}) to (\ref{eq:h1L-2}) as
\begin{align}
\E^0_\Vcal\left[e^{-\gamma \hat{H}^2}\right]  &= \sum_{l\in\{1,2\}}\frac{\Lambda_l}{\la_{\Vcal}} \left(1+\frac{\gamma}{\la\pi}\right)^{-3/2}
+2\sum_{k\in \{1,2\}} \frac{\Lambda^{(k)}_{2, 1}}{\la_{\Vcal}} \left(1+\frac{\gamma}{\la\pi}\right)^{-3/2}\nn\\
&= \frac{1}{\la_{\Vcal}} \left(1+\frac{\gamma}{\la\pi}\right)^{-3/2} \left[\sum_{l\in\{1,2\}} \Lambda_l+ 2\sum_{k\in \{1,2\}} \Lambda^{(k)}_{2, 1}\right] = \left(1+\frac{\gamma}{\la\pi}\right)^{-3/2}. 
%
\label{eq:H0-ij3}
\end{align}
%
Hence, under the Palm probability distribution with respect to $\Vcal$, $\hat{H}$ follow the Nakagami distribution with parameters $\left(\frac{3}{2}, \frac{3}{2\la\pi}\right)$.
\end{proof}
\subsection{Distribution of inter-handover time: two-speed case}\label{subsection:distributionT0_two}
The entire subsection is devoted to our final result about the Palm probability distribution of the {\em inter-handover time} $T$. It is based on an extension of the ideas in the proof of Theorem~\ref{theorem:T-Palm}, with the required adaptation due to the existence of different types of handovers at a typical time.
Under the Palm probability measure $\P^0_{\Vcal}$ of handovers, the {\em inter-handover time} is the duration between the handover at time $0$ and the next handover. We denote the inter-handover time by $T$. Before moving to the Palm probability distribution of the inter-handover time, we refer the reader to Subsubsection~\ref{subsubsection:geo-int} and Subsubsection~\ref{subsubsection:geo-int2} for the geometric interpretation of the intersection criteria (Subsubsection~\ref{subsubsection:IC}) and their applications. This enables us to properly manage the rise in the complexity resulting from the introduction of more than one speed.
\subsubsection{Inter-handover time} \label{subsubsection:int-HoT} Suppose for any $s\in \R$, $T_1(s)$ denotes the first time a handover happens after time $s$. To formally define the inter-handover time, suppose there is a handover at time $0$ and at distance $\hat h_0$. For ease of usage, we introduce an alternative notation for handover types. We denote the head points of the radial birds corresponding to the station that takes over the service and leaves the service, by $(t_n,h_n)$ and $(t_p,h_p)$, respectively. Here the subscript $n$ stands for {\em next} and $p$ stands for {\em previous}. We use the notation $\t_n=type(n)$, for the type of the next head point, and $\t_p=type(p)$, for the type of the previous head point. Let us define
\begin{equation}
q(t_n,t_p)\equiv q:=\begin{cases}
    1 & \text{ if } t_n\geq t_p,\\
    2& \text{ if } t_n< t_p. 
\end{cases}   
\end{equation}
Then we define $\bbinom{q}{\t_p,\t_n}$ to be a class or type of a handover where the next serving station is of type $\t_n$, the previous one is of type $\t_p$ and $q=1$ or $2$, depending on whether $t_n\geq t_p$ or $t_n<t_p$, respectively. \label{notation:bbinom1}\!\!\! The new notation for the handover types encodes geometric information about the head points responsible for handovers. The correspondence between the notation used so far and the new one is given below. Consider the pure handovers of type $l$ as $\binom{1}{l,l}$, for $l=1,2$. Then the bijection between the two notations is given in Table~\ref{tab:notation-t}.
\begin{table}[ht!]
\centering
\begin{spacedtable}{1.3}{6pt}
\begin{tabular}{|>{\centering\arraybackslash}m{2cm}|c|c|c|c|c|}
\hline
\multicolumn{1}{|c|}{} & 
\multicolumn{2}{c|}{\textbf{Old notation}} & 
\multicolumn{3}{c|}{\textbf{New notation}} \\ \hline

\multicolumn{1}{|c|}{} & 
\multicolumn{1}{c|}{\textbf{type}} & 
\multicolumn{1}{c|}{\textbf{HoPP}} &
\multicolumn{1}{c|}{\textbf{type}} & 
\multicolumn{1}{c|}{\textbf{$\t_n$}} & 
\multicolumn{1}{c|}{\textbf{HoPP}}
\\ \hline

\multirow{2}{=}{\centering\textbf{Pure handovers}} & $\binom{1}{1,1}$ & $\Vcal_{1}$ & $\bbinom{1}{1,1}$ &$\t_n=1$& $\Wcal_1$ \\ \cline{2-6}
 & $\binom{1}{2,2}$ & $\Vcal_{2}$ & $\bbinom{1}{2,2}$ &$\t_n=2$& $\Wcal_2$ \\ \cline{2-6}
\hline
\multirow{4}{=}{\centering\textbf{Mixed handovers}} & $\binom{1}{1,2}$ & $\Vcal^{(1)}_{1,2}$ & $ \bbinom{2}{2,1}$ &$\t_n=1$& $\Wcal^{(2)}_{2,1}$  \\ \cline{2-6}
 & $\binom{2}{1,2}$ & $\Vcal^{(2)}_{1,2}$ & $\bbinom{1}{1,2}$ &$\t_n=2$& $\Wcal^{(1)}_{1,2}$  \\ \cline{2-6}
 & $\binom{1}{2,1}$ & $\Vcal^{(1)}_{2,1}$ & $\bbinom{1}{2,1}$ &$\t_n=1$& $\Wcal^{(1)}_{2,1}$ \\ \cline {2-6}
& $\binom{2}{2,1}$ & $\Vcal^{(2)}_{2,1}$ & $\bbinom{2}{1,2}$ &$\t_n=2$& $\Wcal^{(2)}_{1,2}$   \\ \hline
\end{tabular}
\end{spacedtable}
\caption{Table of correspondence between old and new notations $\binom{k}{l,r}$ and $\bbinom{q}{\t_p,\t_n}$. Here HoPP stands for the handover point process.}
\label{tab:notation-t}
\end{table}

Suppose the typical handover is created by the intersection of the radial birds $C^{\t_n}_{(t_n,h_n)}$ and $C^{\t_p}_{(t_p,h_p)}$, for some head points $(t_n,h_n)$ and $(t_p,h_p)$ of type $\t_n$ and $\t_p$, respectively, for some $\t_n,\t_p\in \{1,2\}$. Note that $\t_n, \t_p$ can be equal. Beware that, unlike in Subsubsection~\ref{subsubsection:geo-int}, we do not fix here the type of the head points $(t_n,h_n)$ and $(t_p,h_p)$. The handover at time $0$ is pure and of type $\bbinom{1}{\t_p,\t_n}$, if $\t_p=\t_n$. If $\t_p\neq \t_n$, these two birds can lead to a handover of any of the types $\bbinom{1}{\t_p,\t_n}, \bbinom{2}{\t_p,\t_n}$, for $\t_p\neq \t_n\in \{1,2\}$. More precisely, for $\t_p\neq \t_n$ with $\t_n=1$ or $2$, if $t_n<t_p$, then the handover is mixed and can be of type $\bbinom{2}{2,1}$ or $\bbinom{2}{1,2}$, and if $t_n\geq t_p$, then it can be of type $\bbinom{1}{1,2}$ or  $\bbinom{1}{2,1}$.

From now onward, we only use the new notational system of $\bbinom{q}{\t_p,\t_n}$ for pure and mixed handovers. We write $v_{\t_n}$ and $v_{\t_p}$ in short as $v_n$ and $v_p$, respectively and write $C^n_{(t_n,h_n)}$ and $C^p_{(t_p,h_p)}$ in short for $C^{\t_n}_{(t_n,h_n)}$ and $C^{\t_p}_{(t_p,h_p)}$, respectively. Consider the set $\{\hat h_n(t)\}_{t\geq 0}$, where $\hat h_n(t)(\equiv \hat h_{\t_n}(t)):= \left(v_n^2(t-t_n)^2+h_n^2\right)^\half$, denotes the height of $C^n_{(t_n,h_n)}$ at time $t$. Essentially, the curve $\hat h_n(\cdot)$, as a function of time, represents the curve of the radial bird $C^{\t_n}_{(t_n,h_n)}$, which takes over. Observe that $\hat h_n(0)=(v_n^2t_n^2+h_n^2)^\half:=\hat h_0$. 

For $m\in \{1,2\}$ such that $m\neq \t_n$, let $\left\{E^{t,v_n}_{\hat h_n(t)}\right\}_{t\geq 0}$ and $\left\{E^{t,v_m}_{\hat h_n(t)}\right\}_{t\geq 0}$, be the two collections of open half-ellipses of heights $\{\hat h_n(t)\}_{t\geq 0}$. For $t\geq 0$, define the extra regions
\begin{equation}
S^n_t:=\interior{\left(\bigcup_{t'\in [0,t]}E^{t',v_n}_{\hat h_n(t')}\setminus E^{0,v_n}_{\hat h_0}\right)}\text{ and }S^m_t:=\interior{\left(\bigcup_{t'\in [0,t]}E^{t',v_m}_{\hat h_n(t')}\setminus E^{0,v_m}_{\hat h_0}\right)}.
\label{eq:SnSm}
\end{equation}
Due to their definition as unions of sets, we have:
\begin{lemma}
For $m\neq \t_n$, the collections of sets $\left\{S^n_t\right\}_{t\geq 0}$ and $\left\{S^m_t\right\}_{t\geq 0}$ are both non-decreasing in $t$.
\label{lemma:increasingSnm}
\end{lemma}
\begin{definition}[{\em Inter-handover time}]
Let $m\in\{1,2\}$ be such that $m\neq \t_n$. Then, the time of the next handover is defined as
\begin{align}
\! T&:=\inf\Big\{s>0: \forall t\in (0,s], \Hcal^n(S^n_t)=0, \Hcal^m(S^m_t)=0 \text{ and } \Hcal^n\left(\overline{S^n_s}\right)=2 \text{ or } \Hcal^m\left(\overline{S^m_s}\right)=1\Big\},
\label{eq:T1-02}
\end{align}
which coincides with $T_1(0)$, the time point of the first handover happens after time $0$. 
\label{definition:T1-2speed}
\end{definition}
Here is the rationale for the definition of the inter-handover time $T=T_1(0)$, essentially under the Palm distribution. As a consequence of the increasing property of the collections $\left\{S^n_t\right\}_{t\geq 0}$ and $\left\{S^m_t\right\}_{t\geq 0}$, presented in Lemma~\ref{lemma:increasingSnm}, we can conclude that   $\Hcal^n\left(S^n_{T_1(0)}\right)=0$ implies $\Hcal^n\left(S^n_t\right)=0$, and $\Hcal^m\left(S^m_{T_1(0)}\right)=0$ implies $\Hcal^m\left(S^m_t\right)=0$,  for all $t\in [0,T_1(0)]$. On the other hand, by the definition of $T_1(0)$ in (\ref{eq:T1-02}), the condition that $\Hcal^n\left(\overline{S^n_{T_1(0)}}\right)=2$ or $\Hcal^m\left(\overline{S^m_{T_1(0)}}\right)=1$ essentially implies that there exists a head point $(t_u,h_u)$, where $u$ stands for {\em upcoming}, of type $\t_u=\t_n$ or $m$ such that $C^n_{(t_n,h_n)}$ and $C^u_{(t_u,h_u)}$ intersect at $\left(T_1(0), \hat h_{T_1(0)}\right)$, $\Hcal^n\left(S^n_{T_1(0)}\right)=0$ and $\Hcal^m\left(S^m_{T_1(0)}\right)=0$, so that a handover happens at $T_1(0)$, and in addition to that, it is the first handover after $0$. Also, at $t=T_1(0)$, either $\Hcal^n\left(\overline{S^n_{T_1(0)}}\right)=2$ or $\Hcal^m\left(\overline{S^m_{T_1(0)}}\right)=1$, depending on whether $(t_u,h_u)$ is of type $\t_u=\t_n$ or type $\t_u=m$. 

The point $(t_u,h_u)$ can again be the previously considered head point $(t_p,h_p)$ in certain situations, where $\t_n\neq \t_p$, in which case we have $\t_u=\t_p=m$ and the condition $\Hcal^m\left(\overline{S^m_{T_1(0)}}\right)=1$, in the definition of $T_1(0)$, is anyway met. This situation appears when the typical handover is given by the first intersection of two radial birds of different types and the next handover is given by the second intersection of the same pair. The reader should refer to Figure~\ref{figure:mixed_birdij2} and Figure~\ref{figure:mixed_bird122} in the later part of this article, for better understanding of this scenario. Thus we also have the following corollary of Lemma~\ref{lemma:increasingSnm}, similar to Corollary~\ref{corollary:third-pt}, for the two-speed case:
\begin{corollary}
For the time $s$ meeting the stopping condition in the definition of $T_1(0)$ in (\ref{eq:T1-02}), there exists an a.s. unique point  $(t_u,h_u)\in \mathbb H^+$, such that the  radial bird with a head at this point, intersects the radial bird $C^n_{(t_n,h_n)}$ and produces the handover at time $T_1(0)$.
\label{corollary:third-pt2}
\end{corollary}
\begin{remark}
The specific sub-domain of $\mathbb H^+$ where to encounter the head point in question above, will be determined in the individual cases of the Palm distribution, depending on the type of the handover. We call these sub-domains the \textbf{unexplored regions}.
\end{remark}
\subsubsection{Monotonicity properties of the collection $\left\{E^{t,v_n}_{\hat h_n(t)}\setminus E^{0,v_n}_{\hat h_0}\right\}_{t\geq 0}$} Recall the setting of Subsubsection~\ref{subsubsection:int-HoT}, where there exists a handover at time $0$ such that a station corresponding to type $\t_n$ takes over as a serving station after time $0$.   
In the following, we discuss the monotonicity properties of the collections $\left\{E^{t,v_n}_{\hat h_n(t)}\setminus E^{0,v_n}_{\hat h_0}\right\}_{t\geq 0}$ and $\left\{E^{t,v_m}_{\hat h_n(t)}\setminus E^{0,v_m}_{\hat h_0}\right\}_{t\geq 0}$ for $m\neq \t_n$, under all circumstances.
\begin{proposition} Under the foregoing assumption, we have the following monotonicity properties:
\begin{enumerate}[(i).]
    \item \label{monotonic1} If $\t_n=1$, both collections of sets $\left\{E^{t,v_1}_{\hat h_1(t)}\setminus E^{0,v_1}_{\hat h_0}\right\}_{t\geq 0}$ and $\left\{E^{t,v_2}_{\hat h_1(t)}\setminus E^{0,v_2}_{\hat h_0}\right\}_{t\geq 0}$, where $\hat h_1(t):= \left(v_1^2(t-t_n)^2+h_n^2\right)^\half$, are non-decreasing in $t$.
    \item \label{monotonic2} If $\t_n=2$, the collection of sets $\left\{E^{t,v_2}_{\hat h_2(t)}\setminus E^{0,v_2}_{\hat h_0}\right\}_{t\geq 0}$, where $\hat h_2(t):= \left(v_2^2(t-t_n)^2+h_n^2\right)^\half$, is non-decreasing, but the collection of sets $\left\{E^{t,v_1}_{\hat h_2(t)}\setminus E^{0,v_1}_{\hat h_0}\right\}_{t\geq 0}$ is not always monotonic in $t$. 
\end{enumerate}
\label{prop:monotonic-sets}
\end{proposition}
We provide a proof of this monotonicity in Appendix~\ref{subsection:inc-wing2} and Appendix~\ref{subsection:inc-wing3}, with the help of the
{\em increasing wing property} discussed below for the two-speed case, similar to the one for the single-speed case.

\begin{remark}[Increasing wing property]
Under the Palm distribution of handovers, let $(0,\hat h_0)\in C^n_{(t_n,h_n)}$. Define $\hat h_n(t):=\left(v_n^2(t-t_n)^2+h_n^2\right)^\half$. In case $\t_n=1$, both collections of sets $\left\{E^{t,v_1}_{\hat h_1(t)}\setminus E^{0,v_1}_{\hat h_0}\right\}_{t\geq t_1}$ and $\left\{E^{t,v_2}_{\hat h_1(t)}\setminus E^{0,v_2}_{\hat h_0}\right\}_{t\geq t_1}$, are increasing, since for any $t\geq t_1$ and the point $(t,\hat h_1(t))$ lies on the increasing wing of the radial bird $C^1_{(t_1,h_1)}$, see picture~(\subref{subfigure:inc-wing21}) of Figure~\ref{figure:inc-wing2}. For $\t_n=2$, the collection $\left\{E^{t,v_2}_{\hat h_2(t)}\setminus E^{0,v_2}_{\hat h_0}\right\}_{t\geq t_1}$ is increasing, whereas the collection $\left\{E^{t,v_1}_{\hat h_2(t)}\setminus E^{0,v_1}_{\hat h_0}\right\}_{t\geq t_1}$ is non-monotonic, see picture~(\subref{subfigure:inc-wing22}) of Figure~\ref{figure:inc-wing2}. Indeed, for any $\eps>0$ arbitrarily small, and $t\geq t_n$, we have $E^{t,v_1}_{\hat h_2(t)}\setminus E^{0,v_1}_{\hat h_0}\not\subset E^{t+\eps,v_1}_{\hat h_2(t+\eps)}\setminus E^{0,v_1}_{\hat h_0}$.
\label{remark:inc-arm2}
\end{remark}
\begin{figure}[ht!]
\qquad
 \centering
      \begin{subfigure}[t]{0.45\linewidth}
       \centering
      \begin{tikzpicture}[scale=0.7, every node/.style={scale=0.7}]
\pgftransformxscale{0.55}  
    \pgftransformyscale{0.55}    
     \draw[->] (-3, 0) -- (10, 0) node[right] {$t$};
    \draw[red, domain=0.5:3.8, smooth] plot (\x, {(16*\x*\x-64*\x+64+2.5^2)^0.5});
    \draw[red](2,2.5) node{$\bullet$};
    \draw[](1.2,2) node{$(t_1,h_1)$};
    \draw[](1.6,5.2) node{$(0,\hat h_0)$};
     \draw[red]  (-0.15,0) arc (-180:-360:1.163 and 4.7);
     \draw[blue]  (-2,0) arc (-180:-360:3.03 and 4.7);
     \draw[red]  (5.4,0) arc (0:180:1.733 and 7);
     \draw[blue]  (8.2,0) arc (0:180:4.5 and 7);
     \draw[red]  (1.85,0) arc (-180:-360:1.25 and 5.05);
     \draw[blue]  (-0.15,0) arc (-180:-360:3.247 and 5.05);
    \draw[red](1.1,0.6) node{$E^{0, v_1}_{\hat h_0}$};
    \draw[red](3,0.6) node{$E^{\hat s_1, v_1}_{\hat h_1}$};
     \draw[red](6.2,0.6) node{$E^{\hat s_2,v_1}_{\hat h_2}$};
     \draw[blue](-2.5,1.5) node{$E^{0,v_2}_{\hat h_0}$};
     \draw[blue](6.8,2) node{$E^{\hat s_1, v_2}_{\hat h_1}$};
     \draw[blue](9,1.5) node{$E^{\hat s_2,v_2}_{\hat h_2}$};
    \end{tikzpicture}
\captionsetup{width=0.95\linewidth}
    \caption{For $\t_n=1$ and $t\geq t_1$, the regions $E^{t,v_l}_{h_1(t)}\setminus E^{0,v_l}_{\hat h_0}$, for $l=1,2$, are increasing. We just draw two sets of half-ellipses at time $t=\hat s_1, \hat s_2$, where $h_1(\hat s_1)=\hat h_1$ and $h_1(\hat s_2)=\hat h_2$.}
\label{subfigure:inc-wing21}
    \end{subfigure}
    \hspace{0.1in}
    \centering
      \begin{subfigure}[t]{0.45\linewidth}
       \centering
      \begin{tikzpicture}[scale=0.7, every node/.style={scale=0.7}]
\pgftransformxscale{0.7}  
\pgftransformyscale{0.7}    
 \draw[->] (-2.5, 0) -- (8, 0) node[right] {$t$};
   \draw[blue, domain=-2.2:3.8, smooth] plot (\x, {((3/2)^2*\x*\x+2^2)^0.5});
    \draw[blue](0,2) node{$\bullet$};
    \draw[](-1,1.6) node{$(t_2,h_2)$};
    \draw[blue]  (-0.53,0) arc (-180:-360:1.89 and 2.9);
    \draw[red]  (0.58,0) arc (-180:-360:0.84 and 2.9);
    \draw[blue]  (-0.38,0) arc (-180:-360:2.48 and 3.8);
    \draw[red]  (1.1,0) arc (-180:-360:1.1 and 3.8);
     \draw[blue]  (6.5,0) arc (0:180:3.4 and 5.05);
     \draw[red]  (1.7,0) arc (-180:-360:1.462 and 5.05);
    \draw[blue](-1.2,0.6) node{$E^{0, v_2}_{\hat h_0}$};
    \draw[red](1.3,0.6) node{$E^{0,v_1}_{\hat h_0}$};
    \draw[red](3.8,1.3) node{$E^{\hat s_1,v_1}_{\hat h_1}$};
     \draw[blue](3.4,3.8) node{$E^{\hat s_1,v_2}_{\hat h_1}$};
     \draw[blue](6,4) node{$E^{\hat s_2,v_2}_{\hat h_2}$};
     \draw[red](5.1,1.6) node{$E^{\hat s_2,v_1}_{\hat h_2}$};
    \end{tikzpicture}
\captionsetup{width=0.95\linewidth}
    \caption{For $\t_n=2$ and $t\geq t_2$, the regions $E^{t,v_2}_{h_2(t)}\setminus E^{0,v_2}_{\hat h_0}$ are increasing, where as $E^{t,v_1}_{h_2(t)}\setminus E^{0,v_1}_{\hat h_0}$ is non-monotonic. We just draw two sets of half-ellipses at time $t=\hat s_1, \hat s_2$, where $h_2(\hat s_1)=\hat h_1$ and $h_2(\hat s_2)=\hat h_2$.}
\label{subfigure:inc-wing22}
\end{subfigure}
\captionsetup{width=0.95\linewidth}
    \caption{For $(t_n,h_n)$ of type $\t_n=1$ or $2$ and for $t\geq t_n$, the pairs of upper half-ellipses $\left(E^{t,v_1}_{h_1(t)}, E^{t,v_2}_{h_1(t)}\right)$ or $\left(E^{t,v_1}_{h_2(t)}, E^{t,v_2}_{h_2(t)}\right)$, respectively, benefit from the increasing wing property.}
\label{figure:inc-wing2}
\end{figure}
\subsubsection{Palm distribution of inter-handover time}\label{subsubsection:PD-IHT} The following theorem determines the Laplace transform of $T$, which is equal to $T_1(0)$, under the Palm probability distribution $\P^0_\Vcal$ in terms of the head point process $\Hcal$. Using this alternative nomenclature for handover types, for pure handovers of type $i$ or equivalently $\bbinom{1}{i,i}$, we denote the pure handover point processes and their intensities as $\Wcal_i$ and $L_i$, for $i=1,2$, respectively. \label{notation:Wcal1} \label{notation:L1}\!\!\!\!\!\!  Similarly, we denote the mixed handover point processes and their intensities as $\Wcal^{(q)}_{i,j}$ and $L^{(q)}_{i,j}$, with $\t_p=i$ and $\t_n=j\neq i$, respectively, for mixed handovers of type $\bbinom{q}{i,j}$. \label{notation:Wcalijq} \label{notation:Lijq}\!\!\!\!\!\! We refer to Table~\ref{tab:notation-t} for the bijection between the two sets of notations. We write the handover point process as
\[
\Vcal\equiv \Wcal:=\sum_{i\in \{1,2\}} \Wcal_{i}+ \sum_{i\neq j, q\in \{1,2\}} \Wcal^{(q)}_{i,j},
\]
and $\la_{\Wcal}=\la_{\Vcal}$.~\label{notation:Wcal}~\label{notation:laW}
\begin{theorem}
Let $\rho\geq 0$. In the two-speed case,  under the Palm probability measure $\P^0_{\Vcal}$ of handovers, the Laplace transform of inter-handover time $T$ is given by,
\begin{equation}
\E_{\Wcal}^0\left[e^{-\rho T}\right]= \sum_{i\in \{1,2\}}\frac{L_i}{\la_{\Wcal}} \E^0_{\Wcal_i} [e^{-\rho T}] +  \sum_{i\neq j, q\in \{1,2\}} \frac{L^{(q)}_{i,j}}{\la_\Wcal} \E^0_{ \Wcal_{i,j}^{(q)}}[e^{-\rho T}].
\label{eq:decompLT2}
\end{equation}
\label{theorem:T-Palm-MS}
\end{theorem}
\begin{proof}[Proof of Theorem~\ref{theorem:T-Palm-MS}] The proof is based on the application of Theorem~\ref{thm:decompfV} to the decomposition of the Palm probability measure. To get the result, we apply the decomposition (\ref{eq:DPalm1}) to the function $f(\Hcal)= e^{-\rho T}$, for any $\rho\geq 0$. The terms $\E^0_{\Wcal_i} [e^{-\rho T}]$, for $i\in \{1,2\}$, correspond to the Palm probability measures for pure handovers of type $\bbinom{1}{i,i}$ and $\E^0_{\Wcal_{i,j}^{(q)}}[e^{-\rho T}]$, for $i\neq j,q\in \{1,2\}$, correspond to the Palm probability measures for mixed handovers of type $\bbinom{q}{i,j}$, which are determined in the following. 
\end{proof}
\begin{lemma}
For any $\rho\geq 0$, under the Palm probability measure of pure handovers of type $\bbinom{1}{i,i}$ with $\t_p=\t_n=i$, the Laplace transform of $T$ is
\begin{align}
\E^0_{\Wcal_i}\left[e^{-\rho T}\right]
&= 
\frac {4\la_i^2v_i^2} {L_i}
\int_{(\R^+)^3} \zeta_{i,i}(h,t',h'){\rm d}t' {\rm d}h' {\rm d}h,
\label{eq:T1-l}
\end{align}
for $i=1,2$, where the function $\zeta_{1,1}(h,t',h')$ is determined in (\ref{eq:T12d}) and $\zeta_{2,2}(h,t',h')$, in (\ref{eq:zeta1}).  
\label{lemma:LT1-l}
\end{lemma}
\begin{lemma}
For $i\neq j, q\in \{1,2\}$, under the Palm probability measure of mixed handovers of type $\bbinom{q}{i,j}$, the Laplace transform of $T$ is
\begin{equation}
\E^0_{\Wcal_{i,j}^{(q)}}\left[e^{-\rho T}\right]=\frac{4\la_1\la_2 v_1v_2}{L^{(q)}_{i,j}} \, \xi_{i,j}^{q}(\rho,v_i,v_j), 
\label{eq:Tlm-k}
\end{equation}
where the functions $\xi_{i,j}^{(q)}(\rho,v_i,v_j)$ are determined in (\ref{eq:T12j}), (\ref{eq:T12m}), (\ref{eq:T121g}) and (\ref{eq:T1225}), respectively, for the four different cases for  $i\neq j, q\in \{1,2\}$.
\label{lemma:LT_1ijk} 
\end{lemma}
\begin{remark}
In the expressions of the mixed Palm probability measures, the functions $\xi_{i,j}^{(q)}(\rho,v_i,v_j)$ for $i\neq j, q\in \{1,2\}$, are also functions of the underlying parameters $\la_1, \la_2$. We skip the dependence in $\la_1,\la_2$ for brevity.   
\end{remark}
The derivation of  $\E^0_{\Wcal_{i}} [e^{-\rho T}]$ with $i\in \{1,2\}$,  $\t_p=\t_n=i$, can be found in Subsubsection~\ref{subsubsection:Lemma111} for the case $\bbinom{1}{1,1}$ and Appendix~\ref{subsection:Lemma122} for the case $\bbinom{1}{2,2}$. We derive expressions for  $\E^0_{\Wcal^{(q)}_{i,j}} [e^{-\rho T}]$ with  $i\neq j, q\in \{1,2\}$ in Subsubsection~\ref{subsubsection:Lemma121} for the case $\bbinom{1}{2,1}$, Subsubsection~\ref{subsubsection:Lemma212} for $\bbinom{2}{1,2}$, Appendix~\ref{subsection:Lemma221} for $\bbinom{2}{2,1}$, and Appendix~\ref{subsection:Lemma112} for $\bbinom{1}{1,2}$, respectively.
\subsubsection{Handover pair type} For ease of computation in the proof of last results, as a general principle, we denote by $\bbinom{q,q'}{\t_p, \t_n, \t_u}$ the type of all pairs of consecutive handovers with $\bbinom{q}{\t_p, \t_n}$ the type of the first handover, $\t_u$ the type of the radial bird corresponding to the upcoming station and \label{notation:bbinom2}
\begin{equation}
q':=\begin{cases}
    1 & \text{ if } t_u\geq t_n,\\
    2& \text{ if } t_u< t_n. 
\end{cases} 
\label{eq:pair-Ho}
\end{equation}
In this composite notation $\bbinom{q,q'}{\t_p,\t_n,\t_u}$ of a pair of consecutive handovers, $\bbinom{q'}{\t_n,\t_u}$ denotes the type of the next handover. This notation encodes the information about the types and relative positions of three head points involved in two consecutive handovers. Table~\ref{tab:notation-tu} describes all possible pairs of consecutive handovers that can be produced for each type of handover. 
\begin{table}[ht!]
\centering
\begin{spacedtable}{1.4}{6pt}
\begin{tabular}{|>{\centering\arraybackslash}m{2.7cm}|c||c|c|c|c|c|}
\hline
\multicolumn{1}{|c|}{} & 
\multicolumn{1}{c||}{\textbf{Type}} & 
\multicolumn{2}{c|}{\textbf{Pair}, $\t_u=1$}& \multicolumn{3}{c|}{\textbf{Pair}, $\t_u=2$}  \\ \hline
\multicolumn{1}{|c|}{} & 
\multicolumn{1}{c||}{-} & 
\multicolumn{2}{c|}{-}& \multicolumn{2}{c|}{-} &\multicolumn{1}{c|}{Reappearance}  \\ \hline
\multirow{2}{=}{\centering\textbf{Pure handovers} $\bbinom{q}{\t_p,\t_n} $, $\t_p=\t_n$} & $\bbinom{1}{1,1}$ & $\bbinom{1,1}{1,1,1}$ & - & $\bbinom{1,1}{1,1,2}$ & $\bbinom{1,2}{1,1,2}$ &- \\ \cline{2-7}
 & $\bbinom{1}{2,2}$ & $\bbinom{1,1}{2,2,1}$ & $\bbinom{1,2}{2,2,1}$ &$\bbinom{1,1}{2,2,2}$& -&- \\ \cline{2-7}
\hline
\multirow{4}{=}{\centering\textbf{Mixed handovers} $\bbinom{q}{\t_p,\t_n}$, $\t_p\neq \t_n$} & $\bbinom{1}{2,1}$ & $\bbinom{1,1}{2,1,1}$ & - & $\bbinom{1,2}{2,1,2}$ & $\bbinom{1,1}{2,1,2}$ &$\bbinom{1,2}{2,1,2^*}$ \\ 
\cline {2-7}
& $\bbinom{2}{1,2}$ & $\bbinom{2,1}{1,2,1}$ & - & $\bbinom{2,1}{1,2,2}$ & - & -   \\
\cline{2-7}
 & $ \bbinom{2}{2,1}$ & $\bbinom{2,1}{2,1,1}$ & - & $ \bbinom{2,1}{2,1,2}$ & - & $ \bbinom{2,1}{2,1,2^*}$ \\ \cline{2-7}
 & $\bbinom{1}{1,2}$ & $\bbinom{1,2}{1,2,1}$ & $\bbinom{1,1}{1,2,1}$ &$\bbinom{1,1}{1,2,2}$& -&-  \\ 
  \hline
\end{tabular}
\end{spacedtable}
\captionsetup{width=0.9\linewidth}
\caption{Table of pair types $\bbinom{q,q'}{\t_p,\t_n,\t_u}$, associated to each individual type $\bbinom{q}{\t_p,\t_n}$. The special case: $\bbinom{q,q'}{\t_p,\t_n,\t_p^*}$, denotes the reappearance phenomenon where the previous bird coincides with the upcoming bird.}
\label{tab:notation-tu}
\end{table}
This table also summarizes the strategy to sort out the Palm probability distribution of inter-handover time, depending on the type of initial handover and the type of the radial bird producing the next handover. There can be a scenario where the same station of type $\t_p$, that was a serving station previously, becomes the upcoming serving station again. We denote the pair type in this situation as $\bbinom{q,q'}{\t_p,\t_n,\t_p^*}$, where we use $\cdot^*$ to denote this reappearance of the same station corresponding to $\t_p$, as the upcoming serving station.
\begin{remark}
Table~\ref{tab:transition-mat}, representing the adjacency matrix, summarises the possible transitions among the handover types. We use $1$ to denote the possible transitions and $0$ to denote the impossible transitions. We use $\cdot^*$ to denote that the transition can happen in two different ways, including the one with reappearance in our context.
\begin{table}[ht!]
\centering
\begin{spacedtable}{1.2}{2pt}
\begin{tabular}{|c|c|c|c|c|c|c|}
\hline
    \textbf{Types} & $\bbinom{1}{1,1}$ & $\bbinom{1}{2,2}$ & $\bbinom{1}{1,2}$& $\bbinom{1}{2,1}$ & $\bbinom{2}{1,2}$&$\bbinom{2}{2,1}$ \\ \hline
$\bbinom{1}{1,1}$ & $1$ & $0$ & $1$ & $0$ & $1$ & $0$ \\\hline
$\bbinom{1}{2,2}$ & $0$ & $1$ & $0$ & $1$ & $0$ & $1$ \\\hline
$\bbinom{1}{1,2}$ & $0$ & $1$ & $0$ & $1$ & $0$ & $1$ \\\hline
$\bbinom{1}{2,1}$ & $1$ & $0$ & $1$ & $0$ & $1^*$ & $0$ \\\hline
$\bbinom{2}{1,2}$ & $0$ & $1$ & $0$ & $1$ & $0$ & $0$ \\\hline
$\bbinom{2}{2,1}$ & $1$ & $0$ & $1^*$ & $0$ & $0$ & $0$ \\\hline
\end{tabular}
\end{spacedtable}
\captionsetup{width=0.9\linewidth}
\caption{All possible transitions among the handover types, where the transitions are from each type on the left column to each type on the top row.}
\label{tab:transition-mat}
\end{table}
One can also obtain the transition diagram based on the adjacency matrix among the handover types from Table~\ref{tab:transition-mat}, see Figure~\ref{figure:tra-dia}.
\vspace{-0.1in}
\begin{center}
\begin{figure}[ht!]
\begin{tikzpicture}[
    ->, 
    >=stealth', 
    shorten >=1pt, 
    auto, 
    node distance=1.2cm, 
    semithick,
    state/.style={circle, draw, minimum size=1.2em, font=\bfseries, scale=0.6, every node/.style={scale=0.75}} 
]
\pgftransformxscale{0.15}  
\pgftransformyscale{0.15}  
\node[state] (S1) {$\bbinom{1}{1,1}$};
\node[state] (S2) [right=of S1] {$\bbinom{1}{2,2}$};
\node[state] (S3) [below right=of S2] {$\bbinom{1}{1,2}$};
\node[state] (S4) [below left=of S3] {$\bbinom{1}{2,1}$};
\node[state] (S5) [left=of S4] {$\bbinom{2}{1,2}$};
\node[state] (S6) [above left=of S5] {$\bbinom{2}{2,1}$};


\path (S2) edge node {} (S6);
\path (S3) edge node {} (S2);
\path (S3) edge [bend right=15] node {} (S4);
\path (S3) edge [bend right=15] node {} (S6);
\path (S4) edge node {} (S1);
\path (S4) edge[bend right=20] node {} (S3);
\path (S4) edge [bend right=15] node[above]{$*$} (S5);
\path (S5) edge [bend right=15] node {} (S4);
\path (S5) edge node {} (S2);
\path (S6) edge [bend right=15] node[below]{$*$} (S3); 
\path (S6) edge node {} (S1);

\path (S1) edge [bend left=0] node[swap] {} (S3); 
\path (S1) edge [bend right=0] node[swap] {} (S5); 
\path (S2) edge [bend left=0] node[swap] {} (S4); 
\path (S1) edge [loop above] node {} (S1);
\path (S2) edge [loop above] node {} (S2);
\end{tikzpicture}
\captionsetup{width=0.9\linewidth}
\captionof{figure}{Transition diagram for all possible transitions among the handover types, where $*$ is used on two of the directed edges, to denote that the transition can happen in two possible ways.}
\label{figure:tra-dia}
\end{figure}
\end{center}
\end{remark}

\textbf{Proofs of Lemma~\ref{lemma:LT1-l} and Lemma~\ref{lemma:LT_1ijk}: }In determining the individual Palm expectations in (\ref{eq:decompLT2}), we frequently use the coordinates of the pure and mixed intersections of the radial birds. Suppose $(t_1,h_1), (t_2,h_2)\in \mathbb H^+$ are two head points of same or different types. From now onward, the subscript used in the notation of the head point, does not correspond to it's type, unlike in Subsubsection~\ref{subsubsection:int-HoT}. The intersection point is denoted by $(\hat s(t_1,h_1,t_2,h_2), \hat h(t_1,h_1,t_2,h_2))$, when the two birds are of the same type. On the other hand, when the two birds are of different types, the intersection points are denoted by $(\hat s_1(t_1,h_1,t_2,h_2), \hat h_1(t_1,h_1,t_2,h_2))$ and $(\hat s_2(t_1,h_1,t_2,h_2), \hat h_2(t_1,h_1,t_2,h_2))$. In the following, we may use slightly different notations for the intersection point, wherever necessary in the proofs. Let $L_s$ be the vertical line $t=s$. Define $Q_s:=\{(u,h)\in \mathbb H^+: u\geq s\}$, to be the quadrant on the right of the vertical line $L_s$.  \label{notation:hatsk}~\label{notation:Ls}

\subsubsection{Proof of Lemma~\ref{lemma:LT1-l}, type $\binom{1}{1,1}$ or equivalently $\bbinom{1}{1,1}$}\label{subsubsection:Lemma111}
Using Lemma~\ref{lemma:Palm-Lcal-ii} for the first term in (\ref{eq:decompLT2}) with $\t_n=1=\t_p$ and $q=1$, we have
\begin{align}
\E^0_{\Wcal_1} [e^{-\rho T}]
&= \frac {4\la_1^2v_1^2} {L_{1}}
\int_{(\R^+)^3}
\E_\Hcal\left[ e^{-\rho(T\circ\th_{\hat s})} \prod_{l=1,2}\one_{\Hcal^l\left(E^{\hat{s},v_l}_{\hat{h}}\right)=0}  \right] {\rm d}t' {\rm d}h' {\rm d}h ,
\label{eq:T12a}
\end{align}

where we consider two radial birds of type $\t_n=1=\t_p$ at $(t_n,h_n)=(0,h)$ and $(t_p,h_p)=(-t',h')$ and where, in short, we wrote $(\hat s(0,h,-t',h'), \hat h(0,h,-t',h'))= (\hat s,\hat h)$. The station corresponding to the head point $(0,h)$ is the serving station after time $\hat s$, since $\t_n=1$ and $q=1$. Since $(0,h)$ is of type $1$, we define $\hat h_1(t):=\left(v_1^2t^2+h^2\right)^\half$ and write $\hat h_1(t)=\hat h_t$, within this proof for brevity. Observe that 
$T\circ\th_{\hat s}=T_1(\hat s)-\hat s$, where $T_1(\hat s)$ is the time until the first handover after time $\hat s$. The inner expectation in (\ref{eq:T12a}) can be evaluated by looking for the next handover
given by the intersection of the radial bird with its head at $(0,h)$ and a third radial bird with its head at, say $(T_j,H_j)$, in view of Corollary~\ref{corollary:third-pt2}, which is either of type 1 or type 2, as depicted in Figure~\ref{figure:pure_birdii}.
\begin{figure}[ht!]
 \centering
      \begin{subfigure}[t]{0.45\linewidth}
       \centering
      \begin{tikzpicture}[scale=0.7, every node/.style={scale=0.7}]
\pgftransformxscale{0.83}  
\pgftransformyscale{0.83}    
    \draw[->] (-1, 0) -- (7, 0) node[right] {$t$};
    \draw[red](2,2.5) node{$\bullet$};
    \draw[red, domain=0.8:3.1, smooth] plot (\x, {(16*\x*\x-64*\x+64+2.5^2)^0.5});
    \draw[red](3.5,3) node{$\bullet$};
    \draw[](3.8,2.7) node{$(0,h)$};
    \draw[red, domain=2.5:4.5, smooth] plot (\x, {(16*\x*\x-32*3.5*\x+16*3.5^2+3^2)^0.5});
    \draw[](1.2,2.3) node{$(-t',h')$};
    \draw[](4.5,4.1) node{$(\hat s',\hat h')$};
    \draw[](2.1,4.4) node{$(\hat s,\hat h)$};
     \draw[blue]  (1.75,0) arc (-180:-360:2.15 and 3.3);
     \draw[red]  (4.7,0) arc (0:180:0.85 and 3.3);
     \draw[red, domain=3.3:5.4, smooth] plot (\x, {(16*\x*\x-16*2*4.5*\x+16*4.5*4.5+2.1^2)^0.5});
     \draw[red](4.5,2.1) node{$\bullet$};
     \draw[](4.5,1.7) node{$(t_1,h_1)$};
     \draw[red]  (1.8,0) arc (-180:-360:1 and 4.1);
     \draw[blue]  (5.5,0) arc (0:180:2.65 and 4.1);   
    \draw[blue](1.3,0.6) node{$E^{\hat s, v_2}_{\hat h}$};
    \draw[red](2.4,0.6) node{$E^{\hat s, v_1}_{\hat h}$};
    \draw[blue](6.7,0.6) node{$E^{\hat s',v_2}_{\hat h'}$};
     \draw[red](4.25,0.6) node{$E^{\hat s',v_1}_{\hat h'}$};
    \end{tikzpicture}
    \caption{For the intersections $(\hat s, \hat h)$ and $(\hat s', \hat h')$ to represent two consecutive  handovers, the regions \red{$E^{\hat s, v_1}_{\hat h} \cup E^{\hat s',v_1}_{\hat h'}$} and \blue{$E^{\hat s, v_2}_{\hat h}\cup E^{\hat s',v_2}_{\hat h'}$} must have no points from $\Hcal^1$ and $\Hcal^2$, respectively. Here $(\hat s', \hat h')$ is a realization of $(\hat s(0,h,T_j,H_j),\hat h(0,h,T_j,H_j))$.}
\label{subfigure:pure_birdii1}
    \end{subfigure}
    \hspace{0.2in}
    \centering
      \begin{subfigure}[t]{0.45\linewidth}
       \centering
      \begin{tikzpicture}[scale=0.7, every node/.style={scale=0.7}]
\pgftransformxscale{0.83}  
\pgftransformyscale{0.83}    
\draw[->] (-1, 0) -- (8, 0) node[right] {$t$};
    \draw[red](2,2.5) node{$\bullet$};
    \draw[red, domain=0.8:3.4, smooth] plot (\x, {(16*\x*\x-64*\x+64+2.5^2)^0.5});
    \draw[red](3.5,3) node{$\bullet$};
    \draw[](3.8,2.7) node{$(0,h)$};
    \draw[red, domain=2.2:4.7, smooth] plot (\x, {(16*\x*\x-32*3.5*\x+16*3.5^2+3^2)^0.5});
    \draw[](1.2,2.3) node{$(-t',h')$};
    \draw[](1.7,5.5) node{$(\hat s_1,\hat h_1)$};
    \draw[](4.5,4.1) node{$(\hat s_2,\hat h_2)$};
    \draw[](2.1,4.4) node{$(\hat s,\hat h)$};
     \draw[blue]  (1.8,0) arc (-180:-360:2.2 and 3.4);
     \draw[red]  (4.8,0) arc (0:180:0.88 and 3.4);
     \draw[blue, domain=2:7.3, smooth] plot (\x, {((3/2)^2*\x*\x- 2*(3/2)^2*5.8*\x+(3/2)^2*5.8*5.8+2^2)^0.5});
     \draw[blue](5.8,2) node{$\bullet$};
     \draw[](6,1.7) node{$(t_1,h_1)$};
     \draw[red]  (1.8,0) arc (-180:-360:1 and 4.1);
     \draw[blue]  (5.5,0) arc (0:180:2.65 and 4.1);   
    \draw[blue](1.3,0.6) node{$E^{\hat s, v_2}_{\hat h}$};
    \draw[red](2.4,0.6) node{$E^{\hat s, v_1}_{\hat h}$};
    \draw[blue](6.7,0.6) node{$E^{\hat s_2,v_2}_{\hat h_2}$};
     \draw[red](4.25,0.6) node{$E^{\hat s_2,v_1}_{\hat h_2}$};
    \end{tikzpicture}
    \caption{For the intersections $(\hat s, \hat h)$ and $(\hat s_2, \hat h_2)$ to represent two consecutive  handovers, the regions \red{$E^{\hat s, v_1}_{\hat h} \cup E^{\hat s_2, v_1}_{\hat h_2}$} and \blue{$E^{\hat s, v_2}_{\hat h}\cup E^{\hat s_2, v_2}_{\hat h_2}$} must have no points from $\Hcal^1$ and $\Hcal^2$, respectively. Here $(\hat s_2, \hat h_2)$ is a realization of $(\hat s_2(0,h,T_j,H_j),\hat h_2(0,h,T_j,H_j))$.}
    \label{subfigure:pure_birdii2}
    \end{subfigure}
    \captionsetup{width=0.9\linewidth}
    \caption{The two cases in (\ref{eq:T12c}), i.e., the next handover can be given by a bird at $(T_j,H_j)=(t_1,h_1)$ of type $1$ or $2$, with $(t_1,h_1)$ lying outside $\overline{E^{\hat s, v_1}_{\hat h}}$ or $\overline{E^{\hat s, v_2}_{\hat h}}$, respectively.}
\label{figure:pure_birdii}
\end{figure}
\begin{case}[\textbf{\em $(T_j,H_j)$ is of type 1.}]\label{case:H111} In this case, $(T_j,H_j)\in \Hcal^1$ is outside the region $\overline{E^{\hat s,v_1}_{\hat h}}$. Necessarily, $T_j\geq 0$. Indeed, if $T_j<0$, then the intersection of the birds with head at $(0,h)$ and $(T_j,H_j)$, either produces a handover before time $\hat s$ or lies in $\Lcal^+_e(v_1,v_2)$, the open region above the joint lower envelope, defined in (\ref{eq:Lcal2+}). Thus $T_j\geq 0$ and the type of the pair of consecutive handovers is $\bbinom{1,1}{1,1,1}$. The upcoming head point $(T_j,H_j)$ of type $1$ can be found in the unexplored region defined as $Q_0\setminus \overline{E^{\hat s,v_1}_{\hat h}}$. Let $(\hat s(0,h,T_j,H_j),\hat h(0,h,T_j,H_j))$ be the unique intersection between the two birds and, in this case, $T_1(\hat s):=\hat s(0,h,T_j,H_j)$.
For this intersection to give that the next handover takes place immediately after time $\hat s$, the region $\interior{\left(\bigcup_{t\in [\hat s,\hat s(0,h,T_j,H_j)]} E^{t,v_l}_{\hat h_t}\setminus E^{\hat s, v_l}_{\hat h}\right)}$, beyond $\overline{E^{\hat s, v_l}_{\hat h}}$, must have no point of $\Hcal^l$ for $l=1,2$. The required region is constructed as a union, using the increasing sets as in Lemma~\ref{lemma:increasingSnm} and the Definition~\ref{definition:T1-2speed}, due to the potential lack of monotonicity among the individual sets in the union.
This corresponds to the first term in (\ref{eq:T12cc}) and is also depicted in picture~(\subref{subfigure:pure_birdii1}) of Figure~\ref{figure:pure_birdii}.
\end{case}
\begin{case}[\textbf{\em $(T_j,H_j)$ is of type 2.}]\label{case:H112} In this case $(T_j,H_j)\in \Hcal^2$ is outside $\overline{E^{\hat s,v_2}_{\hat h}}$ and $T_j$ can be positive or negative. In this case, the unexplored region is $\mathbb{H}^+\setminus \overline{E^{\hat s,v_2}_{\hat h}}$ for the upcoming head point of type $2$. Then, by part~(\ref{intersection}) of Lemma~\ref{lemma:in1out2}, applied to $(s,u)=(\hat s,\hat h)$, it is true that the radial birds with head points $(T_j,H_j)$ and $(0,h)$ intersect each other, even though the point $(T_j,H_j)$ is far on the left or right of $(0,h)$. The type of the pair of handovers is $\bbinom{1,2}{1,1,2}$ or $\bbinom{1,1}{1,1,2}$, respectively, depending on whether $T_j<0$ or $T_j\geq 0$.
Let $(\hat s_1(0,h,T_j,H_j),\hat h_1(0,h,T_j,H_j))$ and $(\hat s_2(0,h,T_j,H_j),\hat h_2(0,h,T_j,H_j))$, be the two intersection points, as illustrated in picture~(\subref{subfigure:pure_birdii2}) of Figure~\ref{figure:pure_birdii}. By part~(\ref{intersection_order}) of Lemma~\ref{lemma:in1out2}, we have $\hat s_1\leq \hat s\leq \hat s_2$ and, by Observation~\ref{observation:sts}, we also have $\hat s_1(0,h,T_j,H_j)\leq 0\leq \hat s_2(0,h,T_j,H_j)$. The first intersection point cannot correspond to a future handover, since $\hat s_1(0,h,T_j,H_j)\leq \hat s$, see picture~(\subref{subfigure:pure_birdii2}) of Figure~\ref{figure:pure_birdii}.
The only possibility for a future handover is hence $(\hat s_2(0,h,T_j,H_j),\hat h_2(0,h,T_j,H_j))$, and in this case, $T_1(\hat s):=\hat s_2(0,h,T_j,H_j)$. Observe that, using Lemma~\ref{lemma:increasingSnm}, for this intersection to be the next handover, we must have the extra region $\interior{\left(\bigcup_{t\in [\hat s,\hat s_2(0,h,T_j,H_j)]} E^{t,v_l}_{\hat h_t}\setminus E^{\hat s, v_l}_{\hat h}\right)}$ empty of points of $\Hcal^l$, for $l=1,2$. The union is due to the Definition~\ref{definition:T1-2speed}, and therein, a potential lack of monotonicity among the individual sets in the union.
We have used this in the second and third terms in (\ref{eq:T12cc}).
\end{case}
Based on the last discussions on \ref{case:H111} and \ref{case:H112}, and with the help of Corollary~\ref{corollary:third-pt2}, we can decompose the inner expectation in (\ref{eq:T12a}) as
\begin{align}
\lefteqn{\E_\Hcal\left[e^{-\rho(T_1(\hat s)-\hat s)}  \prod_{l=1,2}\one_{\Hcal^l\left(E^{\hat{s},v_l}_{\hat{h}}\right)=0}\right]}\nn\\
&=
\E_\Hcal\left[ \one_{\exists (T_j,H_j)\in \Hcal^1 \,\mbox{:}\, T_j\geq 0, \, \bigcap_{l=1,2}\left\{\Hcal^l\interior{\left(\bigcup_{t\in [\hat s,\hat s(0,h,T_j,H_j)]} E^{t,v_l}_{\hat h_t}\setminus E^{\hat s, v_l}_{\hat h}\right)}=0\right\}} e^{-\rho (\hat s(0,h,T_j,H_j)-\hat s)} \prod_{l=1,2}\one_{\Hcal^l\left(E^{\hat{s},v_l}_{\hat{h}}\right)=0}\right] \nn\\
&\;+ \E_\Hcal\left[ \one_{\exists (T_j,H_j)\in \Hcal^2 \,\mbox{:}\, T_j< 0, \, \bigcap_{l=1,2}\left\{\Hcal^l\interior{\left(\bigcup_{t\in [\hat s,\hat s_2(0,h,T_j,H_j)]} E^{t,v_l}_{\hat h_t}\setminus E^{\hat s, v_l}_{\hat h}\right)}=0\right\}} e^{-\rho (\hat s_2(0,h,T_j,H_j)-\hat s)} \prod_{l=1,2}\one_{\Hcal^l\left(E^{\hat{s},v_l}_{\hat{h}}\right)=0}\right] \nn\\
&\;+ \E_\Hcal\left[ \one_{\exists (T_j,H_j)\in \Hcal^2 \,\mbox{:}\, T_j\ge 0, \, \bigcap_{l=1,2}\left\{\Hcal^l\interior{\left(\bigcup_{t\in [\hat s,\hat s_2(0,h,T_j,H_j)]} E^{t,v_l}_{\hat h_t}\setminus E^{\hat s, v_l}_{\hat h}\right)}=0\right\}} e^{-\rho (\hat s_2(0,h,T_j,H_j)-\hat s)} \prod_{l=1,2}\one_{\Hcal^l\left(E^{\hat{s},v_l}_{\hat{h}}\right)=0}\right].\label{eq:T12cc}
\end{align}
Using the fact that there is a unique point satisfying the condition in each case, the last expression can be written as
\begin{align}
\lefteqn{\E_\Hcal\left[\sum_{ (T_j,H_j)\in \Hcal^1 \,\mbox{:}\, T_j\geq 0}\;\; \prod_{l=1,2}\one_{\Hcal^l\left(\bigcup_{t\in [\hat s,\hat s(0,h,T_j,H_j)]} E^{t,v_l}_{\hat h_t}\cup E^{\hat s, v_l}_{\hat h}\right)=0} e^{-\rho (\hat s(0,h,T_j,H_j)-\hat s)}\right]}\nn\\
&\;\; +\E_\Hcal\left[\sum_{ (T_j,H_j)\in \Hcal^2 \,\mbox{:}\, T_j< 0}\;\; \prod_{l=1,2}\one_{\Hcal^l\left(\bigcup_{t\in [\hat s,\hat s_2(0,h,T_j,H_j)]} E^{t,v_l}_{\hat h_t}\cup E^{\hat s, v_l}_{\hat h}\right)=0} e^{-\rho (\hat s_2(0,h,T_j,H_j)-\hat s)}\right]\nn\\
&\;\; +\E_\Hcal\left[\sum_{ (T_j,H_j)\in \Hcal^2 \,\mbox{:}\, T_j\ge 0}\;\; \prod_{l=1,2}\one_{\Hcal^l\left(\bigcup_{t\in [\hat s,\hat s_2(0,h,T_j,H_j)]} E^{t,v_l}_{\hat h_t}\cup E^{\hat s, v_l}_{\hat h}\right)=0} e^{-\rho (\hat s_2(0,h,T_j,H_j)-\hat s)}\right].
\label{eq:T12c}
\end{align}
Using part~(\ref{monotonic1}) of Proposition~\ref{prop:monotonic-sets}, the increasing property of the sequence of regions $\left\{E^{t,v_l}_{\hat h_t}\setminus E^{\hat s, v_l}_{\hat h}\right\}_{t\geq \hat s}$, we have, for both $l=1,2$,
\[
\bigcup_{t\in [\hat s,\hat s(0,h,T_j,H_j)]}\!\! E^{t,v_l}_{\hat h_t}\cup E^{\hat s, v_l}_{\hat h}
= E^{\hat s(0,h,T_j,H_j),v_l}_{\hat h(0,h,T_j,H_j)}\cup E^{\hat s, v_l}_{\hat h} \text{ and}\!\! \bigcup_{t\in [\hat s,\hat s_2(0,h,T_j,H_j)]} \!\! E^{t,v_l}_{\hat h_t}\cup E^{\hat s, v_l}_{\hat h}
= E^{\hat s(0,h,T_j,H_j),v_l}_{\hat h(0,h,T_j,H_j)}\cup E^{\hat s, v_l}_{\hat h},
\]
depending on whether  $(T_j,H_j)$ is of type $1$ or type $2$. Hence, (\ref{eq:T12c}) can be written as
\begin{align}
\lefteqn{\E_\Hcal\left[\sum_{ (T_j,H_j)\in \Hcal^1 \,\mbox{:}\, T_j\geq 0}\;\; \prod_{l=1,2}\one_{\Hcal^l\left(E^{\hat s(0,h,T_j,H_j),v_l}_{\hat h(0,h,T_j,H_j)}\cup E^{\hat s, v_l}_{\hat h}\right)=0} e^{-\rho (\hat s(0,h,T_j,H_j)-\hat s)}\right]}\nn\\
&\;\; +\E_\Hcal\left[\sum_{ (T_j,H_j)\in \Hcal^2 \,\mbox{:}\, T_j< 0}\;\; \prod_{l=1,2}\one_{\Hcal^l\left(E^{\hat s_2(0,h,T_j,H_j),v_l}_{\hat h_2(0,h,T_j,H_j)}\cup E^{\hat s, v_l}_{\hat h}\right)=0} e^{-\rho (\hat s_2(0,h,T_j,H_j)-\hat s)}\right]\nn\\
&\;\; +\E_\Hcal\left[\sum_{ (T_j,H_j)\in \Hcal^2 \,\mbox{:}\, T_j\ge 0}\;\; \prod_{l=1,2}\one_{\Hcal^l\left(E^{\hat s_2(0,h,T_j,H_j),v_l}_{\hat h_2(0,h,T_j,H_j)}\cup E^{\hat s, v_l}_{\hat h}\right)=0} e^{-\rho (\hat s_2(0,h,T_j,H_j)-\hat s)}\right]\nn\\
&:=\zeta\bbinom{1,1}{1,1,1}+ \zeta\bbinom{1,2}{1,1,2}+\zeta\bbinom{1,1}{1,1,2}.
\label{eq:T12cs}
\end{align}
We formally define the notation $\zeta\bbinom{q,q'}{\t_p,\t_n,\t_u}$, adopting here a general convention to be used for the rest of the article. Suppose the typical handover is given by the head points $(t_p,h_p), (t_n,h_n)$. Then $\zeta\bbinom{q,q'}{\t_p,\t_n,\t_u}$ is a function of $\rho, h_n, |t_p-t_n|, h_p$, i.e., $\zeta\bbinom{q,q'}{\t_p,\t_n,\t_u}\equiv \zeta\bbinom{q,q'}{\t_p,\t_n,\t_u}(\rho,h_n, |t_p-t_n|, h_p)$. The function $\zeta\bbinom{q,q'}{\t_p,\t_n,\t_u}(\rho, h_n,|t_p-t_n|,h_n)$ is the conditional Laplace transform of the inter-handover time $T$ and the event that the upcoming handover is of type $\bbinom{q'}{\t_n,\t_u}$, given that the typical handover type is $\bbinom{q}{\t_p,\t_n}$, with previous head at $(t_p,h_p)$ and next at $(t_n,h_n)$. For $\rho=0$, $\zeta\bbinom{q,q'}{\t_p,\t_n,\t_u}(0,h_n,|t_p-t_n|,h_p)$ is the conditional probability of the event that the head point corresponding to the upcoming station is of type $\t_u$ and the corresponding head point is on the right or left side of $(t_n,h_n)$ depending on whether $q'=1$ or $2$. In the current proof, since $\bbinom{q}{\t_p,\t_n}=\bbinom{1}{1,1}$, we have $\zeta\bbinom{1,q'}{1,1,\t_u}$ as functions of $\rho, h,t',h'$, for different values of $\t_u$ and $q'$, as appeared in (\ref{eq:T12cs}), for which we evaluate the individual terms.

Applying the Campbell-Mecke formula, the first term in (\ref{eq:T12cs}) is given as follows
\begin{align}
\zeta\bbinom{1,1}{1,1,1} & = 2\la_1v_1
\int_0^{\hat s+ \hat h/v_1} \!\! \int_{\left(\hat h^2-v_1^2(t_1-\hat s)^2\right)^\half}^{\infty}
\E_\Hcal\left[ e^{-\rho (\hat s(0,h,t_1,h_1)-\hat s)} \prod_{l=1,2}  \one_{\Hcal^l\left(E^{\hat s(0,h,t_1,h_1),v_l}_{\hat h(0,h,t_1,h_1)}\cup E^{\hat s, v_l}_{\hat h}\right)=0} 
 \right] {\rm d}h_1 \, {\rm d}t_1 \nn\\
&\;\;+2\la_1v_1
\int_{\hat s+\hat h/v_1}^\infty\int_0^\infty
\E_\Hcal\left[ e^{-\rho (\hat s(0,h,t_1,h_1)-\hat s)} \prod_{l=1,2}  \one_{\Hcal^l\left(E^{\hat s(0,h,t_1,h_1),v_l}_{\hat h(0,h,t_1,h_1)}\cup E^{\hat s, v_l}_{\hat h}\right)=0} 
 \right] {\rm d}h_1 \, {\rm d}t_1,
\label{eq:T1-3rd1}
\end{align}
since the point $(t_1,h_1)$ must be such that $t_1\geq 0$ and must lie outside the region $E^{\hat s,v_1}_{\hat h}$. We write in short $(\hat s',\hat h')$ for $(\hat s(0,h,t_1,h_1),\hat h(0,h,t_1,h_1))$. In the last expression, we have separated the integral in two parts, $[0,\hat s+\hat h/v_1]\times\R^+\setminus E^{\hat s,v_1}_{\hat h}$ and $(\hat s+\hat h/v_1, \infty)\times \R^+$. Using the void probability we have
\begin{align}
\zeta\bbinom{1,1}{1,1,1}
&=2\la_1v_1
\int_0^{\hat s+ \hat h/v_1}\int_{\left(\hat h^2-v_1^2(t_1-\hat s)^2\right)^\half}^{\infty}
e^{-\rho (\hat s'-\hat s)} e^{-\sum_{l=1,2} 2\la_lv_l\left\vert E^{\hat s',v_l}_{\hat h' }\cup E^{\hat s, v_l}_{\hat h}\right\vert} 
 {\rm d}h_1 \, {\rm d}t_1 \nn\\
&\;\;+2\la_1v_1
\int_{\hat s+ \hat h/v_1}^\infty\int_0^\infty
e^{-\rho (\hat s' -\hat s)} e^{-\sum_{l=1,2} 2\la_lv_l\left\vert E^{\hat s' ,v_l}_{\hat h' }\cup E^{\hat s, v_l}_{\hat h}\right\vert} {\rm d}h_1 \, {\rm d}t_1.
\label{eq:T1-3rd1a}
\end{align}
In short, we write $(\hat s_2,\hat h_2)$ for $(\hat s_2(0,h,t_1,h_1),\hat h_2(0,h,t_1,h_1))$ and get a representation of the second term in (\ref{eq:T12cs}) as
\begin{align}
\zeta\bbinom{1,2}{1,1,2}
&=2\la_2v_2
\int_{-\infty}^{\hat s-\hat h/v_2}\int_0^\infty
\E_\Hcal\left[ e^{-\rho (\hat s_2-\hat s)} \prod_{l=1,2}  \one_{\Hcal^l\left(E^{\hat s_2,v_l}_{\hat h_2}\cup E^{\hat s, v_l}_{\hat h}\right)=0} 
 \right] {\rm d}h_1 \, {\rm d}t_1\nn\\
&\;\; +  2\la_2v_2
\int_{\hat s-\hat h/v_2}^{0}\int_{\left(\hat h^2-v_2^2(t_1-\hat s)^2\right)^\half}^\infty
\E_\Hcal\left[ e^{-\rho (\hat s_2-\hat s)} \prod_{l=1,2}  \one_{\Hcal^l\left(E^{\hat s_2,v_l}_{\hat h_2}\cup E^{\hat s, v_l}_{\hat h}\right)=0} 
 \right] {\rm d}h_1 \, {\rm d}t_1, 
\label{eq:T1-3rd2}
\end{align}
where the point $(t_1,h_1)$ with $t_1< 0$ and of type $2$, must lie outside the region $E^{\hat s,v_2}_{\hat h}$. So we have separated the integral in two parts, $(-\infty, \hat s-\hat h/v_2]\times \R^+$ and $[\hat s-\hat h/v_2, 0]\times\R^+\setminus E^{\hat s,v_2}_{\hat h}$. Using the void probabilities we again have
\begin{align}
\zeta\bbinom{1,2}{1,1,2}
&= 2\la_2v_2 \int_{-\infty}^{\hat s-\hat h/v_2}\int_0^\infty e^{-\rho (\hat s_2 -\hat s)} e^{-\sum_{l=1,2} 2\la_lv_l\left\vert E^{\hat s_2 ,v_l}_{\hat h_2 }\cup E^{\hat s, v_l}_{\hat h}\right\vert} {\rm d}h_1 \, {\rm d}t_1 \nn\\
&\;\; +  2\la_2v_2
\int_{\hat s-\hat h/v_2}^{0}\int_{\left(\hat h^2-v_2^2(t_1-\hat s)^2\right)^\half}^\infty   e^{-\rho (\hat s_2 -\hat s)} e^{-\sum_{l=1,2} 2\la_lv_l\left\vert E^{\hat s_2 ,v_l}_{\hat h_2 }\cup E^{\hat s, v_l}_{\hat h}\right\vert} {\rm d}h_1 \, {\rm d}t_1. 
\label{eq:T1-3rd2a}
\end{align}
The third term in (\ref{eq:T12cs}) equals
\begin{align}
\zeta\bbinom{1,1}{1,1,2}&=2\la_2v_2
\int_0^{\hat s+\hat h/v_2}\int_{\left(\hat h^2-v_2^2(t_1-\hat s)^2\right)^\half}^\infty
\E_\Hcal\left[ e^{-\rho (\hat s_2-\hat s)} \prod_{l=1,2}  \one_{\Hcal^l\left(E^{\hat s_2,v_l}_{\hat h_2}\cup E^{\hat s, v_l}_{\hat h}\right)=0} 
 \right] {\rm d}h_1 \, {\rm d}t_1 \nn\\
&\;\; +  2\la_2v_2
\int_{\hat s+\hat h/v_2}^\infty\int_0^\infty
\E_\Hcal\left[ e^{-\rho (\hat s_2-\hat s)} \prod_{l=1,2}  \one_{\Hcal^l\left(E^{\hat s_2,v_l}_{\hat h_2}\cup E^{\hat s, v_l}_{\hat h}\right)=0} 
 \right] {\rm d}h_1 \, {\rm d}t_1,
\label{eq:T1-3rd3}
\end{align}
since the point $(t_1,h_1)$ with $t_1\geq 0$ and of type $2$, must lie outside the region $E^{\hat s,v_2}_{\hat h}$. So we have separated the integral in two parts, $[0,\hat s+\hat h/v_2]\times\R^+\setminus E^{\hat s,v_2}_{\hat h}$ and $(\hat s+\hat h/v_2, \infty)\times \R^+$. By computing the void probabilities we have
\begin{align}
\zeta\bbinom{1,1}{1,1,2}
&= 2\la_2v_2
\int_{0}^{\hat s+\hat h/v_2}\int_{\left(\hat h^2-v_2^2(t_1-\hat s)^2\right)^\half}^\infty   e^{-\rho (\hat s_2 -\hat s)} e^{-\sum_{l=1,2} 2\la_lv_l\left\vert E^{\hat s_2 ,v_l}_{\hat h_2 }\cup E^{\hat s, v_l}_{\hat h}\right\vert} {\rm d}h_1 \, {\rm d}t_1 \nn\\
&\;\; +  2\la_2v_2
\int_{\hat s+\hat h/v_2}^\infty\int_0^\infty
e^{-\rho (\hat s_2 -\hat s)} e^{-\sum_{l=1,2} 2\la_lv_l\left\vert E^{\hat s_2 ,v_l}_{\hat h_2 }\cup E^{\hat s, v_l}_{\hat h}\right\vert}  {\rm d}h_1 \, {\rm d}t_1.
\label{eq:T1-3rd3a}
\end{align}
Substituting from (\ref{eq:T1-3rd1a})--(\ref{eq:T1-3rd3a}) to (\ref{eq:T12cs}), we obtain the inner expectation in (\ref{eq:T12a}) as the sum
\begin{align}
\zeta\bbinom{1,1}{1,1,1}+\zeta\bbinom{1,2}{1,1,2}+\zeta\bbinom{1,1}{1,1,2}
%
%
%
%
%
%
&:=\zeta_{1,1}(h,t',h').
\label{eq:T12d}
\end{align}
Using the expression from (\ref{eq:T12d}) in (\ref{eq:T12a}) yields 
\begin{align}
\E^0_{\Wcal_1} [e^{-\rho T}]
&= \frac {4\la_1^2v_1^2} {L_{1}}
\int_{(\R^+)^3} \zeta_{1,1}(h,t',h') {\rm d}t' {\rm d}h' {\rm d}h. 
\end{align}
This completes the proof for the case $\t_n=1=\t_p$.
\qed
%
%
\vspace{-0.25in}
\subsubsection{Proof of Lemma~\ref{lemma:LT_1ijk}, type $\binom{1}{2,1}$ or equivalently $\bbinom{1}{2,1}$}\label{subsubsection:Lemma121} Here, $\t_n=1, \t_p=2$, $q=1$.
Using Lemma~\ref{lemma:Palm-Lcal-ij} for mixed Palm expectations in (\ref{eq:decompLT2}) with these values of $\t_n, \t_p, q$ and $f(\Hcal)=e^{-\rho T}$, we get
\begin{align}
\frac{L^{(1)}_{2,1}}{4\la_1\la_2 v_1v_2}\E^0_{\Wcal_{2,1}^{(1)}} [e^{-\rho T}] &=   \int_{0}^\infty  \int_{0}^h \int_{t^*}^\infty 
\mathbb{E}_\Hcal\Bigg[ e^{-\rho(T\circ\th_{\hat s_1})}  \prod_{l=1,2}\one\left\{\Hcal^l\left(E^{\hat{s}_1,v_l}_{\hat{h}_1}\right)=0\right\}\Bigg] {\rm d}t'  {\rm d}h' {\rm d}h \nn\\ 
&\;+  \int_{0}^\infty  \int_h^\infty \int_0^{\infty}
\mathbb{E}_\Hcal\Bigg[ e^{-\rho(T\circ\th_{\hat s_1})}  \prod_{l=1,2}\one\left\{\Hcal^l\left(E^{\hat{s}_1,v_l}_{\hat{h}_1}\right)=0\right\}\Bigg] {\rm d}t'  {\rm d}h' {\rm d}h ,
\label{eq:T12g}
\end{align}
where $(\hat{s}_1,\hat{h}_1)=(\hat{s}_1(0,h,-t',h'),\hat{h}_1(0,h,-t',h'))$ denotes the coordinates of the first intersection between the radial birds at $(t_n,h_n)=(0,h)$ of type $\t_n=1$ and $(t_p,h_p)=(-t',h')$ type $\t_p=2$, which are responsible for a handover. The coordinates $(\hat{s}_2,\hat{h}_2)=(\hat{s}_2(0,h,-t',h'),\hat{h}_2(0,h,-t',h'))$ denote the second intersection. Since $\t_n=1$ and $q=1$, the station corresponding to the head point $(0,h)$ becomes the serving station after time $\hat s_1$. The head point $(0,h)$ being of type $1$, we define $\hat h_1(t):=\left(v_1^2t^2+h^2\right)^\half$, which we write in short as $\hat h_t$ locally. Observe that 
\[
T\circ\th_{\hat s_1}=T_1(\hat s_1)-\hat s_1,
\]
where $T_1(\hat s_1)$ is the time of the next handover after time $\hat s_1$. For the next handover, the extra region $\interior{\left(\bigcup_{t\in [\hat s_1,T_1(\hat s_1)]} E^{t,v_l}_{\hat h_t}\setminus E^{\hat s_1, v_l}_{\hat h_1}\right)}$ beyond $\overline{E^{\hat s_1, v_l}_{\hat h_1}}$ must have no point of $\Hcal^l$, for $l=1,2$. Here we have for any $\hat s\geq \hat s_2$ 
\[
\bigcup_{t\in [\hat s_1,\hat s_2]} E^{t,v_l}_{\hat h_t}\setminus E^{\hat s_1, v_l}_{\hat h_1}\subset \bigcup_{t\in [\hat s_1,\hat s]} E^{t,v_l}_{\hat h_t}\setminus E^{\hat s_1, v_l}_{\hat h_1}.
\]
Hence, the next handover can happen before or at $\hat s_2$, i.e., $T_1(\hat s_1)<\hat s_2$ or $T_1(\hat s_1)=\hat s_2$, which leads to another layer of decomposition.

\begin{case}[\textbf{$T_1(\hat s_1)<\hat s_2$.}]\label{case:H211}
By Corollary~\ref{corollary:third-pt2}, the inner expectation in both terms in (\ref{eq:T12g}) can be evaluated by looking for the next handover created by a radial bird with head at, say $(T_j,H_j)$, which is of type $1$ or $2$. The head point $(T_j,H_j)$ must be located outside $\overline{E^{\hat{s}_1,v_1}_{\hat{h}_1}}$ and $\overline{E^{\hat{s}_1,v_2}_{\hat{h}_1}}$, depending on its type $1$ or $2$ (see Figure~\ref{figure:mixed_birdij}), as follows: 

\begin{subcase}[\textbf{\em $(T_j,H_j)$ is of type 1.}]\label{subcase:H211x} Here, it must be the case that $T_j\geq 0$ and $(T_j,H_j)\notin \overline{E^{\hat s_1,v_1}_{\hat h_1}}$. Indeed, since, for any radial bird with the head point at $(T_j,H_j)\notin \overline{E^{\hat s_1,v_1}_{\hat h_1}}$, such that $T_j<0$, the intersection to the radial bird with head at $(0,h)$, does not contribute to a future handover as it intersects the one with head at $(0,h)$ before time $\hat s_1$. Thus the unexplored region must be $Q_0\setminus \overline{E^{\hat s_1,v_1}_{\hat h_1}}$ and the type of the pair of handovers is $\bbinom{1,1}{2,1,1}$. 
In this case, the handover is given by the intersection $(\hat s(0,h,T_j,H_j),\hat h(0,h,T_j,H_j))$, as both $(0,h)$ and $(T_j,H_j)$ are of the same type, see picture~(\subref{subfigure:mixed_birdij1}) of Figure~\ref{figure:mixed_birdij}. Thus we have $T_1(\hat s_1):=\hat s(0,h,T_j,H_j)$. For this intersection to give the next handover immediately after time $\hat s_1$, the extra region $\interior{\left(\bigcup_{t\in [\hat s_1,\hat s(0,h,T_j,H_j)]} E^{t,v_l}_{\hat h_t}\setminus E^{\hat s_1, v_l}_{\hat h_1}\right)}$, beyond $\overline{E^{\hat s_1, v_l}_{\hat h_1}}$, must have no point of $\Hcal^l$ for $l=1,2$, where $\hat h_t=(v_1^2t^2+h^2)^\half$. The union is due to Definition~\ref{definition:T1-2speed} and to the fact that the sets in the union potentially lack monotonicity.
\end{subcase}

\begin{subcase}[\textbf{\em  $(T_j,H_j)$ is of type 2.}]\label{subcase:H211y} Since $(T_j,H_j)\notin \overline{E^{\hat s_1,v_2}_{\hat h_1}}$, the pair of points $(0,h), (T_j,H_j)$ naturally satisfies the intersection criterion, by part~(\ref{intersection}) of Lemma~\ref{lemma:in1out2}. For $(T_j,H_j)$ to be an eligible candidate, we must have $T_j\geq -t'$. Otherwise, if $T_j<-t'$, then the radial bird with head point at $(T_j,H_j)$ intersects the one with head point at $(0,h)$, at a point contained in the open region above the lower envelope, and as a result this does not contribute to any handover. Hence we assume that $T_j\geq -t'$ and the unexplored region is $Q_{-t'}\setminus \overline{E^{\hat s_1,v_2}_{\hat h_1}}$. Depending on whether $-t'\leq T_j<0$ or $T_j\geq 0$, the type of the pair of consecutive handovers is $\bbinom{1,2}{2,1,2}$ or $\bbinom{1,1}{2,1,2}$, respectively.

In both cases, $-t'\leq T_j<0$ or $T_j\geq 0$, there are two intersections $(\hat s_1(0,h,T_j,H_j),\hat h_2(0,h,T_j,H_j))$ and $(\hat s_2(0,h,T_j,H_j),\hat h_2(0,h,T_j,H_j))$ between the radial birds at $(0,h)$ and $(T_j,H_j)$. The first intersection $(\hat s_1(0,h,T_j,H_j),\hat h_2(0,h,T_j,H_j))$ can never correspond to a handover, because it is contained in $\Lcal^+_e(v_1,v_2)$. So the only intersection that can possibly be responsible for a handover is $(\hat s_2(0,h,T_j,H_j), \hat h_2(0,h,T_j,H_j))$, as described in Figure~\ref{figure:mixed_birdij2}. Hence, we have, $T_1(\hat s_1):=\hat s_2(0,h,T_j,H_j)$. 
For this intersection to give the next handover after time $\hat s_1$, the region $\interior{\left(\bigcup_{t\in [\hat s_1,\hat s_2(0,h,T_j,H_j)]} E^{t,v_l}_{\hat h_t}\setminus E^{\hat s_1, v_l}_{\hat h_1}\right)}$, beyond $\overline{E^{\hat s_1, v_l}_{\hat h_1}}$, must have no point of $\Hcal^l$ for $l=1,2$, where $\hat h_t=(v_1^2t^2+h^2)^\half$. The union is due to Definition~\ref{definition:T1-2speed} and the individual sets in the union possibly lack monotonicity.
\end{subcase}
\end{case}

\begin{case}
[\textbf{$T_1(\hat s_1)=\hat s_2$.}]\label{case:H212} Then the next handover is given by the second intersection $(\hat s_2,\hat h_2)$ of the radial birds with head points $(0,h)$, $(-t',h')$, provided the region $\interior{\left(\bigcup_{t\in [\hat s_1,\hat s_2]} E^{t,v_l}_{\hat h_t}\setminus E^{\hat s_1, v_l}_{\hat h_1}\right)}$ has no points of $\Hcal^l$, for $l=1,2$, see Figure~\ref{figure:mixed_birdij2}, where the union is due to Definition~\ref{definition:T1-2speed} and to the fact that the individual sets therein may not be monotonic. The type of the pair of consecutive handovers is $\bbinom{1,2}{2,1,2^*}$.
\end{case}
\begin{figure}[ht!]
 \centering
      \begin{subfigure}[t]{0.45\linewidth}
       \centering
      \begin{tikzpicture}[scale=0.75, every node/.style={scale=0.7}]
\pgftransformxscale{0.95}  
    \pgftransformyscale{0.95}    
    \draw[->] (-2, 0) -- (7, 0) node[right] {$t$};
   \draw[blue, domain=-2.2:3.8, smooth] plot (\x, {((3/2)^2*\x*\x+2^2)^0.5});
    \draw[blue](0,2) node{$\bullet$};
    \draw[](-0.7,1.8) node{$(-t',h')$};
    \draw[red](2,2.5) node{$\bullet$};
    \draw[red, domain=1.1:3.4, smooth] plot (\x, {(16*\x*\x-64*\x+64+2.5^2)^0.5});
    \draw[red](3.5,3) node{$\bullet$};
    \draw[](3.8,2.7) node{$(t_1,h_1)$};
    \draw[red, domain=2.2:4.7, smooth] plot (\x, {(16*\x*\x-32*3.5*\x+16*3.5^2+3^2)^0.5});
    \draw[](2.6,2.3) node{$(0,h)$};
    \draw[](0.9,3.3) node{$(\hat s_1,\hat h_1)$};
    \draw[](3.9,5.1) node{$(\hat s_2,\hat h_2)$};
    \draw[](3.4,4.4) node{$(\tilde s,\tilde h)$};
     \draw[blue]  (-0.47,0) arc (-180:-360:2 and 3.07);
     \draw[red]  (2.35,0) arc (0:180:0.8 and 3.07);
     \draw[red]  (1.8,0) arc (-180:-360:1 and 4.1);
     \draw[blue]  (5.5,0) arc (0:180:2.65 and 4.1);   
    \draw[blue](-0.9,0.6) node{$E^{\hat s_1, v_2}_{\hat h_1}$};
    \draw[red](1.3,0.6) node{$E^{\hat s_1, v_1}_{\hat h_1}$};
    \draw[blue](4.7,1.6) node{$E^{\tilde s,v_2}_{\tilde h}$};
     \draw[red](2.9,0.6) node{$E^{\tilde s,v_1}_{\tilde h}$};
    \end{tikzpicture}
    \caption{For the intersections $(\hat s_1, \hat h_1)$ and $(\tilde s, \tilde h)$ to represent two consecutive  handovers, the regions \red{$E^{\hat s_1, v_1}_{\hat h_1} \cup E^{\tilde s,v_1}_{\tilde h}$} and  \blue{$E^{\hat s_1, v_2}_{\hat h_1}\cup E^{\tilde s,v_2}_{\tilde h}$} must have no points from $\Hcal^1$ and $\Hcal^2$, respectively. Here $(\tilde s, \tilde h)$ is a realization of $(\hat s(0,h,T_j,H_j),\hat h(0,h,T_j,H_j))$.}
\label{subfigure:mixed_birdij1}
    \end{subfigure}
    \hspace{0.1in}
    \centering
      \begin{subfigure}[t]{0.45\linewidth}
       \centering
      \begin{tikzpicture}[scale=0.75, every node/.style={scale=0.7}]
\pgftransformxscale{0.95}  
\pgftransformyscale{0.95}    
\draw[->] (-2, 0) -- (7, 0) node[right] {$t$};
   \draw[blue, domain=-2.2:3.8, smooth] plot (\x, {((3/2)^2*\x*\x+2^2)^0.5});
    \draw[blue](0,2) node{$\bullet$};
    \draw[](-0.7,1.8) node{$(-t',h')$};
    \draw[red](2,2.5) node{$\bullet$};
    \draw[red, domain=0.8:3.4, smooth] plot (\x, {(16*\x*\x-64*\x+64+2.5^2)^0.5});
    \draw[blue](4.4,3) node{$\bullet$};
    \draw[](5,2.6) node{$(t_1,h_1)$};
    \draw[blue, domain=1:5.7, smooth] plot (\x, {((3/2)^2*\x*\x-(3/2)^2*2*4.4*\x+(3/2)^2*4.4^2+3^2)^0.5});
    \draw[](2.6,2.3) node{$(0,h)$};
    \draw[](0.9,3.3) node{$(\hat s_1,\hat h_1)$};
    \draw[](3.9,5.1) node{$(\hat s_2,\hat h_2)$};
    \draw[](3.4,4.2) node{$(\tilde s,\tilde h)$};
     \draw[blue]  (-0.47,0) arc (-180:-360:2 and 3.07);
     \draw[red]  (2.35,0) arc (0:180:0.8 and 3.07);
     \draw[red]  (1.8,0) arc (-180:-360:0.9 and 3.9);
     \draw[blue]  (5.38,0) arc (0:180:2.6 and 3.9);   
    \draw[blue](-0.9,0.6) node{$E^{\hat s_1, v_2}_{\hat h_1}$};
    \draw[red](1.3,0.6) node{$E^{\hat s_1, v_1}_{\hat h_1}$};
    \draw[blue](4.7,1.6) node{$E^{\tilde s,v_2}_{\tilde h}$};
     \draw[red](2.9,0.6) node{$E^{\tilde s,v_1}_{\tilde h}$};
    \end{tikzpicture}
    \caption{For the intersections $(\hat s_1, \hat h_1)$ and $(\tilde s, \tilde h)$ to represent two consecutive  handovers, the regions \red{$E^{\hat s_1, v_1}_{\hat h_1} \cup E^{\tilde s,v_1}_{\tilde h}$} and  \blue{$E^{\hat s_1, v_2}_{\hat h_1}\cup E^{\tilde s,v_2}_{\tilde h}$} must have no points from $\Hcal^1$ and $\Hcal^2$, respectively. Here $(\tilde s, \tilde h)$ is a realization of $(\hat s_2(0,h,T_j,H_j),\hat h_2(0,h,T_j,H_j))$.}    \label{subfigure:mixed_birdij2}
    \end{subfigure}
    \caption{The cases corresponding to the first and second terms in (\ref{eq:T12h1}), when the head point $(T_j,H_j)=(t_1,h_1)$ is of type $1$ and $2$, respectively.}    \label{figure:mixed_birdij}
\end{figure}
Thus, based on the last arguments, we write the inner expectation in (\ref{eq:T12g}), as the following sum of three terms:
\begin{align}
\lefteqn{\mathbb{E}_\Hcal\Bigg[ e^{-\rho(T_1(\hat s_1)-\hat s_1)} \prod_{l=1,2}\one_{\Hcal^l\left(E^{\hat{s}_1,v_l}_{\hat{h}_1}\right)=0}
\Bigg]}\nn\\
&=  \E_\Hcal\Bigg[ \one_{\exists (T_j,H_j)\in \Hcal^1 \,\mbox{:}\, T_j\geq 0, \, \bigcap_{l=1,2}\left\{\Hcal^l\interior{\left(\bigcup_{t\in [\hat s_1,\hat s(0,h,T_j,H_j)]} E^{t,v_l}_{\hat h_t}\setminus E^{\hat s_1, v_l}_{\hat h_1}\right)}=0\right\}} \one_{\hat s_1\leq \hat s(0,h,T_j,H_j)\leq \hat s_2} \nn\\
& \hspace{3.8in}\times e^{-\rho (\hat s(0,h,T_j,H_j)-\hat s_1)}\prod_{l=1,2}\one_{\Hcal^l\left(E^{\hat{s}_1,v_l}_{\hat{h}_1}\right)=0}\Bigg] \nn\\
&\;\;+ \E_\Hcal\Bigg[ \one_{\exists (T_j,H_j)\in \Hcal^2 \,\mbox{:}\, T_j\geq -t',\, \bigcap_{l=1,2}\left\{\Hcal^l\interior{\left(\bigcup_{t\in [\hat s_1,\hat s_2(0,h,T_j,H_j)]} E^{t,v_l}_{\hat h_t}\setminus E^{\hat s_1, v_l}_{\hat h_1}\right)}=0\right\}} \one_{\hat s_1\leq \hat s_2(0,h,T_j,H_j)\leq \hat s_2} \nn\\
& \hspace{3.8in}\times e^{-\rho (\hat s(0,h,T_j,H_j)-\hat s_1)} \prod_{l=1,2}\one_{\Hcal^l\left(E^{\hat{s}_1,v_l}_{\hat{h}_1}\right)=0}\Bigg] \nn\\
&\;\; + \E_\Hcal\left[ e^{-\rho (\hat s_2-\hat s_1)} \prod_{l=1,2}\one_{\{\Hcal^l\interior{\left(\bigcup_{t\in [\hat s_1,\hat s_2]} E^{t,v_l}_{\hat h_t}\setminus E^{\hat s_1, v_l}_{\hat h_1}\right)}=0\}} \prod_{l=1,2}\one_{\Hcal^l\left(E^{\hat{s}_1,v_l}_{\hat{h}_1}\right)=0}\right]. 
\label{eq:T12h1}
\end{align}
Similarly to (\ref{eq:T12c}), there exists a unique head point that satisfies the condition in each term of the sum in (\ref{eq:T12h1}), the last expression can be written as
\begin{align}
\lefteqn{\E_\Hcal\left[ \sum_{(T_j,H_j)\in \Hcal^1 \,\mbox{:}\, T_j\geq 0} \prod_{l=1,2}\one_{\Hcal^l\left(\bigcup_{t\in [\hat s_1,\hat s(0,h,T_j,H_j)]} E^{t,v_l}_{\hat h_t}\cup E^{\hat s_1, v_l}_{\hat h_1}\right)=0} \one_{\hat s_1\leq \hat s(0,h,T_j,H_j)\leq  \hat s_2} e^{-\rho (\hat s(0,h,T_j,H_j)-\hat s_1)} \right]}\nn\\
&\;\;+ \E_\Hcal\left[ \sum_{(T_j,H_j)\in \Hcal^2 \,\mbox{:}\, T_j\geq -t'} \prod_{l=1,2}\one_{\Hcal^l\left(\bigcup_{t\in [\hat s_1,\hat s_2(0,h,T_j,H_j)]} E^{t,v_l}_{\hat h_t}\cup E^{\hat s_1, v_l}_{\hat h_1}\right)=0} \one_{\hat s_1\leq \hat s_2(0,h,T_j,H_j)\leq  \hat s_2} e^{-\rho (\hat s(0,h,T_j,H_j)-\hat s_1)}\right]\nn\\
&\;\; + \E_\Hcal\left[ e^{-\rho (\hat s_2-\hat s_1)} \prod_{l=1,2}\one_{\Hcal^l\left(\bigcup_{t\in [\hat s_1,\hat s_2]} E^{t,v_l}_{\hat h_t}\cup E^{\hat s_1, v_l}_{\hat h_1}\right)=0} \right].
\label{eq:T12h}
\end{align}
\begin{figure}[ht!]
\begin{tikzpicture}[scale=0.8, every node/.style={scale=0.7}]
\pgftransformxscale{1}  
\pgftransformyscale{1}    
\draw[->] (-2, 0) -- (7, 0) node[right] {$t$};
   \draw[blue, domain=-2.2:3.6, smooth] plot (\x, {((3/2)^2*\x*\x+2^2)^0.5});
    \draw[blue](0,2) node{$\bullet$};
    \draw[](-0.7,1.8) node{$(-t',h')$};
    \draw[red, domain=0.8:3.3, smooth] plot (\x, {(16*\x*\x-64*\x+64+2.5^2)^0.5});
    \draw[red](2,2.5) node{$\bullet$};
    \draw[](2.6,2.3) node{$(0,h)$};
    \draw[](0.9,3.3) node{$(\hat s_1,\hat h_1)$};
    \draw[](2.5,5.3) node{$(\hat s_2,\hat h_2)$};
     \draw[blue]  (-0.47,0) arc (-180:-360:2 and 3.07);
     \draw[red]  (2.35,0) arc (0:180:0.8 and 3.07);
     \draw[red]  (1.85,0) arc (-180:-360:1.25 and 5.05);
     \draw[blue]  (6.5,0) arc (0:180:3.4 and 5.05);    
    \draw[blue](-0.9,0.6) node{$E^{\hat s_1, v_2}_{\hat h_1}$};
    \draw[red](1.3,0.6) node{$E^{\hat s_1, v_1}_{\hat h_1}$};
    \draw[blue](4.9,1.6) node{$E^{\hat s_2,v_2}_{\hat h_2}$};
     \draw[red](3.4,2.8) node{$E^{\hat s_2,v_1}_{\hat h_2}$};
    \end{tikzpicture}
    \captionsetup{width=0.9\linewidth}
    \caption{The case corresponding to the last term in (\ref{eq:T12h}). For the intersections $(\hat s_1, \hat h_1)$ and $(\hat s_2, \hat h_2)$ to represent two consecutive  handovers, the regions \red{$E^{\hat s_1, v_1}_{\hat h_1} \cup E^{\hat s_2, v_1}_{\hat h_2}$} and  \blue{$E^{\hat s_1, v_2}_{\hat h_1}\cup E^{\hat s_2, v_2}_{\hat h_2}$} must have no points from $\Hcal^1$ and $\Hcal^2$, respectively.}
\label{figure:mixed_birdij2}
\end{figure}
Using the increasing property of the sequence of regions $\left\{E^{t,v_l}_{\hat h_t}\setminus E^{\hat s_1, v_l}_{\hat h_1}\right\}_{t\geq \hat s}$ from part~(\ref{monotonic1}) of Proposition~\ref{prop:monotonic-sets}, in all the terms in (\ref{eq:T12h}), we get
\[
\bigcup_{t\in [\hat s_1,\hat s(0,h,T_j,H_j)]} E^{t,v_l}_{\hat h_t}\cup E^{\hat s_1, v_l}_{\hat h_1}
= E^{\hat s(0,h,T_j,H_j),v_l}_{\hat h(0,h,T_j,H_j)}\cup E^{\hat s_1, v_l}_{\hat h_1},
\]
\[
\bigcup_{t\in [\hat s_1,\hat s_2(0,h,T_j,H_j)]} E^{t,v_l}_{\hat h_t}\cup E^{\hat s_1, v_l}_{\hat h_1}
= E^{\hat s_2(0,h,T_j,H_j),v_l}_{\hat h_2(0,h,T_j,H_j)}\cup E^{\hat s_1, v_l}_{\hat h_1} 
\text{ and } 
\bigcup_{t\in [\hat s_1,\hat s_2]} E^{t,v_l}_{\hat h_t}\cup E^{\hat s_1, v_l}_{\hat h_1}
= E^{\hat s_2,v_l}_{\hat h_2}\cup E^{\hat s_1, v_l}_{\hat h_1},
\]
for $l=1,2$, where the first two equalities depend on whether $(T_j,H_j)$ is of type $1$ or $2$ and the last one is for $\hat s_1$ and $\hat s_2$ to be consecutive handovers. Thus the three terms in (\ref{eq:T12h}) can be written as
\begin{align}
\lefteqn{\E_\Hcal\left[ \sum_{(T_j,H_j)\in \Hcal^1 \,\mbox{:}\, T_j\geq 0} \prod_{l=1,2}\one_{\Hcal^l\left(E^{\hat s(0,h,T_j,H_j),v_l}_{\hat h(0,h,T_j,H_j)}\cup E^{\hat s_1, v_l}_{\hat h_1}\right)=0} \one_{\hat s_1\leq \hat s(0,h,T_j,H_j)\leq  \hat s_2} e^{-\rho (\hat s(0,h,T_j,H_j)-\hat s_1)} \right]}\nn\\
&\;\;+ \E_\Hcal\left[ \sum_{(T_j,H_j)\in \Hcal^2 \,\mbox{:}\, T_j\geq -t'} \prod_{l=1,2}\one_{\Hcal^l\left(E^{\hat s_2(0,h,T_j,H_j),v_l}_{\hat h_2(0,h,T_j,H_j)}\cup E^{\hat s_1, v_l}_{\hat h_1}\right)=0} \one_{\hat s_1\leq \hat s_2(0,h,T_j,H_j)\leq  \hat s_2} e^{-\rho (\hat s(0,h,T_j,H_j)-\hat s_1)}\right]\nn\\
&\;\; + \E_\Hcal\left[ e^{-\rho (\hat s_2-\hat s_1)} \prod_{l=1,2}\one_{\Hcal^l\left(E^{\hat s_2,v_l}_{\hat h_2}\cup E^{\hat s_1, v_l}_{\hat h_1}\right)=0} \right] 
:=\zeta\bbinom{1,1}{2,1,1}+\zeta\bbinom{1,1}{2,1,2}+\zeta\bbinom{1,2}{2,1,2^*},
\label{eq:T12hh}
\end{align}
where the type of the pair of handovers marked with $\cdot^*$ means that we have reappearance of the previous station as an upcoming one, as seen in the last column of Table~\ref{tab:notation-tu}.

Let $R_1(h,t',h')\subset (\R^+)^2\setminus E^{\hat s,v_1}_{\hat h}$ be the region defined as 
\[R_1(h,t',h'):=\left\{(t,u): \hat s_1\leq \hat s(0,h,t,u)\leq \hat s_2\right\},\]
which contains all possible head points that produce such a handover, as discussed in~\ref{subcase:H211x}. It can be proved, using Lemma~\ref{lemma:int1mid2}, that $R_1(h,t',h')=E^{\hat s_2,v_1}_{\hat h_2}\setminus E^{\hat s_1,v_1}_{\hat h_1}$. Applying the Campbell-Mecke formula,  the first term in (\ref{eq:T12hh}) equals 
\vspace{-0.1in}
\begin{align}
\zeta\bbinom{1,1}{2,1,1}&=2\la_1v_1\int_{E^{\hat s_2,v_1}_{\hat h_2}\setminus E^{\hat s_1,v_1}_{\hat h_1}} e^{-\rho (\hat s(0,h,t_1,h_1)-\hat s_1)} \E_\Hcal\left[  \prod_{l=1,2}\one_{\Hcal^l\left(E^{\hat s(0,h,t_1,h_1),v_l}_{\hat h(0,h,t_1,t_1)}\cup E^{\hat s_1,v_l}_{\hat h_1}\right)=0}\right]{\rm d}t_1  {\rm d}h_1\nn\\
&=2\la_1v_1\int_{E^{\hat s_2,v_1}_{\hat h_2}\setminus E^{\hat s_1,v_1}_{\hat h_1}} e^{-\rho (\hat s(0,h,t_1,h_1)-\hat s_1)} e^{-\sum_{l=1,2}2\la_lv_l\left\vert E^{\hat s(0,h,t_1,h_1),v_l}_{\hat h(0,h,t_1,t_1)}\cup E^{\hat s_1,v_l}_{\hat h_1}\right\vert}
{\rm d}t_1  {\rm d}h_1.
\label{eq:R211in}
\end{align}
Let $R_2(h,t',h')\subset Q_{-t'}\setminus E^{\hat s,v_2}_{\hat h}$ be the region defined as
$R_2(h,t',h'):=\left\{(t,u)\,\mbox{:}\, \hat s_1\leq \hat s_2(0,h,t,u)\leq \hat s_2\right\}$, containing the head point that gives the required handover, as seen in~\ref{subcase:H121ly}. It can be proved using Lemma~\ref{lemma:int1mid2}, that the region $R_2(h,t',h')=E^{\hat s_2,v_2}_{\hat h_2}\setminus E^{\hat s_1,v_2}_{\hat h_1}$. Applying the Campbell-Mecke formula, the second term in (\ref{eq:T12hh}) equals
\vspace{-0.1in}
\begin{align}
\zeta\bbinom{1,1}{2,1,2}&=2\la_2v_2 \int_{E^{\hat s_2,v_2}_{\hat h_2}\setminus E^{\hat s_1,v_2}_{\hat h_1}}  e^{-\rho (\hat s_2(0,h,t_1,h_1)-\hat s_1)} \E_\Hcal\left[  \prod_{l=1,2}\one_{\Hcal^l\left(E^{\hat s_2(0,h,t_1,h_1),v_l}_{\hat h_2(0,h,t_1,t_1)}\cup E^{\hat s_1,v_l}_{\hat h_1}\right)=0}\right]{\rm d}t_1  {\rm d}h_1\nn\\
&= 2\la_2v_2 \int_{E^{\hat s_2,v_2}_{\hat h_2}\setminus E^{\hat s_1,v_2}_{\hat h_1}}  e^{-\rho (\hat s_2(0,h,t_1,h_1)-\hat s_1)} e^{-\sum_{l=1,2}2\la_lv_l\left\vert E^{\hat s_2(0,h,t_1,h_1),v_l}_{\hat h_2(0,h,t_1,t_1)}\cup E^{\hat s_1,v_l}_{\hat h_1}\right\vert}
{\rm d}t_1  {\rm d}h_1.
\label{eq:R212in}
\end{align}
\vspace{-0.1in}
The third term in (\ref{eq:T12hh}) equals
\begin{align}
\zeta\bbinom{1,2}{2,1,2^*}&=e^{-\rho (\hat s_2-\hat s_1)} e^{-\sum_{l=1,2}2\la_lv_l\left\vert E^{\hat s_2,v_l}_{\hat h_2}\cup E^{\hat s_1,v_l}_{\hat h_1}\right\vert}.
\label{eq:s1s2}
\end{align}
\vspace{-0.02in}
Define
\vspace{-0.03in}
\begin{align}
\zeta^{(1)}_{2,1}(\rho,h,t',h')    &:=\zeta\bbinom{1,1}{2,1,1}+\zeta\bbinom{1,1}{2,1,2}+\zeta\bbinom{1,2}{2,1,2^*},
\label{eq:zeta221}
\end{align}
using (\ref{eq:R211in})--(\ref{eq:s1s2}). We obtain the result by substituting  the expression of $\zeta^{(2)}_{2,1}(\rho,h,t',h')$ from (\ref{eq:zeta221}) back to both terms in (\ref{eq:T12g}), as 
\begin{align}
\frac{L^{(1)}_{2,1}}{4\la_1\la_2 v_1v_2}\E^0_{ \Wcal_{2,1}^{(1)}} [e^{-\rho T}]& =   \int_{0}^\infty  \int_{0}^h \int_{t^*}^\infty 
\zeta^{(1)}_{2,1}(\rho,t',h',h) {\rm d}t'  {\rm d}h' {\rm d}h   +  \int_{0}^\infty  \int_h^\infty \int_0^{\infty}
\zeta^{(1)}_{2,1}(\rho,t',h',h) {\rm d}t'  {\rm d}h' {\rm d}h \nn\\
&:=  \xi_{2,1}^{(1)}(\rho,v_2,v_1),
\label{eq:T12j}
\end{align}
which is a function of $\rho, v_2, v_1$ only.\qed
\subsubsection{Proof of Lemma~\ref{lemma:LT_1ijk}, type $\binom{2}{2,1}$ or equivalently $\bbinom{2}{1,2}$}~\label{subsubsection:Lemma212}
Using Lemma~\ref{lemma:Palm-Lcal-ij}, the mixed Palm expectations in (\ref{eq:decompLT2}) with $\t_n=2, \t_p=1$, $q=2$, and $f(\Hcal)=e^{-\rho T}$, we have
\begin{align}
\frac{L^{(2)}_{1,2}}{4\la_1\la_2 v_1v_2}\E^0_{\Wcal_{1,2}^{(2)}} [e^{-\rho T}]
& =   \int_{0}^\infty  \int_{0}^h \int_{t^*}^\infty 
	\mathbb{E}_\Hcal\Bigg[ e^{-\rho(T\circ\th_{\hat s_2})}  \prod_{l=1,2}\one_{\Hcal^l\left(E^{\hat{s}_2,v_l}_{\hat{h}_2}\right)=0}
	\Bigg] {\rm d}t'  {\rm d}h' {\rm d}h \nn\\ 
    &\;+  \int_{0}^\infty  \int_h^\infty \int_0^{\infty}
	\mathbb{E}_\Hcal\Bigg[ e^{-\rho(T\circ\th_{\hat s_2})} \prod_{l=1,2}\one_{\Hcal^l\left(E^{\hat{s}_2,v_l}_{\hat{h}_2}\right)=0}
	\Bigg] {\rm d}t'  {\rm d}h' {\rm d}h,
    \label{eq:T12k}
\end{align}
where there are two radial birds of type $\t_n=2$ and $\t_p=1$, respectively, with head at $(t_p,h_p)=(-t',h')$ and $(t_n,h_n)=(0,h)$ and we write in short $(\hat{s}_2,\hat{h}_2)$ for $(\hat{s}_2(0,h,-t',h'),\hat{h}_2(0,h,-t',h'))$ to denote the second intersection, see Figure~\ref{figure:UUtau-MS}. Note that, since $\t_n=2$ and $q=2$, the station corresponding to the head point $(-t',h')$ is the serving station after time $\hat s_2$. Since $(-t',h')$ is of type $2$, we define $\hat h_t:=\left(v_2^2(t+t')^2+h'^2\right)^\half$. Also note that 
\[
T\circ\th_{\hat s_2}:=T_1(\hat s_2)-\hat s_2,
\]
where $T_1(\hat s_2)$ is the time point of the next handover after time $\hat s_2$. The inner expectation in both terms in (\ref{eq:T12k}) can be evaluated by looking for the next handover created by a radial bird with its head at $(T_j,H_j)$ (say), such that it must lie outside $\overline{E^{\hat s_2, v_1}_{\hat h_2}}$ or $\overline{E^{\hat s_2, v_2}_{\hat h_2}}$, depending on its type.
\begin{case}
[\textbf{\em $(T_j,H_j)$ is of type 1.}]\label{case:H2121} We must have $T_j\geq \hat s_2$. Otherwise, if $T_j<\hat s_2$, then the intersection produced by the radial bird with head at $(T_j,H_j)$, with the one at $(0,h)$, is either contained  in $\Lcal^+_e(v_1,v_2)$ or it is responsible for a handover before the time $\hat s_2$. Thus $T_j\geq \hat s_2$ and the unexplored region is $Q_{\hat s_2}\setminus \overline{E^{\hat s_2,v_1}_{\hat h_2}}$. The type of the pair of consecutive handovers is $\bbinom{2,1}{1,2,1}$. Also $(T_j,H_j)$ must be such that, $T_j\in D_r(-t',h', H_j)$, given $t',h', H_j$, for $(T_j,H_j)$ to be in this case. Suppose $(\hat s_1(-t',h',T_j,H_j),\hat h_1(-t',h',T_j,H_j))$ and $(\hat s_2(-t',h',T_j,H_j),\hat h_2(-t',h',T_j,H_j))$ are the two intersection point of the radial birds with head at $(-t',h'), (T_j,H_j)$, such that $\hat s_1(-t',h',T_j,H_j)\leq \hat s_2(-t',h',T_j,H_j)$. 

Observe that, the point $(T_j,H_j)$ being outside $\overline{E^{\hat s_2,v_1}_{\hat h_2}}$,  we have $\hat s_2\leq \hat s_1(-t',h',T_j,H_j)$, see picture~(\subref{subfigure:UUtau-MS1}) of Figure~\ref{figure:UUtau-MS}. For the two intersection points to represent handovers, we must have no point of $\Hcal^l$, for $l=1,2$, in the extra regions $\interior{\left(\bigcup_{t\in [\hat s_2,\hat s_1(-t',h',T_j,H_j)]} E^{t,v_l}_{\hat h_t}\setminus E^{\hat s_2, v_l}_{\hat h_2}\right)}$ and $\interior{\left(\bigcup_{t\in [\hat s_2,\hat s_2(-t',h',T_j,H_j)]} E^{t,v_l}_{\hat h_t}\setminus E^{\hat s_2, v_l}_{\hat h_2}\right)}$, respectively, beyond $\overline{E^{\hat s_2, v_l}_{\hat h_2}}$, where $\hat h_t=(v_2^2(t+t')^2+h'^2)^\half$. Since $\hat s_1(-t',h',T_j,H_j)\leq \hat s_2(-t',h',T_j,H_j)$, we have 
\[
\bigcup_{t\in [\hat s_2,\hat s_1(-t',h',T_j,H_j)]} E^{t,v_l}_{\hat h_t}\setminus E^{\hat s_2, v_l}_{\hat h_2}\subset \bigcup_{t\in [\hat s_2,\hat s_2(-t',h',T_j,H_j)]} E^{t,v_l}_{\hat h_t}\setminus E^{\hat s_2, v_l}_{\hat h_2}.
\]
Hence, the first intersection $(\hat s_1(-t',h',T_j,H_j),\hat h_1(-t',h',T_j,H_j))$ gives the next handover. Thus we have $T_1(\hat s_1):=\hat s_1(0,h,T_j,H_j)$.
\end{case}
\begin{case}
[\textbf{\em $(T_j,H_j)$ is of type 2.}]\label{case:H2122} In this case, the point must be such that $T_j\geq -t'$. Otherwise, the birds with head at $(T_j,H_j)$ and $(-t',h')$, intersects before time $\hat s_2$, which is not of interest here, since we are looking a future handovers after $\hat s_2$. Thus the unexplored region is $Q_{-t'}\setminus \overline{E^{\hat s_2,v_2}_{\hat h_2}}$ and the type of the pair of consecutive handovers is $\bbinom{2,1}{1,2,2}$. The handover must be given by $(\hat s(-t',h',T_j,H_j),\hat h(-t',h',T_j,H_j))$, the only intersection between the birds at $(-t',h')$ and $(T_j,H_j)$, see picture~(\subref{subfigure:UUtau-MS2}) of Figure~\ref{figure:UUtau-MS}. Additionally, we have $T_1(\hat s_1):=\hat s_2(0,h,T_j,H_j)$. The union is due to the Definition~\ref{definition:T1-2speed} and the individual sets in the union may not be monotonic. 
\end{case}
\begin{figure}[ht!]
 \centering
      \begin{subfigure}[t]{0.43\linewidth}
       \centering
      \begin{tikzpicture}[scale=0.65, every node/.style={scale=0.7}]
\pgftransformxscale{0.7}  
\pgftransformyscale{0.7}    
\draw[->] (-1, 0) -- (10, 0) node[right] {$t$};
   \draw[blue, domain=-1.8:5.5, smooth] plot (\x, {((3/2)^2*\x*\x+2^2)^0.5});
    \draw[blue](0,2) node{$\bullet$};
    \draw[](-1.1,1.65) node{$(-t',h')$};
    \draw[red, domain=0.1:3.8, smooth] plot (\x, {(16*\x*\x-64*\x+64+2.5^2)^0.5});
    \draw[red](2,2.5) node{$\bullet$};
    \draw[red, domain=4.5:8, smooth] plot (\x, {(16*\x*\x-2*16*6.5*\x+16*6.5*6.5+4^2)^0.5});
    \draw[red](6.5,4) node{$\bullet$};
    \draw[](7,3) node{$(t_1,h_1)$};
    \draw[](2.8,2.5) node{$(0,h)$};
    \draw[](2.2,5.5) node{$(\hat s_2,\hat h_2)$};
    \draw[blue](0.7,0.6) node{$E^{\hat s_2, v_2}_{\hat h_2}$};
     \draw[red](2.5,0.6) node{$E^{\hat s_2, v_1}_{\hat h_2}$};
     \draw[red]  (1.85,0) arc (-180:-360:1.25 and 5.05);
     \draw[blue]  (6.5,0) arc (0:180:3.4 and 5.05);
     \draw[red]  (3.05,0) arc (-180:-360:1.85 and 7.59);
     \draw[blue]  (9.45,0) arc (0:180:4.8 and 7.59);
     \draw[blue](7.5,1.6) node{$E^{\tilde s, v_2}_{\tilde h}$};
     \draw[red](5,1.6) node{$E^{\tilde s, v_1}_{\tilde h}$};
    \draw[](6,7.4) node{$(\tilde s,\tilde h)$};
\draw[<-] (3.4, 5.2) -- (4.4, 6) node[right] {$S_3$};
\end{tikzpicture}
\caption{For the intersection $(\hat s_2, \hat h_2)$ and $(\tilde s, \tilde h)$ to represent two consecutive  handovers, the regions \red{$S_3\cup E^{\hat s_2, v_1}_{\hat h_2} \cup E^{\tilde s,v_1}_{\tilde h}$} and \blue{$E^{\hat s_2, v_2}_{\hat h_2}\cup E^{\tilde s,v_2}_{\tilde h}$} must have no points from $\Hcal^1$ and $\Hcal^2$, respectively. Here $(\tilde s, \tilde h)$ is a realization of $(\hat s_1(-t',h',T_j,H_j),\hat h_1(-t',h',T_j,H_j))$.}
\label{subfigure:UUtau-MS1}
\end{subfigure}
\hspace{0.2in}
\centering
\begin{subfigure}[t]{0.43\linewidth}
 \centering
\begin{tikzpicture}[scale=0.65, every node/.style={scale=0.7}]
\pgftransformxscale{0.7}  
\pgftransformyscale{0.7}    
    \draw[->] (-2, 0) -- (11, 0) node[right] {$t$};
   \draw[blue, domain=-1.8:5, smooth] plot (\x, {((3/2)^2*\x*\x+2^2)^0.5});
    \draw[blue](0,2) node{$\bullet$};
    \draw[](-1.1,1.65) node{$(-t',h')$};
    \draw[red, domain=0.1:4.3, smooth] plot (\x, {(16*\x*\x-64*\x+64+2.5^2)^0.5});
    \draw[red](2,2.5) node{$\bullet$};
    \draw[blue, domain=3.5:11, smooth] plot (\x, {((3/2)^2*\x*\x-(3/2)^2*2*8.5*\x+(3/2)^2*8.5*8.5+5^2)^0.5});
    \draw[blue](8.5,5) node{$\bullet$};
    \draw[](8,4) node{$(t_1,h_1)$};
    \draw[](2.8,2.5) node{$(0,h)$};
    \draw[](2.2,5.5) node{$(\hat s_2,\hat h_2)$};
    \draw[<-] (3.4, 5.2) -- (4.4, 6) node[right] {$S_4$};
    \draw[blue](0.7,0.6) node{$E^{\hat s_2, v_2}_{\hat h_2}$};
     \draw[red](3,0.6) node{$E^{\hat s_2, v_1}_{\hat h_2}$};
     \draw[red]  (1.85,0) arc (-180:-360:1.25 and 5.05);
     \draw[blue]  (6.5,0) arc (0:180:3.4 and 5.05);
     \draw[red]  (2.85,0) arc (-180:-360:2 and 7.5);
     \draw[blue]  (9.88,0) arc (0:180:5.05 and 7.5);
     \draw[blue](8,1.6) node{$E^{\tilde s, v_2}_{\tilde h}$};
     \draw[red](5.5,5) node{$E^{\tilde s, v_1}_{\tilde h}$};
    \draw[](6,8) node{$(\tilde s,\tilde h)$};
    \end{tikzpicture}
    \caption{For the intersection $(\hat s_2, \hat h_2)$ and $(\tilde s, \tilde h)$ to represent two consecutive  handovers, the regions \red{$S_4\cup E^{\hat s_2, v_1}_{\hat h_2} \cup E^{\tilde s,v_1}_{\tilde h}$} and \blue{$E^{\hat s_2, v_2}_{\hat h_2}\cup E^{\tilde s,v_2}_{\tilde h}$} must have no points from $\Hcal^1$ and $\Hcal^2$, respectively. Here $(\tilde s, \tilde h)$ is a realization of $(\hat s(-t',h',T_j,H_j),\hat h(-t',h',T_j,H_j))$.}
    \label{subfigure:UUtau-MS2}
    \end{subfigure}
    \captionsetup{width=0.9\linewidth}
    \caption{The two pictures correspond to the two cases in (\ref{eq:T_121a}), where the head point $(T_j,H_j)=(t_1,h_1)$ is of type $1$ or $2$.}
\label{figure:UUtau-MS}
\end{figure}
The last discussion leads to a decomposition of the inner expectation in (\ref{eq:T12k}) as
\begin{align}
\lefteqn{\mathbb{E}_\Hcal\Bigg[ e^{-\rho(T_1(\hat s_2)-\hat s_2)}  \prod_{l=1,2}\one_{\Hcal^l\left(E^{\hat{s}_2,v_l}_{\hat{h}_2}\right)=0}
\Bigg]}\nn\\
&= \E_\Hcal\left[ \one_{\exists (T_j,H_j)\in \Hcal^1 \,\mbox{:}\, T_j\geq \hat s_2,\, T_j\in D_r(-t',h', H_j),\, \bigcap_{l=1,2}\left\{\Hcal^l\interior{\left(\bigcup_{t\in [\hat s_2,\hat s_1(-t',h',T_j,H_j)]} E^{t,v_l}_{\hat h_t}\setminus E^{\hat s_2, v_l}_{\hat h_2}\right)}=0\right\}} \right.\nn\\
&\hspace{3.7in}\times\left. e^{-\rho (\hat s(-t',h',T_j,H_j)-\hat s_2)}\prod_{l=1,2}\one_{\Hcal^l\left(E^{\hat{s}_2,v_l}_{\hat{h}_2}\right)=0}\right] \nn\\
&\;\;+ \E_\Hcal\left[ \one_{\exists (T_j,H_j)\in \Hcal^2 \,\mbox{:}\,  \bigcap_{l=1,2}\left\{\Hcal^l\interior{\left(\bigcup_{t\in [\hat s_2,\hat s(-t',h',T_j,H_j)]} E^{t,v_l}_{\hat h_t}\setminus E^{\hat s_2, v_l}_{\hat h_2}\right)}=0\right\}} e^{-\rho (\hat s(-t',h',T_j,H_j)-\hat s_2)}\right.\nn\\
&\hspace{3.7in}\times\left. \prod_{l=1,2}\one_{\Hcal^l\left(E^{\hat{s}_2,v_l}_{\hat{h}_2}\right)=0}\right].
\label{eq:T_121a}
\end{align}
Since there exists a unique head point satisfying the condition in the last sum, it  reduces to
\begin{align}
\lefteqn{\hspace{-0.9in}\E_\Hcal\left[ \sum_{(T_j,H_j)\in \Hcal^1 \,\mbox{:}\, T_j\geq \hat s_2, T_j\in D_r(-t',h', H_j)} \prod_{l=1,2}\one_{\Hcal^l\left(\bigcup_{t\in [\hat s_2,\hat s_1(-t',h',T_j,H_j)]} E^{t,v_l}_{\hat h_t}\cup E^{\hat s_2, v_l}_{\hat h_2}\right)=0}  e^{-\rho (\hat s_1(-t',h',T_j,H_j)-\hat s_2)} \right]}\nn\\
&\hspace{-0.4in} + \E_\Hcal\left[ \sum_{(T_j,H_j)\in \Hcal^2} \prod_{l=1,2}\one_{\Hcal^l\left(\bigcup_{t\in [\hat s_2,\hat s(-t',h',T_j,H_j)]} E^{t,v_l}_{\hat h_t}\cup E^{\hat s_2, v_l}_{\hat h_2}\right)=0}  e^{-\rho (\hat s(-t',h',T_j,H_j)-\hat s_2)} \right]\nn\\
&\hspace{-0.6in}:=\zeta\bbinom{2,1}{1,2,1}+\zeta\bbinom{2,1}{1,2,2}.
\label{eq:T_121}
\end{align}
For the first term in (\ref{eq:T_121}), using Remark~\ref{remark:union} for $l=1$, we have 
\begin{equation}
\bigcup_{t\in [\hat s_2,\hat s_1(-t',h',T_j,H_j)]} E^{t,v_1}_{\hat h_t}\cup E^{\hat s_2, v_1}_{\hat h_2}= E^{\hat s_2, v_1}_{\hat h_2}\cup S_3\cup E^{\hat s_1(-t',h',T_j,H_j),v_1}_{\hat h_1(-t',h',T_j,H_j)},
\label{eq:S3}
\end{equation}
where $\hat h_t=(v_2^2(t+t')^2+h'^2)^\half$ and $S_3$ is the region outside $E^{\hat s_2, v_1}_{\hat h_2}\cup E^{\hat s_1(-t',h',T_j,H_j),v_1}_{\hat h_1(-t',h',T_j,H_j)}$ but below the hyperbola given by the equation $u^2-\frac{v_1^2v_2^2}{v_1^2-v_2^2}(t+t')^2=h'^2$, in the  $(t,u)$-coordinate system similar to (\ref{eq:2bird-hyp}) and within the time interval 
\[
\left[\frac{1}{v_1^2}\left((v_1^2-v_2^2)\hat s_2 -v_2^2t'\right), \frac{1}{v_1^2}\left((v_1^2-v_2^2)\hat s_1(-t',h',T_j,H_j)-v_2^2t'\right)\right].
\]
For $l=2$, using the increasing property of $\left\{E^{t,v_2}_{\hat h_t}\setminus E^{\hat s, v_2}_{\hat h}\right\}_{t\geq \hat s_2}$, from part~(\ref{monotonic2}) of Proposition~\ref{prop:monotonic-sets}, we have
\[
\bigcup_{t\in [\hat s_2,\hat s_1(-t',h',T_j,H_j)]} E^{t,v_2}_{\hat h_t}\cup E^{\hat s, v_2}_{\hat h}= E^{\hat s_1(-t',h',T_j,H_j),v_2}_{\hat h_1(-t',h',T_j,H_j)}\cup  E^{\hat s, v_2}_{\hat h}.
\]
Then the first term in (\ref{eq:T_121}) equals
\begin{align}
\lefteqn{\E_\Hcal\left[ \sum_{(T_j,H_j)\in \Hcal^1\,\mbox{:}\, T_j\geq \hat s_2, T_j\in D_r(-t',h', H_j)} \one_{\Hcal^1\left( E^{\hat s_1(-t',h',T_j,H_j),v_1}_{\hat h_1(-t',h',T_j,H_j)}\cup S_3\cup E^{\hat s_2, v_1}_{\hat h_2}\right)=0} \right.}\nn\\
&\hspace{2.2in} \times\left. \one_{\Hcal^2\left( E^{\hat s_1(-t',h',T_j,H_j),v_2}_{\hat h_1(-t',h',T_j,H_j)}\cup E^{\hat s_2, v_2}_{\hat h_2}\right)=0}  e^{-\rho (\hat s_1(-t',h',T_j,H_j)-\hat s_2)} \right].
\end{align}
The point $(T_j,H_j)$ must lie in the region $R_1(h,t',h')$, where 
\begin{equation}
R_1(h,t',h'):=\left\{(t,u)\,\mbox{:}\,t\geq \hat s_2,\, v_1^2(t-\hat s_2)^2+u^2\geq \hat h_2^2 \; \text{ and }\; u^2-\frac{v_1^2v_2^2}{v_1^2-v_2^2}(t+t')^2\leq h'^2 \right\}.
\label{eq:R1-212}
\end{equation}
For the region $R_1(h,t',h')$ in (\ref{eq:R1-212}), the first inequality $t\geq \hat s_2$ is due to the fact that the head point should lie on the right of $(\hat s_2, \hat h_2)$, as discussed in~\ref{case:H2121} of this proof. The second and the third condition are because the next head point must lie outside the region $\overline{E^{\hat s_2,v_1}_{\hat h_2}}$ and below the hyperbola $u^2- \frac{v_1^2v_2^2}{v_1^2-v_2^2} (t+t')^2=h'^2$, by Property~\ref{hyp}. The latter is equivalent to the intersection criteria of two birds. Let $ \frac{v_1^2v_2^2}{v_1^2-v_2^2}:= v^2$.
Applying the Campbell-Mecke formula, the first term in (\ref{eq:T_121}) equals
\begin{align}
\zeta\bbinom{2,1}{1,2,1}
&=2\la_1v_1 
\int_{R_1(h,t',h')} e^{-\rho (\hat s_1(-t',h',t_1,h_1)-\hat s_2)}  \E_{\Hcal}\left[ \one_{\Hcal^1\left(E^{\hat s_1(-t',h',t_1,h_1),v_1}_{\hat h_1(-t',h',t_1,h_1)}\cup S_3\cup E^{\hat s_2,v_1}_{\hat h_2}\right)=0}\right.\nn\\
&\hspace{2in}\times \left.\one_{\Hcal^2\left(E^{\hat s_1(-t',h',t_1,h_1),v_2}_{\hat h_1(-t',h',t_1,h_1)}\cup E^{\hat s_2,v_2}_{\hat h_2}\right)=0}\right] {\rm d}t_1  {\rm d}h_1.  
\label{eq:R21-1outa}
\end{align}
Using the region $R_1(h,t',h')$ from (\ref{eq:R1-212}) in (\ref{eq:R21-1out}) and computing the void probabilities we get
\begin{align}
\zeta\bbinom{2,1}{1,2,1}
&{=}2\la_1v_1\!\!\! 
\int_{\hat s_2}^{\hat s_2+\hat h_2/v_1} \!\!\! \!\! \int_{\left(\hat h_2^2- v_1^2(t_1-\hat s_2)^2\right)^\half}^{\left(v^2 (t_1+t')^2+ h'^2\right)^\half} \!\!\! \!\!  e^{-\rho (\hat s_1(-t',h',t_1,h_1)-\hat s_2)}  e^{-2\la_1v_1\vert E^{\hat s_1(-t',h',t_1,h_1),v_1}_{\hat h_1(-t',h',t_1,h_1)}\cup S_3\cup E^{\hat s_2,v_1}_{\hat h_2}\vert} \nn\\
&\hspace{3in}\times  e^{-2\la_2 v_2\vert E^{\hat s_1(-t',h',t_1,h_1),v_2}_{\hat h_1(-t',h',t_1,h_1)}\cup E^{\hat s_2,v_2}_{\hat h_2}\vert} {\rm d}h_1  {\rm d}t_1\nn\\
&\;\; +2\la_1v_1 
\int_{\hat s_2+\hat h_2/v_1}^\infty \!\int_0^{\left(v^2 (t_1+t')^2+ h'^2\right)^\half} \!\!\!\!\!\!\!\! e^{-\rho (\hat s_1(-t',h',t_1,h_1)-\hat s_2)}   e^{-2\la_1v_1\vert S_3\cup E^{\hat s_1(-t',h',t_1,h_1),v_1}_{\hat h_1(-t',h',t_1,h_1)}\cup E^{\hat s_2,v_1}_{\hat h_2}\vert}\nn\\
&\hspace{2.5in}\times e^{-2\la_2v_2\vert E^{\hat s_1(-t',h',t_1,h_1),v_2}_{\hat h_1(-t',h',t_1,h_1)}\cup E^{\hat s_2,v_2}_{\hat h_2}\vert} {\rm d}h_1  {\rm d}t_1.
\label{eq:R21-1out}
\end{align}
For the second term in (\ref{eq:T_121}), using Remark~\ref{remark:union} for $l=1$,   
\begin{equation}
\bigcup_{t\in [\hat s_2,\hat s(-t',h',T_j,H_j)]} E^{t,v_1}_{\hat h_t}\cup E^{\hat s_2, v_1}_{\hat h_2}= E^{\hat s_2, v_1}_{\hat h_2}\cup S_4\cup E^{\hat s(-t',h',T_j,H_j),v_1}_{\hat h(-t',h',T_j,H_j)},
\label{eq:S4}
\end{equation}
where $\hat h_t=(v_2^2(t+t')^2+h'^2)^\half$ and $S_4$ is the region outside $E^{\hat s_2, v_1}_{\hat h_2}\cup E^{\hat s(-t',h',T_j,H_j),v_1}_{\hat h(-t',h',T_j,H_j)}$ but below the hyperbola given by the equation $u^2-\frac{v_1^2v_2^2}{v_1^2-v_2^2}(t+t')^2=h'^2$, in the $(t,u)$-coordinate system similar to (\ref{eq:2bird-hyp}) and within the time interval 
\[
\left[\frac{1}{v_1^2}\left((v_1^2-v_2^2)\hat s_2 -v_2^2t'\right), \frac{1}{v_1^2}\left((v_1^2-v_2^2)\hat s(-t',h',T_j,H_j)-v_2^2t'\right)\right].
\]
For $l=2$, using part~(\ref{monotonic2}) of Proposition~\ref{prop:monotonic-sets}, the increasing property of $\left\{E^{t,v_2}_{\hat h_t}\setminus E^{\hat s, v_2}_{\hat h}\right\}_{t\geq \hat s_2}$, we have
\[
\bigcup_{t\in [\hat s_2,\hat s(-t',h',T_j,H_j)]} E^{t,v_2}_{\hat h_t}\cup E^{\hat s, v_2}_{\hat h}= E^{\hat s (-t',h',T_j,H_j),v_2}_{\hat h(-t',h',T_j,H_j)}\cup  E^{\hat s, v_2}_{\hat h}.
\]
The second term in (\ref{eq:T_121}) equals
\begin{align}
\lefteqn{\hspace{-3.3in}\E_\Hcal\left[ \sum_{(T_j,H_j)\in \Hcal^2} \one_{\Hcal^l\left(E^{\hat s_2, v_1}_{\hat h_2}\cup S_4\cup E^{\hat s(-t',h',T_j,H_j),v_1}_{\hat h(-t',h',T_j,H_j)}\right)=0} \one_{\Hcal^l\left(E^{\hat s_2, v_2}_{\hat h_2}\cup E^{\hat s(-t',h',T_j,H_j),v_2}_{\hat h(-t',h',T_j,H_j)}\right)=0} e^{-\rho (\hat s(-t',h',T_j,H_j)-\hat s_2)} \right].}\nn
\end{align}
The point $(T_j,H_j)$ must lie in the region $R_2(h,t',h')$, where 
\begin{equation}
R_2(h,t',h'):=\left\{(t,u):t\geq -t', u\geq 0, \text{ and } v_2^2(t-\hat s_2)^2+u^2\geq \hat h_2^2 \right\}.
\label{eq:R2-211}
\end{equation}
For the region $R_2(h,t',h')$ in (\ref{eq:R2-211}), the first condition $t\geq -t'$, is as discussed in~\ref{case:H2122} of the proof. The last condition is due to the fact that the next head point must lie outside $\overline{E^{\hat s_2, v_2}_{\hat h_2}}$. 

Applying the Campbell-Mecke formula and evaluating void probabilities, the second term in (\ref{eq:T_121}) equals
\begin{align}
\zeta\bbinom{2,1}{1,2,2}&=2\la_2v_2 
\int_{R_2(h,t',h')} e^{-\rho (\hat s(-t',h',t_1,h_1)-\hat s_2)}  \E_{\Hcal}\left[ \one_{\Hcal^1\left(E^{\hat s(-t',h',t_1,h_1),v_1}_{\hat h(-t',h',t_1,h_1)}\cup S_4\cup E^{\hat s_2,v_1}_{\hat h_2}\right)=0} \right.\nn\\
&\hspace{3in}\times \left.\one_{\Hcal^2\left(E^{\hat s(-t',h',t_1,h_1),v_2}_{\hat h(-t',h',t_1,h_1)}\cup E^{\hat s_2,v_2}_{\hat h_2}\right)=0}\right] {\rm d}t_1  {\rm d}h_1\nn\\ 
&=2\la_2v_2 \int_{-t'}^{\hat s_2}\int_{\left(\hat h_2^2-v_2^2(t_1-\hat s_2)^2\right)^\half}^\infty e^{-\rho (\hat s(-t',h',t_1,h_1)-\hat s_2)}  e^{-2\la_1 v_1 \left\vert E^{\hat s(-t',h',t_1,h_1),v_1}_{\hat h(-t',h',t_1,h_1)}\cup S_4\cup E^{\hat s_2,v_1}_{\hat h_2}\right\vert}\nn\\
&\hspace{3.5in}\times  
e^{-2\la_2 v_2 \left\vert E^{\hat s(-t',h',t_1,h_1),v_2}_{\hat h(-t',h',t_1,h_1)}\cup E^{\hat s_2,v_2}_{\hat h_2}\right\vert}{\rm d}t_1  {\rm d}h_1 \nn\\
&\;\;+2\la_2v_2 \int_{\hat s_2}^{\hat s_2+\hat h_2/v_2}\int_{\left(\hat h_2^2-v_2^2(t_1-\hat s_2)^2\right)^\half}^\infty e^{-\rho (\hat s(-t',h',t_1,h_1)-\hat s_2)}  e^{-2\la_1 v_1 \left\vert E^{\hat s(-t',h',t_1,h_1),v_1}_{\hat h(-t',h',t_1,h_1)}\cup S_4\cup E^{\hat s_2,v_1}_{\hat h_2}\right\vert}\nn\\
&\hspace{3.5in}\times  
e^{-2\la_2 v_2 \left\vert E^{\hat s(-t',h',t_1,h_1),v_2}_{\hat h(-t',h',t_1,h_1)}\cup E^{\hat s_2,v_2}_{\hat h_2}\right\vert}{\rm d}t_1  {\rm d}h_1 \nn\\
&\;\;+2\la_2v_2 \int_{\hat s_2+\hat h_2/v_2}^\infty \int_0^\infty e^{-\rho (\hat s(-t',h',t_1,h_1)-\hat s_2)}  e^{-2\la_1 v_1 \left\vert E^{\hat s(-t',h',t_1,h_1),v_1}_{\hat h(-t',h',t_1,h_1)}\cup S_4\cup E^{\hat s_2,v_1}_{\hat h_2}\right\vert}\nn\\
&\hspace{3in}\times  
e^{-2\la_2 v_2 \left\vert E^{\hat s(-t',h',t_1,h_1),v_2}_{\hat h(-t',h',t_1,h_1)}\cup E^{\hat s_2,v_2}_{\hat h_2}\right\vert}{\rm d}t_1  {\rm d}h_1.
\label{eq:R21-2out}
\end{align}
Let us denote the sum of all the terms in (\ref{eq:R21-1out}) and (\ref{eq:R21-2out}) as $\zeta^{(2)}_{1,2}(\rho, t',h',h)= \zeta\bbinom{2,1}{1,2,1}+\zeta\bbinom{2,1}{1,2,2}$. Using the expression for $\zeta^{(2)}_{1,2}(\rho, t',h',h)$ back in (\ref{eq:T12k}) yields
\begin{align}
\frac{L^{(2)}_{1,2}}{4\la_1\la_2 v_1v_2}\E^0_{\Wcal_{1,2}^{(2)}} [e^{-\rho T}]
& =  \int_{0}^\infty  \int_{0}^h \int_{t^*}^\infty 
\zeta^{(2)}_{1,2}(\rho, t',h',h) {\rm d}t'  {\rm d}h' {\rm d}h  +  \int_{0}^\infty  \int_h^\infty \int_0^{\infty}
\zeta^{(2)}_{1,2}(\rho, t',h',h) {\rm d}t'  {\rm d}h' {\rm d}h \nn\\
&:=\xi_{1,2}^{(2)}(\rho, v_2,v_1),
\label{eq:T12m}
\end{align}
which is a function of $\rho, v_2, v_1$ only. \qed
\begin{remark}
In all these expressions, the common thread is the union of the possibly overlapping region enclosed by half ellipses, $E^{s_1,v}_{h_1}\cup E^{s_2,v}_{h_2}$. We have computed the area of these regions in Appendix~\ref{subsection:Uellipses}, using the variables corresponding to the head points.
In particular, the area of the regions of the form of $S_1, S_2, \cdots, S_6$ defined in (\ref{eq:S1a}), (\ref{eq:S2a}), (\ref{eq:S3}), (\ref{eq:S4}), (\ref{eq:S5}) and (\ref{eq:S6}), respectively, can be determined using the idea in Remark~\ref{remark:union}.
\label{remark:MS}
\end{remark}
\begin{remark}
Consider the multi-speed case with set of speeds $v_{[n]}$. Suppose $S^v$ is one of the regions mentioned in Remark~\ref{remark:MS}, for some speed $v\in v_{[n]}$. Then for all $v'\in v_{[n]}$, such that $v'>v$, the region $S^{v'}\subset S^v$. For all $v''\in v_{[n]}$, such that $v''<v$, $S^v$ is contained in the union of half ellipses $E^{\hat s, v''}_{\hat h} \cup E^{\hat s',v''}_{\hat h'}$. This is the only point that needs to be taken care separately in the multi-speed case.
\label{remark:nMS}
\end{remark}
%


\section{Handover process as a factor of a Markov chain: two-speed case} \label{sec-MC2}
Based on the information provided by Table~\ref{tab:notation-tu},  one can construct a discrete Markov chain extending that of Section~\ref{sec-MC1} to the two-speed case. Define the collection of types of handovers as
\[
\Tcal:=\left\{\bbinom{q}{i,j}\,\mbox{:}\, q=1 \text{ if } i=j \text{ and } q=1,2 \text{ for } i\neq j, \text{ where } i,j\in \{1,2\}\right\}.
\]
For each step $k\in \Z$, define $type(k)\in \Tcal$ to be the type of the handover at that step. The goal of this section is to construct a sequence $\left\{\left(H_k^l, T^r_k, H^r_k\right), type(k)\right\}_{k\in \Z}$ that forms a discrete time Markov chain on the state space $(\R^+)^3\times \Tcal$, where the sequence $\left\{\left(H_k^l, T^r_k, H^r_k\right)\right\}_{k\in \Z}$ of triples on $(\R^+)^3$ is determined as follows using the type and coordinate of the head points.
Heuristically, suppose the $k$-th handover is of type $\bbinom{q}{\t_p,\t_n}$ and produced by the head points $(t^p_k,h^p_k)$ and $(t^n_k,h^n_k)$, where the letter $p$ and $n$ stand for {\em previous} and {\em next}. We then define
\begin{equation}
\begin{aligned}
\left(H_k^l, T^r_k, H^r_k\right) &:=\begin{cases}
           (h^p_k,t^n_k-t^p_k,h^n_k) & \text{ if } q=1,\\
           (h^n_k,t^p_k-t^n_k,h^p_k), &\text{ if } q=2.
    \end{cases}
\end{aligned}
\label{eq:HlTrHr}
\end{equation}
The state at the $k$-th step implicitly depends on the previous step as the head point $(t^p_k,h^p_k)$ is involved in previous handover(s). The following theorem establishes the Markovianity of the sequence $\left\{\left(H^l_k, T^r_k, H^r_k\right), type(k) \right\}_{k\in \Z}$, following the foot steps of Theorem~\ref{theorem:markov}.

\begin{theorem}
The sequence $\left\{\left(H^l_k, T^r_k, H^r_k\right), type(k) \right\}_{k\in \Z}$ forms a homogeneous Markov chain on the state space $(\R^+)^3\times \Tcal$. 
\label{theorem:markov2}
\end{theorem}
\begin{proof}
Fix $k\in \Z$. Let  $type(k)=\bbinom{q}{i,j}$ be the type of the $k$-th handover, for some $i,j\in \{1,2\}$. Given $\left(H^l_k, T^r_k, H^r_k\right)=\left(h^l_k, t^r_k, h^r_k\right)$, suppose the locations of the responsible head points are $(0,h^l_k)$ and $(t^r_k, h^r_k)$, respectively and the intersection point of the corresponding radial birds for the handover in question is $(\hat s_k,\hat h_k)$. Define $\t(t,h)$ to be the type of a head point $(t,h)\in \mathbb H^+$. For any $s\in \R$, define $Q_s:=\{(t,h)\in \mathbb H^+: t\geq s\}$, to be the quadrant on $\mathbb H^+$ and on the right of the vertical line $t=s$.  Depending on the type $\bbinom{q}{i,j}$, we can assign their corresponding type to the head points and carry out our analysis in the following two cases:
\begin{case}[\textbf{\em $q=1$.}]\label{case:q1} In this case, the head point corresponding to the new station is on the right and so we have $\t(t^r_k,h^r_k)=j$ and $\t(0,h^l_k)=i$. Here, we have four possibilities for the handover types: $\bbinom{q}{i,j}=\bbinom{1}{1,1}, \bbinom{1}{2,1}, \bbinom{1}{1,2}$ and $\bbinom{1}{2,2}$. Let $m\neq j\in \{1,2\}$.
The precise unexplored region for different cases plays an essential role in establishing the Markovianity of the process in question. To determine the $(k+1)$-st step of the process, one first needs to find the unexplored regions for potential location of the upcoming head point. Observe that there are two unexplored regions, as the upcoming head point can be of two different types. For any given   type $\bbinom{1}{i,j}$ for the $k$-th handover, let $\Ucal^1_j$ and $ \Ucal^1_m\subset \mathbb H^+$ be the unexplored region of the upcoming head point of type $j$ and $m$, respectively. Table~\ref{tab:unexp-regions-q1} lists out the unexplored regions $\Ucal^1_j$ and $ \Ucal^1_m$ for the case $q=1$, by gathering the details in various parts of the proofs of Lemma~\ref{lemma:LT1-l} and Lemma~\ref{lemma:LT_1ijk} corresponding to different handover types. These unexplored regions only depend on the state $\left(\left(h^l_k, t^r_k, h^r_k\right), type(k)\right)$ at step $k$.
\begin{table}[!ht]
\centering
\renewcommand{\arraystretch}{1.2} 
\begin{spacedtable}{1.7}{2pt}
\begin{tabular}{|c|c|c|c|c|}
\hline
\begin{tabular}{@{}c@{}}
\textbf{Types:} $\bbinom{1}{i,j}$\\ \hline $\Ucal^1_j,\Ucal^1_m$, $m\neq j$
\end{tabular}   & $\bbinom{1}{1,1}$, Fig.~\plainref{figure:pure_birdii} &$\bbinom{1}{2,1}$, Fig.~\plainref{figure:mixed_birdij} \& \plainref{figure:mixed_birdij2}& $\bbinom{1}{1,2}$, Fig.~\plainref{figure:mixed_bird122a} & $\bbinom{1}{2,2}$, Fig.~\plainref{figure:pure_bird22} \& \plainref{figure:pure_bird222}  \\ \cline{1-5}
$\Ucal^1_j$ & $Q_{t^r_k}\setminus \overline{E^{\hat s_k,v_1}_{\hat h_k}}$ & $Q_{t^r_k}\setminus \overline{E^{\hat s_k,v_1}_{\hat h_k}}$ &
$Q_{t^r_k}\setminus \overline{E^{\hat s_k,v_2}_{\hat h_k}}$ & 
$Q_{t^r_k}\setminus \overline{E^{\hat s_k,v_2}_{\hat h_k}}$\\
\hline
$\Ucal^1_m$ & 
$\mathbb H^+\setminus \overline{E^{\hat s_k,v_2}_{\hat h_k}}$& $Q_{0}\setminus \overline{E^{\hat s_k,v_2}_{\hat h_k}}$ &
$Q_{\hat s_k}\setminus \overline{E^{\hat s_k,v_1}_{\hat h_k}}$ & 
$Q_{\hat s_k}\setminus \overline{E^{\hat s_k,v_1}_{\hat h_k}}$  \\
\hline
\end{tabular}
\end{spacedtable}
\caption{Unexplored regions $\Ucal^1_j$, $\Ucal^1_m$, for $m\neq j$, when the $k$-th handover type is $\bbinom{1}{i,j}$. The references to figures are mentioned next to the handover types. The corresponding figure resembles the respective scenario and those can be used to determine the unexplored regions.}
\label{tab:unexp-regions-q1}
\end{table}

As explained in (\ref{eq:SnSm}) and Lemma~\ref{lemma:increasingSnm}, one can construct a sequence of non-decreasing open sets $\{S^j_t\}_{t\geq \hat s_k}$ and $\{S^m_t\}_{t\geq \hat s_k}$ in the unexplored region $\Ucal^1_j$ and $\Ucal^1_m$, using the sequence $\{h_j(t)\}_{t\geq \hat s_k}$ where $h_j(t):=\left(v_j^2(t-t^r_k)^2+(h^r_k)^2\right)^\half$. By Corollary~\ref{corollary:third-pt2}, there exists a head point say $(T_u,H_u)$ of type $\t_u$ in the unexplored region $\Ucal^1_j$ or $\Ucal^1_m$ such that the $(k+1)$-st handover is given by the intersection point, say $\left(\hat S_{k+1},\hat H_{k+1}\right)$, of the radial birds at $(t^r_k,h^r_k)$ and $(T_u,H_u)$. Here, the letter $u$ stands for {\em upcoming}. The pair of stopping set is $\left(S^j_{\hat S_{k+1}}, S^m_{\hat S_{k+1}}\right)$, and the upcoming head point $(T_u,H_u)$ lies on the boundary $\partial S^j_{\hat S_{k+1}}$ or $\partial S^m_{\hat S_{k+1}}$ of the respective stopping sets, depending on whether $\t_u=j$ or $m$. Then the type of this handover is $\bbinom{q'}{j,\t_u}$, where 
\begin{equation}
\begin{aligned}
q' &:=\begin{cases}
           1 & \text{ if } T_u\geq t^r_k,\\
           2, &\text{ if } T_u<t^r_k.
    \end{cases}
\end{aligned}
\end{equation}
Subsequently we define 
\begin{equation}
\begin{aligned}
\left(H_{k+1}^l, T^r_{k+1}, H^r_{k+1}\right) &:=\begin{cases}
           (h^r_k,T_u-t^r_k,H_u) & \text{ if } q'=1,\\
           (h^r_k,t^r_k-T_u,H_u), &\text{ if } q'=2.
    \end{cases}
\end{aligned}
\label{eq:HlTrHrk+1}
\end{equation}
This is the `spatial' part of the state at step $k+1$.
\end{case}
\begin{case}[\textbf{\em $q=2$.}]\label{case:q2}  In this case, the head point corresponding to the new station is on the left and so we have $\t(0,h^l_k)=j$ and $\t(t^r_k,h^r_k)=i$. The possibilities of the handover types in this case are $\bbinom{q}{i,j}=\bbinom{2}{2,1}$ and $\bbinom{2}{1,2}$. Let $m\neq j\in \{1,2\}$. The head point of the upcoming station can be found in the unexplored region, say $\Ucal^2_j$ and $\Ucal^2_m$, listed in the Table~\ref{tab:unexp-regions-q2} in the case $q=2$, obtained using the same recipe as for Table~\ref{tab:unexp-regions-q1} in the case $q=1$, as follows:

\begin{table}[!ht]
\centering
\renewcommand{\arraystretch}{1.2} 
\begin{spacedtable}{1.7}{6pt}
\begin{tabular}{|c|c|c|}
\hline
\begin{tabular}{@{}c@{}}
\textbf{Types:} $\bbinom{2}{i,j}$\\ \hline $\Ucal^2_j,\Ucal^2_m$, $m\neq j$
\end{tabular}   & $\bbinom{2}{2,1}$, Fig.~\plainref{figure:mixed_bird121} &$\bbinom{2}{1,2}$, Fig.~\plainref{figure:UUtau-MS}  \\ \cline{1-3}
$\Ucal^2_j$ & $Q_{0}\setminus \overline{E^{\hat s_k,v_1}_{\hat h_k}}$ & $Q_{0}\setminus \overline{E^{\hat s_k,v_2}_{\hat h_k}}$ \\
\hline
$\Ucal^2_m$ & 
$Q_{t^r_k}\setminus \overline{E^{\hat s_k,v_2}_{\hat h_k}}$& $Q_{\hat s_k}\setminus \overline{E^{\hat s_k,v_1}_{\hat h_k}}$   \\
\hline
\end{tabular}
\end{spacedtable}
\captionsetup{width=0.96\linewidth}
\caption{Unexplored regions $\Ucal^2_j$, $\Ucal^2_m$, for $m\neq j$, when the $k$-th handover type is $\bbinom{2}{i,j}$. The figures mentioned next to the handover types are used to determine the unexplored regions.}
\label{tab:unexp-regions-q2}
\end{table}
These unexplored regions only depend on the state $\left(\left(h^l_k, t^r_k, h^r_k\right), type(k)\right)$ at step $k$. Similarly to the case $q=1$, we construct a sequence of non-decreasing open sets $\{S^j_t\}_{t\geq \hat s_k}$ and $\{S^m_t\}_{t\geq \hat s_k}$, representing explored regions in the unexplored region $\Ucal^2_j$ and $\Ucal^2_m$, using the sequence $\{h_j(t)\}_{t\geq \hat s_k}$, where $h_j(t):=\left(v_j^2 t^2+(h^l_k)^2\right)^\half$. By Corollary~\ref{corollary:third-pt2}, there exists a head point, say $(T_u,H_u)$, of type $\t_u$ in the unexplored region $\Ucal^2_j$ or $\Ucal^2_m$ such that $(k+1)$-st handover is given by the intersection point, say $\left(\hat S_{k+1},\hat H_{k+1}\right)$, of the radial birds at $(0,h^l_k)$ and $(T_u,H_u)$. The pair of stopping sets is $\left(S^j_{\hat S_{k+1}}, S^m_{\hat S_{k+1}}\right)$, and the upcoming head point $(T_u,H_u)$ lie on the boundary of the respective stopping sets, depending on whether $\t_u=j$ or $m$. From Table~\ref{tab:unexp-regions-q2}, it is clear that $T_u\geq 0$. Then, the type of the upcoming handover is $\bbinom{q'}{j,\t_u}$, with $q'$ equal to $1$. Also note that $\bbinom{q'}{j,\t_u}=\bbinom{2}{j,j}$ is not possible for obvious reasons. Subsequently, we define the `spatial' part of the state at $(k+1)$-st step as
\begin{equation}
\left(H_{k+1}^l, T^r_{k+1}, H^r_{k+1}\right) := (H^r_k,T_u-t^r_k,H_u).
\label{eq:HlTrHrk+1a}
\end{equation} 
\end{case} 
The Markov property is satisfied by the sequence $\left\{\left(H^l_k, T^r_k, H^r_k\right), type(k) \right\}_{k\in \Z}$, since the state  at the $(k+1)$-st step, $\left(\left(H^l_{k+1}, T^r_{k+1}, H^r_{k+1}\right), type(k+1)\right)$, is determined using the unexplored regions, $\Ucal^1_j$ and $\Ucal^1_m$ in ~\ref{case:q1} for $q=1$ and $\Ucal^2_j$ and $\Ucal^2_m$ in~\ref{case:q2} for $q=2$, where one finds a Poisson point process of heads which is conditionally independent of the past, given the state $\left(\left(h^l_k, t^r_k, h^r_k\right), type(k)\right)$ at step $k$. This completes the proof of Theorem~\ref{theorem:markov2}. 
\end{proof}
\begin{remark}
Our technique allows one to predict the upcoming handover type $\bbinom{q'}{j,\t_u}$ and the location of the responsible head point, say $(T_u,H_u)$ of type $\t_u$, given the state at step $k$. In addition, it enables us to derive a representation of 
the transition kernel of the chain, for a transition to $\left((H^l_{k+1},T^r_{k+1},H^r_{k+1}),\bbinom{q'}{j,\t_u}\right)$, given the $k$-th state is $\left((H^l_k,T^r_k,H^r_k),\bbinom{q}{i,j}\right)$ along the line of the single-speed case. 
\end{remark}
\begin{remark}[Structural features of the Markov chain] In determining  each upcoming handover epoch the Markov chain explores certain region until it finds the next head point responsible for the next handover. For different typical handover types, the unexplored regions are listed in Table~\ref{tab:unexp-regions-q1} and Table~\ref{tab:unexp-regions-q2}. One can say that, as if the Markov chain starts fresh from a particular step onward, if the next head point lies in a region that has minimum dependence from the past. This happens only when the sets $E^{\hat s,v_1}_{\hat h} $ and $E^{\hat s,v_2}_{\hat h}$ are small, for a handover given by an intersection point $(\hat s,\hat h)$. This is equivalent to the condition that the handover distance $\hat h$ is small. In terms of the Markov chain, this scenario happens when all the components of the spatial part $(H^l,T^r,H^r)$ of the Markov chain are small enough, irrespective of the the type of the handover. In other words, in this scenario both the head points are sufficiently close to each other and also close to the time axis.

This idea leads to a renewal structure in the Markov chain using a ``petite set'' or ``small set'' argument~\cite[Theorem~5.2.1]{Meyn-Tweedie}. Moreover, this opens up a door for a new line of research on the structural properties of handover process in terms of the Markov chain, for example: central limit theorem, law of large numbers, large deviation etc. as mentioned in Remark~\ref{remark:petit_set} for the single-speed case.   
\end{remark}


\section{Handover frequency: multi-speed case}\label{sec-HFn_speed}
\subsection{Multi-speed case}
\label{subsection:multi_speed} Our techniques naturally lift to the finite set of speeds case, except for a small adjustment required to quantify the effect of stations of all the different speeds on handovers of the same or mixed types. Suppose the set of speeds is $v_{[n]}:=\{v_i\}_{i\in [n]}$, with the order $v_1> v_2 > \cdots > v_n$ and initial locations of stations are independent Poisson point processes $\tilde{\Phi}^l=\sum_{i\in \N}\delta_{(R^{(l)}_i, \a^{(l)}_i)}$, for $l\in [n]$. \label{notation:vn1}

\subsubsection{Head point process and radial bird particle process}For different speeds $v_l$ with $l\in[n]$, consider $n$ different head point processes for $l\in [n]$,
\[
\Hcal^{l}:=\sum_{i\in \N}\delta_{\left(T^{(l)}_i, H^{(l)}_i\right)}, 
\text{ where }
\left(T^{(l)}_i, H^{(l)}_i\right):=\left(-\frac{R^{(l)}_i}{v_l}\cos \a^{(l)}_i, R^{(l)}_i|\sin \a^{(l)}_i|\right).
\]
The following  result is similar to Lemma~\ref{lemma:tips_density2}.
\begin{lemma}
For each $l\in[n]$, $\Hcal^l$ is a Poisson point process on $\mathbb{H}^+$ with intensity measure $\nu_l$ where $d\nu_l:=v_l {\rm d}t \otimes 2\la_l {\rm d}h $. For $l\in [n]$, the point processes $\Hcal^l$, are independent.
\label{lemma:tips_densityn}
\end{lemma}
We shall write $\Hcal=\sum_{l\in [n]}\Hcal^l$, the superposition of $n$ independent Poisson point processes $\{\Hcal^l\}_{l\in [n]}$. We write the law of the head point process $\Hcal=\sum_{l\in [n]}\Hcal^l$, as a product law  $\P_\Hcal:=\bigotimes_{l\in [n]}\P_{\Hcal^l}$, since the $\Hcal^l$'s are independent. The radial bird particle process can be constructed as the sum $\Pcal_c(v_{[n]}):=\sum_{l\in [n]}\Pcal_c(v_l)$, where,
\[
\Pcal_c(v_l)=\sum_{i}\delta_{C_i^l},
\]
where the support is the set of closed subsets of $\mathbb H^+$ of the form (\ref{eq:cb2}). The Poisson structure and time-stationarity of $\Pcal_c(v_{[n]})$ are also preserved in this case, as seen in Lemma~\ref{lemma:bird_time_stationary2} for the two-speed case. In view of Remark~\ref{remark:bird-to-mhead}, we can equivalently consider $\Pcal_c(v_{[n]})$ as the marked head point point process $\Hcal_c(v_{[n]})= \sum_{l\in [n]}\Hcal_c(v_l)$ with closed set valued marks, where
\[
\Hcal_c(v_l)=\sum_{i}\delta_{\left(T^{(l)}_i, H^{(l)}_i\right), C_i^l}.
\]
\subsubsection{Joint lower envelope process} The {\em joint lower envelope process} is a stochastic process on $\mathbb H^+$, is defined as $\{(t, L_n(t))\}_{t\in \R}$, where 
\[
L_n(t):= \min_{l\in [n]} \inf_{i\in\N} f_{v_l}\left(\left(R^{(l)}_i,\a^{(l)}_i\right), t\right).
\]
The lower envelope is also time-stationary in view of Lemma~\ref{lemma:L_stationary2} for two-speed case. The {\em lower envelope} is defined as a random closed subset of $\mathbb H^+$ as \label{notation:Len}
\[
\Lcal_e(v_{[n]}):= \{(t, L_n(t)): t\in \R\}.
\] 
Let us define the open region above the lower envelope $\Lcal_e(v_{[n]})$ as \label{notation:Len+}
\begin{equation}
\Lcal^+_e(v_{[n]}):=\{(t,h)\in \mathbb H^+\,\mbox{:}\, h>L_n(t)\}. \label{eq:Lcaln+}
\end{equation} 
In the multi-speed case, we have a {\em half-ellipse condition}, analogous to  Lemma~\ref{lem:semiellipse2}, for the characterization for a point to be on the lower envelope $\Lcal_e(v_{[n]})$. We shall state the result without proof.
\begin{lemma}[Half-ellipse condition] For any point $(s, u)\in \cup_{k\in[n]}\cup_{i\in \N}C^l_i$ the following holds true: 
\begin{enumerate}[(i).]
\item \label{semiellipse_n1} $(s,u) \in \Lcal_e(v_{[n]})$ if and only if $\Hcal^l(E^{s, v_l}_{u})=0$ for all $l\in [n]$;
\item \label{semiellipse_n2} the probability of such an event is
\begin{equation}
\P\left((s,u) \in \Lcal_e(v_{[n]})\right)= e^{-\la \pi u^2}, 
\end{equation}
where $\la:=\sum_{l\in [n]}\la_l$.
\end{enumerate}
\label{lem:semiellipse_n}
\end{lemma}
\subsubsection{The handover point process and  intensity}\label{subsubsection:constVn} In view of the construction of the handover point process in the two-speed case in Subsubsection~\ref{subsubsection:constV}, we can also define the handover point process in the multi-speed case using the multi-speed radial bird particle process $\Hcal_c(v_{[n]})$. By Lemma~\ref{lem:semiellipse_n}, for the intersection point $(\hat{s}, \hat{h})$ of two radial birds with their heads at $(t,h), (t',h')$ give rise to a handover if and only if $\Hcal^{l}(E^{\hat s, v_l}_{\hat h})= 0$, for all $l\in [n]$.  

Formally, taking inspiration from the two-speed case in Subsubsection~\ref{subsubsection:constV}, we can construct different handover point processes $\Vcal_l$ corresponding pure handovers of type $l$, for $l\in [n]$, as in (\ref{eq:VcalHPPll}), and $\Vcal^{(k)}_{l,r}$ corresponding to mixed handovers of type $\binom{k}{l,r}$, for $l\neq r\in [n], k\in \{1,2\}$, as in (\ref{eq:VcalHPPlmk}), (\ref{eq:VcalHPPlmks}). Here, we recall that this type of mixed handover is due to the $k$-th intersection of the radial birds of types $l,r$ such that $l\neq r$ and the head point of the former is on the left of that of the latter.

Based on all possible types of pure and mixed intersections, while respecting the {\em intersection criterion} determined in Lemma~\ref{lemma:2int} and Lemma~\ref{eq:DlDr-rev}, the handover point process $\Vcal$, similar to (\ref{eq:VcalHPP}), is defined as 
\begin{equation}
\Vcal:=\sum_{l\in [n]}\Vcal_l+\sum_{l\neq r\in [n], k=1,2}\Vcal_{l,r}^{(k)}.
\label{eq:VcalHPPn}
\end{equation}
In the following, we just state the result about the time stationarity of the handover point process $\Vcal$, which is inherited from the stationarity of the head point process $\Hcal$.
\begin{lemma}[Time-stationarity of $\Vcal$]
In the multi-speed case, the handover point process $\Vcal$ is stationary with respect to time.
\label{lemma:stationary_Tcaln}
\end{lemma}
The construction of the point processes $\Rcal_l$ for $l\in [n]$ and $\Rcal^{(k)}_{l,r}$ for $l\neq r\in [n], k\in \{1,2\}$, corresponding to the right most head points is also the same and so is the existence of the bijections as in Lemma~\ref{lemma:vr}. 
\begin{lemma}
\begin{enumerate}[(i).]
    \item \label{bijection-in}For all $l\in [n]$, there exists a bijection $\beta_l$ between $\Vcal_l$ and $\Rcal_l$.
    \item \label{bijection-ijkn}  For all $l\neq r\in [n], k\in \{1,2\}$, there exists a bijection $\beta^{(k)}_{l,r}$ between $\Vcal^{(k)}_{l,r}$ and $\Rcal^{(k)}_{l,m}$.
\end{enumerate}
\label{lemma:vrn}
\end{lemma} 
We obtain the equality of intensities as a corollary:
\begin{corollary}
\begin{enumerate}[(i).]
    \item For $l\in[n]$, the intensities of the point processes $\Vcal_l$ and $ \Rcal_l$ are equal: $\la_{\Rcal_l}=\la_{\Vcal_l}= \La_l$.
     \item   For all $l\neq r\in [n], k\in \{1,2\}$, the intensities of the point processes $\Vcal_{l,r}^{(k)}$ and $ \Rcal_{l,r}^{(k)}$ are equal: $\la_{\Rcal_{l,r}^{(k)}}=\la_{\Vcal_{l,r}^{(k)}}=\La_{l,r}^{(k)}$.
\end{enumerate} 
\label{corollary:eqVRn}
\end{corollary}
In the multi-speed scenario, the handover frequency can be written as a sum of pure and mixed handovers 
\[
\la_\Vcal= \sum_{l\in [n]}\Lambda_l+\sum_{l\neq r\in [n],k\in \{1,2\}}\Lambda^{(k)}_{l,r},
\]
as described in the following result which generalizes Theorem~\ref{theorem:handover_freq1} and Theorem~\ref{theorem:HO2speed-main}.
\begin{theorem} 
The handover frequency in multi-speed case with set of speeds $v_{[n]}$ is given by:
\begin{enumerate}
\item \label{MHnp1} The pure handover frequencies are 
\begin{equation}
     \Lambda_i= \frac{4v_i\sqrt{\la}}{\pi}\left(\frac{\la_i}{\la}\right)^2, \text{ for }i\in [n].
\end{equation}
\item \label{MHnp2}  For any pair $i\neq j\in [n]$, the mixed handover frequencies satisfies, $\Lambda^{(1)}_{i,j}=\Lambda^{(2)}_{j,i}$ and $\Lambda^{(2)}_{i,j}=\Lambda^{(1)}_{i,j}$.

\item \label{MHnp3}  For any pair $j<i\in [n]$, 
\begin{align}
\Lambda^{(1)}_{i,j}  &= \frac{\la_i\la_j}{\pi\la^{3/2}}(v_i+v_j) \left[\left(1+\frac{v_j}{v_i}\right)E\!\left(\frac{v_i}{v_j}\right)- \frac{v_j^2-v_i^2}{v_i v_j} K\left(\frac{v_i}{v_j}\right)\right],\label{eq:La-ijkx}\\
\Lambda^{(2)}_{i,j}  &= \frac{\la_i\la_j}{\pi\la^{3/2}}(v_j-v_i) \left[\left(1+\frac{v_j}{v_i}\right)E\!\left(\frac{v_i}{v_j}\right)- \frac{v_j^2-v_i^2}{v_i v_j} K\left(\frac{v_i}{v_j}\right)-2\right],
%
\label{eq:La-ijk}
\end{align}
\end{enumerate} 
where $E(\cdot)$ and $K(\cdot)$ are elliptic integrals of second and first kind, as defined in (\ref{eq:Elliptic-int}).
%
%
\label{thm:HOnspeed}
\end{theorem}
\begin{proof}[Proof of Theorem~\ref{thm:HOnspeed},~part~(\ref{MHnp1})]
In the first part, the intensity of pure handovers is computed similar to (\ref{eq:Lam1}) in the proof of two-speed case, as 
\begin{equation}
\Lambda_i= \frac{4v_i\sqrt{\la}}{\pi}\left(\frac{\la_i}{\la}\right)^2,
\label{eq:HOF-same-n}
\end{equation}
for each $i\in [n]$, where $\la=\sum_{i\in [n]}\la_i$.
\end{proof}
\begin{proof}[Proof of Theorem~\ref{thm:HOnspeed},~part~(\ref{MHnp2})] Let $i>j\in [n]$ and assume that $v_j>v_i$. We are interested in the handover between two stations with speed $v_j,v_i$, which we call mixed handovers.
For such mixed handovers, by the symmetry argument as in the two-speed case in the proof of part~(\ref{MH2p2}) Theorem~\ref{theorem:HO2speed-main}, we have that $\Lambda^{(1)}_{i,j}=\Lambda^{(2)}_{j,i}$ and $\Lambda^{(2)}_{i,j}=\Lambda^{(1)}_{j,i}$ for all $j<i\in [n]$. 
\end{proof}
\begin{proof}[Proof of Theorem~\ref{thm:HOnspeed},~part~(\ref{MHnp3})]
We just evaluate $\Lambda^{(k)}_{i,j}$ for $j<i\in [n]$ and $k=1,2$ for handovers of type $\binom{k}{i,j}$. Let us straight away start from Equation (\ref{eq:h3v}) for the two-speed case in Theorem~\ref{theorem:HO2speed-main} and write the mixed handover frequency $\Lambda^{(k)}_{i,j}$ as an integral by applying the Campbell-Mecke formula for the factorial moment measure of order 2 for the sum $\Hcal:=\sum_{i\in [n]}\Hcal^i$, as
\begin{align}
\Lambda^{(k)}_{i,j}\!&=\! 4\la_i\la_jv_iv_j\!\!\int_{0}^{1} \!\!\int_{0}^{\infty}\!\!\int_{0}^{\infty}\!\!\! \int_{D_\ell(t_j, h_j, h_i)}\!\! \!\!\!\!\!\E^{(i,j)}_{\Hcal}\!\left[\one_{A^{(k)}_{i,j}}\! \Big(\Hcal^i+\delta_{(t_i, h_i)}+\Hcal^j+\delta_{(t_j, h_j)}+\!\!\sum_{m\neq i,j}\Hcal^m\Big)\right] {\rm d}t_i {\rm d}h_i {\rm d}h_j {\rm d}t_j.
\label{eq:hnv}
\end{align}
Here $A^{(k)}_{i,j}$ is the event that there is a handover of type $\binom{k}{i,j}$ at $\left(\hat s_{i,j}^{(k)}, \hat h_{i,j}^{(k)}\right)$, for $k=1,2$, due to the intersection of the radial birds that have their heads at $(t_j, h_j)$ and $(t_i, h_i)$, with $t_i<t_j$. Above, we took the expectation inside the integral under the two-point Palm probability measure\label{notation:2nPHlP}
\[
\P^{(i,j),(t_i, h_i), (t_j,h_j)}_{\Hcal}:=\P^{(t_i,h_i)}_{\Hcal^i}\otimes\P^{(t_j, h_j)}_{\Hcal^j}\otimes \bigotimes_{m\neq i,j}  \P_{\Hcal^m}.
\]
Let us evaluate the term $\Lambda^{(k)}_{i,j}$ similarly to the two-speed case as
\begin{align}
\Lambda^{(k)}_{i,j}&=4\la_j\la_iv_jv_i\int_{0}^{1}\! \int_{0}^{\infty}\!\!\int_0^{h_j}\!\!\int_{-\infty}^{t_r-t^*}\!\!\E_{\Hcal}\left[\one_{A^{(k)}_{i,j}}\Big(\Hcal^i+\delta_{(t_i, h_i)}+\Hcal^j+\delta_{(t_j, h_j)}+\!\!\sum_{m\neq i,j}\Hcal^m\Big)\right] {\rm d}t_i\, {\rm d}h_i\, {\rm d}h_j\, {\rm d}t_j\nn\\
&\quad + 4\la_j\la_iv_jv_i\int_{0}^{1}\! \int_{0}^{\infty}\!\!\int_{h_j}^{\infty}\!\!\int_{-\infty}^{t_j} \!\!\E_{\Hcal}\left[\one_{A^{(k)}_{i,j}}\Big(\Hcal^i+\delta_{(t_i, h_i)}+\Hcal^j+\delta_{(t_j, h_j)}+\!\!\sum_{m\neq i,j}\Hcal^m\Big)\right] {\rm d}t_i\, {\rm d}h_i\, {\rm d}h_j\, {\rm d}t_j,
\label{eq:hnvc}
\end{align}
where $t^*$ is defined similarly to (\ref{eq:tstar}):
\begin{equation}
t^* \equiv t^*(h_j,h_i):=\frac{1}{v_jv_i}(h_j^2-h_i^2)^\half (v_j^2-v_i^2)^\half.
\label{eq:star-tj}
\end{equation}
The quantity $t^*$ is well defined, since $v_j>v_i$. In (\ref{eq:hnvc}), we have used the fact that $t_i\leq t_j-t^*$ if $h_i<h_j$ and $t_i\leq t_j$, for $h_i\geq h_j$. 
\begin{figure}[ht!]
\centering
\begin{tikzpicture}[scale=0.6, every node/.style={scale=0.7}]
\pgftransformxscale{1.3}  
\pgftransformyscale{1.3}   
    \draw[->] (-3, 0) -- (6, 0) node[right] {$t$};
   \draw[blue, domain=-2:3.8, smooth] plot (\x, {((3/2)^2*\x*\x+2^2)^0.5});
    \draw[blue](0,2) node{$\bullet$};
        \draw[](0.2,1.8) node{$(t_i,h_i)$};
    \draw[red, domain=0.8:3.4, smooth] plot (\x, {(16*\x*\x-64*\x+64+2.5^2)^0.5});
    \draw[red](2,2.5) node{$\bullet$};
    \draw[](2.1,2.2) node{$(t_j,h_j)$};
    \draw[](0.9,3.3) node{$(\hat s_1,\hat h_1)$};
    \draw[](3.7,5.2) node{$(\hat s_2,\hat h_2)$};
     \draw[blue]  (-0.47,0) arc (-180:-360:2 and 3.07);
     \draw[red]  (2.35,0) arc (0:180:0.8 and 3.07);
     \draw[green]  (3.05,0) arc (0:180:1.5 and 3.07);
     \draw[]  (5,0) arc (0:180:3.5 and 3.07);
     \draw[blue](-0.9,0.6) node{$E^{\hat s_1,v_i}_{\hat h_1}$};
     \draw[red](1.6,0.6) node{$E^{\hat s_1,v_j}_{\hat h_1}$};
     \draw[green](2.9,0.6) node{$E^{\hat s_1,v_l}_{\hat h_1}$};
     \draw[](5.5,0.6) node{$E^{\hat s_1,v}_{\hat h_1}$};
    \end{tikzpicture}
    \captionsetup{width=0.94\linewidth}
    \caption{We consider  $(\hat s_1,\hat h_1)=(\hat s_{i,j}^{(1)},\hat h_{i,j}^{(1)})$, $(\hat s_2,\hat h_2)=(\hat s_{i,j}^{(2)},\hat h_{i,j}^{(2)})$ for simplicity. The mixed intersection $(\hat s_1, \hat h_1)$ represents a handover if the regions $E^{\hat s_1,v_i}_{\hat h_1}$ are empty of heads from $\Hcal^i$ for all $i\in [n]$. The red and blue ellipses are for the speeds $v_j$ and $v_i$, respectively and the green and black ones correspond to other speeds $v_l,v_k\neq v_i,v_j$.}
\label{figure:mixed_birds1}
\end{figure}

Under the two-point Palm probability measure $\P^{(i,j),(t_i, h_i), (t_j,h_j)}_{\Hcal}$, we have 
\begin{align}\!\!
\E_{\Hcal}\left[\one_{A^{(k)}_{i,j}}\Big(\Hcal^j+\delta_{(t_j, h_j)}+\Hcal^i+\delta_{(t_i, h_i)}+\!\!\sum_{m\neq i,j}\Hcal^m\Big)\right] &= \P^{(i,j),(t_i, h_i), (t_j,h_j)}_{\Hcal}\!\!\left(A^{(k)}_{i,j}\right).
    \label{eq:HO-ij}
\end{align}
Applying the Slivnyak-Mecke theorem and using Figure~\ref{figure:mixed_birds1}, the term in (\ref{eq:HO-ij}) evaluates to 
\begin{align}
\prod_{m\in [n]} \P_{\Hcal^m}\left( \Hcal^m(E^{\hat s_k,v_m}_{\hat h_k})=0\right) &= \prod_{m\in [n]}e^{-2\la_m v_m\frac{\pi \hat h_k^2}{2v_m}}= e^{-\pi \hat h_k^2\sum_{m\in [n]}\la_m}= e^{-\pi\la\hat h_k^2},
\label{eq:MHn1}
\end{align}
since $\sum_{m\in [n]}\la_m=\la$, where $(\hat s_k,\hat h_k)=(\hat s_k(t_i,h_i,t_j,h_j),\hat h_k(t_i,h_i,t_j,h_j))\equiv (\hat s_k(t_j-t_i,h_i,h_j),\hat h_k(t_j-t_i,h_i,h_j))$, the $k$-th intersection point of the radial birds at $(t_i,h_i), (t_j,h_j)$. For any pair $i, j\in [n]$, with $j<i$, the mixed handover frequency $\Lambda^{(1)}_{i,j}$ is 
\begin{align}
\Lambda^{(1)}_{i,j}
&=4\la_i\la_jv_iv_j\int_{0}^{1} \int_{0}^{\infty}\int_0^{h_j} \int_{-\infty}^{t_j-t^*} e^{-\pi\la \hat h_1^2(t_j-t_i,h_j, h_i)}\; {\rm d}t_i\, {\rm d}h_i\, {\rm d}h_j\, {\rm d}t_j\nn\\
&\quad + 4\la_i\la_jv_iv_j\int_{0}^{1} \int_{0}^{\infty}\int_{h_j}^{\infty} \int_{-\infty}^{t_i}e^{-\pi\la \hat h_1^2(t_j-t_i,h_j, h_i)}\; {\rm d}t_i\, {\rm d}h_i\, {\rm d}h_j\, {\rm d}t_j\nn\\
&=4\la_i\la_jv_iv_j \int_{0}^{\infty} \left[\int_0^{h_j} \int_{t^*}^\infty e^{-\pi\la \hat h_1^2(t,h_j, h_i)}\; {\rm d}t \, {\rm d}h_i +  \int_{h_j}^{\infty} \int_{0}^{\infty}e^{-\pi\la \hat h_1^2(t,h_j, h_i)}\; {\rm d}t \, {\rm d}h_i\right]{\rm d}h_j,
\label{eq:MHn8}
\end{align}
where $t^* \equiv t^*(h_j,h_i):=\frac{1}{v_jv_i}(h_j^2-h_i^2)^\half (v_j^2-v_i^2)^\half$ and $v_j>v_i$ for $j<i$. By computation similar to that used in  the expressions of mixed handover frequencies in the two-speed case (Theorem~\ref{theorem:mixed-hof}), we have that 
\begin{align}
\Lambda^{(1)}_{i,j}&=\frac{\la_i\la_j}{\pi\la^{3/2}}(v_i+v_j) \left[\left(1+\frac{v_j}{v_i}\right)E\!\left(\frac{v_i}{v_j}\right)- \frac{v_j^2-v_i^2}{v_i v_j} K\left(\frac{v_i}{v_j}\right)\right],
\label{eq:nLamij-1}
\end{align}
Similarly, for $\Lambda^{(2)}_{i,j}$, we have
\begin{align}
\Lambda^{(2)}_{i,j} &= \frac{\la_i\la_j}{\pi\la^{3/2}}(v_j-v_i) \left[\left(1+\frac{v_j}{v_i}\right)E\!\left(\frac{v_i}{v_j}\right)- \frac{v_j^2-v_i^2}{v_i v_j} K\left(\frac{v_i}{v_j}\right)-2\right].
\label{eq:nLamij-2}
\end{align}
This completes the proof of Theorem~\ref{thm:HOnspeed}. 
\end{proof}
\begin{remark}[Total handover frequency $\la_{\Vcal}$]
From (\ref{eq:HOF-same-n}), (\ref{eq:nLamij-1}), (\ref{eq:nLamij-2}) and part~(\ref{MHnp2}), we get the value of total handover frequency
\begin{align}
\la_\Vcal &= \sum_{i\in[n]}\Lambda_i+ 2 \sum_{j<i\in[n], k=1,2}\Lambda^{(k)}_{i,j}\nn\\
&= \frac{4\sqrt{\la}}{\pi} \sum_{i\in[n]} v_i\left(\frac{\la_i}{\la}\right)^2 + 4 \sum_{j<i\in[n]} \frac{\la_i\la_j}{\pi\la^{3/2}} \left(v_j c(v_j,v_i)- v_j+v_i\right)\nn\\
&= \frac{4}{\pi\la^{3/2}} \left[\sum_{i\in[n]} v_i \la_i^2 + \sum_{j<i\in[n]} \la_i\la_j \left(v_j c(v_j,v_i)- v_j+v_i\right)\right],\nn
\end{align}
where $c(v_j,v_i) = \left(1+\frac{v_j}{v_i}\right)E\!\left(\frac{v_i}{v_j}\right)- \frac{v_j^2-v_i^2}{v_i v_j} K\left(\frac{v_i}{v_j}\right)$.
\end{remark}
%


\section{Handover Palm distribution and applications: multi-speed case}\label{sec-Distribution_T1n}
In this section, we first determine the handover Palm probability distribution with the goal to determine the Palm probability distribution of inter-handover times and typical handover distance. An identical analysis follows to come up with a result similar to the mass transport principle (Lemma~\ref{lemma:VR-MTP} for the single-speed and Lemma~\ref{lemma:VR-ijkC} for the two-speed case) for the multi-speed case. 
\subsection{Palm probability with respect to handovers}\label{subsection:Palm-Handover-n}
Without loss of generality, here also, we assume that $v_1>v_2>\cdots > v_n$. Identically to the two-speed case, here also, the reference probability space is considered as $(\Omega,\Fcal,\mathbb P)$, 
with $(\Omega,\Fcal)$ the canonical space of head point process $\Hcal=\sum_{l\in [n]} \Hcal^l$ on $\mathbb{H}^+$, where $\Hcal^l$ is the head point process corresponding to speed $v_l$. A point $\omega\in \Omega$ representing a realization of the point process $\Hcal$. The measure space is equipped with the shift $\{\theta_t\}_{t\in \mathbb{R}}$ along the time axis, defined as
\begin{equation}
\theta_t(\Hcal)=\sum_{l\in [n]}\theta_t(\Hcal^l)=\sum_{l\in [n]}\sum_{i\in \N} \delta_{(T^{l}_i-t,H^{l}_i)},
\label{eq:shift_Hcaln}
\end{equation}
which is ergodic for $\mathbb P$. As in the two-speed case, the law $\mathbb P$ of the head point process $\Hcal$ is left invariant by this shift. 

We have constructed the handover point process $\Vcal$ in Subsubsection~\ref{subsubsection:constVn}, using the pure handover point processes $\Vcal_l$ for $l\in [n]$, and the mixed handover point process $\Vcal^{(k)}_{l,r}$ for $l\neq r\in [n], k\in \{1,2\}$, depending on all possible type of intersections. The handover point process $\Vcal$ is stationary from Lemma~\ref{lemma:stationary_Tcal2} and has a positive and finite intensity from Theorem~\ref{theorem:HO2speed-main}. This allows us to define the Palm probability w.r.t. $\Vcal$ on the probability space $(\Omega,\Fcal)$, which will be denoted as
$\mathbb P_{\Vcal}^0$.

Recall the point processes $\Rcal_l$ for $l\in [n]$ and $\Rcal^{(k)}_{l,r}$ for $l\neq r\in [n], k\in \{1,2\}$, corresponding to the right most head points. The same mass transport principle, Lemma~\ref{lemma:VR-ijkC}, under the bijections defined in Lemma~\ref{lemma:vrn}, applies to all the individual Palm probability measures, stated in the following.
\begin{lemma}
For all non-negative measurable functions $f$ on $(\Omega,\Fcal)$ we have:
\begin{enumerate}[(i).]
    \item \label{beta-in} For all $l\in [n]$,
\begin{equation}
\E_{\Vcal_l}^0 [f (\Hcal)]=\E_{\Rcal_l}^0 \left[f (\theta_{\beta_l(0)}(\Hcal))\right], 
\label{eq:VR-in}
\end{equation}
where $\beta_l$ is the bijection between the point processes $\Vcal_l$ and $\Rcal_l$.
\item \label{beta-ijkn} For all $l\neq r\in [n], k= 1,2$,
\begin{equation}
\E_{\Vcal^{(k)}_{l,r}}^0 [f (\Hcal)]=\E_{\Rcal^{(k)}_{l,r}}^0 \left[f (\theta_{\beta^{(k)}_{l,r}(0)}(\Hcal))\right], 
\label{eq:VR-ijkn}
\end{equation}
where $\beta^{(k)}_{l,r}$ is the bijection between the point processes $\Vcal_{l,r}^{(k)}$ and $\Rcal_{l,r}^{(k)}$, of type $\binom{k}{l,r}$.
\end{enumerate}
\label{lemma:VR-ijkn}
\end{lemma}
Hence, we have the following decomposition of $\E_{\Vcal}^0 [f (\Hcal)]$, similar to Theorem~\ref{thm:decompfV}, based on different types of handovers, using the Palm probability measure for the superposition of multi-type independent point processes (cf. \cite[Subsection 1.6, p.~36]{Baccelli-Bremaud}).
\begin{theorem}
For all non-negative measurable functions $f$ on $(\Omega,\Fcal)$,
\begin{align}
\E_{\Vcal}^0 [f(\Hcal)]= \sum_{l\in [n]}\frac{\Lambda_l}{\la_{\Vcal}} \mathbb E^0_{\Vcal_l} [f(\Hcal)]+ \sum_{l\neq r\in [n], k=1,2} \frac{\Lambda^{(k)}_{l,r}}{\la_\Vcal}\E^0_{\Vcal^{(k)}_{l,r}} [f (\Hcal)].
\label{eq:DPalmn}
\end{align}
\label{thm:decompfVn}
\end{theorem}
Similar to the supporting results, Lemma~\ref{lemma:Palm-Lcal-ii}, Lemma~\ref{lemma:Palm-Lcal-ij}, individual Palm probability measures are expressed as follows:
\begin{lemma}
For $l\in [n]$ and for all non-negative measurable functions $f$ on $(\Omega,\Fcal)$,
\begin{align}
\E^0_{\Vcal_l} [f (\Hcal)]
& = 
\frac {4\la_l^2v_l^2} {\Lambda_{l}}
\int_{(\R^+)^3}
\E_\Hcal\left[ f\big(\theta_{\hat{s}(0,h,-t',h')}(\Hcal^l+\delta_{(0,h)}+\delta_{(-t',h')}+\sum_{m\neq l}\Hcal^m)\big)\right.\nn\\
&\hspace{3in}\times \left.\prod_{i}\one_{\Hcal^i\left(E^{ \hat{s}(0,h,-t',h'),v_i}_{\hat{h}(0,h,-t',h')}\right)=0}  \right] {\rm d}t' {\rm d}h' {\rm d}h ,
\nn
\end{align}
where $(\hat{s}(t,h,t',h'),\hat{h}(t,h,t',h'))$
is the coordinates of the intersection of the two radial birds with heads $(t,h)$ and $(t',h')$.
\label{lemma:Palm-Lcal-iin}
\end{lemma}
The proof Lemma~\ref{lemma:Palm-Lcal-iin} follows the steps of Lemma~\ref{lemma:Palm-Lcal} and Lemma~\ref{lemma:Palm-Lcal-ii}, except for the extra condition, for the void probabilities, the multiple half-ellipses to be empty of head points corresponding to the different speeds, as in Theorem~\ref{thm:HOnspeed} and Figure~\ref{figure:mixed_birds1}. 
%
%
Having said this, we move towards the interesting case of mixed handovers similar to Lemma~\ref{lemma:Palm-Lcal-ij}, for which we provide more details in the following. 
\begin{lemma}
For $r\neq l\in [n], k=1,2$ and for all non-negative measurable functions $f$ on $(\Omega,\Fcal)$,
\begin{align}
 \frac{\Lambda^{(k)}_{l,r}}{4\la_l\la_r v_lv_r} \E^0_{\Vcal^{(k)}_{l,r}} [f (\Hcal)]
& = \int_{0}^\infty  \int_{0}^h \int_{t^*}^\infty 
\mathbb{E}_\Hcal\Bigg[ f \Big(\theta_{\hat{s}(0,h,-t',h')}(\Hcal^r+\delta_{(0,h)}+ \Hcal^l+\delta_{(-t',h')}+\!\!\!\sum_{ m\neq l,r}\Hcal^m)\Big) \nn\\
&\hspace{2.5in}\times \prod_{i}\one_{\Hcal^i\left(E^{\hat{s}_k(0,h,-t',h'),v_i}_{\hat{h}_k(0,h,-t',h')}\right)=0}
\Bigg] {\rm d}t'  {\rm d}h' {\rm d}h \nn\\ 
&\, +  \int_{0}^\infty  \int_h^\infty \int_0^{\infty}
\mathbb{E}_\Hcal\Bigg[ f \Big(\theta_{\hat{s}(0,h,-t',h')}(\Hcal^r+\delta_{(0,h)}+ \Hcal^l+\delta_{(-t',h')}+\!\!\! \sum_{m\neq l,r}\Hcal^m)\Big) \nn\\
&\hspace{2.5in}\times \prod_{i}\one_{\Hcal^i\left(E^{\hat{s}_k(0,h,-t',h'),v_i}_{\hat{h}_k(0,h,-t',h')}\right)=0}
\Bigg] {\rm d}t'  {\rm d}h' {\rm d}h,
\nn
\end{align}
where, $(\hat{s}_k(t,h,t',h'),\hat{h}_k(t,h,t',h'))$
are the coordinates of the $k$-th intersection of the two radial birds with heads at $(t,h)$ and $(t',h')$, $t^*=\frac{1}{v_lv_r}(h^2-h'^2)^\half (v_r^2-v_l^2)^\half$.
\label{lemma:Palm-Lcal-ijn}
\end{lemma}
\begin{remark}
The symmetry envisioned earlier in Lemma~\ref{lemma:symm} can also be shown to hold, using the same time-reversal argument, under the very specific assumption that the function $f$ is symmetric, i.e., $f(\Hcal)= f(-\Hcal)$, defined in Subsubsection~\ref{subsubsection:symmEf}. In other words,
\[
\E^0_{\Vcal^{(1)}_{i,j}} [f (\Hcal)]= \E^0_{\Vcal^{(2)}_{j,i}} [f (\Hcal)] \text{ and }\E^0_{\Vcal^{(2)}_{i,j}} [f (\Hcal)]= \E^0_{\Vcal^{(1)}_{j,i}} [f (\Hcal)],
\]
for all $i\neq j\in [n]$. Under the symmetric assumption on the function $f$, we have the decomposition
\begin{align}
\E_{\Vcal}^0 [f(\Hcal)]= \sum_{i\in [n]}\frac{\Lambda_i}{\la_{\Vcal}} \mathbb E^0_{\Vcal_i} [f(\Hcal)]+ 2\sum_{j<i\in [n], k=1,2} \frac{\Lambda^{(k)}_{i,j}}{\la_\Vcal}\E^0_{\Vcal^{(k)}_{i,j}} [f (\Hcal)].
\label{eq:DPalmn1}
\end{align}
\label{remark:symmetryn}
\end{remark}
\subsection{Typical handover distance: multi-speed case}
Let $\hat{H}$ be the random variable for the typical handover distance. Based on the different types of intersections that happen at time $0$ between the corresponding radial birds, we just state the following result for the Laplace transform of $\hat{H}^2$ under the Palm probability measure $\P^0_\Vcal$, in the multi-speed setting. This can be proved along the same lines of Lemma~\ref{lemma:typH-2sp}, by applying the same ideas along with the formula (\ref{eq:DPalmn1}) and the symmetry, adapted to the multi-speed setting.
\begin{lemma}For any $\gamma\geq 0$, under the Palm probability measure $\P_\Vcal^0$ of handover, the Laplace transform of $\hat{H}^2$ is
\begin{align}
\E^0_{\Vcal}\left[e^{-\gamma \hat{H}^2}\right] &= \left(1+\frac{\gamma}{\la\pi}\right)^{-3/2},
\label{eq:LTH_0-nsp}
\end{align}
and $\hat{H}$ follows the Nakagami distribution with parameters $\left(\frac{3}{2}, \frac{3}{2\la\pi}\right)$, where $\la=\sum_{i\in [n]} \la_i$.
%
\label{lemma:typH-nsp}
\end{lemma}
\subsection{Distribution of inter-handover time: multi-speed case}
Let us first convert the handover types bijectively to the new set of notations 
\[
\Tcal_n:=\left\{\bbinom{q}{l,m}:q=1 \text{ if } l=m, \text{ and } q=1,2, \text{ for } l\neq m, \text{ and } l,m\in [n]\right\},
\]
using a table similar to Table~\ref{tab:notation-t} for the $n$ types of speeds. As a result, for any $l\in [n]$, there exists a handover of type $\binom{1}{l,l}$ or equivalently $\bbinom{1}{l,l}$, and a pure handover point process $\Vcal_l\stackrel{d}{=}\Wcal_l$. Also, for any $i\neq j\in [n]$ and $k\in \{1,2\}$, there exists $l\neq m\in [n]$ and $q\in \{1,2\}$ such that, $\Vcal_{i,j}^{(k)}\stackrel{d}{=}\Wcal^{(q)}_{l,m}$, where the left hand side corresponds to handovers of type $\binom{k}{i,j}$ (in the old notations) and the right hand side corresponds to handovers of type $\bbinom{q}{l,m}$. With this in hand, we write the handover point process as
\[
\Vcal\equiv \Wcal:=\sum_{l\in [n]} \Wcal_{l}\; + \sum_{l\neq m\in [n], q\in \{1,2\}} \Wcal^{(q)}_{l,m}.
\]
Let $L_l$ be the intensity of the pure handover process $\Wcal_l$, for $l\in [n]$ and $L^{(q)}_{l,m}$ be the intensity of the mixed handover process $\Wcal^{(q)}_{l,m}$ of type $\bbinom{q}{l,m}$, for $l\neq m\in [n]$, $q\in \{1,2\}$. Then we have a decomposition of the Laplace transform of the inter-handover time $T$, identical to Theorem~\ref{theorem:T-Palm-MS}, under the Palm probability measure $\P^0_\Wcal\equiv \P^0_\Vcal$:
\begin{theorem}
Let $\rho\geq 0$. In the multi-speed case, the Laplace transform of $T$ under the Palm probability measure $\P^0_{\Vcal}$ is given by:
\begin{equation}
\E_{\Wcal}^0 [e^{-\rho T}]= \sum_{l\in [n]}\frac{L_{l}}{\la_{\Wcal}} \E_{\Wcal_l}^0 [e^{-\rho T}]\; + \sum_{l\neq m\in [n],k\in \{1,2\}} \frac{L_{l,m}^{(q)}}{\la_{\Wcal}}  \E^0_{\Wcal^{(q)}_{l,m}} [e^{-\rho T}].
\label{eq:decompLTn}
\end{equation}
\label{theorem:T-Palm-MSn}
\end{theorem}
\begin{proof}[Proof of Theorem~\ref{theorem:T-Palm-MSn}]
The result can be proved similarly to Theorem~\ref{theorem:T-Palm-MS}. 
We derive the expressions for $\E^0_{\Wcal_l} [e^{-\rho T}]$, for $l\in [n]$ and $\E^0_{\Wcal^{(q)}_{l,m}} [e^{-\rho T}]$, for $l\neq m\in [n], q\in \{1,2\}$, with the help of Lemma~\ref{lemma:Palm-Lcal-iin} and Lemma~\ref{lemma:Palm-Lcal-ijn}, following the same steps as those of Lemma~\ref{lemma:LT1-l} and Lemma~\ref{lemma:LT_1ijk} in the two-speed case, adapted to the multi-speed setting, with the help of Remark~\ref{remark:MS} and Remark~\ref{remark:nMS}.
\end{proof}
%


\section{Future work}\label{sec-future_work}
\label{open_problems} Our model and the corresponding results are the first steps towards understanding the handover phenomenon in models with mobile stations and static users. One can gradually relax the simplifications along the following line of research in order to analyse a richer class of dynamical tessellations. 
\begin{enumerate}[1.]
    \item \textbf{Random environment}: Exploring the motion of the stations on a random environment, for example on the Poisson line process, in the single and the multi-speed settings. 
    \vspace{0.05in}
    \item \textbf{Higher dimensions}: Generalizing all these models to $3$ and higher dimensions in both single and multiple speed scenarios.
    \vspace{0.05in}
    \item \textbf{Finite visibility}: Understanding all these models in $2$ and higher dimension, where the user or the stations or both have finite visibility. This restriction naturally appears in the spherical case. 
    \vspace{0.05in}
    \item \textbf{Spherical case}: It would be interesting to analyze handovers in the spherical case, where stations are moving along randomly oriented orbits, as in the communication model with LEO and MEO satellite constellations. We believe that the tools of spherical stochastic geometry and the knowledge acquired in all the scenarios defined in the present paper will help us understand this spherical case in the future.
    \vspace{0.05in}
    \item \textbf{Markov structure and more}: In general, the Markov structure described in Section~\ref{sec-MC1} and Section~\ref{sec-MC2} for the handover process should allow one to derive many structural properties of the process, e.g. CLT, LLN and large deviations, for all the models defined above.
    \vspace{0.05in}
    \item \textbf{Other distance metrics}: It would be interesting to determine the handover frequency for a general class of dynamic tessellations, where, instead of the Euclidean metric, one uses other metrics, for example: $\ell_p$-distances, Manhattan distance, supremum distance. This can be considered as a more abstract mathematical problem. 
    \vspace{0.05in}
    \item \textbf{Other dynamical tessellations}: The dynamical versions of the tessellations mentioned in the celebrated book of Okabe et al.~\cite{Okabe2000}, Lagurre-Voronoi tessellation defined in~\cite{Aurenhammer} and Q cells defined in~\cite{Haenggi-Qcells} would be important extensions of this article.
    \vspace{0.05in}
    \item \textbf{Other handover protocols}: One can define the handover in terms of SINR based association rule and answer same questions in this case.
    \vspace{0.05in}
    \item\textbf{Multi-hop association rule}: In the architecture of communication networks involving satellites, a user is served by an anchor station via multiple satellites connected through inter-satellite-links. Change of one of the connections: user to satellite, satellite to satellite or satellite to anchor station, leads to handovers of different types. Analyzing the geometric and statistical properties of this handover process in this practical setup and for all the models described above is an important goal of the line of research. 
    \vspace{0.05in}
    \item \textbf{Other types of dynamics}: In this article we have looked at the dynamics of the stations according to a simple way-point model, where the motion is without change of direction and speed. One can consider other classes of dynamics for the  stations according to, for example, more general random way-point models, jump processes, Brownian motion, diffusion processes etc., as long as the displacement theorem holds true for the underlying point process of stations. This gives rise to a rich classes of dynamical tessellations where a similar analysis would be of interest.  
\end{enumerate}


\section{Table of Notation}\label{sec-notation}
For the ease of reading, we provide a table of notation, abbreviations and their meaning, that have been used at various places in the article. 
\begin{longtable}{@{} l l l p{0.9\textwidth} @{}}
\caption{List of Notation} \label{tab:notation} \\
\toprule
\textbf{Symbol}  & \textbf{Page} & \textbf{Description} \\
\midrule
\endfirsthead 

\multicolumn{3}{@{}l@{}}{\textbf{Table (\ref{tab:notation}) (continued from previous page)}} \\
\toprule
\textbf{Symbol} &\textbf{Page}  & \textbf{Description} \\
\midrule
\endhead 

\midrule
\multicolumn{3}{@{}r@{}}{\textit{Continued to next page}} \\
\endfoot 
\bottomrule
\endlastfoot 
$\vert x\vert $&\pageref{notation:norm} &  Euclidean norm $\vert| x\vert|_2$ of $x\in \R^2$\\
$\vert A\vert $& &  Area of a region $A\subsetneq \R^2$\\
$\vert x\vert $& &  Absolute value of $x\in \R$\\
$B_r(x)$ & \pageref{notation:ball} & Open ball centered at $x$ and radius $r$\\
$\mathbb{H}^+$ &\pageref{notation:UH} & Upper half plane $\R\times \R^+$\\
$L_s$ & \pageref{notation:L_s}&  The vertical line $t=s$\\
$Q_s$ &\pageref{notation:Qs} & $\{(t,h)\in \mathbb H^+:t\geq s\}$, the quadrant on the right of the vertical line $L_s$\\
$\Lcal_{X}(\gamma)$ &\pageref{notation:LT} &Laplace transform of a random variable $X$ for parameter $\gamma\geq 0$\\
$X\leq_{Lt}Y$ & \pageref{notation:LTO} & $\Lcal_X(\gamma)\leq \Lcal_Y(\gamma)$, i.e., Laplace transforms order between two \\
& & random variables $X$ and $Y$, for any $\gamma\geq 0$\\
$\Phi_{\la}$ or $\Phi$ & \pageref{notation:phi} & Homogeneous Poisson point process (PPP)  on $\R^2$\\
$\tilde{\Phi}_{\la}$ or $\tilde{\Phi}$ &\pageref{notation:phiT} & Marked Poisson point process  on $\R^2$\\
$m_{\Psi}$ &\pageref{notation:m} & Intensity measure of any point process $\Psi$\\
$\Psi\vert_A$ &\pageref{notation:PRest} & Point process $\Psi$ restricted to a set $A$\\
$Supp(\Psi)$ &\pageref{notation:supp} & Support of a point process $\Psi$\\
$\P_\Psi$& \pageref{notation:PPsi} & Probability distribution of a Point process $\Psi$ \\
$\E_\Psi$& \pageref{notation:PPsiE} & Expectation with respect to $\P_\Psi$\\
$\Lcal_{\Psi}(f)$& \pageref{notation:LTf} & Laplace transform of a function $f$ w.r.t. a point process $\Psi$\\
$\P^0_\Psi$& \pageref{notation:PalmPsi} & Palm probability distribution of a Point process $\Psi$ \\
$\E^0_\Psi$& \pageref{notation:PalmPsiE} & Palm expectation of $\Psi$\\
$\Lcal^0_{\Psi}(f)$& \pageref{notation:LTf0} & Laplace transform of a function $f$ w.r.t. a point process $\Psi$ under \\
& & the Palm probability measure $\P^0_\Psi$\\
$\Hcal$& \pageref{notation:Hcal} & Head point process (HPP)\\
$C_{(R,\a)}$& \pageref{notation:BCset1} & Radial bird closed set corresponding to the station given by $(R,\a)$ \\
$C_{(t,h)}$& \pageref{notation:BCset2} & Radial bird with head point at $(t,h)\in \mathbb H^+$ \\
$\Hcal_c$& \pageref{notation:MHPP} & Marked head point process (MHPP), with closed-set valued marks \\
$\Pcal_c$& \pageref{notation:RBPP} & Radial bird particle process (RBPP)\\
$\Lcal_e$& \pageref{notation:Le}  & Lower envelope in single-speed\\
$\Lcal^+_e$&  \pageref{notation:Le+} & The open region above the lower envelope $\Lcal_e$ in single-speed\\
$\Psi^{x,x'}$ or $\Psi+\delta_{x}+\delta_{x'}$  &  \pageref{notation:2point} & Two-point Palm version of a Point process $\Psi$ with two points $x,x'$ \\
$\P^{x,x'}_{\Psi}$ & \pageref{notation:2pointP} & Two-point Palm probability measure of $\Psi$ with two points $x,x'$\\
$\E^{x,x'}_{\Psi}$ & \pageref{notation:2pointE} & Two-point Palm expectation with two points $x,x'$\\
$\Psi^{n, \neq}$& \pageref{notation:FPn} & Factorial power of order $n$ of a point process $\Psi$\\
$U_r^x$ & \pageref{notation:UHB} & Upper half-ball of radius $r$, centered at $(x,0)$\\
$\Vcal_\la$ or $\Vcal$ & \pageref{notation:Vcal1},  \pageref{notation:Vcal2}  & 
Handover point process (HoPP) \\
$\eta_s$ & \pageref{notation:eta_s}&  point process on $\R^+$ for intersection points of all the birds with $L_s$\\
$\la_\Vcal$& \pageref{notation:LVcal1},  \pageref{notation:LVcal2} & Intensity of the handover point process $\Vcal$ or the handover frequency\\
$\Hcal_{\scriptscriptstyle V}$& \pageref{notation:HcalV} & The point process of visible heads\\
$\Acal_{\scriptscriptstyle V}$& \pageref{notation:AcalV} & The point process of abscissas of visible head point from $\Hcal_{\scriptscriptstyle V}$\\
RMHP&  & Right-most head point\\
LMHP&  & Left-most head point\\
$\Rcal$&\pageref{notation:Rcal} & PP for abscissas of RMHP for a handover\\
$\Lcal$&  & PP for abscissa of LMHP for a handover\\
$v_{[n]}$ & \pageref{notation:vn1} & Set of speeds $\{v_1,v_2, \cdots, v_n\}$\\
$\Lcal_e(v_1,v_2)$, $\Lcal_e(v_{[n]})$&\pageref{notation:Le2}, \pageref{notation:Len} & Lower envelope in two-speed and $n$-speed case, respectively\\
$\Lcal^+_e(v_1,v_2)$, $\Lcal^+_e(v_{[n]})$ & \pageref{notation:Le2+}, \pageref{notation:Len+} & The open region above the lower envelope $\Lcal_e(v_1,v_2)$ \\
& & and  $\Lcal_e(v_{[n]})$ in   two and multi-speed case, respectively\\
$(\hat s,\hat h)$ & \pageref{notation:hats0}, \pageref{notation:hats} & $\equiv (\hat s(t_1,h_1,t_2,h_2),\hat h(t_1,h_1,t_2,h_2))$,  intersection point of two\\
& & birds of same type at $(t_1,h_1)$, $(t_2,h_2)$\\
$(\hat s_k,\hat h_k)$ & \pageref{notation:hatsk1}, \pageref{notation:hatsk} & $\equiv (\hat s_k(t_1,h_1,t_2,h_2),\hat h_k(t_1,h_1,t_2,h_2))$, $k$-th intersection point of \\
& & two birds of different type at $(t_1,h_1)$ and $(t_2,h_2)$, for $k=1,2$\\
$E_u^{s,v}$ & \pageref{notation:El} & Upper half-ellipse given by $v^2(t-s)^2+h^2=u^2$, for speed $v$, \\
& & in $(t,h)$-coordinate system \\
$\Hcal^l$ & \pageref{notation:HPPl} & HPP with head points of type $l$ \\
$C^l_{(t,h)}$ & \pageref{notation:RBCSl} & Radial bird of type $l$ with head at $(t,h)\in \mathbb H^+$ \\
$\Hcal^l{+}\delta_{(t_1, h_1)} {+}\delta_{(t_2, h_2)}$& \pageref{notation:2PHl} & Two-point Palm version of $\sum_{i\in [n]}\Hcal^i$, $(t_1,h_1), (t_2,h_2)$ of type $l$\\
$+\sum_{i\neq l}\Hcal^i$ & & \\
$\P^{l,(t_1, h_1), (t_2,h_2)}_{\Hcal}$ & \pageref{notation:2PHlP} & $ \P^{(t_1, h_1),(t_2,h_2)}_{\Hcal^l}\otimes\bigotimes_{i\neq l}  \P_{\Hcal^i}$, the two-point Palm probability measure \\
& & for $\Hcal^l+\delta_{(t_1, h_1)}+\delta_{(t_2, h_2)}+\sum_{i\neq l}\Hcal^i$\\
$\Hcal^l+\delta_{(t_1, h_1)}$ & \pageref{notation:2PHlm} &  Two-point Palm version of  $\sum_{i\in [n]}\Hcal^i$,  $(t_1,h_1)$ and $(t_2,h_2)$ of  \\
$+\Hcal^r+\delta_{(t_2, h_2)}$ & & type $l$ and $r$, respectively\\
$+\sum_{i\neq l,r}\Hcal^i$& & \\
$\P^{(l,r),(t_1, h_1), (t_2,h_2)}_{\Hcal}$ & \pageref{notation:2PHlmP}, \pageref{notation:2nPHlP} & $ \P^{(t_1, h_1)}_{\Hcal^l}\otimes \P^{(t_2,h_2)}_{\Hcal^r}\otimes\bigotimes_{i\neq l,r}  \P_{\Hcal^i}$, the two-point Palm probability \\
& & measure for $\Hcal^l+\delta_{(t_1, h_1)}+\Hcal^r+\delta_{(t_2, h_2)}+\sum_{i\neq l,r}\Hcal^i$\\
$\Vcal_l$& \pageref{notation:Vcall} & HoPP of pure type $l$\\
$\Rcal_l$& \pageref{notation:Rcall} & PP for abscissas of RMHP for pure handovers of type $l$\\
$\Lcal_l$& & PP for abscissas of LMHP for pure handovers of type $l$\\
$\binom{k}{l,r}$& \pageref{notation:binom} & The type of handover due to $k$-th intersection of radial birds of \\
& & type $l$ and $r$, with $l\neq r$ and the head point of type $l$ is on the \\
& & left of the head point of type $r$, where $l,r,k\in \{1,2\}$\\
$\Vcal_{l,r}^{(k)}$ for $k=1,2$& \pageref{notation:Vcallrk} & HoPP of type $\binom{k}{l,r}$\\
$\Vcal$ & \pageref{notation:Vcal3} & $:=\sum_{l\in \{1,2\}} \Vcal_{l}+ \sum_{r\neq l, k\in \{1,2\}} \Vcal^{(k)}_{l,r}$\\
$\La_{l,r}^{(k)}$ for $k=1,2$& \pageref{notation:Lalrk} & Intensity of the HoPP $\Vcal_{l,r}^{(k)}$ of type $\binom{k}{l,r}$\\
$\la_\Vcal$ & \pageref{notation:La} & $:=\sum_{l\in \{1,2\}} \La_l+ \sum_{r\neq l, k\in \{1,2\}} \La^{(k)}_{l,r}$, total handover frequency\\
$\Rcal_{l,r}^{(k)}$ for $k=1,2$& \pageref{notation:Rcallrk} &PP for abscissas of RMHP for mixed handover of type $\binom{k}{l,r}$\\
$\Lcal_{l,r}^{(k)}$ for $k=1,2$& \pageref{notation:Lcallrk} & PP for abscissas of LMHP for mixed handovers of type $\binom{k}{l,r}$\\
$\bbinom{q}{i,j}$ & \pageref{notation:bbinom1} & Types of handovers, with new serving station is of type $j$, the  \\
& &  previous serving station is of type $i$ and $q=1$ or $2$, depending    \\
& & on the head point corresponding to the new serving  \\
& & station is on left or right\\
$\Wcal_i$ & \pageref{notation:Wcal1} & The version of $\Vcal_l$, where $\binom{1}{l,l}\mapsto \bbinom{2}{i,i}$ \\
$L_{i}$ & \pageref{notation:L1} & The version of $\La_l$, where $\binom{1}{l,l}\mapsto \bbinom{2}{i,i}$\\
$\Wcal_{i,j}^{(q)}$ & \pageref{notation:Wcalijq} & The version of $\Vcal_{l,r}^{(k)}$, where $\binom{k}{l,r}\mapsto \bbinom{q}{i,j}$ \\
$L_{i,j}^{(q)}$ & \pageref{notation:Lijq} & The version of $\La_{l,r}^{(k)}$, where $\binom{k}{l,r}\mapsto \bbinom{q}{i,j}$ \\
$\Wcal$ & \pageref{notation:Wcal} & $:=\sum_{i\in \{1,2\}} \Wcal_i+ \sum_{i\neq j, q\in \{1,2\}} \Wcal^{(q)}_{i,j}\equiv  \Vcal$\\
$\la_{\Wcal}$ & \pageref{notation:laW} & $\equiv \la_{\Vcal}$\\
$\bbinom{q,q'}{i,j,m}$& \pageref{notation:bbinom2} & Type of pair of consecutive handovers, where the next and the   \\
& & upcoming serving stations are of type $j, m$, respectively and the
\\ 
& &  previous serving station is of type $i$.\\ 
& & The index $q=1$ or $2$, depending on the head point corresponding\\ 
& & to the next serving station is on the left or right that of previous one. \\
& & The index $q'=1$ or $2$, depending on the head point corresponding  \\
& & to the upcoming serving station is on the left or right that of next  \\
\end{longtable}

\appendix
\renewcommand{\thesection}{\Roman{section}}

\section{}\label{sec-appendix0}
\setcounter{equation}{0}
\renewcommand{\theequation}{\arabic{equation}}
\subsection{Displacement theorem for the Poisson point process}
For some $t>0$, define a {\em displacement map} $D_t: \R^2\times (-\pi, \pi] \to \R^2$ such that for any $x\in \R^2$ and direction $\theta\in (-\pi, \pi]$, $D_t(x, \theta):= x+v_xt$,
where $v_x:=(v\cos\theta, v\sin\theta)$, is the velocity vector of the atom at $x$. The displacements are all i.i.d. but all equal in norm for every atom and fixed $t$. For any $t>0$, the {\em displacement kernel} $\kappa_t:\R^2\times (-\pi, \pi]\times \Bcal(\R^2)\to \R^+$, is defined as 
\begin{equation}
\kappa_t(x,\theta, B):=\P(D_t(x, \theta)\in B \vert \tilde{\Phi})= \P(x+v_xt\in B\vert \tilde{\Phi})= \one_{\{x+v_xt\in B\}},
\label{eq:d-kernel}
\end{equation}
for any $B\in \Bcal(\R^2)$, initial location $x\in \R^2$ and direction of motion $\theta\in (-\pi, \pi]$, where $\Bcal(\R^2)$ is the Borel $\sigma$-algebra on $\R^2$. Since $\kappa_t(x, \theta, \R^2)=1$, the measure $\kappa_t(x, \theta, \cdot)$ is a probability kernel for any $x\in \R^2$ and $\theta\in (-\pi, \pi]$, see~\cite[definition 14.D.1]{Baccelli-Bartek-Karray}.  
\begin{proposition}
For any $t>0$, the following statements are true:
\begin{enumerate}[(i)]
\item \label{Phi-Phi-t1} The displacement kernel $\kappa_t$ preserves the Lebesgue measure.   
\item \label{Phi-Phi-t2}  The point process $\tilde{\Phi}^t$ has the same distribution as $\Phi$.
\end{enumerate}
\label{proposition:Phi-Phi-t}
\end{proposition}
\begin{proof}[Proof of Proposition~\ref{proposition:Phi-Phi-t}(\ref{Phi-Phi-t1})]
For the first part, let $B$ be a Borel subset of $\R^2$, $x\in \R^2$ and $\theta\in (-\pi, \pi]$. The displacement kernel is $\kappa_t(x,\theta, B)= \one_{\{x+v_xt\in B\}}$. The intensity measure of the point process $\tilde \Phi^t$ becomes
\begin{align}
m_{\tilde{\Phi}^t}(B)=\E[\tilde{\Phi}^t(B)]= \E\left[\E\left[\sum_{i\in \N}\one_{\{Y_i\in B\}}\vert \tilde{\Phi}\right]\right]
&= \E\left[\sum_{i\in \N} \E\left[\one_{\{Y_i\in B\}}\vert \tilde{\Phi}\right]\right]= \E\left[\sum_{i\in \N} \kappa_t(X_i, \Theta_i, B)\right].
\label{eq:preserve}
\end{align}
Applying the Campbell-Mecke formula in (\ref{eq:preserve}), we get 
\begin{align}
\E\left[\sum_{i\in \N} \kappa_t(X_i, \Theta_i, B)\right]
=\E\left[\int_{\R^2\times (-\pi, \pi]}\kappa_t(x, \theta, B)m_{\tilde{\Phi}}({\rm d}x , {\rm d}\theta )\right]
&= \int_{\R^2}\int_{-\pi}^{\pi}  \kappa_t(x, \theta, B) \frac{\la}{2\pi} {\rm d}x \, {\rm d}\theta.
\label{eq:preserve0}
\end{align}
Using the displacement kernel from (\ref{eq:d-kernel}) the last term in (\ref{eq:preserve0}) equals to
\begin{align}
\frac{\la}{2\pi} \int_{\R^2}\int_{-\pi}^{\pi} \one_{\{x+v_xt\in B\}}{\rm d}x \, {\rm d}\theta  
&=\frac{\la}{2\pi}\int_{-\pi}^{\pi}\int_{\R^2}\one_{\{x\in B-v_xt\}} {\rm d}x \,{\rm d}\theta  
= \frac{\la}{2\pi}\int_{-\pi}^{\pi} |B-v_x t| {\rm d}\theta  =\la |B|.
\label{eq:preserve1}
\end{align}
Combining (\ref{eq:preserve})--(\ref{eq:preserve1}) we can say that the intensity measure $m_{\tilde{\Phi}^t}$ of $\tilde{\Phi}^t$ also has Lebesgue density.
\end{proof}
\begin{proof}[Proof of Proposition~\ref{proposition:Phi-Phi-t}~(\ref{Phi-Phi-t2})]
For the second part, since the displacements are i.i.d., by the displacement theorem \cite[Theorem 2.2.17]{Baccelli-Bartek-Karray}, the point process $\tilde{\Phi}^t$ is Poisson. This can be seen through the Laplace transform with respect to the point process $\tilde{\Phi}^t$. Let us enumerate the points of $\Phi$ and $\Phi^t$ respectively as $\{X_i\}_{i\in \N}$ and $\{Y_i\}_{i\in \N}$ respectively.  

Let $f:\R^2\to \R^+$ be a measurable function.
The Laplace transform of $f$ with respect to $\tilde \Phi^t$ is
\begin{align}
\!\!\!\!\Lcal_{\tilde{\Phi}^t}(f):=\E\left[\exp{\left(-\int_{\R^2}f\, {\rm d}\tilde{\Phi}^t\right)}\right]=\E\left[\exp\left(-\sum_{i\in \N}f(Y_i)\right)\right]& =\E\left[\prod_{i\in \N}e^{-f(Y_i)}\right].
\label{eq:Laplace-t1x}
\end{align}
Conditioning with respect to the point process $\tilde \Phi$ for last term in (\ref{eq:Laplace-t1x}), we have
\begin{align}
\E\left[\prod_{i\in \N}e^{-f(Y_i)}\right]&=\E\left[\E\left[\prod_{i\in \N}e^{-f(Y_i)}\vert\tilde{\Phi} \right]\right]
=\E\left[\prod_{i\in \N}\int_{\R^2}e^{-f(y)}\kappa_t(X_i, \Theta_i, {\rm d}y  )\right].
\label{eq:Laplace-t1}
\end{align}
From (\ref{eq:d-kernel}) we write, $\kappa_t(X_i,\Theta_i, {\rm d}y  )= \one_{\{X_i+V_it\in {\rm d}y  \}} $. Let us now define a measurable function $g:\R^2\times (-\pi, \pi]\to \R^+$ using the inner integral in (\ref{eq:Laplace-t1}), 
\begin{align}
    g(x,\theta)&=-\log\int_{\R^2}e^{-f(y)}\kappa_t(x, \theta, {\rm d}y  ).
    \label{eq:laplace-tphi}
\end{align} 
Using the function $g$ from (\ref{eq:laplace-tphi}) in (\ref{eq:Laplace-t1}) we obtain the Laplace transform as
\begin{align}
\Lcal_{\tilde{\Phi}^t}(f)&=\E\left[\exp{\left(-\sum_{i\in \N}g(X_i, \Theta_i)\right)}\right]=\E\left[\exp{\left(-\int_{\R^2\times (-\pi, \pi]}g {\rm d}\tilde{\Phi} \right)}\right]= \Lcal_{\tilde{\Phi}}(g).
\label{eq:Laplace-t2}
\end{align}
Evaluating the Laplace transform $\Lcal_{\tilde{\Phi}}(g)$ of $g$ with respect to $\tilde{\Phi}$, we obtain 
\begin{align}
\Lcal_{\tilde{\Phi}}(g)&=\exp\left(-\int_{\R^2}\int_{-\pi}^{\pi}\left(1-e^{-g(x, \theta)}\right)m_{\tilde{\Phi}}({\rm d}x , {\rm d}\theta )\right) \nn\\
&\stackrel{(\ref{eq:laplace-tphi})}{=}\exp\left(-\int_{\R^2}\int_{-\pi}^{\pi}\left(1-\int_{\R^2}e^{-f(y)}\one_{\{x+v_xt\in {\rm d}y  \}}\right)m_{\tilde{\Phi}}({\rm d}x , {\rm d}\theta )\right) \nn\\
&\stackrel{(Prop.~\ref{proposition:Phi-Phi-t}~(\ref{Phi-Phi-t1}))}{=}\exp\left(-\int_{\R^2}\int_{-\pi}^{\pi}\int_{\R^2}\left(1-e^{-f(y)}\right)\one_{\{x+v_xt\in {\rm d}y  \}}m_{\tilde{\Phi}}({\rm d}x , {\rm d}\theta )\right) \nn\\
&\stackrel{(Fubini)}{=}\exp\left(-\int_{\R^2}\left(1-e^{-f(y)}\right) \int_{\R^2}\int_{-\pi}^{\pi}\one_{\{x+v_xt\in {\rm d}y  \}}m_{\tilde{\Phi}}({\rm d}x , {\rm d}\theta )\right)\nn\\
&\stackrel{(\ref{eq:preserve1})}{=}\exp\left(-\int_{\R^2}\left(1-e^{-f(y)}\right)m_{\tilde{\Phi}^t}({\rm d}y)\right)
\stackrel{(Prop.~\ref{proposition:Phi-Phi-t}\,(\ref{Phi-Phi-t1}))}{=}\exp\left(-\int_{\R^2}\left(1-e^{-f(y)}\right)\la {\rm d}y  \right)= \Lcal_{\Phi}(f),
\label{eq:Laplace-t3}
\end{align}
since from part~(\ref{Phi-Phi-t1}), we get that $m_{\tilde{\Phi}^t}$ has Lebesgue density. We obtain from (\ref{eq:Laplace-t2}) and (\ref{eq:Laplace-t3}) that $\Lcal_{\tilde{\Phi}^t}(f)=\Lcal_{\Phi}(f)$ and hence by the characterization of the PPP, $\tilde{\Phi}^t$ is also Poisson.  
\end{proof}

%
\section{}\label{sec-appendix}
\setcounter{equation}{0}
\renewcommand{\theequation}{\thesection.\arabic{equation}}
\subsection{Proof of Lemma~\ref{lem:piP'}}~\label{subsection:AcalV} For $(s,h)\in \mathbb H^+$, the event that $\{(s,h)\text{ is an atom of } \Hcal_{\scriptscriptstyle V}\}$ is
equivalent to the event that $ \Hcal(U_h^s)=0$. Thus the intensity of the point process $\Acal_{\scriptscriptstyle V}$, i.e.,
the expected number of points of $\Acal_{\scriptscriptstyle V}$ in the unit interval $[0,1]$ of the time axis is
\begin{align}
\la_{\scriptscriptstyle V}=\E\left[\sum_{T^{\scriptscriptstyle V}_j\in \Acal_{\scriptscriptstyle V}\,\mbox{:}\, 0\leq T^{\scriptscriptstyle V}_j\leq 1} 1\right] = \E\left[\sum_{(T_i,H_i)\in \Hcal\,\mbox{:}\, 0\leq T_i\leq 1} \one_{\Hcal(U^{T_i}_{H_i})=0}\right],\nn
\end{align}
using (\ref{eq:Hcal_V}). By applying the Campbell-Mecke formula for $\Hcal$, we have
\begin{align}
\la_{\scriptscriptstyle V} &=2\la\int_{0}^{1}\int_{0}^{\infty} e^{-2\la |U_h^t|}\, {\rm d} h\,{\rm d} t=2\la\int_{0}^{\infty} \int_{0}^{1}e^{-\la\pi h^2}\, {\rm d} t\,{\rm d} h= 2\la\int_{0}^{\infty} e^{-\la \pi h^2}\,{\rm d} h= \sqrt{\la}.\nn \qed
\end{align}  
\subsection{Proof of Lemma~\ref{lemma:LT_visible_head}}~\label{subsection:L-LTVH}
Let $\gamma\geq 0$. The Laplace transform of the random variable $H_{\scriptscriptstyle V}^2$, under the Palm probability distribution $\P^0_{\Acal_{\scriptscriptstyle V}}$ of the $\Acal_{\scriptscriptstyle V}$, is
\begin{align}
\Lcal^0_{H_{\scriptscriptstyle V}^2}(\gamma){:=} \E^0_{\Acal_{\scriptscriptstyle V}}\left[e^{-\gamma H_{\scriptscriptstyle V}^2}\right] 
{=} \frac{1}{\la_{\scriptscriptstyle V}}\E\left[\sum_{(T^{\scriptscriptstyle V}_j, H^{\scriptscriptstyle V}_j)\in \Hcal_{\scriptscriptstyle V}\,\mbox{:}\, T^{\scriptscriptstyle V}_j\in [0,1]} \!\!\!\!\!\!\!\! e^{-\gamma  \left(H^{\scriptscriptstyle V}_j\right)^2}\right]
&{=} \frac{1}{\la_{\scriptscriptstyle V}}\E\left[\sum_{(T_i,H_i)\in \Hcal\,\mbox{:}\, T_i\in [0,1]} \!\!\!\!\!\!\!\! \!\!\!\!\!\! e^{-\gamma H_i^2}\one_{\{(T_i, H_i)\in \Lcal_e\}}\right].\nn
\end{align}

Using the fact that for any $(t,h)\in \mathbb H^+$, $\{(t,h)\in \Lcal_e\}=\{\Hcal(U^t_h)=0\}$, by part~(\ref{emptyU}) of Lemma~\ref{lem:semicircle}, in the last expression we get 
\begin{align}
\Lcal^0_{H_{\scriptscriptstyle V}^2}(\gamma)
& {=} \frac{1}{\la_{\scriptscriptstyle V}}\E\left[\sum_{(T_i,H_i)\in \Hcal\,\mbox{:}\, T_i\in [0,1]} \!\!\!\!\!\!\!\! \!\!\!\!\! e^{-\gamma  H_i^2}\one_{\{\Hcal(U^{T_i}_{H_i})=0\}}\right],
\label{eq:Hv}
\end{align}
Applying the Campbell-Mecke formula for the point process $\Hcal$ in (\ref{eq:Hv}), we get
\begin{align}
\Lcal^0_{H_{\scriptscriptstyle V}^2}(\gamma)=\frac{2\la}{\sqrt{\la}}\int_{0}^{1} {\rm d}t \int_{0}^{\infty} e^{-\gamma h^2} e^{-\pi \la h^2} \, {\rm d}h &=2\sqrt{\la} \int_{0}^{\infty}  e^{-(\pi \la +\gamma)h^2} \, {\rm d}h
%
%
=\left(1+\frac{\gamma}{\la\pi}\right)^{-\half}.
\end{align}
%
Thus $H_{\scriptscriptstyle V}^2$ is distributed as a $\Gamma\left(\half, \frac{1}{\la\pi}\right)$ random variable and consequently $H_{\scriptscriptstyle V}$ follows a Nakagami distribution with parameter $\left(\half, \frac{1}{2\la\pi}\right)$. \qed
\subsection{Proof of Lemma~\ref{lemma:Laplace_order}}~\label{subsection:Lemma-LT} Recall the probability densities of the random variables  $\hat{H}$, $\tilde{H}$ and $H_{\scriptscriptstyle V}$  
\[
f_{\hat{H}}(h)=4\pi\la^{3/2} h^2 e^{-\la\pi h^2}, f_{\tilde{H}}(h):=2\la\pi h \, e^{-\la\pi h^2} \text{ and } f_{H_{\scriptscriptstyle V}}(h)=2\la^{1/2} e^{-\la\pi h^2},
\text{ for all } h\geq 0,
\]
respectively. Observe that $f_{\tilde{H}}(h)\leq f_{H_{\scriptscriptstyle V}}(h)$ on $[0, 1/{\pi\sqrt{\la}}]$ and $f_{\tilde{H}}(h)>f_{H_{\scriptscriptstyle V}}(h)$ on $(1/{\pi\sqrt{\la}}, \infty)$. In terms of Laplace transforms for all $\gamma\geq 0$, we have
\begin{align}
\Lcal_{H_{\scriptscriptstyle V}}(\gamma)-\Lcal_{\tilde{H}}(\gamma)&= \int_0^\infty e^{-\gamma h}\left(f_{H_{\scriptscriptstyle V}}(h)- f_{\tilde{H}}(h)\right)\,{\rm d}h\nn\\
&= \int_0^{1/{\pi\sqrt{\la}}} e^{-\gamma h}\left(f_{H_{\scriptscriptstyle V}}(h)- f_{\tilde{H}}(h)\right)\,{\rm d}h
- \int_{1/{\pi\sqrt{\la}}}^\infty e^{-\gamma h}\left(f_{\tilde{H}}(h)-f_{H_{\scriptscriptstyle V}}(h)\right)\,{\rm d}h.\nn
\end{align}
Using the value $h=\frac{1}{\pi\sqrt{\la}}$ gives a lower bound, i.e.,
\begin{align}
\Lcal_{H_{\scriptscriptstyle V}}(\gamma)-\Lcal_{\tilde{H}}(\gamma) &\geq e^{-\frac{\gamma}{\pi\sqrt{\la}}}\left[\int_0^{1/{\pi\sqrt{\la}}} \left(f_{H_{\scriptscriptstyle V}}(h)- f_{\tilde{H}}(h)\right)\,{\rm d}h
- \int_{1/{\pi\sqrt{\la}}}^\infty \left(f_{\tilde{H}}(h)-f_{H_{\scriptscriptstyle V}}(h)\right)\,{\rm d}h\right]\nn\\
&=e^{-\frac{\gamma}{\pi\sqrt{\la}}}\int_0^\infty \left(f_{H_{\scriptscriptstyle V}}(h)- f_{\tilde{H}}(h)\right)\,{\rm d}h =0.\nn
\end{align}
Thus $\Lcal_{H_{\scriptscriptstyle V}}(\gamma)\geq\Lcal_{\tilde{H}}(\gamma)$ for all $\gamma\geq 0$, which implies $ H_{\scriptscriptstyle V} \leq_{Lt} \tilde{H}$. On the other hand, observe that $f_{\hat{H}}(h)\leq f_{\tilde{H}}(h)$ on $[0, 1/{2\sqrt{\la}}]$ and $f_{\hat{H}}(h)>f_{\tilde{H}}(h)$ on $(1/{2\sqrt{\la}}, \infty)$. In terms of Laplace transforms, we have
\begin{align}
    \Lcal_{\tilde{H}}(\gamma) - \Lcal_{\hat{H}}(\gamma)&= \int_0^\infty e^{-\gamma h}\left(f_{\tilde{H}}(h)- f_{\hat{H}}(h)\right)\,{\rm d}h\nn\\
    &= \int_0^{1/{2\sqrt{\la}}} e^{-\gamma h}\left(f_{\tilde{H}}(h)- f_{\hat{H}}(h)\right)\,{\rm d}h
    - \int_{1/{2\sqrt{\la}}}^\infty e^{-\gamma h}\left(f_{\hat{H}}(h)-f_{\tilde{H}}(h)\right)\,{\rm d}h.\nn
\end{align}
Using the value $h=\frac{1}{2\sqrt{\la}}$ again gives the lower bound
\begin{align}
\Lcal_{\tilde{H}}(\gamma) - \Lcal_{\hat{H}}(\gamma)
&\geq e^{-\frac{\gamma}{2\sqrt{\la}}}\left[\int_0^{1/{2\sqrt{\la}}} \left(f_{\tilde{H}}(h)- f_{\hat{H}}(h)\right)\,{\rm d}h - \int_{1/{2\sqrt{\la}}}^\infty \left(f_{\hat{H}}(h)-f_{\tilde{H}}(h)\right)\,{\rm d}h\right] =0.\nn
\end{align}
Thus $\Lcal_{\tilde{H}}(\gamma)\geq\Lcal_{\hat{H}}(\gamma)$ for all $\gamma\geq 0$, which implies $\tilde{H} \leq_{Lt} \hat{H}$. The completes the proof of the Laplace transform ordering (\ref{eq:LTO}) among the typical distances.\qed
\subsection{Proof of Lemma~\ref{lemma:increasingS}}\label{subsection:1speed-inc}
Fix $\hat{h}_0>0$ and suppose there are two radial birds with head at $(t_1,h_1)$ and $(t_2,h_2)$, such that $t_2\leq t_1$. The radial birds $C_{(t_1,h_1)}$  and $C_{(t_1,h_1)}$ intersect at $(0,\hat h_0)$ and $(t_2,h_2)$ lies anywhere on the part of the boundary $U^0_{\hat h_0}$ lying on the left of the vertical line $t=t_1$. We prove the result in two cases: $t_1\geq 0$ and $t_1<0$.
\begin{case}[$t_1\geq 0$.] Consider the set $U^t_{\hat{h}_t}\setminus U^0_{\hat{h}_0}$, where $\hat{h}_t:= \left((t-t_1)^2+h_1^2\right)^\half$, for some fixed $(t_1,h_1)\in \partial U^0_{\hat{h}_0}$. Let $Q_{t_1}:=\{(t,h)\in \mathbb H^+: t\geq t_1\}$, be the region on the right of the vertical line $t=t_1$ in $\mathbb H^+$. 
\begin{figure}[ht!]
    \begin{tikzpicture}[line width=0.6pt, scale=0.8, every node/.style={scale=0.65}]
    \pgftransformxscale{0.95}  
    \pgftransformyscale{0.95}    
    \draw[->] (-3.5, 0) -- (8.5, 0) node[right] {$t$};
    \draw[domain=-2:6, smooth] plot (\x, {(\x*\x-3.5*\x+1.75^2+2.15^2)^0.5});
    \draw[](-0.5,-0.35) node{$(0,0)$};
    \draw[](1.75,2.15) node{$\bullet$};
    \draw[](2,3) node{$(t_1,h_1)$}; 
    \draw[]  (2.75,0) arc (0:180:2.82 and 2.82);
    \draw[teal]  (3,0) arc (0:180:2.5 and 2.5);
    \draw[teal] (0.45, 0) -- (0.45, 2.5);
    \draw[teal] (0.45, -0.25)node{$\tau_1$};
    \draw[purple]  (1,0) arc (-180:-360:3.49 and 3.49);
    \draw[blue]  (0.3,0) arc (-180:-360:2.35 and 2.35);
    \draw[red]  (-1.3,0) arc (-180:-360:2.3 and 2.3);
    \draw[red] (0.9, 0) -- (0.9, 2.3);
    \draw[red] (0.9, -0.25)node{$\tau_2$};
    \draw[->] (-0.07, 0) -- (-0.07, 4.3);
    \draw[] (-0.07, 4.6)node{$h$};
    \draw[-] (1.75, 0) -- (1.75, 4.3);
    \draw[] (2, -0.25)node{$t=t_1$};
    \draw[](-0.3, 1.3) node{$\hat h_0$};
     \draw[](-0.5,0.6) node{$U^{0}_{\hat h_0}$};
     \draw[](-2.5,1.42)
     node{$\bullet$};
     \draw[](-3.2,1.2)
     node{$(t_2, h_2)$};
     \draw[domain=-4:1.7, smooth] plot (\x, {(\x*\x+2*2.5*\x+2.5^2+1.42^2)^0.5});
    \end{tikzpicture}
    \captionsetup{width=0.9\linewidth}
    \caption{For $0<\tau_1<\tau_2$, the half balls in color teal and red correspond to $\tau_1$ and $\tau_2$, respectively. Here   $U^t_{\hat h_{\tau_1}}\cap Q_{t_1}\subsetneq U^t_{\hat h_{\tau_2}}\cap Q_{t_1}$. The blue and purple half balls are for reference.}
    \label{figure:increasingT1}
\end{figure}

Define $R_t:=U^t_{\hat h_t}\cap Q_{t_1}$. Our claim holds if and only if the same holds for the collection $\left\{R_t\right\}_{t>0}$, since for each $t>0$, $R_t=(U^t_{\hat{h}_t}\setminus U^0_{\hat{h}_0})\cup (U^0_{\hat h_0}\cap Q_{t_1})$ and $U^t_{\hat{h}_t}\setminus U^0_{\hat{h}_0}\subset R_t$. Let $0<\tau_1<\tau_2$. The proof will be complete if we can prove that $R_{\tau_1}\subsetneq R_{\tau_2}$. Observe that $\partial U^{\tau_1}_{\hat h_{\tau_1}}$ and $\partial U^{\tau_2}_{\hat h_{\tau_2}}$ intersects at $(t_1,h_1)$, and they intersect each other just once. Using properties of two non-concentric and intersecting circles, for all $t'<t_1$, $\left(\hat h_{\tau_1}^2-(t'-\tau_1)^2\right)^\half > \left(\hat h_{\tau_2}^2-(t'-\tau_2)^2\right)^\half$ and for all $t'>t_1$, $\left(\hat h_{\tau_1}^2-(t'-\tau_1)^2\right)^\half < \left(\hat h_{\tau_2}^2-(t'-\tau_2)^2\right)^\half$, irrespective of the values of $\hat h_{\tau_1}$ and $\hat h_{\tau_2}$. This in particular implies $U^t_{\hat h_{\tau_1}}\cap Q_{t_1}\subsetneq U^t_{\hat h_{\tau_2}}\cap Q_{t_1}$, i.e.,  $R_{\tau_1}\subsetneq R_{\tau_2}$, see Figure~\ref{figure:increasingT1}.
\end{case}
\begin{figure}[ht!]
    \begin{tikzpicture}[line width=0.6pt, scale=0.8, every node/.style={scale=0.65}]
    \pgftransformxscale{1}  
    \pgftransformyscale{1}    
    \draw[->] (-4, 0) -- (7, 0) node[right] {$t$};
    \draw[](0,-0.35) node{$(0,0)$};
    \draw[violet,domain=-4:2.5, smooth] plot (\x, {(\x*\x+3.5*\x+1.75^2+2.25^2)^0.5});
    \draw[](-1.75,2.25) node{$\bullet$};
    \draw[](-2.3,2.8) node{$(t_1,h_1)$}; 
    \draw[]  (2.75,0) arc (0:180:2.82 and 2.82);
    \draw[teal]  (4,0) arc (0:180:3.31 and 3.31);
    \draw[teal] (0.68, 0) -- (0.68, 3.31);
    \draw[teal] (0.68, -0.25)node{$\tau_1$};
    \draw[red]  (-2.51,0) arc (-180:-360:3.7 and 3.7);
    \draw[red] (1.25, 0) -- (1.25, 3.7);
    \draw[red] (1.25, -0.25)node{$\tau_2$};
    \draw[purple]  (-2.4,0) arc (-180:-360:4.2 and 4.2);
    \draw[->] (-0.07, 0) -- (-0.07, 4.3);
    \draw[] (-0.07, 4.6)node{$h$};
    \draw[purple] (1.8, 0) -- (1.8, 4.2);
    \draw[-] (-1.74, 0) -- (-1.74, 4.3);
    \draw[] (-1.8, -0.25)node{$t=t_1$};
    \draw[](-0.3, 1.3) node{$\hat h_0$};
     \draw[](-0.5,0.6) node{$U^{0}_{\hat h_0}$};
     \draw[](-2.5,1.42)
     node{$\bullet$};
     \draw[](-3.2,1.2)
     node{$(t_2, h_2)$};
     \draw[domain=-4:1.7, smooth] plot (\x, {(\x*\x+2*2.5*\x+2.5^2+1.42^2)^0.5});
    \end{tikzpicture}
    \captionsetup{width=0.9\linewidth}
    \caption{For $0<\tau_1<\tau_2$, the half balls in color teal and red correspond to $\tau_1$ and $\tau_2$, respectively. Here   $U^t_{\hat h_{\tau_1}}\cap Q_{t_1}\subsetneq U^t_{\hat h_{\tau_2}}\cap Q_{t_1}$. The purple half ball is just a reference with respect to the next handover.}
    \label{figure:increasingT2}
\end{figure}
\begin{case}[$t_1<0$.] One can prove the claim for the case $t_1<0$ similarly, since the intersection point corresponding to the handover at time $0$, lies on the increasing wing of the radial bird $C_{(t_1,h_1)}$, as depicted in Figure~\ref{figure:increasingT2}. By the {\em increasing wing property} as in Remark~\ref{remark:inc-arm}, for all $t\geq t_1$, the extra region $U^t_{\hat{h}_t}\setminus U^0_{\hat{h}_0}$ created by subsequent half-balls, is increasing. \qed
\end{case} 
\subsection{Area of the region $U^{s}_{h}\cap U^{s'}_{h'}$}~\label{subsection:UcapU}
\begin{lemma}
Suppose $s, s'\in \R$ with $s< s'$ and $h, h'\in \R^+$ for which $U^{s}_{h}\cap U^{s'}_{h'}\neq \emptyset$ and $\exists h_1\geq 0$ such that $ \partial^*U^{s}_{h}\cap \partial^*U^{s'}_{h'}= \{(0, h_1)\}$. Then we have 
\begin{align}
\left\vert U^{s}_{h}\cup U^{s'}_{h'}\right\vert&=
\begin{cases}
    \frac{\pi}{2}h^2+\half h_1(s'-s)+\half F_1(h_1, h, h'), &\text{ if } s<s'<0,\\ 
    \frac{\pi}{2}(h^2+h'^2)+\half h_1(s'-s)-\half F_2(h_1, h, h'), &\text{ if } s<0<s',\\ 
    \frac{\pi}{2}h'^2+\half h_1(s'-s)-\half F_1(h_1, h, h'), &\text{ if }0<s<s',
\end{cases}
\label{UcupU}
\end{align}
where 
\[
F_1(h_1, h, h')=h'^2\arcsin(h_1/h')- h^2\arcsin(h_1/h),\;\; F_2(h_1, h, h')=h'^2\arcsin(h_1/h')+h^2\arcsin(h_1/h).
\]
\label{lemma:UcupU}
\end{lemma}
\begin{proof}
Suppose $s, s'\in \R$ with $s\leq s'$ and $h, h'\in \R^+$ for which $U^{s}_{h}\cap U^{s'}_{h'}\neq \emptyset$ and $\exists h_1\geq 0$ such that $ \partial^*U^{s}_{h}\cap \partial^*U^{s'}_{h'}= \{(0, h_1)\}$. Observe that the area of the region $U^{s}_{h}\cup U^{s'}_{h'}$ (see Figure~\ref{figure:UcupU}) is
\begin{align}
    \vert U^{s'}_{h'}\cup U^{s}_{h}\vert &=
    \begin{cases}
       |U^{s}_{h}|+|U^{s'}_{h'}\setminus U^{s}_{h}|, & \text{ if } s<s'<0,\\
    |U^{s'}_{h'}|+|U^{s}_{h}|-|U^{s'}_{h'}\cap U^{s}_{h}|,& \text{ if } s<0<s',\\
    |U^{s'}_{h'}|+|U^{s}_{h}\setminus U^{s'}_{h'}|, & \text{ if } 0<s<s'.
\end{cases}
\label{eq:ucu}
\end{align}
In the first case $s<s'<0$, we have 
\begin{align}
|U^{s'}_{h'}\setminus U^{s}_{h}|&=\int_0^{h_1}\int^{s'+\left(h'^2-y^2\right)^\half}_{s+\left(h^2-y^2\right)^\half} \, {\rm d}x \, {\rm d}y  \nn\\
&= h_1(s'-s)+\int_0^{h_1}\left[\left(h'^2-y^2\right)^\half-\left(h^2-y^2\right)^\half\right] \, {\rm d}y  \nn\\
&= h_1(s'-s)+ \half\left[h_1\left(h'^2-h_1^2\right)^\half + h'^2\arcsin(h_1/h')\right] -\half\left[h_1\left(h^2-h_1^2\right)^\half + h^2\arcsin(h_1/h)\right]\nn\\
&=h_1(s'-s) +\half h_1(|s'|-|s|)+  \half\left[ h'^2\arcsin(h_1/h')-h^2\arcsin(h_1/h)\right]\nn\\
&=\half h_1(s'-s)+\half F_1(h_1, h, h'),
\label{eq:U'-U}
\end{align}
where, 
\begin{equation}
F_1(h_1, h, h'):=h'^2\arcsin(h_1/h')- h^2\arcsin(h_1/h).
\label{eq:Lambda1}
\end{equation}
In the second case $s<0<s'$, we have 
\begin{align}
\vert U^{s'}_{h'} \cap U^s_h\vert &=\int_0^{h_1}\int^{s+\left(h^2-y^2\right)^\half}_{s'-\left(h'^2-y^2\right)^\half} \, {\rm d}x \, {\rm d}y  \nn\\
&= - h_1(s'-s)+\int_0^{h_1}\left[\left(h'^2-y^2\right)^\half+\left(h^2-y^2\right)^\half\right] \, {\rm d}y  \nn\\
&= -h_1(s'-s)+ \half\left[h_1\left(h'^2-h_1^2\right)^\half + h'^2\arcsin(h_1/h')\right] + \half\left[h_1\left(h^2-h_1^2\right)^\half + h^2\arcsin(h_1/h)\right]\nn\\
&=-h_1(s'-s) +\half h_1(|s'|+|s|)+  \half\left[ h'^2\arcsin(h_1/h')+h^2\arcsin(h_1/h)\right]\nn\\
&=-\half h_1(s'-s)+\half F_2(h_1, h, h'),
\label{eq:UcapU2}
\end{align}
where, 
\begin{equation}
F_2(h_1, h, h'):=h'^2\arcsin(h_1/h')+ h^2\arcsin(h_1/h).
\label{eq:Lambda2}
\end{equation}
In the last case $0<s<s'$, similarly we have
\begin{align}
|U^{s}_{h}\setminus U^{s'}_{h'}|&=\int_0^{h_1}\int^{s'-\left(h'^2-y^2\right)^\half}_{s-\left(h^2-y^2\right)^\half} \, {\rm d}x \, {\rm d}y  \nn\\
&= h_1(s'-s)+\int_0^{h_1}\left[\left(h^2-y^2\right)^\half-\left(h'^2-y^2\right)^\half\right] \, {\rm d}y  \nn\\
&= h_1(s'-s)+\half\left[h_1\left(h^2-h_1^2\right)^\half + h^2\arcsin(h_1/h)\right]- \half\left[h_1\left(h'^2-h_1^2\right)^\half + h'^2\arcsin(h_1/h')\right] \nn\\
&=h_1(s'-s) +\half h_1(|s|-|s'|)+  \half\left[ h^2\arcsin(h_1/h)-h'^2\arcsin(h_1/h')\right]\nn\\
&=\half h_1(s'-s)-\half F_1(h_1, h, h'),
\label{eq:U-U'}
\end{align}
where the function $F_1$ is defined in (\ref{eq:Lambda1}). Considering all the cases for $s,s'$, we get the result by combining the expressions for $|U^{s'}_{h'}\setminus U^{s}_{h}|$, $|U^{s}_{h}\cap U^{s'}_{h'}|$ and $|U^{s}_{h}\setminus U^{s'}_{h'}|$ respectively from (\ref{eq:U'-U}), (\ref{eq:UcapU2}) and (\ref{eq:U-U'}) to (\ref{eq:ucu}). 
\end{proof}
\subsection{Proof of Lemma~\ref{lemma:2int}}~\label{subsection:Lemma-2int}
Let $(s,h)\in \Lcal_e(v_1, v_2)$ be an intersection point of the two radial birds $C_1,C_2$ with heads at $(t_1, h_1)$ and $(t_2, h_2)$ and with speeds $v_1, v_2$, respectively. We prove the result only for the case with $t_2\leq t_1$. The other case $t_2\geq t_1$ can be proved similarly.

If $h_2\geq h_1$, the quantity $\Delta$ is larger than or equal to $0$. On the other hand, $\Delta$ can become negative for $h_2<h_1$, but depending on $t_1, t_2$. So one can say that the radial birds with different $t_1, t_2$ and different $h_1, h_2$, intersect each other twice, once or never, depending on whether $\Delta> 0$, $\Delta=0$ or $\Delta<0$, respectively. For $s_1,s_2$ to be real, we need $\Delta$ to be non-negative, that is $(t_1-t_2)^2\geq (h_1^2-h_2^2)\frac{(v_1^2-v_2^2)}{v_1^2v_2^2}$. Given the two speeds $v_1,v_2$, the head point $(t_1, h_1)$ of the first radial bird, and $h_2$, the height of the head point of the second radial bird, let 
\begin{equation}
    D_\ell\equiv D_\ell(t_1, h_1, h_2):=\left\{t_2\leq t_1\,\mbox{:}\,(t_1-t_2)^2\geq (h_1^2-h_2^2)\frac{(v_1^2-v_2^2)}{v_1^2v_2^2}, \text{ i.e., } \Delta \geq 0 \right\},
    \label{eq:S2set}
\end{equation}
be the set of all possible values of $t_2$, the first coordinate of the head of the second radial bird, for which $\Delta$ is non-negative.  In case $h_1<h_2$, we have $D_\ell=(-\infty, t_1]$, but in case $h_1\geq h_2$, we have $D_\ell\subsetneq (-\infty, t_1]$, depending on the specific choice of  $t_1$. The main goal of the rest of the proof is to derive the region $D_\ell$, when $h_1\geq h_2$.

The situations where the two radial birds do not intersect or do intersect, are depicted on the first picture in Figure~\ref{figure:int_limit}. Given the speeds $v_1>v_2$, the heights of the heads $h_1\geq h_2$, and the first coordinate $t_1$ of the first radial bird, there exist $t_2$, far enough from $t_1$, and $t_2'$, close enough to $t_1$, such that $(t_1-t_2)^2\geq (h_1^2-h_2^2)\frac{(v_1^2-v_2^2)}{v_1^2v_2^2} $ and $(t_1-t_2')^2< (h_1^2-h_2^2)\frac{(v_1^2-v_2^2)}{v_1^2v_2^2}$.
In other words, $\Delta \geq 0$ for $t_2$ and $\Delta < 0$ for $t'_2$. As a consequence, $t_2\in D_\ell$, but $t'_2\notin D_\ell$. Figure~\ref{figure:int_limit} illustrates the set $D_\ell$.  

\begin{figure}[ht!]
 \centering
      \begin{subfigure}[t]{0.45\linewidth}
       \centering
      \begin{tikzpicture}[scale=0.6, every node/.style={scale=0.7}]
\pgftransformxscale{0.9}  
    \pgftransformyscale{0.9}    
   \draw[->] (-3, 0) -- (5, 0) node[right] {$t$};
    \draw[dashed] (-3, 2) -- (3, 2);
    \draw[<-] (-2.5, 1.95) -- (-2.5, 1.25)node[below]{$h_2$};
    \draw[<-] (-2.5, 0.05) -- (-2.5, 0.7);
     \draw[<-] (-2, 2.45) -- (-2, 1.55)node[below]{$h_1$};
    \draw[<-] (-2, 0.05) -- (-2, 1.1);
   \draw[blue, domain=-3.8:3.9, smooth] plot (\x, {((3/2)^2*\x*\x+2^2)^0.5});
   \draw[blue](0,2) node{$\bullet$};
   \draw[blue,dashed, domain=-2.3:5.5, smooth] plot (\x, {((3/2)^2*\x*\x- (3/2)^2*2*\x*(3/2)+(3/2)^2*(3/2)^2+2^2)^0.5});
    \draw[blue](1.5,2) node{$\bullet$};
    \draw[](1.7,1.6) node{$(t'_2,h_2)$};
    \draw[](0,1.6) node{$(t_2,h_2)$};
    \draw[red, domain=0.5:3.5, smooth] plot (\x, {(16*\x*\x-64*\x+64+2.5^2)^0.5});
    \draw[red](2,2.5) node{$\bullet$};
    \draw[](2.8,2.8) node{$(t_1,h_1)$};
    \draw[red, dashed] (-3, 2.5) -- (3, 2.5);
    \draw[](1.5,3.6) node{$(\hat s_1,\hat h_1)$};
    \draw[](3.8,5) node{$(\hat s_2,\hat h_2)$};
    \end{tikzpicture}
    \caption{There exists $t_2, t_2'<t_1$, such that the radial bird at $(t_2, h_2)$ intersects the radial bird at $(t_1, h_1)$, but the other one at $(t_2',h_2)$ does not.}
\label{subfigure:int_limit1}
    \end{subfigure}
    \hspace{0.2in}
    \centering
      \begin{subfigure}[t]{0.45\linewidth}
       \centering
      \begin{tikzpicture}[scale=0.6, every node/.style={scale=0.7}]
\pgftransformxscale{1.35}  
    \pgftransformyscale{1.35}    
    \draw[->] (-2, 0) -- (4.5, 0) node[right] {$t$};
    \draw[-] (2.18, 2.6) -- (2.18, 0);
    \draw[](2.7,0.2)node {$\hat s_1=\hat s_2$};
\draw[blue,domain=-1:3.9, smooth] plot (\x, {((3/2)^2*\x*\x- (3/2)^2*2*\x*(1.09)+(3/2)^2*(1.09)^2+2^2)^0.5});
    \draw[blue](1.1,2) node{$\bullet$};
    \draw[](1.1,1.65) node{$(t_1-t^*,h_2)$};
    \draw[blue,dashed, domain=-1:2.6, smooth] plot (\x, {((3/2)^2*\x*\x- (3/2)^2*2*\x*(-0.3)+(3/2)^2*(-0.3)^2+2^2)^0.5});
    \draw[blue](-0.3,2) node{$\bullet$};
    \draw[blue,dashed, domain=-1:3.2, smooth] plot (\x, {((3/2)^2*\x*\x- (3/2)^2*2*\x*(0.3)+(3/2)^2*(0.3)^2+2^2)^0.5});
    \draw[blue](0.3,2) node{$\bullet$};
    \draw[red, domain=1:3, smooth] plot (\x, {(16*\x*\x-64*\x+64+2.5^2)^0.5});
    \draw[red](2,2.5) node{$\bullet$};
    \draw[](1.3,2.65) node{$(t_1,h_1)$};
    \end{tikzpicture}
    \caption{There exists $t^*$ such that, two radial birds at $(t_1,h_1), (t_2, h_2)$ intersects each other twice for all $t_2<t_1-t^*$ and $t_2>t_1+t^*$.}
\label{subfigure:star-t2}
    \end{subfigure}
    \caption{Criterion for intersections with $v_1>v_2$, $h_1>h_2$, fixed $t_1$ but variable $t_2$.}   \label{figure:int_limit}
\end{figure} 
In view of the last observation, given $t_1$ and $ h_1\geq h_2$, there exists a maximum value of $t_2$ such that there is a unique intersection point of the two radial birds $C_1, C_2$, see picture~\subref{subfigure:star-t2} of Figure~\ref{figure:int_limit}. It corresponds to the value of $t_2$ such that $\Delta=0$, which is $t_1-t^*$, given $t_1$ and $ h_1\geq h_2$,  where,
\begin{equation}
t^*:=\frac{1}{v_1v_2}(h_1^2-h_2^2)^\half (v_1^2-v_2^2)^\half.
\label{eq:star-t2}
\end{equation}
Hence the set $D_\ell$ turns out to be $(-\infty, t_1-t^*]$, given $t_1$ and $ h_1\geq h_2$. Similarly we can determine the other part in the case $t_2\geq t_1$, where $D_r\equiv D_r(t_1,h_1,h_2)$ and $ D_r(t_1,h_1,h_2)=[t_1+t^*, \infty)$, for $h_1\geq h_2$ and $D_r(t_1,h_1,h_2)=[t_1, \infty)$, for $h_1< h_2$.\qed
\subsection{Proof of Lemma~\ref{lemma:in1out2}}~\label{subsection:L-in1out2}
\begin{proof}[Proof of part~(\ref{intersection})] 
By the intersection criterion in Lemma~\ref{lemma:2int} and Lemma~\ref{lemma:2int-rev}, the two radial birds $C^1_{(t_1,h_1)}$ and $C^2_{(t_2,h_2)}$ must intersect each other if $h_1\leq h_2$, for any $t_1, t_2\in \R$. We now prove the result for the case $h_1>h_2$. Without loss of generality, we prove the result for $s=0$ and $t_2>t_1>0$; one of the other case, $t_2<t_1<0$, can be proved similarly. The two more cases $t_2<0<t_1$ and $t_1<0<t_2$ can be proved as a corollary of the cases $t_2>t_1>0$ and $t_2<t_1<0$, respectively. 


Suppose $t_2>t_1>0$. Let us denote the vertical line $t=0$ as $L_0$. Since $(t_2,h_2)\notin E^{0,v_2}_u$, by Observation~\ref{observation:in-out} (for the single-speed case), the radial bird of type $2$ with head at $(t_2,h_2)$ must intersect the vertical line $L_0$ at a level higher than $u$. On the other hand, the radial bird of type $1$ at $(t_1,h_1)$ intersects the line $L_0$ exactly at level $u$, see Figure~\ref{figure:1in2out}. This implies that the two radial birds at $(t_1,h_1)$ and $(t_2,h_2)$ intersect each other twice almost surely. This completes the proof of this part.
\begin{figure}[ht!]
\begin{tikzpicture}[scale=0.7, every node/.style={scale=0.6}]
\pgftransformxscale{0.75}  
\pgftransformyscale{0.75}    
\draw[->] (-1, 0) -- (10, 0) node[right] {$t$};
\draw[red](3.65,2) node{$\bullet$};
\draw[](4.5,1.7) node{$(t_1,h_1)$};
\draw[red, domain=2:5.3, smooth] plot (\x, {(16*\x*\x-32*3.65*\x+16*3.65^2+2^2)^0.5});
    \draw[](1.9,4.4) node{$(0,u)$};
     \draw[blue, domain=1.6:10.8, smooth] plot (\x, {((3/2)^2*\x*\x- 2*(3/2)^2*6.2*\x+(3/2)^2*6.2*6.2+1.5^2)^0.5});
     \draw[blue](6.2,1.5) node{$\bullet$};
     \draw[](6.5,1) node{$(t_2,h_2)$};
     \draw[->] (2.75, 0) -- (2.75, 7)node[above]{$L_0$};
     \draw[red]  (1.8,0) arc (-180:-360:1 and 4.1);
     \draw[blue]  (5.5,0) arc (0:180:2.65 and 4.1);
    \draw[red](3.3,1) node{$E^{0, v_1}_{u}$};
    \draw[blue](1,1) node{$E^{0, v_2}_{u}$};
    \end{tikzpicture}
    \captionsetup{width=0.9\linewidth}
    \caption{The bird at $(t_1,h_1)$ intersects the line $L_0$ at level $u$, but the bird at $(t_2,h_2)$ intersects the line $L_0$ at level higher than $u$.}
\label{figure:1in2out}
\end{figure}
\end{proof}
\begin{proof}[Proof of part~(\ref{intersection_order})]
Without loss of generality, we prove the result for $s=0$. We also assume $0<t_1<t_2$, whereas the other case $t_2<t_1<0$, can be proved similarly. The two other cases $t_2<0<t_1$ and $t_1<0<t_2$ can be derived from the cases $0<t_1<t_2$ and $t_2<t_1<0$, respectively. Observe that each side of the curve of the radial birds is monotonic. Since the bird at $(t_2,h_2)$ intersects the vertical line $L_0$ at a level higher than $u$, then using monotonicity and the fact that $0\leq t_1$, the first intersection must be on the left of $(0,u)$ and the second intersection on the right of $(0,u)$, i.e., $\hat s_1\leq 0\leq \hat s_2$. This completes the proof.
\end{proof}
\subsection{Proof of Lemma~\ref{lemma:int1mid2}}~\label{subsection:L-int1mid2}
Assume that $(t_0,h_0)$ is of type $1$ and $\hat s_1\leq \tilde s\leq \hat s_2$. Let $L_{\hat s_1}$ and $L_{\hat s_2}$ be the vertical lines $t=\hat s_1$ and $t=\hat s_2$, respectively. Using the monotonicity of the two sides of the radial bird $C^1_{(t_1,h_1)}$, the condition $\hat s_1\leq \tilde s\leq \hat s_2$ is equivalent to the fact that the radial bird at $(t_0,h_0)$ intersects the line $L_{\hat s_1}$ at a height above $\hat h_1$ and the line $L_{\hat s_2}$ at a height below $\hat h_2$, see Observation~\ref{observation:in-out} for the single speed case. This is possible if and only is $(t_0,h_0)$ lies inside $E^{\hat s_2,v_1}_{\hat h_2}$ and outside $E^{\hat s_1,v_1}_{\hat h_1}$, see picture~(\subref{subfigure:3rdb1}) of Figure~\ref{figure:mixed_bird1mid2}. In fact, $(t_0,h_0)$ cannot lie inside $E^{\hat s_1,v_1}_{\hat h_1}$, by assumption.
\begin{figure}[ht!]
 \centering
      \begin{subfigure}[t]{0.45\linewidth}
      \centering
      \begin{tikzpicture}[scale=0.7, every node/.style={scale=0.7}]
\pgftransformxscale{0.85}  
    \pgftransformyscale{0.85}    
     \draw[->] (-2, 0) -- (7, 0) node[right] {$t$};
    \draw[<-] (1.55, 6.3)--(1.55,0)node[below] {$L_{\hat s_1}$};
    \draw[<-] (3.1, 6.3)--(3.1, 0) node[below] {$L_{\hat s_2}$};
   \draw[blue, domain=-2.2:4.2, smooth] plot (\x, {((3/2)^2*\x*\x+2^2)^0.5});
    \draw[blue](0,2) node{$\bullet$};
    \draw[](-0.7,1.8) node{$(t_2,h_2)$};
    \draw[red, domain=0.8:3.4, smooth] plot (\x, {(16*\x*\x-64*\x+64+2.5^2)^0.5});
    \draw[red](2,2.5) node{$\bullet$};
    \draw[](2.6,2.3) node{$(t_1,h_1)$};
    \draw[](0.9,3.3) node{$(\hat s_1,\hat h_1)$};
    \draw[](2.3,5.5) node{$(\hat s_2,\hat h_2)$};
     \draw[red]  (2.35,0) arc (0:180:0.8 and 3.07);
     \draw[red]  (1.85,0) arc (-180:-360:1.25 and 5.05);
     %
    \draw[red](0.3,0.6) node{$E^{\hat s_1, v_1}_{\hat h_1}$};
     \draw[red](5,0.6) node{$E^{\hat s_2,v_1}_{\hat h_2}$};
    \end{tikzpicture}
    \caption{Any radial bird of type $1$, whose head point is in $E^{\hat s_2,v_1}_{\hat h_2}\setminus E^{\hat s_1,v_1}_{\hat h_1}$, intersects the vertical line $L_{\hat s_1}$ above $\hat h_1$ and the vertical line $L_{\hat s_2}$ below $\hat h_2$.}
\label{subfigure:3rdb1}
    \end{subfigure}
    \hspace{0.2in}
    \centering
      \begin{subfigure}[t]{0.45\linewidth}
      \centering
      \begin{tikzpicture}[scale=0.7, every node/.style={scale=0.7}]
\pgftransformxscale{0.85}  
\pgftransformyscale{0.85}    
 \draw[->] (-2, 0) -- (7, 0) node[right] {$t$};
    \draw[<-] (1.55, 6.3)--(1.55,0)node[below] {$L_{\hat s_1}$};
    \draw[<-] (3.1, 6.3)--(3.1, 0) node[below] {$L_{\hat s_2}$};
   \draw[blue, domain=-2.2:4.2, smooth] plot (\x, {((3/2)^2*\x*\x+2^2)^0.5});
    \draw[blue](0,2) node{$\bullet$};
    \draw[](-0.7,1.8) node{$(t_2,h_2)$};
    \draw[red, domain=0.8:3.4, smooth] plot (\x, {(16*\x*\x-64*\x+64+2.5^2)^0.5});
    \draw[red](2,2.5) node{$\bullet$};
    \draw[](2.6,2.3) node{$(t_1,h_1)$};
    \draw[](0.9,3.3) node{$(\hat s_1,\hat h_1)$};
    \draw[](2.3,5.5) node{$(\hat s_2,\hat h_2)$};
     \draw[blue]  (-0.47,0) arc (-180:-360:2 and 3.07);
     \draw[blue]  (6.5,0) arc (0:180:3.4 and 5.05);
    \draw[blue](-1.2,0.6) node{$E^{\hat s_1, v_2}_{\hat h_1}$};
    \draw[blue](5.1,1.6) node{$E^{\hat s_2,v_2}_{\hat h_2}$};
    \end{tikzpicture}
    \caption{Any radial bird of type $2$, whose head point is in $E^{\hat s_2,v_2}_{\hat h_2}\setminus E^{\hat s_1,v_2}_{\hat h_1}$, intersects the vertical line $L_{\hat s_1}$ above $\hat h_1$ and the vertical line $L_{\hat s_2}$ below $\hat h_2$.}
\label{subfigure:3rdb2}
    \end{subfigure}
    \captionsetup{width=0.9\linewidth}
    \caption{Scenario described in Lemma~\ref{lemma:int1mid2}.}
\label{figure:mixed_bird1mid2}
\end{figure}

Similarly, for $(t_0,h_0)$ of type $2$, $\hat s_1\leq \tilde s_2\leq \hat s_2$ if and only if $(t_0,h_0)\in E^{\hat s_2,v_2}_{\hat h_2}\setminus E^{\hat s_1,v_2}_{\hat h_1}$, see picture~(\subref{subfigure:3rdb2}) of Figure~\ref{figure:mixed_bird1mid2}. This is true because the radial bird of type $2$ at $(t_0,h_0)\in E^{\hat s_2,v_2}_{\hat h_2}\setminus E^{\hat s_1,v_2}_{\hat h_1}$ intersects the line $L_{\hat s_1}$ at a height above $\hat h_1$ and the line $L_{\hat s_2}$ at a height below $\hat h_2$.\qed 
\subsection{Proof of Lemma~\ref{lemma:Palm-Lcal-ij}, $l=1,r=2$ and $k=1,2$}~\label{subsection:l1r2k} Here we derive the result for $l=2,r=1$ and $k=1,2$, i.e., for $\E^0_{\Vcal^{(k)}_{1,2}} [f(\Hcal)]$, similarly to the first case.  Suppose  $(T^2_i,H^2_i)$ and $(T^1_j,H^1_j)$ in the support of $\Hcal^2, \Hcal^1$, respectively, such that $T_j^1 \in D'_\ell(T_i^2, H_i^2, H_j^1)$. The $k$-th intersection point of the two radial birds with these heads is
$(\hat{S}^{(k)}_{1,2},\hat{H}^{(k)}_{1,2})$, in short written as $(\hat{S}^{(k)},\hat{H}^{(k)})$.
By (\ref{eq:VR-ijk}) and the definition of $\P_{\Rcal^{(k)}_{1,2}}^0$ we have
\begin{align}
\E^0_{\Vcal^{(k)}_{1,2}} [f (\Hcal)]
=\E^0_{\Rcal^{(k)}_{1,2}} \left[f (\theta_{\beta^{(k)}_{1,2}(0)}(\Hcal))\right] 
& = \frac 1 {\Lambda_{1,2}^{(k)}}
\mathbb{E} \left[ \sum_{T^2_i\in \Rcal^{(k)}_{1,2}\,\mbox{:}\, 0\le T^2_i \le 1} f \left(\theta_{\beta^{(k)}_{1,2}(0)} \circ \theta_{T^2_i}(\Hcal)\right)\right]\nn\\
&=\frac 1 {\Lambda_{1,2}^{(k)}}
\mathbb{E} \left[ \sum_{T^2_i\in \Rcal^{(k)}_{1,2}\,\mbox{:}\, 0\le T^2_i \le 1} f \left(\theta_{T^2_i+\beta^{(k)}_{1,2}(T^2_i)}(\Hcal)\right) \right],
\label{eq:fH121}
\end{align}
using the fact that $\theta_{\beta^{(k)}_{1,2}(0)} \circ \theta_{T^2_i}=\theta_{T^2_i+\beta^{(k)}_{1,2}(T^2_i)}$. The last expectation in (\ref{eq:fH121}) is equal to

\begin{align}
\lefteqn{\mathbb{E} \left[ \sum_{T^2_i\in \Rcal^{(k)}_{1,2}\,\mbox{:}\, 0\le T^2_i \le 1} f \left(\theta_{T^2_i+\beta^{(k)}_{1,2}(T^2_i)}(\Hcal)\right) \right]}\nn\\
&=\mathbb{E} \left[ \sum_{(T^2_i,H^2_i)\in \Hcal^2\,\mbox{:}\, 0\le T^2_i \le 1} 
\one_{\exists (T^1_j,H^1_j)\in \Hcal^1 \,\mbox{:}\, T^1_j\in D'_\ell(T^2_i,H^2_i,H^1_j)}  f \left(\theta_{T^2_i+\beta^{(k)}_{1,2}(T^2_i)}(\Hcal)\right)  \prod_{l=1,2}\one_{\Hcal^l\left(E^{\hat{S}^{(k)},v_l}_{\hat{H}^{(k)}}\right)=0}\right]\nn\\
& = \mathbb{E} \left[ \sum_{(T^2_i,H^2_i)\in \Hcal^2\,\mbox{:}\, 0\le T^2_i \le 1}\,\sum_{(T^1_j,H^1_j)\in \Hcal^1 \,\mbox{:}\, T^1_j\in D'_\ell(T^2_i,H^2_i,H^1_j)}
f \left(\theta_{T^2_i+\beta^{(k)}_{1,2}(T^2_i)}(\Hcal)\right) \prod_{l=1,2}\one_{\Hcal^l\left(E^{\hat{S}^{(k)},v_l}_{\hat{H}^{(k)}}\right)=0}
\right]\nn \\
& = 
\mathbb{E} \left[ \sum_{(T^2_i,H^2_i)\in \Hcal^2 \,\mbox{:}\, 0\le T^2_i \le 1} \;
\sum_{(T^1_j,H^1_j)\in \Hcal^1 \,\mbox{:}\, T^1_j\in D'_\ell(T^2_i,H^2_i,H^1_j)}
f \left(\theta_{\hat{S}^{(k)}}(\Hcal)\right) \prod_{l=1,2}\one_{\Hcal^l\left(E^{\hat{S}^{(k)},v_l}_{\hat{H}^{(k)}}\right)=0}
\right], 
\label{eq:fH122}
\end{align}
where $T^2_i+\beta^{(k)}_{1,2}(T^2_i)=\hat S^{(k)}$. Applying the Campbell-Mecke formula for the factorial power of order 2 based on the two-point Palm probability measure $\P^{(t, h), (t',h')}_{\Hcal}=\P^{(t, h)}_{\Hcal^2}\otimes \P^{(t',h')}_{\Hcal^1}$, the expectation in (\ref{eq:fH122}) evaluates to
\begin{align}
\lefteqn{4\la_1\la_2 v_1 v_2	\int_{0}^1 \int_{0}^\infty  \int_{0}^\infty \int_{D'_\ell(t,h,h')} 
\mathbb{E}_\Hcal\Bigg[ f \left(\theta_{\hat{s}_k(t,h,t',h')}(\Hcal^2+\delta_{(t,h)}+ \Hcal^1+\delta_{(t',h')})\right)}\nn\\
&\hspace{3.3in} \times \prod_{l=1,2}\one_{\Hcal^l\left(E^{\hat{s}_k(t,h,t',h'),v_l}_{\hat{h}_k(t,h,t',h')}\right)=0}
\Bigg] {\rm d}t'  {\rm d}h' {\rm d}h {\rm d}t,
\label{eq:fH124}
\end{align}
similarly to (\ref{eq:Last-Hij}), where the set $D'_\ell(t,h,h')$ is defined as in (\ref{eq:DlDr-rev}). The set $D'_\ell(t,h,h')=(-\infty, t]$ when $h<h'$ but $D'_\ell(t,h,h')= (-\infty, t-t^*]$, when $h\geq h'$, where $t^*$ is as defined in (\ref{eq:tstar}). This leads to a simplification of the multiple integral in (\ref{eq:fH124}) to
\begin{align}
\lefteqn{\int_{0}^1 \int_{0}^\infty  \int_{0}^h \int_{-\infty}^{t-t^*} 
\mathbb{E}_\Hcal\Bigg[ f \left(\theta_{\hat{s}_k(t,h,t',h')}(\Hcal^2+\delta_{(t,h)}+ \Hcal^1+\delta_{(t',h')})\right) \prod_{l=1,2}\one_{\Hcal^l\left(E^{\hat{s}_k(t,h,t',h'),v_l}_{\hat{h}_k(t,h,t',h')}\right)=0}
\Bigg] {\rm d}t'  {\rm d}h' {\rm d}h {\rm d}t}\nn\\ 
&+ \int_{0}^1 \int_{0}^\infty  \int_h^\infty \int_{-\infty}^{t} 
\mathbb{E}_\Hcal\Bigg[ f \left(\theta_{\hat{s}_k(t,h,t',h')}(\Hcal^2+\delta_{(t,h)}+ \Hcal^1+\delta_{(t',h')})\right)  \prod_{l=1,2}\one_{\Hcal^l\left(E^{\hat{s}_k(t,h,t',h'),v_l}_{\hat{h}_k(t,h,t',h')}\right)=0}
\Bigg] {\rm d}t' {\rm d}h' {\rm d}h {\rm d}t,
\label{eq:fH125a}
\end{align}
where $t^*=\frac{1}{v_1v_2} (h^2-h'^2)^\half (v_1^2-v_2^2)^\half $. 

By making the change of variable $t-t'$ to $t'$ we obtain from (\ref{eq:fH125a}) that
\begin{align}
\lefteqn{\int_{0}^\infty  \int_{0}^h \int_{t^*}^\infty 
\mathbb{E}_\Hcal\Bigg[ f \left(\theta_{\hat{s}_k(0,h,-t',h')}(\Hcal^2+\delta_{(0,h)}+ \Hcal^1+\delta_{(-t',h')})\right) \prod_{l=1,2}\one_{\Hcal^l\left(E^{\hat{s}_k(0,h,-t',h'),v_l}_{\hat{h}_k(0,h,-t',h')}\right)=0}
\Bigg] {\rm d}t'  {\rm d}h' {\rm d}h}\nn \nn\\ 
&+  \int_{0}^\infty  \int_h^\infty \int_0^{\infty}
\mathbb{E}_\Hcal\Bigg[ f \left(\theta_{\hat{s}_k(0,h,-t',h')}(\Hcal^2+\delta_{(0,h)}+ \Hcal^1+\delta_{(-t',h')})\right) \prod_{l=1,2}\one_{\Hcal^l\left(E^{\hat{s}_k(0,h,-t',h'),v_l}_{\hat{h}_k(0,h,-t',h')}\right)=0}
\Bigg] {\rm d}t'  {\rm d}h' {\rm d}h.
\label{eq:fH125}
\end{align}
The result is obtained in this case by combining the three equations (\ref{eq:fH125}), (\ref{eq:fH124}) and (\ref{eq:fH122}). \qed
\subsection{Proof of Lemma~\ref{lemma:betaLR}}~\label{subsection:L-betaLR}
For the equality in (\ref{eq:betaLR1}), we can write similarly to (\ref{eq:Last-Hijjj}), 
\begin{align}
\E^0_{\Lcal^{(1)}_{1,2}} \left[f (\theta_{\tilde \beta^{(1)}_{1,2}(0)}(\Hcal))\right]  
&= \frac 1 {\Lambda_{1,2}^{(1)}}
\mathbb{E} \left[ \sum_{T^1_i\in \Lcal^{(1)}_{1,2}\,\mbox{:}\, 0\le T^1_i \le 1} f \left(\theta_{\tilde\beta^{(1)}_{1,2}(0)} \circ \theta_{T^1_i}(\Hcal)\right)\right]\nn\\
&= \frac 1 {\Lambda_{1,2}^{(1)}} \mathbb{E} \left[ \sum_{T^1_i\in \Lcal^{(1)}_{1,2} \,\mbox{:}\, 0\le T^1_i \le 1} f \left(\theta_{T^1_i+\tilde\beta^{(1)}_{1,2}(T^1_i)}(\Hcal)\right) \right],
\label{eq:eqLR1}
\end{align}
using the fact that $\theta_{\tilde\beta^{(1)}_{1,2}(0)} \circ \theta_{T^1_i}=\theta_{T^1_i+\tilde\beta^{(1)}_{1,2}(T^1_i)}$. Thus from (\ref{eq:eqLR1}) we have
\begin{align}
\lefteqn{\!\!\mathbb{E} \left[ \sum_{T^1_i\in \Lcal^{(1)}_{1,2} \,\mbox{:}\, 0\le T^1_i \le 1} f \left(\theta_{T^1_i+\tilde\beta^{(1)}_{1,2}(T^1_i)}(\Hcal)\right) \right]}\nonumber \\
& = \mathbb{E} \left[ \sum_{(T^1_i,H^1_i)\in \Hcal^1 \,\mbox{:}\, 0\le T^1_i \le 1} \!\!\!
	\one_{\exists (T^2_j,H^2_j)\in \Hcal^2 \,\mbox{:}\, T^2_j\in D_r(T^1_i,H^1_i,H^2_j)}
	\;  f \left(\theta_{T^1_i+\tilde\beta^{(1)}_{1,2}(T^1_i)}(\Hcal)\right) \prod_{l=1,2}\one_{\Hcal^l\left(E^{\hat{S}^{(1)},v_l}_{\hat{H}^{(1)}}\right)=0}\right]\nn \\
& = 	\mathbb{E} \left[ \sum_{(T^1_i,H^1_i)\in \Hcal^1 \,\mbox{:}\, 0\le T^1_i \le 1}\, 
\sum_{(T^2_j,H^2_j)\in \Hcal^2 \,\mbox{:}\, T^2_j\in D_r(T^1_i,H^1_i,H^2_j)}
    \!\!\!\!\!\!f \left(\theta_{T^1_i+\tilde\beta^{(1)}_{1,2}(T^1_i)}(\Hcal)\right) \prod_{l=1,2}\one_{\Hcal^l\left(E^{\hat{S}^{(1)},v_l}_{\hat{H}^{(1)}}\right)=0}
	  \right].
   \label{eq:eqLR2x}
\end{align}
Using the fact that $T^1_i+\tilde\beta^{(1)}_{1,2}(T^1_i)=\hat S^{(1)}$ we have (\ref{eq:eqLR2x}) that
\begin{align}
\lefteqn{\mathbb{E} \left[ \sum_{T^1_i\in \Lcal^{(1)}_{1,2} \,\mbox{:}\, 0\le T^1_i \le 1} f \left(\theta_{T^1_i+\tilde\beta^{(1)}_{1,2}(T^1_i)}(\Hcal)\right) \right]}\nn \\
& = \mathbb{E} \left[ \sum_{(T^1_i,H^1_i)\in \Hcal^1 \,\mbox{:}\,  0\le T^1_i \le 1} \;
\sum_{(T^2_j,H^2_j)\in \Hcal^2 \,\mbox{:}\, T^2_j\in D_r(T^1_i,H^1_i,H^2_j)}
    f \left(\theta_{\hat{S}^{(1)}}(\Hcal)\right) \prod_{l=1,2}\one_{\Hcal^l\left(E^{\hat{S}^{(1)},v_l}_{\hat{H}^{(1)}}\right)=0}
	  \right]\nn\\
& = \mathbb{E} \left[ \sum_{(T^1_i,H^1_i)\in \Hcal^1 \,\mbox{:}\,  0\le T^1_i \le 1} \;\sum_{(T^2_j,H^2_j)\in \Hcal^2 \,\mbox{:}\, T^2_j\in D_\ell(T^1_i,H^1_i,H^2_j)}f \left(\theta_{\tilde{S}^{(2)}}(\Hcal)\right) \prod_{l=1,2}\one_{\Hcal^l\left(E^{\tilde{S}^{(2)},v_l}_{\tilde{H}^{(2)}}\right)=0}\right]\nn\\
&\stackrel{(\ref{eq:Last-Hijj}), (\ref{eq:Last-Hijjj})}{=} \Lambda_{2,1}^{(2)}\, \E^0_{\Rcal^{(2)}_{2,1}} \left[f (\theta_{\beta^{(2)}_{2,1}(0)}(\Hcal))\right],
\label{eq:eqLR2}
\end{align}
where $(\tilde{S}^{(2)},\tilde{H}^{(2)})$ is the intersection of type $\binom{2}{2,1}$. In the second step in (\ref{eq:eqLR2}), after counting backwards in time, we have used the symmetry of the function $f$, i.e., $f(\Hcal)=f(-\Hcal)$. As seen before in part~(\ref{MH2p2}) of Theorem~\ref{theorem:HO2speed-main}, we have $\Lambda_{1,2}^{(1)}=\Lambda_{2,1}^{(2)}$. We obtain the result using the last fact and combining (\ref{eq:eqLR1}) and (\ref{eq:eqLR2}). The equality in (\ref{eq:betaLR2}) can be proved similarly.\qed
\subsection{Proof of Proposition~\ref{prop:monotonic-sets}~part~(\ref{monotonic1})}\label{subsection:inc-wing2}
We first provide the justification for the claim for the three type of handovers at time $0$, for $\t_n=1$, i.e., when the station that takes over at time $0$ is of type $1$. This corresponds to intersections of types $1$ or $\bbinom{1}{1,1}$, $\binom{1}{2,1}$ or $\bbinom{2}{2,1}$ and $\binom{1}{1,2}$ or $\bbinom{1}{2,1}$, respectively. Let us start the proof of part~(\ref{monotonic1}) for these cases. 

\begin{proof}[Proof of Proposition~\ref{prop:monotonic-sets},~part~(\ref{monotonic1}), typical handover type $\binom{1}{2,1}$ or equivalently $\bbinom{2}{2,1}$]
Suppose the handover happens at a distance $\hat h_0$, which is given by the first intersection of two radial birds with head points $(t_n,h_n)=(t_1,h_1)$ of type $\t_n=1$ and $(t_p,h_p)=(t_2,h_2)$ of type $\t_p=2$, such that $t_1\geq t_2$. By construction, $(t_1,h_1)$ lies on $\partial E^{0,v_1}_{\hat h_0}$ and $(t_2,h_2)$ lies on $\partial E^{0,v_2}_{\hat h_0}$. In fact, for any $(t_1,h_1)$ lying on $\partial E^{0,v_1}_{\hat h_0}$ and $(t_2,h_2)$ lying on $\partial E^{0,v_2}_{\hat h_0}$, we get the same handover. Given $\hat h_0, t_1, h_1$, define $\hat h_1(t):=(v_1^2(t-t_1)^2+h_1^2)^\half$ and consider the two collections $\left\{E^{t,v_1}_{\hat h_1(t)}\right\}_{t\geq 0}$ and $\left\{E^{t,v_2}_{\hat h_1(t)}\right\}_{t\geq 0}$, of upper half ellipses $E^{t,v_1}_{\hat h_1(t)}$ and $E^{t,v_2}_{\hat h_1(t)}$ with their top point lying on the radial bird $C^1_{(t_1,h_1)}$.
\begin{case}[\textbf{$\left\{E^{t,v_1}_{\hat h_1(t)}\setminus E^{0,v_1}_{\hat h_0}\right\}_{t\geq 0}$}.]
It can be proved, using the techniques applied for the single-speed case, that the collection $\left\{E^{t,v_1}_{\hat h_1(t)}\setminus E^{0,v_1}_{\hat h_0}\right\}_{t\geq 0}$ is non-decreasing in $t$, since the boundaries $\partial E^{t, v_1}_{\hat h_1(t)}$ pass through $(t_1,h_1)$, for all $t$, see picture~\subref{subfigure:mH-FHoT1} of Figure~\ref{figure:mH-FHoT}. 
\end{case}
\begin{figure}[ht!]
 \centering
      \begin{subfigure}[t]{0.47\linewidth}
       \centering
      \begin{tikzpicture}[line width=0.55pt, scale=0.7, every node/.style={scale=0.7}]
\pgftransformxscale{1.2}  
\pgftransformyscale{1.2}  
\draw[->] (-1, 0) -- (6, 0) node[right] {$t$};
\draw[blue, domain=-1.2:3.3, smooth] plot (\x, {((3/2)^2*\x*\x+2^2)^0.5});
\draw[blue](0,2) node{$\bullet$};
    \draw[](-0.7,1.8) node{$(t_2,h_2)$};
    \draw[red, domain=0.6:3.5, smooth] plot (\x, {(9*\x*\x-2*9*2*\x+9*2*2+2.5^2)^0.5});
    \draw[red](2,2.5) node{$\bullet$};
    \draw[](2.8,2.3) node{$(t_1,h_1)$};
    \draw[](0.9,3.3) node{$(0,\hat h_0)$};
     \draw[red]  (2.48,0) arc (0:180:1 and 2.98);
     \draw[red](-0.1,0.6) node{$E^{0, v_1}_{\hat h_0}$};
     \draw[teal]  (0.7,0) arc (-180:-360:0.93 and 2.75);
     \draw[magenta]  (1.65,0) arc (-180:-360:1.175 and 3.5);
     \draw[violet]  (1.75,0) arc (-180:-360:1.544 and 4.6);
\draw[->] (2, 0) -- (2, 5);
    \end{tikzpicture}
    \caption{The collection $\left\{E^{t,v_1}_{\hat h_1(t)}\setminus E^{0,v_1}_{\hat h_0}\right\}_{t\geq 0}$ is non-decreasing, since all the boundaries pass through $(t_1,h_1)$.}
\label{subfigure:mH-FHoT1}
    \end{subfigure}
    \hspace{0.1in}
    \centering
      \begin{subfigure}[t]{0.47\linewidth}
       \centering
      \begin{tikzpicture}[line width=0.55pt, scale=0.7, every node/.style={scale=0.7}]
\pgftransformxscale{1.2}  
\pgftransformyscale{1.2}    
\draw[->] (-1.4, 0) -- (5.5, 0) node[right] {$t$};
   \draw[blue, domain=-1.2:3.3, smooth] plot (\x, {((3/2)^2*\x*\x+2^2)^0.5});
    \draw[blue](0,2) node{$\bullet$};
    \draw[](-0.7,1.8) node{$(t_2,h_2)$};
    \draw[red, domain=1:4, smooth] plot (\x, {(9*\x*\x-2*9*2.5*\x+9*2.5*2.5+2.7^2)^0.5});
    \draw[red](2.5,2.7) node{$\bullet$};
    \draw[](2.6,1.8) node{$(t_1,h_1)$};
    \draw[](0.9,3.3) node{$(0,\hat h_0)$};
     \draw[blue]  (-0.45,0) arc (-180:-360:2.28 and 3.38);
    \draw[blue](-1,0.6) node{$E^{0, v_2}_{\hat h_0}$};
     \draw[teal]  (-0.1,0) arc (-180:-360:2.092 and 3.1);
     \draw[teal] (2, 0) -- (2, 3.1);
     \draw[]  (0.45,0) arc (-180:-360:1.856 and 2.75);
     \draw[] (2.3, 0) -- (2.3, 2.75);
    \end{tikzpicture}
    \caption{Existence of $0\leq\tilde t\leq t_1$, given $(t_1,h_1)$. The collection $\left\{E^{t,v_2}_{\hat h_1(t)}\setminus E^{0,v_2}_{\hat h_0}\right\}_{t\geq 0}$ is non-decreasing, for all $t\in [\tilde t,t_1]$.}
    \label{subfigure:mH-FHoT2}
    \end{subfigure}
    \caption{Monotonicity property of the collection $\left\{E^{t,v_l}_{\hat h_1(t)}\setminus E^{0,v_l}_{\hat h_0}\right\}_{t\geq 0}$.}
\label{figure:mH-FHoT}
\end{figure}
\begin{case}[\textbf{$\left\{E^{t,v_2}_{\hat h_1(t)}\setminus E^{0,v_2}_{\hat h_0}\right\}_{t\geq 0}$}.] We prove this case in two parts. In the first part, for $t\in [t_1, \infty)$, using the {\em increasing wing property}, as described in Remark~\ref{remark:inc-arm2}, $\left\{E^{t,v_2}_{\hat h_1(t)}\setminus E^{0,v_2}_{\hat h_0}\right\}_{t\geq t_1}$ is non-decreasing. 

We consider now the part where $t\in [0,t_1]$, which corresponds to the decreasing wing of the radial bird $C^1_{(t_1,h_1)}$. We first claim that, given $(t_1,h_1)$, there exists $\tilde t>0$ such that, for all $t<\tilde t$, $E^{t,v_2}_{\hat h(t)}\subset E^{0,v_2}_{\hat h_0}$ and for all $t\geq \tilde t$, $E^{t,v_2}_{\hat h(t)}\cap \left(E^{0,v_2}_{\hat h_0}\right)^c\neq \emptyset$. Indeed, the critical value $\tilde t$ is given by the non-trivial solution of the equation $t+\hat h_t/v_2= \hat h_0/v_2$, which is either $0$ or $\frac{2}{v_1^2-v_2^2}(v_1^2t_1-v_2\hat h_0)$. Define,
\begin{equation}
\tilde t=
\begin{cases}
    0, & \text{ if } v_1^2t_1\leq v_2\hat h_0,\\
    \frac{2}{v_1^2-v_2^2}(v_1^2t_1-v_2\hat h_0), &\text{ if } v_1^2t_1> v_2\hat h_0.
\end{cases}
\label{eq:tilde-t}
\end{equation}
Given $v_1,v_2, \hat h_0$, define
\begin{equation}
t^*:=\begin{cases}
           0, & \text{ if } v_1^2t_1\leq v_2\hat h_0,\\
           \frac{2v_2\hat h_0}{v_1^2+v_2^2}, &\text{ if } v_1^2t_1> v_2\hat h_0.
    \end{cases}
\end{equation}
As a side remark, observe that $t^*$ is given by an intersection of the line segment joining the two points $(0,\hat h_0)$ and $(\hat h_0/v_2, 0)$ and $\partial E_{\hat h_0}^{0,v_1}$, which are $(0,\hat h_0)$ and  $\left(\frac{2v_2\hat h_0}{v_1^2+v_2^2},\frac{\hat h_0(v_1^2-v_2^2)}{v_1^2+v_2^2}\right)$.
\begin{subcase}[$v_1^2t_1> v_2\hat h_0$.] In this case $\tilde t=\frac{2}{v_1^2-v_2^2}(v_1^2t_1-v_2\hat h_0)$ and $t^*=\frac{2v_2\hat h_0}{v_1^2+v_2^2}$.  One can show from (\ref{eq:tilde-t}) that, for all $(t_1,h_1)\in \partial E_{\hat h_0}^{0,v_1}$ such that, for all $t_1\in [0, t^*]$, we have $\tilde t\leq t_1$ and for all $t_1\in (t^*,\hat h_0/v_1]$, we have $\tilde t> t_1$. Indeed, if $t_1\in (t^*,\hat h_0/v_1]$, then there does not exist any radial bird of type $2$ with head point lying outside $E^{0, v_2}_{\hat h_0}$ that intersects the radial bird $C^1_{(t_1,h_1)}$, twice on its decreasing wing. We perform further analysis by dividing in two cases: $(t_1,h_1)$ being such that $t_1\in (t^*,\hat h_0/v_1]$ or  $t_1\in [0,t^*]$.

\begin{subsubcase}{$t_1\in (t^*,\hat h_0/v_1]$.} In this case, by definition of $t^*$, we have $\tilde t>t_1$ and by definition of $\tilde t$, $E^{t,v_2}_{\hat h(t)}\subset E^{0,v_2}_{\hat h_0}$, for all $t<\tilde t$ and in particular for all $t<t_1$. 
Also, since $\tilde t>t_1$, for $t_1\in (t^*,\hat h_0/v_1]$, $\tilde t$ corresponds to the increasing wing of the bird $C^1_{(t_1,h_1)}$. Hence, using the {\em increasing wing property}, as described in Remark~\ref{remark:inc-arm2}, the collection of open sets $\left\{E^{t,v_2}_{\hat h_1(t)}\setminus E^{0,v_2}_{\hat h_0}\right\}_{t\geq t_1}$ is non-decreasing, and so is $\left\{E^{t,v_2}_{\hat h_1(t)}\setminus E^{0,v_2}_{\hat h_0}\right\}_{t\geq \tilde t}$ in particular. Also by definition of $\tilde t$, $E^{t,v_2}_{\hat h_1(t)}\setminus E^{0,v_2}_{\hat h_0}=\emptyset$, for all $t\in [0, \tilde t]$. As a whole, the collection $\left\{E^{t,v_2}_{\hat h_1(t)}\setminus E^{0,v_2}_{\hat h_0}\right\}_{t\geq 0}$ is non-decreasing.
\end{subsubcase}

\begin{subsubcase}{$t_1\in [0,t^*]$.}\label{subsubcase:t_1less}  Here again, by definition of $t^*$, we have $\tilde t<t_1$. We are interested in understanding the behavior of the collection $\left\{ E^{t,v_2}_{\hat h_1(t)}\setminus E^{0,v_2}_{\hat h_0}\right\}_{t\in [\tilde t, t_1]}$. It only concerns the interval $[\tilde t, t_1]$, because $\left\{ E^{t,v_2}_{\hat h_1(t)}\setminus E^{0,v_2}_{\hat h_0}\right\}_{t\geq t_1}$ is anyway non-decreasing using the {\em increasing wing property}, as in Remark~\ref{remark:inc-arm2} and $E^{t,v_2}_{\hat h_1(t)}\setminus E^{0,v_2}_{\hat h_0}=\emptyset$, for $t\in [0, \tilde t]$. Let us present a general argument for the monotonicity property of interest for the interval $[\tilde t, t_1]$. 

Let us take a time point $\t\in [\tilde t, t_1)$ and choose $\eps>0$ arbitrarily small. The boundaries of $\partial E^{\tau,v_2}_{\hat h_1(\tau)}$ and $\partial E^{\tau+\eps,v_2}_{\hat h_1(\tau+\eps)}$ are given by the equations
\begin{equation}
    v_2^2(t-\tau)^2+h^2=\hat h_1(\tau)^2,\,\,
v_2^2(t-\tau-\eps)^2+h^2=\hat h_1(\tau+\eps)^2,
\label{eq:e-tau1-tau2}
\end{equation}
respectively, where 
\begin{equation}
\hat h_1(\t)= \left(v_1^2(\t-t_1)^2+h_1^2\right)^\half,\,\, \hat h_1(\t+\eps)= \left(v_1^2(\t+\eps-t_1)^2+h_1^2\right)^\half.
\label{eq:htau-1-2}
\end{equation}
Suppose $(\bar t, \bar h)\equiv (\bar t(\t,\eps), \bar h(\t,\eps))\in \mathbb H^+$ is an intersection point, if it exists, of the two ellipses in (\ref{eq:e-tau1-tau2}). We have the following proposition: 
\begin{proposition}
For any $\t$ and $\eps>0$, with $\tilde t \leq \t< t_1$, the following holds true:
\begin{enumerate}[(a).]
\item \label{existence} $(\bar t, \bar h)$ exists and $(\bar t, \bar h)\in E^{0,v_2}_{\hat h_0}$,
\item \label{t+eps} $E^{\t,v_2}_{\hat h_1(\t)}\setminus E^{0,v_2}_{\hat h_0}\subset E^{\t+\eps,v_2}_{\hat h_1(\t+\eps)}\setminus E^{0,v_2}_{\hat h_0}$.
\end{enumerate}
\label{prop:tau1-tau2}
\end{proposition}
\begin{figure}[ht!]
    \begin{tikzpicture}[line width=0.6pt, scale=0.8, every node/.style={scale=0.75}]
    \pgftransformxscale{1.3}  
    \pgftransformyscale{1.3}
    \draw[->] (-1.4, 0) -- (5.5, 0) node[right] {$t$};
   \draw[blue, domain=-1.2:2.5, smooth] plot (\x, {((3/2)^2*\x*\x+2^2)^0.5});
    \draw[blue](0,2) node{$\bullet$};
    \draw[](-0.7,1.8) node{$(t_2,h_2)$};
    \draw[red, domain=1.5:4, smooth] plot (\x, {(2.3^2*\x*\x-2*2.3^2*2.8*\x+2.3^2*2.8*2.8+2.4^2)^0.5});
    \draw[red](2.8,2.4) node{$\bullet$};
    \draw[](2.8,2.8) node{$(t_1,h_1)$};
    \draw[](1.2,3.5) node{$(0,\hat h_0)$};
     \draw[blue]  (-0.45,0) arc (-180:-360:2.28 and 3.35);
     \draw[blue](-0.8,0.8) node[below]{$E^{0,v_2}_{\hat h_0}$};
     \draw[]  (0.65,0) arc (-180:-360:1.757 and 2.58);
     \draw[] (2.4,2.58) -- (2.4,0);
     \draw[](2.3,0) node[below]{$\t$};
     \draw[](0.3,0.8) node[below]{$E^{\t,v_2}_{\hat h_1(\t)}$};
     \draw[magenta] (2.6, 2.45) -- (2.6,0);
     \draw[](2.8,0)node[below]{$\t+\eps$};
     \draw[magenta](1.5,0.8) node[below]{$E^{\t+\eps,v_2}_{\hat h_1(\t+\eps)}$};
     \draw[magenta]  (0.95,0) arc (-180:-360:1.669 and 2.45);
     \draw[](3.2,1.4) node{$(\bar t,\bar h)$};
      \draw[->] (3.2, 1.6) -- (3.2, 2.25);
\end{tikzpicture}
\captionsetup{width=0.9\linewidth}
    \caption{For any $\t\in [\tilde t, t_1)$ and the intersection point $(\bar t,\bar h)$ of the two boundaries  $\partial E^{\t,v_2}_{\hat h_1(\t)}$ and $\partial E^{\t+\eps,v_2}_{\hat h_1(\t+\eps)}$, belongs to $E^{0,v_2}_{\hat h_0}$.} 
    \label{figure:int-in}
\end{figure}
By Proposition~\ref{prop:tau1-tau2}, the collection of sets $\left\{E^{t,v_2}_{\hat h_1(t)}\setminus E^{0,v_2}_{\hat h_0}\right\}_{t\in [\tilde t, t_1]}$ is non-decreasing, in particular. Finally, the sequence $\left\{E^{t,v_2}_{\hat h_1(t)}\setminus E^{0,v_2}_{\hat h_0}\right\}_{t\geq 0}$ is non-decreasing.
\end{subsubcase}
\end{subcase}
\begin{subcase}[$v_1^2t_1\leq v_2\hat h_0$.]\label{subcase:two-2} Then $t^*=0$, and $\tilde t=0$ as well. By definition of $\tilde t$, $E^{\t,v_2}_{\hat h_1(\t)}\cap (E^{0,v_2}_{\hat h_0})^c\neq \emptyset$, for all $\t\geq 0$. We establish the claim using the same steps of \ref{subsubcase:t_1less}, along with the help of Proposition~\ref{prop:tau1-tau2}, with $\tilde t=0$. We do not provide the complete proof of Proposition~\ref{prop:tau1-tau2}, with $\tilde t=0$. As a matter of fact, we just show that $(\hat h_1(0))^{'}\geq -v_2$. The last inequality holds true, since from (\ref{eq:h1-tilde-t}), we have 
\begin{align}
(\hat h_1(0))^{'}= -\frac{v_1^2t_1}{\hat h_0}\geq -v_2,
\end{align}
using the assumption $v_1^2t_1\leq v_2\hat h_0$ of this~\ref{subcase:two-2}.
\end{subcase}
\end{case}
This completes the proof of part~(\ref{monotonic1}) of Proposition~\ref{prop:monotonic-sets}, when the typical handover is of type $\binom{1}{2,1}$ or equivalently $\bbinom{2}{2,1}$.
\end{proof}
\begin{proof}[Proof of Proposition~\ref{prop:tau1-tau2},\,part~(\ref{existence})] We first prove that the two boundaries  $\partial E^{\tau,v_2}_{\hat h_1(\tau)}$ and $\partial E^{\tau+\eps,v_2}_{\hat h_1(\tau+\eps)}$ intersect for any $0\leq \t<t_1$ and arbitrarily small $\eps>0$. For this, it is enough to show that, $\hat h_1(\t+\eps)\leq \hat h_1(\t)$ and 
\begin{align}
\t+\eps+\frac{1}{v_2}\hat h_1(\t+\eps)\geq \t+\frac{1}{v_2}\hat h_1(\t) \text{ i.e., } \frac{\hat h_1(\t+\eps)-\hat h_1(\t)}{\eps}\geq -v_2,
\label{eq:ineq-eps}
\end{align}
i.e., the right most point $(\t+\eps+\frac{1}{v_2}\hat h_1(\t+\eps), 0)$ of the upper half-ellipse $E^{\tau+\eps,v_2}_{\hat h_1(\tau+\eps)}$, is on the right of the right-most point $(\t+\frac{1}{v_2}\hat h_1(\t), 0)$ of the upper half-ellipse $E^{\tau,v_2}_{\hat h_1(\tau)}$.
The first inequality $\hat h_1(\t+\eps)\leq\hat h_1(\t)$ is true since, $\t<t_1$, which correspond to the decreasing wing of $C^1_{(t_1,h_1)}$.
The second inequality (\ref{eq:ineq-eps}) is true if we can show that $(\hat h_1(\t))'\geq -v_2$ for all $0\leq \t<t_1$, where
\begin{align}
\left(\hat h_1(\t)\right)^{'} 
&= \frac{v_1^2(\t-t_1)}{\hat h_1(\t)}.
\end{align}
If we can show that $(\hat h_1(\tilde t ))'\geq -v_2$, then, by the property of the decreasing wing of the radial bird $C^1_{(t_1,h_1)}$, we have $(\hat h_1(\t))'>(\hat h_1(\tilde t ))'\geq -v_2$, for all $\t\in [\tilde t, t_1)$. Indeed, using the equality $\tilde t+\hat h_1(\tilde t)/v_2= \hat h_0/v_2$, we get
\begin{align}
\left(\hat h_1(\tilde t )\right)'= \frac{v_1^2(\tilde t-t_1)}{\hat h_1(\tilde t)}= \frac{v_1^2(\tilde t-t_1)}{\hat h_0-v_2\tilde t}.
\label{eq:h1-tilde-t}
\end{align}
We have $\frac{v_1^2(\tilde t-t_1)}{\hat h_0-v_2\tilde t}\geq-v_2$, since 
\begin{align}
v_1^2(\tilde t-t_1)+ v_2 (\hat h_0-v_2\tilde t) = (v_1^2-v_2^2) \left( \tilde t- \frac{1}{v_1^2-v_2^2}(v_1^2t_1-v_2\hat h_0)\right) \stackrel{(\ref{eq:tilde-t})}{=} v_1^2t_1-v_2\hat h_0\geq 0.
\end{align}
This proves the inequality (\ref{eq:ineq-eps}) and as a result shows the existence of the intersection between the two boundaries  $\partial E^{\tau,v_2}_{\hat h_1(\tau)}$ and $\partial E^{\tau+\eps,v_2}_{\hat h_1(\tau+\eps)}$. 

We find the coordinate of the intersection $(\bar t,\bar h)$ of the two ellipses using the governing equations in (\ref{eq:e-tau1-tau2}), where   
\begin{equation}
\bar t = \half\left[2\t+\eps+\frac{\hat h_1^2(\t)-\hat h^2_1(\t+\eps)}{v_2^2\eps}\right], \,\, \bar h^2= \hat h_1^2(\t)-  v_2^2(\bar t-\t)^2= \hat h_1^2(\t)-\frac{v_2^2}{4}\left[\eps+\frac{\hat h_1^2(\t)-\hat h^2_1(\t+\eps)}{v_2^2\eps}\right]^2.
\label{eq:b-th}
\end{equation}
Let us first prove that for any $\t$, with $\tilde t\leq \t<t_1$, $(\bar t, \bar h)\in E^{0,v_2}_{\hat h_0}$, i.e., $0<v_2^2\bar t^2+\bar h^2<\hat h_0^2$. Using the expressions of $\bar t, \bar h^2$ from (\ref{eq:b-th}) we get that,
\begin{align}
v_2^2\bar t^2+\bar h^2 &= \frac{v_2^2}{4}\left[2\t+\eps+\frac{\hat h_1^2(\t)-\hat h^2_1(\t+\eps)}{v_2^2\eps}\right]^2  + \hat h_1^2(\t)-\frac{v_2^2}{4}\left[\eps+\frac{\hat h_1^2(\t)-\hat h^2_1(\t+\eps)}{v_2^2\eps}\right]^2\nn\\
&=v_2^2 \t \left[\t+\eps+\frac{\hat h_1^2(\t)-\hat h^2_1(\t+\eps)}{v_2^2\eps}\right]+ \hat h_1^2(\t)\label{eq:tau-eps}\\
&=v_2^2\t(\t+\eps) +\frac{\t+\eps}{\eps}\hat h_1^2(\t)-\frac{\t}{\eps}\hat h_1^2(\t+\eps).
\label{eq:inside-h0}
\end{align}
Now, due to the fact that $(\t,\hat h_1(\t)), (\t+\eps, \hat h_1(\t+\eps)) \in C^1_{(t_1,h_1)}$, using the expressions for $\hat h_1(\t) $ and $ \hat h_1(\t+\eps)$
from (\ref{eq:htau-1-2}) to the last expression in (\ref{eq:inside-h0}), we obtain that
\begin{align}
v_2^2\bar t^2+\bar h^2  &=v_2^2\t(\t+\eps) +\frac{\t+\eps}{\eps}\left[v_1^2(\t-t_1)^2+h_1^2\right]-\frac{\t}{\eps}\left[v_1^2(\t+\eps-t_1)^2+h_1^2\right]\nn\\
&=v_2^2\t(\t+\eps)+h_1^2 +v_1^2(t_1^2-\t(\t+\eps))\nn\\
&=h_1^2 +v_1^2t_1^2- (v_1^2-v_2^2)\t(\t+\eps) =\hat h_0^2- (v_1^2-v_2^2)\t(\t+\eps), 
\label{eq:bar-th}
\end{align}
since $h_1^2 +v_1^2t_1^2=\hat h_0^2$. The quantity in the last step of (\ref{eq:bar-th}) is less than $\hat h_0^2$, since $v_1>v_2$ and greater than $0$, which is clear from (\ref{eq:tau-eps}), since $\t<t_1$ and $\eps>0$ is arbitrarily small. This implies $(\bar t,\bar h)\in E^{0,v_2}_{\hat h_0}$, see Figure~\ref{figure:int-in}.
\end{proof}
\begin{proof}[Proof of Proposition~\ref{prop:tau1-tau2},\,part~(\ref{t+eps})]
Let us take $\t\in [\tilde t,t_1)$ and $\eps>0$ arbitrarily small. Then $(\bar t,\bar h)$ being the intersection of $\partial E^{\tau,v_2}_{\hat h_1(\tau_1)}$ and $\partial E^{\tau+\eps,v_2}_{\hat h_1(\tau+\eps)}$, implies that for all $t'<\bar t$, we have $ v_2^2(t'-\tau)^2+h^2<
v_2^2(t'-\tau-\eps)^2+h^2$ and for all $t'>\bar t$, we have $ v_2^2(t'-\tau)^2+h^2>
v_2^2(t'-\tau-\eps)^2+h^2$, using the curves as in (\ref{eq:e-tau1-tau2}) of the boundaries, see Figure~\ref{figure:int-in}. This along with the fact that $(\bar t, \bar h)\in E^{0,v_2}_{\hat h_0}$, allows one to conclude that $E^{\t,v_2}_{\hat h_1(\t)}\setminus E^{0,v_2}_{\hat h_0}\subset E^{\t+\eps,v_2}_{\hat h_1(\t+\eps)}\setminus E^{0,v_2}_{\hat h_0}$. This justifies the claim.
\end{proof}
\begin{proof}[Proof of Proposition~\ref{prop:monotonic-sets},\,part~(\ref{monotonic1}), typical handover type $1$ or $\bbinom{1}{1,1}$ and $\binom{1}{1,2}$ or $\bbinom{1}{2,1}$]
\begin{figure}[ht!]
 \centering
      \begin{subfigure}[t]{0.48\linewidth}
      \centering
      \begin{tikzpicture}[scale=0.7, every node/.style={scale=0.7}]
\pgftransformxscale{0.8}  
    \pgftransformyscale{0.8}    
    \draw[->] (-1, 0) -- (7, 0) node[right] {$t$};
    \draw[red](2,2.5) node{$\bullet$};
    \draw[red, domain=0.8:3.4, smooth] plot (\x, {(16*\x*\x-64*\x+64+2.5^2)^0.5});
    \draw[red](3.5,3) node{$\bullet$};
    \draw[](3.8,2.7) node{$(t_1,h_1)$};
    \draw[red, domain=2.3:4.7, smooth] plot (\x, {(16*\x*\x-32*3.5*\x+16*3.5^2+3^2)^0.5});
    \draw[](1.2,2.3) node{$(t_2,h_2)$};
    \draw[](2.1,4.4) node{$(0,\hat h_0)$};
     \draw[red]  (1.8,0) arc (-180:-360:1 and 4.1);
     \draw[blue]  (5.5,0) arc (0:180:2.65 and 4.1);   
    \draw[blue](1.15,0.6) node{$E^{0, v_2}_{\hat h_0}$};
    \draw[red](2.4,0.6) node{$E^{0, v_1}_{\hat h_0}$};
    \end{tikzpicture}
    \caption{The case where typical handover is of type $1$ or equivalently $\bbinom{1}{1,2}$. The subsequent half ellipses corresponding to both the speeds $v_1,v_2$, have their top point lying on $C^1_{(t_1,h_1)}$.}
\label{subfigure:pure_bird-ap1}
    \end{subfigure}
    \hspace{0.08in}
    \centering
      \begin{subfigure}[t]{0.48\linewidth}
      \centering
      \begin{tikzpicture}[scale=0.7, every node/.style={scale=0.7}]
\pgftransformxscale{0.8}  
\pgftransformyscale{0.8}    
\draw[->] (-5, 0) -- (4, 0) node[right] {$t$};
   \draw[red, domain=-1:2, smooth] plot (\x, {(16*\x*\x- 16*2*0.5*\x+16*(0.5)^2+2^2)^0.5});
    \draw[red](0.5,2) node{$\bullet$};
    \draw[](-0.8,2) node{$(t_2,h_2)$};
    \draw[blue](2,2.5) node{$\bullet$};
    \draw[](2.7,2.2) node{$(t_1,h_1)$};
    \draw[blue, domain=-1.5:4, smooth] plot (\x, {((3/2)^2*\x*\x-2*(3/2)^2*2*\x+(3/2)^2*2^2+2.5^2)^0.5});
    \draw[](-1.7,4.8) node{$(0,\hat h_0)$};
     \draw[blue]  (-3.6,0) arc (-180:-360:3.05 and 4.5);
     \draw[red]  (0.6,0) arc (0:180:1.15 and 4.5);
    \draw[blue](-2.5,0.6) node{$E^{0, v_2}_{\hat h_0}$};
    \draw[red](-1.1,0.9) node{$E^{0, v_1}_{\hat h_0}$};
    \end{tikzpicture}
    \caption{The case where typical handover is of type $\binom{1}{1,2}$ or $\bbinom{1}{2,1}$. The subsequent half ellipses corresponding to both the speeds $v_1,v_2$, have their top point lying on $C^1_{(t_2,h_2)}$.}
\label{subfigure:pure_bird-ap2}
    \end{subfigure}
    \caption{The other two cases of typical handovers being of type $1$ and type $\binom{1}{1,2}$ or $\bbinom{1}{2,1}$.}
\label{figure:pure_bird-ap}
\end{figure}
Our claim for this part in the case of typical handovers of type $1$ or equivalently $\bbinom{1}{1,1}$ and $\binom{1}{1,2}$ or equivalently $\bbinom{1}{2,1}$, can be proved similarly, since the proof only concerns the type of the radial bird, say $C^1_{(t_1,h_1)}$, corresponding to the station that takes over the service after the handover at the intersection, see picture~(\subref{subfigure:pure_bird-ap1}) of Figure~\ref{figure:pure_bird-ap} and $C^1_{(t_2,h_2)}$, corresponding to picture~(\subref{subfigure:pure_bird-ap2}) of Figure~\ref{figure:pure_bird-ap} respectively. The sequence of upper half-ellipses that have their top point lying on $C^1_{(t_1,h_1)}$ and $C^1_{(t_2,h_2)}$, respectively. 
\end{proof}
\begin{remark}
The converse of the argument used in the proof of Proposition~\ref{prop:tau1-tau2}, part~(\ref{t+eps}), can be proved to be false in general case, using a counter example depicted in picture~\subref{subfigure:mH-FHoT2} of Figure~\ref{figure:mH-FHoT}. Indeed, there exists $t> t_1$ such that the boundaries $\partial E^{t,v_2}_{\hat h_1(t)}$, $\partial E^{t+\eps,v_2}_{\hat h_1(t+\eps)}$, with the top point lying on the increasing wing of $C^1_{(t_1,h_1)}$, intersect each other outside of $E^{0,v_2}_{\hat h_0}$ or never, but still we have $E^{t,v_2}_{\hat h_1(t)}\setminus E^{0,v_2}_{\hat h_0}\subset E^{t+\eps,v_2}_{\hat h_1(t+\eps)}\setminus E^{0,v_2}_{\hat h_0}$.
\end{remark}
\subsection{Proof of Proposition~\ref{prop:monotonic-sets}~part~(\ref{monotonic2})}~\label{subsection:inc-wing3}
This scenario appears when at the typical handover a station of type $2$ takes over, which correspond to one of the types: $2$ or $\bbinom{1}{2,2}$, $\binom{2}{1,2}$ or $\bbinom{2}{1,2}$ and $\binom{2}{2,1}$ or $\bbinom{1}{2,1}$, respectively.
\begin{case}[Type $\binom{2}{1,2}$ or equivalently $\bbinom{2}{1,2}$.]\label{case:two-12}
Here $(t_n,h_n)=(t_1,h_1)$, which is of type $2$. Define $\hat h_2(t):=(v_2^2(t-t_1)^2+h_1^2)^\half$. In this case, the boundaries $\partial E^{t,v_2}_{\hat h_2(t)}$ of the sets in the collection $\{E^{t,v_2}_{\hat h_2(t)}\}_{t\geq 0}$, passes through the head point $(t_1,h_1)$, see picture~(\subref{subfigure:2in1nota}) of Figure~\ref{figure:2in1not}. By applying the same argument of single speed case, Lemma~\ref{lemma:increasingS}, we can show that $\left\{E^{t,v_2}_{\hat h_2(t)}\setminus E^{0,v_2}_{\hat h_0}\right\}_{t\geq 0}$ is non-decreasing. 
\begin{figure}[ht!]
 \centering
      \begin{subfigure}[t]{0.43\linewidth}\begin{tikzpicture}[scale=0.7, every node/.style={scale=0.7}]
\pgftransformxscale{1}  
    \pgftransformyscale{1}    
\draw[->] (-2, 0) -- (6, 0) node[right] {$t$};
   \draw[red, domain=-0.3:1.3, smooth] plot (\x, {(16*\x*\x- 16*2*0.5*\x+16*(0.5)^2+2^2)^0.5});
    \draw[red](0.5,2) node{$\bullet$};
    \draw[](-0.8,2) node{$(t_2,h_2)$};
    \draw[blue](2,2.5) node{$\bullet$};
    \draw[](2.7,2.2) node{$(t_1,h_1)$};
    \draw[blue, domain=-0.3:4.4, smooth] plot (\x, {((3/2)^2*\x*\x-2*(3/2)^2*2*\x+(3/2)^2*2^2+2.5^2)^0.5});
    \draw[](1.7,3.4) node{$(0,\hat h_0)$};
     \draw[blue]  (2.9,0) arc (0:180:1.8 and 2.9); 
     \draw[](-1.2,0.4) node{\blue{$E^{0,v_2}_{\hat h_0}$}};
     \draw[blue]  (-0.3,0) arc (-180:-360:1.676 and 2.7);
     \draw[blue]  (1.05,0) arc (-180:-360:1.62 and 2.73);
     \draw[blue]  (1.35,0) arc (-180:-360:2.1 and 3.4);
    \end{tikzpicture}
    \caption{The collection $\left\{E^{t,v_2}_{\hat h_2(t)}\setminus E^{0,v_2}_{\hat h_0}\right\}_{t\geq 0}$ is non-decreasing in $t$.}
\label{subfigure:2in1nota}
    \end{subfigure}
    \hspace{0.1in}
    \centering
\begin{subfigure}[t]{0.43\linewidth}\begin{tikzpicture}[scale=0.7, every node/.style={scale=0.7}]
\pgftransformxscale{1}  
\pgftransformyscale{1}    
\draw[->] (-2.3, 0) -- (5, 0) node[right] {$t$};
\draw[red, domain=-0.9:0.4, smooth] plot (\x, {(16*\x*\x+ 16*2*0.5*\x+16*(0.5)^2+3^2)^0.5});
    \draw[red](-0.5,3) node{$\bullet$};
    \draw[](-0.8,2.5) node{$(t_2,h_2)$};
    \draw[blue](2,2.5) node{$\bullet$};
    \draw[](1.8,3) node{$(t_1,h_1)$};
    \draw[blue, domain=-0.3:4, smooth] plot (\x, {((3/2)^2*\x*\x-2*(3/2)^2*2*\x+(3/2)^2*2^2+2.5^2)^0.5});
    \draw[](1,4) node{$(0,\hat h_0)$};
    \draw[](-1.4,0.4) node{\red{$E^{0,v_1}_{\hat h_0}$}};
     \draw[red]  (-0.95,0) arc (-180:-360:1.043 and 3.78); 
     \draw[red]  (0,0) arc (-180:-360:0.84 and 3.05);  
     \draw[red]  (2.2,0) arc (0:180:0.731 and 2.65);
     \draw[red](3.15,0) arc (0:180:0.7 and 2.6);   
     \draw[red](3.95,0) arc (0:180:0.827 and 3);
    \end{tikzpicture}
    \caption{The collection $\left\{E^{t,v_1}_{\hat h_2(t)}\setminus E^{0,v_1}_{\hat h_0}\right\}_{t\geq 0}$ is non-monotonic in $t$, as the sequence of sets are not comparable.}
    \label{subfigure:2in1notb}
    \end{subfigure}
    \caption{Sequence of sets $\left\{E^{t,v_l}_{\hat h_2(t)}\setminus E^{0,v_l}_{\hat h_0}\right\}_{t\geq 0}$, for $l=1,2$.}   
    \label{figure:2in1not}
\end{figure}

The other collection $\left\{E^{t,v_1}_{\hat h_{2}(t)}\setminus E^{0,v_1}_{\hat h_0}\right\}_{t\geq 0}$ corresponding to the speed $v_1$, can be shown to be not always monotonic, by creating artificial example depicted in picture~(\subref{subfigure:2in1notb}) of Figure~\ref{figure:2in1not}.
\end{case}
\begin{case}[Type $2$ or equivalently $\bbinom{1}{2,2}$.]
This case can be proved similarly to \ref{case:two-12}, since the subsequent half-ellipses corresponding to different speeds $v_1,v_2$ are such that their top point lies on the radial bird $C^2_{(t_1,h_1)}$, see picture~(\subref{subfigure:two121}) of Figure~\ref{figure:two12}.
\end{case}
\begin{figure}[ht!]
 \centering
      \begin{subfigure}[t]{0.45\linewidth}
      \centering
      \begin{tikzpicture}[scale=0.7, every node/.style={scale=0.7}]
\pgftransformxscale{0.8}  
\pgftransformyscale{0.8}    
\draw[->] (-2, 0) -- (7, 0) node[right] {$t$};
    \draw[blue](0.73,2.5) node{$\bullet$};
    \draw[blue, domain=-1:4, smooth] plot (\x, {((3/2)^2*\x*\x-2*(3/2)^2*0.73*\x+(3/2)^2*0.73*0.73+2.5^2)^0.5});
    \draw[blue](4.7,2.85) node{$\bullet$};
    \draw[](5.8,2.5) node{$(t_1,h_1)$};
    \draw[blue, domain=1.6:7, smooth] plot (\x, {((3/2)^2*\x*\x-2*(3/2)^2*4.7*\x+(3/2)^2*4.7^2+2.85^2)^0.5});
    \draw[](-0.2,2.3) node{$(t_2,h_2)$};
    \draw[](1.8,4.2) node{$(0,\hat h_0)$};
     \draw[red]  (1.8,0) arc (-180:-360:1 and 4.05);
     \draw[blue]  (5.5,0) arc (0:180:2.65 and 4.05);   
    \draw[blue](1.1,0.6) node{$E^{0, v_2}_{\hat h_0}$};
    \draw[red](2.4,0.6) node{$E^{0, v_1}_{\hat h_0}$};
\end{tikzpicture}
\caption{The case of typical handover is of type $2$ or equivalently $\bbinom{1}{2,2}$. The subsequent half-ellipses corresponding to speeds $v_1,v_2$ are such that their top point lies on the radial bird $C^2_{(t_1,h_1)}$.}
\label{subfigure:two121}
\end{subfigure}
\hspace{0.1in}
\centering
\begin{subfigure}[t]{0.45\linewidth}
\centering
\begin{tikzpicture}[scale=0.7, every node/.style={scale=0.7}]
\pgftransformxscale{0.6}  
\pgftransformyscale{0.6}    
    \draw[->] (-2, 0) -- (8, 0) node[right] {$t$};
   \draw[blue, domain=-1.8:4, smooth] plot (\x, {((3/2)^2*\x*\x+2^2)^0.5});
    \draw[blue](0,2) node{$\bullet$};
    \draw[](-1.1,1.65) node{$(t_2,h_2)$};
    \draw[red, domain=0.5:3.5, smooth] plot (\x, {(16*\x*\x-64*\x+64+2.5^2)^0.5});
    \draw[red](2,2.5) node{$\bullet$};
    \draw[](2.8,2.2) node{$(t_1,h_1)$};
    \draw[](2.2,5.5) node{$(0,\hat h_0)$};
     \draw[blue](3,0.6) node{$E^{0, v_2}_{\hat h_0}$};
     \draw[red](4.5,0.6) node{$E^{0, v_1}_{\hat h_0}$};
     \draw[red]  (1.85,0) arc (-180:-360:1.25 and 5.05);
     \draw[blue]  (6.5,0) arc (0:180:3.4 and 5.05);
    \end{tikzpicture}
    \caption{The case of typical handover is of type $\binom{2}{2,1}$ or $\bbinom{1}{1,2}$. The subsequent half-ellipses corresponding to speeds $v_1,v_2$ are such that their top point lies on the increasing wing of the  radial bird $C^2_{(t_2,h_2)}$.}
    \label{subfigure:two122}
    \end{subfigure}
    \caption{The other two cases of typical handovers being of type $2$ or equivalently $\bbinom{1}{2,2}$ and type $\binom{2}{2,1}$ or equivalently $\bbinom{1}{1,2}$.}
\label{figure:two12}
\end{figure}
\begin{case}[Type $\binom{2}{2,1}$ or equivalently $\bbinom{1}{1,2}$.]
In this case the subsequent half-ellipses are such that their top point lies on the increasing wing of the radial bird $C^2_{(t_2,h_2)}$, see picture~(\subref{subfigure:two122}) of Figure~\ref{figure:two12}. The proof of monotonicity of $\left\{E^{t,v_2}_{\hat h_2(t)}\setminus E^{0,v_2}_{\hat h_0}\right\}_{t\geq 0}$ follows when using the increasing wing property, as in Remark~\ref{remark:inc-arm2}. The non-monotonicity of $\left\{E^{t,v_1}_{\hat h_2(t)}\setminus E^{0,v_1}_{\hat h_0}\right\}_{t\geq 0}$ can be shown using a similar counter example in \ref{case:two-12}.
\end{case}
\qed
\subsection{Proof of Lemma~\ref{lemma:LT1-l}, type $\binom{1}{2,2}$ or equivalently $\bbinom{1}{2,2}$}~\label{subsection:Lemma122}
Using Lemma~\ref{lemma:Palm-Lcal-ii} for the second term in (\ref{eq:decompLT2}), with $\t_n=2=\t_p$ and $q=1$, we have
\begin{align}
\E^0_{\Wcal_2} [e^{-\rho T}]
&= \frac {4\la_2^2v_2^2} {L_{2}}
\int_{(\R^+)^3}
\E_\Hcal\left[ e^{-\rho(T_1\circ\th_{\hat s})}  \prod_{l=1,2}\one_{\Hcal^l\left(E^{\hat{s},v_l}_{\hat{h}}\right)=0}  \right] {\rm d}t' {\rm d}h' {\rm d}h, 
\label{eq:T22a}
\end{align}
where we consider two radial birds of type $\t_n=2=\t_p$ at $(t_n,h_n)=(0,h)$ and $(t_p,h_p)=(-t',h)$ and where, in short, we wrote $(\hat s(0,h,-t',h'), \hat h(0,h,-t',h'))= (\hat s,\hat h)$. Note that, since $\t_n=2$ and $q=1$, the station corresponding to the head point $(0,h)$ is the serving station after time $\hat s$. Since $(0,h)$ is of type $2$, we define $\hat h_2(t):=\left(v_2^2t^2+h^2\right)^\half$, which we write in short as $\hat h_t$ within this proof. Notice that 
\[
T\circ \th_{\hat s}=T_1(\hat s)-\hat s,
\]
where $T_1(\hat s)$ is the time to the handover after time $\hat s$. By Corollary~\ref{corollary:third-pt2}, the inner expectation in (\ref{eq:T22a}) can be determined by exploring the next handover as an intersection of the radial bird at $(0,h)$ and a third radial bird with its head at, say $(T_j,H_j)$, that lies outside of $\overline{E^{\hat s, v_1}_{\hat h}}$ or $\overline{E^{\hat s, v_2}_{\hat h}}$, depending on its type. We have the following cases:
\begin{case}[\textbf{\em $(T_j,H_j)$ is of type 1.}]\label{case:H221} In this case, we must have $(T_j,H_j)\notin E^{\hat s, v_1}_{\hat h}$. 
Along with this, we must assume that $T_j\geq \hat s$. In case $T_j<\hat s$,  the intersection created by the bird with head at $(T_j,H_j)$ either contributes to a handover that happens before time $\hat s$, or is such that  the intersection lies in $\Lcal^+_e(v_1,v_2)$, see Figure~\ref{figure:ge-hats}. In this case $T_j\geq \hat s$ and the unexplored region is $Q_{\hat s}\setminus \overline{E^{\hat s,v_1}_{\hat h}}$. There are two intersection points $(\hat s_1(0,h, T_j, H_j),\hat h_1(0,h, T_j, H_j))$ and $(\hat s_2(0,h, T_j, H_j),\hat h_2(0,h, T_j, H_j))$. For any of these points to represent a handover, we need the extra regions $\interior{\left(\bigcup_{t\in [\hat s,\hat s_1(0,h,T_j,H_j)]} E^{t,v_l}_{\hat h_t}\setminus E^{\hat s, v_l}_{\hat h}\right)}$ and $\interior{\left(\bigcup_{t\in [\hat s,\hat s_2(0,h,T_j,H_j)]} E^{t,v_l}_{\hat h_t}\setminus E^{\hat s, v_l}_{\hat h}\right)}$ to be empty of points from $\Hcal^l$, for $l=1,2$, where $\hat h_t=(v_2^2t^2+h^2)^\half$. The reason for the union is the same as above. Since $\hat s_1(0,h, T_j, H_j)\leq \hat s_2(0,h, T_j, H_j)$, we have the inclusion 
\[
\bigcup_{t\in [\hat s,\hat s_1(0,h,T_j,H_j)]} E^{t,v_l}_{\hat h_t}\setminus E^{\hat s, v_l}_{\hat h}\subset \bigcup_{t\in [\hat s,\hat s_2(0,h,T_j,H_j)]} E^{t,v_l}_{\hat h_t}\setminus E^{\hat s, v_l}_{\hat h}.
\]
Hence, it cannot happen that $\hat s_2(0,h, T_j, H_j)$ corresponds to a handover, while $\hat s_1(0,h, T_j, H_j)$ does not. We must take the first intersection point as a handover, and thus $T_1(\hat s):=\hat s_1(0,h, T_j, H_j)$.

Additionally, the point $(T_j,H_j)$ must be chosen in such a way that $T_j\in D_\ell(0,h, H_j)$ or $T_j\in D_r(0,h, H_j)$, depending on whether $T_j<0$ or $T_j\geq 0$, respectively, given $h,H_j$, see Figure~\ref{figure:pure_bird22}. The type of the pair of consecutive handovers is $\bbinom{1,2}{2,2,1}$ or $\bbinom{1,1}{2,2,1}$, depending on whether $T_j<0$ or $T_j\geq 0$, respectively.
\end{case}
\begin{figure}[ht!]
\centering
      \begin{subfigure}[t]{0.43\linewidth}
       \centering
      \begin{tikzpicture}[scale=0.7, every node/.style={scale=0.7}]
\pgftransformxscale{1}  
\pgftransformyscale{1}    
\draw[->] (-1, 0) -- (8.5, 0) node[right] {$t$};
\draw[blue](0.73,2.5) node{$\bullet$};
\draw[blue, domain=-1:4, smooth] plot (\x, {((3/2)^2*\x*\x-2*(3/2)^2*0.73*\x+(3/2)^2*0.73*0.73+2.5^2)^0.5});
\draw[blue](4.7,2.85) node{$\bullet$};
\draw[](5.8,2.7) node{$(0,h)$};
\draw[dashed] (4.7,0) -- (4.7,5);
\draw[blue, domain=2:7, smooth] plot (\x, {((3/2)^2*\x*\x-2*(3/2)^2*4.7*\x+(3/2)^2*4.7^2+2.85^2)^0.5});
\draw[](1.2,2.3) node{$(-t',h')$};
\draw[](4,4.4) node{$(\hat s,\hat h)$};
\draw[red, domain=0.5:2.5, smooth] plot (\x, {(16*\x*\x-16*2*1.5*\x+16*1.5*1.5+3^2)^0.5});
\draw[red](1.5,3) node{$\bullet$};
\draw[](0.5,3) node{$(t_1,h_1)$};
\draw[red]  (1.8,0) arc (-180:-360:1 and 4.05);
\draw[blue]  (5.5,0) arc (0:180:2.65 and 4.05);   \draw[blue](1.3,0.6) node{$E^{\hat s, v_2}_{\hat h}$};
\draw[red](2.4,0.6) node{$E^{\hat s, v_1}_{\hat h}$};
\end{tikzpicture}
\caption{In this case, the intersection of $(0,h)$ and $(t_1,h_1)$, lies \blue{in $\Lcal^+_e(v_1,v_2)$}.}
\label{subfigure:ge-hats1}
\end{subfigure}
\hspace{0.2in}
\centering
\begin{subfigure}[t]{0.43\linewidth}
 \centering
\begin{tikzpicture}[scale=0.7, every node/.style={scale=0.7}]
\pgftransformxscale{1}  
\pgftransformyscale{1}    
    \draw[->] (-1, 0) -- (7, 0) node[right] {$t$};
    \draw[blue](0.73,2.5) node{$\bullet$};
    \draw[blue, domain=-1:4, smooth] plot (\x, {((3/2)^2*\x*\x-2*(3/2)^2*0.73*\x+(3/2)^2*0.73*0.73+2.5^2)^0.5});
    \draw[blue](4.7,2.85) node{$\bullet$};
    \draw[](5.8,2.5) node{$(0,h)$};
    \draw[blue, domain=1.6:7, smooth] plot (\x, {((3/2)^2*\x*\x-2*(3/2)^2*4.7*\x+(3/2)^2*4.7^2+2.85^2)^0.5});
    \draw[](-0.2,2.3) node{$(-t',h')$};
    \draw[dashed] (4.7,0) -- (4.7,5);
    \draw[](4,4.2) node{$(\hat s,\hat h)$};
     \draw[red, domain=0.3:2.6, smooth] plot (\x, {(16*\x*\x-16*2*1.5*\x+16*1.5*1.5+2.3^2)^0.5});
     \draw[red](1.5,2.3) node{$\bullet$};
     \draw[](1.5,1.8) node{$(t_1,h_1)$};
     \draw[red]  (1.8,0) arc (-180:-360:1 and 4.05);
     \draw[blue]  (5.5,0) arc (0:180:2.65 and 4.05);   
    \draw[blue](1.1,0.6) node{$E^{\hat s, v_2}_{\hat h}$};
    \draw[red](2.4,0.6) node{$E^{\hat s, v_1}_{\hat h}$};
    \end{tikzpicture}
    \caption{In this case, the intersection of $(0,h)$ and $(t_1,h_1)$ gives a handover before time $\hat s$.}
\label{subfigure:ge-hats2}
    \end{subfigure}
    \caption{The next head point $(T_j,H_j)=(t_1,h_1)\notin \overline{E^{\hat s, v_1}_{\hat h}}$, but it cannot be such that $t_1<\hat s$.}
\label{figure:ge-hats}
\end{figure}
\begin{figure}[ht!]
 \centering
      \begin{subfigure}[t]{0.43\linewidth}
       \centering
      \begin{tikzpicture}[scale=0.7, every node/.style={scale=0.7}]
\pgftransformxscale{1}  
\pgftransformyscale{1}    
    \draw[->] (-1, 0) -- (7, 0) node[right] {$t$};
    \draw[blue](0.73,2.5) node{$\bullet$};
    \draw[blue, domain=-1:4, smooth] plot (\x, {((3/2)^2*\x*\x-2*(3/2)^2*0.73*\x+(3/2)^2*0.73*0.73+2.5^2)^0.5});
    \draw[blue](4.7,2.85) node{$\bullet$};
    \draw[](5.8,2.5) node{$(0,h)$};
    \draw[blue, domain=1.9:7, smooth] plot (\x, {((3/2)^2*\x*\x-2*(3/2)^2*4.7*\x+(3/2)^2*4.7^2+2.85^2)^0.5});
    \draw[](-0.2,2.3) node{$(-t',h')$};
    \draw[dashed] (4.7,0) -- (4.7,5.5);
    \draw[](4,4.7) node{$(\hat s',\hat h')$};
    \draw[->] (3.8, 4.4) -- (3.7,3.4);
    \draw[](1.8,4.2) node{$(\hat s,\hat h)$};
     \draw[blue]  (1.65,0) arc (-180:-360:2 and 3.28);
     \draw[red]  (4.45,0) arc (0:180:0.85 and 3.28);
     \draw[red, domain=3.1:5.2, smooth] plot (\x, {(16*\x*\x-16*2*4.3*\x+16*4.3*4.3+1.8^2)^0.5});
     \draw[red](4.3,1.8) node{$\bullet$};
     \draw[](4.5,1.3) node{$(t_1,h_1)$};
     \draw[red]  (1.8,0) arc (-180:-360:1 and 4.05);
     \draw[blue]  (5.5,0) arc (0:180:2.65 and 4.05);   
    \draw[blue](1.1,0.6) node{$E^{\hat s, v_2}_{\hat h}$};
    \draw[red](2.4,0.6) node{$E^{\hat s, v_1}_{\hat h}$};
    \draw[blue](6.2,0.6) node{$E^{\hat s',v_2}_{\hat h'}$};
     \draw[red](3.45,0.6) node{$E^{\hat s',v_1}_{\hat h'}$};
    \end{tikzpicture}
    \caption{For the intersection $(\hat s, \hat h)$ and $(\hat s', \hat h')$ to represent two consecutive  handovers, the regions \red{$S_1\cup E^{\hat s, v_1}_{\hat h} \cup E^{\hat s',v_1}_{\hat h'}$} and \blue{$E^{\hat s, v_2}_{\hat h}\cup E^{\hat s',v_2}_{\hat h'}$} must have no points from $\Hcal^1$ and $\Hcal^2$, respectively. Here $(\hat s', \hat h')$ is a realization of $(\hat s_1(0,h,T_j,H_j),\hat h_1(0,h,T_j,H_j))$.}
\label{subfigure:pure_bird221}
    \end{subfigure}
    \hspace{0.2in}
    \centering
      \begin{subfigure}[t]{0.43\linewidth}
       \centering
      \begin{tikzpicture}[scale=0.7, every node/.style={scale=0.7}]
\pgftransformxscale{0.95}  
\pgftransformyscale{0.95}    
\draw[->] (-1, 0) -- (8.5, 0) node[right] {$t$};
    \draw[blue](0.73,2.5) node{$\bullet$};
    \draw[blue, domain=-1:3.2, smooth] plot (\x, {((3/2)^2*\x*\x-2*(3/2)^2*0.73*\x+(3/2)^2*0.73*0.73+2.5^2)^0.5});
    \draw[blue](4.7,2.85) node{$\bullet$};
    \draw[](3.8,2.7) node{$(0,h)$};
    \draw[dashed] (4.7,0) -- (4.7,5);
    \draw[blue, domain=2:7.3, smooth] plot (\x, {((3/2)^2*\x*\x-2*(3/2)^2*4.7*\x+(3/2)^2*4.7^2+2.85^2)^0.5});
    \draw[](1.2,2.3) node{$(-t',h')$};
    \draw[](5.4,3.7) node{$(\hat s',\hat h')$};
    \draw[](2.1,4.4) node{$(\hat s,\hat h)$};
     \draw[blue]  (3.45,0) arc (-180:-360:1.95 and 3.05);
     \draw[red]  (6.2,0) arc (0:180:0.75 and 3.05);
     \draw[red, domain=4.9:7.1, smooth] plot (\x, {(16*\x*\x-16*2*6*\x+16*6*6+2.1^2)^0.5});
     \draw[red](6,2.1) node{$\bullet$};
     \draw[](6.2,1.7) node{$(t_1,h_1)$};
     \draw[red]  (1.8,0) arc (-180:-360:1 and 4.05);
     \draw[blue]  (5.5,0) arc (0:180:2.65 and 4.05);   
    \draw[blue](1.3,0.6) node{$E^{\hat s, v_2}_{\hat h}$};
    \draw[red](2.4,0.6) node{$E^{\hat s, v_1}_{\hat h}$};
    \draw[blue](6.7,0.6) node{$E^{\hat s',v_2}_{\hat h'}$};
     \draw[red](5.3,0.6) node{$E^{\hat s',v_1}_{\hat h'}$};
    \end{tikzpicture}
    \caption{For the intersections $(\hat s, \hat h)$ and $(\hat s', \hat h')$ to represent two consecutive  handovers, the regions \red{$S_1\cup E^{\hat s, v_1}_{\hat h} \cup E^{\hat s',v_1}_{\hat h'}$} and \blue{$E^{\hat s, v_2}_{\hat h}\cup E^{\hat s',v_2}_{\hat h'}$} must have no points from $\Hcal^1$ and $\Hcal^2$, respectively. Here $(\hat s', \hat h')$ is a realization of $(\hat s_1(0,h,T_j,H_j),\hat h_1(0,h,T_j,H_j))$.}
\label{subfigure:pure_bird222}
    \end{subfigure}
    \captionsetup{width=0.9\linewidth}
    \caption{The case where $(T_j,H_j)=(t_1,h_1)$ is of type $1$. The two cases considered are, $t_1<0$ and $t_1\geq 0$, as in the first two terms in (\ref{eq:T22bb}).}
\label{figure:pure_bird22}
\end{figure}
\begin{figure}[ht!]
\begin{tikzpicture}[scale=0.7, every node/.style={scale=0.7}]
\pgftransformxscale{1.15}  
\pgftransformyscale{1.15}    
\draw[->] (-1, 0) -- (8.5, 0) node[right] {$t$};
    \draw[blue](0.73,2.5) node{$\bullet$};
    \draw[blue, domain=-1:3.7, smooth] plot (\x, {((3/2)^2*\x*\x-2*(3/2)^2*0.73*\x+(3/2)^2*0.73*0.73+2.5^2)^0.5});
    \draw[blue](4.7,2.85) node{$\bullet$};
    \draw[](3.8,2.7) node{$(0,h)$};
    \draw[dashed] (4.7,0) -- (4.7,5);
    \draw[blue, domain=2:7, smooth] plot (\x, {((3/2)^2*\x*\x-2*(3/2)^2*4.7*\x+(3/2)^2*4.7^2+2.85^2)^0.5});
    \draw[](1.2,2.3) node{$(-t',h')$};
    \draw[](5.4,3.7) node{$(\hat s',\hat h')$};
    \draw[](2.1,4.4) node{$(\hat s,\hat h)$};
     \draw[blue]  (3.4,0) arc (-180:-360:1.95 and 3.03);
     \draw[red]  (6.1,0) arc (0:180:0.75 and 3.03);
     \draw[blue, domain=4:8, smooth] plot (\x, {((3/2)^2*\x*\x-(3/2)^2*2*6.8*\x+(3/2)^2*6.8*6.8+2.05^2)^0.5});
     \draw[blue](6.8,2.05) node{$\bullet$};
     \draw[](7,1.7) node{$(t_1,h_1)$};
     \draw[red]  (1.8,0) arc (-180:-360:1 and 4.05);
     \draw[blue]  (5.5,0) arc (0:180:2.65 and 4.05);   
    \draw[blue](1.3,0.6) node{$E^{\hat s, v_2}_{\hat h}$};
    \draw[red](2.4,0.6) node{$E^{\hat s, v_1}_{\hat h}$};
    \draw[blue](6.7,0.6) node{$E^{\hat s',v_2}_{\hat h'}$};
     \draw[red](5.3,0.6) node{$E^{\hat s',v_1}_{\hat h'}$};
    \end{tikzpicture}
    \captionsetup{width=0.9\linewidth}
    \caption{The case where $(t_1,h_1)$ is of type $2$, with $t_1\geq 0$, as in the last term in (\ref{eq:T22bb}). For the intersections $(\hat s, \hat h)$ and $(\hat s', \hat h')$ to represent two consecutive  handovers, the regions \red{$S_2\cup E^{\hat s, v_1}_{\hat h} \cup E^{\hat s',v_1}_{\hat h'}$} and \blue{$E^{\hat s, v_2}_{\hat h}\cup E^{\hat s',v_2}_{\hat h'}$} must have no points from $\Hcal^1$ and $\Hcal^2$, respectively. Here $(\hat s', \hat h')$ is a realization of $(\hat s(0,h,T_j,H_j),\hat h(0,h,T_j,H_j))$.}
\label{figure:pure_bird222}
\end{figure}
\begin{case}[\textbf{\em $(T_j,H_j)$ is of type 2.}]\label{case:H222} The intersection of the radial birds of type $2$ at $(0,h)$ and $(T_j,H_j)\notin E^{\hat s,v_2}_{\hat h}$, gives rise to a handover if $T_j\geq 0$, see Figure~\ref{figure:pure_bird222}. Otherwise, if $T_j<0$, then the intersection given by the birds at $(T_j,H_j)$ and $(0,h)$ is either contained in $\Lcal^+_e(v_1,v_2)$ or it gives rise to a handover before time $\hat s$. So we must have $T_j\geq 0$ and the unexplored region is $Q_0\setminus \overline{E^{\hat s,v_2}_{\hat h}}$. The type of the pair of handovers is $\bbinom{1,1}{2,2,2}$. Let $(\hat s(0,h, T_j, H_j),\hat h(0,h, T_j, H_j))$ be their intersection point. Here, we can consider that $T_1(\hat s):=\hat s(0,h, T_j, H_j)$. For this intersection to give that the next handover takes place after time $\hat s$, the extra region $\interior{\left(\bigcup_{t\in [\hat s,\hat s(0,h,T_j,H_j)]} E^{t,v_l}_{\hat h_t}\setminus E^{\hat s, v_l}_{\hat h}\right)}$, beyond $\overline{E^{\hat s, v_l}_{\hat h}}$, must have no point of $\Hcal^l$ for $l=1,2$, where $\hat h_t=(v_2^2t^2+h^2)^\half$.
\end{case}

Based on the last discussions with \ref{case:H221} and \ref{case:H222}, we can decompose the inner expectation in (\ref{eq:T22a}) as
\begin{align}
\lefteqn{\E_\Hcal\left[ e^{-\rho(T_1(\hat s)-\hat s)}  \prod_{l=1,2}\one_{\Hcal^l\left(E^{\hat{s},v_l}_{\hat{h}}\right)=0}  \right]}\nn\\
&=
\E_\Hcal\left[ \one_{\exists (T_j,H_j)\in \Hcal^1 \,\mbox{:}\, \hat s\leq T_j\leq 0,\, T_j\in D_\ell(0,h,H_j) \, \bigcap_{l=1,2}\left\{\Hcal^l\interior{\left(\bigcup_{t\in [\hat s,\hat s_1(0,h,T_j,H_j)]} E^{t,v_l}_{\hat h_t}\setminus E^{\hat s,v_l}_{\hat h}\right)}=0\right\}}  \right.\nn\\
&\hspace{3in}\times e^{-\rho (\hat s_1(0,h,T_j,H_j)-\hat s)}\prod_{l=1,2}\one_{\Hcal^l\left(E^{\hat{s},v_l}_{\hat{h}}\right)=0}\Bigg] \nn\\
&\;+ \E_\Hcal\left[ \one_{\exists (T_j,H_j)\in \Hcal^1 \,\mbox{:}\, T_j\in D_r(0,h,H_j),\, \bigcap_{l=1,2}\left\{\Hcal^l\interior{\left(\bigcup_{t\in [\hat s,\hat s_1(0,h,T_j,H_j)]} E^{t,v_l}_{\hat h_t}\setminus E^{\hat s,v_l}_{\hat h}\right)}=0\right\}} e^{-\rho (\hat s_1(0,h,T_j,H_j)-\hat s)} \right.\nn\\
&\hspace{4in}\times \prod_{l=1,2}\one_{\Hcal^l\left(E^{\hat{s},v_l}_{\hat{h}}\right)=0}\Bigg] \nn\\
&\;\;+ \E_\Hcal\left[ \one_{\exists (T_j,H_j)\in \Hcal^2 \,\mbox{:}\, T_j\ge 0, \, \bigcap_{l=1,2}\left\{\Hcal^l\interior{\left(\bigcup_{t\in [\hat s,\hat s(0,h,T_j,H_j)]} E^{t,v_l}_{\hat h_t}\setminus E^{\hat s,v_l}_{\hat h}\right)}=0\right\}} e^{-\rho (\hat s(0,h,T_j,H_j)-\hat s)}\right.\nn\\
&\hspace{4in}\times \prod_{l=1,2}\one_{\Hcal^l\left(E^{\hat{s},v_l}_{\hat{h}}\right)=0}\Bigg].
\label{eq:T22bb}
\end{align}
Similarly to the final step in (\ref{eq:T12c}), there exists a unique head point $(T_j,H_j)$ satisfying the condition inside the expectation in each individual term in (\ref{eq:T22bb}), the last expression equals
\begin{align}
\lefteqn{\E_\Hcal\left[\sum_{ (T_j,H_j)\in \Hcal^1 \,\mbox{:}\, \hat s\leq T_j\leq 0,\,T_j\in D_\ell(0,h,H_j)} \prod_{l=1,2}\one_{\Hcal^l\left(\bigcup_{t\in [\hat s,\hat s_1(0,h,T_j,H_j)]} E^{t,v_l}_{\hat h_t}\cup  E^{\hat s,v_l}_{\hat h}\right)=0} e^{-\rho (\hat s_1(0,h,T_j,H_j)-\hat s)}\right]}\nn\\
&\;\;+\E_\Hcal\left[\sum_{ (T_j,H_j)\in \Hcal^1 \,\mbox{:}\, T_j\in D_r(0,h,H_j)}\;\; \prod_{l=1,2}\one_{\Hcal^l\left(\bigcup_{t\in [\hat s,\hat s_1(0,h,T_j,H_j)]} E^{t,v_l}_{\hat h_t}\cup  E^{\hat s,v_l}_{\hat h}\right)=0} e^{-\rho (\hat s_1(0,h,T_j,H_j)-\hat s)}\right]\nn\\
&\;\; +\E_\Hcal\left[\sum_{ (T_j,H_j)\in \Hcal^2 \,\mbox{:}\, T_j\ge 0}\;\; \prod_{l=1,2}\one_{\Hcal^l\left(\bigcup_{t\in [\hat s,\hat s(0,h,T_j,H_j)]} E^{t,v_l}_{\hat h_t}\cup E^{\hat s,v_l}_{\hat h}\right)=0} e^{-\rho (\hat s(0,h,T_j,H_j)-\hat s)} \right].
\label{eq:T22b}
\end{align}
We write the sum of three terms in (\ref{eq:T22b}) as $\zeta\bbinom{1,2}{2,2,1}+\zeta\bbinom{1,1}{2,2,1}+\zeta\bbinom{1,1}{2,2,2}$. Using Remark~\ref{remark:union} for $l=1$, the region in the first two terms in (\ref{eq:T22b}) satisfies
\begin{equation}
\bigcup_{t\in [\hat s,\hat s_1(0,h,T_j,H_j)]} E^{t,v_1}_{\hat h_t}\cup E^{\hat s, v_1}_{\hat h}= E^{\hat s, v_1}_{\hat h}\cup S_1\cup E^{\hat s_1(0,h,T_j,H_j),v_1}_{\hat h_1(0,h,T_j,H_j)}.
\label{eq:S1a}
\end{equation}
In the last union, $S_1$ is the region outside $E^{\hat s, v_1}_{\hat h}\cup E^{\hat s_1(0,h,T_j,H_j),v_1}_{\hat h_1(0,h,T_j,H_j)}$, but below the hyperbola with equation $u^2-\frac{v_1^2v_2^2}{v_1^2-v_2^2}t^2=h^2$, in the $(t,u)$-coordinate system similar to (\ref{eq:2bird-hyp}), and in the time interval 
\[
\left[\left(1-\frac{v_2^2}{v_1^2}\right)\hat s, \left(1-\frac{v_2^2}{v_1^2}\right) \hat s_1(0,h,T_j,H_j)\right].
\]
For $l=2$, using part~(\ref{monotonic2}) of Proposition~\ref{prop:monotonic-sets}, the increasing property of $\left\{E^{t,v_2}_{\hat h_t}\setminus E^{\hat s, v_2}_{\hat h}\right\}_{t\geq \hat s}$ in the first and second term in (\ref{eq:T22b}), we get that 
\[
\bigcup_{t\in [\hat s,\hat s_1(0,h,T_j,H_j)]} E^{t,v_2}_{\hat h_t}\cup E^{\hat s, v_2}_{\hat h}= E^{\hat s_1(0,h,T_j,H_j),v_2}_{\hat h_1(0,h,T_j,H_j)}\cup  E^{\hat s, v_2}_{\hat h}.
\]
The next head point $(T_j,H_j)$ must be in the region $R_1(h, \hat s, \hat h)$, where
\begin{equation}
R_1(h, \hat s, \hat h):=\left\{(t,u): u\geq 0,\;\;  t\geq \hat s,\;\; v_1^2(t-\hat s)^2+u^2\geq \hat h^2 \text{ and } u^2-\frac{v_1^2v_2^2}{v_1^2-v_2^2} t^2\leq h^2\right\},
\label{eq:R1-221}
\end{equation}
as indicated in the picture~\subref{subfigure:pure_bird221} of Figure~\ref{figure:pure_bird22}. In (\ref{eq:R1-221}), the first inequality $t\geq \hat s$ is due to the discussion in~\ref{case:H221} of this proof. The second inequality $v_1^2(t-\hat s)^2+u^2\geq \hat h^2$, signifies that the next head point must be outside $\overline{E^{\hat{s},v_1}_{\hat{h}}}$. The last inequality $u^2-\frac{v_1^2v_2^2}{v_1^2-v_2^2} t^2\leq h^2$, signifies that the next head point must be below the hyperbola (\ref{eq:2bird-hyp}) corresponding to the head point $(0,h)$, which is a more geometric version of the intersection criteria. Then, implementing the last discussions, the first two terms in  (\ref{eq:T22b}) together are equal to
\begin{align}
\lefteqn{\E_\Hcal\left[\sum_{ (T_j,H_j)\in \Hcal^1 \,\mbox{:}\, \hat s\leq T_j\leq 0,\,T_j\in D_\ell(0,h,H_j)} \one_{\Hcal^l\left(E^{\hat s, v_1}_{\hat h}\cup S_1\cup E^{\hat s_1(0,h,T_j,H_j),v_1}_{\hat h_1(0,h,T_j,H_j)}\right)=0} \one_{\Hcal^l\left(E^{\hat s, v_2}_{\hat h}\cup E^{\hat s_1(0,h,T_j,H_j),v_2}_{\hat h_1(0,h,T_j,H_j)}\right)=0} \right.}\nn\\
&\hspace{4in}\times e^{-\rho (\hat s_1(0,h,T_j,H_j)-\hat s)}\Bigg]\nn\\
&\;\;+\E_\Hcal\left[\sum_{ (T_j,H_j)\in \Hcal^1 \,\mbox{:}\,T_j\in D_r(0,h,H_j)} \one_{\Hcal^l\left(E^{\hat s, v_1}_{\hat h}\cup S_1\cup E^{\hat s_1(0,h,T_j,H_j),v_1}_{\hat h_1(0,h,T_j,H_j)}\right)=0} \one_{\Hcal^l\left(E^{\hat s, v_2}_{\hat h}\cup E^{\hat s_1(0,h,T_j,H_j),v_2}_{\hat h_1(0,h,T_j,H_j)}\right)=0} \right.\nn\\
&\hspace{4in}\times e^{-\rho (\hat s_1(0,h,T_j,H_j)-\hat s)}\Bigg].
\label{eq:T22cx}
\end{align}
Applying the Campbell-Mecke formula in (\ref{eq:T22cx}) and using the region $R_1(h, \hat s, \hat h)$ defined in (\ref{eq:R1-221}), we get
\begin{align}
\zeta\bbinom{1,2}{2,2,1}+\zeta\bbinom{1,1}{2,2,1}
&=2\la_1v_1
\int_{R_1(h, \hat s, \hat h)} \!\!\!\! e^{-\rho (\hat s_1-\hat s)}
\E_\Hcal\left[   \one_{\Hcal^l\left(E^{\hat s,v_1}_{\hat h}\cup S_1\cup E^{\hat s_1,v_1}_{\hat h_1}\right)=0} \one_{\Hcal^l\left(E^{\hat s,v_2}_{\hat h}\cup E^{\hat s_1,v_2}_{\hat h_1}\right)=0}\right] {\rm d}t_1 \, {\rm d}h_1,
\label{eq:T22c}
\end{align}
where in short, we wrote $(\hat s_1,\hat h_1)$ for $(\hat s_1(0,h,t_1,h_1),\hat h_1(0,h,t_1,h_1))$. Observe that the curves $v_1^2(t-\hat s)^2+u^2= \hat h^2 $ and $u^2-\frac{v_1^2v_2^2}{v_1^2-v_2^2} t^2=h^2$ intersect exactly at one point, whose abscissa is $ \frac{v_1^2-v_2^2}{v_1^2}\, \hat s$. So we divide the domain of integration in two parts and use the void probabilities to get that the last expression is equal to  
\begin{align}
\lefteqn{2\la_1v_1
\int_{\frac{v_1^2-v_2^2}{v_1^2}\, \hat s}^{\hat s+\hat h/v_1}\int_{\left(\hat h^2-v_1^2(t_1-\hat s)^2\right)^\half}^{\left(\frac{v_1^2v_2^2}{v_1^2-v_2^2} t_1^2+ h^2\right)^\half} e^{-\rho (\hat s_1-\hat s)}
e^{-  2\la_1v_1\left\vert E^{\hat s,v_1}_{\hat h}\cup S_1\cup E^{\hat s_1,v_1}_{\hat h_1}\right\vert} e^{-  2\la_2v_2\left\vert E^{\hat s,v_2}_{\hat h}\cup E^{\hat s_1,v_2}_{\hat h_1}\right\vert} {\rm d}h_1 \, {\rm d}t_1}\nn\\
&\;\;+ 2\la_1v_1
 \int_{\hat s+ \hat h/v_1}^\infty\int_0^{\left(\frac{v_1^2v_2^2}{v_1^2-v_2^2} t_1^2+ h^2\right)^\half} e^{-\rho (\hat s_1-\hat s)}
e^{-  2\la_1v_1\left\vert E^{\hat s,v_1}_{\hat h}\cup S_1\cup E^{\hat s_1,v_1}_{\hat h_1}\right\vert} e^{-  2\la_2v_2\left\vert E^{\hat s,v_2}_{\hat h}\cup E^{\hat s_1,v_2}_{\hat h_1}\right\vert} {\rm d}h_1 \, {\rm d}t_1.
\label{eq:T22d}
\end{align}
For the third term in  (\ref{eq:T22b}), using Remark~\ref{remark:union} for $l=1$,
\begin{equation}
\bigcup_{t\in [\hat s,\hat s(0,h,T_j,H_j)]} E^{t,v_1}_{\hat h_t}\cup E^{\hat s, v_1}_{\hat h}= E^{\hat s, v_1}_{\hat h}\cup S_2\cup E^{\hat s(0,h,T_j,H_j),v_1}_{\hat h(0,h,T_j,H_j)},
\label{eq:S2a}
\end{equation}
where $S_2$ is the region outside $E^{\hat s, v_1}_{\hat h}\cup E^{\hat s(0,h,T_j,H_j),v_1}_{\hat h(0,h,T_j,H_j)}$, but below the hyperbola given by the equation $u^2-\frac{v_1^2v_2^2}{v_1^2-v_2^2}t^2=h^2$ in the $(t,u)$-coordinate system similar to (\ref{eq:2bird-hyp}) and in the time interval 
\[
\left[\left(1-\frac{v_2^2}{v_1^2}\right)\hat s, \left(1-\frac{v_2^2}{v_1^2}\right) \hat s(0,h,T_j,H_j)\right].
\]
For $l=2$, using part~(\ref{monotonic1}) of Proposition~\ref{prop:monotonic-sets}, the increasing property of $\left\{E^{t,v_2}_{\hat h_t}\setminus E^{\hat s, v_2}_{\hat h}\right\}_{t\geq \hat s}$, we have
\[
\bigcup_{t\in [\hat s,\hat s(0,h,T_j,H_j)]} E^{t,v_2}_{\hat h_t}\cup E^{\hat s, v_2}_{\hat h}= E^{\hat s(0,h,T_j,H_j),v_2}_{\hat h(0,h,T_j,H_j)}\cup  E^{\hat s, v_2}_{\hat h}.
\]
The next responsible head point $(T_j,H_j)$ must belong to the region $R_2(h, \hat s, \hat h)$, where 
\begin{equation}
R_2(h, \hat s, \hat h):=\left\{(t,u): v_2^2(t-\hat s)^2+u^2\geq \hat h^2, \;\; u\geq 0 \text{ and } t\geq 0 \right\}.
\label{eq:R2-222}
\end{equation}
The region $R_2(h, \hat s, \hat h)$ in (\ref{eq:R2-222}) is defined in such a way that, the next head point must lie outside $\overline{E^{\hat s, v_2}_{\hat h}}$ and on the right of $(0,h)$ as discussed in~\ref{case:H222} above. Applying the Campbell-Mecke formula, the third term in  (\ref{eq:T22b}), after invoking the last discussions, equals,
\begin{align}
\zeta\bbinom{1,1}{2,2,2}
&{=}\E_\Hcal\!\!\left[\!\sum_{ (T_j,H_j)\in \Hcal^2 \,\mbox{:}\, T_j\ge 0} \!\!\!\!\one_{\Hcal^l\left(E^{\hat s(0,h,T_j,H_j),v_1}_{\hat h(0,h,T_j,H_j)}\cup S_2 \cup  E^{\hat s, v_1}_{\hat h}\right)=0} \!\!\one_{\Hcal^l\left(E^{\hat s(0,h,T_j,H_j),v_2}_{\hat h(0,h,T_j,H_j)}\cup  E^{\hat s, v_2}_{\hat h}\right)=0} \!\! e^{-\rho (\hat s(0,h,T_j,H_j)-\hat s)} \right]\nn\\
&=2\la_2v_2
\int_{R_2(h,\hat s,\hat h)} e^{-\rho (\hat s'-\hat s)}
\E_\Hcal\left[   \one_{\Hcal^1\left(E^{\hat s',v_1}_{\hat h'}\cup S_2\cup  E^{\hat s,v_1}_{\hat h}\right)=0} \one_{\Hcal^2\left(E^{\hat s',v_2}_{\hat h'}\cup E^{\hat s,v_2}_{\hat h}\right)=0}\right] {\rm d}t_1 \, {\rm d}h_1\nn\\
&= 2\la_2v_2
\int_0^{\hat s+\hat h/v_2}\int_{\left(\hat h^2-v_2^2(t_1-\hat s)^2\right)^\half}^\infty e^{-\rho (\hat s'-\hat s)}
e^{-2\la_1v_1\left\vert E^{\hat s',v_1}_{\hat h'}\cup S_2\cup  E^{\hat s,v_1}_{\hat h}\right\vert} e^{-2\la_2v_2\left\vert E^{\hat s',v_2}_{\hat h'}\cup E^{\hat s,v_2}_{\hat h}\right\vert} {\rm d}h_1 \, {\rm d}t_1\nn\\
&\;\;+ 2\la_2v_2
\int_{\hat s+\hat h/v_2}^\infty\int_0^\infty e^{-\rho (\hat s'-\hat s)}
e^{-2\la_1v_1\left\vert E^{\hat s',v_1}_{\hat h'}\cup S_2\cup  E^{\hat s,v_1}_{\hat h}\right\vert} e^{-2\la_2v_2\left\vert E^{\hat s',v_2}_{\hat h'}\cup E^{\hat s,v_2}_{\hat h}\right\vert} {\rm d}h_1 \, {\rm d}t_1,
\label{eq:T22f}
\end{align}
where in short, we wrote $(\hat s',\hat h')$ for $(\hat s(0,h,t_1,h_1),\hat h(0,h,t_1,h_1))$. We denote the sum of all the terms in (\ref{eq:T22d}) and (\ref{eq:T22f}) as 
\begin{equation}
\zeta_{2,2}(h,t',h'):= \zeta\bbinom{1,2}{2,2,1}+\zeta\bbinom{1,1}{2,2,1}+\zeta\bbinom{1,1}{2,2,2}.
\label{eq:zeta1}
\end{equation}
Finally, we obtain the result by substituting the expression of $\zeta_{2,2}(h,t',h')$ to  
(\ref{eq:T22a}).\qed
\subsection{Proof of Lemma~\ref{lemma:LT_1ijk}, type $\binom{1}{1,2}$ or equivalently $\bbinom{2}{2,1}$}~\label{subsection:Lemma221}
Using Lemma~\ref{lemma:Palm-Lcal-ij}, for the mixed Palm expectations in (\ref{eq:decompLT2}) with $\t_n=2, \t_p=1$, $q=2$ and $f(\Hcal)=e^{-\rho T}$, we have
\begin{align}
\frac{L^{(2)}_{2,1}}{4\la_1\la_2 v_1v_2}\E^0_{\Wcal_{2,1}^{(2)}} [e^{-\rho T}]
& =   \int_{0}^\infty  \int_{0}^h \int_{t^*}^\infty 
\mathbb{E}_\Hcal\Bigg[ e^{-\rho(T\circ\th_{\hat s_1})}  \prod_{l=1,2}\one_{\Hcal^l\left(E^{\hat{s}_1,v_l}_{\hat{h}_1}\right)=0}
\Bigg] {\rm d}t'  {\rm d}h' {\rm d}h \nn\\ 
&\;+  \int_{0}^\infty  \int_h^\infty \int_0^{\infty}
\mathbb{E}_\Hcal\Bigg[ e^{-\rho(T\circ\th_{\hat s_1})} \prod_{l=1,2}\one_{\Hcal^l\left(E^{\hat{s}_1,v_l}_{\hat{h}_1}\right)=0}
\Bigg] {\rm d}t'  {\rm d}h' {\rm d}h.
\label{eq:T1211}
\end{align}
where $(\hat{s}_1,\hat{h}_1)=(\hat{s}_1(0,h,-t',h'),\hat{h}_1(0,h,-t',h'))$ is the first intersection of the birds at $(t_p,h_p)=(0,h)$ and $(t_n,h_n)=(-t',h')$, which are of type $\t_p=2$ and $\t_n=1$, respectively. 

The station corresponding to the head point $(-t',h')$ is the serving station after time $\hat s_1$. Since $(-t',h')$ is of type $1$, we define $\hat h_1(t):=\left(v_1^2(t+t')^2+h'^2\right)^\half$, which write in short as $\hat h_t$. Observe that 
\[
T\circ\th_{\hat s_1}=T_1(\hat s_1)-\hat s_1,
\]
where $T_1(\hat s_1)$ is the time point of the first handover after time $\hat s_1$. The intersection point $(\hat{s}_1,\hat{h}_1)$ corresponding to the typical handover in this case, the second intersection $(\hat{s}_2,\hat{h}_2)$ will correspond to the next handover, if there is no radial bird that intersects the bird at $(0,h)$ before time $\hat s_2$, i.e., $T_1(\hat s_1)<\hat s_2$ or the next handover happens at time $\hat s_2$, in which case $T_1(\hat s_1)=\hat s_2$. The reason for the next handover to happen at most before $\hat s_2$, will be clear later in the following cases.
\begin{case}[\textbf{$T_1(\hat s_1)<\hat s_2$.}]\label{case:H121l} The next handover can be given by a bird with head point, say $(T_j,H_j)$, outside the region $\overline{E^{\hat{s}_1,v_1}_{\hat{h}_1}}$ or $\overline{E^{\hat{s}_1,v_2}_{\hat{h}_1}}$, depending on its type. Thus we have two sub-cases: 

\begin{subcase}[\textbf{\em $(T_j,H_j)$ is of type 1.}]\label{subcase:H121lx} In this case, $(T_j,H_j)$ must be at least on the right of $(-t',h')$, i.e., $T_j\geq -t'$. Otherwise, the next radial bird with head point $(T_j,H_j)$ being outside $\overline{E^{\hat s_1, v_1}_{\hat h_1}}$ and $T_j< -t'$, intersects the one with head at $(-t',h')$, at a point that is either contained in $\Lcal^+_e(v_1,v_2)$ or contributes to a handover before time $\hat s_1$, which is not of interest. Then, we must have $T_j\geq -t'$ and the unexplored region is $Q_{-t'}\setminus \overline{E^{\hat s_1,v_1}_{\hat h_1}}$. The type of the pair of consecutive handovers is $\bbinom{2,1}{2,1,1}$.
There is just one intersection, say $(\hat{s}(-t',h',T_j,H_j),\hat{h}(-t',h',T_j,H_j))$ between the radial birds at $(-t',h')$ and $(T_j,H_j)$ and the handover is given by that, see picture~(\subref{subfigure:mixed_bird12a}) of Figure~\ref{figure:mixed_bird121}. Thus we have $T_1(\hat s_1):=\hat s(-t',h',T_j,H_j)$. In this case, the extra region $\interior{\left(\bigcup_{t\in [\hat s_1,\hat s(-t',h',T_j,H_j)]} E^{t,v_l}_{\hat h_t}\setminus E^{\hat s, v_l}_{\hat h}\right)}$ must have no point of $\Hcal^l$ for $l=1,2$. The union is due to Definition~\ref{definition:T1-2speed} and to the fact that the individual sets in the union may lack monotonicity.
\end{subcase}

\begin{subcase}[\textbf{\em  $(T_j,H_j)$ is of type 2.}]\label{subcase:H121ly} Here it must be the case that $T_j\geq 0$. Otherwise, if $T_j<0$, then the two birds with head point $(T_j,H_j)$ and $(0,h)$ intersect at a point before time $\hat s_1$, which is not relevant in this case. Then, we must have $T_j\geq 0\geq -t'$ and the unexplored region is $Q_{0}\setminus \overline{E^{\hat s_1,v_2}_{\hat h_1}}$. The type of the pair of consecutive handovers is $\bbinom{2,1}{2,1,2}$. By Lemma~\ref{lemma:in1out2}, there are two intersections $(\hat s_1(-t',h',T_j,H_j),\hat h_1(-t',h',T_j,H_j))$ and $(\hat s_2(-t',h',T_j,H_j),\hat h_2(-t',h',T_j,H_j))$ between the birds at $(-t',h'), (T_j,H_j)$. The first intersection cannot give rise to a handover, as it is contained in $\Lcal^+_e(v_1,v_2)$. So the next handover must be given by the second intersection, see picture~(\subref{subfigure:mixed_bird12b}) of Figure~\ref{figure:mixed_bird121}. Thus we have $T_1(\hat s_1):=\hat s_2(-t',h',T_j,H_j)$. For the last intersection to give a handover, the extra region $\interior{\left(\bigcup_{t\in [\hat s_1,\hat s_2(-t',h',T_j,H_j)]} E^{t,v_l}_{\hat h_t}\setminus E^{\hat s_1, v_l}_{\hat h_1}\right)}$ must have no point of $\Hcal^l$ for $l=1,2$. The union is due to Definition~\ref{definition:T1-2speed} and the fact that individual sets in the union potentially lack  monotonicity.
\end{subcase}
\end{case}
\begin{figure}[ht!]
 \centering
      \begin{subfigure}[t]{0.45\linewidth}
       \centering
      \begin{tikzpicture}[scale=0.7, every node/.style={scale=0.7}]
\pgftransformxscale{0.97}  
    \pgftransformyscale{0.97}    
    \draw[->] (-5, 0) -- (4, 0) node[right] {$t$};
   \draw[red, domain=-1:2, smooth] plot (\x, {(16*\x*\x- 16*2*0.5*\x+16*(0.5)^2+2^2)^0.5});
    \draw[red](0.5,2) node{$\bullet$};
    \draw[](-0.8,2) node{$(-t',h')$};
    \draw[blue](2,2.5) node{$\bullet$};
    \draw[](2.7,2.2) node{$(0,h)$};
    \draw[blue, domain=-1.5:4, smooth] plot (\x, {((3/2)^2*\x*\x-2*(3/2)^2*2*\x+(3/2)^2*2^2+2.5^2)^0.5});
    \draw[red](1.2,1.4) node{$\bullet$};
    \draw[](1.4,0.9) node{$(t_1,h_1)$};
    \draw[red, domain=-0.3:2.3, smooth] plot (\x, {(16*\x*\x-32*1.2*\x+16*1.2^2+1.4^2)^0.5});
    \draw[](-1.7,4.8) node{$(\hat s_1,\hat h_1)$};
    \draw[](1.7,3.4) node{$(\hat s_2,\hat h_2)$};
    \draw[](1,4.4) node{$(\tilde s,\tilde h)$};
    \draw[->] (0.9, 4.15) -- (0.8, 2.5);
     \draw[blue]  (-3.6,0) arc (-180:-360:3.05 and 4.5);
     \draw[red]  (0.6,0) arc (0:180:1.15 and 4.5);
     \draw[red]  (0.15,0) arc (-180:-360:0.6 and 2.25);
     \draw[blue]  (2.16,0) arc (0:180:1.4 and 2.25);   
    \draw[blue](-2.5,0.6) node{$E^{\hat s_1, v_2}_{\hat h_1}$};
    \draw[red](-1.1,0.9) node{$E^{\hat s_1, v_1}_{\hat h_1}$};
    \draw[blue](1.8,-0.5) node{$E^{\tilde s,v_2}_{\tilde h}$};
     \draw[red](0.6,-0.5) node{$E^{\tilde s,v_1}_{\tilde h}$};
    \end{tikzpicture}
    \caption{For the intersections $(\hat s_1, \hat h_1)$ and $(\tilde s, \tilde h)$ to represent two consecutive  handovers, the regions \red{$E^{\hat s_1, v_1}_{\hat h_1} \cup E^{\tilde s,v_1}_{\tilde h}$} and  \blue{$E^{\hat s_1, v_2}_{\hat h_1}\cup E^{\tilde s,v_2}_{\tilde h}$} must have no points from $\Hcal^1$ and $\Hcal^2$, respectively. Here $(\tilde s, \tilde h)$ is a realization of $(\hat s(-t',h',T_j,H_j),\hat h(-t',h',T_j,H_j))$.}
\label{subfigure:mixed_bird12a}
    \end{subfigure}
    \hspace{0.1in}
    \centering
      \begin{subfigure}[t]{0.45\linewidth}
       \centering
      \begin{tikzpicture}[scale=0.7, every node/.style={scale=0.7}]
\pgftransformxscale{0.9}  
    \pgftransformyscale{0.9}    
\draw[->] (-5, 0) -- (4, 0) node[right] {$t$};
   \draw[red, domain=-1:2, smooth] plot (\x, {(16*\x*\x- 16*2*0.5*\x+16*(0.5)^2+2^2)^0.5});
    \draw[red](0.5,2) node{$\bullet$};
    \draw[](-0.8,2) node{$(-t',h')$};
    \draw[blue](2,2.5) node{$\bullet$};
    \draw[](2.7,2.2) node{$(0,h)$};
    \draw[blue, domain=-1.5:4, smooth] plot (\x, {((3/2)^2*\x*\x-2*(3/2)^2*2*\x+(3/2)^2*2^2+2.5^2)^0.5});
    \draw[blue](2.6,0.8) node{$\bullet$};
    \draw[](3.3,0.6) node{$(t_1,h_1)$};
    \draw[blue, domain=-1.5:4, smooth] plot (\x, {((3/2)^2*\x*\x-2*(3/2)^2*2.6*\x+(3/2)^2*2.6^2+0.8^2)^0.5});
    \draw[](-1.7,4.8) node{$(\hat s_1,\hat h_1)$};
    \draw[](1.7,3.4) node{$(\hat s_2,\hat h_2)$};
    \draw[](1,4.6) node{$(\tilde s,\tilde h)$};
    \draw[->] (0.9, 4.15) -- (0.85, 2.8);
     \draw[blue]  (-3.6,0) arc (-180:-360:3.05 and 4.5);
     \draw[red]  (0.6,0) arc (0:180:1.15 and 4.5);
     \draw[red]  (0.25,0) arc (-180:-360:0.66 and 2.63);
     \draw[blue]  (2.65,0) arc (0:180:1.78 and 2.63);   
    \draw[blue](-2.5,0.6) node{$E^{\hat s_1, v_2}_{\hat h_1}$};
    \draw[red](-1.1,0.9) node{$E^{\hat s_1, v_1}_{\hat h_1}$};
    \draw[blue](2,-0.5) node{$E^{\tilde s,v_2}_{\tilde h}$};
     \draw[red](0.9,-0.5) node{$E^{\tilde s,v_1}_{\tilde h}$};
    \end{tikzpicture}
    \caption{For the intersections $(\hat s_1, \hat h_1)$ and $(\tilde s, \tilde h)$ to represent two consecutive  handovers, the regions \red{$E^{\hat s_1, v_1}_{\hat h_1} \cup E^{\tilde s,v_1}_{\tilde h}$} and  \blue{$E^{\hat s_1, v_2}_{\hat h_1}\cup E^{\tilde s,v_2}_{\tilde h}$} must have no points from $\Hcal^1$ and $\Hcal^2$, respectively. Here $(\tilde s, \tilde h)$ is a realization of $(\hat s_2(-t',h',T_j,H_j),\hat h_2(-t',h',T_j,H_j))$.}
\label{subfigure:mixed_bird12b}
    \end{subfigure}
    \captionsetup{width=0.9\linewidth}
    \caption{The cases corresponding to the first and second term in (\ref{eq:T1212}), when the head point $(T_j,H_j)=(t_1,h_1)$ is of type $1$ and type $2$, respectively.}    \label{figure:mixed_bird121}
\end{figure}
\begin{case}
[\textbf{$T_1(\hat s_1)=\hat s_2$.}]\label{case:H121a} In this case, the next handover is given by the second intersection $(\hat s_2,\hat h_2)$ of the radial birds with head points $(0,h)$, $(-t',h')$ and the extra region $\interior{\left(\bigcup_{t\in [\hat s_1,\hat s_2]} E^{t,v_l}_{\hat h_t}\setminus E^{\hat s_1, v_l}_{\hat h_1}\right)}$ has no point of $\Hcal^l$, for $l=1,2$, see Figure~\ref{figure:mixed_bird122}. The next handover in question happens as late as $\hat s_2$, since for any $\hat s>\hat s_2$
\[
\bigcup_{t\in [\hat s_1,\hat s_2]} E^{t,v_l}_{\hat h_t}\setminus E^{\hat s_1, v_l}_{\hat h_1}
\subset \bigcup_{t\in [\hat s_1,\hat s]} E^{t,v_l}_{\hat h_t}\setminus E^{\hat s_1, v_l}_{\hat h_1}.
\]
Thus the type of the pair of consecutive handovers is $\bbinom{2,1}{2,1,2^*}$.
\end{case}
Based on last discussions we have the following decomposition of the inner expectation in (\ref{eq:T1211}),
\begin{align}
\lefteqn{\mathbb{E}_\Hcal\Bigg[ e^{-\rho(T_1(\hat s_1)-\hat s_1)} \prod_{l=1,2}\one_{\Hcal^l\left(E^{\hat{s}_1,v_l}_{\hat{h}_1}\right)=0}\Bigg]}\nn\\
&=  \E_\Hcal\Bigg[ \one_{\exists (T_j,H_j)\in \Hcal^1 \,\mbox{:}\, T_j\geq -t' ,\, \bigcap_{l=1,2}\{\Hcal^l\interior{\left(\bigcup_{t\in [\hat s_1,\hat s(-t',h',T_j,H_j)]} E^{t,v_l}_{\hat h_t}\setminus E^{\hat s, v_l}_{\hat h}\right)}=0\}} \one_{\hat s_1\leq \hat s(-t',h',T_j,H_j)\leq \hat s_2} \nn\\
& \hspace{3.4in}\times e^{-\rho (\hat s(-t',h',T_j,H_j)-\hat s_1)}\prod_{l=1,2}\one_{\Hcal^l\left(E^{\hat{s}_1,v_l}_{\hat{h}_1}\right)=0}\Bigg] \nn\\
&\;\;+ \E_\Hcal\Bigg[ \one_{\exists (T_j,H_j)\in \Hcal^2 \,\mbox{:}\, T_j\geq 0,\, \bigcap_{l=1,2}\{\Hcal^l\interior{\left(\bigcup_{t\in [\hat s_1,\hat s_2(-t',h',T_j,H_j)]} E^{t,v_l}_{\hat h_t}\setminus E^{\hat s, v_l}_{\hat h}\right)}=0\}} \one_{\hat s_1\leq \hat s_2(-t',h',T_j,H_j)\leq \hat s_2}\nn\\
& \hspace{3.4in}\times e^{-\rho (\hat s_2(-t',h',T_j,H_j)-\hat s_1)}\prod_{l=1,2}\one_{\Hcal^l\left(E^{\hat{s}_1,v_l}_{\hat{h}_1}\right)=0}\Bigg] \nn\\
&\quad + \E_\Hcal\left[ e^{-\rho (\hat s_2-\hat s_1)} \prod_{l=1,2}\one_{\Hcal^l\interior{\left(\bigcup_{t\in [\hat s_1,\hat s_2]} E^{t,v_l}_{\hat h_t}\setminus E^{\hat s_1, v_l}_{\hat h_1}\right)}=0} \prod_{l=1,2}\one_{\Hcal^l\left(E^{\hat{s}_1,v_l}_{\hat{h}_1}\right)=0}\right]\nn\\
&:=\zeta\bbinom{2,1}{2,1,1}+\zeta\bbinom{2,1}{2,1,2}+\zeta\bbinom{2,1}{2,1,2^*},
\label{eq:T1212}
\end{align}
where the pair type marked with $\cdot^*$ means that we have recurrence of the previous station as an upcoming station, seen in the last column of Table~\ref{tab:notation-tu}.
Similarly to (\ref{eq:T12h}), there exists a unique head point satisfying the condition in each term in (\ref{eq:T1212}), and hence the sum can be written as
\begin{align}
\lefteqn{\E_\Hcal\Bigg[ \sum_{ (T_j,H_j)\in \Hcal^1 \,\mbox{:}\, T_j\geq -t'} \prod_{l=1,2}\one_{\Hcal^l(\bigcup_{t\in [\hat s_1,\hat s(-t',h',T_j,H_j)]} E^{t,v_l}_{\hat h_t}\cup E^{\hat s, v_l}_{\hat h})=0} \one_{\hat s_1\leq \hat s(-t',h',T_j,H_j)\leq \hat s_2}}\nn\\
&\hspace{4in}\times e^{-\rho (\hat s(-t',h',T_j,H_j)-\hat s_1)}\Bigg] \nn\\
&\;+ \E_\Hcal\left[ \sum_{(T_j,H_j)\in \Hcal^2 \,\mbox{:}\, T_j\geq 0}\prod_{l=1,2}\one_{\Hcal^l(\bigcup_{t\in [\hat s_1,\hat s_2(-t',h',T_j,H_j)]} E^{t,v_l}_{\hat h_t}\cup E^{\hat s, v_l}_{\hat h})=0} \one_{\hat s_1\leq \hat s_2(0,h,T_j,H_j)\leq \hat s_2}\right.\nn\\ 
&\hspace{4in}\times e^{-\rho (\hat s_1\leq \hat s(0,h,T_j,H_j)-\hat s_1)}\Bigg] \nn\\
&\; + \E_\Hcal\left[ e^{-\rho (\hat s_2-\hat s_1)} \prod_{l=1,2}\one_{\Hcal^l( \bigcup_{t\in [\hat s_1,\hat s_2]} E^{t,v_l}_{\hat h_t}\cup E^{\hat s_1, v_l}_{\hat h_1})=0} \right].\label{eq:T121b}
\end{align}
\begin{figure}[ht!]
\begin{tikzpicture}[scale=0.7, every node/.style={scale=0.7}]
\pgftransformxscale{1}  
\pgftransformyscale{1}    
\draw[->] (-5, 0) -- (4, 0) node[right] {$t$};
   \draw[red, domain=-0.8:1.7, smooth] plot (\x, {(16*\x*\x- 16*2*0.5*\x+16*(0.5)^2+2^2)^0.5});
    \draw[red](0.5,2) node{$\bullet$};
    \draw[](-0.8,2) node{$(-t',h')$};
    \draw[blue](2,2.5) node{$\bullet$};
    \draw[](1.7,2) node{$(0,h)$};
    \draw[blue, domain=-1.5:4.4, smooth] plot (\x, {((3/2)^2*\x*\x-2*(3/2)^2*2*\x+(3/2)^2*2^2+2.5^2)^0.5});
    \draw[](-1.7,4.8) node{$(\hat s_1,\hat h_1)$};
    \draw[](1.7,3.4) node{$(\hat s_2,\hat h_2)$};
     \draw[blue]  (-3.6,0) arc (-180:-360:3.05 and 4.5);
     \draw[red]  (0.6,0) arc (0:180:1.15 and 4.5);
     \draw[red]  (0.28,0) arc (-180:-360:0.8 and 2.9);
     \draw[blue]  (2.9,0) arc (0:180:1.8 and 2.9);   
    \draw[blue](-2.5,0.6) node{$E^{\hat s_1, v_2}_{\hat h_1}$};
    \draw[red](-1.1,0.9) node{$E^{\hat s_1, v_1}_{\hat h_1}$};
     \draw[red](1.2,0.7) node{$E^{\hat s_2,v_1}_{\hat h_2}$};
    \draw[blue](3.8,0.7) node{$E^{\hat s_2,v_2}_{\hat h_2}$};
    \end{tikzpicture}
    \captionsetup{width=0.9\linewidth}
    \caption{The case corresponding to the last term in (\ref{eq:T1212}). For the intersections $(\hat s_1, \hat h_1)$ and $(\hat s_2, \hat h_2)$ to represent two consecutive  handovers, the regions \red{$E^{\hat s_1, v_1}_{\hat h_1} \cup E^{\hat s_2, v_1}_{\hat h_2}$} and  \blue{$E^{\hat s_1, v_2}_{\hat h_1}\cup E^{\hat s_2, v_2}_{\hat h_2}$} must have no points from $\Hcal^1$ and $\Hcal^2$, respectively.}
\label{figure:mixed_bird122}
\end{figure}

For the first term in (\ref{eq:T121b}), by part~(\ref{monotonic1}) of Proposition~\ref{prop:monotonic-sets}, the increasing property of $\left\{E^{t,v_l}_{\hat h_t}\setminus E^{\hat s_1, v_l}_{\hat h_1}\right\}_{t\geq \hat s_1}$, we have 
\[
\bigcup_{t\in [\hat s_1,\hat s(-t',h',T_j,H_j)]} E^{t,v_l}_{\hat h_t}\cup E^{\hat s_1, v_l}_{\hat h_1}=  E^{\hat s(-t',h',T_j,H_j),v_l}_{\hat h(-t',h',T_j,H_j)}\cup E^{\hat s_1, v_l}_{\hat h_1}.
\]

The point $(T_j,H_j)$ must lie in the region $R_1(h,t',h')\subset [-t',\infty)\times \R^+\setminus E^{\hat s_1,v_1}_{\hat h_1}$ be the region defined as 
\[R_1(h,t',h'):=\left\{(t,u)\,\mbox{:}\, \hat s_1\leq \hat s(-t',h',t,u)\leq \hat s_2\right\},\]
which contains such a possible head point, as seen in~\ref{subcase:H211x}. Using Lemma~\ref{lemma:int1mid2} it can be proved that the region $R_1(h,t',h')=E^{\hat s_2,v_1}_{\hat h_2}\setminus E^{\hat s_1,v_1}_{\hat h_1}$. By the Campbell-Mecke formula, after implementing the last discussions, the first term in (\ref{eq:T121b}) equals 
\begin{align}
\zeta\bbinom{2,1}{2,1,1}&=2\la_1v_1\int_{E^{\hat s_2,v_1}_{\hat h_2}\setminus E^{\hat s_1,v_1}_{\hat h_1}} e^{-\rho (\hat s(-t',h',t_1,h_1)-\hat s_1)} \E_\Hcal\left[  \prod_{l=1,2}\one_{\Hcal^l\left(E^{\hat s(-t',h',t_1,h_1),v_l}_{\hat h(-t',h',t_1,t_1)}\cup E^{\hat s_1,v_l}_{\hat h_1}\right)=0}\right]{\rm d}t_1  {\rm d}h_1\nn\\
&=2\la_1v_1\int_{E^{\hat s_2,v_1}_{\hat h_2}\setminus E^{\hat s_1,v_1}_{\hat h_1}} e^{-\rho (\hat s(-t',h',t_1,h_1)-\hat s_1)} e^{-\sum_{l=1,2}2\la_lv_l\left\vert E^{\hat s(-t',h',t_1,h_1),v_l}_{\hat h(-t',h',t_1,t_1)}\cup E^{\hat s_1,v_l}_{\hat h_1}\right\vert}
{\rm d}t_1  {\rm d}h_1.
\label{eq:T121c}
\end{align}
For the second term in (\ref{eq:T121b}), by the increasing property of $\left\{E^{t,v_l}_{\hat h_t}\setminus E^{\hat s_1, v_l}_{\hat h_1}\right\}_{t\geq \hat s_1}$ from part~(\ref{monotonic1}) of Proposition~\ref{prop:monotonic-sets}, we have
\[
\bigcup_{t\in [\hat s_1,\hat s_2(-t',h',T_j,H_j)]} E^{t,v_l}_{\hat h_t}\cup E^{\hat s_1, v_l}_{\hat h_1}=  E^{\hat s_2(-t',h',T_j,H_j),v_l}_{\hat h_2(-t',h',T_j,H_j)}\cup E^{\hat s_1, v_l}_{\hat h_1}.
\]
The point $(T_j,H_j)$ must lie in the region $R_2(h,t',h')\subset (\R^+)^2\setminus E^{\hat s_1,v_2}_{\hat h_1}$ defined as 
\[
R_2(h,t',h'):=\left\{(t,u)\,\mbox{:}\, \hat s_1\leq \hat s_2(-t',h',t,u)\leq \hat s_2\right\},
\]
that contains the required head points, as discussed in~\ref{subcase:H211y}. It can be proved using Lemma~\ref{lemma:int1mid2}, that $R_2(h,t',h')=E^{\hat s_2,v_2}_{\hat h_2}\setminus E^{\hat s_1,v_2}_{\hat h_1}$. 
Using this along with the Campbell-Mecke formula, we get that the second term in (\ref{eq:T121b}) equals 
\hspace{-0.2in}
\begin{align}
\zeta\bbinom{2,1}{2,1,2}&=2\la_2v_2\int_{E^{\hat s_2,v_2}_{\hat h_2}\setminus E^{\hat s_1,v_2}_{\hat h_1}} e^{-\rho (\hat s_2(-t',h',t_1,h_1)-\hat s_1)} \E_\Hcal\left[  \prod_{l=1,2}\one_{\Hcal^l\left(E^{\hat s_2(-t',h',t_1,h_1),v_l}_{\hat h_2(-t',h',t_1,t_1)}\cup E^{\hat s_1,v_l}_{\hat h_1}\right)=0}\right]{\rm d}t_1  {\rm d}h_1.
\label{eq:T121dx}
\end{align}
Using void probabilities in (\ref{eq:T121dx}) we obtain that
\begin{align}
\zeta\bbinom{2,1}{2,1,2}
&=2\la_2v_2\int_{E^{\hat s_2,v_2}_{\hat h_2}\setminus E^{\hat s_1,v_2}_{\hat h_1}} e^{-\rho (\hat s_2(-t',h',t_1,h_1)-\hat s_1)} e^{-\sum_{l=1,2}2\la_lv_l\left\vert E^{\hat s_2(-t',h',t_1,h_1),v_l}_{\hat h_2(-t',h',t_1,t_1)}\cup E^{\hat s_1,v_l}_{\hat h_1}\right\vert}
{\rm d}t_1  {\rm d}h_1.
\label{eq:T121d}
\end{align}
For the third term in (\ref{eq:T121b}), we have
\[
\bigcup_{t\in [\hat s_1,\hat s_2]} E^{t,v_l}_{\hat h_t}\cup E^{\hat s_1, v_l}_{\hat h_1}=  E^{\hat s_2,v_l}_{\hat h_2(-t',h',T_j,H_j)}\cup E^{\hat s_1, v_l}_{\hat h_1},
\]
and hence this third term equals
\begin{align}
\zeta\bbinom{2,1}{2,1,2^*} &= e^{-\rho (\hat s_2-\hat s_1)} e^{-\sum_{l=1,2}2\la_lv_l\left\vert  E^{\hat s_2,v_l}_{\hat h_2}\cup E^{\hat s_1,v_l}_{\hat h_1}\right\vert}.
\label{eq:12s1s2}
\end{align}
Let us denote the sum of all the terms in (\ref{eq:T121c})--(\ref{eq:12s1s2}) as $\zeta^{(2)}_{2,1}(\rho, t',h',h)=\zeta\bbinom{2,1}{2,1,1}+\zeta\bbinom{2,1}{2,1,2}+\zeta\bbinom{2,1}{2,1,2^*}$. Using the expression for $\zeta^{(2)}_{2,1}(\rho, t',h',h)$ in (\ref{eq:T1211}) yields
\begin{align}
\frac{L^{(2)}_{2,1}}{4\la_1\la_2 v_1v_2}\E^0_{\Wcal_{2,1}^{(2)}} [e^{-\rho T}]
& =  \int_{0}^\infty  \int_{0}^h \int_{t^*}^\infty 
\zeta^{(2)}_{2,1}(\rho, t',h',h) {\rm d}t'  {\rm d}h' {\rm d}h  +  \int_{0}^\infty  \int_h^\infty \int_0^{\infty}
\zeta^{(2)}_{2,1}(\rho, t',h',h) {\rm d}t'  {\rm d}h' {\rm d}h \nn\\
&:=\xi_{2,1}^{(2)}(\rho, v_1,v_2),
\label{eq:T121g}
\end{align}
which is a function of $\rho, v_1, v_2$ only.\qed
\subsection{Proof of Lemma~\ref{lemma:LT_1ijk}, type $\binom{2}{1,2}$ or equivalently $\bbinom{1}{1,2}$}~\label{subsection:Lemma112}
Using Lemma~\ref{lemma:Palm-Lcal-ij}, the mixed Palm expectations in (\ref{eq:decompLT2}) with $\t_n=2, \t_p=1$, $q=1$, and $f(\Hcal)=e^{-\rho T}$, we have
\begin{align}
\frac{L^{(1)}_{1,2}}{4\la_1\la_2 v_1v_2}\E^0_{\Wcal_{1,2}^{(1)}} [e^{-\rho T}]
& =   \int_{0}^\infty  \int_{0}^h \int_{t^*}^\infty 
\mathbb{E}_\Hcal\Bigg[ e^{-\rho(T\circ\th_{\hat s_2})} \prod_{l=1,2}\one_{\Hcal^l\left(E^{\hat{s}_2,v_l}_{\hat{h}_2}\right)=0}
\Bigg] {\rm d}t'  {\rm d}h' {\rm d}h \nn\\ 
&\;+  \int_{0}^\infty  \int_h^\infty \int_0^{\infty}
\mathbb{E}_\Hcal\Bigg[e^{-\rho(T\circ\th_{\hat s_2})} \prod_{l=1,2}\one_{\Hcal^l\left(E^{\hat{s}_2,v_l}_{\hat{h}_2}\right)=0}
\Bigg] {\rm d}t'  {\rm d}h' {\rm d}h.
\label{eq:T1221}
\end{align}
where $(\hat{s}_2,\hat{h}_2)=(\hat{s}_2(0,h,-t',h'),\hat{h}_2(0,h,-t',h'))$ is the second intersection of the birds at $(t_n,h_n)=(0,h), (t_p,h_p)=(-t',h')$, which are of type $\t_n=2$ and $\t_p=1$, respectively. We write this intersection  $(\hat{s}_2,\hat{h}_2)$ in short. Note that the station corresponding to the head point $(0,h)$ is the serving station after time $\hat s_2$. Since $(0,h)$ is of type $2$, we define $\hat h_2(t):=\left(v_2^2t^2+h^2\right)^\half$, which we write in short as $\hat h_t$. Notice that 
\[
T\circ\th_{\hat s_2}:= T_1(\hat s_2)-\hat s_2,
\]
where $T_1(\hat s_2)$ is the handover that happens for the first time after time $\hat s_2$. The next handover point can be found by looking for the first head point, say $(T_j,H_j)$, the radial bird of which intersects the one with head point $(0,h)$. Depending on the type of $(T_j,H_j)$ there are two cases as follows:

\begin{case}[\textbf{\em $(T_j,H_j)$ is of type 1.}]\label{case:H1221} For the point $(T_j,H_j)$ to be eligible for a handover, it must be outside $\overline{E^{\hat s_2,v_1}_{\hat h_2}}$ and satisfy $T_j\geq \hat s_2$. Otherwise, if $T_j<\hat s_2$, then the intersection between the birds with head at $(T_j,H_j)$ and $(0,h)$, is either contained in $\Lcal^+_e(v_1,v_2)$ or it produces a handover before time $\hat s_2$. So it must be the case that $T_j\geq \hat s_2$ and the unexplored region is $Q_{\hat s_2}\setminus \overline{E^{\hat s_2,v_1}_{\hat h_2}}$.

In addition to that, if $\hat s_2< 0$, we must have $T_j\in D_\ell(0,h,H_j)$ or $T_j\in D_r(0,h,H_j)$, when $\hat s_2\leq T_j\leq 0$ and the type of the pair of consecutive handovers can be $\bbinom{1,2}{1,2,1}$ or $\bbinom{1,1}{1,2,1}$, depending on whether $\hat s_2\leq T_j<0$ or $T_j\geq 0$. If $\hat s_2\geq 0$, we must have $T_j\in D_r(0,h,H_j)$ and the type of the pair of consecutive handovers can be $\bbinom{1,1}{1,2,1}$ only.
Suppose the two intersection points are $(s_1(0,h,T_j,H_j), \hat h_1(0,h,T_j,H_j))$ and $(s_2(0,h,T_j,H_j), \hat h_2(0,h,T_j,H_j))$, both of which can be eligible for handovers. For the two intersection points to represent handovers we must have no point of $\Hcal^l$, for $l=1,2$, in the extra regions $\interior{\left(\bigcup_{t\in [\hat s_2,\hat s_1(0,h,T_j,H_j)]} E^{t,v_l}_{\hat h_t}\setminus E^{\hat s_2, v_l}_{\hat h_2}\right)}$ and $\interior{\left(\bigcup_{t\in [\hat s_2,\hat s_2(0,h,T_j,H_j)]} E^{t,v_l}_{\hat h_t}\setminus E^{\hat s_2, v_l}_{\hat h_2}\right)}$, respectively, beyond $\overline{E^{\hat s_2, v_l}_{\hat h_2}}$. Since $\hat s_1(0,h,T_j,H_j)\leq \hat s_2(0,h,T_j,H_j)$, we have 
\[
\bigcup_{t\in [\hat s_2,\hat s_1(0,h,T_j,H_j)]} E^{t,v_l}_{\hat h_t}\setminus E^{\hat s_2, v_l}_{\hat h_2}\subset \bigcup_{t\in [\hat s_2,\hat s_2(0,h,T_j,H_j)]} E^{t,v_l}_{\hat h_t}\setminus E^{\hat s_2, v_l}_{\hat h_2},
\]
Hence, we choose the first intersection to be responsible for the next handover. Thus we have $T_1(\hat s_2):= \hat s_1(0,h,T_j,H_j)$ and the extra region $\interior{\left(\bigcup_{t\in [\hat s_2, \hat s_1(0,h,T_j,H_j)]}E^{t, v_l}_{\hat h_t}\setminus E^{\hat s_2,v_l}_{\hat h_2}\right)}$ has no points of $\Hcal^l$, for $l=1,2$, where $\hat h_t=(v_2^2t^2+h^2)^\half$, see picture~(\subref{subfigure:mixed_bird122a}) of Figure~\ref{figure:mixed_bird122a}. This union is again due to Definition~\ref{definition:T1-2speed} and to the fact that the individual sets in the union may not be monotonic.
\end{case}
\begin{figure}[ht!]
\centering
\begin{subfigure}[t]{0.45\linewidth}\begin{tikzpicture}[scale=0.7, every node/.style={scale=0.7}]
\pgftransformxscale{1}  
\pgftransformyscale{1}    
\draw[->] (-2, 0) -- (5, 0) node[right] {$t$};
\draw[red, domain=-0.7:1.7, smooth] plot (\x, {(16*\x*\x- 16*2*0.5*\x+16*(0.5)^2+2^2)^0.5});
\draw[red](0.5,2) node{$\bullet$};
    \draw[](-0.8,2) node{$(-t',h')$};
    \draw[blue](2,2.5) node{$\bullet$};
    \draw[](2.7,2.2) node{$(0,h)$};
    \draw[blue, domain=-1:4.4, smooth] plot (\x, {((3/2)^2*\x*\x-2*(3/2)^2*2*\x+(3/2)^2*2^2+2.5^2)^0.5});
    \draw[red](3.2,1.7) node{$\bullet$};
    \draw[](3.5,1.2) node{$(t_1,h_1)$};
    \draw[red, domain=2:4.3, smooth] plot (\x, {(16*\x*\x-32*3.2*\x+16*3.2^2+1.7^2)^0.5});
    \draw[](-1.7,4.8) node{$(\hat s_1,\hat h_1)$};
    \draw[](1.7,3.4) node{$(\hat s_2,\hat h_2)$};
    \draw[](3,3.4) node{$(\tilde s,\tilde h)$};
     \draw[red]  (0.28,0) arc (-180:-360:0.8 and 2.9);
     \draw[blue]  (2.9,0) arc (0:180:1.8 and 2.9); 
     \draw[blue]  (1,0) arc (-180:-360:1.62 and 2.73);
     \draw[red]  (3.35,0) arc (0:180:0.67 and 2.73);
    \draw[blue](-0.2,0.6) node{$E^{\hat s_2,v_2}_{\hat h_2}$};
     \draw[red](1,0.6) node{$E^{\hat s_2,v_1}_{\hat h_2}$};
      \draw[red](2.5,0.6) node{$E^{\tilde s,v_1}_{\tilde h}$};
     \draw[blue](3.7,0.6) node{$E^{\tilde s,v_2}_{\tilde h}$};
    \end{tikzpicture}
    \caption{For the intersections $(\hat s_2, \hat h_2)$ and $(\tilde s_1, \tilde h_1)$ to represent two consecutive  handovers, the regions \red{$S_5\cup E^{\hat s_2, v_1}_{\hat h_2} \cup E^{\tilde s,v_1}_{\tilde h}$} and  \blue{$E^{\hat s_2, v_2}_{\hat h_2}\cup E^{\tilde s,v_2}_{\tilde h}$} must have no points from $\Hcal^1$ and $\Hcal^2$, respectively. Here $(\tilde s, \tilde h)$ is a realization of $(\hat s_1(0,h,T_j,H_j),\hat h_1(0,h,T_j,H_j))$.}
\label{subfigure:mixed_bird122a}
\end{subfigure}
\hspace{0.1in}
\centering
\begin{subfigure}[t]{0.45\linewidth}\begin{tikzpicture}[scale=0.7, every node/.style={scale=0.7}]
\pgftransformxscale{1}  
\pgftransformyscale{1}    
    \draw[->] (-2, 0) -- (5, 0) node[right] {$t$};
   \draw[red, domain=-0.7:1.7, smooth] plot (\x, {(16*\x*\x- 16*2*0.5*\x+16*(0.5)^2+2^2)^0.5});
    \draw[red](0.5,2) node{$\bullet$};
    \draw[](-0.8,2) node{$(-t',h')$};
    \draw[blue](2,2.5) node{$\bullet$};
    \draw[](1.7,2) node{$(0,h)$};
    \draw[blue, domain=-1:4.4, smooth] plot (\x, {((3/2)^2*\x*\x-2*(3/2)^2*2*\x+(3/2)^2*2^2+2.5^2)^0.5});
    \draw[blue](4,1) node{$\bullet$};
    \draw[](4,0.7) node{$(t_1,h_1)$};
    \draw[blue, domain=1:5, smooth] plot (\x, {((3/2)^2*\x*\x-2*(3/2)^2*4*\x+(3/2)^2*4^2+1^2)^0.5});
    \draw[](-1.7,4.8) node{$(\hat s_1,\hat h_1)$};
    \draw[](1.7,3.4) node{$(\hat s_2,\hat h_2)$};
    \draw[](3,2.4) node{$(\tilde s_1,\tilde h_1)$};
     \draw[red]  (0.28,0) arc (-180:-360:0.8 and 2.9);
     \draw[blue]  (2.9,0) arc (0:180:1.8 and 2.9); 
     \draw[blue]  (0.75,0) arc (-180:-360:1.7 and 2.6);
     \draw[red]  (3.15,0) arc (0:180:0.7 and 2.6);   
   \draw[blue](-0.2,0.6) node{$E^{\hat s_2,v_2}_{\hat h_2}$};
     \draw[red](1,0.6) node{$E^{\hat s_2,v_1}_{\hat h_2}$};
      \draw[red](2.3,0.6) node{$E^{\tilde s_1,v_1}_{\tilde h_1}$};
     \draw[blue](5,0.6) node{$E^{\tilde s_1,v_2}_{\tilde h_1}$};
    \end{tikzpicture}
\captionsetup{width=0.9\linewidth}
\caption{For the intersections $(\hat s_2, \hat h_2)$ and $(\tilde s_1, \tilde h_1)$ to represent two consecutive  handovers, the regions \red{$S_6\cup E^{\hat s_2, v_1}_{\hat h_2} \cup E^{\tilde s_1, v_1}_{\tilde h_1}$} and  \blue{$E^{\hat s_2, v_2}_{\hat h_2}\cup E^{\tilde s_1, v_2}_{\tilde h_1}$} must have no points from $\Hcal^1$ and $\Hcal^2$, respectively. Here $(\tilde s_1, \tilde h_1)$ is a realization of $(\hat s(0,h,T_j,H_j),\hat h(0,h,T_j,H_j))$.}
\label{subfigure:mixed_bird122b}
\end{subfigure}
\captionsetup{width=0.9\linewidth}
\caption{\ref{case:H1221} and~\ref{case:H1222}, when the head point $(T_j,H_j)=(t_1,h_1)$ is of type $1$ or $2$, respectively.}    \label{figure:mixed_bird122a}
\end{figure}
\begin{case}[\textbf{\em $(T_j,H_j)$ is of type 2.}]\label{case:H1222} For the point $(T_j,H_j)$ to be eligible, it must be from outside of $\overline{E^{\hat s_2,v_2}_{\hat h_2}}$ and satisfy $T_j\geq 0$. Otherwise if $T_j<0$, the two birds with heads at $(T_j,H_j)$ and $(0,h)$ intersect at a point which is contained in $\Lcal^+_e(v_1,v_2)$ or produce a handover before time $\hat s_2$, which is not interesting in this case. Thus we must have $T_j\geq 0$ and the unexplored region is $Q_0\setminus \overline{E^{\hat s_2,v_2}_{\hat h_2}}$. The type of the pair of consecutive handovers is $\bbinom{1,1}{1,2,2}$. Hence we have $T_1(\hat s_2):= \hat s(0,h,T_j,H_j)$ and the region $\interior{\left(\bigcup_{t\in [\hat s_2, \hat s(0,h,T_j,H_j)]}E^{t, v_l}_{\hat h_t}\setminus E^{\hat s_2,v_l}_{\hat h_2}\right)}$ has no points of $\Hcal^l$, for $l=1,2$, where $\hat h_t=(v_2^2t^2+h^2)^\half$, see picture~(\subref{subfigure:mixed_bird122b})  of Figure~\ref{figure:mixed_bird122a}. 
\end{case}
Based on the reasoning in~\ref{case:H1221} and~\ref{case:H1222} we can decompose the inner expectation in (\ref{eq:T1221}) as:
\begin{align}
\lefteqn{\mathbb{E}_\Hcal\Bigg[ e^{-\rho(T_1(\hat s_2)-\hat s_2)} \prod_{l=1,2}\one_{\Hcal^l\left(E^{\hat{s}_2,v_l}_{\hat{h}_2}\right)=0}\Bigg]}\nn\\
&= \E_\Hcal\left[ \one_{\exists (T_j,H_j)\in \Hcal^1 \,\mbox{:}\, \hat s_2\leq T_j\leq 0, T_j\in D_\ell(0,h,H_j),\, \bigcap_{l=1,2}\left\{\Hcal^l\interior{\left(\bigcup_{t\in [\hat s_2, \hat s_1(0,h,T_j,H_j)]}E^{t, v_l}_{\hat h_t}\setminus E^{\hat s_2,v_l}_{\hat h_2}\right)}=0\right\}} \right.\nn\\
&\hspace{2.5in}\times e^{-\rho (\hat s_1(-t',h',T_j,H_j)-\hat s_2)}  \left.\prod_{l=1,2}\one_{\Hcal^l\left(E^{\hat{s}_2,v_l}_{\hat{h}_2}\right)=0}\right] \nn\\
&\;+\E_\Hcal\left[ \one_{\exists (T_j,H_j)\in \Hcal^1 \,\mbox{:}\, T_j\in D_r(0,h,H_j),\, \bigcap_{l=1,2}\left\{\Hcal^l\interior{\left(\bigcup_{t\in [\hat s_2, \hat s_1(0,h,T_j,H_j)]}E^{t, v_l}_{\hat h_t}\setminus E^{\hat s_2,v_l}_{\hat h_2}\right)}=0\right\}} e^{-\rho (\hat s_1(-t',h',T_j,H_j)-\hat s_2)} \right.\nn\\
&\hspace{3.5in}\times \left.\prod_{l=1,2}\one_{\Hcal^l\left(E^{\hat{s}_2,v_l}_{\hat{h}_2}\right)=0}\right] \nn\\
&\;\;+ \E_\Hcal\left[ \one_{\exists (T_j,H_j)\in \Hcal^2 \,\mbox{:}\, T_j\geq 0,\,  \bigcap_{l=1,2}\left\{\Hcal^l\interior{\left(\bigcup_{t\in [\hat s_2, \hat s(0,h,T_j,H_j)]}E^{t, v_l}_{\hat h_t}\setminus E^{\hat s_2,v_l}_{\hat h_2}\right)}=0\right\}} e^{-\rho (\hat s(0,h,T_j,H_j)-\hat s_2)}\right.\nn\\
&\hspace{3.5in}\times \left. \prod_{l=1,2}\one_{\Hcal^l\left(E^{\hat{s}_2,v_l}_{\hat{h}_2}\right)=0}\right].
\label{eq:T1222y}
\end{align}
There exists a unique head point satisfying the condition in each term of last sum in (\ref{eq:T1222y}),  and hence it can be written as 
\begin{align}
\lefteqn{\E_\Hcal\left[ \sum_{(T_j,H_j)\in \Hcal^1\,\mbox{:}\, \hat s_2\leq T_j\leq 0, T_j\in D_\ell(0,h,H_j)} \prod_{l=1,2}\one_{\Hcal^l\left(\bigcup_{t\in [\hat s_2, \hat s_1(0,h,T_j,H_j)]}E^{t, v_l}_{\hat h_t}\cup E^{\hat s_2,v_l}_{\hat h_2}\right)=0}  e^{-\rho (\hat s_1(0,h,T_j,H_j)-\hat s_2)} \right]}\nn\\
&\;\;+\E_\Hcal\left[ \sum_{(T_j,H_j)\in \Hcal^1\,\mbox{:}\, T_j\in D_r(0,h,H_j)} \prod_{l=1,2}\one_{\Hcal^l\left(\bigcup_{t\in [\hat s_2, \hat s_1(0,h,T_j,H_j)]}E^{t, v_l}_{\hat h_t}\cup E^{\hat s_2,v_l}_{\hat h_2}\right)=0}  e^{-\rho (\hat s_1(0,h,T_j,H_j)-\hat s_2)} \right]\nn\\
&\;\;+ \E_\Hcal\left[ \sum_{(T_j,H_j)\in \Hcal^2 \,\mbox{:}\, T_j\geq 0} \prod_{l=1,2}\one_{\Hcal^l\left(\bigcup_{t\in [\hat s_2, \hat s(0,h,T_j,H_j)]}E^{t, v_l}_{\hat h_t}\cup E^{\hat s_2,v_l}_{\hat h_2}\right)=0}  e^{-\rho (\hat s(0,h,T_j,H_j)-\hat s_2)} \right] \label{eq:T1222z}\\
&:=\zeta\bbinom{1,2}{1,2,1}+\zeta\bbinom{1,1}{1,2,1}+\zeta\bbinom{1,1}{1,2,2}.
\label{eq:T1222}
\end{align}
Using Remark~\ref{remark:union} for the first two terms in (\ref{eq:T1222z}) for $l=1$, we have
\begin{equation}
\bigcup_{t\in [\hat s_2,\hat s_1(0,h,T_j,H_j)]} E^{t,v_1}_{\hat h_t}\cup E^{\hat s_2, v_1}_{\hat h_2}= E^{\hat s_2, v_1}_{\hat h_2}\cup S_5\cup E^{\hat s_1(0,h,T_j,H_j),v_1}_{\hat h_1(0,h,T_j,H_j)},
\label{eq:S5}
\end{equation}
where $\hat h_t=(v_2^2t^2+h^2)^\half$ and $S_5$ is the region outside $E^{\hat s_2, v_1}_{\hat h_2}\cup E^{\hat s_1(0,h,T_j,H_j),v_1}_{\hat h_1(0,h,T_j,H_j)}$, but below the hyperbola given by the equation $u^2-\frac{v_1^2v_2^2}{v_1^2-v_2^2}t^2=h^2$, in the $(t,u)$-coordinate system similar to (\ref{eq:2bird-hyp}) and within the time interval 
\[
\left[\frac{v_1^2-v_2^2}{v_1^2}\hat s_2, \frac{v_1^2-v_2^2}{v_1^2}\hat s_1(0,h,T_j,H_j)\right].
\]
For $l=2$, using part~(\ref{monotonic2}) of Proposition~\ref{prop:monotonic-sets}, the increasing property of $\left\{E^{t,v_2}_{\hat h_t}\setminus E^{\hat s, v_2}_{\hat h}\right\}_{t\geq \hat s_2}$, we have
\[
\bigcup_{t\in [\hat s_2,\hat s_1(0,h,T_j,H_j)]} E^{t,v_2}_{\hat h_t}\cup E^{\hat s, v_2}_{\hat h}= E^{\hat s_1(0,h,T_j,H_j),v_2}_{\hat h_1(0,h,T_j,H_j)}\cup  E^{\hat s, v_2}_{\hat h}.
\]
The first two terms in (\ref{eq:T1222}) become
\begin{align}
\lefteqn{\zeta\bbinom{1,2}{1,2,1}+\zeta\bbinom{1,1}{1,2,1}}\nn\\ &=\E_\Hcal\left[ \sum_{(T_j,H_j)\in \Hcal^1 \,\mbox{:}\, \hat s_2\leq T_j\leq 0, T_j\in D_\ell(0,h,H_j)} \!\!\! \one_{\Hcal^1\left(E^{\hat s_2, v_1}_{\hat h_2}\cup S_5\cup E^{\hat s_1(0,h,T_j,H_j),v_1}_{\hat h_1(0,h,T_j,H_j)}\right)=0} \one_{\Hcal^2\left(E^{\hat s_2, v_2}_{\hat h_2}\cup E^{\hat s_1(0,h,T_j,H_j),v_2}_{\hat h_1(0,h,T_j,H_j)}\right)=0} \right.\nn\\
&\hspace{3.5in}\times  e^{-\rho (\hat s_1(0,h,T_j,H_j)-\hat s_2)} \Bigg]\nn\\
&\;+\E_\Hcal\left[ \sum_{(T_j,H_j)\in \Hcal^1\,\mbox{:}\, T_j\in D_r(0,h,H_j)} \one_{\Hcal^1\left(E^{\hat s_2, v_1}_{\hat h_2}\cup S_5\cup E^{\hat s_1(0,h,T_j,H_j),v_1}_{\hat h_1(0,h,T_j,H_j)}\right)=0} \one_{\Hcal^2\left(E^{\hat s_2, v_2}_{\hat h_2}\cup E^{\hat s_1(0,h,T_j,H_j),v_2}_{\hat h_1(0,h,T_j,H_j)}\right)=0} \right.\nn\\
&\hspace{3.5in}\times  e^{-\rho (\hat s_1(0,h,T_j,H_j)-\hat s_2)} \Bigg].
\label{eq:T1222a}
\end{align}
The point $(T_j,H_j)$ must belong to the region $R_1(h,t',h')\subset [\hat s_2,\infty)\times \R^+\setminus E^{\hat s_2,v_1}_{\hat h_2}$, where
\begin{equation}
R_1(h,t',h'):=\left\{(t,u)\,\mbox{:}\, t>\hat s_2,\, u\geq 0,\,v_1^2(t-\hat s_2)^2+u^2\geq\hat h_2^2,\, u^2-\frac{v_1^2v_2^2}{v_1^2-v_2^2}t^2\leq h^2\right\}.
\label{eq:R1-122}
\end{equation}
For the region $R_1(h,t',h')$ in (\ref{eq:R1-122}), the first condition $t>\hat s_2$ is due to the discussion in~\ref{case:H1221}. The last two conditions are due to the fact that the next head point must lie outside $\overline{E^{\hat s_2, v_1}_{\hat h_2}}$ and below the curve of the hyperbola $u^2-\frac{v_1^2v_2^2}{v_1^2-v_2^2}t^2=h^2$. Let $\frac{v_1^2v_2^2}{v_1^2-v_2^2}:= v^2$.

Applying the Campbell-Mecke formula, the first two terms in (\ref{eq:T1222}) become
\begin{align}
\lefteqn{\zeta\bbinom{1,2}{1,2,1}+\zeta\bbinom{1,1}{1,2,1}}\nn\\ 
&=2\la_1v_1\int_{R_1(h,t',h')} e^{-\rho (\hat s_1(0,h,t_1,h_1)-\hat s_2)} e^{-2\la_1v_1\left\vert  E^{\hat s_1(0,h,t_1,h_1),v_1}_{\hat h_1(0,h,t_1,h_1)}\cup S_5\cup E^{\hat s_2,v_1}_{\hat h_2}\right\vert} e^{-2\la_2v_2\left\vert E^{\hat s_1(0,h,t_1,h_1),v_2}_{\hat h_1(0,h,t_1,h_1)}\cup E^{\hat s_2,v_2}_{\hat h_2}\right\vert}{\rm d}t_1\, {\rm d}h_1\nn\\
&=2\la_1v_1\int_{\frac{v_1^2-v_2^2}{v_1^2+v_2^2}\,\hat s_2}^{\hat s_2+\hat h_2/v_1} \!\!\!\int_{\left(\hat h_2^2-v_1^2(t_1-\hat s_2)^2\right)^\half}^{\left(v^2 t_1^2+h^2\right)^\half} \!\! 
e^{-\rho (\hat s_1(0,h,t_1,h_1)-\hat s_2)} 
e^{-2\la_1v_1\left\vert  E^{\hat s_1(0,h,t_1,h_1),v_1}_{\hat h_1(0,h,t_1,h_1)}\cup S_5\cup E^{\hat s_2,v_1}_{\hat h_2}\right\vert}\nn\\
&\hspace{3.5in}\times e^{-2\la_2v_2\left\vert E^{\hat s_1(0,h,t_1,h_1),v_2}_{\hat h_1(0,h,t_1,h_1)}\cup E^{\hat s_2,v_2}_{\hat h_2}\right\vert}{\rm d}t_1\, {\rm d}h_1\nn\\
&\;\;+ 2\la_1v_1  \!\!\! \int_{\hat s_2+\hat h_2/v_1}^\infty  \!\! \int_0^{\left(v^2 t_1^2+h^2\right)^\half}   \!\!\!\! \!\!\!\! e^{-\rho (\hat s_1(0,h,t_1,h_1)-\hat s_2)} e^{-2\la_1v_1\left\vert  E^{\hat s_1(0,h,t_1,h_1),v_1}_{\hat h_1(0,h,t_1,h_1)}\cup S_5\cup E^{\hat s_2,v_1}_{\hat h_2}\right\vert}\nn\\
&\hspace{3.5in}\times e^{-2\la_2v_2\left\vert E^{\hat s_1(0,h,t_1,h_1),v_2}_{\hat h_1(0,h,t_1,h_1)}\cup E^{\hat s_2,v_2}_{\hat h_2}\right\vert}{\rm d}t_1\, {\rm d}h_1.
\label{eq:T1223}
\end{align}
For the third term in (\ref{eq:T1222}), using Remark~\ref{remark:union} for $l=2$, the region satisfies
\begin{equation}
\bigcup_{t\in [\hat s_2,\hat s(0,h,T_j,H_j)]} E^{t,v_1}_{\hat h_t}\cup E^{\hat s_2, v_1}_{\hat h_2}= E^{\hat s_2, v_1}_{\hat h_2}\cup S_6\cup E^{\hat s(0,h,T_j,H_j),v_1}_{\hat h(0,h,T_j,H_j)},
\label{eq:S6}
\end{equation}
where $\hat h_t=(v_2^2t^2+h^2)^\half$ and $S_2$ is the region outside $E^{\hat s_2, v_1}_{\hat h_2}\cup E^{\hat s(0,h,T_j,H_j),v_1}_{\hat h(0,h,T_j,H_j)}$, but below the hyperbola given by the equation $u^2-\frac{v_1^2v_2^2}{v_1^2-v_2^2}t^2=h^2$, in the $(t,u)$-coordinate system similar to (\ref{eq:2bird-hyp}) and within the time interval 
\[
\left[\frac{v_1^2-v_2^2}{v_1^2}\hat s_2, \frac{v_1^2-v_2^2}{v_1^2}\hat s(0,h,T_j,H_j)\right].
\]
For $l=2$, using the increasing property of $\left\{E^{t,v_2}_{\hat h_t}\setminus E^{\hat s, v_2}_{\hat h}\right\}_{t\geq \hat s_2}$ from part~(\ref{monotonic2}) of Proposition~\ref{prop:monotonic-sets}, we have that
\[
\bigcup_{t\in [\hat s_2,\hat s(0,h,T_j,H_j)]} E^{t,v_2}_{\hat h_t}\cup E^{\hat s, v_2}_{\hat h}= E^{\hat s(0,h,T_j,H_j),v_2}_{\hat h(0,h,T_j,H_j)}\cup  E^{\hat s, v_2}_{\hat h}.
\]
Then the third term in (\ref{eq:T1222}) equals,
\begin{align}
\zeta\bbinom{1,1}{1,2,2}
&=\E_\Hcal\left[ \sum_{(T_j,H_j)\in \Hcal^2 \,\mbox{:}\, T_j\geq 0} \one_{\Hcal^1\left(E^{\hat s(0,h,T_j,H_j), v_1}_{\hat h(0,h,T_j,H_j)}\cup S_6\cup E^{\hat s_2,v_l}_{\hat h_2}\right)=0} \one_{\Hcal^2\left(E^{\hat s(0,h,T_j,H_j), v_2}_{\hat h(0,h,T_j,H_j)}\cup E^{\hat s_2,v_2}_{\hat h_2}\right)=0}\right.\nn\\
&\hspace{4in}\times  e^{-\rho (\hat s(0,h,T_j,H_j)-\hat s_2)} \Big].
\label{eq:T1222x}
\end{align}
The point $(T_j,H_j)$ must lie in the region $R_2(h,t',h')\subset  (\R^+)^2\setminus E^{\hat s_2,v_2}_{\hat h_2}$, defined as 
\begin{equation}
R_2(h,t',h'):=\left\{(t,u)\,\mbox{:}\, t\geq 0,\, u\geq 0,\,v_2^2(t-\hat s_2)^2+u^2\geq \hat h_2^2\right\}.
\label{eq:R2-122}
\end{equation}
For the region $R_2(h,t',h')$ in (\ref{eq:R2-122}), the first condition $t\geq 0$ is due to the discussion in~\ref{case:H1222}. The last two conditions are due to the fact that the next head point must lie outside $\overline{E^{\hat s_2, v_2}_{\hat h_2}}$. By the Campbell-Mecke formula for (\ref{eq:T1222x}), we have
\begin{align}
\zeta\bbinom{1,1}{1,2,2}
&=2\la_2v_2\int_{R_2(h,t',h')} e^{-\rho (\hat s(0,h,t_1,h_1)-\hat s_2)} \E_\Hcal\left[  \one_{\Hcal^1\left(E^{\hat s(0,h,t_1,h_1), v_1}_{\hat h(0,h,t_1,h_1)}\cup S_6\cup E^{\hat s_2,v_l}_{\hat h_2}\right)=0} \right.\nn\\
&\hspace{3in}\left.\times \one_{\Hcal^2\left(E^{\hat s(0,h,t_1,h_1), v_2}_{\hat h(0,h,t_1,h_1)}\cup E^{\hat s_2,v_2}_{\hat h_2}\right)=0}\right]{\rm d}t_1  {\rm d}h_1.
\label{eq:T1224a}
\end{align}
Computing the void probabilities in (\ref{eq:T1224a}) we have
\begin{align}
\zeta\bbinom{1,1}{1,2,2}
&= 2\la_2v_2\int_{0}^{\hat s_2+\hat h_2/v_2}\int_{\left(\hat h_2^2-v_2^2(t_1-\hat s_2)^2\right)^\half}^{\infty} e^{-\rho (\hat s(0,h,t_1,h_1)-\hat s_2)} e^{-2\la_1v_1\left\vert E^{\hat s(0,h,t_1,h_1), v_1}_{\hat h(0,h,t_1,h_1)}\cup S_6\cup  E^{\hat s_2,v_1}_{\hat h_2}\right\vert}\nn\\
&\hspace{3.1in} \times e^{-2\la_2v_2\left\vert E^{\hat s(0,h,t_1,h_1), v_2}_{\hat h(0,h,t_1,h_1)}\cup   E^{\hat s_2,v_2}_{\hat h_2}\right\vert}
{\rm d}t_1  {\rm d}h_1\nn\\
&\;\;+  2\la_2v_2\int_{\hat s_2+\hat h_2/v_2}^\infty\int_0^{\infty} e^{-\rho (\hat s(0,h,t_1,h_1)-\hat s_2)} e^{-2\la_1v_1\left\vert E^{\hat s(0,h,t_1,h_1), v_1}_{\hat h(0,h,t_1,h_1)}\cup S_6\cup  E^{\hat s_2,v_1}_{\hat h_2}\right\vert}
\nn\\
&\hspace{3.1in} \times
e^{-2\la_2v_2\left\vert E^{\hat s(0,h,t_1,h_1), v_2}_{\hat h(0,h,t_1,h_1)}\cup   E^{\hat s_2,v_2}_{\hat h_2}\right\vert}
{\rm d}t_1  {\rm d}h_1.
\label{eq:T1224}
\end{align}
Let us denote the sum of all terms in (\ref{eq:T1223}), (\ref{eq:T1224}) as $\zeta^{(1)}_{1,2}(\rho, t',h',h)=\zeta\bbinom{1,2}{1,2,1}+\zeta\bbinom{1,1}{1,2,1}+\zeta\bbinom{1,1}{1,2,2}$. Using the expression for $\zeta^{(1)}_{1,2}(\rho, t',h',h)$ in (\ref{eq:T1221}) yields
\begin{align}
\frac{L^{(1)}_{1,2}}{4\la_1\la_2 v_1v_2}\E^0_{\Wcal_{1,2}^{(1)}} [e^{-\rho T}]
& =  \int_{0}^\infty  \int_{0}^h \int_{t^*}^\infty 
\zeta^{(1)}_{1,2}(\rho, t',h',h) {\rm d}t'  {\rm d}h' {\rm d}h  +  \int_{0}^\infty  \int_h^\infty \int_0^{\infty}
\zeta^{(1)}_{1,2}(\rho, t',h',h) {\rm d}t'  {\rm d}h' {\rm d}h \nn\\
&:=\xi_{1,2}^{(1)}(\rho, v_1,v_2),
\label{eq:T1225}
\end{align}
which is a function of $\rho, v_1, v_2$ only. \qed
\subsection{Area of the regions $E^{s_1,v}_{\ell_1}\cup E^{s_2,v}_{\ell_2}$}~\label{subsection:Uellipses} The area of union of half ellipses are computed using elementary geometry in the following Lemma. We compute the area of the region $E^{s_1,v}_{\ell_1}\cup E^{s_2,v}_{\ell_2}$ for general speed $v$, also as preparation of the multi-speed case. For any speed $v$, two half-ellipses $E^{s_1,v}_{\ell_1}, E^{s_2,v}_{\ell_2}$ are given by the equations
\begin{align}
    v^2(t-s_1)^2+h^2=\ell_1^2 \text{ and }v^2(t-s_2)^2+h^2=\ell_2^2.
    \label{eq:v_ellipses}
\end{align}

Let the intersection point in $\mathbb H^+$, of the two ellipses in (\ref{eq:v_ellipses}) be $(t(v), h(v))$, which is given by
\begin{equation}
t(v):=\frac{1}{2(s_2-s_1)}\left[(\ell_1^2-\ell_2^2)/v^2-s_1^2+s_2^2\right] \text{ and }h(v):=\left(\ell_1^2-v^2(t(v)-s_1)^2\right)^\half. 
\label{eq:tvhv}
\end{equation}
\vspace{-0.1in}
\begin{lemma}
Let $v>0$ and $(s_1, \ell_1), (s_2, \ell_2)$ be two points such that $s_1<s_2$. Suppose the ellipses in (\ref{eq:v_ellipses}) intersects at $(t(v), h(v))$. Then the area of the region $E^{s_1,v}_{\ell_1}\cup E^{s_2,v}_{\ell_2}$ is given by
\begin{align}
|E^{s_1,v}_{\ell_1}\cup E^{s_2,v}_{\ell_2}|&= \begin{cases}
    \frac{1}{2v}\left[\pi \ell_1^2+vh(v)(s_2-s_1)+F_1(h(v), \ell_1,\ell_2)\right], & \text{ if }s_1<s_2< t(v),\\
    \frac{1}{2v}\left[ \pi \ell_2^2+\pi \ell_1^2+vh(v)(s_2-s_1)-F_2(h(v), \ell_1, \ell_2)\right], &\text{ if } s_1< t(v)<s_2,\\
    \frac{1}{2v}\left[\pi \ell_2^2+vh(v)(s_2-s_1)-F_1(h(v), \ell_1,\ell_2)\right], & \text{ if } t(v)<s_1<s_2,\\
    \end{cases}
\label{eq:MHn72}
\end{align}
where $t(v), h(v)$ as in (\ref{eq:tvhv}) and the functions $F_1, F_2$ are as in (\ref{eq:Lambda1}) and (\ref{eq:Lambda2}).
\label{lemma:EUE}
\end{lemma}
\begin{proof}
The area of the region $E^{s_1,v}_{\ell_1}\cup E^{s_2,v}_{\ell_2}$ is given by
\begin{align}
|E^{s_1,v}_{\ell_1}\cup E^{s_2,v}_{\ell_2}|&=
\begin{cases}
|E^{s_1,v}_{\ell_1}|+|E^{s_2,v}_{\ell_2}\setminus E^{s_1,v}_{\ell_1}|, &\text{ if } s_1<s_2<t(v),\\
|E^{s_1,v}_{\ell_1}|+| E^{s_2,v}_{\ell_2}|- |E^{s_1,v}_{\ell_1}\cap E^{s_2,v}_{\ell_2}|, &\text{ if } s_1<t(v)<s_2,\\
|E^{s_2,v}_{\ell_2}|+|E^{s_1,v}_{\ell_1}\setminus E^{s_2,v}_{\ell_2}|, &\text{ if } t(v)<s_1<s_2.
    \end{cases}
\end{align}
In the first case $s_1<s_2<t(v)$, we have 
\begin{align}
|E^{s_2,v}_{\ell_2}\setminus E^{s_1,v}_{\ell_1}| &= \int_{0}^{h(v)}\int^{s_2+\frac{1}{v}\left(\ell_2^2-h^2\right)^\half}_{s_1+\frac{1}{v}\left(\ell_1^2-h^2\right)^\half} \, {\rm d}t \, {\rm d}h\nn\\
&=h(v)(s_2-s_1)+\frac{1}{v}\int_{0}^{h(v)}\left(\left(\ell_2^2-h^2\right)^\half-\left(\ell_1^2-h^2\right)^\half\right) \, {\rm d}h\nn\\
&=h(v)(s_2-s_1)+\frac{1}{2v}\left[h(v)\left(\ell_2^2-h(v)^2\right)^\half-h(v)\left(\ell_1^2-h(v)^2\right)^\half\right]\nn\\
&\quad +\frac{1}{2v}\left[ \ell_2^2\arcsin{(h(v)/\ell_2)}- \ell_1^2\arcsin{(h(v)/\ell_1)}\right]\nn\\
&\stackrel{(\ref{eq:v_ellipses}),  (\ref{eq:Lambda1})}{=}h(v)(s_2-s_1)+\frac{1}{2v}\left[v h(v) |t(v)-s_2|-vh(v) |t(v)-s_1|+ F_1(h(v), \ell_1,\ell_2)\right]\nn\\
&=h(v)(s_2-s_1)- \half h(v) (s_2-s_1)+\frac{1}{2v}F_1(h(v), \ell_1,\ell_2)\nn\\
&=\half h(v)(s_2-s_1)+\frac{1}{2v} F_1(h(v), \ell_1,\ell_2),
\label{eq:s1s2tv}
\end{align}
\begin{figure}[ht!]
\begin{subfigure}[b]
{0.3\linewidth}\begin{tikzpicture}[scale=0.6, every node/.style={scale=0.7}]
\pgftransformxscale{0.7}  
\pgftransformyscale{0.7}    
    \draw[->] (-1, 0) -- (7, 0) node[right] {$t$};
    \draw[](0.1,2.3) node{$\bullet$};
    \draw[](1.3,1.9) node{$(t(v),h(v))$};
    \draw[](1.4,3.4) node{$(s_1,\ell_1)$};
    \draw[](3.3,5.5) node{$(s_2,\ell_2)$};
     \draw[]  (-0.6,0) arc (-180:-360:2 and 3.07);
     \draw[]  (6.5,0) arc (0:180:3.4 and 5.05);
    \draw[](0.7,0.6) node{$E^{s_1,v}_{\ell_1}$};
    \draw[](4.7,2.6) node{$E^{s_2,v}_{\ell_2}$};
    \end{tikzpicture}
\end{subfigure}
    \begin{subfigure}[b]{0.3\linewidth}
      \begin{tikzpicture}[scale=0.5, every node/.style={scale=0.7}]
\pgftransformxscale{0.7}  
\pgftransformyscale{0.7}    
    \draw[->] (-1, 0) -- (9, 0) node[right] {$t$};
    \draw[](2.25,2.8) node{$\bullet$};
    \draw[](3.85,2.95) node{$(t(v),h(v))$};
    \draw[](1.5,3.4) node{$(s_1,\ell_1)$};
    \draw[](5.1,5.4) node{$(s_2,\ell_2)$};
    \draw[]  (-0.47,0) arc (-180:-360:2 and 3.07);
     \draw[]  (8.5,0) arc (0:180:3.4 and 5.05);
     %
    \draw[](0.7,1) node{$E^{s_1,v}_{\ell_1}$};
     \draw[](6,2) node{$E^{s_2,v}_{\ell_2}$};
    \end{tikzpicture}
\end{subfigure}
\;
\begin{subfigure}[b]
{0.3\linewidth}\begin{tikzpicture}[scale=0.6, every node/.style={scale=0.7}]
\pgftransformxscale{0.7}  
\pgftransformyscale{0.7}    
    \draw[->] (-1, 0) -- (8, 0) node[right] {$t$};
    \draw[](5.95,2.65) node{$\bullet$};
    \draw[](7.4,2.9) node{$(t(v),h(v))$};
    \draw[](4.7,3.4) node{$(s_2,\ell_2)$};
    \draw[](3.3,5.5) node{$(s_1,\ell_1)$};
     \draw[]  (3,0) arc (-180:-360:2 and 3.07);
     \draw[]  (6.5,0) arc (0:180:3.4 and 5.05);
    \draw[](0.6,0.6) node{$E^{s_1,v}_{\ell_1}$};
    \draw[](4.7,1.6) node{$E^{s_2,v}_{\ell_2}$};
    \end{tikzpicture}
\end{subfigure}
\captionsetup{width=0.9\linewidth}
\caption{Given $(s_1,\ell_1)$ and $(s_2,\ell_2)$, the area of the union of two half-ellipses should be different for different speed and depending on $t(v)<s_1<s_2$ or $s_1< t(v)<s_2$ or $s_1<s_2<t(v)$.}
\label{figure:mixed_birds12v}
\end{figure}
Let us now consider the second case, namely $s_1<t(v)<s_2$. Then $
|E^{s_1,v}_{\ell_1}\cup E^{s_2,v}_{\ell_2}|=|E^{s_1,v}_{\ell_1}|+| E^{s_2,v}_{\ell_2}|- |E^{s_1,v}_{\ell_1}\cap E^{s_2,v}_{\ell_2}|$, where the area of the common region is
\begin{align}
|E^{s_1,v}_{\ell_1}\cap E^{s_2,v}_{\ell_2}| &= \int_{0}^{h(v)}\int_{s_2-\frac{1}{v}\left(\ell_2^2-h^2\right)^\half}^{s_1+\frac{1}{v}\left(\ell_1^2-h^2\right)^\half} \, {\rm d}t \, {\rm d}h\nn\\
&= h(v)(s_1-s_2)+\frac{1}{2v}\left[h(v)\left(\ell_1^2-h(v)^2\right)^\half+h(v)\left(\ell_2^2-h(v)^2\right)^\half\right] \nn\\
&\qquad +\frac{1}{2v}\left[\ell_1^2\arcsin{(h(v)/\ell_1)}+ \ell_2^2\arcsin{(h(v)/\ell_2)}\right]\nn\\
&=h(v)(s_1-s_2)+\frac{1}{2v}\left[h(v)\left(\ell_2^2-h(v)^2\right)^\half+ h(v)\left(\ell_1^2-h(v)^2\right)^\half\right] + \frac{1}{2v}F_2(h(v), \ell_1,\ell_2)\nn\\
&\stackrel{(\ref{eq:v_ellipses})}{=} h(v)(s_1-s_2)+\half h(v)\left[|t(v)-s_2|+ |t(v)-s_1| \right]+\frac{1}{2v}F_2(h(v), \ell_1,\ell_2)\nn\\
&= -\half h(v)(s_2-s_1)+\frac{1}{2v} F_2(h(v), \ell_1,\ell_2),
\label{eq:MH5}
\end{align}
since $|t(v)-s_2|+|t(v)-s_1| =s_2-s_1$, since $s_1\leq t(v)\leq s_2$, from the second picture of Figure~\ref{figure:mixed_birds12v}. Above $F_2(h(v), \ell_1,\ell_2)$ is as defined in (\ref{eq:Lambda2}).

In the last case, $t(v)<s_1<s_2$, we have $|E^{s_1,v}_{\ell_1}\cup E^{s_2,v}_{\ell_2}| = |E^{s_2,v}_{\ell_2}|+|E^{s_1,v}_{\ell_1}\setminus E^{s_2,v}_{\ell_2}|$. Thus
\begin{align}
|E^{s_1,v}_{\ell_1}\setminus E^{s_2,v}_{\ell_2}| &= \int_{0}^{h(v)}\int_{s_1-\frac{1}{v}\left(\ell_1^2-h^2\right)^\half}^{s_2-\frac{1}{v}\left(\ell_2^2-h^2\right)^\half} \, {\rm d}t \, {\rm d}h\nn\\
&=h(v)(s_2-s_1)+\frac{1}{v}\int_{0}^{h(v)}\left(\left(\ell_1^2-h^2\right)^\half-\left(\ell_2^2-h^2\right)^\half\right) \, {\rm d}h\nn\\
&=h(v)(s_2-s_1)+\frac{h(v)}{2v}\left[\left(\ell_1^2-h(v)^2\right)^\half-\left(\ell_2^2-h(v)^2\right)^\half\right]- \frac{1}{2v} F_1(h(v), \ell_1,\ell_2) ,
\label{eq:MH4a}
\end{align}
where $F_1(h(v), \ell_1,\ell_2)$ is as defined in (\ref{eq:Lambda1}). Using (\ref{eq:v_ellipses}) in (\ref{eq:MH4a}) we have
\begin{align}
|E^{s_1,v}_{\ell_1}\setminus E^{s_2,v}_{\ell_2}|
&=h(v)(s_2-s_1)+\frac{h(v)}{2v}\left[v \left\{|t(v)-s_1|-|t(v)-s_2|\right\}\right]- \frac{1}{2v}F_1(h(v), \ell_1,\ell_2)\nn\\
&=h(v)(s_2-s_1)+ \half h(v) (s_1-s_2)-\frac{1}{2v}F_1(h(v), \ell_1,\ell_2)\nn\\
&=\half h(v)(s_2-s_1)-\frac{1}{2v} F_1(h(v), \ell_1,\ell_2),
\label{eq:MH4}
\end{align}
since $|t(v)-s_1|- |t(v)-s_2|=s_1-s_2$, as $t(v)<s_1<s_2$, see the first picture of  Figure~\ref{figure:mixed_birds12v}. Thus we have the result.
\end{proof}


\section{}\label{sec-appendix_integrals}
\begin{proof}[Proof of Lemma~\ref{lemma:two-int}~(equation~\ref{two-int1})]
We want to evaluate the integral:
\begin{equation}
I = \int_{t^*}^\infty  e^{-\la\pi \bar v^2 \left(v_2t- v_1\sqrt{t^2- (t^*)^2}\right)^2} \, {\rm d}t,\nn
\end{equation}
where $t^*:=c$ and $\la\pi \bar v^2:=a$ and $\bar v=\frac{v_1v_2}{v_1^2-v_2^2}$. Let us take a hyperbolic substitution as, $t = c \cosh \theta$, which implies 
\begin{equation}
I = c \int_{0}^\infty e^{-a c^2 \left(v_2 \cosh \theta - v_1 \sinh \theta\right)^2} \sinh \theta \, {\rm d}\theta.\nn
\end{equation}
Since $v_1 > v_2$, we define a scale parameter $V = \sqrt{v_1^2 - v_2^2}$ and a hyperbolic phase angle $\phi$ such that
$v_1 = V \cosh \phi, v_2 = V \sinh \phi$, which implies $\tanh \phi = \frac{v_2}{v_1}$. Substituting these into the argument of the exponent yields, $v_2 \cosh \theta - v_1 \sinh \theta = V \sinh \phi \cosh \theta - V \cosh \phi \sinh \theta = -V \sinh(\theta - \phi)$.
Using this, we have
\begin{equation}
I = c \int_{0}^\infty e^{-a c^2 V^2 \sinh^2(\theta - \phi)} \sinh \theta \, {\rm d}\theta.\nn
\end{equation}
Substituting $x = \theta - \phi$ and expanding the remaining $\sinh \theta$ term using $\th=x+\phi$, splits the integral into two distinct parts ($I = I_1 + I_2$) as
\begin{align}
I &= c \cosh \phi \int_{-\phi}^\infty e^{-a c^2 V^2 \sinh^2 x} \sinh x \, {\rm d}x + c \sinh \phi \int_{-\phi}^\infty e^{-a c^2 V^2 \sinh^2 x} \cosh x \, {\rm d}x.\nn
\end{align}
For the first integral $I_1$, we utilize the identity $\sinh^2 x = \cosh^2 x - 1$ as
\begin{equation}
I_1 = c \cosh \phi \, e^{a c^2 V^2} \int_{-\phi}^\infty e^{-a c^2 V^2 \cosh^2 x} \sinh x \, {\rm d}x.\nn
\end{equation}
Using the substitution $u = \cosh x$, the integrand becomes
\begin{equation}
I_1 = c \cosh \phi \, e^{a c^2 V^2} \int_{\cosh\phi}^\infty e^{-a c^2 V^2 u^2 } {\rm d}u.\nn
\end{equation}
The integral can be evaluated directly via the complementary error function as
\begin{align}
I_1 &{=} \frac{\cosh \phi}{\sqrt{a}V} \, e^{a c^2 V^2}\!\!\!\! \int_{\sqrt{a}c v_1}^\infty \!\!\!\! e^{-z^2 } {\rm d}z {=} \frac{\sqrt{\pi}\cosh \phi}{2\sqrt{a}V} \, e^{a c^2 V^2} \mbox{erfc}(\sqrt{a}c v_1)  = \frac{v_1 e^{a c^2 (v_1^2 - v_2^2)}}{2 \sqrt{a} (v_1^2 - v_2^2)} \sqrt{\pi} \, \text{erfc}\left(\sqrt{a} c v_1\right).\label{eq:I1}
\end{align}
For the second integral $I_2$, we use the direct substitution $u = \sinh x$, transforming it into a standard Gaussian integral
\begin{equation}
I_2 = c \sinh \phi \int_{-\sinh \phi}^\infty e^{-a c^2 V^2 u^2} \, {\rm d}u.\nn
\end{equation}
Splitting the domain from $[-\sinh \phi, 0]$ and $[0, \infty)$ yields the error function ($\text{erf}$) and a constant term
\begin{equation}
I_2 = \frac{v_2}{2 \sqrt{a} (v_1^2 - v_2^2)} \sqrt{\pi} \left[ 1 + \text{erf}\left(\sqrt{a} c v_2\right) \right].
\label{eq:I2}
\end{equation}
Combining $I_1$ from (\ref{eq:I2}) and $I_2$ from (\ref{eq:I2}) gives
\begin{equation}
I = \frac{\sqrt{\pi}}{2 \sqrt{a} (v_1^2 - v_2^2)} \left[ v_1 e^{a c^2 (v_1^2 - v_2^2)} \text{erfc}\left(\sqrt{a} c v_1\right) + v_2 \left(1 + \text{erf}\left(\sqrt{a} c v_2\right)\right) \right].\nn
\end{equation}
Taking  $c=t^*= \frac{1}{v_1v_2}(h_1^2-h_2^2)^\half (v_1^2-v_2^2)^\half$, $a=\la\pi \bar v^2$,  $\bar v=\frac{v_1v_2}{v_1^2-v_2^2}$ we have 
\begin{align}
I &= \frac{\sqrt{\pi}}{2 \sqrt{\la\pi \bar v^2} (v_1^2 - v_2^2)} \left[ v_1 e^{\la\pi \bar v^2 (t^*)^2 (v_1^2 - v_2^2)} \text{erfc}\left(\sqrt{\la\pi \bar v^2} t^* v_1\right) + v_2 \left(1 + \text{erf}\left(\sqrt{\la\pi \bar v^2} t^* v_2\right)\right) \right]\nn\\
&= \frac{1}{2 \sqrt{\la}v_1v_2} \left[ v_1 e^{\la\pi (h_1^2-h_2^2)} \text{erfc}\left(\sqrt{\la\pi} \left(\frac{h_1^2-h_2^2}{v_1^2-v_2^2}\right)^\half v_1\right) + v_2 \left(1 + \text{erf}\left(\sqrt{\la\pi}\left(\frac{h_1^2-h_2^2}{v_1^2-v_2^2}\right)^\half  v_2\right)\right) \right]\nn\\
&= \frac{1}{2 \sqrt{\la}v_1v_2} \left[ v_1 e^{\la\pi (h_1^2-h_2^2)} \text{erfc}\left( \sqrt{\la\pi}\left(h_1^2-h_2^2\right)^\half \bar{v}_1\right) + v_2  \left(1 + \text{erf}\left(\sqrt{\la\pi}\left(h_1^2-h_2^2\right)^\half  \bar{v}_2\right)\right) \right],\nn
\end{align}
where $\bar v_1=\frac{v_1}{(v_1^2-v_2^2)^\half}$ and $\bar v_2=\frac{v_2}{(v_1^2-v_2^2)^\half}$.
\end{proof}
\begin{proof}[Proof of Lemma~\ref{lemma:two-int}~(equation~\ref{two-int2})]
In the case $h_1<h_2$, in the second term, we take $a=\la\pi \bar v^2$,  $\bar v=\frac{v_1v_2}{v_1^2-v_2^2}$ and $c=t^{**}= \frac{1}{v_1v_2}(h_2^2-h_1^2)^\half (v_1^2-v_2^2)^\half$. We want to evaluate the integral:
\begin{equation}
I = \int_{0}^\infty  e^{-a \left(v_2t- v_1\sqrt{t^2+c^2}\right)^2} \, {\rm d}t.\nn
\end{equation}
Let $t = c \sinh \theta$. This yields 
\begin{equation}
I = c \int_{0}^\infty e^{-a c^2 \left(v_2 \sinh \theta - v_1 \cosh \theta\right)^2} \cosh \theta \, {\rm d}\theta.\nn
\end{equation}
Since $v_1 > v_2$, we define the scale parameter $V = (v_1^2-v_2^2)^\half$ and a hyperbolic phase angle $\phi$ such that, $v_1 = V \cosh \phi, v_2 = V \sinh \phi$ implies that $\tanh \phi = \frac{v_2}{v_1}$. Substituting these definitions into the exponent's inner term gives $v_2 \sinh \theta - v_1 \cosh \theta = V \sinh \phi \sinh \theta - V \cosh \phi \cosh \theta = -V \cosh(\theta - \phi)$, which implies that
\begin{equation}
I = c \int_{0}^\infty e^{-a c^2 V^2 \cosh^2(\theta - \phi)} \cosh \theta \, {\rm d}\theta.\nn
\end{equation}
We shift our variable by setting $x = \theta - \phi$ and expand the remaining $\cosh \theta$ term for $\th=x+\phi$, to split the integral into two  parts as 
\begin{align}
I &= c \cosh \phi \int_{-\phi}^\infty e^{-a c^2 V^2 \cosh^2 x} \cosh x \, {\rm d}x  + c \sinh \phi \int_{-\phi}^\infty e^{-a c^2 V^2 \cosh^2 x} \sinh x \, {\rm d}x.\nn
\end{align}
Using the identity $\cosh^2 x = \sinh^2 x + 1$ in the first part, we get 
\begin{align}
I_1 &= c \cosh \phi \, e^{-a c^2 V^2} \int_{-\phi}^\infty e^{-a c^2 V^2 \sinh^2 x} \cosh x \, {\rm d}x.\nn
\end{align}
For $I_1$, we substitute $u = \sinh x$ to get
\begin{align}
I_1 &= c \cosh \phi \, e^{-a c^2 V^2} \int_{-\sinh\phi}^\infty e^{-a c^2 V^2 u^2} {\rm d}u.\nn
\end{align}
Using Gaussian integral, we get
\begin{align}
I_1 &= \frac{v_1 e^{-a c^2 (v_1^2 - v_2^2)}}{2 \sqrt{a} (v_1^2 - v_2^2)} \sqrt{\pi} \left[ 1 + \text{erf}\left(\sqrt{a} c v_2\right) \right].\nn
\end{align}
For $I_2$, we substitute $u = \cosh x$ to get 
\begin{align}
I_2 = c \sinh \phi  \int_{-\phi}^\infty e^{-a c^2 V^2 \cosh^2 x} \sinh x \, {\rm d}x &= c \sinh \phi \int_{\cosh \phi}^\infty e^{-a c^2 V^2 u^2} \, {\rm d}u \nn\\
&= \sinh \phi \frac{\sqrt{\pi}}{2V\sqrt{a}} \text{erfc}(\sqrt{a}c v_1) = \frac{v_2\sqrt{\pi}}{2(v_1^2-v_2^2)\sqrt{a}} \text{erfc}(\sqrt{a}c v_1).\nn
\end{align}
Combining $I_1$ and $I_2$ yields the final closed form
\begin{equation}
I = \frac{\sqrt{\pi} e^{-a c^2 (v_1^2 - v_2^2)}}{2 \sqrt{a} (v_1^2 - v_2^2)}  v_1 \left(1 + \text{erf}\left(\sqrt{a} c v_2\right)\right) + \frac{v_2\sqrt{\pi}}{2(v_1^2-v_2^2)\sqrt{a}} \text{erfc}\left(\sqrt{a} c v_1\right).
\label{eq:I3}
\end{equation}
Finally replacing  $c=t^{**}= \frac{1}{v_1v_2}(h_2^2-h_1^2)^\half (v_1^2-v_2^2)^\half$, $a=\la\pi \bar v^2$,  $\bar v=\frac{v_1v_2}{v_1^2-v_2^2}$ in (\ref{eq:I3}), we have 
\begin{align}
I &= \frac{\sqrt{\pi} e^{-a c^2 (v_1^2 - v_2^2)}}{2 \sqrt{a} (v_1^2 - v_2^2)}  v_1 \left(1 + \text{erf}\left(\sqrt{a} c v_2\right)\right) + \frac{v_2\sqrt{\pi}}{2(v_1^2-v_2^2)\sqrt{a}} \text{erfc}\left(\sqrt{a} c v_1\right)\nn\\
&= \frac{\sqrt{\pi} e^{-\la\pi (h_2^2 - h_1^2)}}{2 \sqrt{\la\pi} v_1v_2}  v_1 \left(1 + \text{erf}\left(\sqrt{\la\pi}\left(\frac{h_2^2-h_1^2}{v_1^2-v_2^2}\right)^\half v_2\right)\right) + \frac{v_2\sqrt{\pi}}{2v_1v_2\sqrt{\la\pi}} \text{erfc}\left(\sqrt{\la\pi} \left(\frac{h_2^2-h_1^2}{v_1^2-v_2^2}\right)^\half v_1\right)\nn\\
&=\frac{1}{2v_1v_2\sqrt{\la}} \left[ v_1 e^{-\la\pi (h_2^2 - h_1^2)}  \left(1 + \text{erf}\left(\sqrt{\la\pi}\left(\frac{h_2^2-h_1^2}{v_1^2-v_2^2}\right)^\half v_2\right)\right) + v_2\text{erfc}\left(\sqrt{\la\pi} \left(\frac{h_2^2-h_1^2}{v_1^2-v_2^2}\right)^\half v_1\right) \right]\nn\\
&= \frac{1}{2v_1v_2\sqrt{\la}} \left[ v_1 e^{-\la\pi (h_2^2-h_1^2)}  \left(1 + \text{erf}\left(\sqrt{\la\pi}\left(h_2^2-h_1^2\right)^\half \bar v_2\right)\right) + v_2  \text{erfc}\left( \sqrt{\la\pi}\left(h_2^2-h_1^2\right)^\half \bar v_1\right) \right],
\end{align}
where $\bar v_1:=\frac{v_1}{(v_1^2-v_2^2)^\half}$ and $\bar v_2:=\frac{v_2}{(v_1^2-v_2^2)^\half}$.
\end{proof}
\begin{proof}[Proof of Lemma~\ref{lemma:h1h2}]
We wish to evaluate
\[
I_1=\int_{0}^{\infty}\int_{0}^{h_1}
e^{-\lambda\pi h_2^2}\,
\text{erfc}\!\Big(\sqrt{\lambda\pi\left(h_1^2-h_2^2\right)}\;\bar v_1\Big)\,
{\rm d}h_2\,{\rm d}h_1.
\]
For brevity set $a=\lambda\pi, k=\bar v_1$, so that
\[
I_1=\int_0^\infty\int_0^{h_1} e^{-ah_2^2}\,
\text{erfc}\Big(\sqrt{a}\,k\sqrt{h_1^2-h_2^2}\Big)\,{\rm d}h_2\,{\rm d}h_1 .
\]
The domain of integration is $D=\{(h_1,h_2): 0\le h_2\le h_1<\infty\}$. Introduce polar-type coordinates
$h_1=\rho, h_2=\rho\cos\theta$ such that $\rho\in[0,\infty), \theta \in \left[0,\tfrac{\pi}{2}\right]$, and $\sqrt{h_1^2-h_2^2}=\rho\sin\theta$. By Fubini's theorem the double integral becomes
\begin{align}
I_1=\int_0^{\pi/2}\sin\theta\,{\rm d}\theta
\int_0^\infty \rho\,e^{-a\rho^2\cos^2\theta}\,
\text{erfc}\big(\sqrt{a}\,k\,\rho\sin\theta\big)\,{\rm d}\rho .
\label{eq:GI}
\end{align}
We use the following standard identity involving erfc function from~\cite[Integral 5, Subsection 4.3]{Ng-Geller69}:
\vspace{0.12in}
\begin{identity}
For any $p>0,\ q\ge 0$ 
\begin{align}
\int_0^\infty x\,e^{-px^2}\,\mbox{erfc}(qx)\,{\rm d}x
=\frac{1}{2p}\left[1-\frac{q}{\sqrt{p+q^2}}\right].
\label{eq:id1}
\end{align}
\label{identity:id1}
\end{identity}
Here $p=a\cos^2\theta, q=\sqrt{a}\,k\sin\theta$, and hence $p+q^2=a\cos^2\theta+a k^2\sin^2\theta
=a\big(\cos^2\theta+k^2\sin^2\theta\big)$. Hence the inner ($\rho$) integral in (\ref{eq:GI}) equals $\frac{1}{2a\cos^2\theta}
\left[1-\frac{k\sin\theta}{\sqrt{\cos^2\theta+k^2\sin^2\theta}}\right]$. Substituting back, we get

\[
I_1 =\frac{1}{2a}\int_0^{\pi/2}
\frac{\sin\theta}{\cos^2\theta}
\left[1-\frac{k\sin\theta}{\sqrt{\cos^2\theta+k^2\sin^2\theta}}\right]
{\rm d}\theta.
\]
Let $u=\cos\theta$. The integral becomes
\[
I_1 =\frac{1}{2a}\int_0^1 \frac{1}{u^2}
\left[1-\frac{k\sqrt{1-u^2}}{\sqrt{k^2+(1-k^2)u^2}}\right]{\rm d}u.
\]
Write the part in bracket as $F(u)=1-k\,g(u)$ where $g(u)=\frac{\sqrt{1-u^2}}{\sqrt{k^2+(1-k^2)u^2}}$.
First note that 
$g(0)=\frac1k$, and hence $F(0)=1-k g(0)=0$, so the bracket vanishes at $u=0$. 

Note that the integrand $F(u)/u^2$ in $I_1$ has a removable singularity at $u=0$. To see the precise order of vanishing, we expand for small $u$ and rewrite the function $g(u)^2$ as
\[
g(u)^2=\frac{1-u^2}{k^2+(1-k^2)u^2}
=\frac{1}{k^2}\cdot\frac{1-u^2}{1+\frac{1-k^2}{k^2}u^2}
=\frac{1}{k^2}\left[1-u^2\left(1+\frac{1-k^2}{k^2}\right)+O(u^4)\right]
=\frac{1}{k^2}-\frac{u^2}{k^4}+O(u^4),
\]
so 
\[
g(u)=\frac1k\sqrt{1-\frac{u^2}{k^2}+O(u^4)}
=\frac1k-\frac{u^2}{2k^3}+O(u^4),
\]
by using another binomial expansion. Therefore $F(u)=1-k\,g(u)=\frac{u^2}{2k^2}+O(u^4)$ as $u\to 0$. Consequently
\[
\frac{F(u)}{u^2}\;\xrightarrow[u\to0]{}\;\frac{1}{2k^2},
\qquad
\frac{F(u)}{u}\;\xrightarrow[u\to0]{}\;0,
\]
so both the integrand and the boundary term below are perfectly well behaved at $u=0$. Integrating by parts, we get that
\[
\int_0^1 \frac{1-k g(u)}{u^2}\,{\rm d}u
=\Big[-\frac{1-kg(u)}{u}\Big]_0^1+\int_0^1 \frac{-k\,g'(u)}{u}\,{\rm d}u .
\]

At $u=1$: $g(1)=0$, which gives the boundary value $-1$.
At $u=0$, we have that 
$\frac{F(u)}{u}\;\xrightarrow[u\to0]{}\;0$. This leaves
\[
\int_0^1\frac{1-kg(u)}{u^2}\,{\rm d}u=-1-k\int_0^1 \frac{g'(u)}{u}\,{\rm d}u.
\]
Differentiating the function $g(u)=\frac{\sqrt{1-u^2}}{\sqrt{k^2+(1-k^2)u^2}}$
\begin{align}
g'(u) 
&= -\frac{u}{(1-u^2)^{1/2}\big[k^2+(1-k^2)u^2\big]^{3/2}}.\nn
\end{align}
Therefore the integral $I_1$ becomes
\begin{align}
I_1 =\frac{1}{2a}\left[-1+k\int_0^1
\frac{du}{(1-u^2)^{1/2}\big[k^2+(1-k^2)u^2\big]^{3/2}}\right].
\label{eq:IJE}
\end{align}
Let $u=\sin\phi$. Then the integral is
\[
K:=\int_0^1
\frac{{\rm d}u}{(1-u^2)^{1/2}\big[k^2+(1-k^2)u^2\big]^{3/2}}
=\int_0^{\pi/2}
\frac{{\rm d}\phi}{\big[k^2+(1-k^2)\sin^2\phi\big]^{3/2}} .
\]
We write the denominator term as
\[
k^2+(1-k^2)\sin^2\phi
= k^2\left[1-\frac{k^2-1}{k^2}\sin^2\phi\right].
\]
Setting $m^2:=\frac{k^2-1}{k^2}$, we have that $0< m<1$, for $k=\bar v_1= \frac{v_1}{(v_1^2-v_2^2)^\half}>1$, and we rewrite the integral $K$ as 
\[
K=\frac{1}{k^3}\int_0^{\pi/2}\frac{{\rm d}\phi}{\big(1-m^2\sin^2\phi\big)^{3/2}} := \frac{1}{k^3} J(m).
\]
Let $E(m)=\int_0^{\pi/2}\sqrt{1-m^2\sin^2\psi}\,d\psi$ be the complete elliptic integral of the second kind. The following relation can be shown to hold among the functions $J(m)$ and $E(m)$, the proof of which is given later.
\begin{lemma}
For any $m\in (0,1)$ we have that 
\begin{align}
J(m)=\frac{E(m)}{1-m^2}.
\label{eq:JmE}
\end{align}
\label{lemma:JmEm}
\end{lemma}
Using (\ref{eq:JmE}) we have that
\[
K= \frac{1}{k^3}\int_0^{\pi/2}\frac{{\rm d}\phi}{\big(k^2\cos^2\phi+\sin^2\phi\big)^{3/2}}
=\frac{1}{k^3 \left(1-\frac{k^2-1}{k^2}\right)}\,E\!\left(\sqrt{\frac{k^2-1}{k^2}}\right)= \frac{1}{k}\,E\!\left(\frac1k\sqrt{k^2-1}\right).
\]
Assembling the results back to (\ref{eq:IJE}) we finally, get that
\[
I_1 =\frac{1}{2a}\left[-1+k\cdot\frac{1}{k}\,E\!\left(\frac1k\sqrt{k^2-1}\right)\right]
=\frac{1}{2a}\left[E\!\left(\frac1k\sqrt{1-k^2}\right)-1\right].
\]
Restoring $a=\lambda\pi$ and $k=\bar v_1$ and using the relations $\bar v_1^{\,2}-\bar v_2^2=1$ and $\frac{\bar v_2}{\bar v_1}= \frac{v_2}{v_1}$, we obtain that
\begin{align}
I_1 &=\frac{1}{2\pi\lambda}
\left[E\!\left(\frac{1}{\bar v_1}\sqrt{\bar v_1^{\,2}-1}\right)-1\right]
= \frac{1}{2\pi\lambda}\left[E\!\left(\frac{\bar v_2}{\bar v_1}\right)-1\right] = \frac{1}{2\pi\lambda}\left[E\!\left(\frac{v_2}{ v_1}\right)-1\right].\nn
\end{align}
We now wish to evaluate the second integral 
\[
I_2=\int_{0}^{\infty}\int_{0}^{h_1}
e^{-\lambda\pi h_1^2}\,
\text{erfc}\!\Big(\sqrt{\lambda\pi\left(h_1^2-h_2^2\right)}\;\bar v_2\Big)\,
{\rm d}h_2\,{\rm d}h_1 .
\]
Again set $a=\lambda\pi, k=\bar v_2$, so that
\[
I_2=\int_0^\infty\int_0^{h_1} e^{-ah_1^2}\,
\text{erfc}\Big(\sqrt{a}\,k\sqrt{h_1^2-h_2^2}\Big)\,{\rm d}h_2\,{\rm d}h_1 .
\]
Fix $h_1$ and substitute $h_2=h_1\sin\phi$, so that $\sqrt{h_1^2-h_2^2}=h_1\cos\phi$.
Then the inner integral is
\[
\int_0^{h_1}\text{erfc}\!\Big(\sqrt a\,k\sqrt{h_1^2-h_2^2}\Big)\,{\rm d}h_2
=h_1\int_0^{\pi/2}\cos\phi\,\text{erfc}\big(\sqrt a\,k\,h_1\cos\phi\big)\,{\rm d}\phi.
\]
Substituting into $I_2$ and swapping the (now finite) $\phi$-integral with the $h_1$-integral (justified by absolute convergence, since the integrand is non-negative and bounded by an integrable Gaussian in $h_1$),
\begin{align}
I_2=\int_0^{\pi/2}\cos\phi\,{\rm d}\phi
\int_0^\infty h_1\,e^{-ah_1^2}\,\text{erfc}\big(\sqrt a\,k\,h_1\cos\phi\big)\,{\rm d}h_1.
\label{eq:GIa}
\end{align}
We again use the Identity~\ref{identity:id1} with  $p=a$ and $q=\sqrt a\,k\cos\phi$, so that
$p+q^2=a+a k^2\cos^2\phi=a\big(1+k^2\cos^2\phi\big)$. The inner integral in (\ref{eq:GIa}) therefore equals $\frac{1}{2a}\left[1-\frac{k\cos\phi}{\sqrt{1+k^2\cos^2\phi}}\right]$. Hence the total integral becomes
\[
I_2=\frac{1}{2a}\int_0^{\pi/2}\cos\phi
\left[1-\frac{k\cos\phi}{\sqrt{1+k^2\cos^2\phi}}\right]{\rm d}\phi
=\frac{1}{2a}\Big(1-k\,L\Big),
\]
where $L:= \int_0^{\pi/2}\frac{\cos^2\phi}{\sqrt{1+k^2\cos^2\phi}}\,d\phi$.
Rewriting the denominator using $\cos^2\phi=1-\sin^2\phi$, we get
\[
1+k^2\cos^2\phi=1+k^2-k^2\sin^2\phi=(1+k^2)\left[1-\frac{k^2}{1+k^2}\sin^2\phi\right].
\]
Define $m=\frac{k}{\sqrt{1+k^2}}$, and note that  $m \in[0,1)$. Then
\begin{align}
L=\frac{1}{\sqrt{1+k^2}}\int_0^{\pi/2}\frac{\cos^2\phi}{\sqrt{1-m^2\sin^2\phi}}\,{\rm d}\phi .
\label{eq:Lint}
\end{align}
Let $w(\phi) = 1-m^2\sin^2\phi$, which implies that $\sin^2\phi=\frac{1-w}{m^2}$. Then
\[
\cos^2\phi=1-\sin^2\phi=1-\frac{1-w}{m^2}=\frac{w-(1-m^2)}{m^2}.
\]
Substituting $\cos^2\phi$ in (\ref{eq:Lint}) we get
\[
\int_0^{\pi/2}\frac{\cos^2\phi}{\sqrt{w}}\,{\rm d}\phi
=\frac{1}{m^2}\int_0^{\pi/2}\frac{w-(1-m^2)}{\sqrt w}\,{\rm d}\phi
=\frac{1}{m^2}\left[\int_0^{\pi/2}\sqrt{w}\,{\rm d}\phi-(1-m^2)\int_0^{\pi/2}\frac{{\rm d}\phi}{\sqrt w}\right].
\]
Using the complete elliptic integrals
\[
E(m)=\int_0^{\pi/2}\sqrt{1-m^2\sin^2\phi}\,{\rm d}\phi ,
\qquad
K(m)=\int_0^{\pi/2}\frac{{\rm d}\phi}{\sqrt{1-m^2\sin^2\phi}} ,
\]
we get that
\[
\int_0^{\pi/2}\frac{\cos^2\phi}{\sqrt{1-m^2\sin^2\phi}}\,{\rm d}\phi
=\frac{1}{m^2}\Big[E(m)-(1-m^2)K(m)\Big].
\]
Assembling the simplified expressions, we get
\[
L=\frac{1}{\sqrt{1+k^2}}\cdot\frac{1}{m^2}\Big[E(m)-(1-m^2)K(m)\Big].
\]
Since $m^2=\dfrac{k^2}{1+k^2}$, we have
\[
L=\frac{1}{\sqrt{1+k^2}}\cdot\frac{1+k^2}{k^2}
\left[E(m)-\frac{K(m)}{1+k^2}\right]
=\frac{\sqrt{1+k^2}}{k^2}\left[E(m)-\frac{K(m)}{1+k^2}\right].
\]
From $I_2=\dfrac{1}{2a}\big(1-kL\big)$, we obtain
\[
I_2=\frac{1}{2a}\left[1-\frac{\sqrt{1+k^2}}{k}\,E(m)+\frac{K(m)}{k\sqrt{1+k^2}}\right]
=\frac{1}{2a}\left[1+\frac{K(m)-(1+k^2)E(m)}{k\sqrt{1+k^2}}\right].
\]
Restoring $a=\lambda\pi$, $k=\bar v_2$ and $m=\frac{k}{\sqrt{1+k^2}}= \frac{\bar v_2}{\bar v_1}= \frac{v_2}{v_1}$, we get 
\begin{align}
I_2 &=\frac{1}{2\la\pi}\left[1+\frac{K(v_2/v_1)-\bar v_1^2 E(v_2/v_1)}{\bar v_2 \bar v_1}\right] = \frac{1}{2\la\pi}\left[1+\frac{K(v_2/v_1)}{\bar v_2 \bar v_1}- \frac{v_1}{v_2} E(v_2/v_1)\right].\nn
\end{align}
\end{proof}
\begin{proof}[Proof of Lemma~\ref{lemma:JmEm}] 
For $0\le m<1$, we prove that 
\begin{align}
J(m)=\frac{E(m)}{1-m^2}.
\label{eq:Jm}
\end{align}
For this, we first note that
\begin{align}
\sin^2\psi=\frac{1-\big(1-m^2\sin^2\psi\big)}{m^2}.
\label{eq:sin-psi}
\end{align}
Using this we state further useful identities. Define
\[
A_1(m):=\int_0^{\pi/2}\sin^2\psi\,(1-m^2\sin^2\psi)^{-1/2}\,{\rm d}\psi .
\]
Using the identity (\ref{eq:sin-psi}), we have
\begin{equation}
A_1(m)=\frac{1}{m^2}\int_0^{\pi/2}\Big[(1-m^2\sin^2\psi)^{-1/2}-(1-m^2\sin^2\psi)^{1/2}\Big]{\rm d}\psi
=\frac{K(m)-E(m)}{m^2}. \label{eq:A1}
\end{equation}

Similarly, define
\[
B_1(m):=\int_0^{\pi/2}\sin^2\psi\,(1-m^2\sin^2\psi)^{-3/2}\,{\rm d}\psi .
\]
Using the same identity (\ref{eq:sin-psi}),
\begin{equation}
B_1(m)=\frac1{m^2}\int_0^{\pi/2}\Big[(1-m^2\sin^2\psi)^{-3/2}-(1-m^2\sin^2\psi)^{-1/2}\Big]{\rm d}\psi
=\frac{J(m)-K(m)}{m^2}. 
\label{eq:B1}
\end{equation}
Let us define 
\[
B_2(m):=\int_0^{\pi/2}\sin^4\psi\,(1-m^2\sin^2\psi)^{-3/2}\,{\rm d}\psi.
\]
Using the same identity (\ref{eq:sin-psi}) again, we have 
\begin{equation}
B_2(m)=\frac1{m^2}\int_0^{\pi/2}\sin^2\psi\Big[(1-m^2\sin^2\psi)^{-3/2}-(1-m^2\sin^2\psi)^{-1/2}\Big]{\rm d}\psi
=\frac{B_1(m)-A_1(m)}{m^2}.
\label{eq:B2}
\end{equation}
Differentiating $E(m)$ under the integral sign, we get
\[
\frac{{\rm d}E(m)}{{\rm d}m}=\int_0^{\pi/2}\frac{-m\sin^2\psi}{\sqrt{1-m^2\sin^2\psi}}\,{\rm d}\psi=-m\,A_1(m) =\frac{E(m)-K(m)}{m}.
\]
Also note that $\frac{{\rm d}K(m)}{{\rm d}m}=m\,B_1(m)$. We now claim that the following identity holds true:
\vspace{0.12in}
\begin{identity}
For any $m\in (0,1)$ we have that
\begin{align}
K(m)- A_1(m)+ (m^2-1)B_1(m) &= 0.   
\label{eq:id}
\end{align}
\label{identity:K0}
\end{identity}
Substituting $A_1(m)= \frac{K(m)-E(m)}{m^2}$ and $B_1(m)= \frac 1 m \frac{{\rm d}K(m)}{{\rm d}m}$ in the (\ref{eq:id}), we get
\[
K(m)-\frac{K(m)-E(m)}{m^2}+\Big(m^2-1\Big)\frac1m\frac{{\rm d}K(m)}{{\rm d}m} = 0.
\]
Using this, we solve for $\frac{{\rm d}K(m)}{{\rm d}m}$ as
\begin{equation}
\frac{{\rm d}K(m)}{{\rm d}m}=\frac{m}{m^2-1}\frac{(1-m^2)K(m)-E(m)}{m^2}
=\frac{E(m)-(1-m^2)K(m)}{m(1-m^2)}.
\label{eq:dK}
\end{equation}
From (\ref{eq:B1}) we have $J(m)=m^2B_1(m)+K(m)$, and since $B_1(m)=\frac 1 m\frac{{\rm d}K(m)}{{\rm d}m}$, we further have that
\[
J(m)=m\,\frac{{\rm d}K(m)}{{\rm d}m}+K(m). 
\]
Substituting the expression for $\frac{{\rm d}K(m)}{{\rm d}m}$ from (\ref{eq:dK}), we finally get that
\[
J(m)=m\cdot\frac{E(m)-(1-m^2)K(m)}{m(1-m^2)}+K(m) = \frac{E(m)}{1-m^2}.
\]
\end{proof}
\begin{proof}[Proof of Identity~\ref{identity:K0}]
Consider the function $h(\psi)= \sin\psi\cos\psi (1-m^2\sin^2\psi)^{-\half}$, and note that 
\begin{align}
    \int_0^{\pi/2} h'(\psi){\rm d}\psi = h(\pi/2)-h(0)= 0.
\end{align}
The derivative of the function $h$ is
\begin{align}
h'(\psi) &= (\cos^2\psi - \sin^2\psi)(1 - m^2\sin^2\psi)^{-\half} + m^2\sin^2\psi\cos^2\psi (1 - m^2\sin^2\psi)^{-\frac 3 2}\nn\\
&= (1 - 2\sin^2\psi)(1 - m^2\sin^2\psi)^{-\half} + m^2\sin^2\psi (1 - m^2\sin^2\psi)^{-\frac 3 2}- m^2\sin^4\psi (1 - m^2\sin^2\psi)^{-\frac 3 2}.\nn
\end{align}
Hence the integral of $h'(\psi)$ is
\begin{align}
0 = \int_0^{\pi/2} h'(\psi){\rm d}\psi &= K(m)- 2A_1(m)+ m^2 B_1(m)-  m^2 B_2(m)\nn\\
&=  K(m)- A_1(m)+ (m^2-1) B_1(m),\nn
\end{align}
using the fact that $m^2B_2(m)= B_1(m)-A_1(m)$ from (\ref{eq:B2}).
\end{proof}
%


\let\oldaddcontentsline\addcontentsline
\renewcommand{\addcontentsline}[3]{}%

\section*{Acknowledgement}\label{sec-Acknowledgement}
The authors would like to thank Bart\l{}omiej B\l{}aszczyszyn, Rapha{\"e}l Lachi\`{e}ze-Rey and Gourab Ghatak for many important discussions and for providing many references. The authors are thankful to the ERC-NEMO grant, under the European Union’s Horizon 2020 research and innovation program, grant agreement number 788851 to INRIA Paris. This research was also funded in part by the France 2030 BPI ``5G NTN mmWave'' project to T{\'e}l{\'e}com Paris.

\let\addcontentsline\oldaddcontentsline 


\begin{thebibliography}{10}

\bibitem{Albers-Roos}
G.~{Albers} and T.~{Roos}.
\newblock Voronoi diagrams of moving points in higher dimensional spaces.
\newblock{\em Lecture Notes in Computer Science}, 621.
\newblock{(Proceedings of the Third Scandinavian Workshop on Algorithm Theory, SWAT 1992}. Springer, Berlin, Heidelberg.

\bibitem{Albers-Th}
G.~{Albers}.
\newblock Three-Dimensional Dynamic Voronoi Diagrams.
\newblock{\em Diploma thesis, University of W{\"u}rzburg}, 1991.

\bibitem{Andrews-etal}
J.~G.~{Andrews}, F.~{Baccelli} and R.~K.~{Ganti}.
\newblock A Tractable Approach to Coverage and Rate in Cellular Networks. 
\newblock{\em IEEE Trans. Comm.}, 59(11):3122-3134, 2010. 

\bibitem{Andrieu-Moulines-MCMC}
C.~{Andrieu} and \'{E}.~{Moulines}.
\newblock On the ergodicity properties of some adaptive MCMC algorithms.
\newblock{\em Ann. Appl. Probab.}, 16(3):1462--1505, 2006.

\bibitem{Amdeberhan-etal}
T.~{Amdeberhan}, M.~L.~{Glasser}, M.~C.~{Jones}, V.~H.~{Moll}, R.~{Posey} and D.~{Varela}
\newblock The Cauchy-Schl{\"o}milch transformation.
\newblock{\em arXiv:1004.2445}, 2010.

\bibitem{Athreya-Ney}
K.~B.~{Athreya} and P.~{Ney}. 
\newblock A New Approach to the Limit Theory of Recurrent Markov Chains.
\newblock{\em Trans. Amer. Math. Soc.}, 245:493–501, 1978.

\bibitem{Aurenhammer}
F.~{Aurenhammer}. 
\newblock Power diagrams: properties, algorithms and applications.
\newblock{\em SIAM Journal on Computing}, 16 (1): 78-–96, 1987.

\bibitem{Baccelli-Bartek-Karray}
F.~{Baccelli}, B.~{B\l{}aszczyszyn} and M.~{Karray}.
\newblock Random Measures, Point Processes, Stochastic Geometry.
\newblock {\em Inria}, 2020.
\newblock \url{https://hal.inria.fr/hal-02460214}.
 
\bibitem{Baccelli-Bremaud}
F.~{Baccelli} and P.~{Br{\'e}maud}.
\newblock Elements of Queueing Theory: Palm Martingale Calculus and Stochastic Recurrences.
\newblock {\em Springer Berlin, Heidelberg}, 2003.

\bibitem{Baccelli-Choi}
F.~{Baccelli} and C.~S.~{Choi}.
\newblock A Novel Analytical Model for LEO Satellite
 Constellations Leveraging Cox Point Processes,
\newblock {\em IEEE Trans. Commun.}, 73(4):2265--2279, 2025.

\bibitem{BKLZ}
F.~{Baccelli}, M.~{Klein}, M.~{Lebourges} and S.~{Zuyev}.
\newblock Stochastic geometry and architecture of communication networks.
\newblock {\em Telecommun. Syst.}, 7:209--227, 1997.

\bibitem{Baccelli_Madadi_Gustavo}
F.~{Baccelli}, P.~{Madadi} and G. ~{de Veciana}.
\newblock Shared Rate Process for Mobile Users in Poisson Networks and Applications.
\newblock {\em IEEE Trans. Inf. Theory}, 64(3):2121--2141, 2018.

\bibitem{Baccelli_Zuyev}
F.~{Baccelli} and S.~{Zuyev}.
\newblock Stochastic geometry models of mobile communication networks. 
\newblock In {\em Frontiers in Queueing: Models and Applications in Science and Engineering}, 
\newblock ed. J. Dshalalow, Chapter 8:227--243. CRC Press, 1996.

\bibitem{Baccelli_Zuyev2}
F.~{Baccelli} and S.~{Zuyev}. \newblock Poisson-Voronoi Spanning Trees with Applications to the Optimization of
Communication Networks.
\newblock{Oper. Res.}, 47(4):619--631, 1999.

\bibitem{Banagar-etal}
M.~{Banagar}, V.~V.~{Chetlur} and H.~S.~{Dhillon}.
\newblock Stochastic Geometry-based Performance Analysis of Drone Cellular Networks. \newblock{\em Book chapter in: UAV Communications for 5G and Beyond. Y. Zeng, I. Guvenc, R. Zhang, G. Geraci, and D. W. Matolak, eds. Wiley}, 233--254, 2020.

\bibitem{Alex-Beutel}
A.~{Beutel}.
\newblock Interactive Voronoi Diagram Generator with WebGL.
\newblock{\url{https://alexbeutel.com/webgl/voronoi.html}}, 2012.

\bibitem{Choi-Baccelli-MA}
C.~S.~{Choi} and F.~{Baccelli}.
\newblock Cox Point Processes for Multi Altitude LEO Satellite Networks.
\newblock{\em IEEE Trans. Veh. Technol.}, 73(10):15916--15921, 2024.

\bibitem{Choi-Baccelli-WC}
C.~S.~{Choi} and F.~{Baccelli}.
\newblock Stochastic Geometry and Dynamical System Analysis of Walker Satellite Constellations.
\newblock{\em IEEE Trans. Veh. Technol.}, 1--6, 2025. 

\bibitem{Dedecker-Rio-FCLT}
J.~{Dedecker} and E.~{Rio}.
\newblock On the functional central limit theorem for stationary processes.
\newblock{\em Ann. Inst. H. Poincaré Probab. Stat.}, 36(1):1-34, 2000.

\bibitem{Diaz-etal}
J.~{Diaz}, D.~{Mitsche} and X.~{Perez}.
\newblock Dynamic random geometric graphs. 
\newblock{\em arXiv preprint cs/0702074}, 2007.

\bibitem{Fortelle-LDP-MC}
A.~{de La Fortelle}.
\newblock Large Deviation Principle for Markov Chains in Continuous Time.
\newblock{\em Probl. Inf. Transm.}, 37:120-–139, 2001.

\bibitem{Gilbert}
E.~N.~{Gilbert}.
\newblock Random subdivisions of space into chrystals.
\newblock {\em Adv. Appl. Prob.}, 33:958--972, 1962.

\bibitem{Gowda83}
I.~{Gowda}, D.~{Kirkpatrick}, D.~{Lee} and A.~{Naamad}.
\newblock Dynamic Voronoi diagrams.
\newblock{\em IEEE Trans. Inf. Theory}, 29(5):724--731, 1983.

\bibitem{Haenggi-Qcells}
M.~{Haenggi}.
\newblock Q Cells in Wireless Networks.
\newblock{\em IEEE Trans. Wireless Commun.}, 25:9719--9730, 2026.

\bibitem{L_Heinrich}
L.~{Heinrich}.
\newblock Contact and chord length distribution of a stationary Voronoi tessellation.
\newblock {\em Adv. Appl. Probab.}, 30:603--618, 1998.

\bibitem{Jensen-Rahbek-LLN}
S.~T.~{Jensen} and A.~{Rahbek}.
\newblock On the law of large numbers for (geometrically) ergodic Markov chains. \newblock{\em Econ. Theory}, 23(4):761--766, 2007.

\bibitem{Jones-CLT}
G.~L.~{Jones}.
\newblock On the Markov chain central limit theorem.
\newblock{\em Probab. Surveys}, 1:299--320, 2004.

\bibitem{JOS-FCLT}
P.~{Jung}, T.~{Owada} and G.~{Samorodnitsky}.
\newblock Functional central limit theorem for a class of negatively dependent heavy-tailed stationary infinitely divisible processes generated by conservative flows.
\newblock{\em Ann. Probab.}, 45(4):2087--2130, 2017.

\bibitem{kingman}
J.~F.~C.~{Kingman}.
\newblock Poisson Processes,
\newblock {\em Oxford University Press, USA}, 1993.

\bibitem{Slim-etal}
Z.~{Lou}, B.~E.~Y.~{Belmekki} and M.~S.~{Alouini}.
\newblock Coverage Analysis of Large-Scale HAPS Networks Using Directional Beams.
\newblock{\em IEEE Trans. Aerosp. Electron. Syst.}, 61(4):9260--9275, 2025.

\bibitem{Mecke1981}
J.~{Mecke}. 
\newblock Formulas for stationary planar fiber processes III—intersection with fiber systems.
\newblock{Math. Oper. Statist.} 12:201--210, 1981.

\bibitem{Miles}
R.~E.~{Miles}.
\newblock Sectional Voronoi tessellation.
\newblock {\em Rev. Un. Mat. Argentina}, 29:310--327, 1984.

\bibitem{Jasper_Moller}
J.~{M{\"o}ller}.
\newblock Random tessellations in $\R^d$.
\newblock{Adv. Appl. Probab.}, 21:37--73, 1989.

\bibitem{Muche2010}
L.~{Muche}.
\newblock Contact and Chord Length Distribution Functions of the Poisson-Voronoi Tessellation in High Dimensions.
\newblock{\em Adv. Appl. Probab.} 42(1):48--68, 2010.

\bibitem{Muche_Stoyan}
L.~{Muche} and D.~{Stoyan}.
\newblock Contact and chord length distributions of the Poisson Voronoi tessellation.
\newblock {\em J. Appl. Probab.}, 29:467--471, 1992.

\bibitem{Meyn-Tweedie}
S.P.~{Meyn} and R.~L.~{Tweedie}.
\newblock Markov Chains and Stochastic Stability. 
\newblock{\em Springer-Verlag}, 1993.

\bibitem{Ng-Geller69}
E.~W.~{Ng} and M.~{Geller}. 
\newblock A Table of Integrals of the Error Functions.
\newblock{\em J. Res. Natl. Bur. Stand. B Math. Sci.}, 
\newblock 73~B(1), Section 4, 1969.

\bibitem{Nummelin}
E.~{Nummelin}.
\newblock A splitting technique for Harris recurrent Markov chains. 
\newblock{\em Z. Wahrsch. Verw. Gebiete}, 43, 309–-318, 1978.

\bibitem{Okabe2000}
A.~{Okabe}, B.~{Boots}, K.~{Sugihara} and S.~N.~{Chiu}. 
\newblock Spatial tessellations:
Concepts and applications of Voronoi diagrams. 
\newblock{\em John Wiley and Sons}, second edition, 2000.

\bibitem{Okati-2020}
N.~{Okati}, T.~{Riihonen}, D.~{Korpi}, I.~{Angervuori} and R.~{Wichman}. 
\newblock Downlink Coverage and Rate Analysis of
Low Earth Orbit Satellite Constellations
Using Stochastic Geometry.
\newblock{\em IEEE Trans. Commun.}, 68(8):5120--5134, 2020.

\bibitem{Okati-thesis}
N.~{Okati}. 
\newblock Modeling and Analysis of Massive Low Earth Orbit Communication Networks. 
\newblock{\em Ph.D. Thesis, Tempere University}:
\newblock{\url{https://trepo.tuni.fi/handle/10024/146346}}, 2023.

\bibitem{Palm}
C.~{Palm}. 
\newblock Intensit{\"a}tsschwankungen im fernsprechverkehr.
\newblock{\em Ericsson Technics}, 1943.

\bibitem{duchemin-HMC}
D.~{Quentin}.
\newblock Growth dynamics of large networks using hidden Markov chains.
\newblock{\em Ph.D. Thesis, Universit{\'e} Gustave Eiffel}: \newblock{\url{https://theses.hal.science/tel-03749513}}, 2022.

\bibitem{T-Roos93}
T.~{Roos}.
\newblock Voronoi diagrams over dynamic scenes.
\newblock{\em Discrete Appl. Math.}, 43(3):243--259, 1993.

\bibitem{Roos-th}
T.~{Roos}.
\newblock Dynamic Voronoi diagrams.
\newblock{\em Ph.D. Thesis, University of W{\"u}rzburg}, 1991. 

\bibitem{Roberts-Rosenthal-GeoEr}
G.~{Roberts} and J.~{Rosenthal}.
\newblock Geometric Ergodicity and Hybrid Markov Chains.
\newblock{\em Electron. Commun. Probab.}, 2:13--25, 1997.


\bibitem{Roberts-Rosenthal-MCMC}
G.~{Roberts} and J.~{Rosenthal}.
\newblock Markov-Chain Monte Carlo: Some Practical Implications of Theoretical Results.
\newblock{\em Canad. J. Stat.}, 26(1):5--20, 1998.

\bibitem{Salehi-Hossain}
M.~{Salehi} and E.~{Hossain}. \newblock{Stochastic Geometry Analysis of Sojourn Time in Multi-Tier Cellular Networks}. \newblock{\em IEEE Trans. Wireless Commun.}, 20(3):1816--1830, 2021.

\bibitem{Stoyan}
D.~{Stoyan}. 
\newblock Comparison Methods for Queues and other Stochastic Models. {\em J. Wiley and Sons}, Chichester, 1983. 

\bibitem{Tanelli-etal}
I.~M.~{Tanash}, A.~K.~{Dwivedi} and T.~{Riihonen}.
\newblock Urban RIS-Assisted HAP Networks: Performance Analysis Using Stochastic Geometry.
\newblock{\em 2025 IEEE 101st Vehicular Technology Conference (VTC2025-Spring), Oslo, Norway}, 2025.

\bibitem{Touchette-LDP}
H.~{Touchette}.
\newblock A basic introduction to large deviations: Theory, applications, simulations.
\newblock{\em arXiv:1106.4146}, 2011.
\end{thebibliography}
\end{document}